\documentclass{memo-l}

\usepackage{amssymb,amscd,amsxtra,soul}

\usepackage{hyperref}

\usepackage{esint}

\usepackage[shortlabels]{enumitem}
\usepackage[capitalise]{cleveref}
\crefrangeformat{enumi}{items #3#1#4--#5#2#6}

\usepackage[all]{xy}

\hypersetup{colorlinks=false}

\usepackage{orcidlink}

\usepackage{color}
\definecolor{darkgreen}{cmyk}{1,0,1,.2}
\definecolor{m}{rgb}{1,0.1,1}
\long\def\red#1{\textcolor {red}{#1}}

\newdimen\theight
\def\TeXref#1{%
             \leavevmode\vadjust{\setbox0=\hbox{{\tt
                     \quad\quad  {\small \textrm #1}}}%
             \theight=\ht0
             \advance\theight by \lineskip
             \kern -\theight \vbox to
             \theight{\rightline{\rlap{\box0}}%
             \vss}%
             }}%
\newcommand{\fnote}[1]{\TeXref{*}{\tiny{#1}}}

\long\def\rnote#1{\red{\fnote{#1}}}

\theoremstyle{plain}

\newtheorem{thm}{Theorem}[section]
\newtheorem{lem}[thm]{Lemma}
\newtheorem{cor}[thm]{Corollary}
\newtheorem{prop}[thm]{Proposition}

\theoremstyle{definition}

\newtheorem{ex}[thm]{Example}

\theoremstyle{remark}

\newtheorem{rem}[thm]{Remark}
\newtheorem{claim}[thm]{Claim}

\newtheorem{notation}[thm]{Notation}

\crefname{thm}{Theorem}{Theorems}
\crefname{lem}{Lemma}{Lemmas}
\crefname{cor}{Corollary}{Corollaries}
\crefname{prop}{Proposition}{Propositions}
\crefname{mainthm}{Theorem}{Theorems}
\crefname{maincor}{Corollary}{Corollaries}

\crefname{defn}{Definition}{Definitions}
\crefname{conj}{Conjecture}{Conjectures}
\crefname{ex}{Example}{Examples}
\crefname{exs}{Examples}{Examples}
\crefname{prob}{Problem}{Problems}
\crefname{quest}{Question}{Questions}

\crefname{rem}{Remark}{Remarks}
\crefname{claim}{Claim}{Claims}
\crefname{case}{Case}{Cases}
\crefname{hyp}{Hypothesis}{Hypotheses}
\crefname{notation}{Notation}{Notations}

\newcommand{\C}{\mathbb{C}}
\newcommand{\K}{\mathbb{K}}
\newcommand{\N}{\mathbb{N}}

\newcommand{\R}{\mathbb{R}}

\newcommand{\Z}{\mathbb{Z}}

\renewcommand{\AA}{\mathcal{A}}

\newcommand{\CC}{\mathcal{C}}

\newcommand{\FF}{\mathcal{F}}
\newcommand{\GG}{\mathcal{G}}
\newcommand{\HH}{\mathcal{H}}

\newcommand{\KK}{\mathcal{K}}
\newcommand{\LL}{\mathcal{L}}
\newcommand{\PP}{\mathcal{P}}
\newcommand{\RR}{\mathcal{R}}
\renewcommand{\SS}{\mathcal{S}}

\newcommand{\VV}{\mathcal{V}}

\newcommand{\fg}{\mathfrak{g}}

\newcommand{\fG}{\mathfrak{G}}

\newcommand{\fU}{\mathfrak{U}}
\newcommand{\fX}{\mathfrak{X}}

\newcommand{\bfc}{\boldsymbol{c}}

\newcommand{\bfg}{\boldsymbol{g}}
\newcommand{\bfh}{\boldsymbol{h}}

\newcommand{\bfr}{\boldsymbol{r}}
\newcommand{\bfs}{\boldsymbol{s}}

\newcommand{\bfx}{\boldsymbol{x}}

\newcommand{\bfE}{\boldsymbol{E}}
\newcommand{\bfF}{\boldsymbol{F}}
\newcommand{\bfG}{\boldsymbol{G}}
\newcommand{\bfH}{\boldsymbol{H}}

\newcommand{\bfK}{\boldsymbol{K}}

\newcommand{\bfV}{\boldsymbol{V}}
\newcommand{\bfM}{\boldsymbol{M}}
\newcommand{\bfP}{\boldsymbol{P}}

\newcommand{\bfT}{\boldsymbol{T}}

\newcommand{\bfFF}{\boldsymbol{\FF}}

\newcommand{\bfphi}{\boldsymbol{\phi}}

\newcommand{\bfnu}{\boldsymbol{\nu}}

\newcommand{\bfrho}{\boldsymbol{\rho}}
\newcommand{\bfpi}{\boldsymbol{\pi}}
\newcommand{\bfvarpi}{\boldsymbol{\varpi}}

\newcommand{\bfeta}{\boldsymbol{\eta}}
\newcommand{\bfomega}{\boldsymbol{\omega}}

\newcommand{\sh}{\mathsf{h}}
\newcommand{\sw}{\mathsf{w}}
\newcommand{\sA}{\mathsf{A}}

\newcommand{\sH}{\mathsf{H}}
\newcommand{\sJ}{\mathsf{J}}

\newcommand{\sT}{\mathsf{T}}

\newcommand{\supp}{\operatorname{supp}}
\newcommand{\esssup}{\operatorname{ess\,sup}}
\newcommand{\codim}{\operatorname{codim}}

\newcommand{\im}{\operatorname{im}}

\newcommand{\Cl}{\operatorname{Cl}}
\newcommand{\id}{\operatorname{id}}
\newcommand{\Tr}{\operatorname{Tr}}

\newcommand{\Aut}{\operatorname{Aut}}

\newcommand{\Diffeo}{\operatorname{Diffeo}}
\newcommand{\Diff}{\operatorname{Diff}}
\newcommand{\GL}{\operatorname{GL}}
\newcommand{\Hol}{\operatorname{Hol}}

\newcommand{\sign}{\operatorname{sign}}
\newcommand{\Fix}{\operatorname{Fix}}
\newcommand{\rank}{\operatorname{rank}}
\newcommand{\End}{\operatorname{End}}
\newcommand{\dom}{\operatorname{dom}}
\newcommand{\vol}{\operatorname{vol}}
\newcommand{\dvol}{\operatorname{dvol}}
\newcommand{\ev}{\operatorname{ev}}
\newcommand{\Pen}{\operatorname{Pen}}

\newcommand{\inj}{\operatorname{inj}}
\newcommand{\pr}{\operatorname{pr}}
\newcommand{\lb}{\operatorname{lb}}
\newcommand{\rb}{\operatorname{rb}}
\newcommand{\ff}{\operatorname{ff}}
\newcommand{\length}{\operatorname{length}}
\newcommand{\sing}{\operatorname{sing}}

\newcommand{\fiber}{{\text{\rm fiber}}}
\newcommand{\topd}{{\text{\rm top}}}
\newcommand{\co}{{\text{\rm c}}}
\newcommand{\cv}{{\text{\rm cv}}}
\newcommand{\trans}{{\text{\rm t}}}
\newcommand{\rnabla}{\mathring{\nabla}}
\newcommand{\rP}{\mathring{P}}

\newcommand{\rexp}{\mathring{\exp}}

\newcommand{\cnabla}{\check{\nabla}}
\newcommand{\cexp}{\check{\exp}}
\newcommand{\cO}{\check{O}}
\newcommand{\cV}{\check{V}}

\newcommand{\bTr}{{}^{\text{\rm b}}\!\Tr}

\newcommand{\Str}{\operatorname{Str}}
\newcommand{\str}{\operatorname{str}}
\newcommand{\bStr}{{}^{\text{\rm b}}\!\Str}
\newcommand{\bOmega}{{}^{\text{\rm b}}\Omega}
\newcommand{\bT}{{}^{\text{\rm b}}T}
\newcommand{\bsigma}{{}^{\text{\rm b}}\!\sigma}
\newcommand{\Psib}{\Psi_{\text{\rm b}}}
\newcommand{\Diffb}{\Diff_{\text{\rm b}}}
\newcommand{\Deltab}{\Delta_{\text{\rm b}}}
\newcommand{\Hb}{H_{\text{\rm b}}}
\newcommand{\nuint}{\!\!\!\!\!{\phantom{\int}}^\nu\!\!\!\!\int}
\newcommand{\smallnuint}{\!\!\!{\phantom{\int}}^\nu\!\!\!\int}
\newcommand{\muint}{\!\!\!\!\!{\phantom{\int}}^\mu\!\!\!\!\int}

\newcommand{\nulint}{\!\!\!\!\!{\phantom{\int}}^{\nu_l}\!\!\!\!\int}
\newcommand{\bfnuint}{\!\!\!\!\!{\phantom{\int}}^{\bfnu}\!\!\!\!\int}
\newcommand{\smallbsnuint}{\!\!\!{\phantom{\int}}^{\bfnu}\!\!\!\int}
\newcommand{\Fsigma}{{}^\FF\!\sigma}
\newcommand{\Diffub}{\Diff_{\text{\rm ub}}}
\newcommand{\Cinftyub}{C^\infty_{\text{\rm ub}}}
\newcommand{\Cinftyc}{C^\infty_{\text{\rm c}}}
\newcommand{\Cinftycv}{C^\infty_{\text{\rm cv}}}
\newcommand{\bchi}{{}^{\text{\rm b}}\!\chi}

\newcommand{\Ldis}{L_{\text{\rm dis}}}

\newcommand{\LdisK}{L_{\text{\rm dis},K}}
\newcommand{\LdisJ}{L_{\text{\rm dis},J}}

\newcommand{\LdisdK}{L_{\text{\rm dis},K'}}
\newcommand{\LdisdJ}{L_{\text{\rm dis},J'}}
\newcommand{\LdisKdK}{L_{\text{\rm dis},K,K'}}
\newcommand{\LdisJdJ}{L_{\text{\rm dis},J,J'}}
\newcommand{\LdisIdI}{L_{\text{\rm dis},I,I'}}
\newcommand{\olfX}{\overline{\fX}}
\newcommand{\fXcom}{\fX_{\text{\rm com}}}
\newcommand{\fXub}{\fX_{\text{\rm ub}}}

\newcommand{\fXb}{\fX_{\text{\rm b}}}
\newcommand{\olfXcom}{\overline{\fX}_{\text{\rm com}}}

\newcommand{\betaNo}{\beta_{\text{\rm No}}}

\numberwithin{section}{chapter}
\numberwithin{equation}{section}

\makeindex

\begin{document}

\frontmatter

\title{A trace formula for foliated flows}

\author[J.A. \'Alvarez L\'opez]{Jes\'us A. \'Alvarez L\'opez\,\orcidlink{0000-0001-6056-2847}}
\address{Department of Mathematics and CITMAga\\
         University of Santiago de Compostela\\
         15782 Santiago de Compostela\\ Spain}
\email{jesus.alvarez@usc.es}
\thanks{The authors are partially supported by the grants MTM2017-89686-P and PID2020-114474GB-I00 (AEI/FEDER, UE) and ED431C 2019/10 (Xunta de Galicia, FEDER)}

\author[Y.A. Kordyukov]{Yuri A. Kordyukov\,\orcidlink{0000-0003-2957-2873}}
\address{Institute of Mathematics\\ 
Ufa Federal Research Center\\
Russian Academy of Sciences\\
112 Chernyshevsky street\\ 450008 Ufa\\ Russia}
\email{yurikor@matem.anrb.ru}

\author[E. Leichtnam]{Eric Leichtnam\,\orcidlink{0000-0002-5058-5508}}
\address{Institut de Math\'ematiques de Jussieu-PRG\\ CNRS\\ Batiment Sophie Germain (bureau 740)\\ Case~7012\\ 75205 Paris Cedex 13, France}
\email{eric.leichtnam@imj-prg.fr}

\date{\today}




\maketitle

\begin{abstract}
Let $\mathcal{F}$ be a transversely oriented foliation of codimension one on a closed manifold $M$, and let $\phi=\{\phi^t\}$ be a foliated flow on $(M,\mathcal{F})$. Assume the closed orbits of $\phi$ are simple and its preserved leaves are transversely simple. In this case, there are finitely many preserved leaves, which are compact. Let $M^0$ denote their union, $M^1=M\setminus M^0$ and $\mathcal{F}^1=\mathcal{F}|_{M^1}$. We consider two topological vector spaces, $I(\mathcal{F})$ and $I'(\mathcal{F})$, consisting of the leafwise currents on $M$ that are conormal and dual-conormal to $M^0$, respectively. They become topological complexes with the differential operator $d_{\mathcal{F}}$ induced by the de~Rham derivative on the leaves, and they have an $\mathbb{R}$-action $\phi^*=\{\phi^{t\,*}\}$ induced by $\phi$. Let $\bar H^\bullet I(\mathcal{F})$ and $\bar H^\bullet I'(\mathcal{F})$ denote the corresponding leafwise reduced cohomologies, with the induced $\mathbb{R}$-action $\phi^*=\{\phi^{t\,*}\}$. $\bar H^\bullet I(\mathcal{F})$ and $\bar H^\bullet I'(\mathcal{F})$ are shown to be the central terms of short exact sequences in the category of continuous linear maps between locally convex spaces, where the other terms are described using Witten's perturbations of the de~Rham complex on $M^0$ and leafwise Witten's perturbations for $\mathcal{F}^1$. This is used to define some kind of Lefschetz distribution $L_{\text{\rm dis}}(\phi)$ of the actions $\phi^*$ on both $\bar H^\bullet I(\mathcal{F})$ and $\bar H^\bullet I'(\mathcal{F})$, whose value is a distribution on $\mathbb{R}$. Its definition involves several renormalization procedures, the main one is the b-trace of some smoothing b-pseudodifferential operator on the compact manifold with boundary obtained by cutting $M$ along $M^0$. We also prove a trace formula describing $L_{\text{\rm dis}}(\phi)$ in terms of infinitesimal data from the closed orbits and preserved leaves. This solves a conjecture of C.~Deninger involving two leafwise reduced cohomologies instead of a single one. This memoir is the conclusion of a program started about ten years ago by the three authors.
\end{abstract}

\tableofcontents


\mainmatter

\chapter{Introduction}\label{ch: intro}

\section{Deninger's program}\label{s: Deninger}

Let $(M,\FF)$ be a smooth foliated manifold. The leafwise cohomology, $H^\bullet(\FF)$, is defined with the complex of differential forms on the leaves that are smooth on $M$, $C^\infty(M;\Lambda\FF)$ ($\Lambda\FF=\bigwedge T^*\FF\otimes\C$), equipped with de~Rham differential operator along the leaves, $d_\FF$. This differential complex is not elliptic, it is only leafwise elliptic. Therefore $H^\bullet(\FF)$ may be of infinite dimension and non-Hausdorff with the topology induced by the $C^\infty$ topology. Thus it makes sense to consider the reduced leafwise cohomology, $\bar H^\bullet(\FF)=H^\bullet(\FF)/\overline0$. (The reduced cohomology is defined and denoted in a similar way for any complex with a compatible topology, called a topological complex.)

A flow $\phi=\{\phi^t\}$ on $M$ is said to be foliated if it maps leaves to leaves; equivalently, its infinitesimal generator $Z$ is an infinitesimal transformation of $(M,\FF)$, or the induced section $\overline Z$ of the normal bundle $N\FF=TM/T\FF$ is parallel with respect to the Bott partial connection. In this case, there is an induced $\R$-action $\phi^*=\{\phi^{t\,*}\}$ on $(C^\infty(M;\Lambda\FF),d_\FF)$, which induces an $\R$-action $\phi^*=\{\phi^{t\,*}\}$ on $\bar H^\bullet(\FF)$. Moreover, $\phi$ induces a local flow $\bar\phi$ on local transversals of $\FF$. Some leaves may be preserved by $\phi$, which correspond to the fixed points of $\bar\phi$. If these fixed points of $\bar\phi$ are simple, then the leaves preserved by $\phi$ are called transversely simple (\Cref{ss: simple fol flows}).

Assume $M$ is closed, $\codim\FF=1$, the closed orbits are simple, the preserved leaves are transversely simple, and $\phi$ is transverse to the non-preserved leaves. With these conditions, C.~Deninger has conjectured that the supertrace of $\phi^*$ on $\bar H^\bullet(\FF)$ makes sense as a distribution $\Ldis(\phi)$ on $\R$ (its Lefschetz distribution), and it has an expression involving infinitesimal data from the preserved leaves and closed orbits (a dynamical Lefschetz trace formula).

This problem is a part of a program proposed by Deninger, whose goal is the study of arithmetic zeta functions by finding an interpretation of the explicit formulae as a dynamical Lefschetz trace formula for some $(M,\FF,\phi)$ of this type \cite{Deninger1998,Deninger2001,Deninger2002,Deninger-Arith_geom_anal_fol_sps,Deninger2008}. The precise expression of the trace formula was previously suggested by Guillemin \cite{Guillemin1977}. Further developments of these ideas were made in \cite{DeningerSinghof2002,Mumken2006,Kopei2006,Leichtnam2008,Kopei2011,Leichtnam2014,KordyukovPavlenko2015,
Kim-fiber_bdls,Deninger-Dyn-Syst-for-arith-schemes,Deninger-Primes-knots-periodic-orbits}.

It became clear that more generality is needed to draw arithmetic consequences (perhaps foliated flows on possibly singular foliated spaces of arithmetic nature). But, even for $(M,\FF,\phi)$ as above, this problem is difficult and interesting; for instance, $\bar H^\bullet(\FF)$ is not appropriate in general \cite{DeningerSinghof2001}. Besides its own interest, a solution might provide techniques to deal with more general settings. Moreover, we believe that 
the techniques developed in this paper will be useful in arithmetic once the appropriate framework 
allowing to interpret the Weil’s explicit formulae for arithmetic zeta functions as Lefschetz trace formulae will have been discovered.

\section{Case with no preserved leaves}\label{s: no preserved leaves}

The first two authors proved such a trace formula when $\phi$ has no preserved leaves \cite{AlvKordy2002}, and extended it for transverse actions of Lie groups \cite{AlvKordy2008a}. In this case, $\FF$ is Riemannian; i.e., it is locally described by Riemannian submersions for some Riemannian metric $g$ on $M$ (a bundle-like metric). Using $g$, we get the leafwise coderivative $\delta_\FF$ and the leafwise Laplacian $\Delta_\FF$. Then the leafwise heat operator defines a continuous map \cite{AlvKordy2001}
\begin{equation}\label{heat op - FF Riem - M closed}
C^\infty(M;\Lambda\FF)\times[0,\infty]\to C^\infty(M;\Lambda\FF)\;,\quad(\alpha,u)\mapsto e^{-u\Delta_\FF}\alpha\;.
\end{equation}
It follows that there is a leafwise Hodge decomposition
\begin{equation}\label{leafwise Hodge dec - FF Riem - M closed}
C^\infty(M;\Lambda\FF)=\ker\Delta_\FF\oplus\overline{\im d_\FF}\oplus\overline{\im\delta_\FF}\;,
\end{equation}
and therefore the orthogonal projection $\Pi_\FF=e^{-\infty\Delta_\FF}$ to $\ker\Delta_\FF$ induces a leafwise Hodge isomorphism
\begin{equation}\label{leafwise Hodge iso - FF Riem - M closed}
\bar H^\bullet(\FF)\cong\ker\Delta_\FF\;.
\end{equation}
This is surprising because $\Delta_\FF$ is only leafwise elliptic; somehow, the transverse rigidity of Riemannian foliations makes up for the lack of transverse ellipticity. These properties may fail for non-Riemannian foliations \cite{DeningerSinghof2001}. 

Furthermore, for all $f\in\Cinftyc(\R)$ and $0<u\le\infty$, the operator
\begin{equation}\label{P_u f - FF Riem - M closed}
P_{u,f}=\int_\R\phi^{t*}e^{-u\Delta_\FF}f(t)\,dt
\end{equation}
is smoothing, and therefore of trace class, its supertrace $\Str P_{u,f}$ depends continuously on $f$ and is independent of $u$, and the limit of $\Str P_{u,f}$ as $u\downarrow0$ gives the expected contribution of the closed orbits \cite{AlvKordy2002,AlvKordy2008a}. By~\eqref{leafwise Hodge iso - FF Riem - M closed} and~\eqref{P_u f - FF Riem - M closed}, the mapping $f\mapsto\Str P_{\infty,f}$ can be considered as the Lefschetz distribution $\Ldis(\phi)$, solving the problem in this case.

\section{General case}\label{s: general}

This publication is a continuation of the works \cite{AlvKordy2001,AlvKordy2002,AlvKordy2008a}, recalled in \Cref{s: no preserved leaves}. Our main goal is to propose an extension of the trace formula to the case where there are (compact) leaves preserved by $\phi$, which are very relevant in Deninger's program. Examples of foliations with such foliated flows can be easily constructed by using foliation surgeries.

\subsection{Ingredients of the trace formula}\label{ss: general - ingredients}

Assume $\FF$ is transversely oriented for the sake of simplicity. Thus, by Frobenius theorem, $\FF$ is defined by a 1-form $\omega$ with $d\omega=\eta\wedge\omega$ ($T\FF=\ker\omega$). Except in trivial cases, the existence of leaves preserved by $\phi$ prevents $\FF$ from being Riemannian (it is impossible to choose $\eta=0$), yet $\FF$ has a precise description \cite{AlvKordyLeichtnam2022}. For instance, there is a finite number of preserved leaves, which are compact. Let $M^0$ denote the union of the leaves preserved by $\phi$, $M^1=M\setminus M^0$ and $\FF^1=\FF|_{M^1}$.

All versions of leafwise reduced cohomologies we will consider have an action $\phi^*=\{\phi^{t\,*}\}$ induced by $\phi$, which is invariant by leafwise homotopy equivalences. Thus, up to leafwise homotopies, we can assume $\phi^t=\id$ on $M^0$. Then, for every leaf $L\subset M^0$, there is some $\varkappa_L\in\R^\times$ such that, on the normal bundle $NL=T_LM/TL$, the normal tangent map is $\phi^t_*=e^{\varkappa_Lt}$. The numbers $\varkappa_L$ will be ingredients of the trace formula. Moreover $\FF^1$ becomes a transversely complete $\R$-Lie foliation with the restriction of $\overline Z$. So $\FF$ is a particular case of foliation almost without holonomy \cite{Hector1972c,Hector1978}.

Take a Riemannian metric $g$ on $M$ so that $\omega$ is the transverse volume form. The corresponding leafwise metric is denoted by $g_\FF$. We can suppose $\eta$ vanishes on $T\FF^\perp$, and therefore it can be considered as a leafwise form, and we have $d_\FF\eta=0$. Furthermore, on some tubular neighborhood $T\equiv(-\epsilon,\epsilon)\times M^0$ ($\epsilon>0$) of $M^0$ in $M$, we can suppose $\eta$ and $g_\FF$ are lifts of their restrictions to $M^0$, and the fibers of the projection $\varpi:T\to M^0$ are orthogonal to the leaves and agree with the orbits of $\phi$. Thus there are no closed orbits of $\phi$ in $T$. The projection $\rho:T\to(-\epsilon,\epsilon)$ is a defining function of $M^0$ on $T$ ($d\rho\ne0$ on $M^0=\rho^{-1}(0)$), which can be assumed to satisfy $d_\FF\rho=\rho\eta$ on $T$ and $\phi^{t*}\rho=e^{\varkappa_L t}\rho$ around every leaf $L\subset M^0$. We can choose any $\eta|_{M^0}$ in some fixed real cohomology class $\xi\in H^1(M^0)$ determined by $\FF$, and there is no restriction on the choice of $g|_{M^0}$.

For every closed orbit $c$ of $\phi$, let $\ell(c)$ denote its smallest positive period. The condition on $c$ to be simple means that $\id-\phi^{k\ell(c)}_*:T_p\FF\to T_p\FF$ is an isomorphism for any $p\in c$ and $k\in\Z^\times$, whose determinant is independent of $p$, and its sign denoted by $\epsilon_c(k)$. The integers $\ell(c)$ and $\epsilon_c(k)$ will be also ingredients of the trace formula.

Let $g^1$ be the bundle-like metric of $\FF^1$ such that it defines the same orthogonal complement $(T\FF^1)^\perp$ as $g$, its restriction to $T\FF^1$ is $g_\FF$, and $\overline Z|_{M^1}$ is of norm one with the induced Euclidean structure on $N\FF^1$. Then $\FF^1$ has bounded geometry with $g^1$ in the sense of \cite{Sanguiao2008,AlvKordyLeichtnam2014}. Let $\omega^1$ denote the transverse volume form of $\FF^1$ defined by $g^1$ and the transverse orientation given by $Z|_{M^1}$. The transverse density $|\omega^1|$ can be considered as an invariant transverse measure of $\FF^1$.

By cutting $M$ along $M^0$, we get a compact manifold with boundary $\bfM$ with a foliation $\bfFF$ tangent to $\partial\bfM$. This allows us to apply tools from b-calculus \cite{Melrose1993,Melrose1996}. For instance, $g^1$ and $\omega^1$ are restrictions to $M^1\equiv\mathring\bfM$ of a b-metric $\bfg_{\text{\rm b}}$ and a b-form $\bfomega_{\text{\rm b}}$ on $\bfM$, and therefore $|\omega^1|$ is the restriction of the b-density $|\bfomega_{\text{\rm b}}|$. 

We can suppose there is some boundary-defining function $\bfrho$ on $\bfM$ ($\bfrho\ge0$ and $d\bfrho\ne0$ on $\partial\bfM=\bfrho^{-1}(0)$) such that the lift $\bfeta$ of $\eta$ to $\bfM$ satisfies $d_{\bfFF}\bfrho=\bfrho\bfeta$ on $\mathring\bfM$, and $\bfrho$ is the lift of $|\rho|$ on a collar neighborhood $\bfT\equiv[0,\epsilon)\times\partial\bfM$ of $\partial\bfM$. The lift of $\phi$ to $\bfM$ is a foliated flow $\bfphi=\{\bfphi^t\}$ of $(\bfM,\bfFF)$.

We will use the b-integral $\smallbsnuint_{\bfM}$, depending on the choice of a trivialization $\bfnu$ of $N\partial\bfM$ satisfying $d\bfrho(\bfnu)=1$. We can apply $\smallbsnuint_{\bfM}$ to b-densities on $\bfM$; the usual integral of their restrictions to $\mathring\bfM$ may not be defined. Assume $\dim\FF$ is even, which is the relevant case in Deninger's program. Then the product of the leafwise Euler density $e(\bfFF)$ and $|\bfomega_b|$ is the restriction of a b-density on $\bfM$, obtaining a b-calculus version of the Connes' $|\bfomega_{\text{\rm b}}|$-Euler characteristic of $\bfFF$,
\[
\bchi_{|\bfomega_{\text{\rm b}}|}(\bfFF)
=\bfnuint_{\bfM}e(\bfFF)\,|\bfomega_{\text{\rm b}}|\;,
\]
which will be called the \emph{b-Connes-Euler characteristic} \index{b-Connes-Euler characteristic} of $\bfFF$ defined by $|\bfomega_{\text{\rm b}}|$ (or of $\FF^1$ defined by $|\omega^1|$). This number will be another ingredient of the trace formula, also denoted by $\bchi_{|\omega^1|}(\FF^1)$. The b-integral can be used to define the b-trace $\bTr$ of smoothing b-pseudodifferential operators on $\bfM$; these operators may not be of trace class. The corresponding concept of b-supertrace will be used, denoted by $\bStr$.

With this generality,~\eqref{heat op - FF Riem - M closed}--\eqref{leafwise Hodge iso - FF Riem - M closed} are not true for $C^\infty(M;\Lambda\FF)$. Using the space $C^{-\infty}(M;\Lambda\FF)$ of leafwise currents does not work either. Instead, we will use the topological complex of leafwise currents that are conormal and dual-conormal at $M^0$ \cite{KohnNirenberg1965,Hormander1971}, \cite[Section~18.2]{Hormander1985-III}, \cite[Chapters~4 and~6]{Melrose1996}, \cite{AlvKordyLeichtnam-conormal}.

\subsection{Conormal and dual-conormal leafwise currents}
\label{ss: general - conormal and dual-conormal leafwise currents}

We first recall the definitions and some properties of conormal and dual-conormal distributions at $M^0$. Let $\Diff(M,M^0)$ be the filtered algebra of differential operators on $C^\infty(M)$ generated by $C^\infty(M)$ and the vector fields on $M$ tangent to $M^0$, and let $H^s(M)$ be the Sobolev space of order $s\in\R$. A distribution $u\in C^{-\infty}(M)$ is said to be conormal at $M^0$ of Sobolev order $s$ if $\Diff(M,M^0)u\subset H^s(M)$. These distributions form a Fr\'echet space $I^{(s)}=I^{(s)}(M,M^0)$ endowed with the projective topology given by the maps $P:I^{(s)}\to H^s(M)$ ($P\in\Diff(M,M^0)$). The spaces $I^{(s)}$ form an inductive spectrum defining an LF-space $I=I(M,M^0)=\bigcup_sI^{(s)}$, with continuous inclusions $C^\infty(M)\subset I\subset C^{-\infty}(M)$. (All inclusions considered here are continuous.) See \cite{AlvKordyLeichtnam-conormal} for the properties of $I$ and of other related spaces.

All spaces of distributions considered here, and their properties, have straightforward extensions for distributional sections of vector bundles.  In particular, for the density bundle $\Omega=\Omega M$, we get the strong dual $I'(M,L)=I(M,L;\Omega)'$, simply denoted by $I'$. The elements of $I'$ are called dual-conormal distributions; in fact, $C^\infty(M)\subset I'\subset C^{-\infty}(M)$ with $I\cap I'=C^\infty(M)$

Let also $K=K(M,M^0)\subset I$ be the closed subspace consisting of elements supported in $M^0$. On the other hand, via the lift to $\bfM$, we get another space, $J=J(M,M^0)$, which is isomorphic to the space of extendable distributions on $\bfM$ conormal at the boundary \cite[Chapter~4]{Melrose1996}. There are canonical injections $C^\infty(M)\subset J\subset C^\infty(M^1)$. Let $K'=K'(M,L)$ and $J'=J'(M,L)$ be defined like $I'$. We get $J'\subset C^{-\infty}(M)$. Moreover there are short exact sequences in the category of continuous linear maps between locally convex spaces \cite[Chapter~2]{Wengenroth2003},
\begin{gather}
0\to K \xrightarrow{\iota} I \xrightarrow{R} J\to0\;,\label{intro - conormal exact sequence}\\
0\leftarrow K' \xleftarrow{R'} I' \xleftarrow{\iota'} J'\leftarrow0\;,
\label{intro - dual-conormal exact sequence}
\end{gather}
where $\iota$ is the inclusion map and $R$ is defined by restriction to $M^1$, and~\eqref{intro - dual-conormal exact sequence} is the transpose of the version of~\eqref{intro - conormal exact sequence} with $\Omega$ ($R'=\iota^\trans$ and $\iota'=R^\trans$). These sequences are relevant because $K$, $J$, $K'$ and $J'$ have better descriptions than $I$ and $I'$. So~\eqref{intro - conormal exact sequence} and~\eqref{intro - dual-conormal exact sequence} will play an important role.

Using the vector bundle $\Lambda\FF$, we get the spaces of conormal and dual-conormal leafwise currents at $M^0$, $I(\FF)=I(M,M^0;\Lambda\FF)$ and $I'(\FF)=I'(M,M^0;\Lambda\FF)$, 
as well as the spaces $K(\FF)$, $J(\FF)$, $K'(\FF)$ and $J'(\FF)$, with a similar notation. All of them are topological complexes with $d_\FF$, and have $\R$-actions $\phi^*=\{\phi^{t\,*}\}$ induced by $\phi$, compatible with $d_\FF$. They give rise to the conormal and dual-conormal leafwise reduced cohomologies, $\bar H^\bullet I(\FF)$ and $\bar H^\bullet I'(\FF)$, as well as the reduced cohomologies $\bar H^\bullet K(\FF)$, $\bar H^\bullet J(\FF)$, $\bar H^\bullet K'(\FF)$ and $\bar H^\bullet J'(\FF)$. All of them with induced $\R$-actions $\phi^*=\{\phi^{t\,*}\}$. The bars are omitted from the notation if the cohomologies are not reduced. There are versions of~\eqref{intro - conormal exact sequence} and~\eqref{intro - dual-conormal exact sequence} for the spaces $K(\FF)$, $I(\FF)$, $J(\FF)$, $K'(\FF)$, $I'(\FF)$ and $J'(\FF)$, where $\iota$, $R$, $\iota'$ and $R'$ are cochain maps. The induced maps in cohomology (resp., reduced cohomology) are denoted by $\iota_*$, $R_*$, $\iota'_*$ and $R'_*$ (resp., $\bar\iota_*$, $\bar R_*$, $\bar\iota'_*$ and $\bar R'_*$).

\subsection{Witten's perturbed complexes}\label{ss: intro - Witten}

To describe the reduced cohomologies of \Cref{ss: general - conormal and dual-conormal leafwise currents} with the $\R$-actions $\phi^*$, we will use the Witten's perturbation $d_\mu=d+\mu{\eta\wedge}$ on $C^{\pm\infty}(L;\Lambda)$ ($\Lambda=\Lambda L=\bigwedge T^*L\otimes\C)$, for $\mu\in\R$ and every leaf $L\subset M^0$. Its cohomology is denoted by $H_\mu^\bullet(L)$. The corresponding perturbed codifferential and Laplace operators are denoted by $\delta_\mu$ and $\Delta_\mu$.

\subsection{Leafwise Witten's perturbed complexes}\label{ss: intro - leafwise Witten}

Recall that $d_{\bfFF}\bfrho=\bfrho\bfeta$ on $\mathring\bfM$ and $\partial\bfM=\bfrho^{-1}(0)$. We will also use the leafwise Witten's perturbation
\[
d_{\bfFF,\mu}=d_{\bfFF}+\mu{\bfeta\wedge}=\bfrho^{-\mu}d_{\bfFF}\bfrho^\mu
\]
on the Sobolev spaces $H^{\pm\infty}(\mathring\bfM;\Lambda\bfFF)\equiv H^{\pm\infty}(M^1;\Lambda\FF^1)$ defined with $\bfg_{\text{\rm b}}\equiv g^1$. Their reduced cohomologies are denoted by $\bar H^\bullet_\mu H^{\pm\infty}(\mathring\bfFF)$ ($\mathring\bfFF=\bfFF|_{\mathring\bfM}\equiv\FF^1$). They satisfy obvious versions of~\eqref{heat op - FF Riem - M closed}--\eqref{leafwise Hodge iso - FF Riem - M closed}. We have the isomorphisms
\begin{equation}\label{intro - multiplication iso}
\bfrho^\mu:(H^{\pm\infty}(\mathring\bfM;\Lambda\bfFF),d_{\bfFF,\mu})
\xrightarrow{\cong}(\bfrho^\mu H^{\pm\infty}(\mathring\bfM;\Lambda\bfFF),d_{\bfFF})\;.
\end{equation}
Let also $\bfphi^{t\,*}_\mu=\bfrho^{-\mu}\bfphi^{t\,*}\bfrho^\mu$ on $H^{\pm\infty}(\mathring\bfM;\Lambda\bfFF)$, which induces an endomorphism $\bfphi^{t\,*}_\mu$ of $\bar H^\bullet_\mu H^{\pm\infty}(\mathring\bfFF)$. For $\mu<\mu'$, the inclusions
\[
\bfrho^{\mu'} H^{\pm\infty}(\mathring\bfM;\Lambda\bfFF)
\subset\bfrho^\mu H^{\pm\infty}(\mathring\bfM;\Lambda\bfFF)
\]
correspond via~\eqref{intro - multiplication iso} to the maps
\begin{equation}\label{bfrho^mu'-mu}
\bfrho^{\mu'-\mu}:H^{\pm\infty}(\mathring\bfM;\Lambda\bfFF)
\to H^{\pm\infty}(\mathring\bfM;\Lambda\bfFF)\;.
\end{equation}
The corresponding perturbed leafwise codifferential and Laplace operators are denoted by $\delta_{\bfFF,\mu}$ and $\Delta_{\bfFF,\mu}$. Finally, for $f\in\Cinftyc(\R)$, $\mu\in\R$ and $0<u\le\infty$, we will use the operator
\[
\bfP_{\mu,u,f}=\int_\R\bfphi^{t*}_\mu\, e^{-u\Delta_{\bfFF},\mu}\,f(t)\,dt
\]
on $H^{\pm\infty}(\mathring\bfM;\Lambda\bfFF)$, which is a version of~\eqref{P_u f - FF Riem - M closed}.

\subsection{Main results leading to the trace formula}\label{ss: main results}

Concerning the above reduced cohomologies, the following are our main achievements.

\begin{thm}\label{t: intro - H^bullet K(FF)} 
We have\footnote{With some abuse of notation, we write $\bigoplus_mA=\bigoplus_mA_m$ and $\prod_mA=\prod_mA_m$ if $A_m=A$ for all $m$.}
\begin{gather*}
K(\FF)\equiv\bigoplus_{L,k}C^\infty(L;\Lambda)\;,\quad
d_\FF\equiv\bigoplus_{L,k}d_{-k-1}\;,\quad
\phi^{t\,*}\equiv\bigoplus_{L,k}e^{-(k+1)\varkappa_Lt}\;,\\
H^\bullet K(\FF)\equiv\bar H^\bullet K(\FF)\equiv\bigoplus_{L,k}H_{-k-1}^\bullet(L)\;,\quad
\phi^{t\,*}\equiv\bigoplus_{L,k}e^{-(k+1)\varkappa_Lt}\;,
\end{gather*}
where $L$ runs over the set of leaves contained in $M^0$ and $k$ runs over $\N_0$.
\end{thm}

The first identity of \Cref{t: intro - H^bullet K(FF)} follows by considering the partial derivatives $\partial_{\rho}^k$ ($k\in\N_0$) of leafwise currents of $(M,\FF)$ that are of Dirac type at the leaves $L\subset M^0$. It is a consequence of the properties of $\rho$, $\eta$ and $\phi^*$ on $T$.

Now consider $\bfrho$, $\bfeta$ and $\bfphi$ on $(\bfM,\bfFF)$.

\begin{thm}\label{t: intro - bar H^bullet J(FF)} 
Using~\eqref{intro - multiplication iso} with $H^\infty(\mathring\bfM;\Lambda\bfFF)$, we get
\begin{gather*}
J(\FF)=\bigcup_\mu\bfrho^\mu H^\infty(\mathring\bfM;\Lambda\bfFF)
\equiv\varinjlim H^\infty(\mathring\bfM;\Lambda\bfFF)\;,\\
d_{\bfFF}\equiv\varinjlim d_{\bfFF,\mu}\;,\quad
\phi^{t\,*}\equiv\varinjlim\bfphi^{t\,*}_\mu\;,
\end{gather*}
where the inductive limits are defined with the maps~\eqref{bfrho^mu'-mu} as $\mu\downarrow-\infty$. Moreover, there are linear identities, 
\[
\bar H^\bullet J(\FF)\equiv\varinjlim\bar H_\mu^\bullet H^\infty(\mathring\bfFF)\;,\quad
\phi^{t\,*}\equiv\varinjlim\bfphi^{t\,*}_\mu\;.
\]
\end{thm}

\begin{thm}\label{t: intro - reduced conormal cohomology exact sequence}
We have a short exact sequence 
\[
0\to H^\bullet K(\FF) \xrightarrow{\bar\iota_*} \bar H^\bullet I(\FF) \xrightarrow{\bar R_*} \bar H^\bullet J(\FF)\to0\;.
\]
\end{thm}

\begin{thm}\label{t: intro - H^bullet K'(FF)} 
Using $L$ and $k$ like in \Cref{t: intro - H^bullet K(FF)}, we have
\begin{gather*}
K'(\FF)\equiv\prod_{L,k}C^{-\infty}(L;\Lambda)\;,\quad
d_\FF\equiv\prod_{L,k}d_k\;,\quad
\phi^{t\,*}\equiv\prod_{L,k}e^{k\varkappa_Lt}\;,\\
H^\bullet K'(\FF)\equiv\bar H^\bullet K'(\FF)\equiv\prod_{L,k}H_k^\bullet(L)\;,\quad
\phi^{t\,*}\equiv\prod_k e^{k\varkappa_Lt}\;.
\end{gather*}
\end{thm}

The identity of \Cref{t: intro - H^bullet K'(FF)} is a consequence of the version of \Cref{t: intro - H^bullet K(FF)} for $K(\FF;\Omega M)$. The shift in the role played by $k$ is due to the introduction of $\Omega M$.

\begin{thm}\label{t: intro - bar H^bullet J'(FF)} 
Using~\eqref{intro - multiplication iso} with $H^{-\infty}(\mathring\bfM;\Lambda\bfFF)$, we get
\begin{gather*}
J'(\FF)=\bigcap_\mu\bfrho^\mu H^{-\infty}(\mathring\bfM;\Lambda\bfFF)
\equiv\varprojlim H^{-\infty}(\mathring\bfM;\Lambda\bfFF)\;,\\
d_{\bfFF}\equiv\varprojlim d_{\bfFF,\mu}\;,\quad
\phi^{t\,*}\equiv\varprojlim\bfphi^{t\,*}_\mu\;.
\end{gather*}
where the projective limits are defined with the maps~\eqref{bfrho^mu'-mu} as $\mu\uparrow+\infty$. Moreover, there are linear identities, 
\[
\bar H^\bullet J'(\FF)\equiv\varprojlim\bar H_\mu^\bullet H^{-\infty}(\mathring\bfFF)\;,\quad
\phi^{t\,*}\equiv\varprojlim\bfphi^{t\,*}_\mu\;.
\]
\end{thm}

There is no essential difference between $J(\FF)$ and $J(\FF;\Omega M)$ because $\mathring\bfFF$ has the invariant transverse density $|\bfomega_{\text{\rm b}}|$. Thus \Cref{t: intro - bar H^bullet J'(FF)} follows from \Cref{t: intro - bar H^bullet J(FF)}. 

\begin{thm}\label{t: intro - reduced dual-conormal cohomology exact sequence} 
We have a short exact sequence
\[
0\leftarrow H^\bullet K'(\FF)\xleftarrow{\bar R'_*}\bar H^\bullet I'(\FF)\xleftarrow{ \bar\iota'_*}\bar H^\bullet J'(\FF)\leftarrow0\;.
\]
\end{thm}

Recall the definition of $\bfP_{\mu,u,f}$ given in \Cref{ss: intro - Witten}.

\begin{thm}\label{t: intro - bfP_u f} 
$\bfP_{\mu,u,f}$ is a smoothing b-pseudodifferential operator, and the map $f\mapsto\bStr\bfP_{\mu,u,f}$ defines a distribution on $\R$.
\end{thm}  

Now we will use the integers $\ell(c)$ and $\epsilon_c(k)$ associated to every closed orbit $c$, and the b-Connes-Euler characteristic $\bchi_{|\bfomega_{\text{\rm b}}|}(\bfFF)=\bchi_{|\omega^1|}(\FF^1)$.

\begin{thm}\label{t: intro - lim_u->0 bTrs(bfP_mu u f} 
We have
\[
\lim_{u\downarrow0}\bStr\bfP_{\mu,u,f}=\bchi_{|\omega^1|}(\FF^1)\,f(0)+\sum_c\ell(c)\sum_{k\in\Z^\times}\epsilon_c(k)\,f(k\ell(c))\;,
\]
where $c$ runs in the set of closed orbits of $\phi$.
\end{thm}

Recall that the definition of $\eta$ was given in \Cref{ss: general - ingredients}.

\begin{thm}\label{t: intro - lim_mu->pm infty ...} If $\dim\FF$ is even, then we can choose $\eta$ and $g$ on $M^0$ so that 
\[
f\mapsto\lim_{u_1\uparrow+\infty,\ u_0\downarrow0}\big(\bStr\bfP_{\mu,u_1,f}-\bStr\bfP_{\mu,u_0,f}\big)
\]
defines a tempered distribution $Z_\mu$ on $\R$, and $Z_\mu\to0$ as $\mu\to\pm\infty$.
\end{thm} 

In \Cref{t: intro - lim_mu->pm infty ...}, for more general choices of $\eta$ and $g$ on $M^0$, the limits of $Z_\mu$ as $\mu\to\pm\infty$ are multiples of the Dirac mass $\delta_0$. These limits may not be zero because the b-trace does not vanish on commutators (it is not a trace). This additional contribution of the b-trace shows up like the eta-invariant of manifolds with boundary \cite{Melrose1993}. When $\dim\FF$ is even, we can prescribe any limit of $Z_\mu$ as $\mu\to\pm\infty$ with appropriate choices of $\eta$ and $g$ on $M^0$ \cite{AlvKordyLeichtnam-ziomf} (see \Cref{t: Z}); in particular, we can prescribe the zero limit. This makes $\bStr\bfP_{\mu,u,f}$ behave like a supertrace as $\mu\to\pm\infty$.

\subsection{The Lefschetz distribution}\label{ss: L_dis(phi)}

It seems there is no reasonable definition of $\Ldis(\phi)$ with a single leafwise reduced cohomology. However, $\bar H^\bullet I(\FF)$ and $\bar H^\bullet I'(\FF)$ together will do the job. Though this may look strange, we hope this idea will be valid in further developments of Deninger's program.

To begin with, by \Cref{t: intro - reduced conormal cohomology exact sequence} and \Cref{t: intro - reduced dual-conormal cohomology exact sequence}, it is enough to consider the actions $\phi^*$ on $H^\bullet K(\FF)$, $\bar H^\bullet J(\FF)$, $H^\bullet K'(\FF)$ and $\bar H^\bullet J'(\FF)$.  

Let us try to define Lefschetz distributions $\LdisK(\phi)$ and $\LdisdK(\phi)$ of $\phi$ on $\bar H^\bullet K(\FF)$ and $\bar H^\bullet K'(\FF)$. By \Cref{t: intro - H^bullet K(FF)} and \Cref{t: intro - H^bullet K'(FF)}, and since all twisted cohomologies $H_\mu^\bullet(L)$ have the same Euler characteristic $\chi(L)$, it makes some sense to define, on $\R^\times$,
\begin{align*}
\LdisK(\phi)&=\sum_{\varkappa_Lt>0}\chi(L)\sum_{k=0}^\infty e^{-(k+1)\varkappa_Lt}
=\sum_{\varkappa_Lt>0}\frac{\chi(L)}{e^{\varkappa_Lt}-1}\;,\\
\LdisdK(\phi)&=\sum_{\varkappa_Lt<0}\chi(L)\sum_{k=0}^\infty e^{k\varkappa_Lt}
=\sum_{\varkappa_Lt<0}\frac{\chi(L)}{1-e^{\varkappa_Lt}}\;.
\end{align*}
In each of these distributions, the conditions on the leaves $L\subset M^0$ guarantee that their contribution to the trace is defined; the other leaves in $M^0$ are omitted as a way of renormalization. Every $L$ has a contribution to just one of these distributions on $\R^\pm$. Taking into account all contributions from leaves $L\subset M^0$ in $\LdisK(\phi)$ and $\LdisdK(\phi)$, we get a combined Lefschetz distribution on $\R^\times$,
\[
\LdisKdK(\phi)=\sum_L\frac{\chi(L)}{|e^{\varkappa_Lt}-1|}\;.
\]
By changing variables and using L'H\^ospital's rule, it follows that every function $|e^{\varkappa_Lt}-1|^{-1}$ on $\R^\times$ can be extended to a distribution $W_L$ on $\R$ given by \cite{Barner1981}
\[
\langle W_L,f\rangle=\int_0^\infty\bigg(\frac{f(t)+f(-t)}{|e^{\varkappa_Lt}-1|}-\frac{2f(0)}{|\varkappa_L|\,t}\bigg)\,dt\;.
\]
Thus $\LdisKdK(\phi)$ can be extended to $\R$ as the distribution
\[
\LdisKdK(\phi)=\sum_L\chi(L)\,W_L\;.
\]

Next, by \Cref{t: intro - bfP_u f}, like in the case of~\eqref{P_u f - FF Riem - M closed}, we can consider the mapping
\[
f\mapsto\lim_{u\uparrow+\infty}\bStr\bfP_{\mu,u,f}
\] 
as the distributional supertrace of the action $\phi^*_\mu$ on $\bar H^\bullet_\mu H^{\pm\infty}(\mathring\bfFF)$. Since $\bfP_{\mu,u,f}$ is not of trace class, its b-supertrace is used here instead of the supertrace as a way of renormalization. By \Cref{t: intro - bar H^bullet J(FF)} and \Cref{t: intro - bar H^bullet J'(FF)}, it makes sense to define the Lefschetz distributions of $\phi$ on $\bar H^\bullet J(\FF)$ and $\bar H^\bullet J'(\FF)$, denoted by $\LdisJ(\phi)$ and $\LdisdJ(\phi)$, by
\begin{align*}
\langle\LdisJ(\phi),f\rangle&=\lim_{\mu\downarrow-\infty}\lim_{u\uparrow+\infty}\bStr\bfP_{\mu,u,f}\;,\\
\langle\LdisdJ(\phi),f\rangle&=\lim_{\mu\uparrow+\infty}\lim_{u\uparrow+\infty}\bStr\bfP_{\mu,u,f}\;.
\end{align*}

From now on, assume $\dim\FF$ is even (the relevant case in Deninger's program is $\dim\FF=2$). By \Cref{t: intro - lim_u->0 bTrs(bfP_mu u f,t: intro - lim_mu->pm infty ...}, we can choose $\eta$ and $g$ on $M^0$ so that
\[
\LdisJ(\phi)=\LdisdJ(\phi)=\bchi_{|\omega^1|}(\FF^1)\,\delta_0
+\sum_c\ell(c)\sum_{k\in\Z^\times}\epsilon_c(k)\,\delta_{k\ell(c)}\;.
\]
The notation $\LdisJdJ(\phi)$ may be used for this distribution, which is considered as a common feature of the actions $\phi^*$ on $\bar H^\bullet I(\FF)$ and $\bar H^\bullet I'(\FF)$.

Finally, by \Cref{t: intro - reduced conormal cohomology exact sequence,t: intro - reduced dual-conormal cohomology exact sequence}, it makes sense to define the combined Lefschetz distribution
\[
\Ldis(\phi)=\LdisIdI(\phi)=\LdisKdK(\phi)+\LdisJdJ(\phi)\;.
\]
By \Cref{t: intro - lim_u->0 bTrs(bfP_mu u f,t: intro - lim_mu->pm infty ...}, the trace formula conjectured by Deninger is satisfied:

\begin{thm}\label{t: L_dis(phi)}
Using the preserved leaves $L$ and the closed orbits $c$, we have
\[
\Ldis(\phi)=\sum_L\chi(L)\,W_L+\bchi_{|\omega^1|}(\FF^1)\,\delta_0
+\sum_c\ell(c)\sum_{k\in\Z^\times}\epsilon_c(k)\,\delta_{k\ell(c)}\;.
\]
\end{thm}

\section{Short guide}\label{s: guide}

Our arguments involve tools from two different sources: Analysis and Foliations. Concerning Analysis, we mainly use conormal and dual-conormal distributions, analysis on manifolds of bounded geometry and small b-calculus. Concerning Foliations, we mainly use local Reeb's stability, suspension foliations, Riemannian foliations, and differential forms and currents on foliated manifolds.  For the readers' convenience, the needed basic concepts and results from those areas are recalled in \Cref{ch: analytic tools,ch: foln tools}. The specialists on any of them may skip the corresponding chapter, except perhaps the notation. A few short proofs are also recalled in \Cref{ch: analytic tools} because their arguments will be used.

\Cref{ch: fols with simple fol flows} contains a more specific description of foliations with simple foliated flows, explaining all topological and geometric objects that will be used in our analysis. We specially focus on the case of suspension foliations, which describe $\FF$ on a tubular neighborhood $T$ of $M^0$. 

\Cref{ch: conormal} is devoted to the study of the action $\phi^*$ on $\bar H^\bullet I(\FF)$ and $\bar H^\bullet I'(\FF)$, showing \Cref{t: intro - reduced conormal cohomology exact sequence,t: intro - H^bullet K(FF),t: intro - bar H^bullet J(FF),t: intro - reduced dual-conormal cohomology exact sequence,t: intro - H^bullet K'(FF),t: intro - bar H^bullet J'(FF)}.

Finally, \Cref{ch: contribution from M^1} is devoted to the study of $\bStr\bfP_{\mu,u,f}$, showing \Cref{t: intro - bfP_u f,t: intro - lim_u->0 bTrs(bfP_mu u f,t: intro - lim_mu->pm infty ...}.

\chapter{Analytic tools}\label{ch: analytic tools}

\section{Section spaces and operators on manifolds}\label{s: section sps and opers}

The field of coefficients is $\K$, equal to $\R$ or $\C$. We typically consider $\K=\C$, and the few cases where $\K=\R$ will be indicated without changing the notation.

\subsection{Topological vector spaces}\label{ss: TVS}

Let us recall some concepts and fix some conventions concerning topological vector spaces (TVSs); \index{TVS} see \cite{Edwards1965,Horvath1966-I,Kothe1969-I,Schaefer1971,NariciBeckenstein2011,Wengenroth2003} for other concepts we use. We always consider (possibly non-Hausdorff) locally convex spaces (LCSs); \index{LCS} the abbreviation LCHS \index{LCHS} is used in the Hausdorff case. Local convexity is preserved by all operations we use. For instance, we will use the (locally convex) inductive/projective limit of any inductive/projective spectrum (or system) of continuous linear maps between LCSs. If the inductive/projective spectrum is a sequence of continuous inclusions, then the inductive/projective limit is the union/intersection, always endowed with the inductive/projective limit topology. This applies to the locally convex direct sum and the topological product of LCSs. LF-spaces are not assumed to be strict. The (continuous) dual $X'$ of any LCS $X$ is always endowed with the strong topology.

Now fix an inductive spectrum of LCSs of the form $(X_k)=(X_0\subset X_1\subset\cdots)$, and let $X=\bigcup_kX_k$. The condition on $(X_k)$ to be \emph{acyclic} \index{acyclic} means that, for all $k$, there is some $k'\ge k$ such that, for all $k''\ge k'$, the topologies of $X_{k'}$ and $X_{k''}$ coincide on some $0$-neighborhood of $X_k$ \cite[Theorem~6.1]{Wengenroth2003}. In this case, $X$ is Hausdorff if and only if all $X_k$ are Hausdorff \cite[Proposition~6.3]{Wengenroth2003}. It is said that $(X_k)$ is \emph{regular} \index{regular} if any bounded $B\subset X$ is contained and bounded in some step $X_k$. If moreover the topologies of $X$ and $X_k$ coincide on $B$, then $(X_k)$ is said to be \emph{boundedly retractive}. The conditions of being \emph{compactly retractive} \index{compactly retractive} or \emph{sequentially retractive} are similarly defined, using compact sets or convergent sequences. 

If the steps $X_k$ are Fr\'echet spaces, the above properties of $(X_k)$ only depend on the LF-space $X$ \cite[Chapter~6, p.~111]{Wengenroth2003}, and therefore they are considered as properties of $X$. In this case, $X$ is acyclic if and only if it is boundedly/compactly/sequentially retractive \cite[Proposition~6.4]{Wengenroth2003}. As a consequence, acyclic LF-spaces are complete and regular \cite[Corollary~6.5]{Wengenroth2003}. A topological vector subspace $Y\subset X$ is called a \emph{limit subspace} \index{limit subspace} if $Y\equiv\bigcup_k(X\cap Y_k)$ as TVSs.

Assume the steps $X_k$ are LCHSs. It is said that $(X_k)$ is \emph{compact} if the inclusion maps are compact operators. Then $(X_k)$ is acyclic, and so $X$ is Hausdorff. Moreover $X$ is a complete bornological DF Montel space \cite[Theorem~6']{Komatsu1967}.

The above concepts and properties also apply to an inductive/projective spectrum of LCSs consisting of continuous inclusions $X_r\subset X_{r'}$ for $r<r'$ in $\R$ because $\bigcap_rX_r=\bigcap_kX_{r_k}$ and $\bigcup_rX_r=\bigcup_kX_{s_k}$ for sequences $r_k\downarrow-\infty$ and $s_k\uparrow+\infty$.

In the category of continuous linear maps between LCSs, the exactness of a sequence $0\to X \to Y \to Z\to0$ means that it is exact as a sequence of linear maps and consists of topological homomorphisms \cite[Sections~2.1 and~2.2]{Wengenroth2003}.

Given LCSs $X$ and $Y$, let $L(X,Y)$ \index{$L(X,Y)$} denote the LCS of continuous linear maps $X\to Y$ with the topology of uniform convergence over bounded subsets. If $X$ and $Y$ are Banach spaces, then $L(X,Y)$ is also a Banach space whose norm may be denoted by $\|{\cdot}\|_{X,Y}$, with possible simplifications to avoid redundant notation. If $X=Y$, then the notation $\End(X)$ \index{$\End(X)$} is used, as well as $\|{\cdot}\|_X$ if $X$ is a Banach space. 

The following construction will be often used. Given a linear subspace $\AA$ of closed operators, densely defined in $X$ and with values in $Y$, we get the LCS
\begin{equation}\label{Z = u in bigcup_A in AA dom A | AA cdot u subset Y}
Z=\Big\{\,u\in\bigcap_{A\in\AA}\dom A\mid \AA\cdot u\subset Y\,\Big\}
\end{equation}
with the projective topology given by the maps $A:Z\to Y$ ($A\in\AA$).  If $Y$ is a Fr\'echet space, $L(X,Y)\subset\AA$ and $\AA/L(X,Y)$ is countably generated, then $Z$ is easily seen to be a Fr\'echet space. If moreover $Y$ is a Hilbertian space, then $Z$ is easily seen to be a totally reflexive Fr\'echet space using \cite[Theorem~4]{Valdivia1989}.

A \emph{Hilbertian space} \index{Hilbertian space} is a TVS $X$ endowed with a family of Hilbert-space scalar products, all of them with equivalent norms defining the topology of $X$, but none of them is distinguished.

\subsection{Smooth functions on open subsets of $\R^n$}\label{ss: smooth functions on U}

For any open $U\subset\R^n$ ($n\in\N_0=\N\cup\{0\}$), we use the Fr\'echet space $C^\infty(U)$ of smooth ($\K$-valued) functions on $U$, whose topology is described by the semi-norms
\begin{equation}\label{| u |_K C^k}
\|u\|_{K,C^k}=\sup_{x\in K,\ |I|\le k}|\partial^I u(x)|\;,
\end{equation}
for any compact $K\subset U$, $k\in\N_0$ and $I\in\N_0^n$, with standard multi-index notation. For any $S\subset U$, let $C^\infty_S(U)\subset C^\infty(U)$ be the topological vector subspace of smooth functions supported in $S$. The strict LF-space of compactly supported functions is
\begin{equation}\label{Cinftyc(U)}
\Cinftyc(U)=\bigcup_KC^\infty_K(U)\;,
\end{equation}
for compact subsets $K\subset U$.

Straightforward generalizations to the case of functions with values in $\K^l$ ($l\in\N$) can be given by 
\begin{equation}\label{C^infty_./c(U K^l) equiv C^infty_./c(U) otimes K^l}
C^\infty_{{\cdot}/\co}(U,\K^l)\equiv C^\infty_{{\cdot}/\co}(U)\otimes\K^l\;.
\end{equation}
(The notation $C^\infty_{{\cdot}/\co}$ or $C^\infty_{\co/{\cdot}}$ refers to both $C^\infty$ and $\Cinftyc$.)

\subsection{Vector bundles}\label{ss: vector bundless}

We fix a smooth $n$-manifold $M$ and a ($\K$-) vector bundle $E$ of rank $l$ over $M$. Let $E_x\subset E$ ($x\in M$) denote the fibers of $E$, $0_x$ the zero element of $E_x$, and $0_M$ the zero section of $E$. Let $\Omega^aE$ \index{$\Omega^aE$} ($a\in\R$) be the line bundle of $a$-densities of $E$, and $o(E)$ \index{$o(E)$} the flat line bundle of its orientations; as usual, we write $\Omega E=\Omega^1E$. \index{$\Omega E$} Recall that $\Omega^aE\otimes\Omega^bE\equiv\Omega^{a+b}E$. We use the notation $\Lambda E=\bigwedge E^*$ for the exterior bundle of the dual bundle. We may denote $\Lambda^\topd E=\Lambda^lE$, and use similar notation with other gradings and bigradings. \index{$\Lambda E$} \index{$\Lambda^\topd E$} For any submanifold $L\subset M$, we also write $E_L=E|_L$. As particular cases, we have the tangent and cotangent $\R$-vector bundles, $TM$ and $T^*M$, and the associated $\K$-vector bundles $o(M)=o(TM)\otimes\K$, \index{$o(M)$} $\Lambda M=\Lambda TM\otimes\K$, $\Omega^aM=\Omega^aTM\otimes\K$ and $\Omega M=\Omega TM\equiv\Lambda^n M\otimes o(M)$. \index{$\Lambda M$} \index{$\Omega^aM$} \index{$\Omega M$}

\subsection{Smooth and distributional sections}\label{ss: smooth/distributional sections}

Concerning spaces of distributional sections, we follow the notation of \cite{Melrose1996,Hormander1983-I,Hormander1985-III}, with some minor changes to fit our notation for foliations. The precise references of the properties recalled here are given in \cite[Section~2.4]{AlvKordyLeichtnam-conormal}.

Consider the Fr\'echet space $C^\infty(M;E)$ of smooth sections of $E$, whose topology is described by semi-norms $\|{\cdot}\|_{K,C^k}$ defined like in~\eqref{| u |_K C^k}, using charts $(U,x)$ of $M$ and diffeomorphisms of triviality $E_U\equiv U\times\K^l$ with $K\subset U$. Redundant notation is simplified as usual. For instance, in the case of the trivial vector bundle of rank $1$ (resp., $l$), we write $C^\infty(M)$ (resp., $C^\infty(M,\K^l)$). We also write $C^\infty(L,E)=C^\infty(L,E_L)$ and $C^\infty(M;\Omega^a)=C^\infty(M;\Omega^aM)$. If $M$ is fixed, the notation $C^\infty(E)=C^\infty(M;E)$ can be used, but it may be confusing because the space of smooth functions on $E$ is also used. In particular, $\fX(M)=C^\infty(M;TM)$ is the Lie algebra of vector fields. The subspace $C^\infty_S(M;E)$ \index{$C^{\pm\infty}_S(M;E)$} is defined like in \Cref{ss: smooth functions on U}, and the strict LF-space $\Cinftyc(M;E)$ \index{$C^{\pm\infty}_{{\cdot}/\co}(M;E)$} is defined like in~\eqref{Cinftyc(U)}, using compact subsets $K\subset M$. There is a continuous inclusion $\Cinftyc(M;E)\subset C^\infty(M;E)$.

The notation $C^\infty(M;E)$, or $C^\infty(E)$, is also used with any smooth fiber bundle $E$, obtaining a completely metrizable topological space with the weak $C^\infty$ topology.

The space of distributional sections with arbitrary/compact support is
\begin{equation}\label{C^-infty_cdot/c(M;E)}
C^{-\infty}_{{\cdot}/\co}(M;E)=C^\infty_{\co/{\cdot}}(M;E^*\otimes\Omega)'\;.
\end{equation}
The canonical pairing of any $u\in C^{-\infty}_{{\cdot}/\co}(M;E)$ and $v\in C^\infty_{\co/{\cdot}}(M;E^*\otimes\Omega)$ is denoted by $\langle u,v\rangle$ (or $(u,v)$ if the notation $\langle{\cdot},{\cdot}\rangle$ is used for other purposes). Integration of smooth densities on $M$ and the canonical pairing of $E$ and $E^*$ define a continuous dense inclusion $C^\infty_{{\cdot}/\co}(M;E)\subset C^{-\infty}_{{\cdot}/\co}(M;E)$. If $U\subset M$ is open, the extension by zero defines a TVS-embedding $C^{\pm\infty}_\co(U;E)\subset C^{\pm\infty}_\co(M;E)$. 

The above spaces of distributional sections can be also described in terms of the corresponding spaces of distributions as the algebraic tensor product as $C^\infty(M)$-modules \cite[Eq.~(2.5)]{AlvKordyLeichtnam-conormal}
\begin{equation}\label{C^infty(M)-tensor product description of C^pm infty_./c(M E)}
C^{-\infty}_{{\cdot}/\co}(M;E)\equiv C^{-\infty}_{{\cdot}/\co}(M)\otimes_{C^\infty(M)}C^\infty(M;E)\;.
\end{equation}
This tensor product has an induced topology so that this is a TVS-identity. Expressions like~\eqref{C^infty(M)-tensor product description of C^pm infty_./c(M E)} hold for most of the LCSs of distributional sections we will consider, which are also $C^\infty(M)$-modules. Thus, from now on, we will often define and study those spaces for the trivial line bundle or density bundles, and then the notation for arbitrary vector bundles will be used without further comment, and the properties have straightforward extensions.

Given a smooth submersion $\phi:M\to M'$, a smooth/distributional section of $E$ has \emph{compact support in the vertical direction} if its support has compact intersections with the fibers of $\phi$. They form the LCHSs $C^{\pm\infty}_\cv(M;E)$. \index{$C^{\pm\infty}_\cv(M;E)$} Here, $\Cinftycv(M;E)$ has the inductive topology defined like in the case of $\Cinftyc(M;E)$, using~\eqref{| u |_K C^k} and~\eqref{Cinftyc(U)} with closed subsets $K\subset M$ whose intersection with the fibers is compact. $C^{-\infty}_\cv(M;E)$ has the projective topology defined by the (product) maps $f:C^{-\infty}_\cv(M;E)\to C^{-\infty}_\co(M;E)$, for $f\in\Cinftyc(M)$. A version of~\eqref{C^infty(M)-tensor product description of C^pm infty_./c(M E)} is also true for $C^{-\infty}_\cv(M;E)$ in this case.

Consider also the Fr\'echet space $C^k(M)$ ($k\in\N_0$) of $C^k$ functions, with the semi-norms $\|{\cdot}\|_{K,C^k}$ given like in~\eqref{| u |_K C^k}, the LF-space $C^k_\co(M)$ of $C^k$ functions with compact support, defined like in~\eqref{Cinftyc(U)}, and the space $C^{\prime\,-k}_{{\cdot}/\co}(M)$ \index{$C^{\prime\,-k}_{{\cdot}/\co}(M)$} of distributions of order $k$ with arbitrary/compact support, defined like in~\eqref{C^-infty_cdot/c(M;E)}. There are continuous dense inclusions
\begin{equation}\label{C^prime -k'(M E) supset C^prime -k(M E)}
C^{k'}_{{\cdot}/\co}(M)\subset C^k_{{\cdot}/\co}(M)\;,\quad 
C^{\prime\,-k'}_{\co/{\cdot}}(M)\supset C^{\prime\,-k}_{\co/{\cdot}}(M)\quad(k<k')\;,
\end{equation}
with
\begin{equation}\label{bigcap_k C^k_./c(M) = C^infty_./c(M)}
\bigcap_kC^k_{{\cdot}/\co}(M)=C^{\infty}_{{\cdot}/\co}(M)\;,\quad
\bigcup_kC^{\prime\,-k}_\co(M)=C^{-\infty}_\co(M)\;.
\end{equation}
The space $\bigcup_kC^{\prime\,-k}(M)$ consists of the distributions with some order. If $M$ is compact, then every $C^k(M)$ is a Banach space and $\bigcup_kC^{\prime\,-k}(M)=C^{-\infty}(M)$.

$C^{\infty}_\co(M)$ and $C^{k}_\co(M)$ are complete and Hausdorff. $C^{\infty}_{{\cdot}/\co}(M)$ and $C^{k}_{{\cdot}/\co}(M)$ are ultrabornological and barreled. $C^{\pm\infty}_{{\cdot}/\co}(M)$ is a Montel space (in particular, barreled) and reflexive. $C^{\infty}_{{\cdot}/\co}(M)$ is a Schwartz space, and therefore $C^{-\infty}_{{\cdot}/\co}(M)$ is ultrabornological. $C^\infty(M)$ is distinguished. $C^{\pm\infty}_{{\cdot}/\co}(M)$ is webbed.

The type of notation introduced in this section will be used with any LCHS and $C^\infty(M)$-module continuously included in $C^{\pm\infty}(M;E)$.

\subsection{Linear operators on section spaces}\label{ss: ops}

Let $E$ and $F$ be vector bundles over $M$, and let $A:\Cinftyc(M;E)\to C^\infty(M;F)$ be a continuous linear map. The \emph{transpose} of $A$ is the continuous linear map
\begin{gather*}
A^\trans:C^{-\infty}_\co(M;F^*\otimes\Omega)\to C^{-\infty}(M;E^*\otimes\Omega)\;,\\
\langle A^\trans v,u\rangle=\langle v,Au\rangle\;,\quad u\in\Cinftyc(M;E)\;,\quad v\in C^{-\infty}_\co(M;F^*\otimes\Omega)\;.
\end{gather*}
For instance, the transpose of $\Cinftyc(M;E^*\otimes\Omega)\subset C^\infty(M;E^*\otimes\Omega)$ is a continuous dense injection $C^{-\infty}_\co(M;E)\subset C^{-\infty}(M;E)$. If $A^\trans$ restricts to a continuous linear map $\Cinftyc(M;F^*\otimes\Omega)\to C^\infty(M;E^*\otimes\Omega)$, then $A^{tt}:C^{-\infty}_\co(M;E)\to C^{-\infty}(M;F)$ is a continuous extension of $A$, also denoted by $A$. The \emph{Schwartz kernel},\index{Schwartz kernel} $K_A\in C^{-\infty}(M^2;F\boxtimes(E^*\otimes\Omega))$,\index{$K_A$} is determined by the condition $\langle K_A,v\otimes u\rangle = \langle v,Au\rangle$ for $u\in\Cinftyc(M;E)$ and $v\in\Cinftyc(M;F^*\otimes\Omega)$. The Schwartz kernel theorem \cite[Theorem~5.2.1]{Hormander1971} states that we have a linear isomorphism
\begin{equation}\label{Schwartz kernel theorem}
L(\Cinftyc(M;E),C^{-\infty}(M;F))\xrightarrow{\cong} C^{-\infty}(M^2;F\boxtimes (E^*\otimes\Omega M))\;,\quad 
A\mapsto K_A\;.
\end{equation}
Using that $(F^*\otimes \Omega)^*\otimes \Omega\equiv F$, we get
\[
K_{A^\trans}=R^*K_A\in C^{-\infty}(M^2;(E^*\otimes\Omega)\boxtimes F)\;,
\]
where $R: M^2\to M^2$ is given by $R(x,y)=(y,x)$. If $K_A$ is $C^\infty$, we can write
\begin{equation}\label{Au(y) = int_M K_A(y x) u(x)}
Au(x)=\int_MK_A(x,y)u(y)\;,\quad A^\trans v(x)=\int_MK_A(y,x)v(y)\;,
\end{equation}
for $u\in\Cinftyc(M;E)$ and $v\in\Cinftyc(M;F^*\otimes\Omega)$.

There are versions of the construction of $A^\trans$ and $A^{\text{\rm tt}}$ when both the domain and codomain of $A$ have compact support, or no support restriction. For example, for any open $U\subset M$, the transpose of the extension by zero $\Cinftyc(U;E^*\otimes\Omega)\subset\Cinftyc(M;E^*\otimes\Omega)$ is the restriction map
\begin{equation}\label{restriction map}
C^{-\infty}(M;E)\to C^{-\infty}(U,E)\;,\quad u\mapsto u|_U\;,
\end{equation}
and the transpose of the restriction map $C^\infty(M;E^*\otimes\Omega)\to C^\infty(U,E^*\otimes\Omega)$ is the extension by zero
\begin{equation}\label{ext by 0}
C^{-\infty}_\co(U;E)\subset C^{-\infty}_\co(M;E)\;.
\end{equation}
Inclusion maps may be denoted by $\iota$ \index{$\iota$} and restriction maps by $R$, \index{$R$} without further comment. The \emph{singular support} of any $u\in C^{-\infty}(M;E)$, $\sing\supp u$, is the complement of the maximal open subset $U\subset M$ with $u|_U\in C^\infty(U;E)$.

\subsection{Pull-back and push-forward of distributional sections}
\label{ss: pull-back and push-forward of distrib sections}

Any smooth map $\phi:M'\to M$ induces the continuous linear pull-back map
\begin{equation}\label{phi^*: C^infty(M E) -> C^infty(M' phi^*E)}
\phi^*:C^\infty(M;E)\to C^\infty(M';\phi^*E)\;.
\end{equation}
If $\phi$ is a submersion, then it also induces the continuous linear push-forward map
\begin{equation}\label{phi_*: Cinftyc(M' phi^*E otimes Omega_fiber) -> Cinftyc(M E)}
\phi_*:\Cinftyc(M';\phi^*E\otimes\Omega_\fiber)\to \Cinftyc(M;E)\;,
\end{equation}
where $\Omega_\fiber=\Omega_\fiber M'=\Omega\VV$ \index{$\Omega_\fiber$} for the vertical subbundle $\VV=\ker\phi_*\subset TM'$. Moreover, the map~\eqref{phi_*: Cinftyc(M' phi^*E otimes Omega_fiber) -> Cinftyc(M E)} has a continuous extension
\begin{equation}\label{phi_*: C^infty_cv(M' phi^*E otimes Omega_fiber) -> Cinfty(M E)}
\phi_*:C^\infty_\cv(M';\phi^*E\otimes\Omega_\fiber)\to C^\infty(M;E)\;,
\end{equation}
also called push-forward map. Using~\eqref{phi_*: Cinftyc(M' phi^*E otimes Omega_fiber) -> Cinftyc(M E)} and any partition of unity $\{\lambda_j\}$ of $M$ consisting of compactly supported smooth functions, the map~\eqref{phi_*: C^infty_cv(M' phi^*E otimes Omega_fiber) -> Cinfty(M E)} is given by
\begin{equation}\label{phi_*u}
\phi_*u=\sum_j\phi_*(\phi^*\lambda_j\cdot u)\;.
\end{equation}

Since $\phi^*\Omega M\equiv\Omega(TM/\VV)\equiv\Omega^{-1}_{\text{\rm fiber}}\otimes\Omega M'$, transposing the versions of~\eqref{phi^*: C^infty(M E) -> C^infty(M' phi^*E)} and~\eqref{phi_*: Cinftyc(M' phi^*E otimes Omega_fiber) -> Cinftyc(M E)} with $E^*\otimes\Omega M$ and using~\eqref{C^-infty_cdot/c(M;E)}, we obtain continuous extensions of~\eqref{phi_*: Cinftyc(M' phi^*E otimes Omega_fiber) -> Cinftyc(M E)} and~\eqref{phi^*: C^infty(M E) -> C^infty(M' phi^*E)} \cite[Theorem~6.1.2]{Hormander1971},
\begin{gather}
\phi_*:C^{-\infty}_\co(M';\phi^*E\otimes\Omega_\fiber)\to C^{-\infty}_\co(M;E)\;,
\label{phi_*: C^-infty_c(M' phi^*E otimes Omega_fiber) -> C^-infty_c(M E)}\\
\phi^*:C^{-\infty}(M;E)\to C^{-\infty}(M';\phi^*E)\;,
\label{phi_*: C^-infty(M E) -> C^-infty(M' phi^*E)}
\end{gather}
also called push-forward and pull-back maps. Again,~\eqref{phi_*: C^-infty_c(M' phi^*E otimes Omega_fiber) -> C^-infty_c(M E)} has a continuous extension,
\begin{equation}\label{phi_*: C^-infty_cv(M' phi^*E otimes Omega_fiber) -> C^-infty(M E)}
\phi_*:C^{-\infty}_\cv(M';\phi^*E\otimes\Omega_\fiber)\to C^{-\infty}(M;E)\;,
\end{equation}
also called push-forward map, defined like~\eqref{phi_*: C^infty_cv(M' phi^*E otimes Omega_fiber) -> Cinfty(M E)} with~\eqref{phi_*: C^-infty_c(M' phi^*E otimes Omega_fiber) -> C^-infty_c(M E)}.

If $\phi:M'\to M$ is a local diffeomorphism, we can omit $\Omega_\fiber$ in the push-forward maps. If moreover $\phi$ is proper, the compositions $\phi_*\phi^*$ and $\phi^*\phi_*$ are defined on smooth or distributional sections with compact support or no support condition.

The spaces $C^\infty(M';\phi^*E)$ and $C^\infty(M';\phi^*E\otimes\Omega_\fiber)$ become $C^\infty(M)$-modules via the homomorphism of algebras, $\phi^*:C^\infty(M)\to C^\infty(M')$, and we have
\begin{align}
C^{\pm\infty}_{{\cdot}/\co}(M';\phi^*E)&=C^{\pm\infty}_{{\cdot}/\co}(M')\otimes_{C^\infty(M)}C^\infty(M;E)\;,
\label{C^infty(M)-tensor product description of C^pm infty_./c(M' phi^*E)}\\
C^{\pm\infty}_{{\cdot}/\co}(M';\phi^*E\otimes\Omega_\fiber)
&=C^{\pm\infty}_{{\cdot}/\co}(M';\Omega_\fiber)\otimes_{C^\infty(M)}C^\infty(M;E)\;.
\label{C^infty(M)-tensor product description of C^pm infty_cv/c(M' phi^*E otimes Omega_fiber)}
\end{align}
Using~\eqref{C^infty(M)-tensor product description of C^pm infty_./c(M E)},~\eqref{C^infty(M)-tensor product description of C^pm infty_./c(M' phi^*E)} and~\eqref{C^infty(M)-tensor product description of C^pm infty_cv/c(M' phi^*E otimes Omega_fiber)}, we can describe~\eqref{phi^*: C^infty(M E) -> C^infty(M' phi^*E)}--\eqref{phi_*: C^-infty_cv(M' phi^*E otimes Omega_fiber) -> C^-infty(M E)} as the $C^\infty(M)$-tensor products of their trivial-line-bundle versions with the identity map on the space $C^\infty(M;E)$. This kind of description is valid with other spaces of distributional sections with the obvious extensions of~\eqref{C^infty(M)-tensor product description of C^pm infty_./c(M' phi^*E)} and~\eqref{C^infty(M)-tensor product description of C^pm infty_cv/c(M' phi^*E otimes Omega_fiber)}. Thus, in this chapter, we will mainly consider the pull-back and push-forward between spaces of distributions. Only the special case of the pull-back and push-forward between spaces of currents will be briefly indicated a few times.

\subsection{Differential operators}\label{ss: diff ops}

Let $\Diff(M)\subset\End(C^\infty_{{\cdot}/\co}(M))$ be the subalgebra and $C^\infty(M)$-submodule of differential operators, filtered by the order. Every $\Diff^m(M)$ ($m\in\N_0$) is spanned as $C^\infty(M)$-module by all compositions of up to $m$ elements of $\fX(M)$, considered as the Lie algebra of derivations of $C^\infty_{{\cdot}/\co}(M)$. In particular, $\Diff^0(M)\equiv C^\infty(M)$.

On the other hand, let 
\[
P(T^*M)=\bigoplus_{m=0}^\infty P^{(m)}(T^*M)\subset C^\infty(T^*M)
\]
be the subalgebra and $C^\infty(M)$-submodule of functions whose restriction to the fibers are  polynomials, equipped with the grading given by the degree; in particular,
\[
P^{(0)}(T^*M)\equiv C^\infty(M)\;,\quad P^{(1)}(T^*M)\equiv\fX(M)\otimes\C\;. 
\]
For every order $m$, the principal symbol exact sequence
\begin{equation}\label{sigma_m}
0\to\Diff^{m-1}(M)\hookrightarrow\Diff^m(M) \xrightarrow{\sigma_m} P^{(m)}(T^*M)\to0
\end{equation}
is defined so that the principal symbol of any $X\in\fX(M)\subset\Diff^1(M)$ is $\sigma_1(X)=iX\in P^{(1)}(T^*M)$, and $\bigoplus_m\sigma_m$ induces an isomorphism of graded algebras and $C^\infty(M)$-modules,
\[
\bigoplus_{m=0}^\infty\Diff^m(M)/\Diff^{m-1}(M)\xrightarrow{\cong} P(T^*M)\;.
\]

 For vector bundles $E$ and $F$ over $M$, the above concepts can be extended by taking the $C^\infty(M)$-tensor product with $C^\infty(M;F\otimes E^*)$, obtaining
\begin{align*}
\Diff^m(M;E,F)&\subset L(C^\infty_{{\cdot}/\co}(M;E),C^\infty_{{\cdot}/\co}(M;F))\;,\\
P^{(m)}(T^*M;F\otimes E^*)&\subset C^\infty(T^*M;\pi^*(F\otimes E^*))\;,
\end{align*}
where $\pi:T^*M\to M$ is the projection. So $\Diff^0(M;E,F)\equiv C^\infty(M;F\otimes E^*)$. If $E=F$, we write $\Diff(M;E)$, which is a filtered algebra. The principal symbol $\sigma_m$ on $\Diff^m(M;E,F)$ \index{$\Diff^m(M;E,F)$} is given by the $C^\infty(M)$-tensor product of~\eqref{sigma_m} with the identity map on $C^\infty(M;F\otimes E^*)$. Redundant notation is simplified like in \Cref{ss: smooth/distributional sections}. Recall that $A\in\Diff^m(M;E,F)$ is \emph{elliptic} if $\sigma_m(A)(p,\xi)$ is an isomorphism for all $p\in M$ and $0\ne\xi\in T^*_pM$.  If $E$ is a line bundle, then \cite[Eq.~(2.13)]{AlvKordyLeichtnam-conormal}
\begin{equation}\label{Diff^m(M E) equiv Diff^m(M)}
\Diff^m(M;E)\equiv\Diff^m(M)\;.
\end{equation}
For $m=0$, we get $C^\infty(M;E\otimes E^*)\equiv C^\infty(M)$.

For all $A\in\Diff^m(M;E)$, we have $A^\trans\in\Diff^m(M;E^*\otimes\Omega)$, and therefore $A$ has continuous extensions to an endomorphism $A$ of $C^{-\infty}_{{\cdot}/\co}(M;E)$ (\Cref{ss: ops}). A similar map is defined when $A\in\Diff^m(M;E,F)$.

The canonical coordinates of $\R^n\times\R^n\equiv\R^n\times\R^{n*}\equiv T^*\R^n$ are denoted by $(x,\xi)=(x^1,\dots,x^n,\xi^1,\dots,\xi^n)$. Let $dx=dx^1\wedge\dots\wedge dx^n$, $d\xi=d\xi^1\wedge\dots\wedge d\xi^n$, $D^I=D_x^I=(-i)^{|I|}\partial_I=(-i)^{|I|}\partial_{x,I}$ ($i=\sqrt{-1}$) and $\xi^I=\xi^{i_1}\cdots\xi^{i_n}$ ($I=(i_1,\dots,i_n)\in\N_0^n$). For any open $U\subset\R^n$ and $A=\sum_{|I|\le m}a_I(x)\,D^I\in\Diff^m(U)$, write $A=a(x,D)$ for $a(x,\xi)=\sum_{|I|\le m}a_I(x)\,\xi^I$, and then $\sigma_m(A)=\sum_{|I|=m}a_I(x)\,\xi^I$. We have
\begin{equation}\label{Au(x) = ...}
Au(x)=(2\pi)^{-n}\int_{\R^n}e^{i\langle x,\xi\rangle}a(x,\xi)\hat u(\xi)\,d\xi\;,
\end{equation}
for all $u\in\Cinftyc(U)$, where $\hat u$ is the Fourier transform of $u$. The local extension of this expression to the case where $A\in\Diff^m(M;E,F)$ is straightforward, using charts of $M$ and local trivializations of $E$ and $F$, and taking local coefficients $a_I$ with values in $\C^{l'}\otimes\C^{l*}\equiv\C^{l\times l'}$ ($l$ and $l'$ are the ranks of $E$ and $F$).

\subsection{Symbols}\label{ss: symbols}

For any open $U\subset\R^n$ and $l\in\N_0$, a \emph{symbol} of \emph{order} at most $m\in\R$ on $U\times\R^l$, or simply on $U$, is a function $a\in C^\infty(U\times\R^l)$ such that, for any compact $K\subset U$, $I\in\N_0^n$ and $J\in\N_0^l$,
\begin{equation}\label{| a |_K I J m}
\|a\|_{K,I,J,m}:=\sup_{x\in K,\ \xi\in\R^l}\frac{|D_x^I D_\xi^J a(x,\xi)|}{(1+|\xi|)^{m-|J|}}<\infty\;.
\end{equation}
They form a Fr\'echet space $S^m(U\times\R^l)$ with the semi-norms~\eqref{| a |_K I J m}. There are continuous inclusions
\begin{equation}\label{S^m(T^*M) subset S^m'(T^*M)}
S^m(U\times\R^l)\subset S^{m'}(U\times\R^l)\quad(m<m')\;,
\end{equation}
giving rise to the LCSs
\[
S^\infty(U\times\R^l)=\bigcup_mS^m(U\times\R^l)\;,\quad S^{-\infty}(U\times\R^l)=\bigcap_mS^m(U\times\R^l)\;.
\]
$S^\infty(U\times\R^l)$ is an LF-space, and therefore barreled, ultrabornological and webbed \cite[Proposition~3.1]{AlvKordyLeichtnam-conormal}. It is also a filtered algebra and $C^\infty(U)$-module with the pointwise multiplication. The homogeneous components of the corresponding graded algebra are denoted by $S^{(m)}(U\times\R^l)$. \index{$S^{(m)}(U\times\R^l)$} The Fr\'echet space $S^{-\infty}(U\times\R^l)$ is a filtered ideal and $C^\infty(U)$-submodule of $S^\infty(U\times\R^l)$.
The notation $S^m(\R^l)$, $S^{\pm\infty}(\R^l)$ and $S^{(m)}(\R^l)$ is used when $U=\R^0=\{0\}$.

Consider the first-factor projection $U\times\R^l\to U$ to define $\Cinftycv(U\times\R^l)$. There are continuous inclusions
\begin{equation}\label{S^infty(U x R^l) subset C^infty(U x R^l)}
\Cinftycv(U\times\R^l)\subset S^{-\infty}(U\times\R^l)\;,\quad S^\infty(U\times\R^l)\subset C^\infty(U\times\R^l)\;;
\end{equation}
in particular, $S^\infty(U\times\R^l)$ is Hausdorff. The following properties hold \cite[Corollaries~3.4--3.6 and Remark~3.8]{AlvKordyLeichtnam-conormal}: The topologies of $S^\infty(U\times\R^l)$ and $C^\infty(U\times\R^l)$ coincide on $S^m(U\times\R^l)$, however the second inclusion of~\eqref{S^infty(U x R^l) subset C^infty(U x R^l)} is not a TVS-embedding; $\Cinftyc(U\times\R^l)$ is dense in $S^\infty(U\times\R^l)$; and $S^\infty(U\times\R^l)$ is an acyclic Montel space, and therefore complete, boundedly/compactly/sequentially retractive and reflexive.

With more generality, a symbol of order $m$ on a vector bundle $E$ over $M$ is a smooth function on $E$ satisfying~\eqref{| a |_K I J m} via charts of $M$ and local trivializations of $E$, with $K$ contained in the domains of charts where $E$ is trivial. As above, they form a Fr\'echet space $S^m(E)$ with the topology described by the semi-norms given by this version of~\eqref{| a |_K I J m}. The version of~\eqref{S^m(T^*M) subset S^m'(T^*M)} in this setting is true, obtaining the corresponding spaces $S^{\pm\infty}(E)$ and $S^{(m)}(E)$. The above properties have obvious extensions to this setting.

Given another vector bundle $F$ over $M$, the $C^\infty(M)$-tensor product of the above spaces with $C^\infty(M;F)$ gives spaces $S^m(E;F)$, $S^{\pm\infty}(E;F)$ and $S^{(m)}(E;F)$, satisfying analogous properties. Now~\eqref{S^infty(U x R^l) subset C^infty(U x R^l)} becomes $\Cinftycv(E;\pi^*F)\subset S^{-\infty}(E;F)$ and $S^{\infty}(E;F)\subset C^\infty(E;\pi^*F)$, where $\pi:E\to M$ is the projection.

\subsection{Pseudodifferential operators}\label{ss: pseudodiff ops}

The notation of \Cref{ss: symbols} is used here. For any $a\in S^m(U\times\R^n)$, the expression~\eqref{Au(x) = ...} defines a continuous linear map $A=a(x,D):\Cinftyc(U)\to C^\infty(U)$, with Schwartz kernel
\[
K_A(x,y)=(2\pi)^{-n}\left(\int_{\R^n}e^{i\langle x-y,\xi\rangle}a(x,\xi)\,d\xi\right)|dy|\;,
\]
using an oscillatory integral, which is defined as a tempered distribution \cite[Eq.~(4.2)]{Melrose1981}, \cite[Section~7.8]{Hormander1983-I}.

Take an atlas $\{U_k,x_k\}$ of $M$ and an associated $C^\infty$ partition of unity $\{f_k\}$. Via every chart $(U_k,x_k)$, for all $a\in S^m(T^*U_k)$, the above procedure defines a continuous linear map $a(x_k,D_{x_k}):\Cinftyc(U_k)\to C^\infty(U_k)$. 

Let $\Delta\subset M^2$ be the diagonal. A \emph{pseudodifferential operator} of \emph{order} at most $m$ on $M$ is a continuous linear map $A:\Cinftyc(M)\to C^\infty(M)$ such that $K_A$ is $C^\infty$ on $M^2\setminus\Delta$, and, for all $k$, the operator $f_kA:\Cinftyc(U_k)\to\Cinftyc(U_k)$ is of the form $a_k(x_k,D_{x_k})$ for some $a_k\in S^m(T^*U_k)$, which is supported in $\pi^{-1}(\supp f_k)$, where $\pi:T^*M\to M$ is the projection. They form a $C^\infty(M^2)$-module $\Psi^m(M)$ with the pointwise multiplication of their Schwartz kernels by smooth functions on $M^2$. Moreover $\sum_ka_k\in S^m(T^*M)$ defines a class $\sigma_m(A)\in S^{(m)}(T^*M)$, called the \emph{principal symbol}, which is independent of the choices involved, obtaining an exact sequence of $C^\infty(M^2)$-modules,
\[
0\to\Psi^{m-1}(M)\hookrightarrow\Psi^m(M) \xrightarrow{\sigma_m} S^{(m)}(T^*M)\to0\;,
\]
where $S^{(m)}(T^*M)$ is a $C^\infty(M^2)$-module via the restriction linear map $C^\infty(M^2)\to C^\infty(\Delta)\equiv C^\infty(M)$. Then $\Psi(M):=\bigcup_m\Psi^m(M)$ is a filtered $C^\infty(M^2)$-module, and $\Psi^{-\infty}(M):=\bigcap_m\Psi^m(M)$ is the submodule of the operators with $C^\infty$ Schwartz kernel (the \emph{smoothing} operators). All of these concepts are independent of the choices involved. If $m\in\N_0$, then
\[
\Diff^m(M)=\{\,A\in\Psi^m(M)\mid\supp K_A\subset\Delta\,\}\;.
\]

These concepts and properties can be extended to vector bundles by taking the $C^\infty(M^2)$-tensor product with $C^\infty(M^2;F\boxtimes E^*)$, like in the case of differential operators (\Cref{ss: diff ops}). In this case, we use the notation $\Psi^m(M;E,F)$ \index{$\Psi^m(M;E)$} (or $\Psi^m(M;E)$ if $E=F$), \index{$\Psi^m(M;E)$} $S^{(m)}(T^*M;F\otimes E^*)$, etc.  Recall that an operator $A\in\Psi^m(M;E,F)$ is called \emph{elliptic} if $\sigma_m(A)$ has an inverse in $S^{(-m)}(T^*M;F,E)$; i.e., any representative of $\sigma_m(A)$ is an isomorphism at $(p,\xi)\in T^*M$ if $\xi$ is far enough from $0_p$ in $T^*_pM$. The space $\Psi^m(M;E,F)$ is preserved by taking transposes. Thus any $A\in\Psi^m(M;E,F)$ has a continuous extension (\Cref{ss: ops})
\[
A:C^{-\infty}_\co(M;E)\to C^{-\infty}(M;F)\;,
\]
and $\sing\supp Au\subset\sing\supp u$ for all $u\in C^{-\infty}_\co(M;E)$ (\emph{pseudolocality}). Moreover $A\in\Psi^{-\infty}(M;E,F)$ just when it defines a continuous map
\[
A:C^{-\infty}_\co(M;E)\to C^\infty(M;F)\;.
\]
It is said that $A$ is \emph{properly supported} if both factor projections $M^2\to M$ have proper restrictions to $\supp K_A$. In this case, $A$ defines continuous linear maps (\Cref{ss: ops})
\[
A:\Cinftyc(M;E)\to\Cinftyc(M;F)\;,\quad A:C^{-\infty}(M;E)\to C^{-\infty}(M;F)\;,
\]
which gives sense to the composition of properly supported pseudodifferential operators. Any pseudodifferential operator is properly supported modulo smoothing operators, and the symbol map is multiplicative. 

If $A\in\Psi^{-\infty}(M;E)$ and $P,Q\in\Diff(M;E)$, then
\begin{equation}\label{K_PAQ(x y)}
K_{PAQ}(x,y)=P_x\,Q_y^\trans K_A(x,y)\;.
\end{equation}

\subsection{$L^2$ and $L^\infty$ sections}\label{ss: L^2 and L^infty}

The Hilbert space $L^2(M;\Omega^{1/2})$ of square-integrable half-densities is the completion of $\Cinftyc(M;\Omega^{1/2})$ with the scalar product $\langle u,v\rangle=\int_Mu\bar v$. The induced norm is denoted by $\|{\cdot}\|$.

If $M$ is compact, $L^2(M;E)$ \index{$L^2(M;E)$} can be described as the $C^\infty(M)$-tensor product of $L^2(M;\Omega^{1/2})$ and $C^\infty(M;\Omega^{-1/2}\otimes E)$. It is a Hilbertian space with the scalar products $\langle u,v\rangle=\int_M(u,v)\,\omega$, determined by the choice of a Euclidean/Hermitian structure $({\cdot},{\cdot})$ on $E$ and a non-vanishing $\omega\in C^\infty(M;\Omega)$.

When $M$ is not assumed to be compact, any choice of $({\cdot},{\cdot})$ and $\omega$ can be used to define $L^2(M;E)$ and $\langle{\cdot},{\cdot}\rangle$. Now $L^2(M;E)$ and the equivalence class of $\|{\cdot}\|$ depends on the choices involved. The independence still holds for sections supported in any compact $K\subset M$, obtaining the Hilbertian space $L^2_K(M;E)$. These spaces give rise to the strict LF-space $L^2_\co(M;E)$ like in~\eqref{Cinftyc(U)}. We also get the Fr\'echet space
\[
L^2_{\text{\rm loc}}(M;E)=\{\,u\in C^{-\infty}(M;E)\mid\Cinftyc(M)\,u\subset L^2_\co(M;E)\,\}\;,
\]
defining the topology like in~\eqref{Z = u in bigcup_A in AA dom A | AA cdot u subset Y}. If $M$ is compact, then $L^2_{\text{\rm loc/c}}(M;E)\equiv L^2(M;E)$ as TVSs. The spaces $L^2_{\text{\rm loc/c}}(M;E)$ satisfy the obvious version of~\eqref{C^-infty_cdot/c(M;E)}.

Any $A\in\Diff^m(M;E)$ can be considered as a densely defined operator in $L^2(M;E)$. Its adjoint $A^*$ is the closure of the \emph{formal adjoint} $A^*\in\Diff^m(M;E)$, determined by the condition $\langle u,A^*v\rangle=\langle Au,v\rangle$ for all $u,v\in \Cinftyc(M;E)$.

We can also use $({\cdot},{\cdot})$ to define the Banach space $L^\infty(M;E)$ of essentially bounded sections, with the norm $\|u\|_{L^\infty}=\esssup_{x\in M}|u(x)|$. There is a continuous injection $L^\infty(M;E)\subset L^2_{\text{\rm loc}}(M;E)$. \index{$L^\infty(M;E)$} If $M$ is compact, the equivalence class of $\|{\cdot}\|_{L^\infty}$ is independent of $({\cdot},{\cdot})$.

\subsection{Sobolev spaces}\label{ss: Sobolev sps}

Suppose first that $M$ is compact. The \emph{Sobolev space} of \emph{order} $s\in\R$ is the Hilbertian space \index{$H^s(M;E)$}
\begin{equation}\label{H^s(M E)}
H^s(M;E)=\{\,u\in C^{-\infty}(M;E)\mid\Psi^s(M;E)\,u\subset L^2(M;E)\,\}\;,
\end{equation}
with the topology like in~\eqref{Z = u in bigcup_A in AA dom A | AA cdot u subset Y}. It can be equipped with any scalar product $\langle u,v\rangle_s=\langle(1+P)^su,v\rangle$, for any nonnegative symmetric elliptic $P\in\Diff^2(M;E)$ (by the elliptic estimate), where $\langle{\cdot},{\cdot}\rangle$ is defined like in \Cref{ss: L^2 and L^infty} and $(1+P)^s$ is given by the spectral theorem. Let $\|{\cdot}\|_s$ denote the corresponding norm. We have
\begin{equation}\label{H^-s(M E) = Psi^s(M E) cdot L^2(M E) = H^s(M E^* otimes Omega)'}
\Psi^s(M;E)\,L^2(M;E)=H^{-s}(M;E)=H^s(M;E^*\otimes\Omega)'\;.
\end{equation}
If $s\in\N$, we can use $\Diff^s(M;E)$ instead of $\Psi^s(M;E)$ in~\eqref{H^s(M E)} and the first equality of~\eqref{H^-s(M E) = Psi^s(M E) cdot L^2(M E) = H^s(M E^* otimes Omega)'}. There are dense compact inclusions (Rellich theorem) \begin{equation}\label{H^s(M E) subset H^s'(M E)}
H^s(M;E)\subset H^{s'}(M;E)\quad(s'<s)\;.
\end{equation}
So the spaces $H^s(M,E)$ form a compact spectrum. Moreover, there are continuous dense inclusions, for $s>k+n/2$,
\begin{gather}
H^s(M;E)\subset C^k(M;E)\subset H^k(M;E)\;,\label{H^s(M E) subset C^k(M E) subset H^k(M E)}\\
H^{-s}(M;E)\supset C^{\prime\,-k}(M;E)\supset H^{-k}(M;E)\;.\label{H^-s(M E) supset C^prime -k(M E) supset H^-k(M E)}
\end{gather}
The first inclusion of~\eqref{H^s(M E) subset C^k(M E) subset H^k(M E)} is the Sobolev embedding theorem, and~\eqref{H^-s(M E) supset C^prime -k(M E) supset H^-k(M E)} is the transpose of the version of~\eqref{H^s(M E) subset C^k(M E) subset H^k(M E)} with $E^*\otimes\Omega M$. So
\begin{equation}\label{C^infty(M E) = bigcap_s H^s(M E)}
C^\infty(M;E)=\bigcap_sH^s(M;E)\;\quad 
C^{-\infty}(M;E)=\bigcup_sH^s(M;E)\;.
\end{equation}

Any $A\in\Psi^m(M;E)$ defines a bounded operator $A:H^{s+m}(M;E)\to H^s(M;E)$. This can be considered as a densely defined operator in $H^s(M;E)$, which is closable because, after fixing a scalar product in $H^s(M;E)$, the adjoint of $A$ in $H^s(M;E)$ is densely defined since it is induced by $\bar A^\trans\in\Psi^m(M;\bar E^*\otimes\Omega)$ via the identity of real Hilbert spaces, $H^s(M;E)\equiv H^s(M;\bar E)'=H^{-s}(M;\bar E^*\otimes\Omega)$, where the bar stands for the complex conjugate. In the case $s=0$, the adjoint of $A$ is induced by the formal adjoint $A^*\in\Psi^m(M;E)$; if $A\in\Diff^m(M;E)$, then $A^*\in\Diff^m(M;E)$. 

If $M$ is not assumed to be compact, then $H^s(M;E)$ can be defined as the completion of $\Cinftyc(M;E)$ with respect to the scalar product $\langle{\cdot},{\cdot}\rangle_s$ defined by the above choices of $({\cdot},{\cdot})$, $\omega$ and $P$; in this case, $H^s(M;E)$ and the equivalence class of $\|{\cdot}\|_s$ depend on the choices involved. With this generality,~\eqref{H^s(M E)} and the first equality of~\eqref{H^-s(M E) = Psi^s(M E) cdot L^2(M E) = H^s(M E^* otimes Omega)'} are wrong, but the second equality of~\eqref{H^-s(M E) = Psi^s(M E) cdot L^2(M E) = H^s(M E^* otimes Omega)'} is true. 

Like $L^2_{\text{\rm loc/c}}(M;E)$ (\Cref{ss: L^2 and L^infty}), we can define the Fr\'echet space $H^s_{\text{\rm loc}}(M;E)$ and the strict LF-space $H^s_\co(M;E)$, which satisfy the versions of the second equality of~\eqref{H^-s(M E) = Psi^s(M E) cdot L^2(M E) = H^s(M E^* otimes Omega)'} (switching the support condition like in~\eqref{C^-infty_cdot/c(M;E)}) and~\eqref{H^s(M E) subset H^s'(M E)}--\eqref{H^-s(M E) supset C^prime -k(M E) supset H^-k(M E)}. These spaces agree with $H^s(M;E)$ if $M$ is compact. For any open $U\subset M$, the restriction map~\eqref{restriction map} defines a continuous linear map $H^s_{\text{\rm loc}}(M;E)\to H^s_{\text{\rm loc}}(U;E)$, and the extension by zero~\eqref{ext by 0} defines a TVS-embedding $H^s_\co(U;E)\subset H^s_\co(M;E)$. In this case, any $A\in\Psi^m(M;E)$ defines continuous linear maps $A:H^s_\co(M;E)\to H^{s-m}_{\text{\rm loc}}(M;E)$. If $A\in\Diff^m(M;E)$, then it defines continuous linear maps $A:H^s_{\text{\rm c/loc}}(M;E)\to H^{s-m}_{\text{\rm c/loc}}(M;F)$.

For example, $H^s(\R^n)$ can be defined with $\langle u,v\rangle_s=\langle(1+\Delta)^su,v\rangle$, involving the Laplacian $\Delta=-\sum_k\partial_k^2$ and the standard scalar product on $L^2(\R^n)$. Recall that the Fourier transform, $f\mapsto\hat f$, defines an automorphism of the Schwartz space $\SS(\R^n)$, which extends to an automorphism of the space $\SS(\R^n)'$ of tempered distributions \cite[Section~7.1]{Hormander1983-I}, which in turn restricts to a TVS-isomorphism
\begin{equation}\label{H^s(R^n) cong L^2(R^n (1+|xi|^2)^s d xi)}
H^s(\R^n)\xrightarrow{\cong} L^2(\R^n,(1+|\xi|^2)^s\,d\xi)\;,\quad f\mapsto\hat f\;.
\end{equation}
We can use~\eqref{H^s(R^n) cong L^2(R^n (1+|xi|^2)^s d xi)} to give an alternative description of $H^s_{\text{\rm c/loc}}(M;E)$ for arbitrary $M$ and $E$. First, $H^s_K(\R^n)\subset H^s(\R^n)$ has the subspace topology for any compact $K\subset\R^n$. Next, for any open $U\subset\R^n$, we can describe $H^s_\co(U)$ by using $H^s_K(U)\equiv H^s_K(\R^n)$ for all compact $K\subset U$, and we can describe $H^s_{\text{\rm loc}}(U)$ by using $H^s_\co(U)$, as explained before. Then a locally finite atlas and a subordinated $C^\infty$ partition of unity can be used in a standard way to describe $H^m_{\text{\rm c/loc}}(M)$. Finally, $H^s_{\text{\rm c/loc}}(M;E)$ can be described as the $C^\infty(M)$-tensor product of $H^s_{\text{\rm c/loc}}(M)$ with $C^\infty(M;E)$, or, equivalently, using local diffeomorphisms of triviality of $E$.

The norm on $L(H^m(M;E),H^{m'}(M;F))$ (resp., $\End(H^m(M;E))$) will be simply denoted by $\|{\cdot}\|_{m,m'}$ (resp., $\|{\cdot}\|_m$).

\subsection{Weighted spaces}\label{ss: weighted sps}

Assume first that $M$ is compact. Take any $h\in C^\infty(M)$ which is positive almost everywhere. Then the \emph{weighted Sobolev space} $hH^s(M;E)$ is a Hilbertian space; a scalar product $\langle{\cdot},{\cdot}\rangle_{hH^s}$ is given by $\langle u,v\rangle_{hH^s}=\langle h^{-1}u,h^{-1}v\rangle_s$, depending on the choice of a scalar product $\langle{\cdot},{\cdot}\rangle_s$ on $H^s(M;E)$ (\Cref{ss: Sobolev sps}). The corresponding norm is denoted by $\|{\cdot}\|_{hH^s}$. In particular, we get the \emph{weighted $L^2$ space} $hL^2(M;E)$. We have $h>0$ just when $hH^m(M;E)=H^m(M;E)$; in this case, $\langle{\cdot},{\cdot}\rangle_{hH^s}$ can be described like $\langle{\cdot},{\cdot}\rangle_s$ using $h^{-2}\omega$ instead of $\omega$. Thus the notation $hH^m(M;E)$ for $h>0$ is used when changing the density; e.g.,  if it is different from a distinguished choice, say a Riemannian volume.

If $M$ is not compact, $hH^s(M;E)$ and $\langle u,v\rangle_{hH^s}$ depend on $h$ and the chosen definitions of $H^s(M;E)$ and $\langle u,v\rangle_s$ (\Cref{ss: Sobolev sps}). We also get the weighted spaces $hH^s_{\text{\rm c/loc}}(M;E)$, \index{$hH^s_{\text{\rm c/loc}}(M;E)$} and the weighted Banach space $hL^\infty(M;E)$ \index{$hL^\infty(M;E)$} with the norm $\|u\|_{hL^\infty}=\|h^{-1}u\|_{L^\infty}$. There is a continuous injection $hL^\infty(M;E)\subset hL^2_{\text{\rm loc}}(M;E)$.

\subsection{Topological complexes}\label{ss: top complexes}

Recall that a complex $(C,d)$ (over $\C$) consists of a ($\Z$-) graded vector space $C=C^\bullet$ and a linear map $d:C\to C$ which is homogeneous of degree $1$ and satisfies $d^2=0$. If moreover $C$ is a TVS and $d$ is continuous, then $(C,d)$ is called a \emph{topological complex}. Then $\ker d$ and $\im d$ are topological graded subspaces, and the cohomology $H^\bullet(C,d)=\ker d/\im d$ becomes a graded TVS. Its maximal Hausdorff quotient, $\bar H^\bullet(C,d):=H^\bullet(C,d)/\overline 0\equiv\ker d/\overline{\im d}$, is called the \emph{reduced cohomology}. Let $[u]\in H^\bullet(C,d)$ and $\overline{[u]}\in\bar H^\bullet(C,d)$ denote the elements defined by any $u\in\ker d$. If $C$ is a LCS, then $H^\bullet(C,d)$ and $\bar H^\bullet(C,d)$ are also LCSs because this property is inherited by subspaces and quotients \cite[Section~II.4]{Schaefer1971}. We may use the notation $Z=ZC=\ker d$, $B=BC=\im d$ and $\bar B=\bar BC=\overline{\im d}$. 

We always assume $C$ has finitely many nonzero homogeneous components, say $C=C^0\oplus\dots\oplus C^N$. So $d$ is given by a finite sequence of \emph{length} $N$,
\[
C^0 \xrightarrow{d_0} C^1 \xrightarrow{d_1} \cdots \xrightarrow{d_{N-1}} C^N\;.
\]
Negative or decreasing degrees may be also considered without any essential change. Continuous homomorphisms between topological complexes induce continuous linear maps between the corresponding cohomologies and reduced cohomologies. (Usually, the term \emph{chain/cochain complex} is used for decreasing/increasing degrees, and \emph{chain/cochain maps} for the corresponding homomorphisms, but we will ignore that difference.)

The \emph{transpose} of $(C,d)$ is the topological complex $(C',d^\trans)$, graded by $(C')^r=(C^r)'$ ($r=0,\dots,N$). 
For any $[f]\in H^\bullet(C',d^\trans)$, we have $fd=d^\trans(f)=0$, and therefore $f$ induces an element of $H^\bullet(C,d)'$. This defines a canonical continuous linear map $H^\bullet(C',d^\trans)\to H^\bullet(C,d)'$.

\begin{prop}\label{p: H(C' d^t) -> H(C d)'}
The canonical map $H^\bullet(C',d^\trans)\to H^\bullet(C,d)'$ is:
\begin{enumerate}[{\rm(i)}]
\item\label{i: H(C' d^t) -> H(C d)' is surjective} surjective if $C$ is a LCHS; and 
\item\label{i: H(C' d^t) -> H(C d)' is injective} injective if $C$ is a Fr\'echet space and $\im d$ is closed.
\end{enumerate}
\end{prop} 

\begin{proof}
Property~\ref{i: H(C' d^t) -> H(C d)' is surjective} is an easy consequence of the Hahn-Banach theorem \cite[Theorem~II.4.2]{Schaefer1971}.

Property~\ref{i: H(C' d^t) -> H(C d)' is injective} follows easily from the open mapping theorem \cite[Theorem~III.2.1]{Schaefer1971} and the Hahn-Banach theorem.
\end{proof}

\begin{rem}\label{r: H(C' d^t) -> H(C d)'}
Extensions of~\ref{i: H(C' d^t) -> H(C d)' is injective} can be given by more general versions of the open mapping theorem (see e.g.\ \cite{Bourles2014}).
\end{rem}

\subsection{Differential complexes}\label{ss: diff complexes}

Recall that a \emph{differential complex} of \emph{order} at most $m$ is a topological complex of the form $(C^\infty(M;E),d)$, where $E$ is a ($\Z$-) graded vector bundle and $d\in\Diff^m(M;E)$; it will be simply denoted by $(E,d)$. Necessarily, it is of finite length, say $E=E^0\oplus\dots\oplus E^N$  and $d$ is given by the sequence
\[
C^\infty(M;E^0) \xrightarrow{d_0} C^\infty(M;E^1) \xrightarrow{d_1} \cdots \xrightarrow{d_{N-1}} C^\infty(M;E^N)\;.
\]
The compactly supported version $(\Cinftyc(M;E),d)$ may be also considered, as well as the distributional versions $(C^{-\infty}_{{\cdot}/\co}(M;E),d)$. Recall that $(E,d)$ is called an \emph{elliptic complex} of \emph{order} $m$ if moreover the symbol sequence,
\begin{equation}\label{symbol sequence}
0\to E^0_p \xrightarrow{\sigma_m(d_0)(p,\xi)} E^1_p \xrightarrow{\sigma_m(d_1)(p,\xi)} \cdots 
\xrightarrow{\sigma_m(d_{N-1})(p,\xi)} E^N_p\to0\;,
\end{equation}
is exact for all $p\in M$ and $0\ne\xi\in T^*_pM$. If $N=1$, this agrees with the ellipticity of $d_0\in\Diff^m(M;E^0,E^1)$.
 
Equip $E$ with a Hermitian structure so that its homogeneous components are orthogonal, and equip $M$ with a Riemannian metric $g$, inducing a volume density on $M$. Consider the corresponding scalar product on $L^2(M;E)$. Then the formal adjoint $\delta=d^*$ also defines a differential complex, giving rise to symmetric differential operators $D=d+\delta$ and $\Delta=D^2=d\delta+\delta d$. The ellipticity of the differential complex $d$ is equivalent to the ellipticity of the differential complex $\delta$, and it is also equivalent to the ellipticity of the differential operator $D$ (or $\Delta$).

In the rest of \Cref{ss: diff complexes}, suppose $M$ is closed and $d$ is elliptic. Then $D$ and $\Delta$ have a discrete spectrum. Moreover, we have the following Hodge-type decomposition, and associated equalities and isomorphism:
\begin{equation}\label{Hodge}
\left\{
\begin{array}{c}
C^\infty(M;E)=\ker\Delta\oplus\im\delta\oplus\im d\;,\\[4pt]
\im\delta\oplus\im d=\im D=\im\Delta\;,\\[4pt]
\ker d\cap\ker\delta=\ker D=\ker\Delta\cong H^\bullet(C^\infty(M;E),d)\;.
\end{array}
\right.
\end{equation}
Writing $C=C^\infty(M;E)$, it follows from~\eqref{Hodge} that $d:\im\delta\to\im d$ and $\delta:\im d\to\im\delta$ are TVS-isomorphisms.

Consider also the operators $d$, $\delta$, $D$ and $\Delta$ on $C^{-\infty}(M;E)$ (\Cref{ss: diff ops}). Then $(C^{-\infty}(M;E),d)$ is another differential complex, and the analogue of~\eqref{Hodge} is satisfied with $C^{-\infty}(M;E)$. By ellipticity and since $M$ is compact, $\Delta$ has the same kernel in $C^\infty(M;E)$ and in $C^{-\infty}(M;E)$, obtaining a canonical isomorphism $H^\bullet(C^\infty(M;E),d)\cong H^\bullet(C^{-\infty}(M;E),d)$.

\section{Conormal distributions}\label{s: conormal distribs}

The space of conormal distributions plays a very important role in our work. We mainly follow  \cite{KohnNirenberg1965,Hormander1971}, \cite[Section~18.2]{Hormander1985-III}, \cite[Chapters~3--5]{Simanca1990}, \cite[Chapters~4 and~6]{Melrose1996}, \cite[Chapters~3 and~9]{MelroseUhlmann2008}, which are oriented to the role they play in pseudodifferential operators and generalizations of those operators. The study of its natural topology was begun in \cite[Chapters~4 and~6]{Melrose1996} and continued in \cite{AlvKordyLeichtnam-conormal}.

For the sake of simplicity, we consider the case of the trivial line bundle first. But all definitions, properties and notation have obvious extensions for arbitrary vector bundles, like in \Cref{ss: diff ops,ss: pseudodiff ops}, either by using local trivializations, or by taking $C^\infty(M)$-tensor products with spaces of smooth sections. When needed, the case of arbitrary vector bundles will be used without further comment.

\subsection{Differential operators tangent to a submanifold}\label{ss: Diff(M L)}

Let $L$ be a regular submanifold of $M$ of codimension $n'$ and dimension $n''$, which is a closed subset. Let $\fX(M,L)\subset\fX(M)$ \index{$\fX(M,L)$} be the Lie subalgebra and $C^\infty(M)$-submodule of vector fields tangent to $L$. Using $\fX(M,L)$ instead of $\fX(M)$, we can define the filtered subalgebra and $C^\infty(M)$-submodule $\Diff(M,L)\subset\Diff(M)$ \index{$\Diff(M,L)$} like in \Cref{ss: diff ops}. We have
\begin{equation}\label{A in Diff(M L) => A^t in Diff(M L Omega)}
A\in\Diff(M,L)\Rightarrow A^\trans\in\Diff(M,L;\Omega)\;.
\end{equation}
By the conditions on $L$, every $\Diff^m(M,L)$ ($m\in\N_0$) is locally finitely $C^\infty(M)$-generated, and therefore $\Diff(M,L)$ is countably $C^\infty(M)$-generated. The surjective restriction map $\fX(M,L)\to\fX(L)$, $X\mapsto X|_L$, induces a surjective linear restriction map of filtered algebras and $C^\infty(M)$-modules,
\begin{equation}\label{Diff(M L) -> Diff(L)}
\Diff(M,L)\to\Diff(L)\;,\quad A\mapsto A|_L\;.
\end{equation}

Let $(U,x)$ be a chart of $M$ adapted to $L$; i.e., it is a diffeomorphism
\begin{gather*}
x=(x^1,\dots,x^n)\equiv(x',x''):U\to U'\times U''\;,\\
x'=(x'^1,\dots,x'^{n'})\;,\quad x''=(x''^{1},\dots,x''^{n''})\;,\quad L_0:=L\cap U=\{x'=0\}\;,
\end{gather*}
for some open $U'\subset\R^{n'}$ and $U''\subset\R^{n''}$. If $L$ is of codimension one, then we will use the notation $(x,y)$ instead of $(x',x'')$. For every $m\in\N_0$, $\Diff^m(U,L_0)$ is $C^\infty(U)$-spanned by the operators $x^{\prime I}\partial_{x'}^J\partial_{x''}^K$ with $|J|+|K|\le m$ and $|I|=|J|$; we may use the generators $\partial_{x'}^J\partial_{x''}^K x^{\prime I}$ as well, with the same conditions on the multi-indices.

\subsection{Conormal distributions when $M$ is compact}\label{ss: conormal - Sobolev order - compact} 

Suppose $M$ is compact. Then the space of \emph{conormal distributions} \index{conormal distribution} at $L$ of \emph{Sobolev order} at most $s\in\R$ is the LCS and $C^\infty(M)$-module \index{$I^{(s)}_{{\cdot}/\co}(M,L)$} \index{$I_{{\cdot}/\co}(M,L)$}
\begin{equation}\label{I^(s)(M,L)}
I^{(s)}(M,L)=\{\,u\in C^{-\infty}(M)\mid\Diff(M,L)\,u\subset H^s(M)\,\}\;,
\end{equation}
with the topology like in~\eqref{Z = u in bigcup_A in AA dom A | AA cdot u subset Y}. This is a totally reflexive Fr\'echet space \cite[Proposition~4.1]{AlvKordyLeichtnam-conormal}. We have continuous inclusions
\begin{equation}\label{I^(s)(M L) subset I^(s')(M L)}
I^{(s)}(M,L)\subset I^{(s')}(M,L)\quad(s'<s)\;,
\end{equation}
and consider the LCSs and $C^\infty(M)$-modules \index{$I(M,L)$} \index{$I^{(\infty)}(M,L)$}
\[
I(M,L)=\bigcup_sI^{(s)}(M,L)\;,\quad I^{(\infty)}(M,L)=\bigcap_sI^{(s)}(M,L)\;.
\]
Thus $I(M,L)$ is a Hausdorff LF-space (\Cref{ss: TVS}), and $I^{(\infty)}(M,L)$ is a Fr\'echet space and submodule of $I(M,L)$. The elements of $I(M,L)$ are called \emph{conormal distributions} of $M$ at $L$ (or of $(M,L)$). The spaces $I^{(s)}(M,L)$ form what is called the \emph{Sobolev-order filtration} of $I(M,L)$, or the \emph{Sobolev-order inductive spectrum} defining $I(M,L)$. From~\eqref{I^(s)(M,L)}, it follows that there are canonical continuous inclusions,
\begin{equation}\label{C^infty(M) subset I^(infty)(M L)}
C^\infty(M)\subset I^{(\infty)}(M,L)\;,\quad I(M,L)\subset C^{-\infty}(M)\;.
\end{equation}
Indeed, $C^\infty(M)$ is dense in $I(M,L)$ \cite[Eq.~(6.2.12)]{Melrose1996}, \cite[Corollary~4.6]{AlvKordyLeichtnam-conormal}.

$I(M,L)$ is barreled, ultrabornological, webbed, acyclic and a Montel space, and therefore complete, boundedly/compactly/sequently retractive and reflexive \cite[Corollaries~4.2 and~4.7]{AlvKordyLeichtnam-conormal}  (\Cref{ss: TVS}).

\subsection{Filtration of $I(M,L)$ by the symbol order when $M$ is compact}
\label{ss: conormal - symbol order - compact}

Take a chart of $M$ adapted to $L$, $(U,x=(x',x''))$, like in \Cref{ss: Diff(M L)}. We use the identity $U''\times\R^{n'}\equiv N^*U''$, and the symbol spaces $S^m(U''\times\R^{n'})\equiv S^m(N^*U'')$ (\Cref{ss: symbols}). The following holds true for $s,\bar m\in\R$ \cite[Theorem~18.2.8]{Hormander1985-III}, \cite[Proposition~6.1.1]{Melrose1996}, \cite[Lemma~9.33]{MelroseUhlmann2008}, \cite[Remark~4.4]{AlvKordyLeichtnam-conormal}:
\begin{itemize}
\item If $s<-\bar m-n'/2$, then the map $\Cinftycv(N^*U'')\to C^\infty(U)$, $a\mapsto u$, given by
\[
u(x)=(2\pi)^{-n'}\int_{\R^{n'}}e^{i\langle x',\xi\rangle}a(x'',\xi)\,d\xi\;,
\]
has a continuous extension $S^{\bar m}(N^*U'')\to I^{(s)}(U,L_0)$.
\item If $\bar m>-s-n'/2$, then the map $\Cinftyc(U)\to C^\infty(N^*U'')$, $u\mapsto a$, given by
\[
a(x'',\xi)=\int_{\R^{n'}}e^{-i\langle x',\xi\rangle}u(x',x'')\,dx'\;,
\]
induces a continuous linear map $I_\co^{(s)}(U,L_0)\to S^{\bar m}(N^*U'')$.
\end{itemize}
In what follows, it is convenient to use
\[
a\,|d\xi|\in S^{\bar m}(N^*U'';\Omega N^*U'')\equiv S^{\bar m}(N^*L_0;\Omega N^*L_0)\;.
\]

Assume $M$ is compact. Take a finite cover of $L$ by relatively compact charts $(U_j,x_j)$ of $M$ adapted to $L$, and write $L_j=L\cap U_j$. Let $\{h,f_j\}$ be a $C^\infty$ partition of unity of $M$ subordinated to the open covering $\{M\setminus L,U_j\}$. Then $I(M,L)$ consists of the distributions $u\in C^{-\infty}(M)$ such that $hu\in C^\infty(M\setminus L)$ and $f_ju\in I_\co(U_j,L_j)$ for all $j$. Every $f_ju$ is given by some $a_j\in S^\infty(N^*L_j;\Omega N^*L_j)$ as above. For
\begin{equation}\label{bar m}
\bar m=m+n/4-n'/2\;,
\end{equation}
the condition $a_j\in S^{\bar m}(N^*L_j;\Omega N^*L_j)$ describes the elements $u$ of a $C^\infty(M)$-submodule $I^m(M,L)\subset I(M,L)$, \index{$I^m(M,L)$} which is independent of the choices involved \cite[Proposition~9.33]{MelroseUhlmann2008} (see also \cite[Definition~6.2.19]{Melrose1996} and \cite[Definition~4.3.9]{Simanca1990}). Moreover, applying the versions of semi-norms~\eqref{| u |_K C^k} on $C^\infty(M\setminus L)$ to $hu$ and the versions of semi-norms~\eqref{| a |_K I J m} on $S^{\bar m}(N^*L_j;\Omega N^*L_j)$ to every $a_j$, we get semi-norms on $I^m(M,L)$, which becomes a Fr\'echet space \cite[Sections~6.2 and 6.10]{Melrose1996}.

The version of~\eqref{S^m(T^*M) subset S^m'(T^*M)} for the spaces $S^{\bar m}(N^*L_j;\Omega N^*L_j)$ gives continuous inclusions
\begin{equation}\label{I^m(M L) subset I^m'(M L)}
I^m(M,L)\subset I^{m'}(M,L)\quad(m<m')\;.
\end{equation}
The element $\sigma_m(u)\in S^{(\bar m)}(N^*L;\Omega N^*L)$ represented by $\sum_ja_j\in S^{\bar m}(N^*L;\Omega N^*L)$ is called the \emph{principal symbol} of $u$. This defines the exact sequence
\[
0\to I^{m-1}(M,L)\hookrightarrow I^m(M,L) \xrightarrow{\sigma_m} S^{(\bar m)}(N^*L;\Omega N^*L)\to0\;.
\]
We also get continuous inclusions
\begin{equation}\label{sandwich for I}
I^{(-m-n/4+\epsilon)}(M,L)\subset I^m(M,L)\subset I^{(-m-n/4-\epsilon)}(M,L)\;,
\end{equation}
for all $m\in\R$ and $\epsilon>0$ (cf.\ \cite[Eq.~(6.2.5)]{Melrose1996}, \cite[Eq.~(9.35)]{MelroseUhlmann2008}). So \index{$I^{(\infty)}(M,L)$} \index{$I^{-\infty}(M,L)$}
\[
I(M,L)=\bigcup_mI^m(M,L)\;,\quad I^{(\infty)}(M,L)=I^{-\infty}(M,L):=\bigcap_mI^m(M,L)\;.
\]
The spaces $I^m(M,L)$ form what is called the \emph{symbol-order filtration} of $I(M,L)$, or the  \emph{symbol-order inductive spectrum} defining $I(M,L)$.

\subsection{$I(M,L)$ for non-compact $M$}\label{ss: I(M L) - non-compact}

If $M$ is not assumed to be compact, the spaces and properties of \Cref{ss: conormal - Sobolev order - compact,ss: conormal - symbol order - compact} can be extended as follows \cite[Sections~4.2.2 and~4.3.3]{AlvKordyLeichtnam-conormal}. 

We can similarly define the LCHS $I^{(s)}_{{\cdot}/\co}(M,L)$ by using $C^{-\infty}_{{\cdot}/\co}(M)$ and $H^s_{\text{\rm loc/c}}(M)$. Every $I^{(s)}(M,L)$ is a Fr\'echet space. We can describe $I^{(s)}_\co(M,L)=\bigcup_KI^{(s)}_K(M,L)$ like in~\eqref{Cinftyc(U)}, which is a strict LF-space, and therefore $I_\co(M,L)=\bigcup_sI_\co^{(s)}(M,L)$ is an LF-space; moreover $I_\co(M,L)=\bigcup_KI_K(M,L)$. We also have the LCHS $I^{(\infty)}_\co(M,L)=\bigcap_sI_\co^{(s)}(M,L)$. All of these spaces are modules over $C^\infty(M)$; $I_\co(M,L)$ is a filtered module and $I^{(\infty)}_\co(M,L)$ a submodule. The extension by zero defines a continuous inclusion $I_\co(U,L\cap U)\subset I_\co(M,L)$ for any open $U\subset M$. We also define the space $I^{(\infty)}(M,L)$ like in the compact case, as well as the space $\bigcup_sI^{(s)}(M,L)$, which consists of the conormal distributions with a Sobolev order. But now let (cf.\ \cite[Definition~18.2.6]{Hormander1985-III})
\begin{equation}\label{I(M L) - non-compact M}
I(M,L)=\{\,u\in C^{-\infty}(M)\mid\Cinftyc(M)\,u\subset I_\co(M,L)\,\}\;,
\end{equation}
which is a LCS with the topology like in~\eqref{Z = u in bigcup_A in AA dom A | AA cdot u subset Y}. We have $I(M,L)=\bigcup_sI^{(s)}(M,L)$ if and only if $L$ is compact; thus the spaces $I^{(s)}(M,L)$ form a filtration of $I(M,L)$ just when $L$ is compact. There is an extension of~\eqref{C^infty(M) subset I^(infty)(M L)} for non-compact $M$, taking arbitrary/compact support; in particular, $I_{{\cdot}/\co}(M,L)$ is Hausdorff. The density of the smooth functions with arbitrary/compact support is also true.

The definition of $I^m(M,L)$ can be immediately extended assuming $\{U_j\}$ is locally finite. We can similarly define $I^m_K(M,L)$ for all compact $K\subset M$, and then define $I^m_\co(M,L)$ like in~\eqref{Cinftyc(U)}. The space of conormal distributions with a symbol order is $\bigcup_mI^m(M,L)$, and let $I^{-\infty}_{{\cdot}/\co}(M,L)=\bigcap_mI^m_{{\cdot}/\co}(M,L)$. There are extensions of~\eqref{I^m(M L) subset I^m'(M L)} and~\eqref{sandwich for I}. So $\bigcup_mI^m(M,L)=\bigcup_sI^{(s)}(M,L)$, $I_\co(M,L)=\bigcup_mI^m_\co(M,L)$ and $I^{(\infty)}_{{\cdot}/\co}(M,L)=I^{-\infty}_{{\cdot}/\co}(M,L)$. $\bigcup_mI^m(M,L)$ and $I_{{\cdot}/\co}(M,L)$ are acyclic Montel spaces, and $I(M,L)$ is a Montel space.

If $M$ is the domain of a given smooth submersion, the LCHS $I_\cv(M;E)$ can be defined like $C^{-\infty}_\cv(M;E)$, using $I_\co(M;E)$ instead of $C^{-\infty}_\co(M;E)$. \index{$I_\cv(M;E)$}

\subsection{Pseudodifferential operators vs conormal distributions}\label{ss: Psi^m vs I^m}

Using the diagonal $\Delta\subset M^2$, the Schwartz kernel isomorphism~\eqref{Schwartz kernel theorem} restricts to linear isomorphisms
\[
\Psi^m(M;E,F)\xrightarrow{\cong}I^m(M^2,\Delta;F\boxtimes(E^*\otimes\Omega M))\;,\quad A\mapsto K_A\;,
\]
and a similar one for the whole of $\Psi(M;E,F)$. Via them, $\Psi^m(M;E,F)$ and $\Psi(M;E,F)$ become LCHSs satisfying the properties of the corresponding spaces of conormal distributional sections. In this case, we have $\bar m=m$ in~\eqref{bar m} and $\sigma_m(A)\equiv\sigma_m(K_A)$ for any $A\in\Psi^m(M;E,F)$ \cite{Hormander1965,KohnNirenberg1965}, \cite[Chapter~XVIII]{Hormander1985-III}, \cite[Chapter~6]{Simanca1990}.

\subsection{Dirac sections at submanifolds}\label{ss: Dirac sections}

We have $\Omega NL\otimes\Omega L\equiv\Omega_LM$. The transpose of the restriction map $C^\infty_{\co/{\cdot}}(M;E\otimes\Omega M)\to C^\infty_{\co/{\cdot}}(L;E\otimes\Omega_LM)$ is a continuous inclusion \index{$\delta_L^u$}
\begin{gather}
C^{-\infty}_{{\cdot}/\co}(L;E\otimes\Omega^{-1}NL)\subset C^{-\infty}_{{\cdot}/\co}(M;E)\;,
\label{u in C^-infty_c mapsto delta_L^u}\\
u\mapsto\delta_L^u\;,\quad\langle\delta_L^u,v\rangle=\langle u,v|_L\rangle\;,\quad 
v\in C^\infty_{\co/{\cdot}}(M;E^*\otimes\Omega)\;.\notag
\end{gather}
By restriction of~\eqref{u in C^-infty_c mapsto delta_L^u}, we get a continuous inclusion \cite[p.~310]{GuilleminSternberg1977},
\begin{equation}\label{u in C^infty_c mapsto delta_L^u}
C^\infty_{{\cdot}/\co}(L;E\otimes\Omega^{-1}NL)\subset C^{-\infty}_{{\cdot}/\co}(M;E)\;;
\end{equation}
in this case, we can write $\langle\delta_L^u,v\rangle=\int_Lu\,v|_L$. This is the subspace of \emph{$\delta$-sections} \index{$\delta$-section} or \emph{Dirac sections} at $L$. Actually, the inclusion~\eqref{u in C^infty_c mapsto delta_L^u} induces a continuous injection \cite[Corollary~4.9]{AlvKordyLeichtnam-conormal}
\begin{equation}\label{C^infty_./c(L E otimes Omega^-1 NL) subset H^s_loc/c(M E)}
C^\infty_{{\cdot}/\co}(L;E\otimes\Omega^{-1}NL)\subset H^s_{\text{\rm loc/c}}(M;E)\quad(s<-n'/2)\;,
\end{equation}
with
\[
C^\infty_{{\cdot}/\co}(L;E\otimes\Omega^{-1}NL)\cap H^{-n'/2}_{\text{\rm loc/c}}(M;E)=0\;.
\]

For instance, for any $p\in M$ and $u\in E_p\otimes\Omega_p^{-1}M$, we get $\delta_p^u\in H^s_\co(M;E)$ if $s<-n/2$, with $\langle\delta_p^u,v\rangle=u\cdot v(p)$ for $v\in C^\infty(M;E^*\otimes\Omega)$, obtaining a continuous map
\begin{equation}\label{(p u) mapsto delta_p^u(p)}
M\times C^\infty(M;E\otimes\Omega^{-1})\to H^s_\co(M;E)\;,\quad(p,u)\mapsto\delta_p^{u(p)}\;.
\end{equation}
As a particular case, the Dirac mass at any $p\in\R^n$ is $\delta_p=\delta_p^{1\otimes|dx|^{-1}}\in H^s_\co(\R^n)$. 

The Schwartz kernel of any $A\in L(C^{-\infty}(M;E),C^\infty(M;F))$ has the following description: for all $q\in M$ and $u\in E_q\otimes\Omega_q^{-1}$,
\begin{equation}\label{K_A(cdot q)(u) = A delta_q^u}
K_A(\cdot,q)(u)=A\delta_q^u\;.
\end{equation}

\subsection{Differential operators on conormal distributional sections}\label{ss: diff opers on conormal distribs}

Any $A\in\Diff^k(M;E)$ induces continuous linear maps \cite[Lemma~6.1.1]{Melrose1996}
\begin{equation}\label{A: I^(s)(M L E) -> I^[s-k](M L E)}
A:I^{(s)}_{{\cdot}/\co}(M,L;E)\to I^{(s-k)}_{{\cdot}/\co}(M,L;E)\;,
\end{equation}
which induce a continuous endomorphism $A$ of $I_{{\cdot}/\co}(M,L;E)$. If $A\in\Diff(M,L;E)$, then it clearly induces a continuous endomorphism $A$ of every $I^{(s)}_{{\cdot}/\co}(M,L;E)$.

By~\eqref{u in C^-infty_c mapsto delta_L^u}, for $A\in\Diff(M,L;E)$ and $u\in C^\infty_{{\cdot}/\co}(L;E\otimes\Omega^{-1}NL)$, we have \cite[Eq.~(4.17)]{AlvKordyLeichtnam-conormal}
\begin{equation}\label{A delta_L^u}
A\delta_L^u=\delta_L^{A'u}\;,\quad A'=((A^\trans)|_L)^\trans\in\Diff(L;E\otimes\Omega^{-1}NL)\;,
\end{equation}
where $A^\trans\in\Diff(M,L;E^*\otimes\Omega)$ and $(A^\trans)|_L\in\Diff(L,E^*\otimes\Omega_LM)$ using the vector bundle versions of~\eqref{A in Diff(M L) => A^t in Diff(M L Omega)} and~\eqref{Diff(M L) -> Diff(L)}. By~\eqref{A delta_L^u}, $\Diff(M,L;E)$ preserves the subspace of Dirac sections given by~\eqref{u in C^infty_c mapsto delta_L^u}. Thus~\eqref{C^infty_./c(L E otimes Omega^-1 NL) subset H^s_loc/c(M E)} induces a continuous inclusion
\begin{equation}\label{u mapsto delta_L^u conormal}
C^\infty_{{\cdot}/\co}(L;E\otimes\Omega^{-1}NL)\subset I^{(s)}_{{\cdot}/\co}(M,L;E)\quad(s<-n'/2)\;.
\end{equation}

\subsection{Pull-back of conormal distributions}\label{ss: pull-back of conormal distribs}

If a smooth map $\phi:M'\to M$ is transverse to a regular submanifold $L\subset M$, which is a closed subset, then $L':=\phi^{-1}(L)\subset M'$ is a regular submanifold, which is a closed subset. The trivial-line-bundle version of~\eqref{phi^*: C^infty(M E) -> C^infty(M' phi^*E)} has continuous extensions 
\begin{equation}\label{phi^*: I^m(M L) -> I^m+k/4(M' L')}
\phi^*:I^m(M,L)\to I^{m+k/4}(M',L')\quad(m\in\R)\;,
\end{equation}
where $k=\dim M-\dim M'$  \cite[Theorem~5.3.8]{Simanca1990}, \cite[Proposition~6.6.1]{Melrose1996}. Taking inductive limits and using~\eqref{sandwich for I}, we get a continuous linear map
\begin{equation}\label{phi^*: I(M L) -> I(M' L')}
\phi^*:I(M,L)\to I(M',L')\;.
\end{equation}
If $\phi$ is a submersion, this is a restriction of~\eqref{phi_*: C^-infty(M E) -> C^-infty(M' phi^*E)}. In the case of a vector bundle $E$ over $M$, we get
\begin{equation}\label{phi^*: I(M L E) -> I(M' L' phi^*E)}
\phi^*:I(M,L;E)\to I(M',L';\phi^*E)\;,
\end{equation}
given by the $C^\infty(M)$-tensor product of the map~\eqref{phi^*: I(M L) -> I(M' L')} and the identity map on $C^\infty(M;E)$, using the versions of~\eqref{C^infty(M)-tensor product description of C^pm infty_./c(M E)} and~\eqref{C^infty(M)-tensor product description of C^pm infty_./c(M' phi^*E)} for spaces of conormal distributions (see \Cref{ss: pull-back and push-forward of distrib sections}).

\subsection{Push-forward of conormal distributions}\label{ss: push-forward of conormal distribs}

Let $\phi:M'\to M$ be a smooth submersion, and let $L\subset M$ and $L'\subset M'$ be regular submanifolds, which are closed subsets, such that $\phi(L')\subset L$ and the restriction $\phi:L'\to L$ is also a smooth submersion. Then~\eqref{phi_*: Cinftyc(M' phi^*E otimes Omega_fiber) -> Cinftyc(M E)} and~\eqref{phi_*: C^infty_cv(M' phi^*E otimes Omega_fiber) -> Cinfty(M E)} have continuous extensions
\begin{equation}\label{phi_*: I^m_c(M' L' Omega_fiber) -> I^m+l/2-k/4_c(M L)}
\phi_*:I^m_{\co/\cv}(M',L';\Omega_\fiber)\to I^{m+l/2-k/4}_{\co/{\cdot}}(M,L)\quad(m\in\R)\;,
\end{equation}
where $k=\dim M'-\dim M$ and $l=\dim L'-\dim L$ \cite[Theorem~5.3.6]{Simanca1990}, \cite[Proposition~6.7.2]{Melrose1996}. Taking inductive limits, we get a continuous linear map
\begin{equation}\label{phi_*: I_c/cv(M' L' Omega_fiber) -> I_c/.(M L)}
\phi_*:I_{\co/\cv}(M',L';\Omega_\fiber)\to I_{\co/{\cdot}}(M,L)\;,
\end{equation}
which is a restriction of~\eqref{phi_*: C^-infty_c(M' phi^*E otimes Omega_fiber) -> C^-infty_c(M E)}. In the case of a vector bundle $E$ over $M$, we get
\begin{equation}\label{phi_*: I_c/cv(M' L' phi^*E otimes Omega_fiber) -> I_c/.(M L E)}
\phi_*:I_{\co/\cv}(M',L';\phi^*E\otimes\Omega_\fiber)\to I_{\co/{\cdot}}(M,L;E)\;,
\end{equation}
is given by the $C^\infty(M)$-tensor product of~\eqref{phi_*: I_c/cv(M' L' Omega_fiber) -> I_c/.(M L)} and the identity map on $C^\infty(M;E)$, using the obvious versions of~\eqref{C^infty(M)-tensor product description of C^pm infty_./c(M E)} and~\eqref{C^infty(M)-tensor product description of C^pm infty_cv/c(M' phi^*E otimes Omega_fiber)} for spaces of conormal distributions (see \Cref{ss: pull-back and push-forward of distrib sections}). The map~\eqref{phi_*: I_c/cv(M' L' Lambda) to I_c/.(M L Lambda)} is also a restriction of~\eqref{phi_*: C^-infty_c/cv(M' Lambda) to C^-infty_c/.(M Lambda)}.

\section{Dual-conormal distributions}\label{s: dual-conormal distribs}

The dual space $I(M,L;E)'$ \cite[Chapter~6]{Melrose1996} also plays an important role in our work. Again, the case of $I(M,L)'$ is considered first; its extension for any vector bundle $E$ can be made like in \Cref{s: conormal distribs}, and will be considered without further comment.

\subsection{Dual-conormal distributions when $M$ is compact}\label{ss: dual-conormal distribs - compact} 

Consider the notation of \Cref{ss: conormal - Sobolev order - compact,ss: conormal - symbol order - compact}, where $M$ is assumed to be compact. The space of \emph{dual-conormal distributions} \index{dual-conormal distribution} of $M$ at $L$ (or of $(M,L)$) is \cite[Chapter~6]{Melrose1996} \index{$I'(M,L)$}
\begin{equation}\label{I'(M L) = I(M L Omega)'}
I'(M,L)=I(M,L;\Omega)'\;.
\end{equation}
Let also
\begin{equation}\label{I^prime (s)(M L) = I^(-s)(M L Omega)'}
I^{\prime\,(s)}(M,L)=I^{(-s)}(M,L;\Omega)'\;,\quad I^{\prime\,m}(M,L)=I^{-m}(M,L;\Omega)'\;.
\end{equation}
$I'(M,L)$ is a complete Montel space, and every $I^{\prime\,(s)}(M,L)$ is bornological and barreled \cite[Corollaries~5.1 and~5.2]{AlvKordyLeichtnam-conormal}. 

Transposing the versions of~\eqref{I^(s)(M L) subset I^(s')(M L)} and~\eqref{I^m(M L) subset I^m'(M L)} with $\Omega M$, we get continuous linear restriction maps, for $s'<s$ and $m<m'$,
\begin{equation}\label{I^prime m(M L) to I^prime m'(M L)}
I^{\prime\,(s')}(M,L)\leftarrow I^{\prime\,(s)}(M,L)\;,\quad I^{\prime\,m'}(M,L)\leftarrow I^{\prime\,m}(M,L)\;.
\end{equation}
These maps form projective spectra (the Sobolev-order and symbol-order spectra), giving rise to $\varprojlim I^{\prime\,(s)}(M,L)$ as $s\uparrow+\infty$ and $\varprojlim I^{\prime\,m}(M,L)$ as $m\downarrow-\infty$. Similarly, from~\eqref{C^infty(M) subset I^(infty)(M L)}, we get continuous inclusions,
\begin{equation}\label{C^-infty(M) supset I'(M L) supset C^infty(M)}
C^{-\infty}(M)\supset I'(M,L)\supset C^\infty(M)\;,
\end{equation}
and~\eqref{sandwich for I} gives rise to continuous linear restriction maps
\begin{equation}\label{sandwich for I'}
I^{\prime\,(-m+n/4-\epsilon)}(M,L)\leftarrow I^{\prime\,m}(M,L)\leftarrow I^{\prime\,(-m+n/4+\epsilon)}(M,L)\;,
\end{equation}
for all $m\in\R$ and $\epsilon>0$. We also have \cite[Corollary 5.3]{AlvKordyLeichtnam-conormal}
\begin{equation}\label{I'(M L) equiv projlim I^prime (s)(M L) equiv projlim I^prime m(M L)}
I'(M,L)\equiv\varprojlim I^{\prime\,(s)}(M,L)\equiv\varprojlim I^{\prime\,m}(M,L)\;,
\end{equation}
as $s\uparrow+\infty$ and $m\downarrow-\infty$, where the last equality follows from~\eqref{sandwich for I'}.

The left-hand-side maps of~\eqref{I^prime m(M L) to I^prime m'(M L)} have dense images, which follows from consequences of the Hahn-Banach theorem \cite[Theorems~7.7.5 and~7.7.7~(c)]{NariciBeckenstein2011}, using that their transposes are the analogs of the inclusions~\eqref{I^(s)(M L) subset I^(s')(M L)} with $\Omega M$ by the reflexivity of the spaces $I^{(s)}(M,L;\Omega)$ (\Cref{ss: conormal - Sobolev order - compact}). Similarly, the inclusions~\eqref{C^-infty(M) supset I'(M L) supset C^infty(M)} are dense.

\subsection{Dual-conormal distributions when $M$ is non-compact}\label{ss: dual-conormal distribs - non-compact}

If $M$ is not supposed to be compact, the above concepts and properties can be extended as follows. We can similarly define the space $I'_K(M,L)$ of dual-conormal distributions supported in any compact $K\subset M$. Then define the LCHSs, $I'_\co(M,L)=\bigcup_KI'_K(M,L)$ like in~\eqref{Cinftyc(U)}, and $I'(M,L)$ like in~\eqref{I(M L) - non-compact M} using $I'_\co(M,L)$ instead of $I_\co(M,L)$. These spaces satisfy a version of~\eqref{I'(M L) = I(M L Omega)'}, interchanging arbitrary/compact support like in~\eqref{C^-infty_cdot/c(M;E)}. $I'(M,L)$ is a complete Montel space, and~\eqref{C^-infty(M) supset I'(M L) supset C^infty(M)} is also true. Similarly, we can define the spaces $I^{\prime\,(s)}_{{\cdot}/\co}(M,L)$ and $I^{\prime\,m}_{{\cdot}/\co}(M,L)$, \index{$I^{\prime\,m}_{{\cdot}/\co}(M,L)$} which satisfy a version of~\eqref{I^prime (s)(M L) = I^(-s)(M L Omega)'} interchanging the support condition. Moreover~\eqref{sandwich for I'} and~\eqref{I'(M L) equiv projlim I^prime (s)(M L) equiv projlim I^prime m(M L)} have obvious extensions. 

If $M$ is the domain of a given smooth submersion, the LCHS $I'_\cv(M;E)$ can be defined like $C^{-\infty}_\cv(M;E)$, using $I'_\co(M;E)$ instead of $C^{-\infty}_\co(M;E)$. \index{$I_\cv(M;E)$}

\subsection{Conormal distributions vs dual-conormal distributions}\label{ss: conormal distribs vs dual-conormal distribs}

Assume $M$ is compact. Then \cite[Theorem~8.11]{AlvKordyLeichtnam-conormal}
\[
I(M,L)\cap I'(M,L)=C^\infty(M)\;.
\]

\subsection{Differential operators on dual-conormal distributional sections}
\label{ss: diff opers on dual-conormal distribs}

For any $A\in\Diff(M;E)$, the transpose of $A^\trans$ on $I_{\co/{\cdot}}(M,L;E^*\otimes\Omega)$ (\Cref{ss: diff opers on conormal distribs}) is a continuous endomorphism $A$ of $I'_{{\cdot}/\co}(M,L;E)$, which is a continuous extension of $A$ on $C^\infty(M;E)$, and a restriction of $A$ on $C^{-\infty}(M;E)$ (\Cref{ss: diff ops}). By~\eqref{A: I^(s)(M L E) -> I^[s-k](M L E)}, if $A\in\Diff^m(M;E)$, we get induced continuous linear maps
\begin{equation}\label{A: I^prime [s](M L E) -> I^prime (s-m)(M L E)}
A:I^{\prime\,(s)}_{{\cdot}/\co}(M,L;E)\to I^{\prime\,(s-m)}_{{\cdot}/\co}(M,L;E)\;,
\end{equation}
If $A\in\Diff(M,L;E)$, the transpose of $A^\trans$ of $I^{(-s)}_{\co/{\cdot}}(M,L;E^*\otimes\Omega)$ is a continuous endomorphism $A$ of $I^{\prime\,(s)}_{{\cdot}/\co}(M,L;E)$.

\subsection{Pull-back of dual-conormal distributions}
\label{ss: pull-back of dual-conormal distributions}

With the notation and conditions of \Cref{ss: push-forward of conormal distribs}, transposing the compactly supported cases  of~\eqref{phi_*: I^m_c(M' L' Omega_fiber) -> I^m+l/2-k/4_c(M L)} and~\eqref{phi_*: I_c/cv(M' L' Omega_fiber) -> I_c/.(M L)} with $\Omega M$, we get continuous linear maps 
\begin{gather}
\phi^*:I^{\prime\,m}(M,L)\to I^{\prime\,m+l/2-k/4}(M',L')\quad(m\in\R)\;,\notag\\
\phi^*:I'(M,L)\to I'(M',L')\;.\label{phi^*: I'(M L) -> I'(M' L')}
\end{gather}
 In the case of a vector bundle $E$ over $M$, like in~\eqref{phi^*: I(M L E) -> I(M' L' phi^*E)}, we get
\begin{equation}\label{phi^*: I'(M L E) -> I'(M' L' phi^*E)}
\phi^*:I'(M,L;E)\to I'(M',L';\phi^*E)\;.
\end{equation}
The map~\eqref{phi^*: I'(M L E) -> I'(M' L' phi^*E)} is an extension of~\eqref{phi^*: C^infty(M E) -> C^infty(M' phi^*E)} and a restriction of~\eqref{phi_*: C^-infty(M E) -> C^-infty(M' phi^*E)}.

\subsection{Push-forward of dual-conormal distributions}
\label{ss: push-forward of dual-conormal distributions}

With the notation and conditions of \Cref{ss: pull-back of conormal distribs}, suppose $\phi$ is a submersion. Transposing the versions of~\eqref{phi^*: I^m(M L) -> I^m+k/4(M' L')} and~\eqref{phi^*: I(M L) -> I(M' L')} with $\Omega M$, and using an analog of~\eqref{phi_*u}, we get continuous linear maps,
\begin{gather}
\phi_*:I^{\prime\,m}_{\co/\cv}(M',L'\otimes\Omega_\fiber)\to I^{\prime\,m-k/4}_{\co/{\cdot}}(M,L)\quad(m\in\R)\;, \notag\\
\phi_*:I'_{\co/\cv}(M',L';\Omega_\fiber)\to I'_{\co/{\cdot}}(M,L)\;.
\label{phi_*: I'_c/cv(M' L' Omega_fiber) -> I'_c/.(M L)}
\end{gather}
In the case of a vector bundle $E$ over $M$, like in~\eqref{phi_*: I_c/cv(M' L' phi^*E otimes Omega_fiber) -> I_c/.(M L E)}, we get
\begin{equation}\label{phi_*: I'_c/cv(M' L' phi^*E otimes Omega_fiber) -> I'_c/.(M L E)}
\phi_*:I'_{\co/\cv}(M',L';\phi^*E\otimes\Omega_\fiber)\to I'_{\co/{\cdot}}(M,L;E)\;.
\end{equation}
The map~\eqref{phi_*: I'_c/cv(M' L' Omega_fiber) -> I'_c/.(M L)} is an extension of~\eqref{phi_*: Cinftyc(M' phi^*E otimes Omega_fiber) -> Cinftyc(M E)} and a restriction of~\eqref{phi_*: C^-infty_c(M' phi^*E otimes Omega_fiber) -> C^-infty_c(M E)}.

\section{Bounded geometry}\label{s: bd geom}

\subsection{Basic notation}\label{ss: bd geom - basic notation}

The concepts recalled here become relevant when $M$ is not compact. Equip $M$ with a Riemannian metric $g$, and let $\nabla$ denote its Levi-Civita connection, $R$ its curvature tensor, and $\inj_M\ge0$ its injectivity radius (the infimum of the injectivity radius at all points). If $M$ is connected, we have an induced distance function $d$. If $M$ is not connected, we can also define $d$ taking $d(p,q)=\infty$ if $p$ and $q$ belong to different connected components. Observe that $M$ is complete if $\inj_M>0$. For $r>0$, $p\in M$ and $S\subset M$, let $B(p,r)$ and $\overline B(p,r)$ denote the open and closed $r$-balls centered at $p$, and $\Pen(S,r)$ and $\overline{\Pen}(S,r)$ denote the open and closed $r$-penumbras of $S$ (defined by the conditions $d(\cdot,S)<r$ and $d(\cdot,S)\le r$, respectively). We may add the subscript ``$M$'' to this notation if needed, or a subscript ``$a$'' if we are referring to a family of Riemannian manifolds $M_a$.

\subsection{Manifolds and vector bundles of bounded geometry}\label{ss: mfds and vector bdls of bd geom}

Recall that $M$ is said to be of \emph{bounded geometry} \index{bounded geometry} if $\inj_M>0$ and $\sup|\nabla^mR|<\infty$ for every $m\in\N_0$. This concept has the following chart description. 

\begin{thm}[Eichhorn \cite{Eichhorn1991}; see also \cite{Roe1988I,Schick1996,Schick2001}]\label{t: mfd of bd geom}
$M$ is of bounded geometry if and only if, for some open ball $B\subset\R^n$ centered at $0$, there are normal coordinates $y_p:V_p\to B$ at every $p\in M$ such that the corresponding Christoffel symbols $\Gamma^i_{jk}$, as a family of functions on $B$ parametrized by $i$, $j$, $k$ and $p$, lie in a bounded set of the Fr\'echet space $C^\infty(B)$. This equivalence holds as well replacing the Cristoffel symbols with the metric coefficients $g_{ij}$.
\end{thm}

\begin{rem}\label{r: equi-bounded geometry}
Any non-connected Riemannian manifold of bounded geometry can be considered as a family of Riemannian manifolds (the connected components), which are of \emph{equi-bounded geometry} in the sense that they satisfy the condition of bounded geometry with the same bounds.
\end{rem}

\begin{ex}\label{ex: mfds of bd geom}
Typical examples of manifolds of bounded geometry are Lie groups with left invariant metrics, covering spaces of closed Riemannian manifolds and leaves of foliations on closed manifolds.
\end{ex}

From now on in this section, assume $M$ is of bounded geometry and consider the charts $y_p:V_p\to B$ given by \Cref{t: mfd of bd geom}. The radius of $B$ will be denoted by $r_0$. 

\begin{prop}[{Schick \cite[Theorem~A.22]{Schick1996}, \cite[Proposition~3.3]{Schick2001}}]\label{p: |partial_I(y_q y_p^-1)|}
For every multi-index $\alpha$, the function $|\partial_I(y_qy_p^{-1})|$ is bounded on $y_p(V_p\cap V_q)$, uniformly on $p,q\in M$.
\end{prop}

\begin{prop}[{Shubin \cite[Appendix~A1.1, Lemma~1.2]{Shubin1992}}]\label{p: p_k}
For any $0<2r\le r_0$, there is a subset $\{p_k\}\subset M$ and some $N\in\N$ such that the balls $B(p_k,r)$ cover $M$, and every intersection of $N+1$ sets $B(p_k,2r)$ is empty.
\end{prop}

A vector bundle $E$ of rank $l$ over $M$ is said to be of \emph{bounded geometry} when it is equipped with a family of local trivializations over the charts $(V_p,y_p)$, for small enough $r_0$, with corresponding defining cocycle $a_{pq}:V_p\cap V_q\to\GL(\C,l)\subset\C^{l\times l}$, such that, for all multi-index $\alpha$, the function $|\partial_I(a_{pq}y_p^{-1})|$ is bounded on $y_p(V_p\cap V_q)$, uniformly on $p,q\in M$. When referring to local trivializations of a vector bundle of bounded geometry, we always mean that they satisfy this condition. If the corresponding defining cocycle is valued in $\operatorname{U}(l)$, then $E$ is said to be of \emph{bounded geometry} as a Hermitian vector bundle. Euclidean vector bundles of bounded geometry are similarly defined.

\begin{ex}\label{ex: bundles of bd geom}
The vector bundle $E$ associated to the principal $\operatorname{O}(n)$-bundle $P$ of orthonormal frames of $M$ and any unitary representation of $\operatorname{O}(n)$ is of bounded geometry in a canonical way. In particular, this applies to $T_\C M$ and $\Lambda M$. If the representation is unitary, then bounded geometry holds as a Hermitian vector bundle. The same is true if we use any reduction $Q$ of $P$ with structural group $H\subset\operatorname{O}(n)$ and any unitary representation of $H$. 
\end{ex}

\begin{ex}\label{ex: operations of bundles of bd geom}
Bounded geometry is preserved by operations of vector bundles induced by operations of vector spaces, like dual vector bundles, direct sums, tensor products, exterior products, densities, etc.
\end{ex}

\begin{ex}\label{ex: widetilde E}
Let $E$ be a vector bundle $E$ over a closed Riemannian manifold $M$, and let $\widetilde M$ be a covering of $M$. Then the lift $\widetilde E$ of $E$ to $\widetilde M$ is of bounded geometry in a canonical way.
\end{ex}

\subsection{Uniform spaces}\label{ss: uniform sps}

For every $m\in\N_0$, a function $u\in C^m(M)$ is said to be \emph{$C^m$-uniformy bounded} if there is some $C_m\ge0$ with $|\nabla^{m'}u|\le C_m$ on $M$ for all $m'\le m$. These functions form the \emph{uniform $C^m$ space} $C_{\text{\rm ub}}^m(M)$, which is a Banach space with the norm $\|{\cdot}\|_{C^m_{\text{\rm ub}}}$ defined by the best constant $C_m$. As usual, we write $C_{\text{\rm ub}}(M)=C_{\text{\rm ub}}^0(M)=C(M)\cap L^\infty(M)$. Equivalently, we may take the norm $\|{\cdot}\|'_{C^m_{\text{\rm ub}}}$ defined by the best constant $C'_m\ge0$ such that $|\partial_I(uy_p^{-1})|\le C'_m$ on $B$ for all $p\in M$ and $|I|\le m$; in fact, it is enough to consider any subset of points $p$ so that $\{V_p\}$ covers $M$ \cite[Theorem~A.22]{Schick1996}, \cite[Proposition~3.3]{Schick2001}. The \emph{uniform $C^\infty$ space} is the Fr\'echet space $\Cinftyub(M)=\bigcap_mC_{\text{\rm ub}}^m(M)$, with the semi-norms $\|{\cdot}\|_{C^m_{\text{\rm ub}}}$ or $\|{\cdot}\|'_{C^m_{\text{\rm ub}}}$. It consists of the functions $u\in C^\infty(M)$ such that all functions $uy_p^{-1}$ lie in a bounded set of $C^\infty(B)$. 

The same definitions apply to functions with values in $\C^l$. Moreover the definition of uniform spaces with covariant derivative can be also considered for non-complete Riemannian manifolds.

\begin{prop}[{Shubin \cite[Appendix~A1.1, Lemma~1.3]{Shubin1992}; see also \cite[Proposition~3.2]{Schick2001}}]\label{p: f_k}
Given $r$, $\{p_k\}$ and $N$ like in \Cref{p: p_k} there is a partition of unity $\{f_k\}$ subordinated to the open covering $\{B(p_k,r)\}$, which is bounded in the Fr\'echet space $\Cinftyub(M)$.
\end{prop}

For a Hermitian vector bundle $E$ of bounded geometry over $M$, the \emph{uniform $C^m$ space} $C_{\text{\rm ub}}^m(M;E)$ can be defined by introducing $\|{\cdot}\|'_{C^m_{\text{\rm ub}}}$ like the case of functions, using local trivializations of $E$ to consider every $uy_p^{-1}$ in $C^m(B,\C^l)$ for all $u\in C^m(M;E)$. Then, as above, we get the \emph{uniform $C^\infty$ space} $\Cinftyub(M;E)$, \index{$\Cinftyub(M;E)$} which consists of the sections $u\in C^\infty(M;E)$ such that all functions $uy_p^{-1}$ define a bounded set of $\Cinftyub(B;\C^l)$. In particular, $\fXub(M):=\Cinftyub(M;TM)$ is a $\Cinftyub(M)$-submodule and Lie subalgebra of $\fX(M)$.

The subset $\fXcom(M)\subset\fX(M)$ of complete vector fields satisfies $\fXub(M)\subset\fXcom(M)$ \cite[Proposition~3.8]{AlvKordyLeichtnam2020}.

\subsection{Differential operators of bounded geometry}\label{ss: diff ops of bd geom}

Like in \Cref{ss: diff ops}, by using $\fXub(M)$ and $\Cinftyub(M)$ instead of $\fX(M)$ and $C^\infty(M)$, we get the filtered subalgebra and $\Cinftyub(M)$-submodule $\Diffub(M)\subset\Diff(M)$ of differential operators of \emph{bounded geometry}. Observe that
\begin{equation}\label{C^m_ub(M)}
C^m_{\text{\rm ub}}(M)=\{\,u\in C^m(M)\mid\Diffub^m(M)\,u\subset L^\infty(M)\,\}\;.
\end{equation}
The concept of $\Diffub(M)$ can be extended to vector bundles of bounded geometry $E$ and $F$ over $M$ by taking the $\Cinftyub(M)$-tensor product with $\Cinftyub(M;F\otimes E^*)$, obtaining the filtered $\Cinftyub(M)$-submodule $\Diffub(M;E,F)\subset\Diff(M;E,F)$ (or $\Diffub(M;E)$ \index{$\Diffub(M;E)$} if $E=F$). Bounded geometry of differential operators is preserved by compositions and by taking transposes, and by taking formal adjoints in the case of Hermitian vector bundles of bounded geometry; in particular, $\Diffub(M;E)$ is a filtered subalgebra of $\Diff(M;E)$. Using local trivializations of $E$ and $F$ over the charts $(V_p,y_p)$, we get a local description of any operator in $\Diffub^m(M;E,F)$ by requiring its local coefficients to define a bounded subset of the Fr\'echet space $C^\infty(B,\C^{l'}\otimes\C^{l*})$, where $l$ and $l'$ are the ranks of $E$ and $F$ (\Cref{ss: diff ops}). If $E$ is a line bundle of bounded geometry, then \cite[Eq.~(2.24)]{AlvKordyLeichtnam-conormal}
\begin{equation}\label{Diffub^m(M E) equiv Diffub^m(M)}
\Diffub^m(M;E)\equiv\Diffub^m(M)\;.
\end{equation}

Let $P_{\text{\rm ub}}(T^*M)\subset P(T^*M)$ be the graded subalgebra generated by $P_{\text{\rm ub}}^{(0)}(T^*M)\equiv \Cinftyub(M)$ and $P_{\text{\rm ub}}^{(1)}(T^*M)\equiv \fXub(M)$, which is also a $\Cinftyub(M)$-submodule. Restricting~\eqref{sigma_m}, we get a short exact sequence with $\sigma_m:\Diffub^m(M)\to P_{\text{\rm ub}}^{(m)}(T^*M)$. By taking the $\Cinftyub(M)$-tensor product with $\Cinftyub(M;F\otimes E^*)$, we get $P_{\text{\rm ub}}^{(m)}(T^*M;F\otimes E^*)$ and a short exact sequence with $\sigma_m:\Diffub^m(M;E,F)\to P_{\text{\rm ub}}^{(m)}(T^*M;F\otimes E^*)$. 

Using the norms $\|{\cdot}\|'_{C^m_{\text{\rm ub}}}$, it easily follows that every $A\in\Diffub^m(M;E,F)$ defines bounded operators $A:C^{m+s}_{\text{\rm ub}}(M;E)\to C^s_{\text{\rm ub}}(M;F)$ ($s\in\N_0$), which induce a continuous linear map $A:\Cinftyub(M;E)\to \Cinftyub(M;F)$.

\begin{ex}\label{ex: connections of bd geom}
In Example~\ref{ex: bundles of bd geom}, the Levi-Civita connection $\nabla$ induces a connection of bounded geometry on $E$, also denoted by $\nabla$. In particular, $\nabla$ itself is of bounded geometry on $TM$, and induces a connection $\nabla$ of bounded geometry on $\Lambda M$. This holds as well for the connection on $E$ induced by any other Riemannian connection of bounded geometry on $TM$. 
\end{ex}

\begin{ex}\label{ex: connections induced by natural operations} Bounded geometry of connections is preserved by taking the induced connections in the operations with vector bundles of bounded geometry indicated in Example~\ref{ex: operations of bundles of bd geom}. 
\end{ex}

Suppose $E$ and $F$ are Hermitian vector bundles of bounded geometry. Then any unitary connection $\nabla$ of bounded geometry on $E$ can be used to define an equivalent norm $\|{\cdot}\|_{C^m_{\text{\rm ub}}}$ on every Banach space $C_{\text{\rm ub}}^m(M;E)$, like in the case of $C_{\text{\rm ub}}^m(M)$. 

It is said that $A\in\Diff^m(M;E,F)$ is \emph{uniformly elliptic} \index{uniformly elliptic} if, given Hermitian metrics of bounded geometry on $E$ and $F$, there is some $C\ge1$ such that, for all $p\in M$ and $\xi\in T^*_pM$,
\begin{equation}\label{uniformly elliptic}
C^{-1}|\xi|^m\le|\sigma_m(A)(p,\xi)|\le C|\xi|^m\;.
\end{equation}
This condition is independent of the choice of the Hermitian metrics of bounded geometry on $E$ and $F$. Any $A\in\Diff^m_{\text{\rm ub}}(M;E,F)$ satisfies the second inequality. 

\begin{ex}\label{ex: widetilde A}
In \Cref{ex: widetilde E}, for any $A\in\Diff^m(M;E)$, its lift $\widetilde A\in\Diff^m(\widetilde M;\widetilde E)$ is of bounded geometry in a canonical way. Moreover $\widetilde A$ is uniformly elliptic if $A$ is elliptic.
\end{ex}


\subsection{Sobolev spaces of manifolds of bounded geometry}\label{ss: Sobolev bd geom}

For any Hermitian vector bundle $E$ of bounded geometry over $M$, any nonnegative symmetric uniformly elliptic $P\in\Diff^2_{\text{\rm ub}}(M;E)$ can be used to define the Sobolev space $H^s(M;E)$ ($s\in\R$) with a scalar product $\langle {\cdot},{\cdot}\rangle_s$ (\Cref{ss: Sobolev sps}). Any choice of $P$ defines the same Hilbertian space $H^s(M;E)$, which is a $\Cinftyub(M)$-module. In particular, $L^2(M;E)$ is the $\Cinftyub(M)$-tensor product of $L^2(M;\Omega^{1/2})$ and $\Cinftyub(M;E\otimes\Omega^{-1/2})$, and $H^s(M;E)$ is the $\Cinftyub(M)$-tensor product of $H^s(M)$ and $\Cinftyub(M;E)$. For instance, we may take $P=\nabla^*\nabla$ for any unitary connection $\nabla$ of bounded geometry on $E$. 

\begin{ex}\label{ex: widetilde P}
In \Cref{ex: widetilde E} and according to \Cref{ex: widetilde A}, $H^s(\widetilde M;\widetilde E)$ can be defined with the lift $\widetilde P$ of any nonnegative symmetric uniformly elliptic $P\in\Diff^2(M;E)$.
\end{ex}

For $s\in\N_0$, the Sobolev space $H^s(M)$ can be also described with the scalar product
\[
\langle u,v\rangle'_s=\sum_k\sum_{|I|\le s}\int_Bf_k^2(x)\cdot\partial_I(uy_{p_k}^{-1})(x)\cdot\overline{\partial_I(vy_{p_k}^{-1})(x)}\,dx\;,
\]
using the partition of unity $\{f_k\}$ given by \Cref{p: p_k} \cite[Theorem~A.22]{Schick1996}, \cite[Propositions~3.2 and~3.3]{Schick2001}, \cite[Appendices~A1.2 and~A1.3]{Shubin1992}. A similar scalar product $\langle{\cdot},{\cdot}\rangle'_s$ can be defined for $H^s(M;E)$ with the help of local trivializations defining the bounded geometry of $E$. Every $A\in\Diffub^m(M;E,F)$ defines bounded operators $A:H^{m+s}(M;E)\to H^s(M;F)$ ($s\in\R$), which induce continuous maps $A:H^{\pm\infty}(M;E)\to H^{\pm\infty}(M;F)$. For any almost everywhere positive $h\in C^\infty(M)$, we have $hH^m(M;E)=H^m(M;E)$ if and only if $h>0$ and $h^{\pm1}\in\Cinftyub(M)$.

If $m'>m+n/2$, then $H^{m'}(M;E)\subset C_{\text{\rm ub}}^m(M;E)$, continuously, and therefore $H^\infty(M;E)\subset \Cinftyub(M;E)$, continuously \cite[Proposition~2.8]{Roe1988I}. The Schwartz kernel mapping, $A\mapsto K_A$, defines a continuous linear map \cite[Proposition~2.9]{Roe1988I}
\begin{equation}\label{A mapsto K_A - bd geom}
L(H^{-\infty}(M;E),H^\infty(M;F))\to\Cinftyub(M;F\boxtimes(E^*\otimes\Omega))\;.
\end{equation}

\begin{rem}\label{r: supp K_A subset r-penumbra <=> supp Au subset r-penumbra for all u}
By~\eqref{K_A(cdot q)(u) = A delta_q^u}, for any $A\in L(H^{-\infty}(M;E),H^\infty(M;F))$ and $r>0$,
\[
\supp K_A\subset\{\,(p,q)\in M^2\mid d(p,q)\le r\,\}
\]
if and only if $\supp Au\subset\overline{\Pen}(\supp u,r)$ for all $u\in H^{-\infty}(M;E)$.
\end{rem}

Let $\RR$ \index{$\RR$} be the Fr\'echet space of rapidly decreasing functions on the real line. If $P\in\Diffub^m(M;E)$ is uniformly elliptic and essentially self-adjoint, then the spectral theorem defines a continuous functional calculus 
\[
\RR\to L(H^{-\infty}(M;E),H^\infty(M;E))\;,\quad\psi\mapsto\psi(P)\;.
\]
Thus, by~\eqref{A mapsto K_A - bd geom}, the linear map
\begin{equation}\label{psi mapsto K_psi(P)}
\RR\to \Cinftyub(M;E\boxtimes(E^*\otimes\Omega))\;,\quad\psi\mapsto K_{\psi(P)}\;,
\end{equation}
is continuous \cite[Proposition~2.10]{Roe1988I}.

\subsection{Maps of bounded geometry}\label{ss: maps of bd geom}

For $a\in\{1,2\}$, let $M_a$ be a Riemannian manifold of bounded geometry, of dimension $n_a$. Consider a normal chart $y_{a,p}:V_{a,p}\to B_a$ at every $p\in M_a$ satisfying the statement of \Cref{t: mfd of bd geom}. Let $r_a$ denote the radius of $B_a$. For $0<r\le r_a$, let $B_{a,r}\subset\R^{n_a}$ denote the ball centered at the origin with radius $r$. We have $B_a(p,r)=y_{a,p}^{-1}(B_{a,r})$.

A smooth map $\phi:M_1\to M_2$ is said to be of \emph{bounded geometry} if, for some $0<r<r_1$ and all $p\in M_1$, we have $\phi(B_1(p,r))\subset V_{2,\phi(p)}$, and the compositions $y_{2,\phi(p)}\phi y_{1,p}^{-1}$ define a bounded set in the Fr\'echet space $C^\infty(B_{1,r},\R^{n_2})$. This condition is preserved by the composition of maps. The set of smooth maps $M_1\to M_2$ of bounded geometry is denoted by $\Cinftyub(M_1,M_2)$.

Let $\phi\in \Cinftyub(M_1,M_2)$. For every $m\in\N_0\cup\{\infty\}$, using $\|{\cdot}\|'_{C^m_{\text{\rm ub}}}$ in the case where $m<\infty$ (\Cref{ss: uniform sps}) it follows that $\phi^*$ induces a continuous linear map \cite[Eq.~(19)]{AlvKordyLeichtnam2020}
\begin{equation}\label{phi^* on C_ub^m}
\phi^*:C_{\text{\rm ub}}^m(M_2;\Lambda)
\to C_{\text{\rm ub}}^m(M_1;\Lambda)\;.
\end{equation}

Recall that $\phi$ is called \emph{uniformly metrically proper} if, for any $s\ge0$, there is some $t_s\ge0$ so that, for all $p,q\in M_1$,
\[
d_2(\phi(p),\phi(q)) \le s \Rightarrow d_1(p,q) \le t_s\;.
\]
For all $m\in\N_0\cup\{\infty\}$, if $\phi\in \Cinftyub(M_1,M_2)$ is uniformly metrically proper, then $\phi^*$ induces a continuous linear map \cite[Eq.~(21)]{AlvKordyLeichtnam2020}
\begin{equation}\label{phi^* on H^m}
\phi^*:H^m(M_2;\Lambda)\to H^m(M_1;\Lambda)\;.
\end{equation}
If $\phi\in\Diffeo(M_1,M_2)$, and both $\phi$ and $\phi^{-1}$ are of bounded geometry, then $\phi$ is uniformly metrically proper. In this case,~\eqref{phi^* on H^m} can be continuously extended to Sobolev spaces of order $-m$.

The pull-back of a vector bundle of bounded geometry by a map of bounded geometry is of bounded geometry.

Homomorphisms of bounded geometry between vector bundles of bounded geometry have an obvious definition, but we will not use them.

\subsection{Smooth families of bounded geometry}\label{ss: families of bd geom}

Let $T$ be a manifold, and let $\pr_1:M\times T\to M$ denote the first factor projection. A section $u\in C^\infty(M\times T;\pr_1^*E)$ is called a \emph{smooth family of smooth sections} of $E$ (\emph{parametrized by} $T$), and we may use the notation $u=\{\,u_t\mid t\in T\,\}$, where $u_t=u(\cdot,t)\in C^\infty(M;E)$. Its \emph{$T$-support} is $\overline{\{\,t\in T\mid u_t\ne0\,\}}$. If the $T$-support is compact, then $u$ is said to be \emph{$T$-compactly supported}. It is said that $u$ is \emph{$T$-locally $C^\infty$-uniformly bounded} if any $t\in T$ is in some chart $(O,z)$ of $T$ such that the maps $u(y_p\times z)^{-1}$ define a bounded subset of the Fr\'echet space $C^\infty(B\times z(O),\C^l)$, using local trivializations of $E$ over the normal charts $(V_p,y_p)$. 

A \emph{smooth family of differential operators}, $A=\{\,A_t\mid t\in T\,\}\subset\Diff(M;E,F)$, can be defined by using smooth families of $\C$-valued functions, tangent vector fields and sections of $C^\infty(M;F\otimes E^*)$, like in \Cref{ss: diff ops}. For this $A$, the \emph{$T$-support} and the property of being \emph{$T$-compactly supported} is defined like in the case of smooth families of sections. If the smooth families of functions, tangent vector fields and sections used to describe $A$ are $T$-locally $C^\infty$-uniformly bounded, then it is said that $A$ is of \emph{$T$-local bounded geometry} (cf.\ \Cref{ss: diff ops of bd geom}).

A smooth map $\phi:M_1\times T\to M_2$ is called a \emph{smooth family of smooth maps} $M_1\to M_2$ (with \emph{parameters} in $T$). It may be denoted by $\phi=\{\,\phi^t\mid t\in T\,\}$, where $\phi^t=\phi(\cdot,t):M_1\to M_2$. It is said that $\phi$ is of \emph{$T$-local bounded geometry} if every $t\in T$ is in some chart $(O,z)$ of $T$ such that, for some $0<r<r_1$, we have $\phi(B_1(p,r)\times O)\subset V_{2,\phi(p)}$ for all $p\in M_1$, and the compositions $y_{2,\phi(p)}\phi(y_{1,p}\times z)^{-1}$, for $p\in M_1$, define a bounded subset of the Fr\'echet space $C^\infty(B_{1,r}\times z(O),\R^{n_2})$. The composition of smooth families of maps parametrized by $T$ has the obvious sense and preserves the $T$-local bounded geometry condition. In particular, the $\R$-local bounded geometry condition makes sense for a flow $\phi=\{\phi^t\}$ on $M$. Given $X\in\fXcom(M)$ with flow $\phi$, we have $X\in\fXub(M)$ if and only if $\phi$ is of $\R$-local bounded geometry \cite[Proposition~3.18]{AlvKordyLeichtnam2020}.

\subsection{Differential complexes of bounded geometry}\label{ss: diff complexes of bd geom}

With the notation of \Cref{ss: diff complexes}, assume that $M$, $E$ and $d$ are of bounded geometry (\Cref{ss: mfds and vector bdls of bd geom}). Then we may also consider the topological complexes $(\Cinftyub(M;E),d)$ and $(H^{\pm\infty}(M;E),d)$ (\Cref{ss: uniform sps,ss: Sobolev bd geom}).

$(E,d)$ is said to be \emph{uniformly elliptic} if $D$ (or $\Delta$) is uniformly elliptic (\Cref{ss: diff ops of bd geom}); this is equivalent to the obvious extension of~\eqref{uniformly elliptic} for~\eqref{symbol sequence}. In this case, a version of~\eqref{Hodge} is true for $(H^{\pm\infty}(M;E),d)$, where the reduced cohomology is used instead of the cohomology, and the closures of the images of $d$, $\delta$, $D$ and $\Delta$ are used instead of their images.

\section{Small b-calculus}\label{s: b-calculus}

 R.~Melrose introduced b-calculus, a way to extend calculus to manifolds with boundary \cite{Melrose1993,Melrose1996}. We will only use a part of it, called small b-calculus. For the sake of simplicity, we consider only compact manifolds with boundary, and the concepts and notation given here can be extended to the non-compact case like in \Cref{s: section sps and opers}, using compactly supported versions or local versions; some non-compact manifolds with boundary will be used in the paper. For the same reason, several kinds of section spaces and operators will be only defined in the case of functions or half-b-densities. Their extension to arbitrary vector bundles can be defined with tensor product expressions, like in \Cref{s: section sps and opers}. Most of these extensions will be used without further comments.

\subsection{Some notions of b-geometry}\label{ss: b-geometry}

Let $M$ be a compact (smooth) $n$-manifold with boundary, whose interior is denoted by $\mathring M$. There exists a function $x\in C^\infty(M)$ so that $x\ge0$, $\partial M=\{x=0\}$ and $dx\ne0$ on $\partial M$, which is called a \emph{boundary-defining function}. \index{boundary-defining function} Let ${}_+N\partial M\subset N\partial M$ be the inward-pointing subbundle of the normal bundle to the boundary. There is a unique trivialization $\nu\in C^\infty(\partial M;{}_+N\partial M)$ of ${}_+N\partial M$ so that $dx(\nu)=1$. Take a collar neighborhood $T\equiv[0,\epsilon_0)_x\times\partial M_\varpi$ of $\partial M$. (In a product expression, every factor projection may be indicated as subscript of the corresponding factor.) Given coordinates $y=(y^1,\dots,y^{n-1})$ on some open $V\subset\partial M$, we get via $\varpi$ coordinates $(x,y)=(x,y^1,\dots,y^{n-1})$, adapted (to $\partial M$), on the open subset $U\equiv[0,\epsilon_0)\times V\subset M$. There are vector bundles over $M$, $\bT M$ and $\bT^*M$, \index{$\bT M$} \index{$\bT^*M$} called \emph{b-tangent} and \emph{b-cotangent} bundles, which have the same restrictions to $\mathring M$ as $TM$ and $T^*M$, and such that $x\partial_x,\partial_{y^1},\dots,\partial_{y^{n-1}}$ and $x^{-1}dx,dy^1,\dots,dy^{n-1}$ extend to local frames around boundary points. This gives rise to versions of induced vector bundles, like $\bOmega^sM:=\Omega^s(\bT M)$ ($s\in\R$) and $\bOmega M:=\bOmega^1M$. We have
\begin{equation}\label{C^infty(M Omega^s) equiv x^s C^infty(M bOmega^s)}
C^\infty(M;\Omega^s)\equiv x^sC^\infty(M;\bOmega^s)\;.
\end{equation}
Thus the integration operator $\int_M$ is defined on $xC^\infty(M;\bOmega)$, and induces a pairing between $C^\infty(M)$ and $xC^\infty(M;\bOmega)$.

At the points of $\partial M$, the local section $x\partial_x$ is independent of the choice of adapted local coordinates, spanning a trivial line subbundle ${}^{\text{\rm b}}\!N\partial M\subset\bT_{\partial M}M$ with $T\partial M=\bT_{\partial M}M/{}^{\text{\rm b}}\!N\partial M$. So $\bOmega^s_{\partial M}M\equiv\Omega^s\partial M\otimes\Omega^s({}^{\text{\rm b}}\!N\partial M)$, and a restriction map $C^\infty(M;\bOmega^s)\to C^\infty(\partial M;\Omega^s)$ is locally given by
\[
u=a(x,y)\,\Big|\frac{dx}{x}dy\Big|^s\mapsto u|_{\partial M}=a(0,y)\,|dy|^s\;.
\]

A Riemannian structure $g$ on $\bT M$ is called a \emph{b-metric}. \index{b-metric} Locally,
\[
g=a_0\Big(\frac{dx}{x}\Big)^2+2\sum_{j=1}^{n-1}a_{0j}\,\frac{dx}{x}\,dy^j
+\sum_{j,k=1}^{n-1}a_{jk}\,dy^j\,dy^k\;,
\]
where $a_0$, $a_{0j}$ and $a_{jk}$ are $C^\infty$ functions, provided that $g$ is positive definite. If moreover $a_0=1+O(x^2)$ and $a_{0j}=O(x)$ as $x\downarrow0$, then $g$ is called \emph{exact}. In this case, the restriction of $g$ to $\mathring T\equiv(0,\epsilon_0)\times\partial M$ is asymptotically cylindrical, and therefore the restriction of $g$ to $\mathring M$ is a complete Riemannian metric. This restriction is of bounded geometry if it is \emph{cylindrical} around the boundary; i.e., taking $\epsilon_0$ small enough, we have $g=(\frac{dx}{x})^2+h$ on $\mathring T$ for some Riemannian metric $h$ on $\partial M$ (considering $h\equiv\varpi^*h$).

\subsection{Supported and extendible smooth functions}\label{ss: supported and extendible smooth funcs}

Let $\breve M$ \index{$\breve M$} be any closed manifold containing $M$ as submanifold of dimension $n$ (for instance, $\breve M$ can be the double of $M$). Let $M'=\breve M\setminus\mathring M$, which is another compact submanifold with boundary of $\breve M$, of dimension $n$ and with $\partial M'=M\cap M'=\partial M$. 

The concepts, notation and conventions of \Cref{ss: smooth/distributional sections} have straightforward extensions to manifolds with boundary, like the Fr\'echet space $C^\infty(M)$. Its elements are called \emph{extendible functions} \index{extendible function} because the continuous linear restriction map
\begin{equation}\label{R: C^infty(widetilde M) -> C^infty(M)}
R:C^\infty(\breve M)\to C^\infty(M)
\end{equation}
is surjective; in fact, there is a continuous linear extension map $E:C^\infty(M)\to C^\infty(\breve M)$ \cite{Seeley1964}. Since $C^\infty(\breve M)$ and $C^\infty(M)$ are Fr\'echet spaces, the map~\eqref{R: C^infty(widetilde M) -> C^infty(M)} is open by the open mapping theorem, and therefore it is a surjective topological homomorphism. Its null space is $C^\infty_{M'}(\breve M)$.

The Fr\'echet space of \emph{supported functions} \index{supported function} is the closed subspace of the smooth functions on $M$ that vanish to all orders at the points of $\partial M$, \index{$\dot C^\infty(M)$}
\begin{equation}\label{dot C^infty(M) = bigcap_m ge 0 x^m C^infty(M) subset C^infty(M)}
\dot C^\infty(M)=\bigcap_{m\ge0}x^mC^\infty(M)\subset C^\infty(M)\;.
\end{equation}
The extension by zero realizes $\dot C^\infty(M)$ as the closed subspace of functions on $\breve M$ supported in $M$,
\begin{equation}\label{dot C^infty(M) subset C^infty(breve M)}
\dot C^\infty(M)\equiv C^\infty_M(\breve M)\subset C^\infty(\breve M)\;.
\end{equation} 
By~\eqref{dot C^infty(M) = bigcap_m ge 0 x^m C^infty(M) subset C^infty(M)},
\begin{equation}\label{x^m dot C^infty(M) = dot C^infty(M)}
x^m\dot C^\infty(M)=\dot C^\infty(M)\quad(m\in\R)\;,
\end{equation}
and therefore, by~\eqref{C^infty(M Omega^s) equiv x^s C^infty(M bOmega^s)},
\begin{equation}\label{dot C^infty(M bOmega^s) equiv dot C^infty(M Omega^s)}
\dot C^\infty(M;\bOmega^s)\equiv\dot C^\infty(M;\Omega^s)\quad(s\in\R)\;.
\end{equation}

We can similarly define Banach spaces $C^k(M)$ and $\dot C^k(M)$ ($k\in\N_0$) satisfying the analogs of~\eqref{R: C^infty(widetilde M) -> C^infty(M)}--\eqref{dot C^infty(M) subset C^infty(breve M)}, which in turn yield analogs of the first inclusions of~\eqref{C^prime -k'(M E) supset C^prime -k(M E)}, obtaining $C^\infty(M)=\bigcap_kC^k(M)$ and $\dot C^\infty(M)=\bigcap_k\dot C^k(M)$.

\subsection{Supported and extendible distributions}\label{ss: supported and extendible distribs}

The spaces of \emph{supported} and \emph{extendible distributions} \index{supported distribution} \index{extendible distribution} on $M$ are \index{$\dot C^{-\infty}(M)$}
\[
\dot C^{-\infty}(M)=C^\infty(M;\Omega)'\;,\quad C^{-\infty}(M)=\dot C^\infty(M;\Omega)'\;.
\]
These are barreled, ultrabornological, webbed, acyclic DF Montel spaces, and therefore complete, boundedly/compactly/sequentially retractive and reflexive \cite[Proposition~6.1]{AlvKordyLeichtnam-conormal}. Transposing the version of~\eqref{R: C^infty(widetilde M) -> C^infty(M)} with $\Omega M$, we get \cite[Proposition~3.2.1]{Melrose1996}
\begin{equation}\label{dot C^-infty(M) subset C^-infty(breve M)}
\dot C^{-\infty}(M)\equiv C^{-\infty}_M(\breve M)\subset C^{-\infty}(\breve M)\;.
\end{equation}
Similarly,~\eqref{dot C^infty(M) subset C^infty(breve M)} and~\eqref{dot C^infty(M) = bigcap_m ge 0 x^m C^infty(M) subset C^infty(M)} give rise to continuous linear restriction maps
\begin{gather}
R:C^{-\infty}(\breve M)\to C^{-\infty}(M)\;,\label{R: C^-infty(breve M) -> C^-infty(M)}\\
R:\dot C^{-\infty}(M)\to C^{-\infty}(M)\;,\label{R: dot C^-infty(M) -> C^-infty(M)}
\end{gather}
which are surjective topological homomorphisms \cite[Proposition~6.2]{AlvKordyLeichtnam-conormal}. 
According to~\eqref{dot C^-infty(M) subset C^-infty(breve M)}, the map~\eqref{R: dot C^-infty(M) -> C^-infty(M)} is a restriction of~\eqref{R: C^-infty(breve M) -> C^-infty(M)}. There are continuous dense inclusions \cite[Lemma~3.2.1]{Melrose1996}
\begin{equation}\label{Cinftyc(mathring M) subset dot C^infty(M) subset C^infty(M) subset dot C^-infty(M)}
\Cinftyc(\mathring M)\subset\dot C^\infty(M)\subset C^\infty(M)\subset\dot C^{-\infty}(M)\;,
\end{equation}
the last one given by the integration pairing between $C^\infty(M)$ and $C^\infty(M;\Omega)$. The restriction of this pairing to $\dot C^\infty(M;\Omega)$ induces a continuous dense inclusion
\begin{equation}\label{C^infty(M) subset C^-infty(M)}
C^\infty(M)\subset C^{-\infty}(M)\;.
\end{equation}
Moreover~\eqref{R: dot C^-infty(M) -> C^-infty(M)} is the identity map on $C^\infty(M)$. 

As before, from~\eqref{x^m dot C^infty(M) = dot C^infty(M)} and~\eqref{dot C^infty(M bOmega^s) equiv dot C^infty(M Omega^s)}, we get
\begin{align}
x^mC^{-\infty}(M)&=C^{-\infty}(M)\quad(m\in\R)\;,\label{x^m C^-infty(M) = C^-infty(M)}\\
C^{-\infty}(M;\bOmega^s)&\equiv C^{-\infty}(M;\Omega^s)\quad(s\in\R)\;.\label{C^-infty(M bOmega^s) equiv C^-infty(M Omega^s)}
\end{align} 

The Banach spaces $C^{\prime-k}(M)$ and $\dot C^{\prime-k}(M)$ ($k\in\N_0$) \index{$\dot C^{\prime-k}(M)$} are similarly defined. They satisfy the analogs of~\eqref{dot C^-infty(M) subset C^-infty(breve M)}--\eqref{C^-infty(M bOmega^s) equiv C^-infty(M Omega^s)}, and the analogs of the second inclusions of~\eqref{C^prime -k'(M E) supset C^prime -k(M E)}, obtaining $\bigcup_kC^{\prime-k}(M)=C^{-\infty}(M)$ and $\bigcup_k\dot C^{\prime-k}(M)=\dot C^{-\infty}(M)$.

\subsection{Supported and extendible Sobolev spaces}\label{ss: supported and extendible Sobolev sps} 

The \emph{supported Sobolev space} \index{supported Sobolev space} of order $s\in\R$ is the closed subspace of the elements supported in $M$, \index{$\dot H^s(M)$}
\begin{equation}\label{dot H^s(M) subset H^s(breve M)}
\dot H^s(M)=H^s_M(\breve M)\subset H^s(\breve M)\;.
\end{equation}
On the other hand, using the map~\eqref{R: dot C^-infty(M) -> C^-infty(M)}, the \emph{extendible Sobolev space} \index{extendible Sobolev space} of order $s$ is $H^s(M)=R(H^s(\breve M))$ with the inductive topology given by the linear map
\begin{equation}\label{R: H^s(breve M) -> H^s(M)}
R:H^s(\breve M)\to H^s(M)\;.
\end{equation}
The null space of~\eqref{R: H^s(breve M) -> H^s(M)} is $H^s_{M'}(\breve M)$. The analogs of~\eqref{H^s(M E) subset H^s'(M E)}--\eqref{C^infty(M E) = bigcap_s H^s(M E)} hold true in this setting using $\dot C^{\pm\infty}(M)$ and $C^{\pm\infty}(M)$.  Furthermore the spaces $\dot H^s(M)$ and $H^s(M)$ form compact spectra of Hilbertian spaces.

The following properties are satisfied \cite[Proposition~3.5.1]{Melrose1996}. $C^\infty(M)$ is dense in $H^s(M)$, we have
\begin{equation}\label{dot H^s(M) equiv H^-s(M Omega)'}
\dot H^s(M)\equiv H^{-s}(M;\Omega)'\;,\quad H^s(M)\equiv\dot H^{-s}(M;\Omega)'\;,
\end{equation}
and the map~\eqref{R: dot C^-infty(M) -> C^-infty(M)} has a continuous restriction
\begin{equation}\label{R: dot H^s(M) -> H^s(M)}
R:\dot H^s(M)\to H^s(M)\;,
\end{equation}
which is surjective if $s\le1/2$, and injective if $s\ge-1/2$. In particular, $\dot H^0(M)\equiv H^0(M)\equiv L^2(M)$. The null space of~\eqref{R: dot H^s(M) -> H^s(M)} is $\dot H^s_{\partial M}(M)$.

\subsection{The space $\dot C^{-\infty}_{\partial M}(M)$}\label{ss: dot C^-infty_partial M(M)} 

The indicated properties of~\eqref{R: C^-infty(breve M) -> C^-infty(M)} and~\eqref{R: dot C^-infty(M) -> C^-infty(M)} mean that we have short exact sequences in the category of continuous linear maps between LCSs (see also \cite[Proposition~3.3.1]{Melrose1996}), \index{$\dot C^{-\infty}_{\partial M}(M)$}
\begin{gather}
0\to\dot C^{-\infty}(M') \xrightarrow{\iota} C^{-\infty}(\breve M) \xrightarrow{R} C^{-\infty}(M)\to0\;,\notag\\
0\to\dot C^{-\infty}_{\partial M}(M) \xrightarrow{\iota} \dot C^{-\infty}(M) \xrightarrow{R} C^{-\infty}(M)\to0\;.
\label{0 -> dot C^-infty_partial M(M) -> dot C^-infty(M) -> C^-infty(M) -> 0}
\end{gather}

From~\eqref{dot C^-infty(M) subset C^-infty(breve M)}, we get \index{$\dot C^{-\infty}_{\partial M}(M)$}
\begin{equation}\label{dot C^-infty_partial M(M) equiv C^-infty_partial M(breve M)}
\dot C^{-\infty}_{\partial M}(M)\equiv C^{-\infty}_{\partial M}(\breve M)\subset C^{-\infty}(\breve M)\;.
\end{equation}

The analogs of the second inclusion of~\eqref{C^prime -k'(M E) supset C^prime -k(M E)},~\eqref{H^s(M E) subset H^s'(M E)} and~\eqref{H^-s(M E) supset C^prime -k(M E) supset H^-k(M E)} hold true for the spaces $\dot C^{\prime\,-k}_{\partial M}(M)$ and $\dot H^s_{\partial M}(M)$. Thus the spaces $\dot C^{\prime\,-k}_{\partial M}(M)$ \index{$\dot C^{\prime\,-k}_{\partial M}(M)$} and $\dot H^s_{\partial M}(M)$ \index{$\dot H^s_{\partial M}(M)$} form spectra with the same union; the spectrum of spaces $\dot H^s_{\partial M}(M)$ is compact.

The following properties hold for $\dot C^{-\infty}_{\partial M}(M)$ \cite[Corollary~6.4 and~6.5]{AlvKordyLeichtnam-conormal}: it is a limit subspace of the LF-space $\dot C^{-\infty}(M)$; and it is barreled, ultrabornological, webbed acyclic DF Montel space, and therefore complete, reflexive and boundedly/compactly/sequentially retractive. A description of $\dot C^{-\infty}_{\partial M}(M)$ will be indicated in \Cref{r: description of dot C^infty_partial M(M)}.

\subsection{Differential operators acting on $C^{-\infty}(M)$ and $\dot C^{-\infty}(M)$}
\label{ss: diff ops on supp/ext distribs}

The notions of \Cref{ss: diff ops} also have straightforward extensions to manifolds with boundary. The action of any $A\in\Diff(M)$ on $C^\infty(M)$ preserves $\dot C^\infty(M)$, giving rise to extended continuous actions of $A$ on $C^{-\infty}(M)$ and $\dot C^{-\infty}(M)$. They fit into commutative diagrams
\begin{equation}\label{AR = RA}
\begin{CD}
\dot C^{-\infty}(M) @>A>> \dot C^{-\infty}(M) \\
@VRVV @VVRV \\
C^{-\infty}(M) @>A>> C^{-\infty}(M)
\end{CD}
\qquad
\begin{CD}
C^{-\infty}(M) @>A>> C^{-\infty}(M) \\
@A{\iota}AA @AA{\iota}A \\
C^\infty(M) @>A>> C^\infty(M)\;.\hspace{-.2cm}
\end{CD}
\end{equation}
However the analogous diagram
\begin{equation}\label{A inclusion ne inclusion A}
\begin{CD}
\dot C^{-\infty}(M) @>A>> \dot C^{-\infty}(M) \\
@A{\iota}AA @AA{\iota}A \\
C^\infty(M) @>A>> C^\infty(M)
\end{CD}
\end{equation}
may not be commutative. Using the notation $u\mapsto u_\co$ for the injection $C^\infty(M)\subset\dot C^{-\infty}(M)$ of~\eqref{Cinftyc(mathring M) subset dot C^infty(M) subset C^infty(M) subset dot C^-infty(M)}, we have $A(u_\co)-(Au)_\co\in C^{-\infty}_{\partial M}(M)$ for all $u\in C^\infty(M)$ \cite[Eq.~(3.4.8)]{Melrose1996}. 

From~\eqref{R: C^infty(widetilde M) -> C^infty(M)} and its version for vector fields, we get a surjective restriction map
\begin{equation}\label{restriction map Diff(breve M) -> Diff(M)}
\Diff(\breve M)\to\Diff(M)\;,\quad\breve A\mapsto\breve A|_M\;.
\end{equation}
For any $\breve A\in\Diff(\breve M)$ with $\breve A|_M=A$, we have the commutative diagrams
\begin{equation}\label{AR = RA with breve M}
\begin{CD}
C^{-\infty}(\breve M) @>{\breve A}>> C^{-\infty}(\breve M) \\
@VRVV @VVRV \\
C^{-\infty}(M) @>A>> C^{-\infty}(M)\;,\hspace{-.2cm}
\end{CD}
\qquad
\begin{CD}
C^{-\infty}(\breve M) @>{\breve A}>> C^{-\infty}(\breve M) \\
@A{\iota}AA @AA{\iota}A \\
\dot C^{-\infty}(M) @>A>> \dot C^{-\infty}(M)\;,\hspace{-.2cm}
\end{CD}
\end{equation}
where the left-hand side square extends the left-hand side square of~\eqref{AR = RA}.

If $A\in\Diff^m(M)$ ($m\in\N_0$), its actions on $\dot C^{-\infty}(M)$ and $C^{-\infty}(M)$ define continuous linear maps,
\begin{equation}\label{A:H^s(M) -> H^s-m(M)}
A:\dot H^s(M)\to\dot H^{s-m}(M)\;,\quad A:H^s(M)\to H^{s-m}(M)\;.
\end{equation}
The maps~\eqref{R: dot H^s(M) -> H^s(M)} and~\eqref{A:H^s(M) -> H^s-m(M)} fit into a commutative diagram given by the left-hand side square of~\eqref{AR = RA}.

\subsection{Differential operators tangent to the boundary}\label{ss: Diffb(M)}

The concepts of \Cref{s: conormal distribs} can be generalized to the case with boundary when $L=\partial M$ \cite[Chapter~6]{Melrose1996} (see also \cite[Section~4.9]{Melrose1993}), giving rise to the Lie subalgebra and $C^\infty(M)$-submodule $\fXb(M)\subset\fX(M)$ \index{$\fXb(M)$} of vector fields tangent to $\partial M$, called \emph{b-vector fields}. \index{b-vector field} We have $\fXb(M)\equiv C^\infty(M;\bT M)$. \index{$\fXb(M)$} Using $\fXb(M)$  like in \Cref{ss: diff ops}, we get the filtered $C^\infty(M)$-submodule and filtered subalgebra $\Diffb(M)\subset\Diff(M)$ \index{$\Diffb(M)$} of \emph{b-differential operators}; \index{b-differential operator} they are the operators $A\in\Diff(M)$ such that~\eqref{A inclusion ne inclusion A} is commutative \cite[Exercise~3.4.20]{Melrose1996}. The definition of $\Diffb(M)$ can be extended to arbitrary vector bundles like in \Cref{ss: diff ops}. The condition of being tangent to the boundary is closed by taking transposes and formal adjoints. The restriction map~\eqref{restriction map Diff(breve M) -> Diff(M)} satisfies
\begin{equation}\label{Diff(breve M,partial M) = Diffb(M)}
\Diff(\breve M,\partial M)|_M=\Diffb(M)\;.
\end{equation}
For all $a\in\R$ and $k\in\N_0$, we have \cite[Eqs.~(4.2.7) and~(4.2.8)]{Melrose1996}
\begin{equation}\label{Diffb^k(M) x^a = x^a Diffb^k(M)}
\Diffb^k(M)\,x^a=x^a\Diffb^k(M)\;.
\end{equation}
$\Diff(M)$ is spanned by $\partial_x$ and $\Diffb(M)$ as algebra, and therefore
\begin{equation}\label{Diff^k(M) x^a subset x^a-k Diff^k(M)}
\Diff^k(M)\,x^a\subset x^{a-k}\Diff^k(M)\;.
\end{equation}

\subsection{Conormal distributions at the boundary}
\label{ss: conormality at the boundary - Sobolev order}

The spaces of \emph{supported} and \emph{extendible conormal distributions} \index{supported conormal distributions} \index{extendible conormal distribution} at the boundary of Sobolev order $s\in\R$ are the $C^\infty(M)$-modules and LCSs, \index{$\dot\AA^{(s)}(M)$} \index{$\AA^{(s)}(M)$}
\begin{align*}
\dot\AA^{(s)}(M)&=\{\,u\in\dot C^{-\infty}(M)\mid\Diffb(M)\,u\subset\dot H^s(M)\,\}\;,\\
\AA^{(s)}(M)&=\{\,u\in C^{-\infty}(M)\mid\Diffb(M)\,u\subset H^s(M)\,\}\;,
\end{align*}
with the topologies defined like in~\eqref{Z = u in bigcup_A in AA dom A | AA cdot u subset Y}, which are totally reflexive Fr\'echet spaces \cite[Proposition~6.6]{AlvKordyLeichtnam-conormal}. They satisfy the analogs of the continuous inclusions~\eqref{I^(s)(M L) subset I^(s')(M L)}, giving rise to the filtered $C^\infty(M)$-modules and LCSs of \emph{supported} and \emph{extendible conormal distributions} at the boundary, \index{$\dot\AA(M)$} \index{$\AA(M)$}
\begin{equation}\label{dot AA(M) - AA(M)}
\dot\AA(M)=\bigcup_s\dot\AA^{(s)}(M)\;,\quad\AA(M)=\bigcup_s\AA^{(s)}(M)\;,
\end{equation}
which are barreled, ultrabornological and webbed \cite[Corollary~6.7]{AlvKordyLeichtnam-conormal}. By definition, there are continuous inclusions
\begin{equation}\label{dot AA(M) subset dot C^-infty(M)}
\dot\AA(M)\subset\dot C^{-\infty}(M)\;,\quad\AA(M)\subset C^{-\infty}(M)\;.
\end{equation}
Thus $\dot\AA(M)$ and $\AA(M)$ are Hausdorff. We have
\begin{equation}\label{bigcap_s AA^(s)(M) = C^infty(M)}
\bigcap_s\dot\AA^{(s)}(M)=\dot C^\infty(M)\;,\quad\bigcap_s\AA^{(s)}(M)=C^\infty(M)\;,
\end{equation}
obtaining continuous dense inclusions \cite[Proposition~4.1.1 and Lemma~4.6.1]{Melrose1996}
\begin{equation}\label{C^infty(M) subset AA(M)}
\dot C^\infty(M)\subset\dot\AA(M)\;,\quad C^\infty(M)\subset\AA(M),\dot\AA(M)\;.
\end{equation}
By~\eqref{C^infty(M) subset AA(M)} and the density of the inclusions~\eqref{Cinftyc(mathring M) subset dot C^infty(M) subset C^infty(M) subset dot C^-infty(M)} and~\eqref{C^infty(M) subset C^-infty(M)}, it follows that the inclusions~\eqref{dot AA(M) subset dot C^-infty(M)} are also dense. On the other hand, by elliptic regularity, we get continuous inclusions \cite[Eq.~(4.1.4)]{Melrose1996}
\begin{equation}\label{dot AA(M)|_mathring M AA(M) subset C^infty(mathring M)}
\dot\AA(M)|_{\mathring M},\AA(M)\subset C^\infty(\mathring M)\;.
\end{equation}
The maps~\eqref{R: dot H^s(M) -> H^s(M)} restrict to continuous linear maps
\begin{equation}\label{R: dot AA^(s)(M) -> AA^(s)(M)}
R:\dot\AA^{(s)}(M)\to\AA^{(s)}(M)\;,
\end{equation}
which are surjective for $s\le1/2$ and injective for $s\ge-1/2$. If $s=0$, then~\eqref{R: dot AA^(s)(M) -> AA^(s)(M)} is a TVS-isomorphism because $\dot H^0(M)\equiv H^0(M)$. The maps~\eqref{R: dot AA^(s)(M) -> AA^(s)(M)} induce a surjective topological homomorphism \cite[Proposition~4.1.1]{Melrose1996}, \cite[Proposition~6.8]{AlvKordyLeichtnam-conormal}
\begin{equation}\label{R: dot AA(M) -> AA(M)}
R:\dot\AA(M)\to\AA(M)\;,
\end{equation}
which is the identity on $C^\infty(M)$.

\subsection{The spaces $x^mL^\infty(M)$}\label{ss: x^m L^infty(M)}

For $m\in\R$, consider the weighted space $x^mL^\infty(M)$ (\Cref{ss: weighted sps}). There is a continuous inclusion
\[
x^mL^\infty(M)\subset C^{-\infty}(M)\;. 
\]
For $m'<m$, we also have a continuous inclusion
\begin{equation}\label{x^m L^infty(M) subset x^m' L^infty(M)}
x^mL^\infty(M)\subset x^{m'}L^\infty(M)\;,
\end{equation}
and $\Cinftyc(\mathring M)$ is dense in $x^mL^\infty(M)$ with the topology of $x^{m'}L^\infty(M)$ \cite[Proposition~6.10]{AlvKordyLeichtnam-conormal}.

\subsection{Filtration of $\AA(M)$ by bounds}\label{ss: description of AA(M) by bounds}

For every $m\in\R$, let \index{$\AA^m(M)$}
\[
\AA^m(M)=\{\,u\in C^{-\infty}(M)\mid\Diffb(M)\,u\subset x^mL^\infty(M)\,\}\;.
\]
This is another $C^\infty(M)$-module and Fr\'echet space with the topology like in~\eqref{Z = u in bigcup_A in AA dom A | AA cdot u subset Y}. By~\eqref{x^m L^infty(M) subset x^m' L^infty(M)}, there is a continuous inclusion
\begin{equation}\label{AA^m(M) subset AA^m'(M)}
\AA^m(M)\subset\AA^{m'}(M)\quad(m'<m)\;.
\end{equation}
Moreover there are continuous inclusions \cite[Proof of Proposition~4.2.1]{Melrose1996}
\begin{equation}\label{sandwich for AA}
\AA^{(s)}(M)\subset\AA^m(M)\subset\AA^{(\min\{m,0\})}(M)\quad(m<s-n/2-1)\;.
\end{equation}
Hence
\begin{equation}\label{AA(M) = bigcup_m AA^m(M)}
\AA(M)=\bigcup_m\AA^m(M)\;.
\end{equation}
Despite of defining the same LF-space, the filtrations of $\AA(M)$ given by the spaces $\AA^{(s)}(M)$ and $\AA^m(M)$ are not equivalent because, in contrast with~\eqref{bigcap_s AA^(s)(M) = C^infty(M)},
\[
\dot C^\infty(M)=\bigcap_m\AA^m(M)\;.
\]
The following is true \cite[Corollaries~6.14--6.16 and~6.39 and Remark~6.41]{AlvKordyLeichtnam-conormal}: the topologies of $\AA(M)$ and $C^\infty(\mathring M)$ coincide on every $\AA^m(M)$ (however the second inclusion of~\eqref{dot AA(M)|_mathring M AA(M) subset C^infty(mathring M)} is not a TVS-embedding); $\Cinftyc(\mathring M)$ is dense in every $\AA^m(M)$, and therefore in every $\AA^{(s)}(M)$ and $\AA(M)$; and $\AA(M)$ is an acyclic Montel space, and therefore complete, boundedly/compactly/sequentially retractive and reflexive.

\subsection{$\dot\AA(M)$ and $\AA(M)$ vs $I(\breve M,\partial M)$}
\label{ss: dot AA(M) and AA(M) vs I(breve M partial M)}

The restriction maps~\eqref{R: H^s(breve M) -> H^s(M)} define continuous linear maps
\[
R:I^{(s)}(\breve M,\partial M)\to\AA^{(s)}(M)\;,
\]
which induce a surjective topological homomorphism \cite[Proposition~6.18]{AlvKordyLeichtnam-conormal}
\begin{equation}\label{R: I(breve M partial M) -> AA(M)}
R:I(\breve M,\partial M)\to\AA(M)\;.
\end{equation}
The null space of~\eqref{R: I(breve M partial M) -> AA(M)} is $I_{M'}(\breve M,\partial M)$. There are TVS-identities
\begin{equation}\label{dot AA^(s)(M) equiv I^(s)_M(breve M partial M)}
\dot\AA^{(s)}(M)\equiv I^{(s)}_M(\breve M,\partial M)\;,
\end{equation}
inducing a TVS-isomorphism \cite[Corollary~6.20]{AlvKordyLeichtnam-conormal}
\begin{equation}\label{dot AA(M) cong I_M(breve M partial M)}
\dot\AA(M) \xrightarrow{\cong} I_M(\breve M,\partial M)\;.
\end{equation}
Moreover $I_M(\breve M,\partial M)$ is a limit subspace of the LF-space $I(\breve M,\partial M)$ \cite[Proposition~6.19]{AlvKordyLeichtnam-conormal}.

\subsection{Filtration of $\dot\AA(M)$ by the symbol order}
\label{ss: conormality at the boundary - symbol order}

Like in~\eqref{dot AA^(s)(M) equiv I^(s)_M(breve M partial M)}, let \index{$\dot\AA^m(M)$}
\begin{equation}\label{dot AA^m(M) = I^m_M(breve M partial M)}
\dot\AA^m(M)=I^m_M(\breve M,\partial M)\subset I^m(\breve M,\partial M)\quad(m\in\R)\;,
\end{equation}
which are closed subspaces satisfying the analogs of~\eqref{I^m(M L) subset I^m'(M L)} and~\eqref{sandwich for I}. Thus
\[
\dot\AA(M)=\bigcup_m\dot\AA^m(M)\;,\quad\dot C^\infty(M)=\bigcap_m\dot\AA^m(M)\;,
\]
and the TVS-isomorphism~\eqref{dot AA(M) cong I_M(breve M partial M)} is also compatible with the symbol filtration. $\dot\AA(M)$ is an acyclic Montel space, and therefore complete, boundedly/compactly/sequentially retractive and reflexive \cite[Corollary~6.22]{AlvKordyLeichtnam-conormal}.

\subsection{The space $\KK(M)$}\label{ss: KK(M)}

The condition of being supported in $\partial M$ defines the LCHSs and $C^\infty(M)$-modules \index{$\KK^{(s)}(M)$} \index{$\KK^m(M)$} \index{$\KK(M)$}
\[
\KK^{(s)}(M)=\dot\AA^{(s)}_{\partial M}(M)\;,\quad\KK^m(M)=\dot\AA^m_{\partial M}(M)\;,\quad
\KK(M)=\dot\AA_{\partial M}(M)\;.
\]
These are the null spaces of the corresponding restrictions of the map~\eqref{R: dot AA(M) -> AA(M)} to $\dot\AA^{(s)}(M)$, $\dot\AA^m(M)$ and $\dot\AA(M)$. They satisfy the analogs of~\eqref{I^(s)(M L) subset I^(s')(M L)},~\eqref{I^m(M L) subset I^m'(M L)} and~\eqref{sandwich for I}, obtaining $\bigcup_s\KK^{(s)}(M)=\bigcup_m\KK^m(M)$. 

The properties of~\eqref{R: dot AA(M) -> AA(M)} mean that the following sequence is exact in the category of continuous linear maps between LCSs:
\begin{equation}\label{0 -> KK(M) -> dot AA(M) -> AA(M) -> 0}
0\to\KK(M) \xrightarrow{\iota} \dot\AA(M) \xrightarrow{R} \AA(M)\to0\;.
\end{equation}
It is called the \emph{conormal sequence at the boundary}. We have \index{conormal sequence at the boundary}
\[
\KK^{(s)}(M)=\{\,u\in\dot C^{-\infty}_{\partial M}(M)\mid\Diffb(M)\,u\subset\dot H^s_{\partial M}(M)\,\}\;,
\]
with the topology defined like in~\eqref{Z = u in bigcup_A in AA dom A | AA cdot u subset Y}. The following properties hold \cite[Propositions~6.24 and~6.25 and Corollaries~6.26--6.28]{AlvKordyLeichtnam-conormal}: every $\KK^{(s)}(M)$ is a totally reflexive Fr\'echet space; $\KK(M)$ is a limit subspace of the LF-space $\dot\AA(M)$; and $\KK(M)$ is barreled, ultrabornological, webbed and an acyclic Montel space, and therefore complete, boundedly/compactly/sequentially retractive and reflexive.

The TVS-isomorphism~\eqref{dot AA(M) cong I_M(breve M partial M)} restricts to a TVS-identity
\begin{equation}\label{KK(M) cong I_partial M(breve M partial M)}
\KK(M)\equiv I_{\partial M}(\breve M,\partial M)\;,
\end{equation}
which in turn restricts to identities between the LCHSs defining the Sobolev-order and symbol-order filtrations, according to~\eqref{dot AA^(s)(M) equiv I^(s)_M(breve M partial M)} and~\eqref{dot AA^m(M) = I^m_M(breve M partial M)}.

A description of $\KK^{(s)}(M)$ and $\KK(M)$ will be indicated in \Cref{r: description of KK(M)}.

\subsection{Action of $\Diff(M)$ on $\dot\AA(M)$, $\AA(M)$ and $\KK(M)$}
\label{ss: diff ops on dot AA(M) and AA(M)}

Any $A\in\Diff(M)$ defines continuous endomorphisms $A$ of $\dot\AA(M)$, $\AA(M)$ and $\KK(M)$. If $A\in\Diff^k(M)$, these maps also satisfy the analogs of~\eqref{A: I^(s)(M L E) -> I^[s-k](M L E)}. If $A\in\Diffb(M)$, then it defines continuous endomorphisms of $\dot\AA^{(s)}(M)$, $\AA^{(s)}(M)$, $\AA^m(M)$ and $\KK^{(s)}(M)$. All of these maps are restrictions of the endomorphisms $A$ of $\dot C^{-\infty}(M)$, $C^{-\infty}(M)$ and $C^\infty(\mathring M)$, and extensions of the endomorphisms $A$ of $\dot C^\infty(M)$ and $C^\infty(M)$.

\subsection{Partial extension maps}\label{ss: partial extension maps}

Given linear subspaces, $X\subset\AA(M)$ and $Y\subset\dot\AA(M$), a map $E:X\to Y$ is called a \emph{partial extension map} \index{partial extension map} if $R(Y)\subset X$ and $RE=1$ on $X$. The surjectivity of~\eqref{R: dot AA(M) -> AA(M)} is given by the following result. Its proof is recalled here because it will play an important role in our work.

\begin{prop}[Cf.\ {\cite[Section~4.4]{Melrose1996}}]\label{p: E_m}
For all $m\in\R$, there is a continuous linear partial extension map $E_m:\AA^m(M)\to\dot\AA^{(s)}(M)$, \index{$E_m$} where $s=0$ if $m\ge0$, and $m>s\in\Z^-$ if $m<0$. For $m\ge0$, $E_m:\AA^m(M)\to\dot\AA^{(0)}(M)$ is a continuous inclusion map.
\end{prop}

\begin{proof}
First, let us consider the non-compact $n$-manifold with boundary $\R^n_1:=[0,\infty)\times\R^{n-1}$, whose double is $\R^n$. Consider the canonical coordinates on $\R^n_1$ given by the factor projections, $x:\R^n_1\to[0,\infty)$ and $y:\R^n_1\to\R^{n-1}$. We use the obvious generalization to the non-compact case of the spaces of extendible and supported conormal distributions at the boundary, of Sobolev order $s$, whose definitions involve $H^s_{\text{\rm loc}}(\R^n_1)$ and $\dot H^s_{\text{\rm loc}}(\R^n_1)$ like in \Cref{ss: I(M L) - non-compact}. 

For $m\ge0$, since $x^mL^\infty(\R^n_1)\subset L^2_{\text{\rm loc}}(\R^n_1)$, continuously, we get $\AA^m(\R^n_1)\subset\dot\AA^{(0)}(\R^n_1)$, continuously. This also follows from~\eqref{sandwich for AA} using that $\dot\AA^{(0)}(\R_1^n)\equiv\AA^{(0)}(\R_1^n)$. Thus $E_m$ must be the inclusion map in this case.

Now fix $m<0$. For $0<\delta\le1$ such that $m+\delta\le0$ if $m\ne-1$, and $m+\delta<0$ if $m=-1$, we have a continuous linear map $J:\AA^m(\R^n_1)\to\AA^{m+\delta}(\R^n_1)$ defined by
\begin{equation}\label{Ju(x y)}
Ju(x,y)=\int_1^xu(\xi,y)\,d\xi\;.
\end{equation}
 So, for $-m<-s=:N\in\N$, we get the continuous linear maps (see \Cref{ss: diff ops on dot AA(M) and AA(M)})
\[
\AA^m(\R^n_1) \xrightarrow{J^N} \AA^0(\R^n_1) \xrightarrow{E_0} \dot\AA^{(0)}(\R^n_1) \xrightarrow{\partial_x^N} \dot\AA^{(s)}(\R^n_1)\;,
\]
whose composition is the desired extension $E_m$. For all $u\in\AA^m(\R^n_1)$, we have
\begin{equation}\label{supp Eu}
\partial\R^n_1\cap\supp E_mu\subset\{0\}\times y(\supp u)\;.
\end{equation}

Consider now a compact manifold with boundary $M$. Cover $\partial M$ with a finite collection of adapted charts $(U_j,(x_j,y_j))$, and let $\{\lambda_j,\mu\}$ be a partition of unity subordinated to the open covering $\{U_j,\mathring M\}$ of $M$. By the case of $\R^n_1$, we directly get $\AA^m(U_j)\subset\dot\AA^{(0)}(U_j)$, continuously, if $m\ge0$. By~\eqref{supp Eu}, if $m<0$ and $-m<N\in\N$, we get a continuous linear partial extension map $E_{m,j}:\AA^m(U_j)\to\dot\AA^{(-N)}(U_j)$, which preserves the condition of being compactly supported. Then the result follows with $E_m:\AA^m(M)\to\dot\AA^{(s)}(M)$ defined by $E_mu=\mu u+\sum_jE_{m,j}(\lambda_ju)$.
\end{proof}

\begin{rem}\label{r: E_m m < 0}
Consider the case where $m<0$ in the proof of \Cref{p: E_m}. Taking a collar neighborhood of the boundary, $T\equiv[0,\epsilon)_x\times\partial M_\varpi$, we can use adapted charts $(U_j\equiv[0,\epsilon)\times V_j,(x,y_j))$ defined by charts $(V_j,y_j)$ of $\partial M$, like in \Cref{ss: b-geometry}. Then the operators $\partial_x\in\Diff(U_j)$ can be combined to define an operator $\partial_x\in\Diff(T)$, which indeed is the derivative operator on $C^\infty(T)\equiv C^\infty([0,\epsilon),C^\infty(\partial M))$. On the other hand, by integrating from $\epsilon$ to $x$, like in~\eqref{Ju(x y)}, we get a continuous linear map $J:\AA^m(T)\to\AA^{m+\delta}(T)$; in fact, this defines a continuous endomorphism $J$ of $C^\infty(\mathring T)$. In this way, a continuous linear extension map $E_{m,T}:\AA^m(T)\to\dot\AA^{(s)}(T)$ can be defined like in the case of $\R^n_1$. Then $E_m:\AA^m(M)\to\dot\AA^{(s)}(M)$ can be given by $E_mu=\mu u+E_{m,T}(\lambda u)$, where $\{\lambda,\mu\}$ is a partition of unity subordinated to the open covering $\{T,\mathring M\}$ of $M$.
\end{rem}

\begin{rem}\label{r: E_m vector bundle}
A version of \Cref{p: E_m} with a vector bundle $E$ over $M$ can be achieved by taking the $C^\infty(M)$-tensor product with the identity map on $C^\infty(M;E\otimes E^*)$. We can also adapt the proof as follows. With the notation of \Cref{r: E_m m < 0}, there is an identity $E_T\equiv\varpi^*E_{\partial M}\equiv[0,\epsilon)\times E_{\partial M}$ over $T$, which induces trivializations $E_{U_j}\equiv[0,\epsilon)\times E_{V_j}\equiv[0,\epsilon)\times V_j\times \C^l$ over domains $U_j\equiv[0,\epsilon)\times V_j$. Like in \Cref{r: E_m m < 0}, these local trivializations can be used to define $\partial_x\in\Diff^1(T;E)$, which is considered as the derivative operator on $C^\infty([0,\epsilon),C^\infty(\partial M;E))\equiv C^\infty(T;E)$. As usual, integration by parts shows that
\begin{equation}\label{partial_x^t = -partial_x}
\partial_x^\trans=-\partial_x\in\Diff^1(T;E^*\otimes\Omega)\;.
\end{equation}
If $E=\Lambda M$, then $\partial_x\in\Diff^1(T;\Lambda)$ is the Lie derivative with respect to $\partial_x\in\fX(T)$. 
\end{rem}

\begin{rem}\label{r: E_m compactly supported} 
By~\eqref{supp Eu}, all steps of the proof of \Cref{p: E_m} have obvious compactly supported versions. This also applies to \Cref{r: E_m m < 0,r: E_m vector bundle}.
\end{rem}

Given $m$ and $s$ satisfying the conditions of \Cref{p: E_m}, let us denote by $E_{m,s}$ the partial extension map constructed in the proof of \Cref{p: E_m}. This notation will make it easier to analyze its dependence on $m$ and $s$ in the following result.

\begin{prop}\label{p: E_m s}
Let $s'\le s$ and $m'\le m$ such that the maps $E_{m,s}$, $E_{m,s'}$ and $E_{m',s'}$ are defined. Then $E_{m,s}u=E_{m',s'}u$ for all $u\in\AA^m(M)$.
\end{prop}

\begin{proof}
According to the proof of \Cref{p: E_m}, it is enough to consider the case of $\R^n_1$. 

If $m'\ge0$, there is nothing to prove.

In the case $m<0$, we have $s,s'\in\Z^-$ with $m'>s>s'$. Let $N=-s$, $N'=-s'$ and $k=s-s'=N'-N$ in $\Z^+$. Since $\AA^0(\R^n_1)\subset L^\infty(\R^n_1)$, the composition
\[
\AA^0(\R^n_1) \xrightarrow{J^k} \AA^0(\R^n_1) \hookrightarrow \dot\AA^{(0)}(\R^n_1) 
\xrightarrow{\partial_x^k} \dot\AA^{(-k)}(\R^n_1)
\]
is equal to the inclusion map $\AA^0(\R^n_1)\hookrightarrow\dot\AA^{(-k)}(\R^n_1)$. So, for all $u\in\AA^m(\R^n_1)$, since $J^Nu\in\AA^0(\R^n_1)$, we have
\[
E_{m',s'}u=E_{m,s'}u=\partial_x^{N'}J^{N'}u=\partial_x^N\partial_x^kJ^kJ^Nu=\partial_x^NJ^Nu=E_{m,s}u\;.
\]

In the case $m'<0\le m$, we have $s=0$ and $m'>s'\in\Z^-$. Then the result follows with a similar argument using $k=-s'\in\Z^+$.
\end{proof}

\begin{cor}\label{c: E_m s}
For all $s$ and $m$ such that the map $E_{m,s}$ is defined, we have $E_{m,s}u=u$ for all $u\in\Cinftyc(\mathring M)$.
\end{cor}

\begin{proof}
Use that $\Cinftyc(\mathring M)\subset\AA^m(M)$ and apply \Cref{p: E_m s}.
\end{proof}

\begin{rem}\label{r: E_m s with m ge 0 and s < 0}
The proof of \Cref{p: E_m s} can be also applied to maps $E_{m,s}$ with $m\ge0$ and $s\in\Z^-$, defined $E_m$ like in the case $m<0$. Including these maps, the map $E_{m,s'}$ of the statement is always defined under the other assumptions.
\end{rem}

\begin{rem}\label{r: E_m s - variants of the defn}
\Cref{p: E_m s,c: E_m s} are also true with the definitions of $E_m$ given in \Cref{r: E_m m < 0,r: E_m vector bundle}, with similar proofs.
\end{rem}

\subsection{$L^2$ and $L^\infty$ half-b-densities}\label{ss: L^2 and L^infty and L^infty half-b-densities}

We have 
\begin{align}
L^2(M;\bOmega^{\frac12})&\equiv x^{-\frac12}L^2(M;\Omega^{\frac12})\;,
\label{L^2(M bOmega^1/2) equiv x^-1/2 L^2(M Omega^1/2)}\\
L^\infty(M;\bOmega^{\frac12})&\equiv x^{-\frac12}L^\infty(M;\Omega^{\frac12})\;,
\label{L^infty(M bOmega^1/2) equiv x^-1/2 L^infty(M Omega^1/2)}
\end{align}
where~\eqref{L^2(M bOmega^1/2) equiv x^-1/2 L^2(M Omega^1/2)} holds as Hilbert spaces, and~\eqref{L^infty(M bOmega^1/2) equiv x^-1/2 L^infty(M Omega^1/2)} holds as LCHSs endowed with a family of equivalent Banach space norms \cite[Eqs.~(6.51) and~(6.52)]{AlvKordyLeichtnam-conormal}.

Equip $M$ with a b-metric $g$ (\Cref{ss: b-geometry}), and endow $\mathring M$ with the restriction of $g$, also denoted by $g$. With the corresponding Euclidean/Hermitean structures on $\Omega^{1/2}\mathring M$ and $\bOmega^{1/2}M$, we get $L^\infty(\mathring M;\Omega^{\frac12})\equiv L^\infty(M;\bOmega^{\frac12})$ as Banach spaces.

\subsection{b-Sobolev spaces}\label{ss: b-Sobolev}

For $m\in\N_0$, the \emph{b-Sobolev spaces} \index{b-Sobolev space} of \emph{order} $\pm m$ are the $C^\infty(M)$-modules and Hilbertian spaces defined by the following analogs of~\eqref{H^s(M E)} and~\eqref{H^-s(M E) = Psi^s(M E) cdot L^2(M E) = H^s(M E^* otimes Omega)'}: \index{$\Hb^m(M;\bOmega^{\frac12})$}
\begin{gather*}
\Hb^m(M;\bOmega^{\frac12})=\{\,u\in L^2(M;\bOmega^{\frac12})\mid
\Diffb^m(M;\bOmega^{\frac12})\,u\subset L^2(M;\bOmega^{\frac12})\,\}\;,\\
\Diffb^m(M;\bOmega^{\frac12})\,L^2(M;\bOmega^{\frac12})=\Hb^{-m}(M;\bOmega^{\frac12})
= \Hb^m(M;\bOmega^{\frac12})'\;.
\end{gather*}
Any finite set of $C^\infty(M)$-generators of $\Diffb^m(M;\bOmega^{1/2})$ defines a scalar product on $\Hb^{\pm m}(M;\bOmega^{1/2})$. The intersections/unions of the spaces $\Hb^m(M;\bOmega^{1/2})$ ($m\in\Z$) are denoted by $\Hb^{\pm\infty}(M;\bOmega^{1/2})$. In particular, $\Hb^\infty(M;\bOmega^{1/2})=\AA^{(0)}(M;\bOmega^{1/2})$.

\subsection{Weighted b-Sobolev spaces}\label{ss: weighted b-Sobolev}

We will also use the \emph{weighted b-Sobolev space} $x^a\Hb^m(M;\bOmega^{1/2})$ ($a\in\R$), \index{$x^a\Hb^m(M;\bOmega^{1/2})$} another Hilbertian space defined like in \Cref{ss: weighted sps}. We have \cite[Section~6.19]{AlvKordyLeichtnam-conormal}
\[
\bigcap_{a,m}x^aH^m_{\text{\rm b}}(M;\bOmega^{\frac12})=\dot C^\infty(M;\bOmega^{\frac12})\;.
\]

\subsection{Action of $\Diffb^m(M)$ on weighted b-Sobolev spaces}
\label{ss: (pseudo) diff ops on weighted b-Sobolev}

We have
\begin{equation}\label{Diffb^m(M bOmega^1/2) equiv Diffb^m(M) equiv Diffb^m(M Omega^1/2)}
\Diffb^m(M;\bOmega^{\frac12})\equiv\Diffb^m(M)\equiv\Diffb^m(M;\Omega^{\frac12})\;,
\end{equation}
like in~\eqref{Diff^m(M E) equiv Diff^m(M)}. Moreover any $A\in\Diffb^k(M;\bOmega^{1/2})$ defines continuous linear maps \cite[Lemma 5.14]{Melrose1993}
\[
A:x^aH^m_{\text{\rm b}}(M;\bOmega^{\frac12})\to x^aH^{m-k}_{\text{\rm b}}(M;\bOmega^{\frac12})\;,
\]
which induce a continuous endomorphism $A$ of $x^aH^{\pm\infty}_{\text{\rm b}}(M;\bOmega^{1/2})$.

\subsection{A description of $\AA(M)$}\label{ss: a description of AA(M)}

In this subsection, unless the contrary is indicated, assume the following properties:
\begin{enumerate}[{\rm(A)}]

\item\label{i: g is of bounded geometry} $\mathring M$ is of bounded geometry with $g$.

\item\label{i: A'}  The collar neighborhood $T$ of $\partial M$ can be chosen so that:
\begin{enumerate}[(a)]

\item\label{i: extension A'} every $A\in\fX(\partial M)$ has an extension $A'\in\fXb(T)$ such that $A'$ is $\varpi$-projectable to $A$, and $A'|_{\mathring T}$ is orthogonal to the $\varpi$-fibers; and

\item\label{i: fX_ub(mathring M)|_mathring T is generated by x partial_x and the vector fields A'} $\fXub(\mathring M)|_{\mathring T}$ is $\Cinftyub(\mathring M)|_{\mathring T}$-generated by $x\partial_x$ and the restrictions $A'|_{\mathring T}$ of the vector fields $A'$ of~\ref{i: extension A'}, for $A\in\fX(\partial M)$.

\end{enumerate}
\end{enumerate}
For instance,~\ref{i: g is of bounded geometry} and~\ref{i: A'} are true if $\mathring T$ is cylindrical with $g$ (\Cref{ss: b-geometry}). The following properties hold \cite[Corollaries~6.32,~6.34,~6.35,~6.37,~6.38 and~6.40 and Propositions~6.33 and~6.36]{AlvKordyLeichtnam-conormal}: the restriction to $\mathring M$ defines a continuous injection $C^\infty(M)\subset\Cinftyub(\mathring M)$ (thus $\Cinftyub(\mathring M)$ becomes a $C^\infty(M)$-module); as $\Cinftyub(\mathring M)$-modules,
\begin{align*}
\Diffub^m(\mathring M)&\equiv\Diffb^m(M)\otimes_{C^\infty(M)}\Cinftyub(\mathring M)\;,\\
\Diffub^m(\mathring M;\Omega^{\frac12})&\equiv\Diffb^m(M;\bOmega^{\frac12})\otimes_{C^\infty(M)}\Cinftyub(\mathring M)\;;
\end{align*}
as $C^\infty(M)$-modules and Hilbertian spaces, for $m\in\Z$,
\begin{alignat*}{2}
H^m(\mathring M;\Omega^{1/2})&\equiv\Hb^m(M;\bOmega^{1/2})\;,&\quad
H^m(\mathring M)&=x^{-1/2}\Hb^m(M)\;,\\
H^{\pm\infty}(\mathring M;\Omega^{1/2})&\equiv\Hb^{\pm\infty}(M;\bOmega^{1/2})\;,&\quad
H^{\pm\infty}(\mathring M)&=x^{-1/2}\Hb^{\pm\infty}(M)\;;
\end{alignat*}
as $C^\infty(M)$-modules and LCHSs, for $m\in\R$,
\begin{gather}
\AA^m(M;\Omega^{1/2})\equiv x^{m+1/2}\Hb^\infty(M;\bOmega^{1/2})\;,\notag\\
\AA^m(M)\equiv x^m\Hb^\infty(M)\equiv x^{m+1/2}H^\infty(\mathring M)\;,
\label{AA^m(M) equiv x^m Hb^infty(M) equiv x^m+1/2 H^infty(mathring M)}\\
\AA(M)\equiv\bigcup_mx^m\Hb^\infty(M)=\bigcup_mx^mH^\infty(\mathring M)\;.
\label{AA(M) equiv bigcup_m x^m Hb^infty(M) = bigcup_m x^m H^infty(mathring M)}
\end{gather}
Actually, the first identities of~\eqref{AA^m(M) equiv x^m Hb^infty(M) equiv x^m+1/2 H^infty(mathring M)} and~\eqref{AA(M) equiv bigcup_m x^m Hb^infty(M) = bigcup_m x^m H^infty(mathring M)} are independent of $g$, and therefore they hold true without the assumptions~\ref{i: g is of bounded geometry} and~\ref{i: A'} \cite[Remark~6.41]{AlvKordyLeichtnam-conormal}.

\subsection{Dual-conormal distributions at the boundary}\label{ss: AA'(M)}

Consider the LCHSs \cite[Section~18.3]{Hormander1985-III}, \cite[Chapter~4]{Melrose1996}, \index{$\KK'(M)$} \index{$\AA'(M)$} \index{$\dot\AA'(M)$}
\[
\KK'(M)=\KK(M;\Omega)'\;,\quad\AA'(M)=\dot\AA(M;\Omega)'\;,\quad\dot\AA'(M)=\AA(M;\Omega)'\;,
\]
which are complete Montel spaces \cite[Proposition~6.42]{AlvKordyLeichtnam-conormal}. The elements of $\AA'(M)$ (resp., $\dot\AA'(M)$) will be called \emph{extendible} (resp., \emph{supported}) \emph{dual-conormal distributions} at the boundary. Consider also the LCHSs \index{extendible dual-conormal distributions} \index{supported dual-conormal distributions} \index{$\KK^{\prime\,(s)}(M)$} \index{$\KK^{\prime\,m}(M)$} \index{$\AA^{\prime\,(s)}(M)$} \index{$\dot\AA^{\prime\,(s)}(M)$}
\[
\KK^{\prime\,(s)}(M)=\KK^{(-s)}(M;\Omega)'\;,\quad
\KK^{\prime\,m}(M)=\KK^{-m}(M;\Omega)'\;,
\]
and, similarly, define $\AA^{\prime\,(s)}(M)$, $\AA^{\prime\,m}(M)$, $\dot\AA^{\prime\,(s)}(M)$ and $\dot\AA^{\prime\,m}(M)$. The spaces $\KK^{\prime\,(s)}(M)$, $\AA^{\prime\,(s)}(M)$ and $\dot\AA^{\prime\,(s)}(M)$ are bornological and barreled \cite[Corollary~6.43]{AlvKordyLeichtnam-conormal}. The transpositions of the analogs of~\eqref{I^(s)(M L) subset I^(s')(M L)} and~\eqref{I^m(M L) subset I^m'(M L)} for the spaces $\KK^{(s)}(M;\Omega)$, $\KK^m(M;\Omega)$, $\dot\AA^{(s)}(M;\Omega)$ and $\dot\AA^m(M;\Omega)$ are continuous linear restriction maps
\begin{gather*}
\KK^{\prime\,(s')}(M)\to\KK^{\prime\,(s)}(M)\;,\quad\KK^{\prime\,m}(M)\to\KK^{\prime\,m'}(M)\;,\\
\AA^{\prime\,(s')}(M)\to\AA^{\prime\,(s)}(M)\;,\quad\AA^{\prime\,m}(M)\to\AA^{\prime\,m'}(M)\;,
\end{gather*}
for $s<s'$ and $m<m'$. These maps form projective spectra, giving rise to projective limits. The spaces $\KK^{\prime\,(s)}(M)$, $\KK^{\prime\,m}(M)$, $\AA^{\prime\,(s)}(M)$ and $\AA^{\prime\,m}(M)$ satisfy the analogs of~\eqref{sandwich for I'} and~\eqref{I'(M L) equiv projlim I^prime (s)(M L) equiv projlim I^prime m(M L)} \cite[Corollary~6.44]{AlvKordyLeichtnam-conormal}.

Similarly, transposing the analogs of~\eqref{I^(s)(M L) subset I^(s')(M L)} and~\eqref{AA^m(M) subset AA^m'(M)} for the spaces $\AA(M,\Omega M)$, we get continuous inclusions
\[
\dot\AA^{\prime\,(s)}(M)\supset\dot\AA^{\prime\,(s')}(M)\;,\quad\dot\AA^{\prime\,m}(M)\supset\dot\AA^{\prime\,m'}(M)\;,
\]
for $s<s'$ and $m<m'$. The version of~\eqref{sandwich for AA} with $\Omega M$ yields continuous inclusions
\begin{equation}\label{sandwich for dot AA'}
\dot\AA^{\prime\,(s)}(M)\supset\dot\AA^{\prime\,m}(M)\supset\dot\AA^{\prime\,(\max\{m,0\})}(M)\quad(m>s+n/2+1)\;.
\end{equation}
We also have \cite[Corollary~6.44]{AlvKordyLeichtnam-conormal}
\begin{equation}\label{dot AA'(M) = bigcap_s dot AA^prime (s)(M) = bigcap_m dot AA^prime m(M)}
\dot\AA'(M)=\bigcap_s\dot\AA^{\prime\,(s)}(M)=\bigcap_m\dot\AA^{\prime\,m}(M)\;,
\end{equation}
where the last equality is a consequence of~\eqref{sandwich for dot AA'}.

Transposing the versions of~\eqref{dot C^infty(M) = bigcap_m ge 0 x^m C^infty(M) subset C^infty(M)},~\eqref{dot AA(M) subset dot C^-infty(M)} and~\eqref{C^infty(M) subset AA(M)} with $\Omega M$, we get continuous inclusions \cite[Section~4.6]{Melrose1996}
\begin{align}
C^\infty(M)\subset\AA'(M)\subset C^{-\infty}(M),\dot C^{-\infty}(M)\;,
\label{C^infty(M) subset AA'(M) subset C^-infty(M)}\\
\dot C^\infty(M)\subset\dot\AA'(M)\subset\dot C^{-\infty}(M),C^{-\infty}(M)\;,
\label{dot C^infty(M) subset dot AA'(M) subset dot C^-infty(M)}
\end{align}
and $R:\dot C^{-\infty}(M)\to C^{-\infty}(M)$ restricts to the identity map on $\AA'(M)$ and $\dot\AA'(M)$. The first inclusion of~\eqref{dot C^infty(M) subset dot AA'(M) subset dot C^-infty(M)} is dense; in fact, $\Cinftyc(\mathring M)$ is dense in every $\dot\AA^{\prime\,m}(M)$, and therefore in $\dot\AA'(M)$ \cite[Corollary~6.50 and Remark~6.51]{AlvKordyLeichtnam-conormal}.

\subsection{Dual-conormal sequence at the boundary}\label{ss: dual conomal seq at the boundary}

Transposing maps in the version of~\eqref{0 -> KK(M) -> dot AA(M) -> AA(M) -> 0} with $\Omega M$, we get the sequence
\begin{equation}\label{dual-conormal seq at the boundary}
0\leftarrow\KK'(M) \xleftarrow{R'} \AA'(M) \xleftarrow{\iota'} \dot\AA'(M)\leftarrow0\;,
\end{equation}
where $R'=\iota^\trans$ and $\iota'=R^\trans$. It is called the \emph{dual-conormal sequence at the boundary} \index{dual-conormal sequence at the boundary} of $M$, which is exact in the category of continuous linear maps between LCSs \cite[Proposition~6.45]{AlvKordyLeichtnam-conormal}.

\subsection{$\dot\AA(M)$ and $\AA(M)$ vs $\AA'(M)$}\label{ss: ABC}

Using~\eqref{dot AA(M) subset dot C^-infty(M)},~\eqref{C^infty(M) subset AA(M)} and~\eqref{C^infty(M) subset AA'(M) subset C^-infty(M)}, we have \cite[Proposition~18.3.24]{Hormander1985-III}, \cite[Theorem~4.6.1]{Melrose1996}
\begin{equation}\label{ABC}
\dot\AA(M)\cap\AA'(M)=C^\infty(M)\;. 
\end{equation}

\subsection{A description of $\dot\AA'(M)$}\label{ss: description of dot AA'(M)}

If~\ref{i: g is of bounded geometry} and~\ref{i: A'} are true, then, for $m\in\R$ \cite[Corollaries~6.48 and~6.49]{AlvKordyLeichtnam-conormal},
\begin{gather}
\dot\AA^{\prime\,m}(M)\equiv x^m\Hb^{-\infty}(M)=x^{m-\frac12}H^{-\infty}(\mathring M)\;,
\label{dot AA^prime m(M) equiv ...}\\
\dot\AA'(M)\equiv\bigcap_mx^m\Hb^{-\infty}(M)=\bigcap_mx^mH^{-\infty}(\mathring M)\;.
\label{dot AA'(M) equiv ...}
\end{gather}
The first identities of~\eqref{dot AA^prime m(M) equiv ...} and~\eqref{dot AA'(M) equiv ...} are independent of $g$, and hold without the assumptions~\ref{i: g is of bounded geometry} and~\ref{i: A'}.

\subsection{Action of $\Diff(M)$ on $\AA'(M)$, $\dot\AA'(M)$ and $\KK'(M)$}\label{ss: Diff(M) on AA'(M)}

Any $A\in\Diff(M)$ induces continuous linear endomorphisms $A$ of $\AA'(M)$, $\dot\AA'(M)$ and $\KK'(M)$ \cite[Proposition~4.6.1]{Melrose1996}, which are the transposes of $A^\trans$ on $\dot\AA(M;\Omega)$, $\AA(M;\Omega)$ and $\KK(M;\Omega)$ (\Cref{ss: ops,ss: diff ops on dot AA(M) and AA(M)}). If $A\in\Diff^k(M)$, these maps satisfy the analogs of~\eqref{A: I^prime [s](M L E) -> I^prime (s-m)(M L E)}. If $A\in\Diffb(M)$, it induces continuous endomorphisms of $\AA^{\prime\,(s)}(M)$, $\AA^{\prime\,m}(M)$, $\dot\AA^{\prime\,(s)}(M)$ and $\KK^{\prime\,(s)}(M)$.

\subsection{The b-stretched product}\label{ss: b-stretched product}

Let $Y_1, \ldots, Y_r$ be the connected components of $\partial M$. Consider the submanifold $B:=\bigcup_{j=1}^r Y_j^2$ of the $C^\infty$ manifold with corners $M^2$. Its inward-pointing spherical normal bundle is $S_+NB={}_+NB/\R^+$, where $\R^+$ acts on ${}_+NB$ by multiplication. The \emph{b-stretched product} \index{b-stretched product} $M^2_{\text{\rm b}}$ \index{$M^2_{\text{\rm b}}$} is the compact smooth manifold with corners obtained from $M^2$ by blowing-up $B$ \cite[Sections~4.1 and~4.2]{Melrose1993}, \cite[Chapter~4]{Melrose1996}, with corresponding surjective smooth blow-down map $\beta_{\text{\rm b}}:M^2_{\text{\rm b}}\to M^2$; namely, $M^2_{\text{\rm b}}=S_+NB\sqcup(M^2\setminus B)$, and $\beta_{\text{\rm b}}$ is the combination of the projection $S_+NB\to B$ and the identity map on $M^2\setminus B$. The topology and $C^\infty$ structure of $M^2_{\text{\rm b}}$ can be described as follows. 

For any $C^1$ curve $\chi:[0,1]\to M^2$ with $\chi(0)\in B$ and $\chi((0,1])\subset M^2\setminus B$, let $\tilde\chi:[0,1]\to M^2_{\text{\rm b}}$ be the lift of $\chi$ so that $\tilde\chi(0)$ is defined by $\chi'(0)$. Then a subset $U\subset M^2_{\text{\rm b}}$ is open if it has open intersections with $S_+NB$ and $M^2\setminus B$, and, for any such curve $\chi$ with $\tilde\chi([0,1])\subset U$, we have $\tilde\eta([0,1])\subset U$ for all $C^1$ curve $\eta:[0,1]\to M^2$ of the same type as $\chi$ and $C^1$-close enough to $\chi$. 

Let $x$ and $x'$ denote the lifts to $M^2$ of the boundary-defining function $x$ from the left and right factors. The $C^\infty$ function
\[
\tau:=\frac{x-x'}{x+x'}:M^2\setminus(\partial M)^2\to[-1,1]
\]
has a continuous extension $\tau$ to the open neighborhood $S_+NB\sqcup(M^2\setminus(\partial M)^2)$ of $S_+NB$ in $M^2_{\text{\rm b}}$. Then $C^\infty(M^2_{\text{\rm b}})$ is locally generated by $\tau$ and $\beta_{\text{\rm b}}^*C^\infty(M^2)$.

The manifold with corners $M^2_{\text{\rm b}}$ has three boundary hypersurfaces, \index{$\ff$} \index{$\lb$} \index{$\rb$} \index{front face} \index{left boundary} \index{right boundary}
\[
\ff=\beta_{\text{\rm b}}^{-1}(B)\;,\quad
\lb=\overline{\beta_{\text{\rm b}}^{-1}(\partial M\times\mathring M)}\;,\quad
\rb=\overline{\beta_{\text{\rm b}}^{-1}(\mathring M\times\partial M)}\;,
\]
called the \emph{front face}, and the \emph{left} and \emph{right boundaries}.  They satisfy $\lb\cap\rb=\emptyset$. Another embedded, compact submanifold of $M^2_{\text{\rm b}}$ is the \emph{b-diagonal}, \index{b-diagonal} $\Deltab=\overline{\beta_{\text{\rm b}}^{-1}(\Delta\setminus B)}$, \index{$\Deltab$} where $\Delta\subset M^2$ is the diagonal. We have $\Deltab\pitchfork\ff$, $\Delta_{\text{\rm b},0}:=\Deltab\cap\ff=\partial\Deltab$ and $\Deltab\cap\lb=\Deltab\cap\rb=\emptyset$. Moreover $\beta_{\text{\rm b}}:\Deltab\to\Delta\equiv M$ is a diffeomorphism, where the last identity is given by the diagonal map. 

Let also $x=\beta_{\text{\rm b}}^*x$ and $x'=\beta_{\text{\rm b}}^*x'$ on $M^2_{\text{\rm b}}$. Thus $r=x+x'$ is a defining function of $\ff$ in $M^2_{\text{\rm b}}$ (in the same sense as in \Cref{ss: b-geometry} for $\partial M$). For adapted local coordinates $(x,y)$, the lifts $y$ and $y'$ of $y$ to open subsets of $M^2$ and $M^2_{\text{\rm b}}$ are defined like $x$ and $x'$. Then $(r,\tau,y,y')$ or $(x,\tau,y,y')$ are local coordinates of $M^2_{\text{\rm b}}$ around points of $\ff$, the submanifold $\Deltab$ is locally described by the conditions $\tau=0$ and $y=y'$, and $\Delta_{\text{\rm b},0}$ is locally described by the conditions $r=\tau=0$ and $y=y'$. Other local coordinates $(r,s,y,y')$ or $(x,s,y,y')$ of $M^2_{\text{\rm b}}$ around points of $\mathring\ff$ are defined using the function
\[
s:=\frac{1+\tau}{1-\tau}=\frac{x}{x'}:M^2_{\text{\rm b}}\setminus\rb\to(0,\infty)\;.
\]

With the obvious extensions to manifolds with corners of some concepts of \Cref{ss: b-geometry,ss: Diffb(M)}, we get the following \cite[Section~4.5]{Melrose1993}. First,
\begin{alignat*}{2}
\bT M^2&\equiv(\bT M)^2\;,&\quad\bT M^2_{\text{\rm b}}&\equiv\beta_{\text{\rm b}}^*(\bT M^2)\;,\\
\fXb(M^2)&\equiv C^\infty(M^2;\bT M^2)\;,&\quad \fXb(M^2_{\text{\rm b}})&\equiv C^\infty(M^2;\bT M^2_{\text{\rm b}})\;. 
\end{alignat*}
Second, any vector field in $\fX(M^2,B)$ can be lifted to a vector field in $\fX(M^2,\ff)$; in particular, the lifts to $M^2$ of $\fXb(M)$, from the left and right factors, generate $\fXb(M^2)$ over $C^\infty(M^2)$. Third, there is a lifting map $\beta_{\text{\rm b}}^*:\fXb(M^2)\to\fXb(M^2_{\text{\rm b}})$, whose image spans $\fXb(M^2_{\text{\rm b}})$ over $C^\infty(M^2_{\text{\rm b}})$. It induces a lifting map $\beta_{\text{\rm b}}^*:\Diffb^m(X^2)\to\Diffb^m(M^2_{\text{\rm b}})$, whose image spans $\Diffb^m(M^2_{\text{\rm b}})$ over $C^\infty(M^2_{\text{\rm b}})$ \cite[Exercise~4.11]{Melrose1993}. For instance, using the above local coordinates, the lift of $x\partial_x$ is $\frac{1}{2}(1+\tau)r\partial_r+\frac{1}{2}(1-\tau^2)\partial_\tau$. Finally, the lift to $M^2_{\text{\rm b}}$ of $\fXb(M)$ from the left factor of $M^2$ is a Lie subalgebra of $\fXb(M^2_{\text{\rm b}})$ transverse to $\Deltab$, giving rise to natural isomorphisms $N\Deltab\cong\bT M$ and $N^*\Deltab\cong\bT ^*M$ \cite[Lemmas~4.5 and~4.6]{Melrose1993}. Thus there is a canonical isomorphism $\bOmega^{1/2}(M^2_{\text{\rm b}})|_{\Deltab}\cong\bOmega M$ (cf.\ \cite[Eq.~(4.125)]{Melrose1993}).

\subsection{The b-pseudodifferential operators}\label{ss: b-pseudodiff opers}

A refinement of the Schwartz kernel theorem gives a bijection \cite[Lemma~4.20]{Melrose1993}
\begin{gather*}
L(\dot C^\infty(M;\bOmega^{\frac12}),C^{-\infty}(M;\bOmega^{\frac12}))\to
C^{-\infty}(M^2_{\text{\rm b}}; \beta_{\text{\rm b}}^*(\bOmega^{\frac12}\boxtimes\bOmega^{\frac12}))\;,\\
A\mapsto\kappa_A\;,\quad\langle Au,v\rangle= \langle\kappa_A,\beta_{\text{\rm b}}^*(v\otimes u) \rangle\;,\quad
u,v\in\dot C^\infty(M;\bOmega^{\frac12})\;.
\end{gather*}

The concept of conormal distributional sections can be also extended to submanifolds whose boundary is given by a transverse intersection with the boundary faces, like $\Deltab\subset M^2_{\text{\rm b}}$. Then a continuous linear map $A:\dot C^\infty(M;\bOmega^{1/2})\to C^{-\infty}(M;\bOmega^{1/2})$ is called a \emph{b-pseudodifferential operator} \index{b-pseudodifferential operator} of \emph{order} at most $m\in\R$ if $\kappa_A\in I^m(M^2_{\text{\rm b}},\Deltab;\beta_{\text{\rm b}}^*(\bOmega^{1/2}\boxtimes\bOmega^{1/2}))$ and $\kappa_A$ vanishes to all orders at $\lb\cup\rb$ \cite[Definition~4.22]{Melrose1993}. Such operators form a $C^\infty(M^2_{\text{\rm b}})$-module $\Psib^m(M;\bOmega^{1/2})$, obtaining the filtered $C^\infty(M^2_{\text{\rm b}})$-module $\Psib(M;\bOmega^{1/2})=\bigcup_m\Psib^m(M;\bOmega^{1/2})$. The submodule $\Psib^{-\infty}(M;\bOmega^{1/2}):=\bigcap_m\Psib^m(M;\bOmega^{1/2})$ (resp., $\Diffb(M;\bOmega^{1/2})$) of $\Psib(M;\bOmega^{1/2})$ consists of the operators $A\in\Psib(M;\bOmega^{1/2})$ with smooth $\kappa_A$ (resp., $\supp\kappa_A\subset\Deltab$).  The obvious generalization of the definition of principal symbol, like in \Cref{ss: conormal - symbol order - compact}, now gives the \emph{principal b-symbol} exact sequence, \index{principal b-symbol} \index{$\bsigma_m$}
\[
0\to\Psib^{m-1}(M;\bOmega^{\frac12})\hookrightarrow\Psib^m(M;\bOmega^{\frac12}) 
\xrightarrow{\bsigma_m} S^{(m)}(\bT^*M;\bOmega^{\frac12})\to0\;.
\]
The principal b-symbol is used to define \emph{b-ellipticity} \index{b-elliptic} like ellipticity in the case of pseudodifferential operators (\Cref{ss: pseudodiff ops}).

Omitting $\bOmega^{1/2}$, if $A\in\Psib^{-\infty}(M)$ and $\kappa:=\kappa_A$ is supported in the domain of a chart $(x,s,y,y')$, then we can write $\kappa=\kappa'(x,s,y,y')\,s^{-1}ds\,dy'$ because $\kappa$ is rapidly decreasing as $s\downarrow0$ and as $s\uparrow+\infty$, obtaining 
\begin{equation}\label{Au(x y) = ...}
Au(x,y)=\int_{\partial M}\int_0^\infty\kappa'(x,s,y,y')u\Big(\frac{x}{s},y'\Big)\frac{ds}{s}\,dy'\;,
\end{equation}
for all $u\in\dot C^\infty(M)$ supported in the domain of the chart $(x,y)$.

Any $A\in\Psib^m(M;\bOmega^{1/2})$ defines continuous endomorphisms $A$ of $\dot C^\infty(M;\bOmega^{1/2})$ and $C^{\pm\infty}(M;\bOmega^{1/2})$ \cite[Propositions~4.29 and~4.34 and Exercise 4.33]{Melrose1993}. In this sense, $\Psib(M;\bOmega^{1/2})$ becomes a filtered algebra with the operation of composition \cite[Propositions~5.20]{Melrose1993}, and the principal b-symbol map is multiplicative.

\subsection{The indicial family}\label{s: indicial}

Let $A\in\Psib^m(M;\bOmega^{1/2})$ ($m\in\R$) and write $\kappa=\kappa_A$. Roughly speaking, the \emph{indicial family} \index{indicial family} of $A$ is an entire family, $I_\nu(A,\lambda)\in\Psi^m(\partial M;\Omega^{1/2})$ ($\lambda\in\C)$, \index{$I_\nu(A,\lambda)$} depending on the trivialization $\nu$ of ${}_+N\partial M$ (\Cref{ss: b-geometry}), defined by taking the ``fiberwise'' Mellin transform of certain conormal distributional section defined by $\kappa|_{\ff}$. Thus $I_\nu(A,\lambda)=0$ for all $\lambda$ just when $\kappa|_{\ff}=0$.

The indicial family can be also described as follows. For $z\in\C$ and $m\in\R$, the mapping $A\mapsto x^{-z}Ax^z$ defines an automorphism of $\Psib^m(M;\bOmega^{1/2})$ \cite[Proposition~5.7]{Melrose1993}. Hence every $A\in\Psib(M;\bOmega^{1/2})$ defines a continuous endomorphism $A$ of $x^kC^\infty(M;\bOmega^{1/2})$ ($k\in\N$). Therefore a continuous endomorphism $A_\partial$ of $C^\infty(\partial M;\Omega^{1/2})$ is well defined by $A_\partial v=Au|_{\partial M}$ if $u\in C^\infty(M;\bOmega^{1/2})$ with $u|_{\partial M}=v$ \cite[Eq.~(5.31)]{Melrose1993}. Then \cite[Proposition~5.8]{Melrose1993}
\begin{equation}\label{I_nu(A lambda) = (x^-i lambda A x^i lambda)_partial}
I_\nu(A,\lambda)=(x^{-i\lambda}Ax^{i\lambda})_\partial\;.
\end{equation}

We will only use the indicial family in the following cases, where $\bOmega^{1/2}$ is omitted for the sake of simplicity. First, if $A\in\Psib^{-\infty}(M)$ and $\kappa$ is supported in the domain of a chart $(x,s,y,y')$, as described in~\eqref{Au(x y) = ...}, then $I_\nu(A,\lambda)\in\Psib^{-\infty}(\partial M;\Omega^{1/2})$ is given by
	\begin{equation}\label{K_I_nu(A lambda)(y y')}
		K_{I_\nu(A,\lambda)}(y,y')=\int_0^{\infty}s^{-i\lambda}\kappa'(0,s,y,y')\frac{ds}{s}\;.
	\end{equation}
The support condition can be obviously removed by using a partition of unity or a collar neighborhood of $\partial M$. Second, if $A\in\Diffb^m(M)$ ($m\in\N_0$) is locally given by
	\[
		A=\sum_{j+|I|\le m}a_{j,I}(x,y)(xD_x)^jD^I_y\;,
	\]
using adapted local coordinates $(x,y)$, then
	\begin{equation}\label{I_nu(A lambda) for A in Diffb^m(M)}
		I_\nu(A,\lambda)=\sum_{j+|I|\le m}a_{j,I}(0,y)\lambda^jD^I_y\;.
	\end{equation}

The indicial family is multiplicative  \cite[Corollary of Proposition~5.20]{Melrose1993}, and compatible with the operation of taking formal adjoints of b-differential operators (Cf.~\cite[Eq.~(4.112)]{Melrose1993}).

\subsection{The b-integral}\label{ss: b-int}

The \emph{b-integral} \index{b-integral} is a linear map $\smallnuint=\smallnuint_M:C^1(M;\bOmega)\to\C$, \index{$\smallnuint$} depending on $\nu$, defined by \cite[Lemma~4.62]{Melrose1993}
\[
\nuint u=\lim_{\epsilon\downarrow0}\left(\int_{x\ge\epsilon} u+\ln\epsilon\int_{\partial M} u|_L\right)\;,
\] 
using a boundary-defining function $x$ with $dx(\nu)=1$. Another trivialization $\mu\in C^\infty(M;{}_+N\partial M)$ is of the form $\mu=a\nu$ for some $0<a\in C^\infty(\partial M)$, and 
\[
\muint u-\nuint u=\int_{\partial M}\ln a\cdot u|_{\partial M}\;.
\]

\begin{lem}\label{l: smallnuint is cont}
$\smallnuint$ is continuous with the $C^1$ topology on $C^1(M;\bOmega)$.
\end{lem}

\begin{proof}
Consider a chart $(V,y)$ of $\partial M$, and the adapted local coordinates $(x,y)$ on $U\equiv[0,\epsilon_0)_x\times V\subset M$ ($\epsilon_0>0$). Since $\smallnuint\equiv\int_M$ on $C^1_\co(\mathring M;\bOmega)\equiv C^1_\co(\mathring M;\Omega)$, it is easy to see that it is enough to prove the continuity of $\smallnuint$ on $C^1_\co(U;\bOmega)$. Every $u\in C^1_\co(U;\bOmega)$ is of the form $u(x,y)=a(x,y)\,|x^{-1}dx\,dy|$ for some $a\in C^1_\co(U)$. Then 
\[
\nuint u=\ln\epsilon_0\cdot\int_Va(0,y)\,dy
+ \lim_{\epsilon\downarrow0}\left(\int_V\int_\epsilon^{\epsilon_0}(a(x,y)-a(0,y))\,\frac{dx}{x}\,dy\right)\;.
\]
Hence
\[
\left|\nuint u\right|\le\vol V\cdot\Big(\ln\epsilon_0\cdot\max_{y\in V}|a(0,y)|
+\epsilon_0\cdot\max_{(\xi,y)\in U}|\partial_xa(\xi,y)|\Big)\;.\qedhere
\]
\end{proof}

\begin{cor}\label{c: smallnuint is cont}
Let $T$ and $T'$ be collar neighborhoods of $\partial M$ in $M$ with $\overline{T'}\subset T$. For any sequence $u_k$ in $C^1(M;\bOmega)$, if $u_k|_T\to0$ in $C^1(T;\bOmega)$ and $\lim_k\int_{M\setminus T'}u_k=a\in\C$, then $\lim_k\smallnuint u_k=a$.
\end{cor}

\begin{proof}
Let $\{\lambda,\mu\}$ be a smooth partition of unity of $M$ subordinated to the open covering $\{T,M\setminus\overline{T'}\}$. Then $\lim_k\lambda u_k=0$ in $C^1(M;\bOmega)$, obtaining $\lim_k\smallnuint\lambda u_k=0$ by \Cref{l: smallnuint is cont}. Moreover $\lim_k\int_{M\setminus T'}\lambda u_k=0$. Therefore
\[
\lim_k\nuint u_k=\lim_k\nuint\mu u_k=\lim_k\int_{M\setminus T'}\mu u_k
=\lim_k\int_{M\setminus T'}u_k=a\;.\qedhere
\]
\end{proof}

\begin{rem}\label{r: smallnuint is cont}
Consider a collar neighborhood of $\partial M$ in $M$ of the form $T\equiv[0,\epsilon)_x\times\partial M_\varpi$, and the intermediate space
\begin{align*}
C^1(T;\bOmega)&\equiv C^1([0,\epsilon),C^1(\partial M;\bOmega_{\partial M}))\\
&\subset C^{0,1}_\varpi(T;\bOmega):=C^1([0,\epsilon),C^0(\partial M;\bOmega_{\partial M}))\\
&\subset C^0(T;\bOmega)\equiv C^0([0,\epsilon),C^0(\partial M;\bOmega_{\partial M}))\;.
\end{align*}
Then $\smallnuint_M$ is actually defined on
\[
\{\,u\in C^0(M;\bOmega)\mid u|_T\in C^{0,1}_\varpi(T;\bOmega)\,\}\;,
\]
and the proof of \Cref{l: smallnuint is cont} shows that it is continuous with the obvious topology defined by the topologies of $C^0(M;\bOmega)$ and $C^{0,1}_\varpi(T;\bOmega)$. So Corollary~\ref{c: smallnuint is cont} is true with the weaker condition $u_k|_T\to0$ in $C^{0,1}_\varpi(T;\bOmega)$.
\end{rem}

\subsection{The b-trace}\label{ss: b-trace}

Any $A\in\Psib^{-\infty}(M;\bOmega^{1/2})$ is of trace class  if and only if $A\in r\Psib^{-\infty}(M;\bOmega^{1/2})$ (i.e., $\kappa_A|_{\ff}=0$). The \emph{b-trace} $\bTr:\Psib^{-\infty}(M;\bOmega^{1/2})\to\C$ is an extension of the trace $\Tr:r\Psib^{-\infty}(M;\bOmega^{1/2})\to\C$ given by \index{b-trace} \index{$\bTr$}
\[
\bTr A=\nuint_M\kappa_A|_{\Deltab}\;,
\]
using the canonical isomorphism $\bOmega^{1/2}(M^2_{\text{\rm b}})|_{\Deltab}\cong\bOmega M$ (\Cref{ss: b-stretched product}). If $A,B\in\Psi^{-\infty}_{\text{\rm b}}(M;\bOmega^{1/2})$, then \cite[Proposition 5.9]{Melrose1993}
\begin{equation}\label{bTr[A,B]}
\bTr[A,B]=-\frac{1}{2\pi i} \int_{-\infty}^{+\infty}\Tr(\partial_{\lambda}I_\nu(A,\lambda)\, I_\nu(B,\lambda))\,d\lambda\;.
\end{equation}
This equality also holds if $A\in\Diffb(M;\bOmega^{1/2})$ and $B\in\Psib^{-\infty}(M;\bOmega^{1/2})$ \cite[Lemma 5.10]{Melrose1993}.

If $E$ is a $\Z_2$-graded Hermitian vector bundle over $M$ with degree involution $\sw$ ($\sw u=(-1)^ku$ for $u\in E^k$ and $k\in\Z_2$), and $A\in\Psi^{-\infty}_{\text{\rm b}}(M;E)$ is homogeneous of degree zero, then its \emph{b-supertrace} is $\bStr A=\bTr(A\sw)$. \index{b-supertrace} \index{$\bStr$} This notion extends the supertrace $\Str(B)$ of any homogeneous operator $B\in r\Psi^{-\infty}_{\text{\rm b}}(M;E)$ of degree zero. \index{$\Str$}

\section{Conormal sequence}\label{s: conormal seq}

In this section and the next one, for the sake of simplicity, we only consider submanifolds of codimension one because that is the only case we use. However, the results can be extended to submanifolds of arbitrary codimension with more work.

\subsection{Cutting along a submanifold}\label{ss: cutting}

Again, for brevity reasons, we consider only the case of a closed manifold and the trivial line bundle. Like in other sections, the spaces of distributions we are going to define have obvious extensions to non-compact manifolds and arbitrary vector bundles, taking compact support or no support conditions, and taking regular submanifolds that are closed subspaces. We will consider those types of extensions without further comment.

Let $M$ be a closed connected manifold, and $L\subset M$ be a (possibly non-connected) regular closed submanifold of codimension one. $M\setminus L$ may have several connected components. First assume also that $L$ is transversely oriented. Then, like in the case with boundary (\Cref{ss: b-geometry}), there is some real-valued smooth function $x$ on some tubular neighborhood $T$ of $L$ in $M$, with projection $\varpi:T\to L$, so that $L=\{x=0\}$ and $dx\ne0$ on $L$. Any function $x$ satisfying these conditions is called a \emph{defining function} \index{defining function} of $L$ on $T$. We can suppose $T\equiv(-\epsilon,\epsilon)_x\times L_\varpi$, for some $\epsilon>0$. (If $M$ and $L$ were not compact, and $L$ were a regular submanifold that is a closed subset, then the tubular neighborhood would have a more involved expression, using a smooth positive function $\epsilon(y)$ on $L$ instead of a fixed positive number $\epsilon$.) For any atlas $\{V_j,y_j\}$ of $L$, we get an atlas of $T$ of the form $\{U_j\equiv(-\epsilon,\epsilon)\times V_j,(x,y)\}$, whose charts are adapted to $L$. The corresponding local vector fields $\partial_x\in\fX(U_j)$ can be combined to define a vector field $\partial_x\in\fX(T)$; we can consider $\partial_x$ as the derivative operator on $C^\infty(T)\equiv C^\infty((-\epsilon,\epsilon),C^\infty(L))$. For every $j$, $\Diff(U_j,L\cap U_j)$ is spanned by $x\partial_x,\partial_j^1,\dots,\partial_j^{n-1}$ using the operations of $C^\infty(U_j)$-module and algebra, where $\partial_j^k=\partial/\partial y_j^k$. Using $T\equiv(-\epsilon,\epsilon)_x\times L$, any $A\in\Diff(L)$ induces an operator $1\otimes A\in\Diff(T,L)$, such that $(1\otimes A)(u(x)v(y))=u(x)\,(Av)(y)$ for $u\in C^\infty(-\epsilon,\epsilon)$ and $v\in C^\infty(L)$. This defines a canonical injection $\Diff(L)\equiv1\otimes\Diff(L)\subset\Diff(T,L)$ so that $(1\otimes A)|_L=A$. (This also shows the surjectivity of~\eqref{Diff(M L) -> Diff(L)} in this case.) Moreover $\Diff(T)$ (resp., $\Diff(T,L)$) is spanned by $\partial_x$ (resp., $x\partial_x$) and $1\otimes\Diff(L)$ using the operations of $C^\infty(T)$-module and algebra. Clearly,
\begin{equation}\label{[partial_x 1 otimes Diff(L)] = 0}
[\partial_x,1\otimes\Diff(L)]=0\;,\quad[\partial_x, x\partial_x]=\partial_x\;,
\end{equation}
yielding
\begin{equation}\label{[partial_x Diff^k(T L)] subset ...}
[\partial_x,\Diff^k(T,L)]\subset\Diff^k(T,L)+\Diff^{k-1}(T,L)\,\partial_x\;.
\end{equation}
$\Diff^k(T,L)$ and $\Diff^k(T)$ satisfy the obvious versions of~\eqref{Diffb^k(M) x^a = x^a Diffb^k(M)} and~\eqref{Diff^k(M) x^a subset x^a-k Diff^k(M)}.

For a vector bundle $E$ over $M$, there is an identity $E_T\equiv(-\epsilon,\epsilon)\times E_L$ over $T\equiv(-\epsilon,\epsilon)\times L$, which can be used to define $\partial_x\in\Diff^1(T;E)$. With this interpretation of $\partial_x$ and using tensor products like in~\eqref{C^infty(M)-tensor product description of C^pm infty_./c(M E)}, the vector bundle versions of the concepts and properties of this section are straightforward.

Let $\bfM$ \index{$\bfM$} be the smooth manifold with boundary defined by ``cutting'' $M$ along $L$; i.e., modifying $M$ only on the tubular neighborhood $T\equiv(-\epsilon,\epsilon)\times L$, which is replaced with $\bfT=((-\epsilon,0]\sqcup[0,\epsilon))\times L$ in the obvious way. ($\bfM$ is the blowing-up $[M,L]$ of $M$ along $L$ \cite[Chapter~5]{Melrose1996}.) Thus $\partial\bfM\equiv L\sqcup L$ because $L$ is transversely oriented, and $\mathring\bfM\equiv M\setminus L$. A canonical projection $\bfpi:\bfM\to M$ \index{$\bfpi$} is defined as the combination of the identity map $\mathring\bfM\to M\setminus L$ and the map $\bfT\to T$ given by the product of the canonical projection $(-\epsilon,0]\sqcup[0,\epsilon)\to(-\epsilon,\epsilon)$ and $\id_L$. This projection realizes $M$ as a quotient space of $\bfM$ by the equivalence relation defined by the homeomorphism $h\equiv h_0\times\id$ of $\partial\bfM\equiv\partial\bfT=(\{0\}\sqcup\{0\})\times L$, where $h_0$ switches the two points of $\{0\}\sqcup\{0\}$. Moreover $\bfpi:\bfM\to M$ is a local embedding of a compact manifold with boundary to a closed manifold of the same dimension. 

Like in \Cref{ss: pull-back and push-forward of distrib sections}, we have the continuous linear pull-back map
\begin{equation}\label{bfpi^*: C^infty(M) -> C^infty(bfM)}
\bfpi^*:C^\infty(M)\to C^\infty(\bfM)\;,
\end{equation}
which is clearly injective. The transpose of the version of~\eqref{bfpi^*: C^infty(M) -> C^infty(bfM)} with $\Omega M$ and $\Omega\bfM\equiv\bfpi^*\Omega M$ is the continuous linear push-forward map
\begin{equation}\label{bfpi_*: dot C^-infty(bfM) -> C^-infty(M)}
\bfpi_*:\dot C^{-\infty}(\bfM)\to C^{-\infty}(M)\;,
\end{equation}
which is a surjective topological homomorphism \cite[Proposition~7.4]{AlvKordyLeichtnam-conormal}.

After distinguishing a connected component $L_0$ of $L$, let $\widetilde M$ \index{$\widetilde M$} and $\widetilde L$ be the quotients of $\bfM\sqcup\bfM\equiv\bfM\times\Z_2$ and $\partial\bfM\sqcup\partial\bfM\equiv\partial\bfM\times\Z_2$ by the equivalence relation generated by $(p,a)\sim(h(p),a)$ if $\bfpi(p)\in L\setminus L_0$ and $(p,a)\sim(h(p),a+1)$ if $\bfpi(p)\in L_0$ ($p\in\bfpi^{-1}(L)=\partial\bfM$ in both cases). Let us remark that $\widetilde M$ may not be homeomorphic to the double of $\bfM$, which is the quotient of $\bfM\times\Z_2$ by the equivalence relation generated by $(p,0)\sim(p,1)$, for $p\in\partial\bfM$. Note that $\widetilde M$ is a closed connected manifold and $\widetilde L$ is a closed regular submanifold. Moreover the quotient $\widetilde T$ of $\bfT\sqcup\bfT$ becomes a tubular neighborhood of $\widetilde L$ in $\widetilde M$. The combination $\bfpi\sqcup\bfpi:\bfM\sqcup\bfM\to M$ induces a two-fold covering map $\tilde\pi:\widetilde M\to M$, whose restrictions to $\widetilde L$ and $\widetilde T$ are trivial two-fold coverings of $L$ and $T$, respectively; i.e., $\widetilde L\equiv L\sqcup L$ and $\widetilde T\equiv T\sqcup T$. The group of deck transformations of $\tilde\pi:\widetilde M\to M$ is $\{\id,\sigma\}$, where $\sigma:\widetilde M\to\widetilde M$ is induced by the map $\sigma_0:\bfM\times\Z_2\to\bfM\times\Z_2$ defined by switching the elements of $\Z_2$. The composition of the injection $\bfM\to\bfM\times\Z_2$, $p\mapsto(p,0)$, with the quotient map $\bfM\sqcup\bfM\to\widetilde M$ is a smooth embedding $\bfM\to\widetilde M$. This will be considered as an inclusion map of a regular submanifold with boundary, obtaining $\partial\bfM\equiv\widetilde L$.

Since $\tilde\pi$ is a two-fold covering map, we have continuous linear maps (\Cref{ss: pull-back and push-forward of distrib sections})
\begin{gather}
\tilde\pi_*:C^\infty(\widetilde M) \to C^\infty(M)\;,\quad
\tilde\pi^*:C^\infty(M) \to C^\infty(\widetilde M)\;,\notag \\
\tilde\pi^*:C^{-\infty}(M) \to C^{-\infty}(\widetilde M)\;,\quad
\tilde\pi_*:C^{-\infty}(\widetilde M) \to C^{-\infty}(M)\;,\label{tilde pi^*: C^-infty(M) -> C^-infty(widetilde M)}
\end{gather}
both pairs of maps satisfying
\begin{equation}\label{tilde pi_* tilde pi^* = 2}
\tilde\pi_*\tilde\pi^*=2\;,\quad\tilde\pi^*\tilde\pi_*=A_\sigma\;,
\end{equation}
where $A_\sigma:C^{\pm\infty}(\widetilde M)\to C^{\pm\infty}(\widetilde M)$ is given by $A_\sigma u=u+\sigma_*u$. Using the continuous linear restriction and inclusion maps given by~\eqref{R: C^infty(widetilde M) -> C^infty(M)} and~\eqref{dot C^-infty(M) subset C^-infty(breve M)}, we get the commutative diagrams
\begin{equation}\label{CD: tilde pi_* iota = bfpi_*}
\begin{CD}
C^\infty(\widetilde M) @>R>> C^\infty(\bfM) \\
@A{\tilde\pi^*}AA @AA{\bfpi^*}A \\
C^\infty(M) @= C^\infty(M)\;,\hspace{-.2cm}
\end{CD}
\qquad
\begin{CD}
\dot C^{-\infty}(\bfM) @>{\iota}>> C^{-\infty}(\widetilde M) \\
@V{\bfpi_*}VV @VV{\tilde\pi_*}V \\
C^{-\infty}(M) @= C^{-\infty}(M)\;,\hspace{-.2cm}
\end{CD}
\end{equation}
the second one is the transpose of the density-bundles version of the first one.

\subsection{Lift of differential operators from $M$ to $\widetilde M$}\label{ss: lift of diff ops}

For any $A\in\Diff(M)$, let $\widetilde A\in\Diff(\widetilde M)$ denote its lift via the covering map $\tilde\pi:\widetilde M\to M$.  The action of $\widetilde A$ on $C^{\pm\infty}(\widetilde M)$ corresponds to the action of $A$ on $C^{\pm\infty}(M)$ via $\tilde\pi^*:C^{\pm\infty}(M)\to C^{\pm\infty}(\widetilde M)$ and $\tilde\pi_*:C^{\pm\infty}(\widetilde M)\to C^{\pm\infty}(M)$. According to~\eqref{restriction map Diff(breve M) -> Diff(M)}, $\widetilde A|_{\bfM}\in\Diff(\bfM)$ is the lift of $A$ via the local embedding $\bfpi:\bfM\to M$, sometimes also denoted by $\widetilde A$. \index{$\widetilde A$} The action of $\widetilde A$ on $C^\infty(\bfM)$ (resp., $C^{-\infty}(\bfM)$) corresponds to the action of $A$ on $C^\infty(M)$ (resp., $C^{-\infty}(M)$) via $\bfpi^*:C^\infty(M)\to C^\infty(\bfM)$ (resp., $\bfpi_*:C^{-\infty}(\bfM)\to C^{-\infty}(M)$). If $A\in\Diff(M,L)$, then $\widetilde A\in\Diff(\widetilde M,\widetilde L)$ and $\widetilde A|_{\bfM}\in\Diffb(\bfM)$ by~\eqref{Diff(breve M,partial M) = Diffb(M)}.

\subsection{The spaces $C^{\pm\infty}(M,L)$}\label{ss: C^-infty(M L)}

Consider the closed subspaces, \index{$C^\infty(M,L)$} \index{$C^k(M,L)$}
\begin{equation}\label{C^infty(M L) subset C^infty(M)}
C^\infty(M,L)\subset C^\infty(M)\;,\quad C^k(M,L)\subset C^k(M)\quad(k\in\N_0)\;,
\end{equation}
consisting of functions that vanish to all orders at the points of $L$ in the first case, and that vanish up to order $k$ at the points of $L$ in the second case. Then let \index{$C^{-\infty}(M,L)$} \index{$C^{\prime\,-k}(M,L)$}
\[
C^{-\infty}(M,L)=C^\infty(M,L;\Omega)'\;,\quad C^{\prime\,-k}(M,L)=C^k(M,L;\Omega)'\;.
\]
$C^{-\infty}(M,L)$ is a barreled, ultrabornological, webbed, acyclic DF Montel space, and therefore complete, boundedly/compactly/sequentially retractive and reflexive \cite[Corollary~7.1]{AlvKordyLeichtnam-conormal}. Note that~\eqref{bfpi^*: C^infty(M) -> C^infty(bfM)} restricts to TVS-isomorphisms
\begin{equation}\label{bfpi^*: C^infty(M L) cong dot C^infty(bfM)}
\bfpi^*:C^\infty(M,L)\xrightarrow{\cong}\dot C^\infty(\bfM)\;,\quad
\bfpi^*:C^k(M,L)\xrightarrow{\cong}\dot C^k(\bfM)\;.
\end{equation}
Taking the transposes of its versions with density bundles, it follows that~\eqref{bfpi_*: dot C^-infty(bfM) -> C^-infty(M)} restricts to TVS-isomorphisms
\begin{equation}\label{bfpi_*: C^-infty(bfM) cong C^-infty(M L)}
\bfpi_*:C^{-\infty}(\bfM)\xrightarrow{\cong}C^{-\infty}(M,L)\;,\quad
\bfpi_*:C^{\prime\,-k}(\bfM)\xrightarrow{\cong}C^{\prime\,-k}(M,L)\;.
\end{equation}
So the spaces $C^\infty(M,L)$, $C^k(M,L)$, $C^{-\infty}(M,L)$ and $C^{\prime\,-k}(M,L)$ satisfy the analogs of~\eqref{C^prime -k'(M E) supset C^prime -k(M E)} and~\eqref{bigcap_k C^k_./c(M) = C^infty_./c(M)}.

On the other hand, there are Hilbertian spaces $H^r(M,L)$ ($r>n/2$) \index{$H^r(M,L)$} and $H^{\prime\,s}(M,L)$ ($s\in\R$), \index{$H^{\prime\,s}(M,L)$} continuously included in $C^0(M,L)$ and $C^{-\infty}(M,L)$, resp., such that the second map of~\eqref{bfpi^*: C^infty(M L) cong dot C^infty(bfM)} for $k=0$ and the first map of~\eqref{bfpi_*: C^-infty(bfM) cong C^-infty(M L)} restrict to a TVS-isomorphisms
\begin{equation}\label{bfpi^*: H^r(M L) cong dot H^r(bfM)}
\bfpi^*:H^r(M,L)\xrightarrow{\cong}\dot H^r(\bfM)\;,\quad
\bfpi_*:H^s(\bfM)\xrightarrow{\cong}H^{\prime\,s}(M,L)\;.
\end{equation}
For $s=0$, the second TVS-isomorphism of~\eqref{bfpi^*: H^r(M L) cong dot H^r(bfM)} becomes
\begin{equation}\label{bfpi_*: L^2(bfM) cong L^2(M)}
\bfpi_*:L^2(\bfM)\xrightarrow{\cong}L^2(M)\;.
\end{equation}
By~\eqref{dot H^s(M) equiv H^-s(M Omega)'},
\begin{equation}\label{H^prime -r(M,L) equiv H^r(M L Omega)'}
H^{\prime\,-r}(M,L)\equiv H^r(M,L;\Omega)'\;,\quad H^r(M,L)\equiv H^{\prime\,-r}(M,L;\Omega)'\;.
\end{equation}
Now, the second identity of~\eqref{H^prime -r(M,L) equiv H^r(M L Omega)'} can be used to extend the definition of $H^r(M,L)$ for all $r\in\R$.

Alternatively, we may also use trace theorems \cite[Theorem~7.53 and~7.58]{Adams1975} to define $H^m(M,L)$ for $m\in\Z^+$, and then use the first identity of~\eqref{H^prime -r(M,L) equiv H^r(M L Omega)'} to define $H^{\prime\,-m}(M,L)$.

From~\eqref{bfpi^*: C^infty(M) -> C^infty(bfM)},~\eqref{bfpi_*: dot C^-infty(bfM) -> C^-infty(M)},~\eqref{bfpi^*: H^r(M L) cong dot H^r(bfM)} and the analogs of~\eqref{H^s(M E) subset C^k(M E) subset H^k(M E)}--\eqref{C^infty(M E) = bigcap_s H^s(M E)} mentioned in \Cref{ss: supported and extendible Sobolev sps}, we get
\begin{alignat}{2}
C^\infty(M,L)&=\bigcap_rH^r(M,L)\;,&\quad C^{-\infty}(M,L)&=\bigcup_sH^{\prime\,s}(M,L)\;,
\label{C^infty(M L) = bigcap_r H^r(M L)}\\
\intertext{as well as a continuous inclusion and a continuous linear surjection,}
C^\infty(M)&\subset\bigcap_sH^{\prime\,s}(M,L)\;,&\quad C^{-\infty}(M)&\leftarrow\bigcup_rH^r(M,L)\;.
\label{C^infty(M) subset bigcap_sH^prime s(M L)}
\end{alignat}
By~\eqref{H^prime -r(M,L) equiv H^r(M L Omega)'} and~\eqref{C^infty(M L) = bigcap_r H^r(M L)},
\begin{equation}\label{C^infty(M L) = C^-infty(M L Omega)'}
C^\infty(M,L)=C^{-\infty}(M,L;\Omega)'\;.
\end{equation}

The transpose of the version of the first inclusion of~\eqref{C^infty(M L) subset C^infty(M)} with $\Omega M$ is a  surjective topological homomorphism \cite[Proposition~7.4]{AlvKordyLeichtnam-conormal}
\begin{equation}\label{R :C^-infty(M) -> C^-infty(M L)}
R:C^{-\infty}(M)\to C^{-\infty}(M,L)\;,
\end{equation}
whose restriction to $C^\infty(M)$ is the identity. It can be also described as the composition
\[
C^{-\infty}(M) \xrightarrow{\tilde\pi^*} C^{-\infty}(\widetilde M) \xrightarrow{R} C^{-\infty}(\bfM) \xrightarrow{\bfpi_*} C^{-\infty}(M,L)\;.
\]
The canonical pairing between $C^\infty(M)$ and $C^\infty(M,L;\Omega)$ defines a continuous dense inclusion
\begin{equation}\label{C^infty(M) subset C^-infty(M L)}
C^\infty(M)\subset C^{-\infty}(M,L)
\end{equation}
such that~\eqref{R :C^-infty(M) -> C^-infty(M L)} is the identity on $C^\infty(M)$. We also get commutative diagrams
\begin{equation}\label{CD dot C^-infty(bfM) -R-> C^-infty(bfM)}
\begin{CD}
C^\infty(\bfM) @<{\iota}<< \dot C^\infty(\bfM) \\
@A{\bfpi^*}AA @A{\cong}A{\bfpi^*}A \\
C^\infty(M) @<{\iota}<< C^\infty(M,L)\;,\hspace{-.2cm}
\end{CD}
\qquad
\begin{CD}
\dot C^{-\infty}(\bfM) @>R>> C^{-\infty}(\bfM) \\
@V{\bfpi_*}VV @V{\cong}V{\bfpi_*}V \\
C^{-\infty}(M) @>R>> C^{-\infty}(M,L)\;,\hspace{-.2cm}
\end{CD}
\end{equation}
the second one is the transpose of the density-bundles version of the first one.

\subsection{The space $C^{-\infty}_L(M)$}\label{ss: C^-infty_L(M)}

The condition of being supported in $L$ define closed subspaces, \index{$C^{-\infty}_L(M)$} \index{$C^{\prime\,-k}_L(M)$} \index{$H^s_L(M)$}
\[
C^{-\infty}_L(M)\subset C^{-\infty}(M)\;,\quad C^{\prime\,-k}_L(M)\subset C^{\prime\,-k}(M)\;,\quad H^s_L(M)\subset H^s(M)\;,
\]
which are the null spaces of restrictions of~\eqref{R :C^-infty(M) -> C^-infty(M L)}. These spaces satisfy continuous inclusions analogous to~\eqref{C^prime -k'(M E) supset C^prime -k(M E)},~\eqref{H^s(M E) subset H^s'(M E)} and~\eqref{H^-s(M E) supset C^prime -k(M E) supset H^-k(M E)}. The following properties hold \cite[Corollaries~7.2 and~7.3]{AlvKordyLeichtnam-conormal}: $C^{-\infty}_L(M)$ is a limit subspace of the LF-space $C^{-\infty}(M)$; and it is a barreled, ultrabornological, webbed, acyclic DF Montel space, and therefore complete, boundedly/compactly/sequentially retractive and reflexive.

According to~\eqref{dot C^-infty_partial M(M) equiv C^-infty_partial M(breve M)} and \Cref{ss: cutting}, we have \cite[Eq.~(7.19)]{AlvKordyLeichtnam-conormal}
\begin{equation}\label{dot C^-infty_partial bfM(bfM) equiv C^-infty_L(M) oplus C^-infty_L(M)}
\dot C^{-\infty}_{\partial\bfM}(\bfM)\equiv C^{-\infty}_L(M)\oplus C^{-\infty}_L(M)\;,
\end{equation}
The maps~\eqref{bfpi_*: dot C^-infty(bfM) -> C^-infty(M)} and~\eqref{tilde pi^*: C^-infty(M) -> C^-infty(widetilde M)} have restrictions
\begin{equation}\label{bfpi_*: dot C^-infty_partial bfM(bfM) -> C^-infty_L(M)}
\bfpi_*=\tilde\pi_*:\dot C^{-\infty}_{\partial\bfM}(\bfM)\to C^{-\infty}_L(M)\;,\quad
\tilde\pi^*:C^{-\infty}_L(M)\to\dot C^{-\infty}_{\partial\bfM}(\bfM)\;.
\end{equation}
Using~\eqref{dot C^-infty_partial bfM(bfM) equiv C^-infty_L(M) oplus C^-infty_L(M)}, these maps are given by $\bfpi_*(u,v)=u+v$ and $\tilde\pi^*u=(u,u)$.

Moreover the right-hand side diagram of~\eqref{CD dot C^-infty(bfM) -R-> C^-infty(bfM)} can be completed to get the commutative diagram
\begin{equation}\label{CD: ... dot C^-infty(bfM) -R-> C^-infty(bfM) -> 0}
\begin{CD}
0\to\dot C^{-\infty}_{\partial\bfM}(\bfM) @>{\iota}>> \dot C^{-\infty}(\bfM) @>R>> C^{-\infty}(\bfM)\to0 \\
@V{\bfpi_*}VV @V{\bfpi_*}VV @V{\cong}V{\bfpi_*}V \\
0\to C^{-\infty}_L(M) @>{\iota}>> C^{-\infty}(M) @>R>> C^{-\infty}(M,L)\to0\;.\hspace{-.2cm}
\end{CD}
\end{equation}
The bottom row of this diagram is exact in the category of continuous linear maps between LCSs by the properties of~\eqref{R :C^-infty(M) -> C^-infty(M L)}.

\subsection{A description of $C^{-\infty}_L(M)$}\label{ss: description of C^-infty_L(M)}

According to~\eqref{u in C^-infty_c mapsto delta_L^u} and \Cref{ss: diff ops,ss: cutting}, we have TVS-isomorphisms 
\begin{equation}\label{C^-infty(L Omega^-1 NL) cong partial_x^m C^-infty(L Omega^-1 NL)}
\partial_x^m:C^{-\infty}(L;\Omega^{-1}NL)\xrightarrow{\cong}\partial_x^mC^{-\infty}(L;\Omega^{-1}NL)
\subset C^{-\infty}_L(M)\;,
\end{equation}
for $m\in\N_0$, inducing TVS-isomorphisms \cite[Proposition~7.7]{AlvKordyLeichtnam-conormal}
\begin{gather}
\bigoplus_{m=0}^\infty C^{-\infty}(L;\Omega^{-1}NL)\xrightarrow{\cong}C^{-\infty}_L(M)\;,
\label{bigoplus_m C^0_m -> C^-infty_L(M)}\\
\bigoplus_{m=0}^kC^{m-k}(L;\Omega^{-1}NL)\xrightarrow{\cong}C^{\prime\,-k}_L(M)\quad(k\in\N_0)\;.
\label{bigoplus_m^k C^m-k(L Omega^-1NL) cong C^prime -k_L(M)}
\end{gather}

\begin{rem}[{See \cite[Exercise~3.3.18]{Melrose1996}}]\label{r: description of dot C^infty_partial M(M)}
In \Cref{ss: dot C^-infty_partial M(M)}, for any compact manifold with boundary $M$, the analogs of \eqref{bigoplus_m C^0_m -> C^-infty_L(M)} and~\eqref{bigoplus_m^k C^m-k(L Omega^-1NL) cong C^prime -k_L(M)} for $\dot C^{-\infty}_{\partial M}(M)$ follows from their application to $C^{-\infty}_{\partial M}(\breve M)$.
\end{rem}

\subsection{Action of $\Diff(M)$ on $C^{-\infty}(M,L)$ and $C^{-\infty}_L(M)$}
\label{ss: action of Diff(M) on C^-infty(M L) and C^-infty_L(M)}

For every $A\in\Diff(M)$, $A^\trans$ preserves $C^\infty(M,L;\Omega)$, and therefore $A$ induces a continuous linear map $A=A^{tt}$ on $C^{-\infty}(M,L)$. By locality, it restricts to a continuous endomorphism $A$ of $C^{-\infty}_L(M)$.

\subsection{The space $J(M,L)$}\label{ss: J(M L)}

According to \Cref{ss: conormality at the boundary - Sobolev order,ss: C^-infty(M L)}, there is a LCHS $J(M,L)$, with a dense continuous inclusion
\begin{equation}\label{J^(infty)(M L) subset C^-infty(M L}
J(M,L)\subset C^{-\infty}(M,L)\;,
\end{equation}
so that~\eqref{bfpi_*: C^-infty(bfM) cong C^-infty(M L)} restricts to a TVS-isomorphism \index{$J(M,L)$}
\begin{equation}\label{bfpi_*: AA(bfM) cong J(M L)}
\bfpi_*:\AA(\bfM)\xrightarrow{\cong}J(M,L)\;,
\end{equation}
where $\AA(\bfM)$ is defined in~\eqref{dot AA(M) - AA(M)}. By~\eqref{dot AA(M)|_mathring M AA(M) subset C^infty(mathring M)}, there is a continuous inclusion
\[
J(M,L)\subset C^\infty(M\setminus L)\;.
\]
We also get spaces $J^{(s)}(M,L)$ \index{$J^{(s)}(M,L)$} and $J^m(M,L)$ \index{$J^m(M,L)$} ($s,m\in\R$) corresponding to $\AA^{(s)}(\bfM)$ and $\AA^m(\bfM)$ via~\eqref{bfpi_*: AA(bfM) cong J(M L)}. Let $\bfx$ be an extension of $|x|$ to $M$ so that it is positive and smooth on $M\setminus L$. Its lift $\bfpi^*\bfx$ is a boundary-defining function of $\bfM$, also denoted by $\bfx$. Using the first map of~\eqref{bfpi_*: C^-infty(bfM) cong C^-infty(M L)} and second map of~\eqref{bfpi^*: H^r(M L) cong dot H^r(bfM)}, and according to \Cref{ss: lift of diff ops}, we can also describe
\begin{align}
J^{(s)}(M,L)&=\{\,u\in C^{-\infty}(M,L)\mid\Diff(M,L)\,u\subset H^{\prime\,s}(M,L)\,\}\;,\label{J^(s)(M,L)}\\
J^m(M,L)&=\{\,u\in C^{-\infty}(M,L)\mid\Diff(M,L)\,u\subset\bfx^mL^\infty(M)\,\}\;,\notag
\end{align}
with topologies like in~\eqref{Z = u in bigcup_A in AA dom A | AA cdot u subset Y}. These spaces satisfy the analogs of~\eqref{I^(s)(M L) subset I^(s')(M L)},~\eqref{dot AA(M) - AA(M)} and~\eqref{AA^m(M) subset AA^m'(M)}--\eqref{AA(M) = bigcup_m AA^m(M)}. Using~\eqref{C^infty(M) subset bigcap_sH^prime s(M L)} and~\eqref{J^(s)(M,L)}, we get a continuous dense inclusion \cite[Corollary~7.14]{AlvKordyLeichtnam-conormal}\index{$J^{(\infty)}(M,L)$}
\begin{equation}\label{C^infty(M) subset J^(infty)(M L)}
C^\infty(M)\subset J(M,L)\;.
\end{equation}
In fact, $\Cinftyc(M\setminus L)$ is dense in every $J^{(s)}(M,L)$ and $J^m(M,L)$, and therefore in $J(M,L)$ \cite[Corollaries~7.14 and~7.17 and the analog of Remark~6.41 for $J(M,L)$]{AlvKordyLeichtnam-conormal}. Moreover the following properties hold \cite[Corollaries~7.11--7.13 and~7.15]{AlvKordyLeichtnam-conormal}: every $J^{(s)}(M, L)$ is a totally reflexive Fr\'echet space; $J(M,L)$ is barreled, ultrabornological, webbed and an acyclic Montel space, and therefore complete, boundedly/compactly/sequentially retractive and reflexive; and the topologies of $J(M,L)$ and $C^\infty(M\setminus L)$ coincide on every $J^m(M,L)$.

\subsection{A description of $J(M,L)$}\label{ss: description of J(M L)}

Take a b-metric $\bfg$ on $\bfM$ satisfying~\ref{i: g is of bounded geometry} and~\ref{i: A'} (\Cref{ss: a description of AA(M)}), and consider its restriction to $\mathring\bfM$. Consider also the boundary-defining function $\bfx$ of $\bfM$ (\Cref{ss: J(M L)}). Taking $m\in\R$, we have TVS-isomorphisms \cite[Corollaries~7.16 and~7.18]{AlvKordyLeichtnam-conormal}
\begin{gather}
J^m(M,L)\cong\bfx^m\Hb^\infty(\bfM)\equiv \bfx^{m+1/2}H^\infty(\mathring\bfM)\;,
\label{J^m(M L) cong ...}\\
J(M,L)\cong\bigcup_m\bfx^m\Hb^\infty(\bfM)=\bigcup_m\bfx^mH^\infty(\mathring\bfM)\;.
\label{J(M L) cong ...}
\end{gather}
The first isomorphisms of~\eqref{J^m(M L) cong ...} and~\eqref{J(M L) cong ...} are independent of $\bfg$; thus they hold without assuming~\ref{i: g is of bounded geometry} and~\ref{i: A'} \cite[the analog of Remark~6.41 for $J(M,L)$]{AlvKordyLeichtnam-conormal}.

\subsection{$I(M,L)$ vs $\dot\AA(\bfM)$ and $J(M,L)$}\label{ss: I(M,L) vs dot AA(bfM)$ and J(M L)}

According to \Cref{ss: pull-back of conormal distribs,ss: push-forward of conormal distribs}, we have the continuous linear maps
\begin{equation}\label{tilde pi^* - tilde pi_*}
\tilde\pi^*:I(M,L)\to I(\widetilde M,\widetilde L)\;,\quad
\tilde\pi_*:I(\widetilde M,\widetilde L)\to I(M,L)\;,
\end{equation}
which are restrictions of the maps~\eqref{tilde pi^*: C^-infty(M) -> C^-infty(widetilde M)}, and therefore they satisfy~\eqref{tilde pi_* tilde pi^* = 2}. These maps are compatible with the symbol and Sobolev filtrations because $\tilde\pi:\widetilde M\to M$ is a covering map (\Cref{ss: pull-back of conormal distribs,ss: push-forward of conormal distribs}).

Since~\eqref{dot AA(M) cong I_M(breve M partial M)} gives a TVS-embedding $\dot\AA(\bfM)\subset I(\widetilde M,\widetilde L)$, which preserves the Sobolev-order and symbol-order filtrations, the map $\tilde\pi_*$ of~\eqref{tilde pi^* - tilde pi_*} has the restriction
\begin{equation}\label{bfpi_*: dot AA(bfM) -> I(M L)}
\bfpi_*:\dot\AA(\bfM)\to I(M,L)\;.
\end{equation}
By~\eqref{bfpi_*: L^2(bfM) cong L^2(M)} and according to \Cref{ss: lift of diff ops}, the map~\eqref{bfpi_*: dot AA(bfM) -> I(M L)} restricts to a TVS-isomorphism
\begin{equation}\label{bfpi_*: dot AA^(0)(bfM) cong I^(0)(M L)}
\bfpi_*:\dot\AA^{(0)}(\bfM)\xrightarrow{\cong} I^{(0)}(M,L)\;.
\end{equation}

On the other hand, the map~\eqref{R :C^-infty(M) -> C^-infty(M L)} restricts to a continuous linear map
\begin{equation}\label{R: I(M L) -> J(M L)}
R:I(M,L)\to J(M,L)\;,
\end{equation}
which is the identity on $C^\infty(M)$, and can be also described as the composition
\[
I(M,L) \xrightarrow{\tilde\pi^*} I(\widetilde M,\widetilde L) \xrightarrow{R} \AA(\bfM) \xrightarrow{\bfpi_*} J(M,L)\;.
\]
Both~\eqref{bfpi_*: dot AA(bfM) -> I(M L)} and~\eqref{R: I(M L) -> J(M L)} are surjective topological homomorphisms \cite[Proposition~7.29]{AlvKordyLeichtnam-conormal}, and therefore $C^\infty(M)$ is dense in $J(M,L)$ \cite[Corollary~7.32]{AlvKordyLeichtnam-conormal}.

\subsection{The space $K(M,L)$}\label{ss: K(M L)}

Like in \Cref{ss: KK(M)}, the condition of being supported in $L$ defines the LCHSs and $C^\infty(M)$-modules \index{$K^{(s)}(M,L)$} \index{$K^m(M,L)$} \index{$K(M,L)$}
\[
K^{(s)}(M,L)=I^{(s)}_L(M,L)\;,\quad K^m(M,L)=I^m_L(M,L)\;,\quad K(M,L)=I_L(M,L)\;.
\]
These are closed subspaces of $I^{(s)}(M,L)$, $I^m_L(M,L)$ and $I(M,L)$, respectively; more precisely, they are the null spaces of the corresponding restrictions of the map~\eqref{R: I(M L) -> J(M L)}. The identity~\eqref{dot C^-infty_partial bfM(bfM) equiv C^-infty_L(M) oplus C^-infty_L(M)} restricts to a TVS-identity
\begin{equation}\label{KK(bfM) equiv K(M L) oplus K(M L)}
\KK(\bfM)\equiv K(M,L)\oplus K(M,L)\;.
\end{equation}
Furthermore the maps~\eqref{bfpi_*: dot C^-infty_partial bfM(bfM) -> C^-infty_L(M)} induce continuous linear maps
\begin{equation}\label{bfpi_*: KK(bfM) -> K(M L)}
\bfpi_*:\KK(\bfM)\to K(M,L)\;,\quad
\tilde\pi^*:K(M,L)\to\KK(\bfM)\;.
\end{equation}
Using~\eqref{KK(bfM) equiv K(M L) oplus K(M L)}, these maps are given by $\bfpi_*(u,v)=u+v$ and $\tilde\pi^*u=(u,u)$.

By~\eqref{dot AA^(s)(M) equiv I^(s)_M(breve M partial M)} and~\eqref{dot AA^m(M) = I^m_M(breve M partial M)}, $K^{(s)}(M,L)$ and $K^m(M,L)$ satisfy analogs of~\eqref{KK(bfM) equiv K(M L) oplus K(M L)}, using $\KK^{(s)}(\bfM)$ and $\KK^m(\bfM)$. The following properties hold true \cite[Corollaries~7.19--7.21 and~7.23]{AlvKordyLeichtnam-conormal}: $K(M,L)$ is a limit subspace of the LF-space $I(M,L)$; every $K^{(s)}(M,L)$ is a totally reflexive Fr\'echet space; moreover it is barreled, ultrabornological and webbed, and therefore so is $K(M,L)$; and $K(M,L)$ is an acyclic Montel space, and therefore complete, boundedly/compactly/sequentially retractive and reflexive.

\begin{ex}\label{ex: Diff(M)}
With the notation of \Cref{ss: pseudodiff ops}, $\Diff(M)\equiv K(M^2,\Delta)$ becomes a filtered $C^\infty(M^2)$-submodule of $\Psi(M)$, with the order filtration corresponding to the symbol filtration. In this way, $\Diff(M)$ also becomes a LCHS satisfying the above properties. If $M$ is compact, it is also a filtered subalgebra of $\Psi(M)$.
\end{ex}

\subsection{A description of $K(M,L)$}\label{ss: description of K(M L)}

By~\eqref{u mapsto delta_L^u conormal} and~\eqref{A: I^(s)(M L E) -> I^[s-k](M L E)}, for $s<-1/2$, every isomorphism~\eqref{C^-infty(L Omega^-1 NL) cong partial_x^m C^-infty(L Omega^-1 NL)} restricts to a TVS-isomorphism 
\begin{equation}\label{C^infty(L Omega^-1 NL) cong partial_x^m C^infty(L Omega^-1 NL)}
\partial_x^m:C^\infty(L;\Omega^{-1}NL)\xrightarrow{\cong}\partial_x^mC^\infty(L;\Omega^{-1}NL)
\subset K^{(s-m)}(M,L)\;,
\end{equation}
Then~\eqref{bigoplus_m C^0_m -> C^-infty_L(M)} restricts to a TVS-isomorphisms \cite[Proposition~7.26]{AlvKordyLeichtnam-conormal}
\begin{gather}
\bigoplus_{m=0}^\infty C^\infty(L;\Omega^{-1}NL)\xrightarrow{\cong}K(M,L)\;,\label{bigoplus_m C^1_m -> K(M L)}\\
\bigoplus_{m<-s-\frac12}C^\infty(L;\Omega^{-1}NL)\xrightarrow{\cong}K^{(s)}(M,L)\quad(s<-1/2)\;.\label{bigoplus_m<-s-1/2 C^1_m cong K^(s)(M L)}
\end{gather}

\begin{rem}\label{r: description of KK(M)}
In \Cref{ss: KK(M)}, for any compact manifold with boundary $M$, the analogs of~\eqref{bigoplus_m C^1_m -> K(M L)} and~\eqref{bigoplus_m<-s-1/2 C^1_m cong K^(s)(M L)} for $\KK(M)$ follows from their application to $K(\breve M,\partial M)$ using~\eqref{KK(M) cong I_partial M(breve M partial M)}.
\end{rem}

\subsection{The conormal sequence}\label{conormal sequence}

The diagram~\eqref{CD: ... dot C^-infty(bfM) -R-> C^-infty(bfM) -> 0} has the restriction
\begin{equation}\label{CD: conormal seqs}
\begin{CD}
0\to\KK(\bfM) @>\iota>> \dot\AA(\bfM) @>R>> \AA(\bfM)\to0 \\
@V{\bfpi_*}VV @V{\bfpi_*}VV @V{\cong}V{\bfpi_*}V \\
0\to K(M,L) @>\iota>> I(M,L) @>R>> J(M,L)\to0\;.\hspace{-.2cm}
\end{CD}
\end{equation}
The bottom row of~\eqref{CD: conormal seqs} is exact in the category of continuous linear maps between LCSs \cite[Corollary~7.30]{AlvKordyLeichtnam-conormal}; it will be called the \emph{conormal sequence} of $M$ at $L$ (or of $(M,L)$). \index{conormal sequence}

The surjectivity of~\eqref{R: I(M L) -> J(M L)} can be realized with the partial extension maps given by the following consequence of \Cref{p: E_m}, whose proof is recalled by its relevance in \Cref{ch: conormal,ch: dual-conormal}.

\begin{cor}[{\cite[Corollary~7.31]{AlvKordyLeichtnam-conormal}}]\label{c: E_m: J^m(M L) -> I^(s)(M L)}
For all $m\in\R$, there is a continuous linear partial extension map $E_m:J^m(M,L)\to I^{(s)}(M,L)$, where $s=0$ if $m\ge0$, and $m>s\in\Z^-$ if $m<0$. For $m\ge0$, $E_m:J^m(M,L)\to I^{(0)}(M)$ is a continuous inclusion map.
\end{cor}

\begin{proof}
By the commutativity of~\eqref{CD: conormal seqs} and using \Cref{p: E_m}, we can define $E_m:J^m(M,L)\to I^{(s)}(M,L)$ as the composition
\[
J^m(M,L) \xrightarrow{\bfpi_*^{-1}} \AA^m(\bfM) \xrightarrow{E_m} \dot\AA^{(s)}(\bfM) 
\xrightarrow{\bfpi_*} I^{(s)}(M,L)\;.
\]
The last assertion follows from \Cref{p: E_m,p: E_m s} and~\eqref{bfpi_*: dot AA^(0)(bfM) cong I^(0)(M L)}.
\end{proof}

According to this proof, \Cref{r: E_m m < 0,r: E_m vector bundle,r: E_m compactly supported,p: E_m s,c: E_m s,r: E_m s with m ge 0 and s < 0,r: E_m s - variants of the defn} have obvious versions for the maps given by \Cref{c: E_m: J^m(M L) -> I^(s)(M L)}.

\subsection{Action of $\Diff(M)$ on  the conormal sequence}\label{ss: Diff(M) on the conormal seq}

According to \Cref{ss: diff opers on conormal distribs}, every $A\in\Diff(M)$ defines a continuous linear map $A$ on $I(M,L)$, which preserves $K(M,L)$, and induces a continuous linear map $A$ on $J(M,L)$. This map satisfies the analog of~\eqref{A: I^(s)(M L E) -> I^[s-k](M L E)}.

The map $A$ on $J(M,L)$ can be also described as a restriction of $A$ on $C^{-\infty}(M,L)$ (\Cref{ss: action of Diff(M) on C^-infty(M L) and C^-infty_L(M)}). On the other hand,  according to \Cref{ss: diff ops on dot AA(M) and AA(M)}, the lift $\widetilde A\in\Diff(\bfM)$ defines continuous linear maps on the top spaces of~\eqref{CD: conormal seqs} which correspond to the operators defined by $A$ on the bottom spaces via the maps $\bfpi_*$. If $A\in\Diff(M,L)$, then it defines continuous endomorphisms $A$ of $J^{(s)}(M,L)$ and $J^m(M,L)$.

\subsection{Pull-back maps on the conormal sequence}\label{ss: pull-back of the conormal seq}

Consider the notation and conditions of \Cref{ss: pull-back of conormal distribs}. By the exactness of the conormal sequences of $(M,L)$ and $(M',L')$ in the category of continuous linear maps between LCSs, the map~\eqref{phi^*: I(M L) -> I(M' L')} induces continuous linear maps,
\begin{gather}
\phi^*:K(M,L)\to K(M',L')\;,\label{phi^*: K(M L) to K(M' L')}\\
\phi^*:J(M,L)\to J(M',L')\;.\label{phi^*: J(M L) to J(M' L')}
\end{gather}
The map~\eqref{phi^*: K(M L) to K(M' L')} is the restriction of~\eqref{phi^*: I(M L) -> I(M' L')}, which is well defined because the map~\eqref{phi^*: I(M L) -> I(M' L')} can be locally defined, and~\eqref{phi^*: J(M L) to J(M' L')} is the induced map in the quotient. These maps are compatible with the maps $\iota$ and $R$ of the conormal sequences, and satisfy the analog of~\eqref{phi^*: I(M L E) -> I(M' L' phi^*E)}.

\subsection{Push-forward maps on the conormal sequence}\label{ss: push-forward of the conormal seq}

Consider the notation and conditions of \Cref{ss: push-forward of conormal distribs}. Like in \Cref{ss: pull-back of the conormal seq}, the map~\eqref{phi_*: I_c/cv(M' L' Omega_fiber) -> I_c/.(M L)} induces continuous linear maps,
\begin{gather}
\phi_*:K(M',L';\Omega_\fiber)\to K(M,L)\;,\label{phi_*: K(M' L' Omega_fiber) to K(M L)}\\
\phi_*:J(M',L';\Omega_\fiber)\to J(M,L)\;.
\label{phi_*: J(M' L' Omega_fiber) to J(M L)}
\end{gather}
They are also compatible with the maps $\iota$ and $R$ of the conormal sequences, and satisfy the analog of~\eqref{phi_*: I_c/cv(M' L' phi^*E otimes Omega_fiber) -> I_c/.(M L E)}.

\subsection{Case where $L$ is not transversely orientable}\label{ss: case where L is not transversely orientable}

If $L$ is not transversely orientable, we still have a tubular neighborhood $T$ of $L$ in $M$, but there is no defining function $x$ of $L$ in $T$ trivializing the projection $\varpi:T\to L$. We can cut $M$ along $L$ as well to produce a bounded compact manifold, $\bfM$, with a projection $\bfpi:\bfM\to M$ and a boundary collar $\bfT$ over $T$.

Using a boundary-defining function $\bfx$ of $\bfM$, we get the same definitions, properties and descriptions of $C^{\pm\infty}(M,L)$ and $J(M,L)$ (\Cref{ss: C^-infty(M L),ss: J(M L),ss: description of J(M L)}). $C^{-\infty}_L(M)$ and $K(M,L)$ also have the same definitions (\Cref{ss: C^-infty_L(M),ss: K(M L)}). However~\eqref{dot C^-infty_partial bfM(bfM) equiv C^-infty_L(M) oplus C^-infty_L(M)} and~\eqref{KK(bfM) equiv K(M L) oplus K(M L)} are not true because the covering map $\bfpi:\partial\bfM\to L$ is not trivial, and the descriptions given in~\eqref{bigoplus_m C^0_m -> C^-infty_L(M)},~\eqref{bigoplus_m^k C^m-k(L Omega^-1NL) cong C^prime -k_L(M)},~\eqref{bigoplus_m C^1_m -> K(M L)} and~\eqref{bigoplus_m<-s-1/2 C^1_m cong K^(s)(M L)} need a slight modification. This problem can be solved as follows.

Let $\check\pi:\check L\to L$ \index{$\check\pi$} \index{$\check L$} denote the two-fold covering of transverse orientations of $L$, and let $\check\sigma$ \index{$\check\sigma$} denote its deck transformation different from the identity. Since the lift of $NL$ to $\check L$ is trivial, $\check\pi$ on $\check L\equiv\{0\}\times\check L$ can be extended to a two-fold covering $\check\pi:\check T:=(-\epsilon,\epsilon)_x\times\check L\to T$, \index{$\check T$} for some $\epsilon>0$. Its deck transformation different from the identity is an extension of $\check\sigma$ on $\check L\equiv\{0\}\times\check L$, also denoted by $\check\sigma$. Then $\check L$ is transversely oriented in $\check T$; i.e., its normal bundle $N\check L$ is trivial. Thus $C^{-\infty}_{\check L}(\check T)$ and $K(\check T,\check L)$ satisfy~\eqref{dot C^-infty_partial bfM(bfM) equiv C^-infty_L(M) oplus C^-infty_L(M)},~\eqref{bigoplus_m C^0_m -> C^-infty_L(M)},~\eqref{bigoplus_m^k C^m-k(L Omega^-1NL) cong C^prime -k_L(M)},~\eqref{KK(bfM) equiv K(M L) oplus K(M L)},~\eqref{bigoplus_m C^1_m -> K(M L)} and~\eqref{bigoplus_m<-s-1/2 C^1_m cong K^(s)(M L)}. Since $N\check L\equiv\check\pi^*NL$, the map $\check\sigma$ lifts to a homomorphism of $N\check L$, which induces a homomorphism of $\Omega^{-1}NL$ also denoted by $\check\sigma$. Let $L_{-1}$ be the union of non-transversely oriented connected components of $L$, and $L_1$ the union of its transversely oriented components. Correspondingly, let $\check L_{\pm1}=\check\pi^{-1}(L_{\pm1})$ and $\check T_{\pm1}=(-\epsilon,\epsilon)\times\check L_{\pm1}$. Since $\check\sigma^*x=\pm x$ on $T_{\pm1}$, the isomorphisms~\eqref{bigoplus_m C^0_m -> C^-infty_L(M)},~\eqref{bigoplus_m^k C^m-k(L Omega^-1NL) cong C^prime -k_L(M)},~\eqref{bigoplus_m C^1_m -> K(M L)} and~\eqref{bigoplus_m<-s-1/2 C^1_m cong K^(s)(M L)} become true in this case by replacing $C^r(L;\Omega^{-1}NL)$ ($r\in\Z\cup\{\pm\infty\}$) with the direct sum of the spaces
\[
\{\,u\in C^r(L_{\pm1};\Omega^{-1}NL_{\pm1})\mid\check\sigma^*u=\pm u\,\}\;.
\]

The other results about $C^{-\infty}_L(M)$ and $K(M,L)$ (\Cref{ss: C^-infty_L(M),ss: description of C^-infty_L(M),ss: K(M L),ss: description of K(M L)}) can be obtained by using these extensions of~\eqref{bigoplus_m C^0_m -> C^-infty_L(M)},~\eqref{bigoplus_m^k C^m-k(L Omega^-1NL) cong C^prime -k_L(M)},~\eqref{bigoplus_m C^1_m -> K(M L)} and~\eqref{bigoplus_m<-s-1/2 C^1_m cong K^(s)(M L)} instead of~\eqref{dot C^-infty_partial bfM(bfM) equiv C^-infty_L(M) oplus C^-infty_L(M)} and~\eqref{KK(bfM) equiv K(M L) oplus K(M L)}. \Cref{conormal sequence,ss: Diff(M) on the conormal seq,ss: pull-back of the conormal seq,ss: push-forward of the conormal seq} also have strightforward extensions.

\section{Dual-conormal sequence}\label{s: dual-conormal seq}

\subsection{The spaces $K'(M,L)$ and $J'(M,L)$}\label{ss: K'(M L) and J'(M L)}

Consider the notation of \Cref{s: conormal seq} assuming that $L$ is transversely oriented; the extension to the non-transversely orientable case can be made with the procedure of \Cref{ss: case where L is not transversely orientable}. Like in \Cref{ss: dual-conormal distribs - compact,ss: AA'(M)}, let \index{$K'(M,L)$} \index{$J'(M,L)$} 
\[
K'(M,L)=K(M,L;\Omega)'\;,\quad J'(M,L)=J(M,L;\Omega)'\;.
\]
By~\eqref{bfpi_*: AA(bfM) cong J(M L)} and~\eqref{KK(bfM) equiv K(M L) oplus K(M L)},
\begin{equation}\label{dot AA'(bfM) equiv J'(M L)}
\KK'(\bfM)\equiv K'(M,L)\oplus K'(M,L)\;,\quad\dot\AA'(\bfM)\equiv J'(M,L)\;.
\end{equation}
Let also \index{$K^{\prime\,(s)}(M,L)$} \index{$K^{\prime\,m}(M,L)$} \index{$J^{\prime\,(s)}(M,L)$} \index{$J^{\prime\,m}(M,L)$} 
\begin{equation}\label{K^prime[s](M L)}
\left\{
\begin{gathered}
K^{\prime\,(s)}(M,L)=K^{(-s)}(M,L;\Omega)'\;,\quad K^{\prime\,m}(M,L)=K^{-m}(M,L;\Omega)'\;,\\
J^{\prime\,(s)}(M,L)=J^{(-s)}(M,L;\Omega)'\;,\quad J^{\prime\,m}(M,L)=J^{-m}(M,L;\Omega)'\;,
\end{gathered}
\right.
\end{equation}
which satisfy the analog of~\eqref{dot AA'(bfM) equiv J'(M L)}. Like in \Cref{ss: AA'(M)}, for $s<s'$ and $m<m'$, we get continuous linear restriction maps
\[
K^{\prime\,(s')}(M,L)\to K^{\prime\,(s)}(M,L)\;,\quad K^{\prime\,m}(M,L)\to K^{\prime\,m'}(M,L)\;,
\]
and continuous injections
\[
J^{\prime\,(s')}(M,L)\subset J^{\prime\,(s)}(M,L)\;,\quad J^{\prime\,m'}(M,L)\subset J^{\prime\,m}(M,L)\;,
\]
forming projective spectra. By~\eqref{dot AA'(bfM) equiv J'(M L)} and its analog for the spaces~\eqref{K^prime[s](M L)}, and according to \Cref{ss: AA'(M)}, we get that the spaces $K^{\prime\,(s)}(M,L)$ and $K^{\prime\,m}(M,L)$ satisfy the analogs of~\eqref{sandwich for I'} and~\eqref{I'(M L) equiv projlim I^prime (s)(M L) equiv projlim I^prime m(M L)}, and the spaces $J^{\prime\,(s)}(M,L)$ and $J^{\prime\,m}(M,L)$ satisfy the analogs of~\eqref{sandwich for dot AA'} and~\eqref{dot AA'(M) = bigcap_s dot AA^prime (s)(M) = bigcap_m dot AA^prime m(M)} \cite[Corollary~8.3]{AlvKordyLeichtnam-conormal}. Furthermore, $K'(M,L)$ and $J'(M,L)$ are complete Montel spaces \cite[Corollary~8.1]{AlvKordyLeichtnam-conormal}, and $K^{\prime\,(s)}(M,L)$ and $J^{\prime\,(s)}(M,L)$ are bornological and barreled \cite[Corollary~8.2]{AlvKordyLeichtnam-conormal}.

Like in \Cref{ss: AA'(M)}, the versions of~\eqref{C^infty(M L) = C^-infty(M L Omega)'},~\eqref{J^(infty)(M L) subset C^-infty(M L} and~\eqref{C^infty(M) subset J^(infty)(M L)} with $\Omega M$ induce continuous inclusions
\begin{equation}\label{C^infty(M L) subset J'(M L) subset C^-infty(M)}
C^{-\infty}(M)\supset J'(M,L)\supset C^\infty(M,L)\;.
\end{equation}

\subsection{A description of $J'(M,L)$}\label{ss: description of J'(M L)}

With the notation and conditions of \Cref{ss: description of J(M L)}, we have the following \cite[Corollaries~8.4 and~8.5]{AlvKordyLeichtnam-conormal}:
\begin{gather}
J^{\prime\,m}(M,L)\cong\bfx^m\Hb^{-\infty}(\bfM)=\bfx^{m-\frac12}H^{-\infty}(\mathring\bfM)\;,
\label{J^prime m(M L) cong ...}\\
J'(M,L)\cong\bigcap_m\bfx^m\Hb^{-\infty}(\bfM)=\bigcap_m\bfx^mH^{-\infty}(\mathring\bfM)\;.
\label{J'(M L) cong ...}
\end{gather}
Actually, the first isomorphisms of~\eqref{J^prime m(M L) cong ...} and~\eqref{J'(M L) cong ...} are independent of $g$, and hold true without the assumptions~\ref{i: g is of bounded geometry} and~\ref{i: A'}. Furthermore $\Cinftyc(M\setminus L)$ is dense in every $J^{\prime\,m}(M,L)$ and in $J'(M,L)$ \cite[Corollary~8.6]{AlvKordyLeichtnam-conormal}. Therefore the right-hand side inclusion of~\eqref{C^infty(M L) subset J'(M L) subset C^-infty(M)} is also dense.

\subsection{Description of $K'(M,L)$}\label{ss: description of K'(M L)}

The transposes of the versions of~\eqref{bigoplus_m C^1_m -> K(M L)} and~\eqref{bigoplus_m<-s-1/2 C^1_m cong K^(s)(M L)} with $\Omega M$ are TVS-isomorphisms \cite[Corollary~8.7]{AlvKordyLeichtnam-conormal},
\begin{gather}
K'(M,L)\xrightarrow{\cong}\prod_{m=0}^\infty C^{-\infty}(L)\;,\label{K'(M L) cong prod_m=0^infty C^-infty(L)}\\
K^{\prime\,(s)}(M,L)\xrightarrow{\cong}\prod_{m<s-1/2}C^{-\infty}(L)\quad(s>1/2)\;,
\label{K^prime (s)(M L) cong prod_m<s-1/2 C^-infty(L)}
\end{gather}
because
\[
C^{\infty}(L;\Omega^{-1}NL\otimes\Omega M)'=C^{\infty}(L;\Omega)'=C^{-\infty}(L)\;.
\]

\subsection{Dual-conormal sequence}\label{ss: dual-conormal seq}

The transpose of the density-bundles  version of~\eqref{CD: conormal seqs} is the commutative diagram
\begin{equation}\label{CD: dual-conormal seqs}
\begin{CD}
0\leftarrow\KK'(\bfM) @<{R'}<< \AA'(\bfM) @<{\iota'}<< \dot\AA'(\bfM)\leftarrow0 \\
@A{\bfpi^*}AA @A{\bfpi^*}AA @A{\bfpi^*}A{\cong}A \\
0\leftarrow K'(M,L) @<{R'}<< I'(M,L) @<{\iota'}<< J'(M,L)\leftarrow0\;,\hspace{-.2cm}
\end{CD}
\end{equation}
where $R'=\iota^\trans$ and $\iota'=R^\trans$. Its bottom row is exact in the category of continuous linear maps between LCSs \cite[Proposition~8.8]{AlvKordyLeichtnam-conormal}, and is called the \emph{dual-conormal sequence} of $M$ at $L$ (or of $(M,L)$). \index{dual-conormal sequence}

\subsection{Action of $\Diff(M)$ on the dual-conormal sequence}\label{ss: action of Diff(M) on the dual-conormal seq}

With the notation of \Cref{ss: Diff(M) on the conormal seq}, consider the actions of $A^\trans$ and $\widetilde A^\trans$ on the bottom and top spaces of the version of~\eqref{CD: conormal seqs} with $\Omega M$ and $\Omega\bfM$. Taking transposes again, we get induced actions of $A$ and $\widetilde A$ on the bottom and top spaces of~\eqref{CD: dual-conormal seqs}, which correspond one another via the linear maps $\bfpi^*$. These maps satisfy the analogs of~\eqref{A: I^prime [s](M L E) -> I^prime (s-m)(M L E)}.

\subsection{Pull-back maps on the dual-conormal sequence}
\label{ss: pull-back of the dual-conormal seq}

Consider the notation and conditions of \Cref{ss: pull-back of dual-conormal distributions} (the same as in \Cref{ss: push-forward of conormal distribs}). Like in~\Cref{ss: pull-back of dual-conormal distributions}, transposing the compactly supported case of the analog of~\eqref{phi_*: I_c/cv(M' L' phi^*E otimes Omega_fiber) -> I_c/.(M L E)} for~\eqref{phi_*: K(M' L' Omega_fiber) to K(M L)} and~\eqref{phi_*: J(M' L' Omega_fiber) to J(M L)} with $E=\Omega M$, we get continuous linear maps,
\begin{gather}
\phi^*:K'(M,L)\to K'(M',L')\;,\label{phi^*: K'(M L) to K'(M' L')}\\
\phi^*:J'(M,L)\to J'(M',L')\;.\label{phi^*: J'(M L) to J'(M' L')}
\end{gather}
They are compatible with the maps $\iota'$ and $R'$ of the dual-conormal sequences, and satisfy the analog of~\eqref{phi^*: I'(M L E) -> I'(M' L' phi^*E)}.

\subsection{Push-forward maps on the dual-conormal sequence}
\label{ss: push-forward of the dual-conormal seq}

Consider the notation and conditions of \Cref{ss: push-forward of dual-conormal distributions} (the same as in \Cref{ss: pull-back of conormal distribs}). Like in \Cref{ss: push-forward of dual-conormal distributions}, transposing the analogs of~\eqref{phi^*: I(M L E) -> I(M' L' phi^*E)} for~\eqref{phi^*: K(M L Lambda) to K(M' L' Lambda)} and~\eqref{phi^*: J(M L Lambda) to J(M' L' Lambda)} with $E=\Omega M$, and using an analog of~\eqref{phi_*u}, we get continuous linear maps,
\begin{gather}
\phi_*:K'(M',L';\Omega_\fiber)\to K'(M,L)\;,\label{phi_*: K'(M' L' Omega_fiber) to K'(M L)}\\
\phi_*:J'(M',L';\Omega_\fiber)\to J'(M,L)\;.\label{phi_*: J'(M' L' Omega_fiber) to J'(M L)}
\end{gather}
They are compatible with the maps $\iota'$ and $R'$ of the dual-conormal sequences, and satisfy the analog of~\eqref{phi_*: I'_c/cv(M' L' phi^*E otimes Omega_fiber) -> I'_c/.(M L E)}.

\section{Currents}\label{s: currents}

Here, again, the manifold $M$ may not be compact, and $L\subset M$ is a regular submanifold that is a closed subset. When using $J(M,L;\Lambda)$ or $K(M,L;\Lambda)$, it is also assumed that $L$ is of codimension one. 

\subsection{Differential forms and currents}\label{ss: diff forms and currents}

Consider the space $C^\infty(M;\Lambda)$ of smooth differential forms, and the space $C^{-\infty}(M;\Lambda)$ of currents. The most typical example of elliptic complex is given by the de~Rham derivative $d$ on $C^\infty(M;\Lambda)$, giving rise to the de~Rham cohomology $H^\bullet(M)$. The extension of $d$ to $C^{-\infty}(M;\Lambda)$ is another topological complex, which produces isomorphic cohomology \cite{deRham1984}. We typically consider cohomology with complex coefficients without further comment; real cohomology classes are only considered in a few cases, where it is indicated; the same applies to other cohomologies that will be considered. The basic properties of  $(C^{\pm\infty}(M;\Lambda),d)$ and $H^\bullet(M)$ can be seen in \cite{deRham1984,BottTu1982}; for instance, the general properties of elliptic complexes apply in this setup (\Cref{ss: diff complexes}). Some properties will be seen in \Cref{s: Witten} with more generality.

A Riemannian metric $g$ on $M$ defines a Hermitian structure on $\Lambda M$, also denoted by $g$. Then we have the additional operators $\delta$ (the \emph{de~Rham coderivative}), $D$ and $\Delta$ (the \emph{Laplacian}) of \Cref{ss: diff complexes}. If needed, the subscript ``$M$'' may be added to this notation, and to other similar notation.

We may also consider the de~Rham complex with coefficients in a flat vector bundle $\FF$, $d=d^\FF$ on $C^\infty(M;\Lambda\otimes\FF)$. As above, $g$ and a Hermitian structure on $\FF$ induce additional operators $\delta=\delta^\FF$, $D=D^\FF$ and $\Delta=\Delta^\FF$.

For any $V\in\fX(M)$, let $\iota_V$ and $\LL_V$ denote the corresponding inner product and Lie derivative on $C^\infty(M;\Lambda)$. For $\eta=V^\flat\in C^\infty(M;\Lambda^1)$, we write ${\eta\lrcorner}=-(\eta\wedge)^*=-\iota_V$. Let $\sw$ be the degree involution on $\Lambda M$. For the bundle of Clifford algebra of $T^*M$, we have the identity $\Cl(T^*M)\equiv\Lambda_\R M$ defined by the symbol of filtered algebras. Via this identity, the left Clifford multiplication by $\eta$ is $c(\eta)={\eta\wedge}+{\eta\lrcorner}$, and the composition of $\sw$ with the right Clifford multiplication by $\eta$ is $\hat c(\eta)={\eta\wedge}-{\eta\lrcorner}$.

\subsection{Product of differential forms and currents}
\label{ss: product of diff forms and currents}

The exterior product of smooth differential forms has continuous extensions,\footnote{This holds with more generality under conditions on the wavefront set \cite[Theorem~8.2.10]{Hormander1971}, but we will not use it.}
\begin{equation}\label{extension of exterior product to currents}
C^{\pm\infty}(M;\Lambda)\otimes C^{\mp\infty}(M;\Lambda)\to C^{-\infty}(M;\Lambda)\;,
\end{equation}
For example, with the notation of \Cref{ss: Dirac sections}, assuming that $M$ and $L$ are oriented, it easily follows that, for $\alpha\in C^\infty(M;\Lambda)$ and $\beta\in C^\infty(L;\Lambda\otimes\Omega^{-1}NL)$,
\begin{equation}\label{alpha wedge delta_L^beta}
\alpha\wedge\delta_L^\beta=\delta_L^{\alpha|_L\wedge\beta}\;,\quad
\delta_L^\beta\wedge\alpha=\delta_L^{\beta\wedge\alpha|_L}\;.
\end{equation}

\subsection{Currents on oriented manifolds}

Assume $M$ is oriented. The orientation induces a canonical identity $\Omega M\equiv\Lambda^nM$. Then, for every degree $k$, the non-degenerate pairing $\Lambda^kM\otimes\Lambda^{n-k}M\to\Lambda^nM$ defined by the wedge product induces a canonical identity
\begin{equation}\label{(Lambda^rM)^* otimes Omega M equiv Lambda^n-rM}
(\Lambda^kM)^*\otimes\Omega M\equiv\Lambda^{n-k}M\;.
\end{equation}
By~\eqref{(Lambda^rM)^* otimes Omega M equiv Lambda^n-rM}, the space~\eqref{C^-infty_cdot/c(M;E)} becomes
\begin{equation}\label{C^-infty(M Lambda^r) equiv C^infty_c(M Lambda^n-r)'}
C^{-\infty}_{{\cdot}/\co}(M;\Lambda^k)\equiv C^\infty_{\co/{\cdot}}(M;\Lambda^{n-k})'
\end{equation}
in this case. This identity corresponds to a pairing
\[
C^{\pm\infty}_{{\cdot}/\co}(M;\Lambda^k)\otimes C^{\mp\infty}_{\co/{\cdot}}(M;\Lambda^{n-k})\to\C\;,
\]
which will be denoted with parentheses to distinguish it from the scalar product. This pairing can be given by the composition of~\eqref{extension of exterior product to currents} and the extension
\[
C^{-\infty}_\co(M;\Lambda^n)\to\C\;,\quad\alpha\mapsto(\alpha,1)\;,
\]
of $\int_M:\Cinftyc(M;\Lambda M)\to\C$.

\subsection{Hodge operator on oriented manifolds}\label{ss: perturbed ops on oriented mfds}

Continuing with the assumption of orientation, let $\star$ on $\Lambda M$ denote the $\C$-linear extension of the Hodge operator $\star$ on the real forms, which is unitary, and let $\bar\star$ denote its $\C$-antilinear extension. These operators are determined by the conditions, for $\alpha,\beta\in C^\infty(M;\Lambda)$, 
\[
\alpha\wedge\overline{\star\beta}=g(\alpha,\beta)\,\dvol=\alpha\wedge\bar\star\beta\;,
\]
where $\dvol=\star1$ is the volume form. Recall that, on $C^\infty(M;\Lambda^k)$, 
\begin{equation}\label{star}
\left\{
\begin{gathered}
\star^2=(-1)^{nk+k}\;,\quad
\delta=(-1)^{nk+n+1}\star\,d\,\star\;,\quad
{\eta\lrcorner}=(-1)^{nk+n+1}\star\,{\eta\wedge}\,\star\;,\\
d\,\star=(-1)^k\star\,\delta\;,\quad
\delta\,\star=(-1)^{k+1}\star\,d\;,\quad
\Delta\,\star=\star\,\Delta\;,\\
{\eta\wedge}\,\star=(-1)^k\,\star\,{\eta\lrcorner}\;,\quad
{\eta\lrcorner}\,\star=(-1)^{k+1}\,\star\;{\eta\wedge}\;.
\end{gathered}
\right.
\end{equation}
The equalities~\eqref{star} are also true with $\bar\star$, and can be extended to $C^{-\infty}(M;\Lambda)$.

For all $\alpha\in C^\infty(M;\Lambda^k)$ and $\beta\in\Cinftyc(M;\Lambda^{n-k})$,
\[
\alpha\wedge\beta=(-1)^{kn+k}\alpha\wedge\bar\star^2\beta=(-1)^{kn+k}g(\alpha,\bar\star\beta)\dvol\;,
\]
yielding
\begin{equation}\label{(alpha beta) = (-1)^rn+r langle alpha star beta rangle}
(\alpha,\beta)=(-1)^{kn+k}\langle\alpha,\bar\star\beta\rangle\;.
\end{equation}

\subsection{Pull-back and push-forward of currents}
\label{ss: pull-back and push-forward of currents}

Given a smooth map $\phi:M'\to M$, recall that its tangent map $T\phi=\phi_*:TM'\to TM$ defines a homomorphism $\phi_*:TM'\to \phi^*TM$, which induces a homomorphism
\begin{equation}\label{phi^*: phi^* Lambda M to Lambda M'}
\phi^*:\phi^*\Lambda M\to\Lambda M'\;.
\end{equation}
Then recall that the pull-back homomorphism
\begin{equation}\label{phi^*: C^infty(M Lambda) to C^infty(M' Lambda)}
\phi^*:C^\infty(M;\Lambda)\to C^\infty(M';\Lambda)
\end{equation}
can be given as the composition
\begin{equation}\label{composition -pull-back - diff forms}
C^\infty(M;\Lambda)\xrightarrow{\phi^*}C^\infty(M';\phi^*\Lambda M)\xrightarrow{\phi^*}C^\infty(M';\Lambda)\;,
\end{equation}
where the first map $\phi^*$ is given by~\eqref{phi^*: C^infty(M E) -> C^infty(M' phi^*E)}, and the second map $\phi^*$ is induced by~\eqref{phi^*: phi^* Lambda M to Lambda M'}.

Now, suppose $\phi$ is a submersion and its vertical subbundle $\VV$ is oriented. Let $\pi_\topd:\Lambda\VV\to\Lambda^\topd\VV$ denote the canonical projection. The orientation of $\VV$ gives a canonical identity $\Omega_\fiber\equiv\Lambda^\topd\VV$. So
\[
\phi^*\Lambda M\otimes\Omega_\fiber\equiv\phi^*\Lambda M\otimes\Lambda^\topd\VV\subset\phi^*\Lambda M\otimes\Lambda\VV\equiv\Lambda M'\;.
\]
Moreover, $\pi_\topd:\Lambda\VV\to\Lambda^\topd\VV$ induces a projection
\begin{equation}\label{pi_top: Lambda M' to phi^*Lambda M otimes Omega_fiber}
\pi_\topd:\Lambda M'\to\phi^*\Lambda M\otimes\Omega_\fiber\;.
\end{equation}
The push-forward homomorphism or \emph{integration along the fibers} \cite[Section~I.6]{BottTu1982},
\begin{equation}\label{phi_*: C^infty_c/cv(M' Lambda) to C^infty_c/.(M Lambda)}
\phi_*:C^\infty_{\co/\cv}(M';\Lambda)\to C^\infty_{\co/{\cdot}}(M;\Lambda)\;,
\end{equation}
can be described as the composition
\begin{equation}\label{composition -push-forward - diff forms}
C^\infty_{\co/\cv}(M';\Lambda)\xrightarrow{\pi_\topd}C^\infty_{\co/\cv}(M';\phi^*\Lambda M\otimes\Omega_\fiber)
\xrightarrow{\phi_*}C^\infty_{\co/{\cdot}}(M;\Lambda)\;,
\end{equation}
where $\pi_\topd$ is induced by~\eqref{pi_top: Lambda M' to phi^*Lambda M otimes Omega_fiber}, and $\phi_*$ is given by~\eqref{phi_*: Cinftyc(M' phi^*E otimes Omega_fiber) -> Cinftyc(M E)} with $E=\Lambda M$.

We also get the push-forward and pull-back maps on currents,
\begin{gather}
\phi_*:C^{-\infty}_{\co/\cv}(M';\Lambda)\to C^{-\infty}_{\co/{\cdot}}(M;\Lambda)\;,
\label{phi_*: C^-infty_c/cv(M' Lambda) to C^-infty_c/.(M Lambda)}\\
\phi^*:C^{-\infty}(M;\Lambda)\to C^{-\infty}(M';\Lambda)\;,
\label{phi^*: C^-infty(M Lambda) to C^-infty(M' Lambda)}
\end{gather}
given by the compositions
\begin{gather}
C^{–\infty}_{\co/\cv}(M';\Lambda)\xrightarrow{\pi_\topd}C^{–\infty}_{\co/\cv}(M';\phi^*\Lambda M\otimes\Omega_\fiber)
\xrightarrow{\phi_*}C^\infty_{\co/{\cdot}}(M;\Lambda)\;,
\label{composition -push-forward - currents}\\
C^{-\infty}(M;\Lambda)\xrightarrow{\phi^*}C^{-\infty}(M';\phi^*\Lambda M)\xrightarrow{\phi^*}C^{-\infty}(M';\Lambda)\;,
\label{composition - pull-back - currents}
\end{gather}
where $\phi_*$ and the first map $\phi^*$ are given by~\eqref{phi_*: C^-infty_c(M' phi^*E otimes Omega_fiber) -> C^-infty_c(M E)}--\eqref{phi_*: C^-infty_cv(M' phi^*E otimes Omega_fiber) -> C^-infty(M E)} with $E=\Lambda M$, and $\pi_\topd$ is induced by~\eqref{pi_top: Lambda M' to phi^*Lambda M otimes Omega_fiber}. The notation $\fint_\phi$ is also used for $\phi_*$, or $\fint_F$ if $\phi$ is a trivial bundle with typical fiber $F$.

\begin{prop}\label{p: phi_* is transpose of phi^* - forms - currents}
The compactly supported case of~\eqref{phi_*: C^-infty_c/cv(M' Lambda) to C^-infty_c/.(M Lambda)} is the transpose of~\eqref{phi^*: C^infty(M Lambda) to C^infty(M' Lambda)}, and~\eqref{phi^*: C^-infty(M Lambda) to C^-infty(M' Lambda)} is the transpose of the compactly supported case of~\eqref{phi_*: C^infty_c/cv(M' Lambda) to C^infty_c/.(M Lambda)}.
\end{prop}

\begin{proof}
By passing to double covers of orientations, we can assume $M$ and $M'$ are oriented, and therefore we can use~\eqref{C^-infty(M Lambda^r) equiv C^infty_c(M Lambda^n-r)'}. By the density of the space of smooth forms in the space of currents (\Cref{ss: smooth/distributional sections}), it is enough to check the statement on smooth forms, where it is given by \cite[Proposition~6.15~(b)]{BottTu1982}: for $\alpha\in C^\infty(M;\Lambda)$ and $\beta\in\Cinftyc(M';\Lambda)$,
\[
(\phi^*\alpha,\beta)=\int_{M'}\phi^*\alpha\wedge\beta=\int_M\alpha\wedge\phi_*\beta=(\alpha,\phi_*\beta)\;.\qedhere
\]
\end{proof}

\subsection{Homotopy operators}\label{ss: homotopy opers}

Recall that any smooth homotopy, $H:M'\times I\to M$ ($I=[0,1]$), induces a continuous homotopy operator $\sh:C^\infty(M';\Lambda)\to C^\infty(M;\Lambda)$ (a linear map, which is homogeneous of degree $-1$, and satisfies $H_1^*-H_0^*=\sh d+d\sh$, where $H_t=H(\cdot,t):M'\to M$). For instance, we can take $\sh$ equal to the composition \cite[Section~4]{BottTu1982}
\begin{equation}\label{sh = fint_I H^*}
C^\infty(M;\Lambda) \xrightarrow{H^*} C^\infty(M'\times I;\Lambda)
\xrightarrow{\fint_I} C^\infty(M';\Lambda)\;.
\end{equation}

\subsection{Pull-back of conormal currents}\label{ss: pull-back of conormal currents}

With the notations and conditions of \Cref{ss: pull-back of conormal distribs}, the map~\eqref{phi^*: C^infty(M Lambda) to C^infty(M' Lambda)} has a continuous extension
\begin{equation}\label{phi^*: I(M L Lambda) to I(M' L' Lambda)}
\phi^*:I(M,L;\Lambda)\to I(M',L;\Lambda)\;,
\end{equation}
which can be given as the composition
\begin{equation}\label{composition - pull-back - conormal currents}
I(M,L;\Lambda)\xrightarrow{\phi^*}I(M',L';\phi^*\Lambda M)\xrightarrow{\phi^*}I(M',L;\Lambda)\;,
\end{equation}
where the first map $\phi^*$ is given by~\eqref{phi^*: I(M L E) -> I(M' L' phi^*E)} with $E=\Lambda M$, and the second map $\phi^*$ is induced by~\eqref{phi^*: phi^* Lambda M to Lambda M'}. If $\phi$ is a smooth submersion with oriented vertical subbundle, then~\eqref{phi^*: I(M L Lambda) to I(M' L' Lambda)} is also a restriction of~\eqref{phi^*: C^-infty(M Lambda) to C^-infty(M' Lambda)}.

Similarly, when $L$ is of codimension one, there are continuous homomorphisms,
\begin{gather}
\phi^*:K(M,L;\Lambda)\to K(M',L';\Lambda)\;,\label{phi^*: K(M L Lambda) to K(M' L' Lambda)}\\
\phi^*:J(M,L;\Lambda)\to J(M',L';\Lambda)\;,\label{phi^*: J(M L Lambda) to J(M' L' Lambda)}
\end{gather}
which can be given as the compositions
\begin{gather}
K(M,L;\Lambda)\xrightarrow{\phi^*}K(M',L';\phi^*\Lambda M)\xrightarrow{\phi^*}K(M',L;\Lambda)\;,
\label{composition - pull-back - K-currents}\\
J(M,L;\Lambda)\xrightarrow{\phi^*}J(M',L';\phi^*\Lambda M)\xrightarrow{\phi^*}J(M',L;\Lambda)\;,
\label{composition - pull-back - J-currents}
\end{gather}
where the first maps $\phi^*$ are given by the analogs of~\eqref{phi^*: I(M L E) -> I(M' L' phi^*E)} with $E=\Lambda M$ for~\eqref{phi^*: K(M L) to K(M' L')} and~\eqref{phi^*: J(M L) to J(M' L')}, and the second maps $\phi^*$ are induced by~\eqref{phi^*: phi^* Lambda M to Lambda M'}.

\subsection{Push-forward of conormal currents}\label{ss: push-forward of conormal currents}

With the notations and conditions of \Cref{ss: push-forward of conormal distribs}, assume also that the vertical subbundle of $\phi$ is oriented. Then the push-forward homomorphism~\eqref{phi_*: C^infty_c/cv(M' Lambda) to C^infty_c/.(M Lambda)} has a continuous extension
\begin{equation}\label{phi_*: I_c/cv(M' L' Lambda) to I_c/.(M L Lambda)}
\phi_*:I_{\co/\cv}(M',L';\Lambda)\to I_{\co/{\cdot}}(M,L;\Lambda)\;,
\end{equation}
which can be described as the composition
\begin{equation}\label{composition - push-forward - conormal currents}
I_{\co/\cv}(M',L';\Lambda)\xrightarrow{\pi_\topd}I_{\co/\cv}(M',L';\phi^*\Lambda M\otimes\Omega_\fiber)
\xrightarrow{\phi_*}I_{\co/{\cdot}}(M,L;\Lambda)\;,
\end{equation}
where $\pi_\topd$ is induced by~\eqref{pi_top: Lambda M' to phi^*Lambda M otimes Omega_fiber}, and $\phi_*$ is given by~\eqref{phi_*: I_c/cv(M' L' phi^*E otimes Omega_fiber) -> I_c/.(M L E)} with $E=\Lambda M$. The map~\eqref{phi_*: I_c/cv(M' L' Lambda) to I_c/.(M L Lambda)} is also a restriction of~\eqref{phi_*: C^-infty_c/cv(M' Lambda) to C^-infty_c/.(M Lambda)}.

Similarly, if $L$ is of codimension one, there are continuous homomorphisms,
\begin{gather}
\phi_*:K_{\co/\cv}(M',L';\Lambda)\to K_{\co/{\cdot}}(M,L;\Lambda)\;,
\label{phi_*: K_c/cv(M' L' Lambda) to K_c/.(M L Lambda)}\\
\phi_*:J_{\co/\cv}(M',L';\Lambda)\to J_{\co/{\cdot}}(M,L;\Lambda)\;.
\label{phi_*: J_c/cv(M' L' Lambda) to J_c/.(M L Lambda)}
\end{gather}
which can be described as the compositions
\begin{gather}
K_{\co/\cv}(M',L';\Lambda)\xrightarrow{\pi_\topd}K_{\co/\cv}(M',L';\phi^*\Lambda M\otimes\Omega_\fiber)
\xrightarrow{\phi_*}K_{\co/{\cdot}}(M,L;\Lambda)\;,
\label{composition - push-forward - K-currents}\\
J_{\co/\cv}(M',L';\Lambda)\xrightarrow{\pi_\topd}J_{\co/\cv}(M',L';\phi^*\Lambda M\otimes\Omega_\fiber)
\xrightarrow{\phi_*}J_{\co/{\cdot}}(M,L;\Lambda)\;,
\label{composition - push-forward - J-currents}
\end{gather}
where the maps $\pi_\topd$ are induced by~\eqref{pi_top: Lambda M' to phi^*Lambda M otimes Omega_fiber}, and the maps $\phi_*$ are given by the analogs of~\eqref{phi_*: I_c/cv(M' L' phi^*E otimes Omega_fiber) -> I_c/.(M L E)} with $E=\Lambda M$ for~\eqref{phi_*: K(M' L' Omega_fiber) to K(M L)} and~\eqref{phi_*: J(M' L' Omega_fiber) to J(M L)}.

\subsection{Pull-back of dual-conormal currents}\label{ss: pull-back of dual-conormal currents}

Consider the notations and conditions of \Cref{ss: pull-back of dual-conormal distributions} (the same as in \Cref{ss: push-forward of conormal distribs}). The map~\eqref{phi^*: C^infty(M Lambda) to C^infty(M' Lambda)} has a continuous extension
\begin{equation}\label{phi^*: I'(M L Lambda) to I'(M' L' Lambda)}
\phi^*:I'(M,L;\Lambda)\to I'(M',L';\Lambda)\;,
\end{equation}
which can be given as the composition
\[
I'(M,L;\Lambda)\xrightarrow{\phi^*}I'(M',L';\phi^*\Lambda M)\xrightarrow{\phi^*}I'(M',L';\Lambda)\;,
\]
using~\eqref{phi^*: I'(M L E) -> I'(M' L' phi^*E)} like in~\eqref{composition - pull-back - conormal currents}. The map~\eqref{phi^*: I'(M L Lambda) to I'(M' L' Lambda)} is also a restriction of~\eqref{phi^*: C^-infty(M Lambda) to C^-infty(M' Lambda)}.

Similarly, when $L$ is of codimension one, there are continuous homomorphisms,
\begin{gather}
\phi^*:K'(M,L;\Lambda)\to K'(M',L';\Lambda)\;,\label{phi^*: K'(M L Lambda) to K'(M' L' Lambda)}\\
\phi^*:J'(M,L;\Lambda)\to J'(M',L';\Lambda)\;,\label{phi^*: J'(M L Lambda) to J'(M' L' Lambda)}
\end{gather}
which can be given as the compositions
\begin{gather*}
K'(M,L;\Lambda)\xrightarrow{\phi^*}K'(M',L';\phi^*\Lambda M)\xrightarrow{\phi^*}K'(M',L';\Lambda)\;,\\
J'(M,L;\Lambda)\xrightarrow{\phi^*}J'(M',L';\phi^*\Lambda M)\xrightarrow{\phi^*}J'(M',L';\Lambda)\;,
\end{gather*}
using the analogs of~\eqref{phi^*: I'(M L E) -> I'(M' L' phi^*E)} for~\eqref{phi^*: K'(M L) to K'(M' L')} and~\eqref{phi^*: J'(M L) to J'(M' L')} like in~\eqref{composition - pull-back - K-currents} and~\eqref{composition - pull-back - J-currents}.

\subsection{Push-forward of dual-conormal currents}\label{ss: push-forward of dual-conormal currents}

With the notations and conditions of \Cref{ss: push-forward of dual-conormal distributions}, assume also that the vertical subbundle of $\phi$ is oriented. Then the map~\eqref{phi_*: C^infty_c/cv(M' Lambda) to C^infty_c/.(M Lambda)} has a continuous extension
\begin{equation}\label{phi_*: I'_c/cv(M' L' Lambda) to I'_c/.(M L Lambda)}
\phi_*:I'_{\co/\cv}(M',L';\Lambda)\to I'_{\co/{\cdot}}(M,L;\Lambda)\;,
\end{equation}
which can be described as the composition
\[
I'_{\co/\cv}(M',L';\Lambda)\xrightarrow{\pi_\topd}I'_{\co/\cv}(M',L';\phi^*\Lambda M\otimes\Omega_\fiber)
\xrightarrow{\phi_*}I'_{\co/{\cdot}}(M,L;\Lambda)\;,
\]
using~\eqref{phi_*: I'_c/cv(M' L' phi^*E otimes Omega_fiber) -> I'_c/.(M L E)} like in~\eqref{composition - push-forward - conormal currents}. The map~\eqref{phi_*: I'_c/cv(M' L' Lambda) to I'_c/.(M L Lambda)} is also a restriction of~\eqref{phi_*: C^-infty_c/cv(M' Lambda) to C^-infty_c/.(M Lambda)}.

Similarly, if $L$ is of codimension one, there are continuous homomorphisms,
\begin{gather}
\phi_*:K'_{\co/\cv}(M',L';\Lambda)\to K'_{\co/{\cdot}}(M,L;\Lambda)\;,
\label{phi_*: K'_c/cv(M' L' Lambda) to K'_c/.(M L Lambda)}\\
\phi_*:J'_{\co/\cv}(M',L';\Lambda)\to J'_{\co/{\cdot}}(M,L;\Lambda)\;,
\label{phi_*: J'_c/cv(M' L' Lambda) to J'_c/.(M L Lambda)}
\end{gather}
which can be described as the compositions
\begin{gather*}
K'_{\co/\cv}(M',L';\Lambda)\xrightarrow{\pi_\topd}K'_{\co/\cv}(M',L';\phi^*\Lambda M\otimes\Omega_\fiber)
\xrightarrow{\phi_*}K'_{\co/{\cdot}}(M,L;\Lambda)\;,\\
J'_{\co/\cv}(M',L';\Lambda)\xrightarrow{\pi_\topd}J'_{\co/\cv}(M',L';\phi^*\Lambda M\otimes\Omega_\fiber)
\xrightarrow{\phi_*}J'_{\co/{\cdot}}(M,L;\Lambda)\;,
\end{gather*}
using the analogs of~\eqref{phi_*: I'_c/cv(M' L' phi^*E otimes Omega_fiber) -> I'_c/.(M L E)} for~\eqref{phi_*: K'(M' L' Omega_fiber) to K'(M L)} and~\eqref{phi_*: J'(M' L' Omega_fiber) to J'(M L)} like in~\eqref{composition - push-forward - K-currents} and~\eqref{composition - push-forward - J-currents}.

\begin{prop}\label{p: phi_* is transpose of phi^* - I-currents - I'-currents}
The compact-support cases of~\eqref{phi_*: I'_c/cv(M' L' Lambda) to I'_c/.(M L Lambda)}--\eqref{phi_*: J'_c/cv(M' L' Lambda) to J'_c/.(M L Lambda)} are transposes of~\eqref{phi^*: I(M L Lambda) to I(M' L' Lambda)},~\eqref{phi^*: K(M L Lambda) to K(M' L' Lambda)} and~\eqref{phi^*: J(M L Lambda) to J(M' L' Lambda)}; and~\eqref{phi^*: I'(M L Lambda) to I'(M' L' Lambda)}--\eqref{phi^*: J'(M L Lambda) to J'(M' L' Lambda)} are transposes of the compact-support cases of~\eqref{phi_*: I_c/cv(M' L' Lambda) to I_c/.(M L Lambda)},~\eqref{phi_*: K_c/cv(M' L' Lambda) to K_c/.(M L Lambda)} and~\eqref{phi_*: J_c/cv(M' L' Lambda) to J_c/.(M L Lambda)}.
\end{prop}

\begin{proof}
We have the commutative diagrams
\[
\begin{CD}
I_\co(M',L';\Lambda) @>{\phi_*}>> I_\co(M,L;\Lambda) \\
@AAA @AAA \\
\Cinftyc(M';\Lambda) @>{\phi_*}>> \Cinftyc(M;\Lambda)
\end{CD}
\qquad
\begin{CD}
I'(M',L';\Lambda) @<{\phi^*}<< I'(M,L;\Lambda)\phantom{\;.} \\
@VVV @VVV \\
C^{-\infty}(M';\Lambda) @<{\phi^*}<< C^{-\infty}(M;\Lambda)\;,
\end{CD}
\]
where the vertical arrows are continuous dense inclusions given by~\eqref{C^infty(M) subset I^(infty)(M L)} and~\eqref{C^-infty(M) supset I'(M L) supset C^infty(M)} with $\Lambda M$. By \Cref{p: phi_* is transpose of phi^* - forms - currents}, the transpose of the first diagram is
\[
\begin{CD}
I'(M',L';\Lambda) @<{(\phi_*)^\trans}<< I'(M,L;\Lambda)\phantom{\;.} \\
@VVV @VVV \\
C^{-\infty}(M';\Lambda) @<{\phi^*}<< C^{-\infty}(M;\Lambda)\;,
\end{CD}
\]
where the vertical arrows are again inclusion maps. Comparing the second and third diagrams, we get
\begin{equation}\label{(phi_*)^t = phi^*: I'(M L Lambda) to I'(M' L' Lambda)}
(\phi_*)^\trans=\phi^*:I'(M,L;\Lambda)\to I'(M',L';\Lambda)\;.
\end{equation}

The analogous argument with the commutative diagrams
\[
\begin{CD}
I(M',L';\Lambda) @<{\phi^*}<< I(M,L;\Lambda) \\
@AAA @AAA \\
C^\infty(M';\Lambda) @<{\phi^*}<< C^\infty(M;\Lambda)
\end{CD}
\qquad
\begin{CD}
I'_\co(M',L';\Lambda) @>{\phi_*}>> I'_\co(M,L;\Lambda) \\
@VVV @VVV \\
C^{-\infty}_\co(M';\Lambda) @>{\phi_*}>> C^{-\infty}_\co(M;\Lambda)
\end{CD}
\]
shows that
\begin{equation}\label{(phi^*)^t = phi_*: I'(M' L' Lambda) to I'(M L Lambda)}
(\phi^*)^\trans=\phi_*:I'_\co(M',L';\Lambda)\to I'_\co(M,L;\Lambda)\;.
\end{equation}

Next, consider the commutative diagrams
\[
\begin{CD}
K_\co(M',L';\Lambda) @>{\phi_*}>> K_\co(M,L;\Lambda) \\
@V{\iota}VV @VV{\iota}V \\
I_\co(M',L';\Lambda) @>{\phi_*}>> I_\co(M,L;\Lambda)
\end{CD}
\qquad
\begin{CD}
K'(M',L';\Lambda) @<{\phi^*}<< K'(M,L;\Lambda)\phantom{\;.} \\
@A{R'}AA @AA{R'}A \\
I'(M',L';\Lambda) @<{\phi^*}<< I'(M,L;\Lambda)\;.
\end{CD}
\]
As above, comparing the second one with the transposition of the first one, and using~\eqref{(phi_*)^t = phi^*: I'(M L Lambda) to I'(M' L' Lambda)} and the surjectivity of $R':I'(M,L;\Lambda)\to K'(M,L;\Lambda)$ (\Cref{ss: dual-conormal seq}), we get
\[
(\phi_*)^\trans=\phi^*:K'(M,L;\Lambda)\to K'(M',L';\Lambda)\;.
\] 

A similar argument with the commutative diagrams
\[
\begin{CD}
K(M',L';\Lambda) @<{\phi^*}<< K(M,L;\Lambda) \\
@V{\iota}VV @VV{\iota}V \\
I(M',L';\Lambda) @<{\phi^*}<< I(M,L;\Lambda)
\end{CD}
\qquad
\begin{CD}
K'_\co(M',L';\Lambda) @>{\phi_*}>> K'_\co(M,L;\Lambda)\phantom{\;,} \\
@A{R'}AA @AA{R'}A \\
I'_\co(M',L';\Lambda) @>{\phi_*}>> I'_\co(M,L;\Lambda)\;,
\end{CD}
\]
using~\eqref{(phi^*)^t = phi_*: I'(M' L' Lambda) to I'(M L Lambda)}, shows that
\[
(\phi^*)^\trans=\phi_*:K'_\co(M',L';\Lambda)\to K'_\co(M,L;\Lambda)\;.
\]

Now, consider the commutative diagrams
\[
\begin{CD}
J_\co(M',L';\Lambda) @>{\phi_*}>> J_\co(M,L;\Lambda) \\
@A{R}AA @AA{R}A \\
I_\co(M',L';\Lambda) @>{\phi_*}>> I_\co(M,L;\Lambda)
\end{CD}
\qquad
\begin{CD}
J'(M',L';\Lambda) @<{\phi^*}<< J'(M,L;\Lambda)\phantom{\;.} \\
@V{\iota'}VV @VV{\iota'}V \\
I'(M',L';\Lambda) @<{\phi^*}<< I'(M,L;\Lambda)\;.
\end{CD}
\]
Again, comparing the second one with the transposition of the first one, and using~\eqref{(phi_*)^t = phi^*: I'(M L Lambda) to I'(M' L' Lambda)} and the injectivity of $\iota':J'(M,L;\Lambda)\to I'(M,L;\Lambda)$ (\Cref{ss: dual-conormal seq}), we get
\[
(\phi_*)^\trans=\phi^*:J'(M,L;\Lambda)\to J'(M',L';\Lambda)\;.
\] 

Finally, the same argument with the commutative diagrams
\[
\begin{CD}
J(M',L';\Lambda) @<{\phi^*}<< J(M,L;\Lambda) \\
@A{R}AA @AA{R}A \\
I(M',L';\Lambda) @<{\phi^*}<< I(M,L;\Lambda)
\end{CD}
\qquad
\begin{CD}
J'_\co(M',L';\Lambda) @>{\phi_*}>> J'_\co(M,L;\Lambda)\phantom{\;,} \\
@V{\iota'}VV @VV{\iota'}V \\
I'_\co(M',L';\Lambda) @>{\phi_*}>> I'_\co(M,L;\Lambda)\;,
\end{CD}
\]
using~\eqref{(phi^*)^t = phi_*: I'(M' L' Lambda) to I'(M L Lambda)}, gives
\[
(\phi^*)^\trans=\phi_*:J'_\co(M',L';\Lambda)\to J'_\co(M,L;\Lambda)\;.\qedhere
\]
\end{proof}

\section{Witten's perturbation of the de~Rham complex}\label{s: Witten}

\subsection{Witten's complex}\label{ss: Witten's complex}

The notation $z=\mu+i\lambda\in\C$ ($i=\sqrt{-1}$) will be used for a complex parameter. Any closed real $\eta\in C^\infty(M;\Lambda^1)$ induces the \emph{Witten's operators} on $C^\infty(M;\Lambda)$, depending on the parameter $z\in\C$ \cite{Witten1982,Novikov1981,Novikov1982,Pajitnov1987,BravermanFarber1997}, \index{Witten's operators}
\begin{equation}\label{perturbed opers}
\left\{
\begin{aligned}
d_z&=d+z\,\eta\wedge\;,\quad
\delta_z=d_z^*=\delta-\bar z\,{\eta\lrcorner}\;,\\
D_z&=d_z+\delta_z=D+\mu\hat c(\eta)+i\lambda c(\eta)\;,\\
\Delta_z&=D_z^2=d_z\delta_z+\delta_zd_z=\Delta+\mu\sH_\eta-i\lambda\sJ_\eta+|z|^2|\eta|^2\;,
\end{aligned}
\right.
\end{equation}
where $\sH_\eta=\LL_V+\LL_V^*$ is of order zero and $\sJ_\eta=\LL_V-\LL_V^*$ is of order one. Here, $d_z$ is an elliptic complex, giving rise to the \emph{twisted cohomology} $H_z^\bullet(M)$, whose isomorphism class depends only on the real class $\xi:=[\eta]\in H^1(M)$ and $z\in\C$. The more explicit notation $d_{z\eta}$, $\delta_{z\eta}$, $D_{z\eta}$ and $\Delta_{z\eta}$ may be used if needed.

Suppose the manifold $M$ is closed, and let $n=\dim M$. Then $\Delta_z$ has a discrete spectrum, and the perturbed operators satisfy~\eqref{Hodge}. We get the \emph{twisted Betti numbers}, $\beta_z^k=\beta_z^k(M,\xi)=\dim H_z^k(M)$ ($k=0,\dots,n$), \index{$\beta_z^k$} whose alternate sum is the Euler characteristic, $\sum_k(-1)^k\beta_z^k=\chi(M)$ \cite[Proposition~1.40]{Farber2004}. Every $\beta_z^k$ is independent of $z$ outside a discrete subset of $\C$, where $\beta_z^k$ jumps (Mityagin and Novikov \cite[Theorem~1]{Novikov2002}). This ground value of $\beta_z^k$, denoted by $\betaNo^k=\betaNo^k(M,\xi)$, is called the $k$th \emph{Novikov Betti number}. \index{$\betaNo^k$} \index{Novikov Betti number} Moreover $\beta_z^k=\betaNo^k$ for $|\mu|\gg0$ \cite[Theorem~2.8]{Farber1995}, \cite[Lemma~1.3]{BravermanFarber1997} \cite[Eq.~(2.9)]{AlvKordyLeichtnam-ziomf}.

Since $\eta$ is real, we have $\overline{d_z\alpha}=d_{\bar z}\bar\alpha$ for all $\alpha\in C^\infty(M;\Lambda)$. So conjugation induces a $\C$-antilinear isomorphism $H_z^k(M)\cong H_{\bar z}^k(M)$, yielding $\beta_z^k=\beta_{\bar z}^k$.

For $\alpha\in C^\infty(M;\Lambda^r)$ and $\beta\in C^\infty(M;\Lambda)$, we have
\begin{equation}\label{d antiderivative with d_z and d_-z}
d(\alpha\wedge\beta)=d_z\alpha\wedge\beta+(-1)^r\alpha\wedge d_{-z}\beta\;.
\end{equation}
It follows that the mappings $(\alpha,\beta)\mapsto\alpha\wedge\beta$ and $(\alpha,\beta)\mapsto\alpha\wedge\bar\beta$ induce maps,
\begin{equation}\label{bilinear and sesquilinear maps}
H_z^r(M)\times H_{-z}^s(M) \to H^{r+s}(M)\;,\quad
H_z^r(M)\times H_{-\bar z}^s(M) \to H^{r+s}(M)\;,
\end{equation}
the first one is bilinear and the second one is sesquilinear.  By density and continuity, the formula~\eqref{d antiderivative with d_z and d_-z} has an extension to the product~\eqref{extension of exterior product to currents} of smooth differential forms and currents.

\subsection{Interpretation as coefficients in a flat line bundle}\label{ss: flat line bundle}

If $\eta=dF$ for some real function $F\in C^\infty(M)$, we get the original operators introduced by Witten \cite{Witten1982}, which satisfy
\begin{equation}\label{Witten's opers}
\left\{
\begin{gathered}
d_z=e^{-zF}\,d\,e^{zF}=e^{-i\lambda F}\,d_\mu\,e^{i\lambda F}\;,\quad
\delta_z=e^{\bar zF}\,\delta\,e^{-\bar zF}=e^{-i\lambda F}\,\delta_\mu\,e^{i\lambda F}\;,\\
D_z=e^{-i\lambda F}\,D_\mu\,e^{i\lambda F}\;,\quad
\Delta_z=e^{-i\lambda F}\,\Delta_\mu\,e^{i\lambda F}\;. 
\end{gathered}
\right.
\end{equation}
Thus we have an isomorphism of differential complexes,
\[
e^{zF}:(C^\infty(M;\Lambda),d_z)\xrightarrow{\cong}(C^\infty(M;\Lambda),d)\;,
\]
which induces an isomorphism $H_z^\bullet(M)\cong H^\bullet(M)$.

Let $\LL$ be the trivial line bundle $M\times\C$ with the flat structure that corresponds to the trivial flat structure by the multiplication isomorphism $e^F:\LL\to M\times\C$, $(p,u)\mapsto(p,e^{F(p)}u)$. Its flat covariant derivative is determined by the condition $d^\LL1=dF$. Every power $\LL^z$ is similarly defined by the function $zF$. We have $d_z\equiv d^{\LL^z}$ on $C^\infty(M;\Lambda)\equiv C^\infty(M;\Lambda\otimes\LL^z)$. Moreover $\delta_z\equiv\delta^{\LL^z}$ and $\Delta_z\equiv\Delta^{\LL^z}$ using the standard Hermitian structure on $\LL^z$. 

For arbitrary $\eta$, take the minimal regular covering $\pi:\widetilde M\to M$ so that the lift $\tilde\eta$ of $\eta$ is exact, say $\tilde\eta=dF$ for some real function $F\in C^\infty(\widetilde M)$. Thus $d_{\widetilde M,z}=e^{-zF}\,d_{\widetilde M}\,e^{zF}$ on $C^\infty(\widetilde M;\Lambda)$ corresponds to $d_{M,z}$ on $C^\infty(M;\Lambda)$ via the injection $\pi^*:C^\infty(M;\Lambda)\to C^\infty(\widetilde M;\Lambda)$. Let $\Gamma=\Aut(\pi)$ be the group of deck transformations of $\widetilde M$. The action of every $\gamma\in\Gamma$ will be denoted by $T_\gamma$ or by $\tilde p\mapsto\gamma\cdot\tilde p$. Since $dF$ is $\Gamma$-invariant, there is a monomorphism $\Gamma\to\R$, $\gamma\mapsto c_\gamma$, so that $F(\gamma\cdot\tilde p)=F(\tilde p)+c_\gamma$ for all $\tilde p\in\widetilde M$; its image is the group of periods of the cohomology class $[\eta]$. 

Let $\widetilde\LL$ be the flat line bundle over $\widetilde M$ defined with $F$ as above. The flat structure of $\widetilde\LL$ is invariant by the first factor action of $\Gamma$ on $\widetilde\LL$, given by $\gamma\cdot_1(\tilde p,u)=(\gamma\cdot\tilde p,u)$. Thus the corresponding quotient Hermitian line bundle $\LL\equiv M\times\C$ has an induced flat structure determined by the condition $d^\LL1=z\eta$. We have $d_z\equiv d^{\LL^z}$, $\delta_z\equiv\delta^{\LL^z}$ and $\Delta_z\equiv\Delta^{\LL^z}$ on $C^\infty(M;\Lambda)\equiv C^\infty(M;\Lambda\otimes\LL^z)$.

Using the monomorphism $\Gamma\to\R^\times$, $\gamma\mapsto a_\gamma:=e^{c_\gamma}$, we can also define the diagonal action of $\Gamma$ on $\widetilde M\times\C$, $\gamma\cdot(\tilde p,u)=(\gamma\cdot\tilde p,a_\gamma u)$, which preserves the vector bundle and trivial flat structures. Moreover the isomorphism $e^F:\widetilde\LL\to\widetilde M\times\C$ is equivariant with respect to the first factor and diagonal actions of $\Gamma$. Hence $\LL$ can be also described as the quotient of the trivial flat line bundle $\widetilde M\times\C$ by the diagonal action of $\Gamma$. Let $\widetilde\omega\in C^\infty(\widetilde M;\widetilde\LL)$ be defined by $\widetilde\omega(\tilde p)=(\tilde p,e^{F(\tilde p)})$, which corresponds to $1\in C^\infty(\widetilde M)\equiv C^\infty(\widetilde M;\widetilde\LL)$ by the isomorphism $e^F:\widetilde\LL\to\widetilde M\times\C$. This section is $\Gamma$-invariant and satisfies $d^{\widetilde\LL}\tilde\omega=\tilde\eta\otimes\widetilde\omega$ in $C^\infty(\widetilde M;\Lambda\otimes\widetilde\LL)$. So it induces a non-vanishing section $\omega$ of $\LL$ satisfying $d^\LL\omega=\eta\otimes\omega$ in $C^\infty(M;\Lambda\otimes\LL)$. Furthermore
\begin{equation}\label{C^-infty(M Lambda otimes LL^z) equiv C^-infty(M Lambda)}
\left\{
\begin{gathered}
C^{\pm\infty}(M;\Lambda\otimes\LL^z)\equiv C^{\pm\infty}(M;\Lambda)\otimes\R\omega^z\equiv C^{\pm\infty}(M;\Lambda)\;,\\
d^{\LL^z}\equiv d_z\otimes1\equiv d_z\;,\quad H_z^\bullet(M)\equiv H^\bullet(M,\LL^z)\;,\\
\delta\equiv\delta_z\otimes1\equiv\delta_z\;,\quad
D\equiv D_z\otimes1\equiv D_z\;,\quad\Delta\equiv\Delta_z\otimes1\equiv\Delta_z\;,
\end{gathered}
\right.
\end{equation}
writing $d=d^{\LL^z}$, $\delta=\delta^{\LL^z}$, $D=D^{\LL^z}$, $\Delta=\Delta^{\LL^z}$. Since $(\LL^z)^*\equiv\LL^{-z}$, this gives an interpretation of~\eqref{d antiderivative with d_z and d_-z} and~\eqref{bilinear and sesquilinear maps}.

\subsection{Witten's perturbation vs pull-back and push-forward homomorphisms}
\label{ss: Witten vs pull-back and push-forward}

For a smooth map $\phi:M'\to M$, let $\eta'=\phi^*\eta$. The homomorphism $\phi^*:C^\infty(M;\Lambda)\to C^\infty(M';\Lambda)$ satisfies $\phi^*d_{z\eta}=d_{z\eta'}\phi^*$. If $\phi$ is a smooth submersion, then $\phi_*:C^\infty_{\co/\cv}(M';\Lambda)\to C^\infty_{\co/{\cdot}}(M;\Lambda)$ satisfies $\phi_*d_{z\eta'}=d_{z\eta}\phi_*$ by \cite[Proposition~I.6.14 and~I.6.15~(a)]{BottTu1982}.

\subsection{Perturbation of pull-back homomorphisms}\label{ss: perturbation of pull-back homs}

Consider the notation of \Cref{ss: flat line bundle}. For a smooth map $\phi:M\to M$, take a lift $\tilde\phi:\widetilde M\to\widetilde M$. Then $\tilde\phi^*_z:=e^{-zF}\,\tilde\phi^*\,e^{zF}=e^{z(\tilde\phi^*F-F)}\,\tilde\phi^*$ is an endomorphism of $(C^\infty(\widetilde M;\Lambda),d_{\widetilde M,z})$ by~\eqref{Witten's opers}. We have $T_\gamma^*(\tilde\phi^*F-F)=\tilde\phi^*F-F$ for all $\gamma\in\Gamma$, obtaining $T_\gamma^*\tilde\phi^*_z=\tilde\phi^*_zT_\gamma^*$. So $\tilde\phi^*_z$ induces an endomorphism $\phi^*_z$ of $(C^\infty(M;\Lambda),d_z)$, which depends on the choice of the lift $\tilde\phi$ of $\phi$. In the case of a flow $\phi=\{\phi^t\}$ on $M$, there is a unique lift to a flow $\tilde\phi=\{\tilde\phi^t\}$ on $\widetilde M$, giving rise to a canonical definition of $\phi^{t*}_z$, \index{$\phi^{t*}_z$} called the \emph{perturbation} of $\phi^{t*}$ defined by $\eta$ with parameter $z$.

\subsection{Witten's operators on oriented manifolds}\label{ss: perturbed ops on oriented mfds}

In this subsection, assume $M$ is oriented. If moreover $M$ is closed, then the maps~\eqref{bilinear and sesquilinear maps} and integration on $M$ define nondegenerate pairings,
\begin{equation}\label{H_z^k(M) times H_-z^n-kM) to C}
H_z^k(M)\times H_{-z}^{n-k}(M)\to\C\;,\quad
H_z^k(M)\times H_{-\bar z}^{n-k}(M)\to\C\;,
\end{equation}
the first one is bilinear and the second one is sesquilinear. Therefore $\beta_z^k=\beta_{-z}^{n-k}=\beta_{-\bar z}^{n-k}=\beta_{\bar z}^k$.

\subsection{Witten's operators vs Hodge star operator}\label{ss: Witten vs Hodge star}

Continuing with the condition of orientation, the equalities~\eqref{star} yield
\begin{equation}\label{delta_z}
\left\{
\begin{gathered}
\delta_z=(-1)^{nk+n+1}\star\,d_{-\bar z}\,\star=(-1)^{nk+n+1}\,\bar\star\,d_{-z}\,\bar\star\;,\\
\begin{alignedat}{3}
d_z\,\star&=(-1)^k\star\,\delta_{-\bar z}\;,&\quad
\delta_z\,\star&=(-1)^{k+1}\star\,d_{-\bar z}\;,&\quad
\Delta_z\,\star&=\star\,\Delta_{-\bar z}\;,\\
d_z\,\bar\star&=(-1)^k\,\bar\star\,\delta_{-z}\;,&\quad\delta_z\,\bar\star&=(-1)^{k+1}\,\bar\star\;d_{-z}\;,&\quad
\Delta_z\,\bar\star&=\bar\star\,\Delta_{-z}\;.
\end{alignedat}
\end{gathered}
\right.
\end{equation}
Then we get a linear isomorphism $\star:\ker\Delta_z\to\ker\Delta_{-\bar z}$ and an antilinear isomorphism $\bar\star:\ker\Delta_z\to\ker\Delta_{-z}$. If $M$ is closed, they induce an explicit linear isomorphism $H_z^k(M)\cong H_{-\bar z}^{n-k}(M)$ and an antilinear isomorphism $H_z^k(M)\cong H_{-z}^{n-k}(M)$ by~\eqref{Hodge}.

Using~\eqref{d antiderivative with d_z and d_-z} and the Stokes theorem, we get
\begin{equation}\label{d_z equiv (-1)^k+1 d_-z^t}
d_z\equiv(-1)^{k+1}\,d_{-z}^\trans\;,
\end{equation}
as maps $C^{-\infty}(M;\Lambda^k)\to C^{-\infty}(M;\Lambda^{k+1})$ using~\eqref{C^-infty(M Lambda^r) equiv C^infty_c(M Lambda^n-r)'}. This identity also follows from~\eqref{star},~\eqref{delta_z} and~\eqref{(alpha beta) = (-1)^rn+r langle alpha star beta rangle}: for $\alpha\in C^\infty(M;\Lambda^k)$ and $\beta\in C^\infty(M;\Lambda^{n-k-1})$,
\begin{align*}
(d_z\alpha,\beta)&=(-1)^{(k+1)n+k+1}\langle d_z\alpha,\bar\star\beta\rangle
=(-1)^{(k+1)n+k+1}\langle\alpha,\delta_z{\bar\star\beta}\rangle\\
&=(-1)^{(kn+1}\langle\alpha,\bar\star d_{-z}\beta\rangle=(-1)^{k+1}(\alpha,d_{-z}\beta)\;.
\end{align*}
This argument also applies to $\delta_z$ and $\Delta_z$, giving
\begin{gather}
\delta_z\equiv(-1)^k\,\delta_{-z}^\trans:C^{-\infty}(M;\Lambda^k)\to C^{-\infty}(M;\Lambda^{k-1})\;,\notag\\
\Delta_z\equiv\Delta_{-z}^\trans:C^{-\infty}(M;\Lambda^k)\to C^{-\infty}(M;\Lambda^k)\;.
\label{Delta_z equiv Delta_-bar z^t}
\end{gather}

\subsection{Perturbed operators with two parameters}\label{ss: two parameters} 

We will also consider perturbed operators of the form
\[
D_{z,z'}=d_z+\delta_{z'}\;,\quad\Delta_{z,z'}=D_{z,z'}^2=d_z\delta_{z'}+\delta_{z'}d_z\;,
\]
depending on two parameters $z,z'\in\C$. They are not symmetric if $z\ne z'$, but their leading symbol is symmetric.

\subsection{Witten's operators on manifolds of bounded geometry}\label{ss: Witten - mfds of bd geom}

Consider now the notation of \Cref{ss: diff forms and currents,ss: Witten's complex}. Assume $M$ is of bounded geometry and $\eta\in\Cinftyub(M;\Lambda^1)$ (\Cref{ss: uniform sps}). Then the differential complex $d_z$ is uniformly bounded and uniformly elliptic for all $z\in\C$. 

Using also the notation of \Cref{ss: perturbation of pull-back homs}, assume that $\phi:M\to M$ is of bounded geometry. Then $\tilde\phi^*F-F$ induces a function in $\Cinftyub(M)$. For $m\in\N_0\cup\{\infty\}$, it follows from~\eqref{phi^* on C_ub^m} that $\phi^*_z$ defines a continuous linear endomorphism of $C_{\text{\rm ub}}^m(M;\Lambda)$ f. If moreover $\phi:M\to M$ is uniformly metrically proper, then, by~\eqref{phi^* on H^m}, $\phi^*_z$ also defines a continuous linear endomorphism of $H^m(M;\Lambda)$. 

If $\phi$ is a diffeomorphism and both of $\phi^{\pm1}$ are of bounded geometry, then $\phi^*_z$ defines a continuous linear endomorphism of $H^m(M;\Lambda)$ for all $m\in\Z\cup\{\pm\infty\}$. To show this, we can assume $M$ is oriented with a standard argument using the covering of orientations. Then, by the version of second equality of~\eqref{H^-s(M E) = Psi^s(M E) cdot L^2(M E) = H^s(M E^* otimes Omega)'} for open manifolds and~\eqref{(Lambda^rM)^* otimes Omega M equiv Lambda^n-rM}, $\phi^*_z$ on $H^{-m}(M;\Lambda)$ ($m\in\N_0\cup\{\infty\}$) is the transpose of $(\phi^{-1})^*_{-z}$ on $H^m(M;\Lambda^{n-\bullet})$.

In the cases of $C_{\text{\rm ub}}^\infty(M;\Lambda)$ and $H^{\pm\infty}(M;\Lambda)$, all of the above endomorphisms are cochain maps with $d_z$.

The symmetric hyperbolic equation
\begin{equation}\label{wave}
\partial_t\alpha_t=iD_z\alpha_t\;,\quad\alpha_0=\alpha\;,
\end{equation}
on any open subset of $M$ and with $t$ in any interval containing $0$, any solution satisfies the finite propagation speed property \cite[Proof of Proposition~1.1]{Chernoff1973} (see also \cite[Theorem~1.4]{CheegerGromovTaylor1982}, \cite[Proof of Proposition~7.20]{Roe1998})
\begin{equation}\label{unit propagation speed}
\supp\alpha_t\subset\Pen(\supp\alpha,|t|)\;.
\end{equation}
In particular, given any $\alpha\in C^\infty(M;\Lambda)$, this is true for $\alpha_t=e^{itD_z}\alpha$.

For $\psi\in\RR$ (\Cref{ss: Sobolev bd geom}), we may use the notation $k_z=k_{\psi,z}=K_{\psi(D_z)}$, where $\psi(D_z)$ is given by the spectral theorem. We may also use the notation $k_{u,z}=k_{\psi_u,z}$ for any family of functions $\psi_u\in\RR$ depending on a parameter $u$.  

For any $\psi\in\SS$ (\Cref{ss: Z}), we have \cite[Proof of Theorem~5.5]{Roe1988I}
\begin{equation}\label{psi(D_z)}
\psi(D_z)=(2\pi)^{-1}\int_{-\infty}^{+\infty} e^{i\xi D_z}\hat\psi(\xi)\,d\xi\;.
\end{equation}
According to Remark~\ref{r: supp K_A subset r-penumbra <=> supp Au subset r-penumbra for all u}, it follows from~\eqref{unit propagation speed} and~\eqref{psi(D_z)} that, for all $r>0$,
\begin{equation}\label{supp hat psi subset [-R R] => supp K_psi(D_z) subset ...}
\supp\hat\psi\subset[-r,r]\Rightarrow\supp k_{\psi,z}\subset\{\,(p,q)\in M^2\mid d(p,q)\le r\,\}\;.
\end{equation}

For instance, for $\psi_u(x)=e^{-ux^2}$ ($u>0$), we get the \emph{perturbed heat kernel} $k_{u,z}=K_{e^{-u\Delta_z}}$. It satisfies the following estimate like the usual heat kernel \cite{ButtigEichhorn1991}: for all $u_0>0$ and $m_1,m_2,m_3\in\N_0$, there are $C_1,C_2>0$ so that, for all $0<u\le u_0$, 
\begin{equation}\label{heat kernel estimates}
|\partial_u^{m_1}\nabla_p^{m_2}\nabla_q^{m_3}k_{z,u}(p,q)|\le C_1u^{-(n+m_2+m_3)/2-m_1}e^{-C_2d^2(p,q)/u}\;.
\end{equation}
In particular, $k_{z,u}\in\Cinftyub(M^2;\Lambda\boxtimes(\Lambda^*\otimes\Omega))$ for every $u>0$.

To estimate more general kernels, consider the Fr\'echet algebra and $\C[z]$-module $\AA$ \index{$\AA$} which consists of the functions $\psi:\R\to\C$ that can be extended to entire functions on $\C$ such that, for every compact $K\subset\R$, the set $\{\,x\mapsto \psi(x+iy)\mid y\in K\,\}$ is bounded in $\SS$ \cite[Section~4]{Roe1987}. It has the following properties: $\AA\subset\SS$; $\AA$ contains all functions with compactly supported smooth Fourier transform, as well as the Gaussian $x\mapsto e^{-x^2}$; if $\psi\in\AA$ and $u>0$, then $\psi_u\in\AA$, where $\psi_u(x)=\psi(ux)$; and, by the Paley-Wiener theorem, for every $\psi\in\AA$ and $c>0$, there is some $A_c>0$ such that, for all $\xi\in\R$,
\begin{equation}\label{|hat psi(xi)| le A_c e^-c |xi|}
\big|\hat\psi(\xi)\big|\le A_ce^{-c|\xi|}\;.
\end{equation}
Define the semi-norms $\|{\cdot}\|_{\AA,C,r}$ ($C>0$ and $r\in\N_0$) on $\AA$ by
\[
\|\psi\|_{\AA,C,r}=\max_{j+k\le r} \int_{-\infty}^{+\infty}|\xi^j\partial^k_{\xi}\hat{\psi}(\xi)|\,e^{C|\xi|}\,d\xi\;.
\]

\begin{lem}\label{l: exponential decay}
If $\psi\in\AA$ and $N>n/2$, then, for any $W>0$, there is some $C'_1=C'_1(z,W)>0$ such that, for all $p,q\in M$ and $m,m_1,m_2\in\N_0$ with $m_1+m_2\le m$,
\[
|\nabla_p^{m_1}\nabla_q^{m_2}k_z(p,q)|\le C'_1 e^{-Wd(p,q)}  \|\psi\|_{\AA, W,N+m}\;.
\]
\end{lem}

\begin{proof}
Using~\eqref{unit propagation speed},~\eqref{psi(D_z)} and the Sobolev embedding theorem, one can show that, for every $\epsilon>0$, there is some $C_0=C_0(z,\epsilon)>0$ so that, for all $\psi\in \AA$ and $p,q\in M$,
\[
|k_z(p,q)|\le C_0\int_{|\xi|>d(p,q)-\epsilon} \big|(1-\partial_\xi^2)^N\hat\psi(\xi)\big|\,d\xi\;.
\]
Hence, for some fixed $\epsilon>0$, we obtain that, for any $W>0$, there is some $C_1=C_1(z,W)>0$ such that, for all $p,q\in M$, 
\begin{align*}
|k_z(p,q)| &\le C_0\int_{|\xi|>d(p,q)-\epsilon}  e^{-W|\xi|}\,\big|(1-\partial_\xi^2)^N\hat\psi(\xi)\big|\, e^{W|\xi|}\,d\xi \notag\\ 
&\le C_1 e^{-Wd(p,q)}  \int_{-\infty}^{+\infty}\big|(1-\partial_\xi^2)^N\hat\psi(\xi)\big|\, e^{W|\xi|} \,d\xi \notag\\
&= C_1 e^{-Wd(p,q)}  \|\psi\|_{\AA, W,N}\;.
\end{align*}
By using $(1+x^2)^m\psi(x)$ ($m\in\N_0$) instead of $\psi(x)$, we also get
\[
|(1+\Delta_{z,p})^{m_1}(1+\Delta_{-z,q})^{m_2}k_z(p,q)|\le C_1 e^{-Wd(p,q)}\|\psi\|_{\AA, W,N+m}\;,
\]
according to~\eqref{C^-infty(M Lambda^r) equiv C^infty_c(M Lambda^n-r)'} and~\eqref{Delta_z equiv Delta_-bar z^t}, yielding the estimate of the statement.
\end{proof}

\subsection{Witten's operators on regular coverings of compact manifolds}
\label{ss: perturbed Schwartz kernels on regular coverings of compact mfds}

Let $\pi:\widetilde M\to M$, $\Gamma$, $\gamma\cdot\tilde p$, $T_\gamma$ and $g_{\widetilde M}$ be like in \Cref{ss: flat line bundle}. Recall that $\widetilde M$ is bounded geometry with $g_{\widetilde M}$. 

Let $|{\cdot}|:\Gamma\to\N_0$ denote the word length function defined by any finite set of generators $\gamma_1,\dots,\gamma_k$ of $\Gamma$; recall that $|\gamma|$ is the minimum length of the expressions of $\gamma$ as products of elements $\gamma_i^{\pm1}$. It is well known that there is some $c_1\ge1$ such that
\begin{equation}\label{word-metric}
c_1^{-1}\,|\gamma|\le d_{\widetilde M}(\gamma\cdot \tilde p,\tilde p)\le c_1\,|\gamma|
\end{equation}
for all $\tilde p\in\widetilde M$ and $\gamma\in\Gamma$. Therefore, given any compact $\bfK\subset\widetilde M^2$, we have
\begin{equation}\label{word-metric with bfK}
c_1^{-1}\,|\gamma|-c_2 \le d_{\widetilde M}(\gamma\cdot \tilde p,\tilde q) \le c_1\,|\gamma|+c_2
\end{equation}
for all $\gamma\in\Gamma$ and $(\tilde p,\tilde q)\in\bfK$, where $c_2=\max d_{\widetilde M}(\bfK)\ge0$. 

Let $\eta$ be a closed real $1$-form on $M$ whose lift to $\widetilde M$ is exact; say $\tilde\eta=d_{\widetilde M}F$ for some $F\in C^\infty(\widetilde M,\R)$. For $z\in\C$, let $D_z=D_{M,z}$, $\Delta_z=\Delta_{M,z}$, $\widetilde D_z=D_{\widetilde M,z}$ and $\widetilde\Delta_z=\Delta_{\widetilde M,z}$ (\Cref{ss: Witten's complex}). For any $\psi\in\RR$, let $k_z=K_{\psi(D_z)}$ and $\tilde k_z=K_{\psi(\widetilde D_z)}$ (\Cref{ss: Witten - mfds of bd geom}). For every $\tilde p\in\widetilde M$, let $[\tilde p]=\pi(\tilde p)$. We look for conditions on $\psi$ to get
\begin{equation}\label{k([tilde p] [tilde q]) equiv ...}
k_z([\tilde p],[\tilde q])\equiv\sum_\gamma T_\gamma^*\tilde k_z(\gamma\cdot\tilde p,\tilde q)
\end{equation}
for all $\tilde p,\tilde q\in\widetilde M$, using the identity
\[
\Lambda_{\gamma\cdot\tilde p}\widetilde M\boxtimes(\Lambda_{\tilde q}\widetilde M^*\otimes\Omega_{\tilde q}\widetilde M)
\equiv\Lambda_{[p]} M\boxtimes(\Lambda_{[q]} M^*\otimes\Omega_{[q]} M)\;.
\]
In particular,~\eqref{k([tilde p] [tilde q]) equiv ...} holds if $\hat\psi\in\Cinftyc(\R)$, which can be proved as follows. In this case, $\tilde k_z$ is supported in a penumbra of the diagonal (\Cref{ss: diff ops of bd geom}). By~\eqref{word-metric with bfK}, taking $\bfK=\bfF^2$ for some fundamental domain $\bfF\subset\widetilde M$, it follows that the right-hand side of~\eqref{k([tilde p] [tilde q]) equiv ...} has a finite number of nonzero terms. So it defines a smooth section on $M^2$, which can be checked to be $k_z$ using~\eqref{K_A(cdot q)(u) = A delta_q^u}. 

Examples where~\eqref{k([tilde p] [tilde q]) equiv ...} fails are easy to construct. For instance, if $\Gamma$ is non-amenable, it is well known that the spectrum of $\widetilde\Delta$ on functions has a gap of the form $(0,\epsilon)$ for some $\epsilon>0$, and therefore~\eqref{k([tilde p] [tilde q]) equiv ...} fails for $\psi(D)$ and $\psi(\widetilde D)$ if $\psi$ is even and supported in $(-\epsilon,\epsilon)$, with $\psi(0)\ne0$.

Consider the Fr\'echet algebra and $\C[z]$-module $\AA$ of \Cref{ss: Witten - mfds of bd geom}.

\begin{prop}\label{p: k(y y') equiv ...}
If $\psi\in\AA$, then~\eqref{k([tilde p] [tilde q]) equiv ...} holds, where the series is convergent in the Fr\'echet space $C^\infty(\widetilde M^2;\Lambda\widetilde M\boxtimes(\Lambda\widetilde M^*\otimes\Omega\widetilde M))$.
\end{prop}

\begin{proof}
First, let us prove that the series is uniformly convergent with all covariant derivatives on any fixed compact subset $\bfK\subset{\widetilde M}^2$.

By \Cref{l: exponential decay}, for any $W>0$ and $N>n/2$, there is some $C_1=C_1(z,W)>0$ such that, for all $\tilde p,\tilde q\in\widetilde M$, $m\in\N_0$ and $m_1+m_2\le m$,
\[
\big|\nabla_{\tilde p}^{m_1}\nabla_{\tilde q}^{m_2}\tilde k_z(\tilde p,\tilde q)\big|
\le C_1 e^{-Wd_{\widetilde M}(\tilde p,\tilde q)}  \|\psi\|_{\AA, W,N+m}\;.
\]
Then, by~\eqref{word-metric with bfK}, 
\begin{equation}\label{|... tilde k_z(gamma cdot tilde y tilde y')|}
\big|\nabla_{\tilde p}^{m_1}\nabla_{\tilde q}^{m_2}\tilde k_z(\gamma\cdot\tilde p,\tilde q)\big|\\
\le C'_1 e^{-\frac{W}{c_1}\,|\gamma|}\,\|\psi\|_{\AA, W, N+m}
\end{equation}
for $\gamma\in\Gamma$, $(\tilde p,\tilde q)\in\bfK$ and $m_1+m_2\le m$, where $C'_1=C_1e^{Wc_2}$. Since the growth of $\Gamma$ is at most exponential, there is some $W_0>0$ such that
\begin{equation}\label{sum_gamma in Gamma e^-W_0 |gamma| < infty}
\sum_{\gamma\in \Gamma} e^{-W_0\,|\gamma|}<\infty\;.
\end{equation}
Choosing $W>c_1W_0$, it follows from~\eqref{|... tilde k_z(gamma cdot tilde y tilde y')|} and~\eqref{sum_gamma in Gamma e^-W_0 |gamma| < infty} that there is some $C=C(z,\bfK,W,N)>0$ such that
\begin{equation}\label{| sum_gamma in Gamma ... T_gamma^* tilde k_z(gamma cdot tilde y tilde y') |}
\Big|\sum_{\gamma\in \Gamma}\nabla_{\tilde p}^{m_1}\nabla_{\tilde q}^{m_2}
T_\gamma^*\tilde k_z(\gamma\cdot\tilde p,\tilde q)\Big|
\le C\,\|\psi\|_{\AA, W, N+m}\;.
\end{equation}
So the series in~\eqref{k([tilde p] [tilde q]) equiv ...} is uniformly convergent on $\bfK$ with all covariant derivatives.

The identity~\eqref{k([tilde p] [tilde q]) equiv ...} for any $\psi\in\AA$ follows from~\eqref{| sum_gamma in Gamma ... T_gamma^* tilde k_z(gamma cdot tilde y tilde y') |}, approximating $\psi$ in $\AA$ by a sequence of functions with compactly supported Fourier transform.
\end{proof}

\begin{rem}\label{r: k(y y') equiv ...}
\Cref{p: k(y y') equiv ...} will be applied to an abelian covering. In that case, or, more generally, when $\Gamma$ has polynomial growth, its proof can be slightly modified so that it works for any $\psi\in\SS$. However not only this proposition, but also the estimate~\eqref{|... tilde k_z(gamma cdot tilde y tilde y')|} will be used later, and we need $\psi\in\AA$ to get the exponential factor of this estimate.
\end{rem}

\subsection{Local index formula for the Witten's complex}\label{ss: local index formulae}

Suppose $M$ is of bounded geometry and consider the perturbed heat operator $e^{-t\Delta_z}$ ($t>0$) in $L^2(M;\Lambda)$, defined by the spectral theorem. By the ellipticity of $\Delta_z$, the operator $e^{-t\Delta_z}$ is smoothing and let $k_{z,t}\in\Cinftyub(M^2;\Lambda\boxtimes(\Lambda^*\otimes\Omega))$ denote its Schwartz kernel (the perturbed heat kernel). It has an asymptotic expansion as $t\downarrow0$ in $\Cinftyub(M^2;\Lambda\boxtimes(\Lambda^*\otimes\Omega))$ of the form
\begin{equation}\label{asymptotic expansion}
k_{z,t}(p,q)\sim h_t(p,q)\sum_{j=0}^\infty t^j\Theta_{z,j}(p,q)\cdot|{\dvol}|(q)\;,
\end{equation}
where $|{\dvol}|$ denotes the Riemannian density and
\[
h_t(p,q)=\frac{1}{(4\pi t)^{n/2}}e^{-d(p,q)^2/4t}\;,\quad\Theta_{z,j}\in\Cinftyub(M^2;\Lambda\boxtimes\Lambda^*)\;.
\]
This expression can be formally differentiated to obtain also asymptotic expansions of the derivatives of $k_{z,t}(p,q)$ with respect to $t$, $p$ and $q$. On the diagonal $\Delta\subset M^2$, the terms $\Theta_{z,j}$ can be locally described with algebraic expressions of the local coefficients of the metric and the form $\eta$, and their derivatives. When $z=0$, we simply write $k_t$ and $\Theta_j$. (See e.g.\ \cite[Section~1.8.1]{Gilkey1995} or \cite[Section~2.5]{BerlineGetzlerVergne2004}.) We have $\Theta_{z,j}(p,p)=\Theta_{\mu,j}(p,p)$ by~\eqref{Witten's opers} since $\eta$ is locally exact.

For even $n$, let $e(M,g)\in C^\infty(M;\Lambda\otimes o(M))=C^\infty(M;\Omega)$ denote the Euler density of $(M,g)$ (the representative of the Euler class given by the Chern-Weil theory).

\begin{thm}[{\cite[Theorem 13.4]{BismutZhang1992}; see also \cite[Theorem~1.5]{AlvGilkey2021a}}]\label{t: e_z l}
We have:
\begin{enumerate}[{\rm(i)}]
\item\label{i: str Theta_z j(p p) = 0} $\str\Theta_{z,j}(p,p)=0$ for $j<n/2$; and,
\item\label{i: str Theta_z n/2(p p) |dvol(p)| = e(M g)(p)} if $n$ is even, then $\str\Theta_{z,n/2}(p,p)\,|\dvol(p)|=e(M,g)(p)$.
\end{enumerate}
\end{thm} 

\begin{rem}\label{r: e_z l}
In the given references, \Cref{t: e_z l} was stated for compact manifolds, but its proof is a local computation, and therefore compactness is irrelevant. We have an additional proof of \Cref{t: e_z l}~\ref{i: str Theta_z n/2(p p) |dvol(p)| = e(M g)(p)} using Getzler's rescaling, following \cite[Section~4.3]{BerlineGetzlerVergne2004}. In the case $n=2$, this can be also checked directly. We omit the details of our alternative proof for brevity reasons.
\end{rem}

\subsection{Local Lefschetz trace formula for the Witten's complex}\label{ss: local Lefschetz formulae}

Let $\phi:U\to V$ be a smooth map between open subsets of $M$ with $U\subset V$, whose fixed point set is denoted by $\Fix(\phi)$. Recall that a fixed point $p$ of $\phi$ is called \emph{simple} if the eigenvalues of $\phi_*:T_pM\to T_pM$ are different from $1$. This means that the graph of $\phi$ is transverse to $\Delta$ in $M^2$ at $(p,p)$; in particular, $p$ is isolated in $\Fix(\phi)$. \index{$\Fix(\phi)$} In this case, let \index{$\epsilon_p(\phi)$}
\begin{equation}\label{epsilon_p(phi)}
\epsilon_p=\epsilon_p(\phi)=\sign\det(\id-\phi_*:T_pM\to T_pM)\in\{\pm1\}\;.
\end{equation}
Assume $V$ is simply connected, and therefore $\eta=dF$ on $V$ for some $F\in C^\infty(V)$. Consider the perturbed linear map $\phi^*_z=e^{z(\phi^*F-F)}\,\phi^*:\phi^*\Lambda V\to\Lambda U$ ($z\in\C$) (\Cref{ss: perturbation of pull-back homs}). Take any relatively compact open neighborhood $W$ of $p$ in $U$ such that $\overline W\cap\Fix(\phi)=\{p\}$. Without loss of generality, we can assume that $U$ is an open subset of a manifold of bounded geometry (or even of a closed manifold), where $\eta$ and $\phi$ can be extended to a closed real 1-form and a smooth map.

\begin{prop}\label{p: local Lefschetz formula}
For all $z\in\C$,
\[
\lim_{t\downarrow0}\int_{q\in W}\str(\phi^*_zk_{z,t}(\phi(q),q)=\epsilon_p(\phi)\;.
\]
\end{prop}

\begin{proof}
This follows like in the analytic proof of the Lefschetz trace formula \cite{AtiyahBott1967} (see also \cite[Chapter~10]{Roe1998} or \cite[Section~3.9]{Gilkey1995}), using~\eqref{asymptotic expansion} and the expression
\[
e^{z(F\phi(x)-F(x))}=1+O(|x|)\;,
\]
in terms of normal coordinates $x=(x^1,\dots,x^n)$ centered at $p$.
\end{proof}

\subsection{A tempered distribution associated to some closed 1-forms}\label{ss: Z} 

Assume $M$ is closed, and let $\SS=\SS(\R)$ \index{$\SS$} (\Cref{ss: Sobolev sps}). We would like to define a limit \index{$Z(M,g,\eta)$}
\begin{equation}\label{Z}
Z=Z(M,g,\eta)=\lim_{\mu\to+\infty}Z_\mu
\end{equation}
in $\SS'$, where $Z_\mu=Z_\mu(M,g,\eta)\in\SS'$ ($\mu\gg0$) should be given by
\begin{equation}\label{Z_mu}
\langle Z_\mu,f\rangle=-\frac{1}{2\pi}\int_0^\infty\int_{-\infty}^{+\infty}\Str\left({\eta\wedge}\,\delta_ze^{-u\Delta_z}\right)
\,\hat f(\nu)\,d\lambda\,du\;,
\end{equation}
for all $f\in\SS$, where $\Str$ denotes the supertrace. If $Z(M,g,-\eta)$ is defined, then $Z_\mu(M,g,\eta)\in\SS'$ is defined for $\mu\ll0$, and, in $\SS'$,
\[
-Z(M,g,-\eta)=\lim_{\mu\to-\infty}Z_\mu(M,g,\eta)\;.
\]

\begin{thm}[{\cite[Theorems~1.1--1.4]{AlvKordyLeichtnam-ziomf}}]\label{t: Z}
Let $M$ be a closed manifold of dimension $n$. For every real class $\xi\in H^1(M)$ and $\tau\gg0$, there is some $\eta\in\xi$ and some Riemannian metric $g$ on $M$ such that~\eqref{Z} and~\eqref{Z_mu} define the tempered distribution $Z=\tau\delta_0$, using the Dirac distribution  $\delta_0$ on $\R$. If $n$ is even, this property holds for all $\tau\in\R$, and we can choose $\eta\in\xi$ so that $Z(M,g,\pm\eta)$ is defined and $\pm Z(M,g,\pm\eta)=\tau\delta_0$.
\end{thm}

\begin{rem}\label{r: Z}
If $n$ is even, we can choose $\eta$ and $g$ in \Cref{t: Z} so that $Z(M,g,\pm\eta)=0$. 
\end{rem}

\chapter{Foliation tools}\label{ch: foln tools}

\section{Foliations}\label{s: folns}

Standard references on foliations are \cite{HectorHirsch1981-A,HectorHirsch1983-B,CamachoLinsNeto1985,Godbillon1991,CandelConlon2000-I,CandelConlon2003-II}, and for analysis on foliations see \cite{Connes1982,MooreSchochet1988}. 

\subsection{Basic concepts}\label{ss: folns}

Recall that a (\emph{smooth}) \emph{foliation} \index{foliation} $\FF$ on a manifold $M$, with \emph{codimension} $n'$ and \emph{dimension} $n''$ ($\codim\FF=n'$, $\dim\FF=n''$), can be described by a \emph{foliated atlas} $\{U_k,x_k\}$ of $M$. The \emph{foliated charts} or \emph{foliated coordiates} $(U_k,x_k)$ are of the form
\begin{equation}\label{x = (x' x'')}
x_k=(x'_k,x''_k):U_k\to x_k(U_k)=\Sigma_k\times B''_k\;,
\end{equation}
where $B''_k$ is an open ball of $\R^{n''}$ and $\Sigma_k$ is open in $\R^{n'}$, and the corresponding changes of coordinates are locally of the form
\begin{equation}\label{changes of foliated coordinates}
x_lx_k^{-1}(u,v)=(h_{lk}(u),g_{lk}(u,v))\;.
\end{equation}
We will use the notation
\[
x_k=(x^1_k,\dots,x^n_k)=(x'^1_k,\dots,x'^{n'}_k,x''^{n'+1}_k,\dots,x''^n_k)\;.
\]
It is also said that $(M,\FF)$ is a \emph{foliated manifold}. The open sets $U_k$ and the projections $x'_k:U_k\to\Sigma_k$ are said to be \emph{distinguished}, the fibers of $x'_k$ are called \emph{plaques}, and the fibers of $x''_k$ are called \emph{local transversals} defined by $(U_k,x_k)$, which can be identified with $\Sigma_k$ via $x'_k$. Thus the sets $\Sigma_k$ can be considered as local transversals of $\FF$ with disjoint closures. The open subsets of all plaques form a base of a topology on $M$, called the \emph{leaf topology}, becoming a smooth manifold of dimension $n''$ with the obvious charts induced by $\{U_k,x_k\}$, and its connected components are called \emph{leaves}. The leaf through any point $p$ may be denoted by $L_p$. The \emph{$\FF$-saturation} of a subset $S\subset M$, denoted by $\FF(S)$, is the union of leaves that meet $S$.

Foliations on manifolds with boundary are similarly defined, assuming the boundary is either tangent or transverse to the leaves; we will only use the case where the boundary is tangent to the leaves (it is a union of leaves).

If a smooth map $\phi:M'\to M$ is transverse to (the leaves of) $\FF$, then the connected components of the inverse images $\phi^{-1}(L)$ of the leaves $L$ of $\FF$ are the leaves of a smooth foliation $\phi^*\FF$ on $M'$ of codimension $n'$, called \emph{pull-back} of $\FF$ by $\phi$. In particular, for the inclusion map of any open subset, $\iota:U\hookrightarrow M$, the pull-back $\iota^*\FF$ is the \emph{restriction} $\FF|_U$.

Any connected manifold $M$ can be considered as a foliation with one leaf, also denoted by $M$. On the other hand, we can consider the foliation by points on $M$, denoted by $M^\delta$ ($\delta$ refers to the discreteness of the leaf topology). Given foliations $\FF_a$ on manifolds $M_a$ ($a=1,2$), the products of leaves of $\FF_1$ and $\FF_2$ are the leaves of the \emph{product foliation} $\FF_1\times\FF_2$, whose charts can be defined using products of charts of $\FF_1$ and $\FF_2$.

\subsection{Holonomy}\label{ss: holonomy}

After considering a refinement if necessary, we can assume the foliated atlas $\{U_k,x_k\}$ is \emph{regular} in the following sense: it is locally finite; for every $k$, there is a foliated chart $(\widetilde U_k,\tilde x_k)$ such that $\overline{U_k}\subset\widetilde U_k$ and $\tilde x_k$ extends $x_k$; and, if $U_{kl}:=U_k\cap U_l\ne\emptyset$, then there is another foliated chart $(U,x)$ such that $\overline{U_k}\cup\overline{U_l}\subset U$. In this case,~\eqref{changes of foliated coordinates} holds on the whole of $U_{kl}$, obtaining the \emph{elementary holonomy transformations} $h_{kl}:x'_l(U_{kl})\to x'_k(U_{kl})$, determined by the condition $h_{kl}x'_l=x'_k$ on $U_{kl}$. The collection $\{U_k,x'_k,h_{kl}\}$ is called a \emph{defining cocycle}. The maps $h_{kl}$ generate the \emph{holonomy pseudogroup} \index{holonomy pseudogroup} $\HH$ \index{$\HH$} on $\Sigma:=\bigsqcup_k\Sigma_k$, \index{$\Sigma$} which is unique up to certain \emph{equivalence} of pseudogroups \cite{Haefliger1980}. This $\Sigma$ can be considered as a \emph{complete transversal} of $\FF$, in the sense that it meets all leaves. The notation $(\Sigma,\HH)$ may be also used.  The $\HH$-orbit of every $\bar p\in\Sigma$ is denoted by $\HH(\bar p)$. The maps $x'_k$ induce a homeomorphism between the leaf space, $M/\FF$, and the orbit space, $\Sigma/\HH$.

The paths in the leaves are called \emph{leafwise paths} when considered in $M$. Let $c:I:=[0,1]\to M$ be a leafwise path with $p:=c(0)\in U_k$ and $q:=c(1)\in U_l$. There is a partition of $I=[0,1]$, $0=t_0<t_1<\dots<t_m=1$, and a sequence of indices, $k=k_1,k_2,\dots,k_m=l$, such that $c([t_{i-1},t_i])\subset U_{k_i}$ for $i=1,\dots,m$. The composition $h_c=h_{k_mk_{m-1}}\cdots h_{k_2k_1}$, wherever defined, is a diffeomorphism with $x'_k(p)\in\dom h_c\subset\Sigma_k$ and $x'_l(q)=h_cx'_k(p)\in\im h_c\subset\Sigma_l$. The tangent map $h_{c*}:T_{x'_k(p)}\Sigma_k\to T_{x'_l(q)}\Sigma_l$ is called \emph{infinitesimal holonomy} \index{infinitesimal holonomy} of $c$. The germ $\bfh_c$ of $h_c$ at $x'_k(p)$, called \emph{germinal holonomy} \index{germinal holonomy} of $c$, depends only on $\FF$ and the end-point homotopy class of $c$ in $L=L_p$. In particular, taking $q=p$ and $l=k$, this defines the \emph{holonomy homomorphism} onto the \emph{holonomy group}, $\bfh=\bfh_L:\pi_1(L,p)\to\Hol(L,p)$. The isomorphism class of $\Hol(L,p)$ is independent of $p$; thus the notation $\Hol L$ may be used, like $\pi_1L$. If $\Hol L$ is trivial, then $L$ is said to be \emph{without holonomy}. Residually many leaves have no holonomy \cite{Hector1977a,EpsteinMillettTischler1977}. If all leaves have no holonomy, then $\FF$ is said to be \emph{without holonomy}. The kernel of $\bfh:\pi_1 L\to\Hol L$ defines the \emph{holonomy cover} \index{holonomy cover} $\widetilde L=\widetilde L^{\text{\rm hol}}$ of $L$. If $D$ is a compact domain of a leaf $L$ with smooth boundary, then $\FF$ can be completely described in some neighborhood of $D$ in $M$ by the composition
\[
\pi_1D \to \pi_1L \xrightarrow{\bfh} \Hol L\;,
\]
where the first homomorphism is induced by $D\hookrightarrow L$ \cite[Section~2.7]{Haefliger1962} (see also \cite[Theorem~2.1.7]{HectorHirsch1981-A}, \cite[Theorem~IV.2]{CamachoLinsNeto1985}, \cite[Theorem~II.2.29]{Godbillon1991}, \cite[Theorem~2.3.9]{CandelConlon2000-I}). This description, called Reeb's local stability, involves the so-called \emph{suspension foliation}, which allows the lifting of smooth paths from $L$ to nearby leaves, continuously in the $C^\infty$ topology.

\subsection{Infinitesimal transformations and transverse vector fields}\label{ss: infinitesmal transfs}

The vectors tangent to the leaves form the \emph{tangent bundle} $T\FF\subset TM$, obtaining also the \emph{normal bundle} $N\FF=TM/T\FF$, the  \emph{cotangent bundle} $T^*\FF=(T\FF)^*$ and the \emph{conormal bundle} $N^*\FF=(N\FF)^*$, the flat line bundles of \emph{tangent/normal orientations}, $o(\FF)=o(T\FF)$ and $o(N\FF)$, the \emph{tangent/normal density bundles}, $\Omega^a\FF=\Omega^aT\FF$ ($a\in\R$) and $\Omega^aN\FF$ (removing ``$a$'' from the notation when it is $1$), and the \emph{tangent/normal exterior bundles}, $\Lambda\FF=\bigwedge T^*\FF\otimes\C$ \index{$\Lambda\FF$} and $\Lambda N\FF=\bigwedge N^*\FF\otimes\C$. \index{$\Lambda N\FF$} Again, we typically consider these density and exterior bundles with complex coefficients, without changing the notation; the few cases of real coefficients will be indicated. The terms \emph{tangent/normal} vector fields, densities and differential forms are used for their smooth sections. Sometimes, ``\emph{leafwise}'' is used instead of ``tangent''. Any $X\in TM$ (resp., $X\in\fX(M)$) canonically defines an element of $N\FF$ (resp., $C^\infty(M;N\FF)$) denoted by $\overline{X}$. For any smooth local transversal $\Sigma$ of $\FF$ through a point $p\in M$, there is a canonical isomorphism $T_p\Sigma\cong N_p\FF$.

A smooth vector bundle $E$ over $M$, endowed with a flat $T\FF$-partial connection, is said to be {\em$\FF$-flat}. For instance, $N\FF$ is $\FF$-flat with the Bott $T\FF$-partial connection $\nabla^\FF$, given by $\nabla^\FF_V\overline{X}=\overline{[V,X]}$ for $V\in\fX(\FF):=C^\infty(M;T\FF)$ \index{$\fX(\FF)$} and $X\in\fX(M)$. For every leafwise path $c$ from $p$ to $q$, its infinitesimal holonomy can be considered as a homomorphism $h_{c*}:N_p\FF\to N_q\FF$, which is the $\nabla^\FF$-parallel transport along $c$.

$\fX(\FF)$ is a Lie subalgebra and $C^\infty(M)$-submodule of $\fX(M)$, whose normalizer is denoted by $\fX(M,\FF)$, \index{$\fX(M,\FF)$} obtaining the quotient Lie algebra $\olfX(M,\FF)=\fX(M,\FF)/\fX(\FF)$. \index{$\olfX(M,\FF)$} The elements of $\fX(M,\FF)$ (resp., $\olfX(M,\FF)$) are called \emph{infinitesimal transformations} (resp., \emph{transverse vector fields}). The projection of every $X\in\fX(M,\FF)$ to $\olfX(M,\FF)$ is also denoted by $\overline{X}$; in fact, 
\[
\olfX(M,\FF)\equiv\{\,\overline X\in C^\infty(M;N\FF)\mid\nabla^\FF\overline X=0\,\}\subset C^\infty(M;N\FF)\;. 
\]
Any $X\in\fX(M)$ is in $\fX(M,\FF)$ if and only if every restriction $X|_{U_k}$ can be projected by $x'_k$, defining an $\HH$-invariant vector field on $\Sigma$, also denoted by $\overline X$. This induces a canonical isomorphism of $\olfX(M,\FF)$ to the Lie algebra $\fX(\Sigma,\HH)$ \index{$\fX(\Sigma,\HH)$} of $\HH$-invariant tangent vector fields on $\Sigma$.

When $M$ is not closed, we can consider the subsets of complete vector fields, $\fXcom(\FF)\subset\fX(\FF)$ and $\fXcom(M,\FF)\subset\fX(M,\FF)$. Let $\olfXcom(M,\FF)\subset\olfX(M,\FF)$ be the projection of $\fXcom(M,\FF)$. 


\subsection{Holonomy groupoid}\label{ss: holonomy groupoid}

On the space of leafwise paths in $M$, with the compact-open topology, two leafwise paths are declared to be equivalent if they have the same end points and the same germinal holonomy. This is an equivalence relation, and the corresponding quotient space, $\fG=\Hol(M,\FF)$, becomes a smooth manifold of dimension $n+n'$ in the following way. An open neighborhood $\fU$ of a class $[c]$ in $\fG$, with $c(0)\in U_k$ and $c(1)\in U_l$, is defined by the leafwise paths $d$ such that $d(0)\in U_k$, $d(1)\in U_l$, $x'_kd(0)\in\dom h_c$, and $h_d$ and $h_c$ have the same germ at $x'_kd(0)$. Local coordinates on $\fU$ are given by $[d]\mapsto(d(0),x''_ld(1))$. Moreover, $\fG$ \index{$\fG$} is a Lie groupoid, called the \emph{holonomy groupoid}, \index{holonomy groupoid} where the space of units $\fG^{(0)}\equiv M$ is defined by the constant paths, the source and range projections $\bfs,\bfr:\fG\to M$ are given by the first and last points of the paths, the operation is induced by the opposite of the usual path product, and the inversion is induced by the usual path inversion. Note that $\fG$ is Hausdorff if and only if $\HH$ is \emph{quasi-analytic} in the following sense: for any $h\in\HH$ and open $O\subset\Sigma$ with $\overline O\subset\dom h$, if $h|_O=\id_O$, then $h$ is the identity on some neighborhood of $\overline O$. Observe also that $\bfs,\bfr:\fG\to M$ are smooth submersions, and $(\bfr,\bfs):\fG\to M^2$ is a smooth immersion. Let $\RR_\FF=\{\,(p,q)\in M^2\mid L_p=L_q\,\}\subset M^2$, which is not a regular submanifold in general, and let $\Delta\subset M^2$ be the diagonal. We have $(\bfr,\bfs)(\fG)=\RR_\FF$ and $(\bfr,\bfs)(\fG^{(0)})=\Delta$. For any leaf $L$ and $p\in L$, we have $\Hol(L,p)=\bfs^{-1}(p)\cap\bfr^{-1}(p)$, the map $\bfr:\bfs^{-1}(p)\to L$ is the covering projection $\widetilde L^{\text{\rm hol}}\to L$, and $\bfs:\bfr^{-1}(p)\to L$ corresponds to $\bfr:\bfs^{-1}(p)\to L$ by the inversion of $\fG$. Thus $(\bfr,\bfs):\fG\to M^2$ is injective if and only if all leaves have trivial holonomy groups, but, even in this case, this map may not be a topological embedding. The fibers of $\bfs$ and $\bfr$ define smooth foliations of codimension $n$ on $\fG$. We also have the smooth foliation $\bfs^*\FF=\bfr^*\FF$ of codimension $n'$ with leaves $\bfs^{-1}(L)=\bfr^{-1}(L)=(\bfr,\bfs)^{-1}(L^2)$ for leaves $L$ of $\FF$, and every restriction $(\bfr,\bfs):(\bfr,\bfs)^{-1}(L^2)\to L^2$ is a smooth covering projection.

Let $\FF_k=\FF|_{U_k}$, $\fG_k=\Hol(U_k,\FF_k)$ and $\RR_k=\RR_{\FF_k}$. The set $\bigcup_k\fG_k$ (resp., $\bigcup_k\RR_k$) is an open neighborhood of $\fG^{(0)}$ in $\fG$ (resp., of $\Delta$ in $\RR_\FF$). Furthermore, by the regularity of $\{U_k,x_k\}$, the map $(\bfr,\bfs):\bigcup_k\fG_k\to M^2$ is a smooth embedding with image $\bigcup_k\RR_k$; we will write $\bigcup_k\fG_k\equiv\bigcup_k\RR_k$.

\subsection{The convolution algebra on $\fG$ and its global action}\label{ss: global action}

Consider the notation of \Cref{ss: holonomy groupoid}. For the sake of simplicity, assume $\fG$ is Hausdorff \cite{Connes1979}. The extension of the following concepts to the case where $\fG$ is not Hausdorff can be made like in \cite{Connes1982}.

Given a vector bundle $E$ over $M$, let $S=\bfr^*E\otimes \bfs^*(E^*\otimes\Omega\FF)$, which is a vector bundle over $\fG$. Let $C^\infty_{\text{\rm cs}}(\fG;S)\subset C^\infty(\fG;S)$ denote the subspace of sections $k\in C^\infty(\fG;S)$ such that $\supp k\cap\bfs^{-1}(K)$ is compact for all compact $K\subset M$; in particular, $C^\infty_{\text{\rm cs}}(\fG;S)=\Cinftyc(\fG;S)$ if $M$ is compact. Similarly, define $C^\infty_{\text{\rm cr}}(\fG;S)$ by using $\bfr$ instead of $\bfs$. Both $C^\infty_{\text{\rm cs}}(\fG;S)$ and $C^\infty_{\text{\rm cr}}(\fG;S)$ are associative algebras with the \emph{convolution} product defined by
\[
(k_1*k_2)(\gamma)=\int_{\bfs(\epsilon)=\bfs(\gamma)}k_1(\gamma\epsilon^{-1})\,k_2(\epsilon)
=\int_{\bfr(\delta)=\bfr(\gamma)}k_1(\delta)\,k_2(\delta^{-1}\gamma)\;,
\]
and $C^\infty_{\text{\rm csr}}(\fG;S):=C^\infty_{\text{\rm cs}}(\fG;S)\cap C^\infty_{\text{\rm cr}}(\fG;S)$ and $\Cinftyc(\fG;S)$ are subalgebras. 

The \emph{global action} of $C^\infty_{\text{\rm cr}}(\fG;S)$ on $C^\infty(M;E)$ is the left action defined by
\[
(k\cdot u)(p)=\int_{\bfr(\gamma)=p}k(\gamma)\,u(\bfs(\gamma))\;.
\]
In this way, $C^\infty_{\text{\rm cr}}(\fG;S)$ can be considered as an algebra of operators on $C^\infty(M;E)$. Moreover $C^\infty_{\text{\rm csr}}(\fG;S)$ preserves $\Cinftyc(M;E)$, obtaining an algebra of operators on $\Cinftyc(M;E)$. It can be said that these operators are defined by a leafwise version of a smooth Schwartz kernel (cf.\ \Cref{ss: ops}).

Let $S'=\bfr^*(E^*\otimes\Omega\FF)\otimes \bfs^*E$. The mapping $k\mapsto k^\trans$, $ k^\trans(\gamma)=k(\gamma^{-1})$, defines anti-homomorphisms $C^\infty_{\text{\rm cs/cr}}(\fG;S)\to C^\infty_{\text{\rm cr/cs}}(\fG;S')$ and $C^\infty_{\text{\rm csr}}(\fG;S)\to C^\infty_{\text{\rm csr}}(\fG;S')$, obtaining a leafwise version of the transposition of operators (cf.\ \Cref{ss: ops}). Similarly, using $E=\Omega^{1/2}\FF$, or if $E$ has a Hermitian structure and we fix a non-vanishing leafwise density, we get a leafwise version of taking adjoint operators. Moreover, in this case, $\Cinftyc(\fG;S)$ is $*$-algebra.

\subsection{Leafwise metric}\label{ss: leafwise metric}

A Euclidean structure $g_\FF$ on $T\FF$ is called a \emph{leafwise} (\emph{Riemannian}) \emph{metric} \index{leafwise metric} of $\FF$. The corresponding \emph{leafwise distance} \index{leafwise distance} is the map $d_\FF:M^2\to[0,\infty]$ given by the distance function of the leaves on $\RR_\FF$, taking $d_\FF(M^2\setminus\RR_\FF)=\infty$. For $p\in M$, $S\subset M$ and $r>0$, the \emph{open} and \emph{closed leafwise balls}, $B_\FF(p,r)$ and $\overline B_\FF(p,r)$, and the \emph{open} and \emph{closed leafwise penumbras}, \index{leafwise penumbra} $\Pen_\FF(S,r)$ \index{$\Pen_\FF(S,r)$} and $\overline{\Pen}_\FF(S,r)$, \index{$\overline{\Pen}_\FF(S,r)$} are defined with $d_\FF$ like in the case of Riemannian metrics (\Cref{s: bd geom}). The Levi-Civita connection on the leaves defines a $T\FF$-partial connection on $T\FF$, also denoted by $\nabla^\FF$.

Equip the foliation $\bfr^*\FF$ on $\fG$ with the leafwise Riemannian metric so that the foliated immersion $(\bfr,\bfs):(\fG,\bfr^*\FF)\to(M^2,\FF^2)$ is isometric on the leaves. Let $d_{\bfr}:\fG\to[0,\infty]$ denote the leafwise distance for the foliation on $\fG$ defined by the fibers of $\bfr$, and consider the corresponding open and closed leafwise penumbras, $\Pen_{\bfr}(\fG^{(0)},r)$ and $\overline{\Pen}_{\bfr}(\fG^{(0)},r)$. Note that we get the same penumbras by using $\bfs$ instead of $\bfr$; indeed, they are given by the conditions $d_\FF^{\text{\rm hol}}<r$ and $d_\FF^{\text{\rm hol}}\le r$, resp., where $d_\FF^{\text{\rm hol}}:\fG\to[0,\infty)$ is defined by
\[
d_\FF^{\text{\rm hol}}(\gamma)=\inf_c\length(c)\;,
\]
with $c$ running in the piecewise smooth representatives of $\gamma$.

For example, if $M$ is endowed with a Riemannian metric, its restriction to the leaves defines a leafwise Riemannian metric. In this case, $d_\FF\ge d_M$ (the distance function of $M$), and the leafwise metric of $\bfr^*\FF$ is given by the Riemannian metric on $\fG$ so that the immersion $(\bfr,\bfs):\fG\to M^2$ is isometric. 

By the smooth lifting of leafwise paths to nearby leaves, it easily follows that $d_\FF^{\text{\rm hol}}:\fG\to[0,\infty)$ and $d_\FF:\RR_\FF\to[0,\infty)$ are upper semicontinuous. Moreover $d_\FF^{\text{\rm hol}}\equiv d_\FF$ on $\bigcup_k\fG_k\equiv\bigcup_k\RR_k$. Using the convexity radius (see e.g.\ \cite[Section~6.3.2]{Petersen1998}), it follows that, after refining $\{U_k,x_k\}$ if necessary, we can assume $d_\FF$ is continuous on $\bigcup_k\RR_k$.

\begin{lem}\label{l: d_FF reaches maximum/minimum}
The following properties hold for any compact $K\subset M^2$:
\begin{enumerate}[{\rm(i)}]
\item\label{i: d_FF reaches maximum} If $K\subset\RR_\FF$, then $d_\FF|_K$ reaches a finite maximum at some point.
\item\label{i: d_FF reaches minimum} If $K\cap\Delta=\emptyset$, then $\inf d_\FF(K)>0$. If moreover $\inf d_\FF(K)$ is small enough, then it is the minimum of $d_\FF|_K$. 
\end{enumerate}
\end{lem}

\begin{proof}
Using that $K$ is compact, $\Delta=\{d_\FF=0\}$, $\RR_\FF=\{d_\FF<\infty\}$, and $\bigcup_k\RR_k$ is a neighborhood of $\Delta$ in $\RR_\FF$ containing $K\cap\{d_\FF\le r\}$ for some $r<0$, we get~\ref{i: d_FF reaches maximum} by the upper semicontinuity of $d_\FF$, and~\ref{i: d_FF reaches minimum} by the continuity of $d_\FF$ on $\bigcup_k\RR_k$.
\end{proof}

\begin{rem}\label{r: d_FF^hol reaches maximum/minimum}
The obvious version of \Cref{l: d_FF reaches maximum/minimum} for $d_\FF^{\text{\rm hol}}$ and compact subsets of $\fG$ can be proved with analogous arguments.
\end{rem}

From now on, suppose the leaves with $g_\FF$ are complete Riemannian manifolds. Then their exponential maps define a smooth map $\exp_\FF:T\FF\to M$.

With the notation of \Cref{ss: global action}, let $C^\infty_{\text{\rm p}}(\fG;S)\subset C^\infty(\fG;S)$ denote the subspace of sections supported in leafwise penumbras of $\fG^{(0)}$. This is a subalgebra of $C^\infty_{\text{\rm csr}}(\fG;S)$, and the leafwise transposition restricts to an anti-homomorphism $C^\infty_{\text{\rm p}}(\fG;S)\to C^\infty_{\text{\rm p}}(\fG;S')$ \cite[Section~4.6]{AlvKordyLeichtnam2020}.

\subsection{Foliated maps and foliated flows}\label{ss: fol maps}

A \emph{foliated map} \index{foliated map} $\phi:(M_1,\FF_1)\to(M_2,\FF_2)$ is a map $\phi:M_1\to M_2$ that maps leaves of $\FF_1$ to leaves of $\FF_2$. In this case, assuming that $\phi$ is smooth, its tangent map defines homomorphisms $\phi_*:T\FF_1\to T\FF_2$ and $\phi_*:N\FF_1\to N\FF_2$, where the second one is compatible with the corresponding flat partial connections. We also get an induced Lie groupoid homomorphism $\Hol(\phi):\Hol(M_1,\FF_1)\to\Hol(M_2,\FF_2)$, defined by $\Hol(\phi)([c])=[\phi c]$. The set of smooth foliated maps $(M_1,\FF_1)\to(M_2,\FF_2)$ is denoted by $C^\infty(M_1,\FF_1;M_2,\FF_2)$.  A smooth family $\phi=\{\,\phi^t\mid t\in T\,\}$ of foliated maps $(M_1,\FF_1)\to(M_2,\FF_2)$ can be considered as the smooth foliated map $\phi:(M_1\times T,\FF_1\times T^\delta)\to(M_2,\FF_2)$.

For example, if a smooth map $\psi:M'\to M$ is transverse to a foliation $\FF$ on $M$, then it is a foliated map $(M',\psi^*\FF)\to(M,\FF)$. Moreover $\psi_*:N\psi^*\FF\to N\FF$ restricts to isomorphisms between the fibers; i.e., it induces an isomorphism $\psi_*:N\psi^*\FF\xrightarrow{\cong}\psi^*N\FF$ of $\psi^*\FF$-flat vector bundles over $M'$.

Let $\Diffeo(M,\FF)$ be the group of foliated diffeomorphisms (or transformations) of $(M,\FF)$. A smooth flow $\phi=\{\phi^t\}$ on $M$ is called \emph{foliated} \index{foliated flow} if $\phi^t\in\Diffeo(M,\FF)$ for all $t\in\R$. More generally, a local flow $\phi:\Omega\to M$, defined on some open neighborhood $\Omega$ of $M\times\{0\}$ in $M\times\R$, is called \emph{foliated} if it is a foliated map $(\Omega,(\FF\times\R^\delta)|_\Omega)\to(M,\FF)$. Then $\fX(M,\FF)$ consists of the smooth vector fields whose local flow is foliated, and $\fXcom(M,\FF)$ \index{$\fXcom(M,\FF)$} consists of the complete smooth vector fields whose flow is foliated.

Let $X\in\fXcom(M,\FF)$, with foliated flow $\phi=\{\phi^t\}$, and let $\bar\phi$ be the local flow on $\Sigma$ generated by $\overline X\in\fX(\Sigma,\HH)$ (\Cref{ss: holonomy,ss: infinitesmal transfs}). The following properties hold \cite[Section~4.8]{AlvKordyLeichtnam2020}: via $x'_k:U_k\to\Sigma_k$, the local flow defined by $\phi$ on every $U_k$ corresponds to the restriction of $\bar\phi$ to $\Sigma_k$; and $\bar\phi$ is $\HH$-equivariant in an obvious sense.

Take another vector field $Y\in\fXcom(M,\FF)$ with foliated flow $\psi=\{\psi^t\}$.

\begin{lem}\label{l: overline Y = overline X <=> phi^t(L) = psi^t(L)}
We have $\overline Y=\overline X$ if and only if $\phi^t(L)=\psi^t(L)$ for all $t\in\R$ and every leaf $L$.
\end{lem}

\begin{proof}
The condition $\overline Y=\overline X$ is equivalent to $\bar\phi=\bar\psi$, which means that the local flows defined by $\phi$ and $\psi$ on every $U_k$ correspond to the same local flow on $\Sigma_k$ via $\pi'_k$. In turn, this is equivalent to the existence of some open $\Omega\subset M\times\R$, containing $M\times\{0\}$, such that $\phi(p,t)$ and $\psi(p,t)$ are in the same leaf for all $(p,t)\in\Omega$. But this is equivalent to $\phi^t(L)=\psi^t(L)$ for all leaf $L$ and $t\in\R$ because $\phi$ and $\psi$ are foliated flows.
\end{proof}

A smooth homotopy $H:M_1\times I\to M_2$ ($I=[0,1]$) between foliated maps $\phi,\psi:(M_1,\FF_1)\to(M_2,\FF_2)$ is said to be \emph{leafwise} (or \emph{integrable}) \index{leafwise homotopy} if it is a foliated map $(M_1\times I,\FF_1\times I)\to(M_2,\FF_2)$. When there is such a leafwise homotopy, it is said that $\phi$ and $\psi$ are \emph{leafwisely homotopic}.

A smooth \emph{leafwise homotopy} between foliated flows on $(M,\FF)$, $\phi=\{\phi^t\}$ and $\psi=\{\psi^t\}$, is a smooth family $H=\{H^t\}$, where every $H^t:M\times I\to M$ is a leafwise homotopy between $\phi^t$ and $\psi^t$; in other words, it can be considered as a leafwise homotopy $H:M\times\R\times I\to M$ between the corresponding foliated maps $\phi,\psi:(M\times\R,\FF\times\R^\delta)\to(M,\FF)$. If moreover every $H(\cdot,\cdot,s):M\times\R\to M$ is a flow, then $H$ is called a smooth \emph{flow leafwise homotopy}. \index{flow leafwise homotopy}

\begin{prop}\label{p: there exists a leafwise homotopy}
Let $X,Y\in\fXcom(M,\FF)$, with foliated flows $\phi=\{\phi^t\}$ and $\psi=\{\psi^t\}$, such that $V:=Y-X\in\fX_\co(\FF)$. Then there is a flow leafwise homotopy $H:M\times\R\times I\to M$ between $\phi$ and $\psi$ such that $H(p,t,s)=\phi^t(p)$ for all $p\in M$ with $\phi^t(p)=\psi^t(p)$.
\end{prop}

\begin{proof}
Since $X\in\fXcom(M,\FF)$ and $V\in\fX_\co(\FF)$, we have $Z_s:=X+sV\in\fXcom(M,\FF)$ ($s\in I$). Let $\xi_s:M\times\R\to M$ denote the flow of every $Z_s$. Since $\overline Z_s=\overline X$ for all $s$, it follows from \Cref{l: overline Y = overline X <=> phi^t(L) = psi^t(L)} that the statement holds with $H:M\times\R\times I\to M$ defined by $H(\cdot,\cdot,s)=\xi_s$.
\end{proof}

\subsection{Differential operators on foliated manifolds}\label{ss: diff ops on fol mfds}

Like in \Cref{ss: diff ops}, using $\fX(\FF)$ instead of $\fX(M)$, we get the filtered subalgebra and $C^\infty(M)$-submodule of \emph{leafwise differential operators}, \index{leafwise differential operator} $\Diff(\FF)\subset\Diff(M)$,  \index{$\Diff(\FF)$} and a \emph{leafwise principal symbol} exact sequence for every order $m$, \index{leafwise principal symbol} \index{$\Fsigma_m$}
\[
0\to\Diff^{m-1}(\FF)\hookrightarrow\Diff^m(\FF) \xrightarrow{\Fsigma_m} P^{(m)}(T^*\FF)\to0\;.
\]
Moreover these concepts can be extended to vector bundles $E$ and $F$ over $M$ like in \Cref{ss: diff ops}, obtaining the filtered $C^\infty(M)$-submodule $\Diff(\FF;E,F)$ (or $\Diff(\FF;E)$ if $E=F$) of $\Diff(M;E,F)$, and the \emph{leafwise principal symbol} $\Fsigma_m:\Diff^m(\FF;E,F)\to P^{(m)}(T^*\FF;F\otimes E^*)$. The diagram
\[
\begin{CD}
\Diff^m(\FF;E,F) @>{\Fsigma_m}>> P^{(m)}(T^*\FF;F\otimes E^*) \\
@VVV @VVV \\
\Diff^m(M;E,F) @>{\sigma_m}>> P^{(m)}(T^*M;F\otimes E^*)
\end{CD}
\]
is commutative, where the left-hand side vertical arrow denotes the inclusion homomorphism, and the right-hand side vertical arrow is induced by the restriction homomorphism $T^*M\to T^*\FF$. The condition of being a leafwise differential operator is preserved by compositions and by taking transposes, and by taking formal adjoints in the case of Hermitian vector bundles; in particular, $\Diff(\FF;E)$ is a filtered subalgebra of $\Diff(M;E)$.  It is said that $A\in\Diff^m(\FF;E,F)$ is \emph{leafwisely elliptic} if the symbol $\sigma_m(A)(p,\xi)$ is an isomorphism for all $p\in M$ and $0\ne\xi\in T^*_p\FF$. In this way, the concepts of \emph{leafwise differential complex} \index{leafwise differential complex} and its \emph{leafwise ellipticity} can be defined like in \Cref{ss: diff complexes}.

A smooth family of leafwise differential operators, $A=\{\,A_t\mid t\in T\,\}\subset\Diff^m(\FF;E,F)$, can be canonically considered as a leafwise differential operator $A\in\Diff^m(\FF\times T^\delta;\pr_1^*E,\pr_1^*F)$, where $\pr_1:M\times T\to M$ is the first-factor projection.

On the other hand, considering the canonical injection $N^*\FF\subset T^*M$, it is said that $A\in\Diff^m(M;E,F)$ is \emph{transversely elliptic} \index{transversely elliptic} if the symbol $\sigma_m(A)(p,\xi)$ is an isomorphism for all $p\in M$ and $0\ne\xi\in N^*_p\FF$. The concept of \emph{transverse ellipticity} has an obvious extension to differential complexes like in \Cref{ss: diff complexes}.

We can use $\Diff(\FF;E)$ to define variants of the section spaces recalled in \Cref{ss: smooth/distributional sections}. For instance, for $m\in\N_0$, we have the LCHS
\[
C^{0,m}_\FF(M;E)=\{\,u\in C(M;E)\mid\Diff^m(\FF;E)\cdot u\subset C(M;E)\,\}\;,
\]
with the topology defined like in~\eqref{Z = u in bigcup_A in AA dom A | AA cdot u subset Y}. Let also $C^{0,\infty}_\FF(M;E)=\bigcap_mC^{0,m}(M;E)$. If $\FF$ is described by a submersion $\varpi:M\to M'$, then the subscript $\varpi$ may be used instead of $\FF$, which agrees with the notation already used in Remark~\ref{r: smallnuint is cont}.

\subsection{Transverse structures}\label{ss: transverse structures}

Recall that $(\Sigma,\HH)$ denotes the holonomy pseudogroup of $\FF$. An (\emph{invariant}) \emph{transverse structure} \index{transverse structure} of $\FF$ is an $\HH$-invariant structure on $\Sigma$. It can be also considered as a $\nabla^\FF$-parallel structure on $N\FF$. For our purposes, it is enough to consider structures on $\Sigma$ (resp., on $N\FF$) defined by smooth sections of bundles associated with $T\Sigma$ (resp., $N\FF$) satisfying some conditions. For instance, we will use the concepts of a \emph{transverse orientation}, a \emph{transverse Riemannian metric} and a \emph{transverse parallelism}. The existence of these transverse structures defines the classes of \emph{transversely orientable}, (\emph{transversely}) \emph{Riemannian}, and \emph{transversely parallelizable} (\emph{TP}) \emph{foliations}. \index{transversely oriented foliation} \index{Riemannian foliation} \index{TP foliation}

A transverse orientation of $\FF$ can be simply described as an orientation of $N\FF$, which is necessarily  $\nabla^\FF$-parallel. It can be determined by a non-vanishing real form $\omega\in C^\infty(M;\Lambda^{n'}N\FF)$; i.e., some real $\omega\in C^\infty(M;\Lambda^{n'})$ defining $\FF$ in the sense that $T\FF = \{\,Y\in TM\mid\iota_Y\omega=0\,\}$.  By Frobenius theorem, the integrability of $T\FF$ means that $d\omega=\eta\wedge\omega$ for some real $\eta\in C^\infty(M;\Lambda^1)$, which is unique modulo $C^\infty(M;\Lambda^1N\FF)$. All other pairs of differential forms $\omega'$ and $\eta'$ satisfying these conditions are of the form $\omega'=e^f\omega$ and $\eta'=\eta+df$ for any real function $f\in C^\infty(M)$. We have $d\omega=0$ just when $\omega$ defines an \emph{invariant transverse volume form}. Any invariant transverse volume form $\omega$ defines an invariant transverse density $|\omega|\in C^\infty(M;\Omega N\FF)$, which can be considered as an invariant transverse measure.

\begin{rem}\label{r: non transversely oriented}
Even when $\FF$ is not transversely oriented, it is defined by some real $\omega\in C^\infty(M;\Lambda^{n'}N\FF\otimes o(N\FF))\equiv C^\infty(M;\Omega N\FF)$, and we have $d\omega=\omega\wedge\eta$ for some real 1-form $\eta$, as above.
\end{rem}

A transverse parallelism can be described as a global frame of $N\FF$ consisting of transverse vector fields $\overline{X_1},\dots,\overline{X_{n'}}$. If its linear span is a Lie subalgebra $\fg\subset\overline\fX(M,\FF)$, it is called a \emph{transverse Lie structure}, giving rise to the concept of ($\fg$-)\emph{Lie foliation}. If moreover $\overline{X_1},\dots,\overline{X_{n'}}\in\olfXcom(M,\FF)$, then the TP or Lie foliation $\FF$ is said to be \emph{complete}. \index{Lie foliation}

Let $G$ be the simply connected Lie group with Lie algebra $\fg$ as above. Then $\FF$ is a $\fg$-Lie foliation just when $\HH$ is equivalent to some pseudogroup generated by restrictions of some left translations on some open $T\subset G$, which is complete just when we can take $T=G$.

Similarly, a transverse Riemannian metric can be described as a $\nabla^\FF$-parallel Euclidean structure on $N\FF$. It is always induced by a Riemannian metric on $M$ such that every $x'_k:U_k\to\Sigma_k$ is a Riemannian submersion, which is called a \emph{bundle-like metric}.\index{bundle-like metric} Thus $\FF$ is Riemannian if and only if it can be endowed with a bundle-like metric on $M$.

It is said that $\FF$ is \emph{transitive at} a point $p\in M$ when the evaluation map $\ev_p:\fX(M,\FF)\to T_pM$ is surjective, or, equivalently, the evaluation map $\overline{\ev}_p:\olfX(M,\FF)\to N_p\FF$ is surjective. The transitive point set is open and saturated. If $\FF$ is transitive at every point, then it is called \emph{transitive}. \index{transitive foliation} If $\ev_p(\fXcom(M,\FF))$ spans $T_pM$ for all $p\in M$, then $\FF$ is called \emph{transversely complete} (\emph{TC}). \index{TC foliation} Since $\ev_p:\fXcom(\FF)\to T_p\FF$ is surjective \cite[Section~4.5]{Molino1988}, $\FF$ is TC if and only if $\overline{\ev}_p(\olfXcom(M,\FF))$ spans $N_p\FF$ for all $p\in M$. 

All TP foliations are transitive, and all transitive foliations are Riemannian. On the other hand, Molino's theory describes Riemannian foliations in terms of TP foliations \cite{Molino1988}. A Riemannian foliation is called \emph{complete} if, using Molino's theory, the corresponding TP foliation is TC. Furthermore Molino's theory describes TC foliations in terms of complete Lie foliations with dense leaves. In turn, complete Lie foliations have the following description due to Fedida \cite{Fedida1971,Fedida1973} (see also \cite[Theorem~4.1 and Lemma~4.5]{Molino1988}). Assume $M$ is connected and $\FF$ a complete $\fg$-Lie foliation. Let $G$ be the simply connected Lie group with Lie algebra $\fg$. Then there is a regular covering $\pi:\widetilde M\to M$ (the \emph{holonomy covering}), a fiber bundle $D:\widetilde M\to G$ (the \emph{developing map}) \index{developing map} and a monomorphism $h:\Gamma:=\Aut(\pi)\equiv\pi_1L/\pi_1\widetilde L\to G$ (the \emph{holonomy homomorphism}) \index{holonomy homomorphism} such that the leaves of $\widetilde\FF:=\pi^*\FF$ are the fibers of $D$, and $D$ is $h$-equivariant with respect to the left action of $G$ on itself by left translations. As a consequence, $\pi$ restricts to diffeomorphisms between the leaves of $\widetilde\FF$ and $\FF$. The subgroup $\Hol\FF=\im h\subset G$, isomorphic to $\Gamma$, is called the \emph{global holonomy group}. 

The Molino's description also gives a precise equivalence between the holonomy pseudogroup $\HH$ and the pseudogroup on $G$ generated by the action of $\Hol\FF$ by left translations. Thus the leaves are dense if and only if $\Hol\FF$ is dense in $G$, which means $\fg=\olfX(M,\FF)$. 

The $\widetilde\FF$-leaf through every $\tilde p\in\widetilde M$ will be denoted by $\widetilde L_{\tilde p}$. Since $D$ induces an identity $\widetilde M/\widetilde\FF\equiv G$, the $\pi$-lift and $D$-projection of vector fields define identities
\begin{equation}\label{overline fX(M FF) equiv ... equiv fX(G Hol FF)}
\olfX(M,\FF)\equiv\olfX(\widetilde M,\widetilde\FF,\Gamma)\equiv\fX(G,\Hol\FF)\;.
\end{equation}
(Given an action, the group is added to the notation of a space of vector fields to indicate the subspace of invariant elements.) These identities give a precise realization of $\fg\subset\olfX(M,\FF)$ as the Lie algebra of left invariant vector fields on $G$. 

If a smooth map $\psi:M'\to M$ is transverse to $\FF$, since $\psi_*:N\psi^*\FF\to N\FF$ restricts to isomorphisms between the fibers and is compatible with the corresponding flat partial connections (\Cref{ss: fol maps}), it follows that any transverse structure of $\FF$ canonically induces a transverse structure of $\psi^*\FF$ of the same type.

\subsection{Foliations of codimension one}\label{ss: folns of codim 1}

In this section, assume $\FF$ is of codimension one ($n'=1$ and $n''=n-1$). Then the notation $(x,y)=(x,y^1,\dots,y^{n-1})$ is used for the foliated coordinates instead $(x',x'')$.

Suppose also that $\FF$ is transversely oriented. Thus there are real forms $\omega,\eta\in C^\infty(M;\Lambda^1)$ such that $\omega$ defines $\FF$ and its transverse orientation, and $d\omega=\eta\wedge\omega$ (\Cref{ss: transverse structures}). There is some $X\in\fX(M)$ with $\omega(X)=1$; in fact, $\overline X\in C^\infty(M;N\FF)$ and $\omega$ determine each other. Now $\FF$ is Riemannian just when $\omega$ can be chosen so that $d\omega=0$; i.e., $X\in\fX(M,\FF)$. Actually, $\FF$ is an $\R$-Lie foliation in this case because $\R\cdot\overline{X}$ is a Lie subalgebra of $\olfX(M,\FF)$.

\subsection{Complete $\R$-Lie foliations}\label{ss: complete R-Lie folns} 

$\FF$ is a complete $\R$-Lie foliation when there is some $Z\in\fXcom(M,\FF)$ so that $\overline Z$ has no zeros. This means that the orbits of the foliated flow $\phi:M\times\R\to M$ of $Z$ are transverse to $\FF$. Its Fedida's description is given by some $\pi:\widetilde M\to M$, $D:\widetilde M\to\R$ and $h:\Gamma\to\R$ (\Cref{ss: transverse structures}). Let $\widetilde Z\in\fXcom(\widetilde M,\widetilde\FF)$ and $\tilde\phi:\widetilde M\times\R\to\widetilde M$ be the lifts of $Z$ and $\phi$. Then $\widetilde Z$ is $\Gamma$-invariant and $D$-projectable. Without loss of generality, we can assume $D_*\widetilde Z=\partial_x\in\fX(\R)$, where $x$ denotes the standard global coordinate of $\R$. Thus $\tilde\phi$ is $\Gamma$-equivariant and induces via $D$ the flow $\bar\phi=\{\bar\phi^t\}$ on $\R$ defined by $\bar\phi^t(x)=t+x$. Since $\bar\phi^t$ preserves every $\Hol\FF$-orbit in $\R$ if and only if $t\in\Hol\FF$, it follows that $\phi^t$ preserves every leaf of $\FF$ if and only if $t\in\Hol\FF$.

\subsection{Foliations almost without holonomy}\label{ss: folns almost w/o hol}

Assume $M$ is compact. It is said that $\FF$ is \emph{almost without holonomy} \index{almost without holonomy foliation} when all non-compact leaves have no holonomy. The structure of such a foliation was described by Hector \cite{Hector1972c,Hector1978}. In the case where $\FF$ has a finite number of leaves with holonomy and is transversely oriented, the description of $\FF$ is as follows. Let $M^0$ be the finite union of compact leaves with holonomy. Let $M^1=M\setminus M^0$, whose connected components are denoted by $M^1_l$ ($l=1,\dots,k$), and let $\FF^1_l=\FF|_{M^1_l}$. Then, for every $l$, there is a connected compact manifold $M_l$, possibly with boundary, endowed with a smooth transversely oriented foliation $\FF_l$ tangent to the boundary, such that, equipping $\bfM:=\bigsqcup_lM_l$ with the combination $\bfFF$ of the foliations $\FF_l$, there is foliated smooth local embedding $\bfpi:(\bfM,\bfFF)\to(M,\FF)$, preserving the transverse orientations, so that:
\begin{itemize}

\item $\bfpi:\mathring M_l\to M^1_l$ is a diffeomorphism for all $l$ (we may write $\mathring M_l\equiv M^1_l$);

\item $\bfpi:\partial \bfM\to M^0$ is a $2$-fold covering map; and

\item every $\FF_l$ is one of the following models: 
\begin{enumerate}[{\rm(1)}]\addtocounter{enumi}{-1}

\item\label{i: model 0} $\FF_l$ is given by a trivial bundle over $[0,1]$,

\item\label{i: model 1} $\mathring\FF_l:=\FF_l|_{\mathring M_l}$ is given by a fiber bundle over $S^1$, or 

\item\label{i: model 2} all leaves of $\mathring\FF_l$ are dense in $\mathring M_l$.

\end{enumerate}
\end{itemize}
Thus $M$ is obtained by gluing the manifolds $M_l$ along corresponding pairs of boundary components. Equivalently, $\bfM$ can be described by cutting $M$ along $M^0$ like in \Cref{s: conormal seq}. Since $\FF$ is transversely oriented, the restriction of $\bfpi:\partial\bfM\to M^0$ to every connected component of $\partial\bfM$ is a diffeomorphism to its image. Thus $\partial \bfM\equiv M^0\sqcup M^0$. The restriction of $\bfFF$ to the interior $\mathring\bfM$ is denoted by $\mathring\bfFF$. Thus $\bfpi$ restricts to a foliated diffeomorphism $(\mathring\bfM,\mathring\bfFF)\xrightarrow{\approx}(M^1,\FF^1)$.

\begin{rem}\label{r: Hector's description of folns almost w/o hol}
In the above description, we have the following:
\begin{enumerate}[{\rm(i)}]

\item\label{i: models are Lie foliations} If $\FF_l$ is a model~\ref{i: model 2}, then $\mathring\FF_l$ becomes a complete $\R$-Lie foliation after a possible change of the differentiable structure of $\mathring M_l$, keeping the same differentiable structure on the leaves \cite[Theorem~2]{Hector1978}.

\item\label{i: adding leaves to M^0} The description holds as well if $M^0$ is any finite union of compact leaves, including all leaves with holonomy. In particular, if $\FF_l$ is a model~\ref{i: model 1} with $\partial M_l=\emptyset$, then $M_l=M$ can be cut into models~\ref{i: model 0} by adding compact leaves to $M^0$. Conversely, if all foliations $\FF_l$ are models~\ref{i: model 0}, then $\FF$ is a model~\ref{i: model 1} with $\partial M=\emptyset$.

\end{enumerate}
\end{rem}

\section{Differential forms on foliated manifolds}\label{s: dif forms on fold mfds}

\subsection{The leafwise complex}\label{ss: leafwise complex}

Let $d_\FF\in\Diff^1(\FF;\Lambda\FF)$ be given by $(d_\FF\alpha)|_L=d_L(\alpha|_L)$ for every leaf $L$ and $\alpha\in C^\infty(M;\Lambda\FF)$. \index{$d_\FF$} Then $(C^\infty(M;\Lambda\FF),d_\FF)$ is a differential complex, called the \emph{leafwise} or \emph{tangential} (\emph{de~Rham}) \emph{complex}. \index{leafwise complex} The elements of $C^\infty(M;\Lambda\FF)$ are called \emph{leafwise forms}; \index{leafwise form} the leafwise forms in $\ker d_\FF$ (resp., $\im d_\FF$) are called \emph{leafwise-closed forms} (resp., \emph{leafwise-exact forms}). \index{leafwise-closed form} \index{leafwise-exact form} The leafwise complex gives rise to the \emph{leafwise} or \emph{tangential cohomology} \index{leafwise cohomology} $H^\bullet(\FF)$. \index{$H^\bullet(\FF)$} The leafwise complex is not elliptic if $n'>0$, and therefore it makes sense to consider also its reduced cohomology $\bar H^\bullet(\FF)$ \index{$\bar H^\bullet(\FF)$} (\Cref{ss: top complexes}). The more precise notation $H^\bullet C^\infty(\FF)=H^\bullet(\FF)$ and $\bar H^\bullet C^\infty(\FF)=\bar H^\bullet(\FF)$ may be also used. Recall that we typically take complex coefficients without any comment; the case of real coefficients will be indicated. Compactly supported versions may be also considered when $M$ is not compact.

We can also take coefficients in any complex $\FF$-flat vector bundle $E$ over $M$, obtaining the differential complex $C^\infty(M;\Lambda\FF\otimes E)$ with $d_\FF\in\Diff^1(\FF;\Lambda\FF\otimes E)$, and the corresponding cohomology, $H^\bullet(\FF;E)$, and reduced cohomology, $\bar H^\bullet(\FF;E)$. For example, we can consider the vector bundle $E$ defined by the $\GL(n')$-principal bundle of (real) normal frames and any unitary representation of $\GL(n')$, with the $\FF$-flat structure induced by the $\FF$-flat structure of $N\FF$. A particular case is $\Lambda N\FF$, which gives rise to the differential complex $(C^\infty(M;\Lambda\FF\otimes\Lambda N\FF),d_\FF)$. Note that
\begin{equation}\label{Lambda FF subset Lambda FF otimes Lambda N FF}
\Lambda\FF\equiv\Lambda\FF\otimes\Lambda^0N\FF\subset\Lambda\FF\otimes\Lambda N\FF\;,
\end{equation}
and therefore $C^\infty(M;\Lambda\FF)$ becomes a subcomplex of $C^\infty(M;\Lambda\FF\otimes\Lambda N\FF)$ with $d_\FF$.

\subsection{Bigrading of differential forms}\label{ss: bigrading}

Consider any splitting 
\begin{equation}\label{splitting}
TM=T\FF\oplus\bfH\cong T\FF\oplus N\FF\;,
\end{equation}
for some vector subbundle $\bfH\subset TM$. Recall that $\Lambda\bfH=\bigwedge\bfH^*\otimes\C$. The splitting~\eqref{splitting} induces a decomposition
\begin{equation}\label{Lambda M}
\Lambda M\equiv\Lambda\FF\otimes\Lambda\bfH\cong\Lambda\FF\otimes\Lambda N\FF\;,
\end{equation}
giving rise to the bigrading of $\Lambda M$ defined by \index{$\Lambda^{u,v}M$}
\begin{equation}\label{Lambda^u v M}
\Lambda^{u,v}M\equiv\Lambda^v\FF\otimes\Lambda^u\bfH\cong\Lambda^v\FF\otimes\Lambda^uN\FF\;,
\end{equation}
and the corresponding bigrading of $C^\infty(M;\Lambda)$ with bihomogeneous components
\[
C^\infty(M;\Lambda^{u,v})\equiv C^\infty(M;\Lambda^v\FF\otimes\Lambda^uN\FF)\;.
\]
In particular, $\Lambda^{0,v}M\equiv\Lambda^v\FF$ and $\Lambda^{u,0}M\equiv\Lambda^u\bfH$, and then the identity of~\eqref{Lambda^u v M} becomes\footnote{This order in the wedge product, introduced in  \cite{AlvKordyLeichtnam2020} and different from \cite{AlvKordy2001}, produces simpler sign expressions. However, the transverse degree is written first in the bigrading, like in the extension to foliations of the Leray-Serre spectral sequence.} $\Lambda^{0,v}M\otimes\Lambda^{u,0}M\equiv\Lambda^{u,v}M$, $\alpha\otimes\beta\equiv\alpha\wedge\beta$.

This bigrading depends on $\bfH$, but the spaces $\Lambda^{\ge u,\cdot}M$ and $C^\infty(M;\Lambda^{\ge u,\cdot})$ are independent of $\bfH$ (see e.g.\ \cite{Alv1989a}). There are canonical identities
\begin{equation}\label{bigwedge^ge u cdot T^*M / bigwedge^ge u+1 cdot T^*M}
\Lambda^{\ge u,\cdot}M/\Lambda^{\ge u+1,\cdot}M\equiv\Lambda^{u,\cdot}M\equiv\Lambda\FF\otimes\Lambda^uN\FF\;,
\end{equation}
where only $\Lambda^{u,\cdot}M$ depends on $\bfH$.

\subsection{Bihomogeneous components of the derivative}\label{ss: d_i j}

The de~Rham derivative on $C^\infty(M;\Lambda)$ decomposes into bihomogeneous components,
\begin{equation}\label{d = d_0 1 + d_1 0 + d_2 -1}
d=d_{0,1}+d_{1,0}+d_{2,-1}\;,
\end{equation}
where the double subscript denotes the corresponding bidegree. By comparing bidegrees in the anti-derivation formula of $d$, we also get that every $d_{i,1-i}$ ($i\in\{0,1,2\}$) \index{$d_{i,1-i}$} satisfies the same anti-derivation formula. Thus $d_{2,-1}$ is of order 0. The other components, $d_{0,1}$ and $d_{1,0}$, are of order 1. Moreover, $d_{2,-1}=0$ if and only if $\bfH$ is completely integrable. By comparing bi-degrees in $d^2=0$, we get \cite{Alv1989a}
\begin{equation}\label{d_0,1^2=...=0}
  d_{0,1}^2=d_{0,1}d_{1,0}+d_{1,0}d_{0,1}=0\;.
\end{equation}
So $(C^\infty(M;\Lambda),d_{0,1})$ is a differential complex of order one. In fact, via~\eqref{Lambda M},
\begin{equation}\label{d_0 1 equiv d_FF}
d_{0,1}\equiv d_\FF\;.
\end{equation}
Moreover
\begin{equation}\label{d_0 1 = d on C^infty(M Lambda^n' cdot)}
d_{0,1}=d:C^\infty(M;\Lambda^{n',\bullet})\to C^\infty(M;\Lambda^{n',\bullet+1})\;.
\end{equation}

\subsection{Basic complex}\label{ss: basic complex}

It is said that $\alpha\in C^\infty(M;\Lambda)$ is a \emph{basic form} if $\iota_X\alpha=\iota_Xd\alpha=0$ for all $X\in\fX(\FF)$. This means that $\alpha$ is an $\FF$-parallel section of $\Lambda N\FF\equiv\Lambda^{\bullet,0}M$; i.e., $\alpha\in C^\infty(M;\Lambda^{\bullet,0})\cap\ker d_{0,1}$. The basic forms form a subcomplex of the de~Rham complex, called the \emph{basic complex}. \index{{basic complex}} It is isomorphic to the complex of $\HH$-invariant forms on $\Sigma$ via the distinguished projections $x'_k:U_k\to\Sigma_k$ (\Cref{ss: folns}).

\subsection{Bihomogeneous components of the coderivative}\label{ss: delta_i j}

Given a leafwise metric $g_\FF$, the coderivative on the leaves defines an operator $\delta_\FF\in\Diff^1(M;\Lambda\FF)$, like in the case of $d_\FF$.

Fix a Riemannian metric $g$ on $M$. Using $\bfH=T\FF^\perp$ and taking formal adjoints in~\eqref{d = d_0 1 + d_1 0 + d_2 -1} and~\eqref{d_0,1^2=...=0}, we get a decomposition of the coderivative on $C^\infty(M;\Lambda)$,
\begin{equation}\label{delta = delta_0 -1 + delta_-1 0 + delta_-2 1}
\delta=\delta_{0,-1}+\delta_{-1,0}+\delta_{-2,1}\;,
\end{equation}
and the bihomogeneous components $\delta_{-i,i-1}=d_{i,1-i}^*$ \index{$\delta_{-i,i-1}$} satisfy the analog of~\eqref{d_0,1^2=...=0}. 

The metric $g$ induces a leafwise metric $g_\FF$. It also induces an Euclidean structure on $N\FF$, which in turn induces a Hermitian structure on $\Lambda N\FF$. Thus the adjoint $\delta_\FF=d_\FF^*$ \index{$\delta_\FF$} is also defined on $C^\infty(M;\Lambda\FF\otimes\Lambda N\FF)$. The analogue of~\eqref{d_0 1 equiv d_FF},
\begin{equation}\label{delta_0 -1 equiv delta_FF}
\delta_{0,-1}\equiv\delta_\FF
\end{equation}
via~\eqref{Lambda M}, holds if and only if $g$ is bundle-like \cite[Lemma~4.12]{AlvKordyLeichtnam2020}. Thus, in this case, $\delta=\delta_{0,-1}\equiv\delta_\FF$ on $C^\infty(M;\Lambda^{0,\bullet})\equiv C^\infty(M;\Lambda\FF)$ via~\eqref{Lambda FF subset Lambda FF otimes Lambda N FF} and~\eqref{Lambda^u v M}. 
 
The following operators will be also used:
\begin{equation}\label{D_0 D_perp Delta_0}
\left\{
\begin{gathered}
D_0=d_{0,1}+\delta_{0,-1}\;,\quad
D_\perp=d_{1,0}+\delta_{-1,0}\;,\\
\Delta_0=D_0^2=d_{0,1}\delta_{0,-1}+\delta_{0,-1}d_{0,1}\;.
\end{gathered}
\right.
\end{equation}

\subsection{Bigrading vs orientations}\label{ss: bigrading vs orientations}

Recall that a transverse orientation of $\FF$ can be described by a non-vanishing real form $\omega\in C^\infty(M;\Lambda^{n'}N\FF)\equiv C^\infty(M;\Lambda^{n',0})$. According to \Cref{ss: transverse structures}, there is a real 1-form $\eta$ satisfying $d\omega=\eta\wedge\omega$. We write $\eta=\eta_0+\eta_1$, where $\eta_0\in C^\infty(M;\Lambda^{0,1})$ is determined by $\omega$, and $\eta_1\in C^\infty(M;\Lambda^{1,0})$ can be chosen arbitrarily. 

On the other hand, an orientation of $T\FF$ is called a (\emph{leafwise} or \emph{tangential}) \emph{orientation} of $\FF$, which can be described by a non-vanishing real form $\chi\in C^\infty(M;\Lambda^{n''}\FF)\equiv C^\infty(M;\Lambda^{0,n''})$. It is said $\FF$ is \emph{oriented} \index{oriented foliation} if it is endowed with an orientation. Given transverse and tangential orientations of $\FF$, described by forms $\omega$ and $\chi$ as above, we consider the induced orientation of $M$ defined by the non-vanishing real form $\chi\wedge\omega\in C^\infty(M;\Lambda^{n',n''})=C^\infty(M;\Lambda^n)$.

Suppose that $M$ is a Riemannian manifold and take $\bfH=T\FF^\perp$. Then, using~\eqref{Lambda M}, the induced Hodge star operators, $\star$ on $\Lambda M$, $\star_\FF$ \index{$\star_\FF$} on $\Lambda\FF$ and $\star_\perp$ \index{$\star_\perp$} on $\Lambda\bfH$, satisfy\footnote{The sign of this expression, used in \cite[Eq.~(42)]{AlvKordyLeichtnam2020}, is different from the sign used in \cite[Lemma~3.2]{AlvKordy2001} by the different choices of induced orientation of $M$.} \cite[Lemma~4.8]{AlvTond1991}, \cite[Lemma~3.2]{AlvKordy2001}, \cite[Eq.~(42)]{AlvKordyLeichtnam2020}
\begin{equation}\label{star equiv (-1)^u(n''-v) star_FF otimes star_perp}
\star\equiv(-1)^{u(n''-v)}{\star_\FF}\otimes{\star_\perp}:\Lambda^{u,v}M\to\Lambda^{n'-u,n''-v}M\;.
\end{equation}
If $\omega=\star_\perp1$ and $\chi=\star_\FF1$, then $\chi\wedge\omega=\star1$. We have
\begin{equation}\label{delta_-i i-1 = (-1)^nk+n+1 star d_i 1-i star}
\delta_{-i,i-1}=(-1)^{nk+n+1}\star d_{i,1-i}\,\star
\end{equation}
on $C^\infty(M;\Lambda^k)$, and
\begin{equation}\label{delta_FF = (-1)^n''v+n''+1 star_FF d_FF star_FF}
\delta_\FF=(-1)^{n''v+n''+1}\star_\FF d_\FF\,\star_\FF
\end{equation}
on $C^\infty(M;\Lambda^v\FF)$. Using~\eqref{star equiv (-1)^u(n''-v) star_FF otimes star_perp}--\eqref{delta_FF = (-1)^n''v+n''+1 star_FF d_FF star_FF}, we easily get
\[
\delta_{0,-1}\equiv\delta_\FF+{\eta_0\lrcorner}
\]
on $C^\infty(M;\Lambda^{0,v})\equiv C^\infty(M;\Lambda^v\FF)$.

\subsection{Leafwise Euler form}\label{ss: e(FF g_FF)}

If $\FF$ is oriented, then $\Omega\FF\equiv\Lambda^{0,n''}M\equiv\Lambda^{n''}\FF$. If moreover $\FF$ is equipped with a leafwise Riemannian metric $g_\FF$ and $n''$ is even, then the \emph{leafwise Euler form}\index{leafwise Euler form} $e(\FF,g_\FF)\in C^\infty(M;\Lambda^{n''}\FF)\equiv C^\infty(M;\Omega\FF)$\index{$e(\FF,g_\FF)$} is defined by the Euler form of the leaves (\Cref{ss: local index formulae}). When $\FF$ is not oriented, $e(\FF,g_\FF)$ is defined as an element of $C^\infty(M;\Lambda^{n''}\FF\otimes o(\FF))\equiv C^\infty(M;\Omega\FF)$.

\subsection{Leafwise currents}\label{ss: leafwise currents}

We may also consider the continuous extension of $d_\FF$ to $C^{-\infty}(M;\Lambda\FF)$ (\Cref{ss: diff ops}), defining another topological complex whose cohomology and reduced cohomology are denoted by $H^\bullet C^{-\infty}(\FF)$ and $\bar H^\bullet C^{-\infty}(\FF)$ (see \Cref{ss: leafwise complex}). The elements of $C^{-\infty}(M;\Lambda\FF)$ are called \emph{leafwise currents}.\index{leafwise currents} In general, $C^\infty(M;\Lambda\FF)\hookrightarrow C^{-\infty}(M;\Lambda\FF)$ does not induce an isomorphism in cohomology or reduced cohomology (consider a foliation by points). 

Like in~\eqref{extension of exterior product to currents}, the exterior product has continuous extensions,
\[
C^{\pm\infty}(M;\Lambda\FF)\otimes C^{\mp\infty}(M;\Lambda\FF)\to C^{-\infty}(M;\Lambda\FF)\;,
\]
with a corresponding extension of the property of $d_\FF$ to be a derivation. Given a leafwise metric $g_\FF$ on $M$, we can also consider the continuous extension $\delta_\FF$ to $C^{-\infty}(M;\Lambda\FF)$

The concept of leafwise currents with coefficients in any $\FF$-flat vector bundle $E$ can be also considered, and the obvious notation is used for the corresponding topological complex and its cohomology and reduced cohomology. In particular, $E$ can be any vector bundle associated with $N\FF$.

\subsection{Bigrading of currents}\label{ss: bigrading of currents}

Consider also the bigrading of $C^{-\infty}(M;\Lambda)$ induced by the bigrading of $\Lambda M$, and the continuous extensions to $C^{-\infty}(M;\Lambda)$ of the operators $d_{i,1-i}$, which satisfy~\eqref{d_0 1 equiv d_FF} and~\eqref{d_0 1 = d on C^infty(M Lambda^n' cdot)}. Given a metric $g$ on $M$, we can also consider the continuous extensions of the operators $\delta_{-i,i-1}$ to $C^{-\infty}(M;\Lambda)$.

If $M$ is oriented, then~\eqref{(Lambda^rM)^* otimes Omega M equiv Lambda^n-rM},~\eqref{C^-infty(M Lambda^r) equiv C^infty_c(M Lambda^n-r)'} and~\eqref{d_z equiv (-1)^k+1 d_-z^t} for $z=0$ give
\begin{gather}
(\Lambda^{u,v}M)^*\otimes\Omega M\equiv\Lambda^{n'-u,n''-v}M\;,\
\label{(Lambda^u v M)^* otimes Omega M equiv Lambda^n'-u n''-v M}\\
C^{-\infty}(M;\Lambda^{u,v})\equiv\Cinftyc(M;\Lambda^{n'-u,n''-v})'\;,
\label{C^-infty(M Lambda^u v) equiv C^infty_c(M Lambda^n'-u n''-v)'}\\
d_{0,1}\equiv(-1)^{u+v+1}\,d_{0,1}^\trans:C^{-\infty}(M;\Lambda^{u,v})\to C^{-\infty}(M;\Lambda^{u,v+1})\;. 
\label{d_0 1 equiv (-1)^u+v+1 d_0 1^t}
\end{gather}
When $M$ is not oriented, these identities hold after adding the tensor product with $o(M)$ to the exterior bundles in the right-hand sides, or working locally, or passing to the double cover of orientations. By~\eqref{d_0 1 = d on C^infty(M Lambda^n' cdot)}, if $u=n'$, then~\eqref{d_0 1 equiv (-1)^u+v+1 d_0 1^t} agrees with~\eqref{d_z equiv (-1)^k+1 d_-z^t} for $z=0$ on $C^{-\infty}(M;\Lambda^{n',v})$. By~\eqref{Lambda M},~\eqref{d_0 1 equiv d_FF} and~\eqref{d_0 1 = d on C^infty(M Lambda^n' cdot)}, if $u=0$, then~\eqref{(Lambda^u v M)^* otimes Omega M equiv Lambda^n'-u n''-v M}--\eqref{d_0 1 equiv (-1)^u+v+1 d_0 1^t} become
\begin{gather}
(\Lambda^v\FF)^*\otimes\Omega M\equiv\Lambda^{n',n''-v}M\equiv\Lambda^{n''-v}\FF\otimes\Lambda^{n'}N\FF
\equiv\Lambda^{n''-v}\FF\otimes\Omega N\FF\;,
\label{(Lambda^v FF)^* otimes Omega M equiv Lambda^n' n''-v M}\\
C^{-\infty}(M;\Lambda^v\FF)\equiv\Cinftyc(M;\Lambda^{n',n''-v})'
\equiv\Cinftyc(M;\Lambda^{n''-v}\FF\otimes\Omega N\FF)'\;,
\label{C^-infty(M Lambda^v FF) equiv C^infty_c(M Lambda^n''-v FF otimes Omega N FF)'}\\
d_\FF\equiv(-1)^{v+1}\,d^\trans\equiv(-1)^{v+1}\,d_\FF^\trans:
C^{-\infty}(M;\Lambda^v\FF)\to C^{-\infty}(M;\Lambda^{v+1}\FF)\;. 
\label{d_FF equiv (-1)^v+1 d_FF^t}
\end{gather}

\subsection{Pull-back of leafwise forms}\label{ss: pull-back of leafwise forms}

Let $\phi\in C^\infty(M',\FF';M,\FF)$. Like in~\eqref{phi^*: C^infty(M Lambda) to C^infty(M' Lambda)} and~\eqref{composition -pull-back - diff forms}, the homomorphisms $\phi_*:T\FF'\to T\FF$ and $\phi_*:N\FF'\to N\FF$ induce continuous homomorphisms,
\begin{gather}
\phi^*:(C^\infty(M;\Lambda\FF\otimes\Lambda N\FF),d_\FF)\to(C^\infty(M';\Lambda\FF'\otimes\Lambda N\FF'),d_{\FF'})\;,
\label{phi^*}\\
\phi^*:(C^\infty(M;\Lambda\FF),d_\FF)\to(C^\infty(M';\Lambda\FF'),d_{\FF'})\;,
\label{phi^* on leafwise forms}
\end{gather}
the second one is a restriction of the first one according to~\eqref{Lambda FF subset Lambda FF otimes Lambda N FF}.

On the other hand, $\phi^*:C^\infty(M;\Lambda)\to C^\infty(M';\Lambda)$ has restrictions
\[
\phi^*:C^\infty(M;\Lambda^{\ge u,\cdot})\to C^\infty(M';\Lambda^{\ge u,\cdot})\;,
\]
which induce~\eqref{phi^*} using~\eqref{bigwedge^ge u cdot T^*M / bigwedge^ge u+1 cdot T^*M}.

\subsection{Bihomogeneous components of pull-back homomorphisms}\label{ss: phi^*_i -i}

For any smooth map $\phi:M'\to M$, the homomorphism $\phi^*:C^\infty(M;\Lambda)\to C^\infty(M';\Lambda)$ decomposes into bihomogeneous components,\index{$\phi^*_{i,-i}$}
\[
\phi^*=\cdots+\phi^*_{-1,1}+\phi^*_{0,0}+\phi^*_{1,-1}+\cdots
\]
If $\phi\in C^\infty(M',\FF';M,\FF)$, then $\phi^*_{i,-i}=0$ for $i<0$. Moreover, via~\eqref{Lambda M},
\begin{equation}\label{phi^*_0 0 equiv phi^*}
\phi^*_{0,0}\equiv\phi^*\;,
\end{equation}
where the right-hand side is~\eqref{phi^*}.

\subsection{Bihomogeneous components of the Lie derivative}\label{ss: LL_X i -i}

For any $X\in\fX(M)$, by comparing bidegrees in Cartan's formula, $\LL_X=d\iota_X+\iota_Xd$, we get a decomposition into bi-homogeneous components,
\[
\LL_X=\LL_{X,-1,1}+\LL_{X,0,0}+\LL_{X,1,-1}+\LL_{X,2,-2}\;.
\]
For instance,
\begin{equation}\label{LL_X 0 0}
\LL_{X,0,0}=d_{0,1}\iota_{\bfV X}+\iota_{\bfV X}d_{0,1}+d_{1,0}\iota_{\bfH X}+\iota_{\bfH X}d_{1,0}\;,
\end{equation}
where $\bfV:TM\to T\FF$ and $\bfH:TM\to\bfH$ denote the projections defined by~\eqref{splitting}. By comparing bidegrees in the derivation formula of $\LL_X$, we also get that every $\LL_{X,i,-i}$\index{$\LL_{X,i,-i}$} ($i\in\{-1,0,1,2\}$) satisfies the same derivation formula. Thus $\LL_{X,-1,1}$, $\LL_{X,1,-1}$ and $\LL_{X,2,-2}$ are of order zero. For the sake of simplicity, we will write $\Theta_X=\LL_{X,0,0}$,\index{$\Theta_X$} which is of order 1.

If $X\in\fX(M,\FF)$, then $\LL_{X,-1,1}=0$, obtaining
\begin{equation}\label{Theta_X d_0 1 = d_0 1 Theta_X}
\Theta_Xd_{0,1}=d_{0,1}\Theta_X
\end{equation} 
by comparing bi-degrees in the formula $\LL_Xd=d\LL_X$. On the other hand, by~\eqref{LL_X 0 0}, if $X\in C^\infty(M;\bfH)$, then, for all $f\in C^\infty(M)$,
\begin{equation}\label{Theta_fX = f Theta_X}
\Theta_{fX}=f\Theta_X
\end{equation}
on $C^\infty(M;\Lambda^{0,\bullet})\equiv C^\infty(M;\Lambda\FF)$. If $d_{1,0}f=0$, then~\eqref{Theta_fX = f Theta_X} holds on $C^\infty(M;\Lambda)$ by~\eqref{LL_X 0 0} and the derivation formula of $\Theta_X$.

\subsection{Local descriptions}\label{ss: local descriptions}

Let $(U,x)$ be a foliated chart of $\FF$, with $x=(x',x'')$, like in~\eqref{x = (x' x'')}. To emphasize the difference between the coordinates $x'$ and $x''$, we use the following notation on $U$ or $x(U)$. Let $x^{\prime i}=x^i$ and $\partial_i^{\prime}=\partial_i$ for $i\le n'$, and $x^{\prime\prime i}=x^i$ and $\partial_i^{\prime\prime}=\partial_i$ for $i>n'$. Thus, when using $x^{\prime i}$ or $\partial_i^{\prime}$, it will be understood that $i$ runs in $\{1,\dots,n'\}$, and, when using $x^{\prime\prime i}$ or $\partial_i^{\prime\prime}$, it will be understood that $i$ runs in $\{n'+1,\dots,n\}$. For multi-indices of the form $J=\{j_1,\dots,j_r\}$ with $1\le j_1<\dots<j_r\le n$, let $dx^J=dx^{j_1}\wedge\dots\wedge dx^{j_r}$ be denoted by $dx^{\prime J}$ or $dx^{\prime\prime J}$ if $J$ only contains indices in $\{1,\dots,n'\}$ or $\{n'+1,\dots,n\}$, respectively. Using functions $f_I\in C^\infty(U)$, $d_\FF$ on $U$ can be described by
\begin{equation}\label{d_FF(f_I dx^prime prime I)}
d_\FF(f_I\,dx^{\prime\prime I})=\partial''_jf_I\,dx^{\prime\prime j}\wedge dx^{\prime\prime I}\;.
\end{equation}
Since the forms $dx^{\prime J}$ are basic,~\eqref{d_0 1 equiv d_FF} on $U$ means that
\begin{equation}\label{d_0 1(f_IJ dx''^I wedge dx'^J)}
d_{0,1}(f_{IJ}\,dx^{\prime\prime I}\wedge dx^{\prime J})=d_\FF(f_{IJ}\,dx^{\prime\prime I})\wedge dx^{\prime J}\;,
\end{equation}
using functions $f_{IJ}\in C^\infty(U)$.

Given a metric $g$ on $M$, the local description
\begin{equation}\label{delta_0 -1(f_IJ dx''^I wedge dx'^J)}
\delta_{0,-1}(f_{IJ}\,dx^{\prime\prime I}\wedge dx^{\prime J})=\delta_\FF(f_{IJ}\,dx^{\prime\prime I})\wedge dx^{\prime J}
\end{equation}
is satisfied just when $g$ is bundle-like \cite[Lemma~3.4]{AlvKordy2001}; in fact this is a local expression of~\eqref{d_0 1 equiv d_FF}.

From~\eqref{LL_X 0 0}, we also get that, on $C^\infty(U,\Lambda^{0,{\bullet}})$,
\begin{equation}\label{d_1,0 = dx^prime i wedge Theta_bfH partial'_i}
d_{1,0}={dx^{\prime i}\wedge}\,\Theta_{\bfH\partial'_i}\;.
\end{equation}
Since $d_{1,0}$ is an anti-derivation, it follows that
\begin{equation}\label{d_1 0(f_IJ dx''^I wedge dx'^J)}
d_{1,0}(f_{IJ}\,dx^{\prime\prime I}\wedge dx^{\prime J})=(-1)^{|I|}\Theta_{\bfH\partial'_i}(f_{IJ}\,dx^{\prime\prime I})\wedge dx^{\prime i}\wedge dx^{\prime J}\;.
\end{equation}

\subsection{Bigrading of leafwise forms}\label{ss: bigrading of leafwise forms}

Suppose $\FF$ is subfoliation of another smooth foliation $\GG$ on $M$. Like in \Cref{ss: bigrading}, for any choice of a complement $\bfG$ of $T\FF$ in $T\GG$, we have $\Lambda\GG=\Lambda\FF\otimes\Lambda\bfG$, obtaining a bigrading of $\Lambda\GG$ defined by $\Lambda^{u,v}\GG=\Lambda^v\FF\otimes\Lambda^u\bfG$, and a corresponding bigrading of $C^\infty(M;\Lambda\GG)$. The decomposition~\eqref{d = d_0 1 + d_1 0 + d_2 -1} has an obvious version for $d_\GG$ satisfying analogous properties.

\subsection{Push-forward and pull-back of leafwise currents}\label{ss: push-forward of leafwise currents}

With the notation of \Cref{ss: pull-back and push-forward of currents}, assume $\phi:M'\to M$ is a  smooth submersion and $\VV$ oriented. Using any complement $\HH$ of $\VV$ in $TM'$, we get a corresponding bigrading of $\Lambda M'$ with $\phi^*\Lambda M\otimes\Omega_\fiber M'\equiv\Lambda^{\bullet,\topd}M'$. Suppose $M$ is equipped with a smooth foliation $\FF$, and let $\FF'=\phi^*\FF$. Choose complements, $\bfH$ of $T\FF$ in $TM$ and $\bfH'$ of $T\FF'$ in $TM'$. The tangent map $\phi_*$ defines an identity $\bfH'\equiv\phi^*\bfH$. Consider the bigradings of $\Lambda M$ and $\Lambda M'$ induced by $(\FF,\bfH)$ and $(\FF',\bfH')$. Then the maps~\eqref{phi_*: C^infty_c/cv(M' Lambda) to C^infty_c/.(M Lambda)}--\eqref{phi^*: C^-infty(M Lambda) to C^-infty(M' Lambda)} have restrictions compatible with $d_{0,1}$,
\begin{gather*}
\phi_*:C^{\pm\infty}_{\co/\cv}(M';\Lambda^{u,\bullet})\to C^{\pm\infty}_{\co/{\cdot}}(M;\Lambda^{u,\bullet-p})\quad(p=\dim\VV)\;,\\
\phi^*:C^{-\infty}(M;\Lambda^{u,\bullet})\to C^{-\infty}(M';\Lambda^{u,\bullet})\;.
\end{gather*}
For $u=0$, by~\eqref{Lambda FF subset Lambda FF otimes Lambda N FF}, they are continuous homomorphisms,
\begin{gather}
\phi_*:(C^{\pm\infty}_{\co/\cv}(M';\Lambda\FF'),d_{\FF'})\to(C^{\pm\infty}_{\co/{\cdot}}(M;\Lambda\FF),d_\FF)\;,
\label{phi_*: C^pm infty_c/cv(M' Lambda FF') to C^pm infty_c/.(M Lambda FF)}\\
\phi^*:(C^{-\infty}(M;\Lambda\FF),d_\FF)\to(C^{-\infty}(M';\Lambda\FF'),d_{\FF'})\;.
\label{phi^*: C^-infty(M Lambda FF) to (C^-infty(M' Lambda FF')}
\end{gather}
Like in~\eqref{phi_*: C^-infty_c/cv(M' Lambda) to C^-infty_c/.(M Lambda)}--\eqref{composition - pull-back - currents}, the maps~\eqref{phi_*: C^pm infty_c/cv(M' Lambda FF') to C^pm infty_c/.(M Lambda FF)} and~\eqref{phi^*: C^-infty(M Lambda FF) to (C^-infty(M' Lambda FF')} can be also defined as the compositions
\begin{gather}
C^{\pm\infty}_{\co/\cv}(M';\Lambda\FF')\xrightarrow{\pi_\topd}C^{\pm\infty}_{\co/\cv}(M';\phi^*\Lambda\FF)
\xrightarrow{\phi_*}C^{\pm\infty}_{\co/{\cdot}}(M;\Lambda\FF)\;,
\label{push-forward - composition - leafwise currents}\\
C^{-\infty}(M;\Lambda\FF)\xrightarrow{\phi^*}C^{-\infty}(M';\phi^*\Lambda \FF)\xrightarrow{\phi^*}C^{-\infty}(M';\Lambda\FF')\;.
\label{pull-back - composition - leafwise currents}
\end{gather}

We can directly extend the definition of~\eqref{phi_*: C^pm infty_c/cv(M' Lambda FF') to C^pm infty_c/.(M Lambda FF)} to the case where $M'$ is a manifold with boundary, assuming $\FF'$ is tangent or transverse to the boundary. It is a cochain map when $\FF'$ is tangent to the boundary. If $\FF'$ is transverse to the boundary and $\phi|_{\partial M'}:\partial M'\to M$ is a submersion, the Stokes' formula gives
\begin{equation}\label{phi_* d_FF' - d_FF phi_* = (phi|_partial M')_* iota^*}
\phi_*d_{\FF'}-d_\FF\phi_*=(\phi|_{\partial M'})_*\iota^*:\Cinftyc(M';\Lambda\FF')\to\Cinftyc(M;\Lambda\FF)\;,
\end{equation}
where $\iota:\partial M'\hookrightarrow M'$.

\subsection{Leafwise homotopy operators}\label{ss: leafwise homotopy opers}

With the notation of \Cref{ss: homotopy opers}, suppose $M$ and $M'$ are equipped with respective smooth foliations $\FF$ and $\FF'$, $H$ is a leafwise homotopy, and consider $H_t^*:C^\infty(M;\Lambda\FF)\to C^\infty(M';\Lambda\FF')$ ($t\in I$). Then we similarly get a continuous linear map $\sh:C^\infty(M;\Lambda\FF)\to C^\infty(M';\Lambda\FF')$, called a \emph{leafwise homotopy operator},\index{leafwise homotopy operator} which is homogeneous of degree $-1$ and satisfies $H_1^*-H_0^*=\sh d_\FF+d_{\FF'}\sh$. By using~\eqref{phi^* on leafwise forms},~\eqref{phi_* d_FF' - d_FF phi_* = (phi|_partial M')_* iota^*} and~\eqref{phi_*: C^pm infty_c/cv(M' Lambda FF') to C^pm infty_c/.(M Lambda FF)}, $\sh$ can be given as the composition
\begin{equation}\label{sh = fint_I H^* leafwise}
C^\infty(M;\Lambda\FF) \xrightarrow{H^*} C^\infty(M'\times I;\Lambda(\FF'\times I))
\xrightarrow{\fint_I} C^\infty(M';\Lambda\FF')\;.
\end{equation}
So $H_0$ and $H_1$ induce the same homomorphisms $H^\bullet C^\infty(\FF)\to H^\bullet C^\infty(\FF')$ and $\bar H^\bullet C^\infty(\FF)\to\bar H^\bullet C^\infty(\FF')$. 

Suppose $H$ is transverse to $\FF$ and $H^*\FF=\FF'\times I$. Let $\pr_1:M'\times I\to M'$ denote the first-factor projection. Consider the bigradings defined by $\FF$, $\FF'$, and complements $\bfH$ and $\bfH'$ of their tangent bundles. So $H_*$ defines a homomorphism $\pr_1^*\bfH'\to\bfH$ whose restrictions to the fibers are isomorphisms. Then~\eqref{sh = fint_I H^* leafwise} is the bihomogeneous component of bidegree $(0,-1)$ of~\eqref{sh = fint_I H^*}.

If moreover $H$ is a submersion, then~\eqref{phi_*: C^pm infty_c/cv(M' Lambda FF') to C^pm infty_c/.(M Lambda FF)} and~\eqref{phi^*: C^-infty(M Lambda FF) to (C^-infty(M' Lambda FF')} give a continuous extension of the maps of~\eqref{sh = fint_I H^* leafwise},
\[
C^{-\infty}(M;\Lambda\FF) \xrightarrow{H^*} C^{-\infty}(M'\times I;\Lambda(\FF'\times I))
\xrightarrow{\fint_I} C^{-\infty}(M';\Lambda\FF')\;.
\]
Their composition, $\sh:C^{-\infty}(M;\Lambda\FF)\to C^{-\infty}(M';\Lambda\FF')$, satisfies $H_1^*-H_0^*=\sh d_\FF+d_{\FF'}\sh$. Thus $H_0$ and $H_1$ also induce the same homomorphisms $H^\bullet C^{-\infty}(\FF)\to H^\bullet C^{-\infty}(\FF')$ and $\bar H^\bullet C^{-\infty}(\FF)\to\bar H^\bullet C^{-\infty}(\FF')$.

\section{Witten's perturbation on foliated manifolds}\label{s: Witten's perturbation on fold mfds}

The operators acting on differential forms on foliated manifolds (\Cref{s: dif forms on fold mfds}) are extended now by taking Witten's perturbations (\Cref{s: Witten}).

\subsection{Perturbation vs bigrading}\label{ss: perturbation vs bigrading}

Using the notation of \Cref{ss: Witten's complex,s: dif forms on fold mfds}, write\footnote{In \cite[Section~11]{AlvKordyLeichtnam2020}, we took $\eta\in C^\infty(M;\Lambda^{0,1})$. However a general $\eta$ is needed, and therefore additional work is required in \Cref{s: Witten's perturbation on fold mfds,s: Witten's opers on Riem folns of bd geom}.} $\eta=\eta_0+\eta_1$ with $\eta_0\in C^\infty(M;\Lambda^{0,1})\equiv C^\infty(M;\Lambda^1\FF)$ and $\eta_1\in C^\infty(M;\Lambda^{1,0})$. The condition $d\eta=0$ means
\begin{equation}\label{d_0 1 eta_0 = ... = 0}
d_{0,1}\eta_0=d_{1,0}\eta_1=d_{1,0}\eta_0+d_{0,1}\eta_1=0\;.
\end{equation}
Like in~\eqref{d = d_0 1 + d_1 0 + d_2 -1} and~\eqref{delta = delta_0 -1 + delta_-1 0 + delta_-2 1}, we get
\[
d_z=d_{z,0,1}+d_{z,1,0}+d_{2,-1}\;,\quad\delta_z=\delta_{z,0,-1}+\delta_{z,-1,0}+\delta_{-2,1}\;,
\]
where
\begin{alignat*}{2}
d_{z,0,1}&=d_{0,1}+z\,{\eta_0\wedge}\;,&\quad d_{z,1,0}&=d_{1,0}+z\,{\eta_1\wedge}\;,\\
\delta_{z,0,-1}&=\delta_{0,1}-\bar z\,{\eta_0\lrcorner}\;,&\quad\delta_{z,-1,0}&=\delta_{-1,0}-\bar z\,{\eta_1\lrcorner}\;.
\end{alignat*}
We will also use the perturbed versions of the operators~\eqref{D_0 D_perp Delta_0}, denoted by $D_{z,0}$, $D_{z,\perp}$ and $\Delta_{z,0}$, defined with the operators $d_{z,i,1-i}$ and $\delta_{z,i,i-1}$. \index{$d_{z,i,1-i}$} \index{$\delta_{z,i,i-1}$} \index{$D_{z,0}$} \index{$D_{z,\perp}$} \index{$\Delta_{z,0}$}

There is an obvious analog of~\eqref{d_0,1^2=...=0} for the operators $d_{z,i,1-i}$, giving rise to analogous relations for the operators $\delta_{z,i,i-1}$. In particular, $d_{z,0,1}$ and $\delta_{z,0,-1}$ define leafwise differential complexes. By~\eqref{delta_z}, the expressions~\eqref{delta_-i i-1 = (-1)^nk+n+1 star d_i 1-i star} and~\eqref{delta_FF = (-1)^n''v+n''+1 star_FF d_FF star_FF} have direct extensions to this setting as well. 

Concerning uniform leafwise/transverse ellipticity, symmetry and being non-negative, the perturbations $d_{z,0,1}$, $\delta_{z,0,-1}$, $D_{z,0}$, $D_{z,\perp}$ and $\Delta_{z,0}$ satisfy the same properties as $d_{0,1}$, $\delta_{0,-1}$, $D_0$, $D_{\perp}$ and $\Delta_0$. 

By~\eqref{d_1,0 = dx^prime i wedge Theta_bfH partial'_i}, on a foliated chart $(U,x)$, we get
\[
d_{1,0}\eta_0=dx^{\prime i}\wedge\Theta_{\bfH\partial_i}\eta_0=-\Theta_{\bfH\partial_i}\eta_0\wedge dx^{\prime i}\;.
\]
But, writing $\eta_1=h_i\,dx^{\prime j}$, by~\eqref{d_0 1 eta_0 = ... = 0},
\[
d_{1,0}\eta_0=-d_{0,1}\eta^1=-\partial''_jh_i\,dx^{\prime\prime j}\wedge dx^{\prime i}\;.
\]
So 
\[
\Theta_{\bfH\partial_i}\eta_0=\partial''_jh_i\,dx^{\prime\prime j}\;.
\]
Then, since $\Theta_{\bfH\partial_i}$ is a derivation, on $C^\infty(M;\Lambda\FF)$,
\begin{equation}\label{[Theta_bfH eta_0 wedge]}
[\Theta_{\bfH\partial_i},{\eta_0\wedge}]={(\Theta_{\bfH\partial_i}\eta_0)\wedge}
=\partial''_jh_i\,{dx^{\prime\prime j}\wedge}=(d_\FF h_i)\wedge=[d_\FF,h_i]\;.
\end{equation}
Thus~\eqref{Theta_X d_0 1 = d_0 1 Theta_X} has the following change in this setting:
\[
[\Theta_{\bfH\partial_i},d_{z,0,1}]=z[d_\FF,h_i]\;.
\]

\subsection{Perturbation of the leafwise complex}\label{ss: perturbation of the leafwise complex}

Consider also the perturbed leafwise complex, $d_{\FF,z}=d_\FF+z\eta_0\wedge$\index{$d_{\FF,z}$} on $C^\infty(M;\Lambda\FF)$, or on $C^\infty(M;\Lambda\FF\otimes\Lambda N\FF)$, as well as its formal adjoint $\delta_{\FF,z}=\delta_\FF-\bar z\,{\eta_0\lrcorner}$,\index{$\delta_{\FF,z}$} and the induced perturbations, $D_{\FF,z}$\index{$D_{\FF,z}$} of $D_\FF$ and $\Delta_{\FF,z}$\index{$\Delta_{\FF,z}$} of $\Delta_\FF$. They satisfy the obvious versions of~\eqref{d_0 1 equiv d_FF} and~\eqref{d_0 1(f_IJ dx''^I wedge dx'^J)}. If $g$ is bundle-like, they also satisfy the obvious versions of~\eqref{delta_0 -1 equiv delta_FF} and~\eqref{delta_0 -1(f_IJ dx''^I wedge dx'^J)}.

\subsection{Perturbation with two parameters}\label{ss: leafwise 2 parameters}

For $z,z'\in\C$, the operators $D_{0,z,z'}$\index{$D_{0,z,z'}$} and $\Delta_{0,z,z'}$\index{$\Delta_{0,z,z'}$} are defined like $D_{z,z'}$ and $\Delta_{z,z'}$ (\Cref{ss: two parameters}), by using $d_{z,0,1}$ and $\delta_{z',0,-1}$ instead of $d_z$ and $\delta_{z'}$. In other words, $D_{0,z,z'}$ is the component of $D_{z,z'}$ that preserves the transverse degree, and $\Delta_{z,z'}=D_{z,z'}^2$. They are uniformly leafwise elliptic, with a symmetric leading symbol.

The operators $D_{\FF,z,z'}$\index{$D_{\FF,z,z'}$} and $\Delta_{\FF,z,z'}$\index{$\Delta_{\FF,z,z'}$} on $C^\infty(M;\Lambda\FF)$, or on $C^\infty(M;\Lambda\FF\otimes\Lambda N\FF)$, are defined like $D_{z,z'}$ and $\Delta_{z,z'}$, by using $d_{\FF,z}$ and $\delta_{\FF,z'}$ instead of $d_z$ and $\delta_{z'}$. They are also uniformly leafwise elliptic, with a symmetric leading symbol. If $g$ is bundle-like, they also agree with $D_{z,z'}$ and $\Delta_{z,z'}$ via~\eqref{Lambda M}.

\subsection{Perturbation vs foliated maps}\label{ss: perturbation vs fold maps}

With the notation of \Cref{ss: flat line bundle,ss: perturbation of pull-back homs} for a smooth foliated map $\phi:(M,\FF)\to(M,\FF)$, let $\widetilde\FF$ and $\tilde\eta_j$ ($j=1,2$) be the lifts of $\FF$ and $\eta_j$ to $\widetilde M$. Thus $\tilde\eta_0=d_{0,1}F\equiv d_\FF F$ and $\tilde\eta_1=d_{1,0}F$. Any lift $\tilde\phi$ of $\phi$ to $\widetilde M$ is a foliated diffeomorphism of $(\widetilde M,\widetilde\FF)$. The endomorphism $\phi^{t*}_z$ of $(C^\infty(M;\Lambda),d_z)$ decomposes into the sum of bihomogeneous components $\phi^*_{z,i,-i},$\index{$\phi^*_{z,i,-i}$} like in \Cref{ss: phi^*_i -i}, whose lifts to $C^\infty(\widetilde M;\Lambda)$ are $e^{z(\tilde\phi^*F-F)}\tilde\phi^*_{i,-i}$. Then $\phi^*_{z,0,0}$ is an endomorphism of $(C^\infty(M;\Lambda),d_{z,0,1})$. 

Similarly, the endomorphism $\phi^*$ of $(C^\infty(M;\Lambda\FF\otimes\Lambda N\FF),d_\FF)$ given by~\eqref{phi^*} has a perturbation $\phi^*_z$, which is an endomorphism of $(C^\infty(M;\Lambda\FF\otimes\Lambda N\FF),d_{\FF,z})$. We have $\phi^*_{z,0,0}\equiv\phi^*_z$ like in~\eqref{phi^*_0 0 equiv phi^*}. By restriction using~\eqref{Lambda FF subset Lambda FF otimes Lambda N FF}, we get an endomorphism $\phi^*_z$ of $(C^\infty(M;\Lambda\FF),d_{\FF,z})$, which is a perturbation of the endomorphism $\phi^*$ of $(C^\infty(M;\Lambda\FF),d_{\FF})$ given by~\eqref{phi^* on leafwise forms}.

\section{Analysis on Riemannian foliations of bounded geometry}\label{s: analysis on Riem folns}

In this section, $\FF$ is a Riemannian foliation on a possibly open manifold $M$, equipped with a bundle-like metric $g$. We adopt the notation of \Cref{s: bd geom} for the metric concepts of $M$. 

\subsection{Riemannian foliations of bounded geometry}\label{ss: transverse structures of bd geometry}

The vector subbundle $\bfH:=T\FF^\perp\subset TM$ is called \emph{horizontal},\index{horizontal subbundle} giving rise to the concepts of \emph{horizontal} vectors, vector fields and frames. Consider the corresponding splitting~\eqref{splitting}, obtaining orthogonal projections $\bfV:TM\to T\FF$ and $\bfH:TM\to\bfH$. The O'Neill tensors \cite{ONeill1966} of the local Riemannian submersions defining $\FF$ can be combined to produce $(1,2)$-tensors $\sT$\index{$\sT$} and $\sA$\index{$\sA$} on $M$, defined by
\begin{align*}
\sT_EF&=\bfH\nabla_{\bfV E}(\bfV F)+\bfV\nabla_{\bfV E}(\bfH F)\;,\\
\sA_EF&=\bfH\nabla_{\bfH E}(\bfV F)+\bfV\nabla_{\bfH E}(\bfH F)\;,
\end{align*}
for $E,F\in\fX(M)$.  By \cite[Theorem~4]{ONeill1966}, if $M$ is connected, given $g$ and any $p\in M$, the foliation $\FF$ is determined by $\sT$, $\sA$ and $T_p\FF$. 

The \emph{adapted} Riemannian connection $\rnabla$\index{$\rnabla$} on $M$ is defined by
\[
\rnabla_EF=\bfV\nabla_E(\bfV F)+\bfH\nabla_E(\bfH F)\;,
\]
for $E,F\in\fX(M)$ \cite{AlvTond1991}. It satisfies the following properties \cite[Section~3]{AlvKordyLeichtnam2014}, \cite[Section~5]{AlvKordyLeichtnam2020}: for $V\in\fX(\FF)$ and $X\in C^\infty(M;\bfH)$,
\begin{gather}
\nabla_V-\rnabla_V=\sT_V\;,\quad\nabla_X-\rnabla_X=\sA_X\;,\label{nabla-rnabla}\\
\bfV([X,V])=\rnabla_XV-\sT_VX\;.\label{bV[X,V]}
\end{gather} 
Moreover, the leaves are $\rnabla$-totally geodesic, the $\rnabla$-geodesics in the leaves are the $\nabla^\FF$-geodesics, and $\rnabla$ and $\nabla$ have the same geodesics orthogonal to the leaves.

Let $x':U\to\Sigma$ be a distinguished submersion around any $p\in M$. Consider the Riemannian metric on $\Sigma$ such that $x'$ is a Riemannian submersion, and let $\cnabla$ and $\cexp$ denote the corresponding Levi-Civita connection and exponential map of $\Sigma$. For all horizontal $X,Y\in\fX(U,\FF|_U)$, we have $\rnabla_XY\in\fX(U,\FF|_U)$ and $\overline{\rnabla_XY}=\cnabla_{\overline X}\overline Y$ \cite[Lemma~1~(3)]{ONeill1966}.

Let $\rexp$\index{$\rexp$} denote the exponential map of the geodesic spray of $\rnabla$ (see e.g.\ \cite[pp.~96--99]{Poor1981}). The maps $\rexp$ and $\cexp$ restrict to diffeomorphisms of some open neighborhoods, $V$ of $0$ in $T_pM$ and $\check V$ of $0$ in $T_{x'(p)}\Sigma$, to some open neighborhoods, $O$ of $p$ in $M$ and $\check O$ of $x'(p)$ in $\Sigma$.  Moreover we can suppose $O\subset U$, $x'_*(V)\subset\check V$ and $x'(O)\subset\check O$, and we have $x'\,\rexp=\cexp\,x'_*$ on $V\cap T_p\FF^\perp$. Let $\kappa=\kappa_p$ be the smooth map of some neighborhood $W$ of $0$ in $T_pM$ to $M$ defined by
$$
\kappa_p(X)=\rexp_q(\rP_{\bfH X}\bfV X)\;,
$$
where $q=\rexp_p(\bfH X)$, and $\rP_{\bfH X}:T_pM\to T_qM$ denotes the $\rnabla$-parallel transport along the $\rnabla$-geodesic $t\mapsto\rexp_p(t\bfH X)$, $0\le t\le1$, which is orthogonal to the leaves. Assume $W\subset V$ and $\kappa(W)\subset O$, and therefore $x'_*(W)\subset\cV$ and $x'\kappa(W)\subset\cO$. For $X,Y\in W$, we have $X-Y\in T_p\FF$ if and only if $\kappa(X)$ and $\kappa(Y)$ belong to the same plaque in $U$ \cite[Proposition~6.1]{AlvKordyLeichtnam2014}. Moreover $x'\kappa(X)=\cexp\,x'_*(X)$ for all $X\in W\cap T_p\FF^\perp$, and $\kappa$ defines a diffeomorphism of some neighborhood of $0$ in $T_pM$ to some neighborhood of $p$ in $M$ with $\kappa_*\equiv\id:T_0(TM)\equiv T_pM\to T_pM$  \cite[Proposition~6.2 and Corollary~6.3]{AlvKordyLeichtnam2014}. Consider identities $T_p\FF^\perp\equiv\R^{n'}$ and $T_p\FF\equiv\R^{n''}$ given by the choice of horizontal and vertical orthonormal frames at $p$. Then, for some open balls centered at the origin, $B'$ in $\R^{n'}$ and $B''$ in $\R^{n''}$, we can assume $\kappa$ is a diffeomorphism of $B'\times B''$ to some open neighborhood of $p$, obtaining foliated coordinates $x=(x',x''):=\kappa^{-1}:U:=\kappa(B'\times B'')\to B'\times B''$, which are said to be \emph{normal}. As usual, $g_{ij}$ denotes the corresponding metric coefficients and $(g^{ij})=(g_{ij})^{-1}$. It is said that $\FF$ has \emph{positive injectivity bi-radius}\index{positive injectivity bi-radius} if there are normal foliated coordinates $x_p:U_p\to B'\times B''$ at every $p\in M$ such that the balls $B'$ and $B''$ are independent of $p$. Then $\FF$ is said to be of \emph{bounded geometry}\index{Riemannian foliation of bounded geometry} if it has positive injectivity bi-radius, and the functions $|\nabla^mR|$, $|\nabla^m\sT|$ and $|\nabla^m\sA|$ are uniformly bounded on $M$ for every $m\in\N_0$ \cite[Definition~8.1]{AlvKordyLeichtnam2014}.

\begin{ex}\label{ex: Riem foln of bd geometry}
Let $H$ be a connected Lie group, $L\vartriangleleft H$ a normal connected Lie subgroup and $\Gamma\subset H$ a discrete subgroup. Then the projection of the translates of $L$ to $\Gamma\backslash H$ are the leaves of a Riemannian foliation of bounded geometry with the bundle-like metric induced by any left invariant metric on $H$.
\end{ex}

The following chart characterization of bounded geometry for Riemannian foliations is connected with another definition given by Sanguiao \cite[Definition~1.7]{Sanguiao2008}.

\begin{thm}[{\cite[Theorem~8.4]{AlvKordyLeichtnam2014}}]\label{t: foln of bd geometry}
  With the above notation, $\FF$ is of bounded geometry if and only if there is a normal foliated chart $x_p:U_p\to B'\times B''$ at every $p\in M$, such that the balls $B'$ and $B''$ are independent of $p$, and the corresponding coefficients $g_{ij}$ and $g^{ij}$, as family of smooth functions on $B'\times B''$ parametrized by $i$, $j$ and $p$, lie in a bounded subset of the Fr\'echet space $C^\infty(B'\times B'')$.
\end{thm}

For the rest of \Cref{s: analysis on Riem folns}, let us assume that $\FF$ is of bounded geometry. Then $M$ and the disjoint union of the leaves are of bounded geometry \cite[Remark~8.2 and Proposition~8.6]{AlvKordyLeichtnam2014}. Consider the foliated charts $y_p:V_p\to B$ and foliated charts $x_p:U_p\to B'\times B''$ given by \Cref{t: mfd of bd geom,t: foln of bd geometry}. Let $r_0$, $r'_0$ and $r''_0$ denote the radii of the balls $B$, $B'$ and $B''$. For $0<r\le r_0$, $0<r'\le r'_0$ and $0<r''\le r''_0$, let $R_r$, $B'_{r'}$ and $B''_{r''}$ denote the balls in $\R^{n'}$ and $\R^{n''}$ centered at the origin with radii $r$, $r'$ and $r''$, respectively. If $r$ is small enough, then $V_{p,r}:=x_p^{-1}(B_r)\subset U_p$ for all $p$ \cite[Proposition~8.6]{AlvKordyLeichtnam2014}. On the other hand, if $r'+r''\le r_0$, then $U_{p,r',r''}:=x_p^{-1}(B'_{r'}\times B''_{r''})\subset V_p$ for all $p$ by the triangle inequality. Then the following subsets are bounded in the corresponding Fr\'echet spaces \cite[Proposition~5.6 and~5.7]{AlvKordyLeichtnam2020}:
\begin{equation}\label{x_p y_p^-1}
\left\{
\begin{aligned}
\{\,x_py_p^{-1}\mid p\in M\,\}&\subset C^\infty(B,\R^{n'}\times\R^{n''})\;,\\
\{\,y_px_p^{-1}\mid p\in M\,\}&\subset C^\infty(B'_{r'}\times B''_{r''},\R^n)\;.
\end{aligned}
\right.
\end{equation}

Let $E$ be the Hermitian vector bundle of bounded geometry associated to the principal $\operatorname{O}(n)$-bundle of orthonormal frames on $M$ and a unitary representation of $\operatorname{O}(n)$ (Example~\ref{ex: bundles of bd geom}). Since $\nabla$ on $TM$ is of bounded geometry, it follows from~\eqref{nabla-rnabla} that $\rnabla$ is also of bounded geometry. Thus we get induced connections $\nabla$ and $\rnabla$ of bounded geometry on $E$ (Example~\ref{ex: connections of bd geom}). By~\eqref{nabla-rnabla}, we also get that $\rnabla$ can be used instead of $\nabla$ to define equivalent versions of $\|{\cdot}\|_{C_{\text{\rm ub}}^m}$ and $\langle\cdot,\cdot\rangle_m$ in the spaces $C_{\text{\rm ub}}^m(M;E)$ and $H^m(M;E)$.  Since the subsets~\eqref{x_p y_p^-1} are bounded, if $B'$ and $B''$ are small enough, then we can use the coordinates $(U_p,x_p)$ instead of coordinates of $(V_p,y_p)$ to define equivalent versions of $\|{\cdot}\|'_{C_{\text{\rm ub}}^m}$ and $\langle\cdot,\cdot\rangle'_m$. Similarly, given another bundle $F$ like $E$, we can use the coordinates $(U_p,x_p)$ instead of $(V_p,y_p)$ to describe $\Diffub^m(M;E,F)$ by requiring that the local coefficients form a bounded subset of the Fr\'echet space $C^\infty(B'\times B'';\C^{l'}\otimes\C^{l*})$, where $l$ and $l'$ are the ranks of $E$ and $F$.

The condition of being leafwise differential operators of bounded geometry is preserved by compositions, and by taking transposes and formal adjoints. They form a filtered $\Cinftyub(M)$-submodule $\Diffub(\FF;E,F)\subset\Diff(\FF;E,F)$. The notation $\Diffub(\FF;E)$\index{$\Diffub(\FF;E)$} is used if $E=F$; this is a filtered subalgebra of $\Diff(\FF;E)$. The concepts of \emph{uniform leafwise ellipticity}\index{uniformly leafwise elliptic} for operators in $\Diff^m(\FF;E,F)$ can be defined like uniform ellipticity (\Cref{ss: diff ops of bd geom}), and can be extended to leafwise differential complexes of order $m$ like in \Cref{ss: diff complexes}. The same applies to \emph{uniform transverse ellipticity} for operators in $\Diff^m(M;E,F)$ and for differential complexes of order $m$. If $P\in\Diffub^2(\FF;E)$ is uniformly leafwise elliptic and $Q\in\Diffub^2(M;E)$ is uniformly transversely elliptic, and both $P$ and $Q$ are symmetric and non-negative, then $H^s(M;E)$ ($s\in\R$) can be described with the scalar product $\langle u,v\rangle_s=\langle((1+P)^s+(1+Q)^s)u,v\rangle$. 

Let $\fXub(\FF)$\index{$\fXub(\FF)$} and $\fXub(M,\FF)$\index{$\fXub(M,\FF)$} denote the intersections of $\fXub(M)$ with $\fX(\FF)$ and $\fX(M,\FF)$, respectively. Then $\Diff_{\text{\rm ub}}(\FF)$ can be also described like in \Cref{ss: diff ops}, using $\Cinftyub(M)$ and $\fXub(\FF)$ instead of $C^\infty(M)$ and $\fX(M)$, and $\Diffub(\FF;E,F)$ can be also described as the $\Cinftyub(M)$-tensor product of $\Diff_{\text{\rm ub}}(\FF)$ and $\Cinftyub(M;E,F)$.

\subsection{Operators of bounded geometry on differential forms}\label{ss: opers of bd geom on diff forms}

Since $\nabla$ and $\rnabla$ are of bounded geometry on $TM$, the induced connections $\nabla$ and $\rnabla$ on $\Lambda M$ are of bounded geometry as well (\Cref{ex: connections of bd geom}). Using \Cref{ex: bundles of bd geom,ex: connections of bd geom}, we get that $\bfH$ and $T\FF$ are also of bounded geometry, and the restrictions of $\rnabla$ to $\bfH$ and $T\FF$ are of bounded geometry \cite[Section~6]{AlvKordyLeichtnam2020}. Thus every $\Lambda^{u,v}M$ is of bounded geometry (\Cref{ex: operations of bundles of bd geom}), and $\rnabla$ is of bounded geometry on $\Lambda^{u,v}M$ (\Cref{ex: connections induced by natural operations}). So this also applies to $\Lambda\FF\equiv\Lambda^{0,\bullet}M$.

By using $\rnabla$ instead of $\nabla$ in the definitions of $\|{\cdot}\|_{C_{\text{\rm ub}}^m}$ and $\langle\cdot,\cdot\rangle_s$ ($m\in\N_0$ and $s\in\R$), it follows that the spaces $C_{\text{\rm ub}}^m(M;\Lambda)$ and $H^s(M;\Lambda)$ inherit the bigrading of $\Lambda M$, and therefore $\Cinftyub(M;\Lambda)$ and $H^{\pm\infty}(M;\Lambda)$ have an induced bigrading. 

The following properties hold \cite[Section~3]{AlvKordy2001}, \cite[Section~6]{AlvKordyLeichtnam2020}: the canonical projections $\Lambda M\to\Lambda^{u,v}M$, the operators  $\star$, $\star_\FF$ or $\star_\perp$ (under appropriate orientability assumptions), and the operators of~\eqref{d = d_0 1 + d_1 0 + d_2 -1},~\eqref{delta = delta_0 -1 + delta_-1 0 + delta_-2 1} and~\eqref{D_0 D_perp Delta_0} are of bounded geometry; the differential complexes $d_{0,1}$ and $\delta_{0,-1}$ are uniformly leafwise elliptic; the differential operators $D_0$ and $\Delta_0$ are symmetric and uniformly leafwise elliptic; the differential operator $D_\perp$ is uniformly transversely elliptic; and there is an endomorphism of bounded geometry, $K$ of $\Lambda M$, such that\footnote{In \cite[Eq.~(55)]{AlvKordyLeichtnam2020}, $D_0$ should be $\delta_{0,-1}$, like in \cite[Proposition~3.1]{AlvKordy2001}.}
\begin{equation}\label{D_perp delta_0 -1 + delta_0 -1 D_perp}
D_\perp\delta_{0,-1}+\delta_{0,-1}D_\perp=K\delta_{0,-1}+\delta_{0,-1}K\;.
\end{equation}

Let us recall the definition of $K$ and the proof of~\eqref{D_perp delta_0 -1 + delta_0 -1 D_perp} because an extension will be needed, which is slightly more general than the extension considered in \cite[Section~11]{AlvKordyLeichtnam2020}. Let $\Theta:\fX(\FF)\to C^\infty(M;\bfH^*\otimes T\FF)$ be the differential operator defined by $\Theta_XV=\bfV([X,V])$ (the expression~\eqref{bV[X,V]}), which induces a differential operator $\Theta:C^\infty(M;\Lambda\FF)\to C^\infty(M;\bfH^*\otimes\Lambda\FF)$. If $X\in\fX(M,\FF)\cap C^\infty(M;\bfH)$, then $\Theta_X$ on $C^\infty(M;\Lambda\FF)$ agrees with $\Theta_X$ on $C^\infty(M;\Lambda^{0,{\bullet}})$ via~\eqref{Lambda M} (\Cref{ss: LL_X i -i}). A homomorphism $\Xi:\Lambda\FF\to \bfH^*\otimes\Lambda\FF$ can be locally defined by\index{$\Xi_X$}
\[
\Xi_X=(-1)^{(n''-v)v}[\Theta_X,\star_\FF]\star_\FF
\]
on $C^\infty(M;\Lambda^v\FF)$, for any $X\in C^\infty(M;\bfH)$, where $\star_\FF$ is defined with any choice of local orientation of $\FF$. Using~\eqref{Lambda M}, its tensor product with the identity on $\Lambda\bfH$ is a homomorphism $\Xi:\Lambda M\to \bfH^*\otimes\Lambda M$. Using the notation of \Cref{ss: local descriptions} on any normal foliated chart $(U,x)$, the local expression
\[
K={dx^{\prime i}\wedge}\,\Xi_{\bfH\partial'_i}
\]
defines an endomorphism of $\Lambda M$. A computation using~\eqref{delta_FF = (-1)^n''v+n''+1 star_FF d_FF star_FF},~\eqref{Theta_X d_0 1 = d_0 1 Theta_X},~\eqref{delta_0 -1(f_IJ dx''^I wedge dx'^J)} and~\eqref{d_1 0(f_IJ dx''^I wedge dx'^J)} gives
\begin{equation}\label{d_1 0 delta_0 -1 + delta_0 -1 d_1 0}
d_{1,0}\delta_{0,-1}+\delta_{0,-1}d_{1,0}=K\delta_{0,-1}+\delta_{0,-1}K\;,
\end{equation}
yielding~\eqref{D_perp delta_0 -1 + delta_0 -1 D_perp} by the analog of~\eqref{d_0,1^2=...=0} for the operators $\delta_{-i,i-1}$.

\subsection{Foliated maps of bounded geometry}
\label{ss: foliated maps of bd geom}

For $a=1,2$, let $\FF_a$ be a Riemannian foliation of bounded geometry on a manifold $M_a$ with a bundle-like metric. To refer to each $\FF_a$, the subscript ``$a$'' is added to the notation used in \Cref{ss: transverse structures of bd geometry}: $n'_a$, $n''_a$, $y_{a,p}:V_{a,p}\to B_a$, $x_{a,p}:U_{a,p}\to B'_a\times B''_a$, $r_{a,0}$, $r'_{a,0}$ and $r''_{a,0}$, $V_{a,p,r}$ and $U_{a,p,r',r''}$. Like in the case of uniform spaces and differential operators, in the definition of bounded geometry for maps $M_1\to M_2$, we can replace the charts  $(V_{1,p},y_{1,p})$ and $(V_{2,\phi(p)},y_{2,\phi(p)})$, and sets $B_1(p,r)$ with the charts $(U_{1,p},x_{1,p})$ and $(U_{2,\phi(p)},x_{2,\phi(p)})$, and sets $U_{1,p,r',r''}$. Let $\Cinftyub(M_1,\FF_1;M_2,\FF_2)$ be the subset of $C^\infty(M_1,\FF_1;M_2,\FF_2)$ consisting of foliated maps of bounded geometry. For any $m\in\N_0$ and $\phi\in \Cinftyub(M_1,\FF_1;M_2,\FF_2)$, using the versions of $\|\cdot\|'_{C^m_{\text{\rm ub}}}$ and $\langle\cdot,\cdot\rangle'_m$ defined with the foliated charts $(U_p,x_p)$ in the case where $m<\infty$ (\Cref{ss: transverse structures of bd geometry}), we get the following versions of~\eqref{phi^* on C_ub^m} and~\eqref{phi^* on H^m} \cite[Section~8]{AlvKordyLeichtnam2020}:~\eqref{phi^*} induces continuous homomorphisms,
\begin{equation}\label{phi^*: C_ub^m -> C_ub^m - foliated}
\phi^*:C_{\text{\rm ub}}^m(M_2;\Lambda\FF_2\otimes\Lambda N\FF_2)\to 
C_{\text{\rm ub}}^m(M_1;\Lambda\FF_1\otimes\Lambda N\FF_1)\;,
\end{equation}
and, if $\phi$ is uniformly metrically proper,     
\begin{equation}\label{phi^*: H^m -> H^m - foliated}
\phi^*:H^m(M_2;\Lambda\FF_2\otimes\Lambda N\FF_2)\to H^m(M_1;\Lambda\FF_1\otimes\Lambda N\FF_2)\;.
\end{equation}
In particular, we get~\eqref{phi^*: H^m -> H^m - foliated} if $\phi$ is a foliated diffeomorphism with $\phi^{\pm1}$ of bounded geometry. In this case, it can be continuously extended to Sobolev spaces of order $-m$ using the version of the second equality of~\eqref{H^-s(M E) = Psi^s(M E) cdot L^2(M E) = H^s(M E^* otimes Omega)'} for open manifolds,~\eqref{Lambda M} and~\eqref{(Lambda^u v M)^* otimes Omega M equiv Lambda^n'-u n''-v M}, like in \Cref{ss: maps of bd geom}.

\subsection{Leafwise functional calculus}\label{ss: Leafwise functional calculus}

Consider the notation of \Cref{ss: transverse structures of bd geometry,ss: opers of bd geom on diff forms}. Like in~\eqref{wave} and~\eqref{unit propagation speed}, the hyperbolic equation
\begin{equation}\label{leafwise wave}
\partial_t\alpha_t=iD_0\alpha_t\;,\quad\alpha_0=\alpha\;,
\end{equation}
has a unique solution on any open subset of $M$ and for $t$ in any interval containing zero, which satisfies \cite[Theorem~1.3]{Chernoff1973}, \cite[Proposition~1.2]{Roe1987}
\begin{equation}\label{leafwise unit propagation speed}
\supp\alpha_t\subset\Pen_\FF(\supp\alpha,|t|)\;.
\end{equation}

The operators $D_0$ and $\Delta_0$, with domain $\Cinftyc(M;\Lambda)$, are essentially self-adjoint in $L^2(M;\Lambda)$ \cite[Theorem~2.2]{Chernoff1973}, and their self-adjoint extensions are also denoted by $D_0$ and $\Delta_0$. The functional calculus of $D_0$, given by the spectral theorem, assigns a (bounded) operator $\psi(D_0)$ to every (bounded) measurable function $\psi$ on $\R$; in particular, we have a unitary operator $e^{itD_0}$ and a bounded self-adjoint operator $e^{-t\Delta_0}$ on $L^2(M;\Lambda)$. The notation $\Pi_0=e^{-\infty\Delta_0}$ is used for the orthogonal projection of $L^2(M;\Lambda)$ to $\ker D_0=\ker\Delta_0$ in $L^2(M;\Lambda)$.

If $\alpha\in\Cinftyc(M;\Lambda)$, the solution of~\eqref{leafwise wave} is given by $\alpha_t=e^{itD_0}\alpha$. For every $m\in\N_0$, there is some $C_m\ge0$ such that \cite[Section~IV.2]{Taylor1981}, \cite[Proposition~1.4]{Roe1987}, \cite[Proposition~7.1]{AlvKordyLeichtnam2020}
\begin{equation}\label{|e^itD_0 alpha|_m}
\|e^{itD_0}\alpha\|_m\le e^{C_m|t|}\|\alpha\|_m\;,
\end{equation}
for all $\alpha\in\Cinftyc(M;\Lambda)$.

On the other hand, like in~\eqref{psi(D_z)}, for $\psi\in\SS$, we get
\begin{equation}\label{psi(D_0)}
\psi(D_0)=\frac{1}{2\pi}\int_{-\infty}^{+\infty}\hat\psi(\xi)e^{i\xi D_0}\,d\xi\;.
\end{equation}
Taking $\psi\in\AA$ (\Cref{ss: Witten - mfds of bd geom}), it follows from~\eqref{|hat psi(xi)| le A_c e^-c |xi|},~\eqref{|e^itD_0 alpha|_m} and~\eqref{psi(D_0)} that, for every $m\in\Z\cup\{\pm\infty\}$, the functional calculus $\psi\mapsto\psi(D_0)$ restricts to a continuous homomorphisms of $\C[z]$-modules and algebras \cite[Proposition~4.1]{Roe1987}, \cite[Proposition~7.2]{AlvKordyLeichtnam2020},
\begin{equation}\label{functional calculus}
\AA\to\End(H^m(M;\Lambda))\;,\quad\AA\to\End(H^\infty(M;\Lambda))\;.
\end{equation}
By taking coefficients in $o(M)$ and transposition (see \Cref{ss: bigrading of currents}), $\psi\mapsto\psi(D_0)$ also induces continuous homomorphisms of $\C[z]$-modules and algebras,
\begin{equation}\label{functional calculus on currents}
\AA\to\End(H^{-m}(M;\Lambda))\;,\quad\AA\to\End(H^{-\infty}(M;\Lambda))\;.
\end{equation}

\subsection{Leafwise Hodge decomposition}\label{ss: Leafwise Hodge}

According to~\eqref{functional calculus}, the operator $e^{-t\Delta_0}$ ($t>0$) restricts to a continuous endomorphism of $H^\infty(M;\Lambda)$. As pointed out in \cite{Sanguiao2008}, using the bounded geometry and uniform leafwise/transverse ellipticity of the operators considered in \Cref{ss: opers of bd geom on diff forms}, and applying~\eqref{D_perp delta_0 -1 + delta_0 -1 D_perp} and~\eqref{functional calculus}, the arguments of \cite{AlvKordy2001} can be adapted to show the following, where $\Delta_0$ is considered on $H^\infty(M;\Lambda)$ \cite{Sanguiao2008}, \cite[Theorem~7.3 and Corollary~7.4]{AlvKordyLeichtnam2020}: there is a TVS-direct-sum decomposition, 
\begin{equation}\label{leafwise Hodge decomposition}
H^\infty(M;\Lambda)=\ker\Delta_0\oplus\overline{\im d_{0,1}}\oplus\overline{\im\delta_{0,-1}}\;,
\end{equation}
whose terms are orthogonal in $L^2(M;\Lambda)$; the map
\begin{equation}\label{leafwise heat flow}
[0,\infty]\times H^\infty(M;\Lambda)\to H^\infty(M;\Lambda)\;,\quad(t,\alpha)\mapsto e^{-t\Delta_0}\alpha\;,
\end{equation}
is well-defined and continuous; and $\Pi_0:H^\infty(M;\Lambda)\to\ker\Delta_0$ induces a TVS-isomorphism
\begin{equation}\label{leafwise Hodge iso}
\bar H(H^\infty(M;\Lambda),d_{0,1}) \xrightarrow{\cong} \ker\Delta_0\;,
\end{equation}
whose inverse is induced by $\ker\Delta_0\hookrightarrow H^\infty(M;\Lambda)$. The analogs of~\eqref{leafwise Hodge decomposition}--\eqref{leafwise Hodge iso} with $H^{-\infty}(M;\Lambda)$ are also true.

By~\eqref{Lambda M} and~\eqref{d_0 1 equiv d_FF}, we can consider $(H^\infty(M;\Lambda\FF),d_\FF)$ as a topological subcomplex of $(H^\infty(M;\Lambda),d_{0,1})$, and the notation $H^\bullet H^\infty(\FF)$ and $\bar H^\bullet H^\infty(\FF)$ is used for its cohomology and reduced cohomology. By~\eqref{delta_0 -1(f_IJ dx''^I wedge dx'^J)}, $\delta_\FF$ on $H^\infty(M;\Lambda\FF)$ is also given by $\delta_{0,-1}$. Thus we get the operators $D_\FF=d_\FF+\delta_\FF$ and $\Delta_\FF=D_\FF^2=\delta_\FF d_\FF+d_\FF\delta_\FF$ on $H^\infty(M;\Lambda\FF)$, which are essentially self-adjoint in $L^2(M;\Lambda\FF)$. Let $\Pi_\FF=e^{-\infty\Delta_\FF}$ be the orthogonal projection to $\ker D_\FF=\ker\Delta_\FF$ in $L^2(M;\Lambda\FF)$. Then~\eqref{|e^itD_0 alpha|_m}--\eqref{leafwise Hodge iso} have obvious versions with $D_\FF$, $\Delta_\FF$, $\bar H^\bullet H^\infty(\FF)$ and $\Pi_\FF$ \cite{Sanguiao2008}, \cite[Section~7]{AlvKordyLeichtnam2020}.

\subsection{A class of smoothing operators}\label{ss: smoothing operators}

Suppose $\FF$ is of codimension one for the sake of simplicity. (The case of codimension${}>1$ can be treated like in \cite{AlvKordy2008a}.) Assume also that $M$ is endowed with a bundle-like metric $g$ so that $\FF$ is of bounded geometry. Let $\phi:M\times\R\to M$ be a foliated flow of $\R$-local bounded geometry, whose infinitesimal generator is $Z\in\fXub(M,\FF)$ (\Cref{ss: families of bd geom}). Assume $\inf_M|\overline Z|>0$; in particular, the orbits of $\phi$ are transverse to the leaves. Given $f\in\Cinftyc(\R)$, consider the following operators on $H^{-\infty}(M;\Lambda\FF)$. For every $\psi\in\AA$, the operator
\begin{equation}\label{P with f}
P=\int_\R\phi^{t*}\,f(t)\,dt\,\psi(D_\FF)
\end{equation}
is defined by the version of~\eqref{functional calculus on currents} for $D_\FF$ and the version of~\eqref{phi^*: H^m -> H^m - foliated} for $\phi^{t*}$ on $H^{-\infty}(M;\Lambda\FF)$. The subscripts ``$\psi$'' or ``$f$'' may be added to the notation of $P$ if needed, or the subscript ``$u$'' in the case of functions $\psi_u\in\AA$ depending on a parameter $u$. For example, we may take $\psi_u(x)=e^{-ux^2}$ and the corresponding operators $P_u$\index{$P_u$} ($u>0$) on $H^{-\infty}(M;\Lambda\FF)$. Let also\index{$P_\infty$}
\[
P_\infty=\int_\R\phi^{t*}\,f(t)\,dt\,\Pi_\FF\;.
\]
The following properties hold \cite[Propositions~9.1,~9.4 and~9.6 and Corollaries~9.2,~9.3 and~9.5]{AlvKordyLeichtnam2020}: every $P_{\psi,f}$, given by~\eqref{P with f}, is smoothing, obtaining continuous bilinear maps\footnote{In \cite[Propositions~9.1 and Corollary~9.2]{AlvKordyLeichtnam2020}, only the continuous dependence on $\psi\in\AA$ is indicated, but the additional continuous dependence on $f\in\Cinftyc(\R)$ is given by \cite[Proposition~9.6]{AlvKordyLeichtnam2020}, indicated in~\eqref{|P_psi|_m m'}.}
\begin{equation}\label{P_psi}
\left\{
\begin{alignedat}{2}
\AA\times\Cinftyc(\R)&\to L(\textstyle{H^{-\infty}(M;\Lambda\FF),H^\infty(M;\Lambda\FF)})\;,&\quad(\psi,f)&\mapsto P_{\psi,f}\;,\\
\AA\times\Cinftyc(\R)&\to\Cinftyub(M^2;\Lambda\FF\boxtimes(\Lambda\FF^*\otimes\Omega M))\;,&\quad(\psi,f)&\mapsto K_{P_{\psi,f}}\;;
\end{alignedat}
\right.
\end{equation}
$P_\infty$ is smoothing, with
\begin{equation}\label{P_infty}
\left\{
\begin{alignedat}{2}
\lim_{u\to\infty}P_u&=P_\infty&\quad&\text{in}\quad L(H^{-\infty}(M;\Lambda\FF),H^\infty(M;\Lambda\FF))\;,\\
\lim_{u\to\infty}K_{P_u}&=K_{P_\infty}&\quad&\text{in}\quad
\Cinftyub(M^2;\Lambda\FF\boxtimes(\Lambda\FF^*\otimes\Omega M))\;;
\end{alignedat}
\right.
\end{equation}
and, for any compact $I\subset\R$ containing $\supp f$ and $m,m'\in\N_0$ \rnote{$m\le m'$ in $\N_0$?}, there are some $C,C'>0$ and $N\in\N_0$, depending on $m$, $m'$ and $I$, such that
\begin{equation}\label{|P_psi|_m m'}
\|P_{\psi,f}\|_{m,m'} \le C'\|\psi\|_{\AA,C,N}\,\|f\|_{I,C^N}\;.
\end{equation}

\subsection{Description of some Schwartz kernels}\label{ss: Schwartz kernels}

In \Cref{ss: smoothing operators}, $\overline Z$ defines the structure of a transversely complete $\R$-Lie foliation on $\FF$, and therefore we can consider also the notation of \Cref{ss: complete R-Lie folns}. Then the lift $\tilde g$ of $g$ to $\widetilde M$ is a bundle-like metric of $\widetilde\FF=\pi^*\FF$, and $\widetilde Z\in\fXub(\widetilde M,\widetilde\FF)$. Assume $D_*\widetilde Z=\partial_x\in\fX(\R)$ and $\bar\phi^t(x)=t+x$; hence $\phi^t$ preserves every leaf of $\FF$ if and only if $t\in\Hol\FF$ (\Cref{ss: complete R-Lie folns}).

For any $\psi\in\AA$ and $f\in\Cinftyc(\R)$, we have the smoothing operator $P$ given by~\eqref{P with f}, and a similar smoothing operator $\widetilde P$ is defined by using $\tilde\phi$ and $\widetilde\FF$ instead of $\phi$ and $\FF$. We are going to describe their Schwartz kernels.

Let $\fG=\Hol(M,\FF)$ and $\widetilde\fG=\Hol(\widetilde M,\widetilde\FF)$, whose source and range maps are denoted by $\bfs,\bfr:\fG\to M$ and $\tilde\bfs,\tilde\bfr:\widetilde\fG\to\widetilde M$ (\Cref{ss: holonomy groupoid}). Since the leaves of $\FF$ and $\widetilde\FF$ have trivial holonomy groups, the smooth immersions $(\bfr,\bfs):\fG\to M^2$ and $(\tilde\bfr,\tilde\bfs):\widetilde\fG\to\widetilde M^2$ are injective, with images $\RR_\FF$ and $\RR_{\widetilde\FF}$. Via these injections, the restriction $\pi\times\pi:\RR_{\widetilde\FF}\to\RR_\FF$ corresponds to the Lie groupoid homomorphism $\pi_\fG:=\Hol(\pi):\widetilde\fG\to\fG$ (\Cref{ss: fol maps}), which is a covering map with $\Aut(\pi_\fG)\equiv\Gamma$. In fact, since $\widetilde\FF$ is defined by the fiber bundle $D$, we get that $\RR_{\widetilde\FF}$ is a regular submanifold of $\widetilde M^2$, and $(\tilde\bfr,\tilde\bfs):\widetilde\fG\to\RR_{\widetilde\FF}$ is a diffeomorphism. We may write $\fG\equiv\RR_\FF$ and $\widetilde\fG\equiv\RR_{\widetilde\FF}$.

Consider the $C^\infty$ vector bundles, $S=\bfr^*\Lambda\FF\otimes\bfs^*(\Lambda\FF\otimes\Omega\FF)$ over $\fG$ and $\widetilde S={\tilde\bfr^*\Lambda\widetilde\FF}\otimes{\tilde\bfs^*(\Lambda\widetilde\FF\otimes\Omega\widetilde\FF)}$ over $\widetilde\fG$. Note that $\widetilde S\equiv\pi_\fG^*S$, and any $k\in C^\infty(\fG;S)$ lifts via $\pi_\fG$ to a section $\tilde k\in C^\infty(\widetilde{\fG};\widetilde S)$. Since $\pi$ restricts to diffeomorphisms of the leaves of $\widetilde\FF$ to the leaves of $\FF$, it follows that $\tilde k\in C^\infty_{\text{\rm p}}(\widetilde{\fG};\widetilde S)$ if and only if $k\in C^\infty_{\text{\rm p}}(\fG;S)$.

For any $\psi\in\RR$, the collection of Schwartz kernels $k_L:=K_{\psi(D_L)}$, for all leaves $L$ of $\FF$, defines a section $k=k_\psi$ of $S$. This also applies to the operators $\psi(D_{\widetilde L})$ on the leaves $\widetilde L$ of $\widetilde\FF$, obtaining a section $\tilde k=\tilde k_\psi$ of $\widetilde S$.

If $\hat\psi\in\Cinftyc(\R)$, then $k_\psi\in C^\infty_{\text{\rm p}}(\fG;S)$, and the global action of $k_\psi$ on $\Cinftyc(M;\Lambda\FF)$ (\Cref{ss: global action}) agrees with the restriction of the operator $\psi(D_\FF)$ on $H^\infty(M;\Lambda\FF)$ defined by the version of~\eqref{functional calculus} for $D_\FF$ \cite[Proposition~10.1]{AlvKordyLeichtnam2020}. Precisely, if $\supp\hat\psi\subset[-R,R]$ for some $R>0$, then $\supp k_\psi\subset\overline{\Pen}_\FF(\fG^{(0)},R)$ by~\eqref{supp hat psi subset [-R R] => supp K_psi(D_z) subset ...}, and therefore $\supp\psi(D_\FF)\alpha\subset\overline{\Pen}_\FF(\supp\alpha,R)$ for all $\alpha\in H^\infty(M;\Lambda\FF)$ by Remark~\ref{r: supp K_A subset r-penumbra <=> supp Au subset r-penumbra for all u}.

Let $\widetilde\Lambda=D^*dx\equiv dx$, which is an invariant transverse volume form of $\widetilde\FF$ defining the same transverse orientation as $\overline{\widetilde Z}$. Since $\widetilde\Lambda$ is $\Gamma$-invariant by the $h$-equivariance of $D$, it defines a transverse volume form $\Lambda$ of $\FF$, which defines the same transverse orientation as $\overline Z$. These $\widetilde\Lambda$ and $\Lambda$ define invariant transverse densities $|\widetilde\Lambda|$ and $|\Lambda|$ of $\widetilde\FF$ and $\FF$.

Let $\tilde p,\tilde q\in\widetilde M$ over $p,q\in M$, and write $t_{\tilde p,\tilde q}=D(\tilde q)-D(\tilde p)$. If $\psi\in\AA$, then\footnote{There is an error in the statement of \cite[Proposition~10.3]{AlvKordyLeichtnam2020}: it is written $f(t_{\tilde p,\tilde q})$ instead of $f(t_{\tilde p,\tilde q}-h(\gamma))$. However, its proof shows the expression given in~\eqref{Schwartz kernel}.}
\begin{equation}\label{Schwartz kernel}
K_P(p,q)
\equiv\sum_{\gamma\in\Gamma}
T_\gamma^*\,\tilde\phi^{t_{\tilde p,\tilde q}-h(\gamma)*}
\tilde k\big(T_\gamma\tilde\phi^{t_{\tilde p,\tilde q}-h(\gamma)}(\tilde p),\tilde q\big)\,f(t_{\tilde p,\tilde q}-h(\gamma))\,|\Lambda|(q)\;,
\end{equation}
defining a convergent series in $\Cinftyub(\widetilde M^2;\widetilde S)$ \cite[Proposition~10.3]{AlvKordyLeichtnam2020}. Here, the identity $\widetilde S_{(\tilde p,\tilde q)}\equiv S_{(p,q)}$ is used, and the leafwise part of the density of $K_P(\cdot,q)$ at $q$ is given by the density of $\tilde k(\cdot,\tilde q)$ at $\tilde q$.

\section{Witten's operators on Riemannian foliations of bounded geometry}
\label{s: Witten's opers on Riem folns of bd geom}

Consider the notation of \Cref{ss: perturbation vs bigrading} with our assumption that $\FF$ is Riemannian of bounded geometry. Suppose also that $\eta\in\Cinftyub(M;\Lambda^1)$, and therefore $\eta_0\in\Cinftyub(M;\Lambda^{0,1})\equiv \Cinftyub(M;\Lambda^1\FF)$ and $\eta_1\in\Cinftyub(M;\Lambda^{1,0})$. Thus the operators $d_{z,i,1-i}$, $\delta_{z,i,i-1}$, $D_{z,0}$, $D_{z,\perp}$ and $\Delta_{z,0}$ are of bounded geometry. Arguing like in~\eqref{D_perp delta_0 -1 + delta_0 -1 D_perp}, we get
\begin{multline*}
(d_{1,0}\,{\eta_0\lrcorner}+{\eta_0\lrcorner}\,d_{1,0})(f_{IJ}\,dx^{\prime\prime I}\wedge dx^{\prime J})\\
\begin{aligned}
&=(K\,{\eta_0\lrcorner}+{\eta_0\lrcorner}\,K)(f_{IJ}\,dx^{\prime\prime I}\wedge dx^{\prime J})\\
&\phantom{={}}{}+(-1)^{(n''+1)|I|+n''}\big(\star_\FF[\Theta_{\bfH\partial'_i},{\eta_0\wedge}]\star_\FF(f_{IJ}\,dx^{\prime\prime I})\big)
\wedge dx^{\prime i}\wedge dx^{\prime\prime J}\;.
\end{aligned}
\end{multline*}
Using~\eqref{delta_0 -1(f_IJ dx''^I wedge dx'^J)} and~\eqref{[Theta_bfH eta_0 wedge]}, it follows that
\begin{multline*}
(d_{1,0}\,{\eta_0\lrcorner}+{\eta_0\lrcorner}\,d_{1,0}-K\,{\eta_0\lrcorner}-{\eta_0\lrcorner}\,K)
(f_{IJ}\,dx^{\prime\prime I}\wedge dx^{\prime J})\\
\begin{aligned}
&=(-1)^{(n''+1)|I|+n''}\big(\star_\FF[d_\FF,h_i]\star_\FF(f_{IJ}\,dx^{\prime\prime I})\big)
\wedge dx^{\prime i}\wedge dx^{\prime\prime J}\\
&=-(-1)^{|I|}\big(\delta_\FF(h_if_{IJ}\,dx^{\prime\prime I})\big)
\wedge dx^{\prime i}\wedge dx^{\prime\prime J}\\
&\phantom{={}}{}+(-1)^{|I|}\big(\delta_\FF(f_{IJ}\,dx^{\prime\prime I})\big)
\wedge \eta_1\wedge dx^{\prime\prime J}\\
&=-\delta_{0,-1}\big(\eta_1\wedge f_{IJ}\,dx^{\prime\prime I}\wedge dx^{\prime\prime J}\big)
-\eta_1\wedge\delta_{0,-1}\big(f_{IJ}\,dx^{\prime\prime I}\wedge dx^{\prime\prime J}\big)\;.
\end{aligned}
\end{multline*}
This shows that
\begin{equation}\label{d_1,0 eta_0 lrcorner + eta_0 lrcorner d_1 0}
d_{1,0}\,{\eta_0\lrcorner}+{\eta_0\lrcorner}\,d_{1,0}=K\,{\eta_0\lrcorner}+{\eta_0\lrcorner}\,K
-\delta_{0,-1}\,{\eta_1\wedge}-{\eta_1\wedge}\,\delta_{0,-1}\;.
\end{equation}
Combining~\eqref{d_1 0 delta_0 -1 + delta_0 -1 d_1 0} and~\eqref{d_1,0 eta_0 lrcorner + eta_0 lrcorner d_1 0}, and using that ${\eta_0\lrcorner}$ is an anti-derivation, we compute
\begin{multline*}
d_{z,1,0}\delta_{z,0,-1}+\delta_{z,0,-1}d_{z,1,0}\\
\begin{aligned}
&=K\delta_{0,-1}+\delta_{0,-1}K-\bar z(K\,{\eta_0\lrcorner}+{\eta_0\lrcorner}\,K
-\delta_{0,-1}\,{\eta_1\wedge}-{\eta_1\wedge}\,\delta_{0,-1})\\
&\phantom{={}}{}+z({\eta_1\wedge}\,\delta_{0,-1}+\delta_{0,-1}\,{\eta_1\wedge})
+|z|^2({\eta_1\wedge}\,{\eta_0\lrcorner}+{\eta_0\lrcorner}\,{\eta_1\wedge})\\
&=K\delta_{z,0,-1}+\delta_{z,0,-1}K+2\Re z\,({\eta_1\wedge}\,\delta_{0,-1}+\delta_{0,-1}\,{\eta_1\wedge})\\
&\phantom{={}}{}-2\Re z\,\bar z({\eta_1\wedge}\,{\eta_0\lrcorner}+{\eta_0\lrcorner}\,{\eta_1\wedge})\\
&=K_z\delta_{z,0,-1}+\delta_{z,0,-1}K_z\;,
\end{aligned}
\end{multline*}
where $K_z=K+2\Re z\,{\eta_1\wedge}$ is an endomorphism of $\Lambda M$ of bounded geometry. Using also the analog of~\eqref{d_0,1^2=...=0} for the operators $d_{z,i,1-i}$, it follows that\footnote{The equality
\[
d_{-\bar z,1,0}\delta_{z,0,-1}+\delta_{z,0,-1}d_{-\bar z,1,0}=K\delta_{z,0,-1}+\delta_{z,0,-1}K
\]
is also true, but this does not fit the analog of~\eqref{d_0,1^2=...=0}.}
\begin{equation}\label{D_z perp delta_z 0 -1 + delta_z 0 -1 D_z perp}
D_{z,\perp}\delta_{z,0,-1}+\delta_{z,0,-1}D_{z,\perp}=K_z\delta_{z,0,-1}+\delta_{z,0,-1}K_z\;.
\end{equation}
Using this key equality, we get straightforward generalizations of all results in \Cref{ss: Leafwise functional calculus,ss: Leafwise Hodge} for $d_{z,0,1}$, $\delta_{z,0,-1}$, $D_{z,0}$ and $\Delta_{z,0}$, which also have obvious versions for $d_{\FF,z}$, $\delta_{\FF,z}$, $D_{\FF,z}$ and $\Delta_{\FF,z}$. Let $\Pi_{0,z}$ and $\Pi_{\FF,z}$ denote the corresponding versions of $\Pi_0$ and $\Pi_\FF$.

Let $\phi:(M,\FF)\to(M,\FF)$ be a smooth foliated map of bounded geometry. Since $\eta\in \Cinftyub(M;\Lambda)$, we get versions of the continuity of~\eqref{phi^*: C_ub^m -> C_ub^m - foliated} and~\eqref{phi^*: H^m -> H^m - foliated} for $\phi^*_z$, assuming $\phi$ is uniformly metrically proper for the second one (\Cref{ss: maps of bd geom}). In particular, this applies to any foliated flow of $\R$-local bounded geometry (\Cref{ss: families of bd geom}), $\phi=\{\phi^t\}$ on $(M,\FF)$, using its unique lift $\tilde\phi=\{\tilde\phi^t\}$ to $\widetilde M$. Then the definitions and results of \Cref{ss: smoothing operators,ss: Schwartz kernels} have obvious twisted extensions using $\phi^{t*}_z$, $D_{\FF,z}$ and $\Pi_{\FF,z}$. The subscript ``$z$'' may be added to the notation $P$, $P_u$, $P_\infty$, $k$, $\tilde k$, $k_u$ and $\tilde k_u$ in this setting.

Recall that $D_{0,z,z'}$ and $\Delta_{0,z,z'}$ are uniformly leafwise elliptic with a symmetric leading symbol (\Cref{ss: leafwise 2 parameters}). Moreover, they are of bounded geometry. Then the obvious version of~\eqref{leafwise wave} with $D_{0,z,z'}$ has a unique solution, which satisfies the obvious analogs of~\eqref{leafwise unit propagation speed} and~\eqref{|e^itD_0 alpha|_m}. Thus $\psi(D_{0,z,z'})$ ($\psi\in\SS$) can be defined by the analog of~\eqref{psi(D_0)}, obtaining corresponding analogs of~\eqref{functional calculus} and~\eqref{functional calculus on currents}. Then, using $\phi^{t*}_z$ and $D_{\FF,z,z'}$, we get obvious extensions of the definitions and results of \Cref{ss: smoothing operators,ss: Schwartz kernels}, except for the statements involving $\Pi_\infty$ and $P_\infty$. The double subscript ``$z,z'$'' may be added to the notation $P$, $P_u$, $k$, $\tilde k$, $k_u$ and $\tilde k_u$ in this setting. However, if $z\ne z'$ and $\eta\ne0$,  $D_{0,z,z'}$ and $\Delta_{0,z,z'}$ are not symmetric, and therefore the results of \Cref{ss: Leafwise Hodge} cannot be generalized for these operators.

\chapter{Foliations with simple foliated flows}\label{ch: fols with simple fol flows}

\section{Simple foliated flows}\label{s: simple fol flows}

\subsection{Simple flows}\label{ss: simple flows}

Let $\phi:\Omega\to M$ be a smooth local flow, where $\Omega$ is an open neighborhood of $M\times\{0\}$ in $M\times\R$. Let $Z\in\fX(M)$ be the infinitesimal generator. For $p\in M$ and $t\in\R$, let
\[
\Omega_p=\{\,t\in\R\mid(p,t)\in\Omega\,\}\;,\quad\Omega^t=\{\,q\in M\mid(q,t)\in\Omega\,\}\;,
\]
and let $\phi^t=\phi(\cdot,t):\Omega^t\to M$. The fixed point set is
\[
\Fix(\phi)=\{\,p\in M\mid p\in\Fix(\phi^t)\ \forall t\in\Omega_p\ \text{close enough to $0$}\,\}\;,
\]
which equals the zero set of $Z$. Recall that a fixed point $p$ of $\phi$ is called \emph{simple} (or \emph{transverse}) if it is a simple fixed point of $\phi^t$ for all $t\ne0$ close enough to $0$ in $\Omega_p$ (see \Cref{ss: local Lefschetz formulae}). In this case, the associated number $\epsilon_p(\phi^t)$, defined in~\eqref{epsilon_p(phi)}, is independent of $t>0$ close enough to $0$ in $\Omega_p$, and is denoted by $\epsilon_p=\epsilon_p(\phi)$. If the fixed points of $\phi$ are simple, then $\Fix(\phi)$ is a discrete subset of $M$. For a fixed point $p$, we can write $\phi^t_*=e^{tA}$ on $T_pM$ for some endomorphism $A$ of $T_pM$. Then $p$ is simple just when $A$ is an automorphism.

Now, assume $Z$ is complete, and therefore we can take $\Omega=M\times\R$. On $M\setminus\Fix(\phi)$, let $N\phi$ denote the normal bundle of the foliation defined by the orbits of $\phi$; i.e., $N_p\phi=T_pM/\R Z(p)$ for every $p\in M\setminus\Fix(\phi)$. Let $\CC=\CC(\phi)$\index{$\CC(\phi)$} denote the set of closed orbits of $\phi$ (without including fixed points). For any $c\in\CC$, let $\ell(c)$\index{$\ell(c)$} denote its minimum positive period. For every subset $I\subset\R$, let \index{$\CC_I(\phi)$}
\[
\CC_I=\CC_I(\phi)=\{\,c\in\CC\mid\ell(c)\in I\,\}\;.
\]
The nonzero periods of all closed orbits form the set\index{$\PP(\phi)$}
\[
\PP=\PP(\phi)=\{\,k\ell(c)\mid c\in\CC,\ k\in\Z^\times\,\}\;.
\]
For all $c\in\CC$, $k\in\Z$ and $p\in c$, let $\phi^{k\ell(c)}_*:N_p\phi\to N_p\phi$ be the homomorphism induced by $\phi^{k\ell(c)}_*:T_pM\to T_pM$. Recall that $c$ is called \emph{simple} when the eigenvalues of $\phi^{k\ell(c)}_*:N_p\phi\to N_p\phi$ are different from $1$ for some (and therefore for all) $p\in c$ and $k\in\Z^\times$; in this case, let\index{$\epsilon_c(k)$}
\[
  \epsilon_c(k)=\epsilon_c(k,\phi)=\sign\det\big(\id-\phi^{k\ell(c)}_*:N_p\phi\to N_p\phi\big)\in\{\pm1\}\;.
\]
Every simple closed orbit $c$, there are neighborhoods, $V$ where $c$ in $M$ and $I$ of $\ell(c)$ in $\R$, such that $c$ is the only closed orbit whose first positive period is in $I$, and moreover that $V\cap\Fix(\phi)=\emptyset$.

The flow $\phi$ is called \emph{simple}\index{simple flow} if all of its fixed points and closed orbits are simple. If moreover $M$ is closed, then $\Fix(\phi)$ is finite, and $\CC_I(\phi)$ are finite for all compact $I\subset\R$. Therefore $\PP(\phi)$ is a discrete subset of $\R$.

\subsection{Transversely simple foliated flows}\label{ss: simple fol flows}

Let $\FF$ be a transversely oriented smooth foliation of codimension one on a closed manifold $M$. We assume $M$ is closed for the sake of simplicity, but the concepts and properties recalled here also have obvious versions when $M$ is a manifold with boundary, where both $\FF$ and $\phi$ are tangent to $\partial M$. Some generalizations to non-compact manifolds will be also indicated and needed.

Let $\phi=\{\phi^t\}$ be a foliated flow on $M$ and let $Z\in\fX(M,\FF)$ be its infinitesimal generator (\Cref{ss: fol maps}). Let $M^0$\index{$M^0$} be the union of leaves preserved by $\phi$, and let $M^1=M\setminus M^0$.\index{$M^1$} The $\phi$-invariant set $M^0$ is compact because it is the zero set of $\overline{Z}\in\olfX(M,\FF)\subset C^\infty(M;N\FF)$. Therefore the $\phi$-invariant set $M^1$ is open in $M$. Moreover $\phi$ is transverse to the leaves on $M^1$. So there is a canonical isomorphism $N\phi\cong T\FF$ on $M^1$, and $\FF$ is transitive at every point of $M^1$ (\Cref{ss: transverse structures}); in particular, the leaves in $M^1$ have no holonomy. Consider the notation of \Cref{ss: holonomy,ss: infinitesmal transfs,ss: fol maps}, using the notation $(x_k,y_k)$ instead of $(x'_k,x''_k)$ because $\codim\FF=1$. Let $\bar\phi$ be the local flow on $\Sigma$ generated by $\overline Z\in\fX(\Sigma,\HH)$. Via the homeomorphism $M/\FF\to\Sigma/\HH$ induced by the coordinates $x_k:U_k\to\Sigma_k$, the leaves preserved by $\phi$ correspond to the $\HH$-orbits preserved by $\bar\phi$, which indeed form $\Fix(\bar\phi)$ because the $\HH$-orbits are totally disconnected. $\overline Z$ is $\HH$-invariant, and $\bar\phi$ is $\HH$-equivariant in an obvious sense.

Since $\dim\Sigma=1$, for all simple $\bar p\in\Fix(\bar\phi)$, there is some $\varkappa=\varkappa_{\bar p}\in\R^\times$ such that $\bar\phi^t_*\equiv e^{\varkappa t}$ on $T_{\bar p}\Sigma\equiv\R$. By the $\HH$-equivariance of $\bar\phi$, we get $\varkappa_{\bar p}=\varkappa_{\bar q}$ for all $\bar q\in\HH(\bar p)\subset\Fix(\bar\phi)$. Thus we can use the notation $\varkappa_L=\varkappa_{\bar p}$ if $\HH(\bar p)$ corresponds to a leaf $L$.

The leaves preserved by $\phi$ that correspond to simple fixed points of $\bar\phi$ are said to be \emph{transversely simple}. If all leaves preserved by $\phi$ are transversely simple, then $\phi$ {\rm(}or $Z${\rm)} is called \emph{transversely simple};\index{transversely simple foliated flow} if moreover its closed orbits are simple, then $\phi$ {\rm(}or $Z${\rm)} is said to be \emph{weakly simple}.\index{weakly simple foliated flow} If $\phi$ is weakly simple, every closed orbit is contained either in $M^0$ or in $M^1$, and its (possibly non-simple) fixed points belong to $M^0$.

Suppose $\phi$ is transversely simple unless otherwise stated. Then $M^0$ is a finite union of compact leaves because every fixed point of $\bar\phi$ is isolated. For any point $p$ in a preserved leaf $L$, there are foliated coordinates $(x,y):U\to\Sigma\times B$, where $\Sigma\subset\R$ is an open interval containing $0$, so that  \cite[Lemma~3.2]{AlvKordyLeichtnam2022}
\begin{equation}\label{overline Z = varkappa_Lx partial_x}
x(p)=0\;,\quad\overline Z=\varkappa_Lx\partial_x\;,\quad\bar\phi^t(x)=e^{\varkappa_Lt}x\;.
\end{equation}
Hence the following properties hold \cite[Propositions~3.4 and~3.5]{AlvKordyLeichtnam2022}:
\begin{enumerate}[{\rm(A)}]
\setcounter{enumi}{2}
\item\label{i: almost w/o hol} $\FF$ is almost without holonomy with finitely many leaves with holonomy.
\item\label{i: homotheties} The holonomy groups of the compact leaves are groups of germs at $0$ of homotheties on $\R$, for some choice of $\{U_k,(x_k,y_k)\}$.
\end{enumerate}

According to~\ref{i: almost w/o hol} and Remark~\ref{r: Hector's description of folns almost w/o hol}~\ref{i: adding leaves to M^0}, we can consider Hector's description with this choice of $M^0$ and $M^1$, even though there may be leaves without holonomy in $M^0$. With the notation of \Cref{ss: folns almost w/o hol}, since $\bfpi:(M_l,\FF_l)\to(M,\FF)$ is a foliated local embedding and $\bfpi:\partial M_l\to M^0$ a local diffeomorphism, any $A\in\fX(M,M^0)$ has a lift $A_l\in\fXb(M_l)$. Moreover $A_l\in\fX(M_l,\FF_l)$ if $A\in\fX(M,\FF)$. Thus any (foliated) flow $\zeta=\{\zeta^t\}$ on $(M,\FF)$ preserving $M^0$ can be lifted via $\bfpi$ to a (foliated) flow $\zeta_l=\{\zeta_l^t\}$ on $(M_l,\FF_l)$ preserving $\partial M_l$. If a foliated flow $\zeta$ on $(M,\FF)$ is weakly simple, then $\zeta_l$ is also weakly simple (on the foliated manifold with boundary $(M_l,\FF_l)$). The restrictions $A_l|_{\mathring M_l}\equiv A|_{M^1_l}$ and $\zeta_l|_{\mathring M_l}\equiv\zeta|_{M^1_l}$ are also denoted by $A_l$ and $\zeta_l$. In particular, this notation applies to $Z$ and $\phi$, obtaining $Z_l$ and $\phi_l=\{\phi_l^t\}$, which induces the structure of a complete $\R$-Lie foliation on $\FF^1_l$. According to \Cref{ss: folns almost w/o hol}, we consider the transverse orientation of every $\FF_l$ so that $\bfpi:(M_l,\FF_l)\to(M,\FF)$ is compatible with the transverse orientations. However, we will consider the transverse orientation of every $\FF^1_l$ defined by $\overline Z_l$. Now Hector's description has the following more specific cases \cite[Section~3]{AlvKordyLeichtnam2022}:
\begin{enumerate}[(a)]
\setcounter{enumi}{2}
\item\label{i: fiber bundle} $\FF$ is given by a fiber bundle $M\to S^1$ with connected fibers.

\item\label{i: minimal Lie foln} $\FF$ is an $\R$-Lie foliation with dense leaves.

\item\label{i: all folns FF_l are models (1)} $M^0\ne\emptyset$, $\Hol L\cong\Z$ for all leaves $L\subset M^0$, and the foliations $\FF^1_l$ are given by fiber bundles $M^1_l\to S^1$ with connected fibers.

\item\label{i: all folns FF_l are models (2)} $M^0\ne\emptyset$, $\Hol L$ is a finitely generated abelian group of rank $>1$ for all leaves $L\subset M^0$, and all foliations $\FF^1_l$ are minimal $\R$-Lie foliations.

\end{enumerate}
The case~\ref{i: fiber bundle} can be considered as a model~\ref{i: model 1} with empty boundary, avoiding the use of models~\ref{i: model 0}, or it can be cut into models~\ref{i: model 0} by adding a finite number of leaves without holonomy to $M^0$. Except in this case, $M^1$ is just the transitive point set of $\FF$, and $\fX(M,\FF)$ spans $\fX(M,M^0)$ as $C^\infty(M)$-module by~\eqref{overline Z = varkappa_Lx partial_x}.

For every leaf $L\subset M^0$, its holonomy homomorphism $\bfh=\bfh_L$ is induced by a homomorphism $\hat h=\hat h_L:\pi_1L\to\Diffeo^+(\R,0)$ whose image consists of homotheties; i.e., writing $\Gamma=\Gamma_L=\pi_1 L/\ker\hat h$, $\hat h$ induces a monomorphism $h=h_L:\Gamma\to\Diffeo^+(\R,0)$, $\gamma\mapsto h_\gamma$, with $h_\gamma(x)=a_\gamma x$ for some monomorphism $\Gamma\to\R^+\equiv(\R^+,\times)$, $\gamma\mapsto a_\gamma=a_{L,\gamma}$. The restriction of $\FF$ to some neighborhood of $L$ can be described as the suspension of $h$ (\Cref{ss: holonomy}); its definition for this case will be recalled in \Cref{ss: suspension - homotheties}. On the other hand, every $\FF^1_l$ has a Fedida's description, which will be better analyzed in \Cref{ss: components of M^1}.

\begin{rem}\label{r: case where M is not compact}
The concepts recalled in this subsection do not need the compactness of $M$. Only the completeness of $Z$ and compactness of $M^0$ are needed to extend the indicated notions and properties.
\end{rem}

\subsection{Existence of simple foliated flows}\label{ss: existence}

For a transversely oriented foliation $\FF$ of codimension one on a closed foliated manifold $M$, the following conditions are equivalent \cite[Propositions~6.1 and~6.3 and Theorem~6.9]{AlvKordyLeichtnam2022}:
\begin{enumerate}[{\rm(a)}]
\setcounter{enumi}{6}
\item\label{i: A-B} It satisfies~\ref{i: almost w/o hol} and~\ref{i: homotheties}.
\item\label{i: preserved leaves are simple} There is a transversely simple foliated flow.
\item\label{i: weakly simple foliated flow trivial on preserved leaves} There is a weakly simple foliated flow.
\item\label{i: simple foliated flow} There is a simple foliated flow.
\end{enumerate}
Moreover the families of foliated flows $\phi$ satisfying~\ref{i: preserved leaves are simple},~\ref{i: weakly simple foliated flow trivial on preserved leaves} or~\ref{i: simple foliated flow} induce the same family of local flows $\bar\phi$ on $\Sigma$. In the case~\ref{i: weakly simple foliated flow trivial on preserved leaves}, it can be also assumed that $M^0\subset\Fix(\phi)$ and there are no closed orbits in some neighborhood of $M^0$, obtaining the same family of local flows $\bar\phi$ on $\Sigma$.

A more precise description of the foliations satisfying these equivalent conditions is given in \cite[Theorem~6.9]{AlvKordyLeichtnam2022}, but it will not be needed here.

\section{Case of suspension foliations}\label{s: suspension}

\subsection{Suspension foliations defined with homotheties}\label{ss: suspension - homotheties}

For a pointed connected closed manifold $(L,p)$, let $\hat h:\pi_1L=\pi_1(L,p)\to\Diffeo^+(\R,0)$ be a homomorphism whose image consists of homotheties, like in \Cref{s: simple fol flows}. Therefore, writing $\Gamma=\pi_1L/\ker\hat h$, $\hat h$ induces a monomorphism $h:\Gamma\to\Diffeo^+(\R,0)$, $\gamma\mapsto h_\gamma$, where $h_\gamma(x)=a_\gamma x$ for some monomorphism $\Gamma\to\R^+$, $\gamma\mapsto a_\gamma$; in particular, $\Gamma$ is abelian, torsion free and finitely generated. Let $\pi=\pi_L:(\widetilde L,\tilde p)\to(L,p)$ be the pointed regular covering map with $\pi_1\widetilde L=\pi_1(\widetilde L,\tilde p)\equiv\ker\hat h$, and therefore $\Aut(\pi)\equiv\Gamma$. Like in \Cref{ss: flat line bundle,ss: perturbed Schwartz kernels on regular coverings of compact mfds}, the canonical left action of every $\gamma\in\Gamma$ on $\widetilde L$ is denoted by $T_\gamma$ or $\tilde y\mapsto\gamma\cdot\tilde y$, and write $[\tilde y]=\pi(\tilde y)$ for $\tilde y\in\widetilde L$. For the diagonal left $\Gamma$-action on $\widetilde M=\R\times\widetilde L$, $\gamma\cdot(x,\tilde y)=(a_\gamma x,\gamma\cdot\tilde y)$, the orbit space $M=\Gamma\backslash\widetilde M$ is called a \emph{suspension manifold}. The canonical projection $\pi_M:\widetilde M\to M$ is a $\Gamma$-cover with deck transformations $h_\gamma\times T_\gamma$ ($\gamma\in\Gamma$). Write $[x,\tilde y]=\pi_M(x,\tilde y)$ for $(x,\tilde y)\in\widetilde M$. 

Let $\widetilde\varpi:\widetilde M\to\widetilde L$ denote the second-factor projection, and let $\widetilde\FF$ be the foliation on $\widetilde M$ with leaves $\{x\}\times\widetilde L$ ($x\in\R$). Since $\widetilde\varpi$ is $\Gamma$-equivariant, we get an induced fiber bundle map $\varpi:M\to L$, defined by $\varpi([x,\tilde y])=\pi(\tilde y)$. On the other hand, since $\widetilde\FF$ and its canonical transverse orientation are $\Gamma$-invariant, we also get an induced transversely oriented foliation $\FF$ on $M$, called a \emph{suspension foliation}, which is transverse to the fibers of $\varpi$. The typical fiber of $\varpi$ is $\R$ because the corresponding fibers of $\widetilde\varpi$ and $\varpi$ can be identified via $\pi_M$. Since $0$ is fixed by the $\Gamma$-action on $\R$, the leaf $\{0\}\times\widetilde L\equiv\widetilde L$ of $\widetilde\FF$ is $\Gamma$-invariant, and $\pi_M(\{0\}\times\widetilde L)\equiv L$ is a compact leaf of $\FF$. The other leaves of $\widetilde\FF$ are diffeomorphic via $\pi_M$ to the corresponding leaves of $\FF$ because the elements of $\Gamma\setminus\{e\}$ have no fixed points in $\R^\times$. Given any $\tilde p\in\widetilde L$ with $\pi(\tilde p)=p\in L$, the fiber $\varpi^{-1}(p)\equiv\widetilde\varpi^{-1}(\tilde p)=\R\times\{\tilde p\}\equiv\R$ is a global transversal of $\FF$ through $p\equiv[0,\tilde p]$. Note that the holonomy homomorphism $\bfh:\pi_1L\to\Hol L$ is induced by $h$, and therefore $\widetilde L^{\text{\rm hol}}\equiv\widetilde L$ (\Cref{ss: holonomy}). The standard orientation of $\R$ induces a transverse orientation of $\widetilde\FF$, which is $\Gamma$-invariant because the image of $h$ consists of orientation preserving homotheties, giving rise to a transverse orientation of $\FF$. Let $\bfH\subset TM$ and $\widetilde\bfH\subset T\widetilde M$ be the linear subbundles of vectors tangent to the fibers of $\varpi$ and $\widetilde\varpi$, which induce bigradings of $\Lambda M$ and $\Lambda\widetilde M$ satisfying $d_{2,-1}=0$ (\Cref{ss: bigrading,ss: d_i j}).  For $\tilde p\in\widetilde M$ and $p\in M$, we will use the identities $\Lambda_{\tilde p}\widetilde\FF\equiv\Lambda_{\widetilde\varpi(\tilde p)}\widetilde L$ and $\Lambda_p\FF\equiv\Lambda_{\varpi(p)}L$ induced by $\widetilde\varpi$ and $\varpi$.

\subsection{Transversely simple flows on suspension foliations}\label{ss: transv simple flows - suspension}

Let $\phi=\{\phi^t\}$ be any transversely simple foliated flow on $M$ and let $Z\in\fXcom(M,\FF)$ be its infinitesimal generator. Let us recall the notation of \Cref{ss: simple fol flows} in this case (see also \Cref{ss: folns almost w/o hol}). Without loss of generality, we can assume $M^0=\pi_M(\{0\}\times\widetilde L)\equiv L$ for the description around a compact leaf. By~\eqref{overline Z = varkappa_Lx partial_x}, we can suppose the lifts of $\phi$ and $Z$ to $\widetilde M$, denoted by $\tilde\phi$ and $\widetilde Z$, are of the form
\begin{equation}\label{tilde phi^t(tilde y x) widetilde Z}
\tilde\phi^t(x,\tilde y)=(e^{\varkappa t}x,\tilde\phi_x^t(\tilde y))\;,\quad\widetilde Z=(\varkappa x\partial_x,\widetilde Z_x)\;,
\end{equation}
for some $\varkappa\in\R^\times$, and smooth families, $\{\,\widetilde\phi^t_x\mid x,t\in\R\,\}\subset\Diffeo(\widetilde L)$ and $\{\,\widetilde Z_x\mid x\in\R\,\}\subset\fX(\widetilde L)$, with $\tilde\phi_x^0=\id_{\widetilde L}$. In particular, $\widetilde Z_0$ and $\tilde\phi_0^t$ are the restrictions of $\widetilde Z$ and $\tilde\phi^t$ to $\widetilde L\equiv\{0\}\times\widetilde L$. Thus $\widetilde Z_0$ is $\Gamma$-invariant and $\tilde\phi_0=\{\tilde\phi_0^t\}$ is $\Gamma$-equivariant, inducing the restrictions of $Z$ and $\phi^t$ to $L$, denoted by $Z_0$ and $\phi_0=\{\phi_0^t\}$; we may also use the notation $Z_L=Z_0$ and $\phi_L=\{\phi_L^t\}=\{\phi_0^t\}$. The $\Gamma$-equivariance of $\tilde\phi^t$ and the $\Gamma$-invariance of $\widetilde Z$ mean that, for all $\gamma\in\Gamma$ and $x,t\in\R$,
\begin{equation}\label{T_gamma tilde phi_x^t = tilde phi_a_gamma x^t T_gamma}
T_\gamma\tilde\phi_x^t=\tilde\phi_{a_\gamma x}^tT_\gamma\;,\quad T_{\gamma*}\widetilde Z_x=\widetilde Z_{a_\gamma x}\;.
\end{equation}
The only preserved leaf of $\tilde\phi$, $\{0\}\times\widetilde L\equiv\widetilde L$, is transversely simple. Now $\widetilde M^0=\{0\}\times\widetilde L$ and $\widetilde M^1=\widetilde M\setminus\widetilde M^0=\R^\times\times\widetilde L$, which has two connected components, $\widetilde M^1_\pm=\R^\pm\times\widetilde L$. In this case, $\widetilde M_\pm=(\R^\pm\cup\{0\})\times\widetilde L$, with $\mathring{\widetilde M}_\pm=\widetilde M^1_\pm$ and $\partial\widetilde M_\pm=\widetilde M^0\equiv\widetilde L$. The connected components of $M^1=M\setminus M^0$ are $M^1_\pm=\pi_M(\widetilde M^1_\pm)$, and we have $M_\pm=\pi_M(\widetilde M_\pm)$, with $\mathring M_\pm=M^1_\pm$ and $\partial M_\pm=M^0\equiv L$. The restriction $\pi_M:\widetilde M_\pm\to M_\pm$ will be denoted by $\pi_{M_\pm}$. The foliations $\widetilde\FF^1_\pm=\mathring{\widetilde\FF}_\pm$ on $\widetilde M^1_\pm=\mathring{\widetilde M}_\pm$ and $\widetilde\FF_\pm$ on $\widetilde M_\pm$ are restrictions of $\widetilde\FF$, and the foliations $\FF^1_\pm=\mathring\FF_\pm$ on $M^1_\pm=\mathring M_\pm$ and $\FF_\pm$ on $M_\pm$ are restrictions of $\FF$.  We have $\bfM=M_+\sqcup M_-$ (resp., $\widetilde{\bfM}=\widetilde M_+\sqcup\widetilde M_-$), equipped with the combination $\bfFF$ (resp., $\widetilde\bfFF$) of $\FF_+$ and $\FF_-$ (resp., $\widetilde\FF_+$ and $\widetilde\FF_-$). The restriction of $\bfFF$ to $\mathring\bfM$ is denoted by $\mathring\bfFF$. Now the map $\bfpi:\bfM\to M$ (resp., $\tilde\bfpi:\widetilde\bfM\to\widetilde M$) is the combination of the inclusion maps $M_\pm\hookrightarrow M$ (resp., $\widetilde M_\pm\hookrightarrow\widetilde M$). The combination of the maps $\pi_{M_\pm}$ is a covering projection $\pi_{\bfM}:\widetilde\bfM\to\bfM$. Moreover $\varpi$ (resp., $\widetilde\varpi$) restricts to global collar neighborhoods of the boundaries, $\varpi_\pm:M_\pm\to\partial M_\pm\equiv L$ (resp., $\widetilde\varpi_\pm:\widetilde M_\pm\to\partial\widetilde M_\pm\equiv\widetilde L$), whose combination is a global collar neighborhood of the boundary, $\bfvarpi:\bfM\to\partial\bfM\equiv L\sqcup L$ (resp., $\widetilde\bfvarpi:\widetilde\bfM\to\partial\widetilde\bfM\equiv\widetilde L\sqcup\widetilde L$). Like in \Cref{ss: suspension - homotheties}, for $\tilde p\in\widetilde M_\pm$ and $p\in M_\pm$, we have canonical identities $\Lambda_{\tilde p}\widetilde\FF_\pm\equiv\Lambda_{\widetilde\varpi_\pm(\tilde p)}\widetilde L$ and $\Lambda_p\FF_\pm\equiv\Lambda_{\varpi_\pm(p)}L$.

Recall also that any $A\in\fX(M,\FF)$, with foliated flow $\zeta=\{\zeta^t\}$, induces vector fields $A_\pm\in\fX(M_\pm,\FF_\pm)\subset\fXb(M_\pm)$, with foliated flows $\zeta_\pm=\{\zeta_\pm^t\}$, whose restrictions to $\mathring M_\pm\equiv M^1_\pm$ are denoted in the same way. In particular, we get $Z_\pm$ with flow $\phi_\pm=\{\phi_\pm^t\}$. The same kind of notation is used for vector fields and flows induced by elements of $\fXcom(\widetilde M,\widetilde\FF)$. Then $\FF^1_\pm\equiv\mathring\FF_\pm$ on $M^1_\pm\equiv\mathring M_\pm$ is a transversely complete $\R$-Lie foliation with the structure defined by $Z_\pm\in\fXcom(M^1_\pm,\FF^1_\pm)$ (see Remark~\ref{r: case where M is not compact}). In its Fedida's description (\Cref{ss: transverse structures}), $\widetilde M^1_\pm$ is the holonomy covering of $M^1_\pm$, whose group of deck transformations is also $\Gamma$, the developing map $D_\pm:\widetilde M^1_\pm\to\R$ is given by $D_\pm(x,y)=\varkappa^{-1}\ln|x|=:t$, the holonomy monomorphism $h_\pm:\Gamma\to\R$ is given by $h_\pm(\gamma)=\varkappa^{-1}\ln a_\gamma$, and therefore $\Hol\FF_\pm=\{\,\varkappa^{-1}\ln a_\gamma\mid\gamma\in\Gamma\,\}$. Thus $\widetilde Z_\pm\in\fXcom(\widetilde M^1_\pm,\widetilde\FF^1_\pm)$ is $D_\pm$-projectable and $(D_\pm)_*\widetilde Z_\pm=\partial_t$. Furthermore $\phi^t_\pm$ preserves every leaf of $\FF_\pm$ if and only if $t=h_\pm(\gamma)=\varkappa^{-1}\ln a_\gamma$ for some $\gamma\in\Gamma$ (\Cref{ss: complete R-Lie folns}).

Let $\tilde\xi=\{\tilde\xi^t\}$ be the weakly simple foliated flow on $(\widetilde M,\widetilde\FF)$, with infinitesimal generator is $\widetilde Y\in\fXcom(\widetilde M,\widetilde\FF)$, given by
\begin{equation}\label{tilde xi^t(tilde y x) widetilde Y}
\tilde\xi^t(x,\tilde y)=(e^{\varkappa t}x,\tilde y)\;,\quad\widetilde Y=(\varkappa x\partial_x,0)\;.
\end{equation}
We have $\overline{\widetilde Y}=\overline{\widetilde Z}\equiv\varkappa x\partial_x$, $\Fix(\tilde\xi^t)=\widetilde L$, and the orbits of $\tilde\xi^t$ on $\mathring{\widetilde M}_\pm$ are the fibers of the restriction $\widetilde\varpi:\mathring{\widetilde M}_\pm\to\widetilde L$. Since $\tilde\xi^t$ is $\Gamma$-equivariant and $\widetilde Y$ is $\Gamma$-invariant, they project to $M$ obtaining a weakly simple foliated flow $\xi^t$ on $(M,\FF)$ and its infinitesimal generator $Y\in\fX(M,\FF)$. We have $\overline Y=\overline Z\equiv\varkappa x\partial_x$, $\Fix(\xi^t)=L$, and the orbits of $\xi^t$ on $\mathring M_\pm$ are the fibers of the restriction $\varpi:\mathring M_\pm\to L$.

On the one hand, we consider the restriction of the transverse orientations of $\widetilde\FF$ and $\FF$ to $\widetilde\FF_\pm$ and $\FF_\pm$, and, on the other hand, we consider the transverse orientations of $\widetilde\FF^1_\pm$ and $\FF^1_\pm$ induced by $\widetilde Z_\pm$ and $Z_\pm$, which corresponds to the standard orientation of $\R$ by $D_\pm$ (\Cref{ss: folns almost w/o hol,ss: simple fol flows}). They agree on $\widetilde M^1_+$ and $M^1_+$ (resp., $\widetilde M^1_-$ and $M^1_-$) if and only if $\kappa>0$ (resp., $\kappa<0$).

\subsection{A defining form of $\FF$}\label{ss: omega - suspension}

For $k=\rank\Gamma$, fix generators $\gamma_1,\dots,\gamma_k$ of $\Gamma$. Let $c_i$ be a piecewise smooth loop in $L$ based at $p$ such that $[c_i]\in\pi_1(L,p)$ projects to $\gamma_i$, and let $a_i=a_{\gamma_i}$. By the universal coefficients and Hurewicz theorems, there are closed $1$-forms $\beta_1,\dots,\beta_k$ on $L$ such that $\delta_{ij}=\langle[\beta_i],[c_j]\rangle=\int_0^1c_j^*\beta_i$ and $\langle[\beta_i],\ker\hat h\rangle=0$. Every $\pi^*\beta_i$ is exact on $\widetilde L$. Let $\eta=-\ln(a_1)\,\beta_1-\dots-\ln(a_k)\,\beta_k$ and $\tilde\eta=\pi^*\eta=dF$ for some $F\in C^\infty(\widetilde L)$. Note that $h(\Gamma)\cong\Gamma$ is the group of periods of $\eta$. With some abuse of notation, write also $F\equiv\widetilde\varpi^*F\in C^\infty(\widetilde M)$, and
\begin{gather*}
\eta\equiv\varpi^*\eta\in C^\infty(M;\Lambda^{0,1})\equiv C^\infty(M;\Lambda^1\FF)\;,\\
\tilde\eta\equiv\widetilde\varpi^*\tilde\eta=\pi_M^*\eta=dF\in C^\infty(\widetilde M;\Lambda^{0,1})
\equiv C^\infty(M;\Lambda^1\widetilde\FF)\;,
\end{gather*}
using~\eqref{bigwedge^ge u cdot T^*M / bigwedge^ge u+1 cdot T^*M}. Thus $\eta=\eta_0$ in this case, with the notation of \Cref{ss: perturbation vs bigrading}. It is easy to check that
\begin{equation}\label{T_gamma^* F = F + ln a_gamma}
T_\gamma^*F=F-\ln a_\gamma
\end{equation}
for all $\gamma\in\Gamma$, yielding $T_\gamma^*e^F=a_\gamma^{-1}e^F$. It easily follows that the $1$-form $\tilde\omega=|\varkappa|^{-1}e^F\,dx$ on $\widetilde M$ is $\Gamma$-invariant. (Recall that the $\Gamma$-action on $\widetilde M$ is given by $\gamma \cdot (x,\tilde y) = (a_\gamma x,\gamma\cdot\tilde y)$.) Furthermore $T\widetilde\FF=\ker\tilde\omega$ and $\tilde\omega$ defines the transverse orientation of $\widetilde\FF$. Therefore $\tilde\omega$ induces a $1$-form $\omega$ on $M$ satisfying $T\FF=\ker\omega$ and defining the transverse orientation of $\FF$. On the other hand, it is easy to compute $d\tilde\omega=\tilde\eta\wedge\tilde\omega$, yielding $d\omega=\eta\wedge\omega$. The vector field $\widetilde X=(|\varkappa|e^{-F}\partial_x,0)\equiv|\varkappa|e^{-F}\partial_x\in C^\infty(\widetilde M;\widetilde\bfH)$ is determined by $\tilde\omega(\widetilde X)=1$. Thus $\widetilde X$ is $\Gamma$-invariant and induces the vector field $X\in C^\infty(M;\bfH)$ satisfying $\omega(X)=1$. So $\widetilde X$ and $X$ also define the transverse orientations of $\widetilde\FF$ and $\FF$. On the other hand, $\tilde\omega(\widetilde Z)=\sign(\varkappa)e^Fx$ by~\eqref{tilde phi^t(tilde y x) widetilde Z}, yielding $\sign\tilde\omega(\widetilde Z_\pm)=\pm\sign(\varkappa)$, and therefore
\[
\sign\omega(Z_\pm)=\pm\sign(\varkappa)\;.
\]
So the transverse orientation of $\FF^1_\pm$ is also defined by the restrictions to $M^1_\pm$ of $\pm\sign(\varkappa)\overline X$ or $\pm\sign(\varkappa)\omega$.

\subsection{A defining function of $M^0$}\label{ss: def func of M^0 - suspension}

Let $\tilde\rho=e^Fx$, which is a defining function of $\widetilde M^0\equiv\widetilde L$ on the whole of $\widetilde M$. Moreover $\tilde\rho$ is $\Gamma$-invariant by~\eqref{T_gamma^* F = F + ln a_gamma}, and therefore it induces a defining function $\rho$ of $M^0\equiv L$ on the whole of $M$. It is easy to compute
\[
d\tilde\rho=e^F(x\tilde\eta+dx)=\tilde\rho\tilde\eta+|\varkappa|\tilde\omega\;,
\]
yielding
\begin{equation}\label{d rho}
d\rho=\rho\eta+|\varkappa|\omega\;,
\end{equation}
and therefore
\begin{equation}\label{d rho wedge omega = rho eta wedge omega}
d\rho\wedge\omega=\rho\eta\wedge\omega\;.
\end{equation}
Since $\tilde\xi^{t*}\tilde\rho=e^{\varkappa t}\tilde\rho$ by~\eqref{tilde xi^t(tilde y x) widetilde Y}, we also get
\begin{equation}\label{xi^t* rho = e^varkappa t rho}
\xi^{t*}\rho=e^{\varkappa t}\rho\;.
\end{equation}

The global tubular neighborhood $\widetilde\varpi:\widetilde M\to\widetilde L\equiv\widetilde M^0$ can be trivialized with $\tilde\rho$, obtaining $\widetilde M\equiv\R_{\tilde\rho}\times\widetilde L_{\widetilde\varpi}$ besides $\widetilde M=\R_x\times\widetilde L_{\widetilde\varpi}$. Thus the global tubular neighborhood $\varpi:M\to L\equiv M^0$ can be trivialized with $\rho$, obtaining $M\equiv\R_\rho\times L_\varpi$. According to \Cref{s: conormal seq}, we have corresponding vector fields $\partial_x,\partial_{\tilde\rho}\in\fX(\widetilde M)$ and $\partial_\rho\in\fX(M)$, and operators $\partial_x,\partial_{\tilde\rho}\in\Diff^1(\widetilde M;\Lambda)$ and $\partial_\rho\in\Diff^1(M;\Lambda)$.  We compute
\[
|\varkappa|\partial_{\tilde\rho}=|\varkappa|\partial_{\tilde\rho}(x)\,\partial_x
=|\varkappa|\partial_{\tilde\rho}(e^{-F}\tilde\rho)\,\partial_x
=|\varkappa|e^{-F}\,\partial_x=\widetilde X\in\fX(\widetilde M)\;.
\]
It easily follows that
\[
|\varkappa|\partial_{\tilde\rho}=|\varkappa|e^{-F}\,\partial_x\in\Diff^1(M;\Lambda)\;.
\]
But, with the notation of \Cref{ss: LL_X i -i}, it is easy to check that
\[
\partial_x=\Theta_{\partial_x}\in\Diff^1(\widetilde M;\Lambda)\;.
\]
Since $\widetilde X\in C^\infty(\widetilde M;\widetilde\bfH)$ and $d_{1,0}F=0$, we can apply~\eqref{Theta_fX = f Theta_X} on $C^\infty(\widetilde M;\Lambda)$ to get
\[
\Theta_{\widetilde X}=|\varkappa|e^{-F}\Theta_{\partial_x}=|\varkappa|e^{-F}\partial_x
=|\varkappa|\partial_{\tilde\rho}\in\Diff^1(\widetilde M;\Lambda)\;.
\]
So, by derivation formula of $\Theta_{\widetilde X}$,~\eqref{Theta_X d_0 1 = d_0 1 Theta_X} and~\eqref{Theta_fX = f Theta_X}, and since $\partial_x\in\fX(\widetilde M,\widetilde\FF)\cap C^\infty(\widetilde M;\widetilde\bfH)$ and $\LL_{\widetilde X}\tilde\eta=0$,
\begin{align*}
[d_{0,1},\Theta_{\widetilde X}]&=|\varkappa|[d_{0,1},e^{-F}\Theta_{\partial_x}]
=|\varkappa|[d_{0,1},e^{-F}]\Theta_{\partial_x}\\
&=-|\varkappa|e^{-F}{\tilde\eta\wedge}\,\Theta_{\partial_x}=-{\tilde\eta\wedge}\,\Theta_{\widetilde X}
=-\Theta_{\widetilde X}\,{\tilde\eta\wedge}\;.
\end{align*}
Hence $|\varkappa|\partial_\rho=X\in\fX(M)$, and
\begin{gather}
|\varkappa|\partial_\rho=\Theta_X\in\Diff^1(M;\Lambda)\;,\label{| varkappa | partial_rho = Theta_X}\\
[\Theta_X,d_{0,1}]={\eta\wedge}\,\Theta_X=\Theta_X\,{\eta\wedge}\in\Diff^1(M;\Lambda)\;.\label{[Theta_X d_0 1]}
\end{gather}
Note also that $\Theta_{\widetilde X}\tilde\omega=0$, and therefore $\Theta_X\omega=0$. Moreover, $\Theta_X\eta=0$.

For any $\epsilon>0$, the restriction $\widetilde\varpi:\widetilde T_\epsilon:=\{|\tilde\rho|<\epsilon\}\to\widetilde L$ is a smaller tubular neighborhood of $\widetilde L$ in $\widetilde M$, which induces a smaller tubular neighborhood $\varpi:T_\epsilon:=\{|\rho|<\epsilon\}\to L$ of $L$ in $M$. Let $\widetilde T^1_\epsilon=\widetilde T_\epsilon\cap\widetilde M^1$, $\widetilde T^1_{\pm,\epsilon}=\widetilde T_{\epsilon}\cap\widetilde M^1_\pm$, $T^1_\epsilon=T_\epsilon\cap M^1$ and $T^1_{\pm,\epsilon}=T_\epsilon\cap M^1_\pm$.

\subsection{A boundary-defining function of $M_\pm$}\label{ss: boundary-def func of M_pm - suspension}

Now consider the boundary-defining function $\tilde\rho_\pm=e^F|x|=\pm\tilde\rho=|\tilde\rho|$ on $\widetilde M_\pm$, which is $\Gamma$-invariant, and therefore it induces a boundary-defining function $\rho_\pm$ on $M_\pm$ satisfying
\begin{equation}\label{rho_pm = pm rho = | rho |}
\rho_\pm=\pm\rho=|\rho|\;.
\end{equation}
If there is no danger of confusion, with some abuse of notation, these boundary-defining functions may be simply denoted by $\tilde\rho$ and $\rho$. Furthermore, we have the boundary-defining function $\tau=\tau_\pm:=|x|$ on $\widetilde M_\pm$.

The global collar neighborhood $\widetilde\varpi:\widetilde M_\pm\to\partial\widetilde M_\pm\equiv\widetilde L$ can be trivialized with, either $\tau$, or $\tilde\rho$, obtaining
\begin{align}
\widetilde M_\pm&\equiv[0,\infty)_\tau\times\widetilde L_{\widetilde\varpi}
\label{widetilde M_pm equiv [0 infty)_tau times widetilde L_varpi}\\
&\equiv[0,\infty)_{\tilde\rho}\times\widetilde L_{\widetilde\varpi}\;.
\label{widetilde M_pm equiv [0 infty)_tilde rho times widetilde L_varpi}
\end{align}
So the global collar neighborhood $\varpi:M_\pm\to\partial M\equiv L$ can be trivialized with $\rho$, obtaining
\begin{equation}\label{M_pm equiv [0 infty)_rho times L_varpi}
M_\pm\equiv[0,\infty)_\rho\times L_\varpi\;.
\end{equation}
By~\eqref{d rho},~\eqref{| varkappa | partial_rho = Theta_X},~\eqref{[Theta_X d_0 1]} and~\eqref{rho_pm = pm rho = | rho |},
\begin{gather}
d\rho=\rho\eta\pm|\varkappa|\omega\;,\label{d rho_pm}\\
[\partial_\rho,d_{0,1}]={\eta\wedge}\,\partial_\rho\in\Diff^1(M_\pm;\Lambda)\;,\label{[partial_rho d_0 1]}
\end{gather}
using the operator $\partial_\rho\in\Diff^1(M_\pm;\Lambda)$ introduced in \Cref{r: E_m m < 0,r: E_m vector bundle}. We have $\partial_\rho\omega=0$, and therefore $\partial_\rho\eta=0$.

Observe that $\tilde\rho^{-1}\tilde\omega=|\varkappa x|^{-1}dx$ is a basic form of $\mathring{\widetilde\FF}_\pm$, and therefore $\rho^{-1}\omega$ is a basic form of $\mathring\FF_\pm$. The transverse orientation of $\mathring\FF_\pm=\FF^1_\pm$ is also defined by the basic form 
\begin{equation}\label{omega_pm = sign(omega(Z_pm)) rho_pm^-1omega}
\omega_{\text{\rm b},\pm}=\sign(\varkappa)\,\rho^{-1}\omega=\pm\sign(\varkappa)\,\rho_\pm^{-1}\omega=\sign(\omega(Z_\pm))\,\rho_\pm^{-1}\omega\;,
\end{equation}
whose lift to $\mathring{\widetilde M}_\pm$ is $\tilde\omega_{\text{\rm b},\pm}=(\varkappa x)^{-1}\,dx$; in fact, $\tilde\omega_{\text{\rm b},\pm}(\widetilde Z_\pm)=1$, and therefore $\omega_{\text{\rm b},\pm}(Z_\pm)=1$. By~\eqref{d rho_pm} and~\eqref{omega_pm = sign(omega(Z_pm)) rho_pm^-1omega},
\begin{equation}\label{rho_pm^-1 d rho_pm}
\rho_\pm^{-1}\,d\rho_\pm=\eta+\varkappa\,\omega_{\text{\rm b},\pm}\;.
\end{equation}

Let $\nu=\nu_\pm$ be the unique smooth trivialization of ${}_+N\partial M_\pm$ so that $d\rho(\nu)=1$ (\Cref{ss: b-geometry}). By~\eqref{d rho_pm}, $\nu$ is represented by the restriction of $\pm|\varkappa|^{-1}X$ to $L\equiv\partial M_\pm$.

The combination of $\rho_+$ and $\rho_-$ is a boundary-defining function $\bfrho$ on $\bfM$, and the combination of $\nu_+$ and $\nu_-$ is the unique smooth trivialization $\bfnu$ of ${}_+N\partial\bfM$ so that $d\bfrho(\bfnu)=1$. Similarly, we define $\tilde\bfrho$ on $\widetilde\bfM$ and $\tilde\bfnu$ on $\partial\widetilde\bfM$.

For $\epsilon>0$, the restriction $\widetilde\varpi_\pm=\widetilde\varpi:\widetilde T_{\pm,\epsilon}:=\{\tilde\rho_\pm<\epsilon\}\to\widetilde L$ is a smaller collar neighborhood of the boundary in $\widetilde M_\pm$, which induces a smaller collar neighborhood $\varpi_\pm=\varpi:T_{\pm,\epsilon}:=\{\rho_\pm<\epsilon\}\to L$ of the boundary in $M_\pm$. By combination, we get smaller collar neighborhoods of the boundaries, $\widetilde\bfvarpi:\widetilde\bfT_\epsilon:=\{\tilde\bfrho<\epsilon\}\to\partial\widetilde\bfM=\widetilde L\sqcup\widetilde L$ and $\bfvarpi:\bfT_\epsilon:=\{\bfrho<\epsilon\}\to\partial\bfM=L\sqcup L$. We have $\mathring{\widetilde T}_\epsilon\equiv\widetilde T^1_\epsilon$, $\mathring{\widetilde T}_{\pm,\epsilon}\equiv\widetilde T^1_{\pm,\epsilon}$, $\mathring T_\epsilon\equiv T^1_\epsilon$ and $\mathring T_{\pm,\epsilon}\equiv T^1_{\pm,\epsilon}$.

\subsection{The metric $g_M$}\label{ss: g_M}

Take a Riemannian metric $g_L$ on $L$, and let $g_{\widetilde L}$ be its lift to $\widetilde L$. Consider leafwise metrics, $g_{\FF}=\varpi^* g_L$ for $\FF$ and $g_{\widetilde\FF}=\widetilde\varpi^*g_{\widetilde L}$ for $\widetilde\FF$; their restrictions to $\FF_\pm$ and $\widetilde\FF_\pm$ may be denoted by $g_{\FF_\pm}$ and $g_{\widetilde\FF_\pm}$. Consider also the metric $g_M=\omega^2+g_\FF$ on $M$. The lift of $g_M$ to $\widetilde M$ is
\[
g_{\widetilde M}=\tilde\omega^2+g_{\widetilde\FF}=|\varkappa|^{-2}e^{2F}\,(dx)^2+g_{\widetilde\FF}\;.
\]
With respect to $g_M$, the transverse volume form is $\omega$, $X$ is unitary and orthogonal to $\FF$, and $T\FF^\perp=\bfH$.

\subsection{The b-metrics $g_{\text{\rm b},\pm}$}\label{ss: g_b pm}

Define also the metric $g_{\text{\rm b},\pm}=\rho_\pm^{-2}\omega^2+g_{\FF}=\omega_{\text{\rm b},\pm}^2+g_{\FF}$ on $\mathring M_\pm=M^1_\pm$, where the last equality uses~\eqref{omega_pm = sign(omega(Z_pm)) rho_pm^-1omega}. It is bundle-like for $\mathring\FF_\pm=\FF^1_\pm$, and its lift to $\mathring{\widetilde M}_\pm$ is
\begin{equation}\label{tilde g_pm}
\tilde g_{\text{\rm b},\pm}=\tilde\rho_\pm^{-2}\,\tilde\omega^2+g_{\widetilde\FF}
=\tilde\omega_{\text{\rm b},\pm}^2+g_{\widetilde\FF}
=(\varkappa x)^{-2}\,(dx)^2+g_{\widetilde\FF}\;.
\end{equation}
With respect to $g_{\text{\rm b},\pm}$, the transverse volume form is $\omega_{\text{\rm b},\pm}$, $Z_\pm$ is unitary and orthogonal to $\mathring\FF_\pm$, and $T\mathring\FF_\pm^\perp=\bfH|_{\mathring M_\pm}$. The metrics $\tilde g_{\text{\rm b},\pm}$ and $g_{\text{\rm b},\pm}$ on $\mathring{\widetilde M}_\pm$ and $\mathring M_\pm$ are restrictions of b-metrics on $\widetilde M_\pm$ and $M_\pm$, also denoted by $\tilde g_{\text{\rm b},\pm}$ and $g_{\text{\rm b},\pm}$. In the rest of this subsection, $\mathring{\widetilde M}_\pm$ and $\mathring M_\pm$ (resp., $\widetilde M_\pm$ and $M_\pm$) are assumed to be endowed with the metrics (resp., b-metrics) $\tilde g_{\text{\rm b},\pm}$ and $g_{\text{\rm b},\pm}$. By~\eqref{rho_pm^-1 d rho_pm}, if $\eta\ne0$ or $\varkappa^2\ne1$, then the b-metrics $\tilde g_{\text{\rm b},\pm}$ and $g_{\text{\rm b},\pm}$ are not exact (\Cref{s: b-calculus}).

\begin{prop}\label{p: FF_pm is of bd geom}
$\mathring\FF_\pm$ is of bounded geometry.
\end{prop}

\begin{proof}
$\mathring{\widetilde M}_\pm$ is of bounded geometry because it is the Riemannian product of $(\R^\pm,(\varkappa x)^{-2}(dx)^2)$ and $(\widetilde L,g_{\widetilde L})$, which are of bounded geometry since $L$ is compact and the change of coordinate $t=\varkappa^{-1}\ln|x|$ defines an isometry between $(\R^\pm,(\varkappa x)^{-2}(dx)^2)$ and $(\R,(dt)^2)$. Via this isometry, $\tilde\omega_{\text{\rm b},\pm}\equiv(\varkappa x)^{-1}\,dx$ on $\R^\pm$ is the pull-back of $dt$ on $\R$. On the other hand, the leaves of $\mathring{\widetilde\FF}_\pm$ are the fibers $\{x\}\times\widetilde L$ ($x\in\R^\pm$), and the O'Neill tensors of $\mathring{\widetilde\FF}_\pm$ with $\tilde g_{\text{\rm b},\pm}$ vanish. Hence, on $\mathring M_\pm$ with $\mathring\FF_\pm$ and $g_{\text{\rm b},\pm}$, all covariant derivatives of the curvature tensor are uniformly bounded, and the O'Neill tensors vanish.

Finally, the bi-injectivity radius of $\mathring{\widetilde\FF}_\pm$ with $\tilde g_{\text{\rm b},\pm}$ is positive because a normal foliated chart centered at any $\tilde p=(x_0,\tilde q)$ is given by
\[
\tilde\chi_{\tilde p,\pm}=(t_{x_0},\tilde y_{\tilde q}):\widetilde U_{\tilde p,\pm}=\R^\pm\times B_{\widetilde L}(\tilde q,r)\to\R\times B\;,
\]
where $t_{x_0}=\varkappa^{-1}(\ln|x|-\ln|x_0|)$, $r=\inj_L\le\inj_{\widetilde L}$, $B$ is the open ball in $\R^{n-1}$ of radius $r$ and center $0$, and $\tilde y_{\tilde q}:B_{\widetilde L}(\tilde q,r)\to B$ is a normal chart of $\widetilde L$. Let $q=\pi(\tilde q)$, $p=\pi_M(\tilde p)$ and $U_{p,\pm}=\pi_M(\widetilde U_{\tilde p,\pm}$). Then $\pi:B_{\widetilde L}(\tilde q,r)\to B_L(q,r)$ is a diffeomorphism, obtaining a normal chart $y_q:B_L(q,r)\to B$ of $L$ that corresponds to $\tilde y_{\tilde q}$ via $\pi$. So $\pi_M:\widetilde U_{\tilde p,\pm}\to U_{p,\pm}$ is also a diffeomorphism, and $\tilde\chi_{\tilde p,\pm}$ induces via $\pi_M$ a normal foliated chart of $\mathring\FF_\pm$ with $g_{\text{\rm b},\pm}$, centered at $p$,
\[
\chi_{p,\pm}\equiv(t_{x_0},y_q):U_{p,\pm}\equiv\R^\pm\times B_L(q,r)\to\R\times B\;.
\]
This shows that the injectivity bi-radius of $\mathring\FF_\pm$ with $g_{\text{\rm b},\pm}$ is positive.
\end{proof}

By \Cref{p: FF_pm is of bd geom} and according to \Cref{ss: transverse structures of bd geometry}, $\mathring M_\pm$ is of bounded geometry (the property~\ref{i: g is of bounded geometry} of \Cref{ss: a description of AA(M)}).

\begin{prop}\label{p: (M_pm g_pm) satisfies (B)}
$M_\pm$ satisfies the property~\ref{i: A'} of \Cref{ss: a description of AA(M)}.
\end{prop}

\begin{proof}
According to the proof of \Cref{p: FF_pm is of bd geom}, it is easy to check that $(\widetilde M_\pm,\tilde g_{\text{\rm b},\pm})$ satisfies~\ref{i: A'} on the whole of $\widetilde M_\pm$ with $\tilde\rho\partial_{\tilde\rho}$ and the extensions $B'=(B,0)\in\fX(\widetilde M_\pm)$ of vector fields $B\in\fX(\widetilde L)$. It follows that $(M_\pm,g_{\text{\rm b},\pm})$ also satisfies~\ref{i: A'} on the whole of $M_\pm$ with $\rho\partial_\rho$ and the extensions $A'\in\fX(M_\pm)$ of vector fields $A\in\fX(L)$ defined as follows. For every $A\in\fX(L)$, let $\widetilde A$ denote its lift to $\widetilde L$. Then $\widetilde A'=(\widetilde A,0)\in\fX(\widetilde M_\pm)$ is $\Gamma$-invariant, and therefore it is $\pi_{M_\pm}$-projectable to a vector field $A'\in\fX(M_\pm)$.
\end{proof}

\begin{prop}\label{p: d(ln rho) in C^infty_ub(mathring M_pm Lambda^1)}
We have $d(\ln\rho)\in C^\infty_{\text{\rm ub}}(\mathring M_\pm;\Lambda^1)$.
\end{prop}

\begin{proof}
On the one hand, by the compactness of $L$ and the definition of $g_{\text{\rm b},\pm}$, we have $\eta\in C^\infty_{\text{\rm ub}}(\mathring M_\pm;\Lambda^{0,1})$. On the other hand, $\rho^{-1}\omega\in C^\infty_{\text{\rm ub}}(\mathring M_\pm;\Lambda^{1,0})$ by \Cref{p: FF_pm is of bd geom}, since $\omega_{\text{\rm b},\pm}=\sign(\varkappa)\rho^{-1}\omega$ is the $g_{\text{\rm b},\pm}$-transverse volume form. Hence $d(\ln\rho)=\eta\pm|\varkappa|\rho^{-1}\omega\in C^\infty_{\text{\rm ub}}(\mathring M_\pm;\Lambda^1)$  by~\eqref{d rho_pm}. 
\end{proof}

Let $|{\cdot}|:\Gamma\to\N_0$ be the word length function given by a finite generating set. There is some $c_0>0$ so that, for all $\gamma\in\Gamma$,
\begin{equation}\label{|ln a_gamma| le c_0 |gamma|}
|\ln a_\gamma| \le c_0\,|\gamma|\;.
\end{equation}
By~\eqref{word-metric},~\eqref{tilde xi^t(tilde y x) widetilde Y} and~\eqref{tilde g_pm}, for all $\tilde p\in\mathring{\widetilde M}_\pm$ and $\gamma\in\Gamma$,
\[
c_1^{-1}|\gamma|\le d_{\mathring{\widetilde\FF}_\pm}\big(\gamma^{-1}\cdot\xi^{h_\pm(\gamma)}(\tilde p),\tilde p\big)\le c_1|\gamma|\;,
\]
using the holonomy homomorphism $h_\pm:\Gamma\to\R$ of $\FF^1_\pm\equiv\mathring\FF_\pm$ (\Cref{ss: transv simple flows - suspension}).

\begin{lem}\label{l: d_widetilde L(tilde y tilde y') > C ln(rho([x tilde y]) / c rho([x tilde y'])}
There are $C>0$ and $c\ge1$ so that, for all $\tilde y,\tilde y'\in\widetilde L$ and $x\in\R^\times$,
\[
d_{\widetilde L}(\tilde y,\tilde y')\ge C\ln \frac{\rho([x,\tilde y])}{c\rho([x,\tilde y'])}\;.
\]
\end{lem}

\begin{proof}
Let $\bfF\subset\widetilde L$ be a fundamental domain. Without loss of generality, we can assume $\tilde y\in \bfF$. Take some $\gamma\in\Gamma$ such that
$\gamma\cdot \tilde y' \in \bfF$. Then
\[
\rho([x,\tilde y])=e^{F(\tilde y)}x\;, \quad 
\rho([x,\tilde y'])=\rho([a_\gamma x,\gamma\cdot\tilde y'])=e^{F(\gamma\cdot\tilde y')}a_\gamma x\;.
\]
There is some $C_0\ge1$ such that, for all $\tilde y_1, \tilde y_2\in \bfF$,
\[
C_0^{-1}e^{F(\tilde y_1)}\le e^{F(\tilde y_2)}\le C_0e^{F(\tilde y_1)}\;.
\]
So, using~\eqref{word-metric with bfK} with $\bfK=\bfF^2$ and~\eqref{|ln a_gamma| le c_0 |gamma|},
\[
\frac{\rho([x,\tilde y])}{\rho([x,\tilde y'])}=\frac{e^{F(\tilde y)}}{e^{F(\gamma\cdot\tilde y')}a_\gamma}
\le C_0a_{\gamma^{-1}}\le C_0e^{c_0\,|\gamma^{-1}|}\le C_0e^{c_0c_1(d_{\widetilde L}(\tilde y',\tilde y)+c_2)}\;.\qedhere
\]
\end{proof}

\begin{cor}\label{c: Pen_FF_l(T_epsilon R) subset T_ce^R/C epsilon}
For $R,\epsilon>0$, we have $\Pen_\FF(T_\epsilon, R)\subset T_{ce^{R/C}\epsilon}$.
\end{cor}

\begin{lem}\label{l: d_FF(phi^t(p) p) ge inj_L}
For $p\in\mathring M_\pm$ and $t\in\R^\times$, if $\phi^t(L_p)=L_p$, then $d_{\FF}(\phi^t(p),p)\ge\inj_L$.
\end{lem}

\begin{proof}
We have $p=[x_0,\tilde q]$ for some $x_0\in\R^\pm$ and $\tilde q\in\widetilde L$, and let $q=\pi(\tilde q)\in L$. For $r=\inj_L$, take normal charts $\tilde y_{\tilde q}:B_{\widetilde L}(\tilde q,r)\to B$ and $y_q:B_L(q,r)\to B$ like in the proof of \Cref{p: FF_pm is of bd geom}. We have the foliated chart of $\widetilde\FF$,
\[
\tilde\chi_{\tilde q}=(x,\tilde y_{\tilde q}):\widetilde U_{\tilde q}=\R\times B_{\widetilde L}(\tilde q,r)\to\R\times B\;,
\]
which induces via $\pi_M$ a foliated chart of $\FF$,
\[
\chi_q=(x,y_q):U_q=\pi_M(\widetilde U_{\tilde q})\equiv\R\times B_L(q,r)\to\R\times B\;.
\]
On the one hand, if $\phi^t(p)\in U_q$, then $\chi_q\phi^t(p)=(e^{\varkappa t}x_0,\tilde\phi^t_{x_0}(\tilde q))$ by~\eqref{tilde phi^t(tilde y x) widetilde Z}, with $e^{\varkappa t}x_0\ne x_0$ because $t\ne0$. So $p$ and $\phi^t(p)$ lie in different plaques of $(U_q,\chi_q)$. On the other hand, if $\phi^t(p)\not\in U_q$, then a fortiori $\phi^t(p)$ is not in the plaque of $(U_q,\chi_q)$ through $p$. In any case, $d_\FF(\phi^t(p),p)\ge r$ because the plaque through $p$ is $B_{L_p}(p,r)\equiv B_{\widetilde L}(\tilde q,r)$.
\end{proof}

\begin{prop}\label{p: d_FF(phi^t(p) p) ge c_1^{-1} |gamma| - c_2}
If $Z_\pm\in\fXub(M^1_\pm,\FF^1_\pm)$, then, for any compact $I\subset\R$, there are $c_1,c_2>0$ such that, for all $p\in\mathring M_\pm$ and $\gamma\in\Gamma$ with $h_\pm(\gamma)\in I$,
\[
d_{\FF}\big(\phi^{h_\pm(\gamma)}(p),p\big)\ge c_1^{-1}|\gamma|-c_2\;.
\]
\end{prop}

\begin{proof}
Since $I$ is compact and $\phi_\pm$ is of $\R$-local bounded geometry on $\mathring M_\pm$ (\Cref{ss: families of bd geom}), there is some $R>0$ such that $d_{\widetilde L}(\tilde\phi^t_x(\tilde y),\tilde y)\le R$ for all $x\in\R^\pm$, $t\in I$ and $\tilde y \in\widetilde L$. Given any fundamental domain $\bfF\subset\widetilde L$, let $\bfK=\overline{\Pen}_{\widetilde L}(\bfF,R)\times\bfF$. By~\eqref{word-metric with bfK}, there are $c_1\ge1$ and $c_2\ge0$ such that 
\begin{equation}\label{d_widetilde L(tilde y gamma cdot tilde phi^t_x(tilde y)) ge c_1^-1 |gamma| - c_2}
d_{\widetilde L}(\gamma\cdot\tilde\phi^t_x(\tilde y),\tilde y')\ge c_1^{-1}|\gamma|-c_2
\end{equation}
for all $x\in\R^\pm$, $t\in I$, $\tilde y,\tilde y'\in\bfF$ and $\gamma\in\Gamma$, because $(\tilde\phi^t_x(\tilde y),\tilde y')\in\bfK$.

Any $p\in\mathring M_\pm$ is of the form $p=[x,\tilde y]$ for some $x\in\R^\pm$ and $\tilde y\in\bfF$. Let $\gamma\in\Gamma$ with $t:=h_\pm(\gamma)=\varkappa^{-1}\ln a_\gamma\in I$. Then $\phi^t(L_p)=L_p$ (\Cref{ss: suspension - homotheties}), and, by~\eqref{d_widetilde L(tilde y gamma cdot tilde phi^t_x(tilde y)) ge c_1^-1 |gamma| - c_2},
\begin{align*}
d_\FF(\phi^t(p),p)&=d_\FF([e^{\varkappa t}x,\tilde\phi^t_x(\tilde y)],[x,\tilde y])
=d_\FF([a_\gamma x,\tilde\phi^t_x(\tilde y)],[x,\tilde y])\\
&=d_\FF([x,\gamma^{-1}\cdot\tilde\phi^t_x(\tilde y)],[x,\tilde y])
=d_{\widetilde L}(\gamma^{-1}\cdot\tilde\phi^t_x(\tilde y),\tilde y)
\ge c_1^{-1}|\gamma|-c_2\;.\qedhere
\end{align*}
\end{proof}

\begin{cor}\label{c: d_FF(phi^t(p) p) ge c_3 |gamma|}
If $Z_\pm\in\fXub(M^1_\pm,\FF^1_\pm)$, then, for any compact $I\subset\R^\times$, there is some $c_3>0$ such that, for all $p\in\mathring M_\pm$ and $\gamma\in\Gamma$ with $h_\pm(\gamma)\in I$,
\[
d_{\FF}\big(\phi^{h_\pm(\gamma)}(p),p\big)\ge c_3|\gamma|\;.
\] 
\end{cor}

\begin{proof}
By \Cref{l: d_FF(phi^t(p) p) ge inj_L,p: d_FF(phi^t(p) p) ge c_1^{-1} |gamma| - c_2}, the result follows taking $c_3>0$ such that, for all $\gamma\in\Gamma$,
\[
c_3|\gamma|\le
\begin{cases}
\inj_L & \text{if $|\gamma|\le c_1c_2$}\\
c_1^{-1}|\gamma|-c_2 & \text{if $|\gamma|>c_1c_2$}\;.\qedhere
\end{cases}
\]
\end{proof}

\begin{prop}\label{p: phi^-t(T_epsilon) subset T_kappa epsilon}
If $Z_\pm\in\fXub(M^1_\pm,\FF^1_\pm)$, then, for any compact $I\subset\R$, there exists some $c'>0$ such that $\phi^t(T_\epsilon)\subset T_{c'\epsilon}$ for all $t\in I$ and $\epsilon>0$. 
\end{prop}

\begin{proof}
Take some $R>0$ like in the proof of \Cref{p: d_FF(phi^t(p) p) ge c_1^{-1} |gamma| - c_2}. Let $\tilde y\in\widetilde L$ and $x\in\R^\pm$ such that $\rho([x,\tilde y])<\epsilon$. If $\bfc:[0,1]\to\widetilde L$ is a minimizing geodesic segment from $\tilde y$ to $\tilde\phi^t_x(\tilde y)$, then
\[
F(\tilde\phi^t_x(\tilde y))-F(\tilde y)=\int_0^1\bfc^*dF=\int_0^1\bfc^*\tilde\eta\le\|\tilde\eta\|_{L^\infty}d_{\widetilde L}(\tilde y,\tilde\phi^t_x(\tilde y))\le\|\eta\|_{L^\infty}R\;.
\]
Take also some $c'_1>0$ such that $e^{\varkappa t}\le c'_1$ for all $t\in I$. Then
\[
\rho(\phi^t([x,\tilde y]))=e^{F(\tilde \phi^t_x(\tilde y))-F(\tilde y)+\varkappa t}\rho([x,\tilde y])\le e^{\|\eta\|_{L^\infty}R}c'_1\epsilon\;.\qedhere
\]
\end{proof}

\begin{prop}\label{p: L subset overline W}
Suppose $\Gamma$ is nontrivial. For any $\epsilon>0$, there is some $0<\epsilon'<\epsilon$ such that, for all leaf $L'$ of $\FF$, if a connected component $W$ of $L'\cap T_\epsilon$ meets $T_{\epsilon'}$, then $L\subset\overline W$.
\end{prop}

\begin{proof}
Let $\bfF\subset\widetilde L$ be a fundamental domain. We can choose $0<\epsilon'<\epsilon$ such that $e^{F(\tilde y)-F(\tilde y')}\epsilon'<\epsilon$ for all $\tilde y,\tilde y'\in\bfF$. Let $W$ be a connected component of $L'\cap T_\epsilon$ that meets $T_{\epsilon'}$ at some point $[x,\tilde y]$. We can assume $\tilde y\in\bfF$. For every $\tilde y'\in\bfF$, we have $[x,\tilde y']\in L'$ and
\[
|\rho([x,\tilde y'])|=e^{F(\tilde y')}|x|=e^{F(\tilde y')-F(\tilde y)}|\rho([x,\tilde y])|<e^{F(\tilde y')-F(\tilde y)}\epsilon'<\epsilon\;.
\]
So $\pi_M(\{x\}\times\bfF)\subset W$ because $\bfF$ is connected. Since $\Gamma$ is nontrivial and $h$ is injective, there is some $\gamma\in\Gamma$ such that $\gamma\cdot\bfF\cap\bfF\ne\emptyset$ and $a_\gamma>1$. Then $W_0:=\bigcup_{m=0}^\infty\gamma^m\cdot\bfF$ is connected in $\widetilde L$. Moreover, by~\eqref{T_gamma^* F = F + ln a_gamma}, for all $m\in\N_0$ and $\tilde y'\in\bfF$,
\[
|\rho([x,\gamma^m\tilde y'])|=e^{F(\gamma^m\tilde y')}|x|=a_\gamma^{-m}e^{F(\tilde y')}|x|=a_\gamma^{-m}|\rho([x,\tilde y'])|<a_\gamma^{-m}\epsilon\;,
\]
which is${}<\epsilon$ and converges to $0$ as $m\to\infty$. Since $[x,\gamma^m\tilde y']\in\varpi^{-1}([\tilde y'])$, it follows that $[x,\gamma^m\tilde y']\to[0,\tilde y']\equiv[\tilde y']$ as $m\to\infty$. Hence $\pi_M(\{x\}\times W_0)\subset W$ and $L\subset\overline{\pi_M(\{x\}\times W_0)}$.
\end{proof}

\subsection{The b-metrics $g_{\text{\rm c},\pm}$}\label{ss: g_c pm}

Using~\eqref{M_pm equiv [0 infty)_rho times L_varpi}, we can also define the metric $g_{\text{\rm c},\pm}\equiv(\varkappa\rho)^{-2}\,(d\rho)^2+g_L$ on $\mathring M_\pm$ and its lift $\tilde g_{\text{\rm c},\pm}\equiv(\varkappa\tilde\rho)^{-2}\,(d\rho)^2+g_{\widetilde L}$ on $\mathring{\widetilde M}_\pm$. These are restrictions to the interiors of b-metrics, also denoted by $g_{\text{\rm c},\pm}$ and $\tilde g_{\text{\rm c},\pm}$. The b-metrics $\varkappa^2g_{\text{\rm c},\pm}$ and $\varkappa^2\tilde g_{\text{\rm c},\pm}$ are exact and cylindrical around the boundary (\Cref{ss: b-geometry}); in particular, the level hypersurfaces of $\rho$ and $\tilde\rho$ are totally geodesic for $g_{\text{\rm c},\pm}$ and $\tilde g_{\text{\rm c},\pm}$.

\begin{prop}\label{p: g_c pm & g_pm are quasi-isometric}
The metrics $g_{\text{\rm c},\pm}$ and $g_{\text{\rm b},\pm}$ are quasi-isometric on $\mathring M_\pm$; more precisely, 
\[
|{\cdot}|_{g_{\text{\rm c},\pm}}\le\sqrt{2(1+\varkappa^{-2}\|\eta\|_{L^\infty})}\,|{\cdot}|_{g_{\text{\rm b},\pm}}\;,\quad
|{\cdot}|_{g_{\text{\rm b},\pm}}\le \sqrt2\,|{\cdot}|_{g_{\text{\rm c},\pm}}\;.
\]
\end{prop}

\begin{proof}
Take any $p\in\mathring M_\pm$ and $u\in T_p\mathring M_\pm\equiv T_pM=\bfH_p\oplus\bfV_p$, and let $v=\bfV u$ and $w=\bfH u$. Then
\[
|u|_{g_{\text{\rm b},\pm}}^2=\rho^{-2}\omega(w)^2+|v|_{g_L}^2\;.
\]
By~\eqref{d rho_pm} and since $\omega(X)=1$, it follows that $u$ is the sum of the vectors
\[
v\mp|\varkappa|^{-1}\rho\eta(v)X_p\in\ker(d\rho)_p\;,\quad w\pm|\varkappa|^{-1}\rho\eta(v)X_p\in\bfH_p=\ker\varpi_{*p}\;,
\]
obtaining
\begin{align*}
|u|_{g_{\text{\rm c},\pm}}^2&=(\varkappa\rho)^{-2}\big(\pm|\varkappa|\omega(w)+\rho\eta(v)\big)^2+|v|_{g_L}^2\\
&=|u|_{g_{\text{\rm b},\pm}}^2+\varkappa^{-2}\eta(v)^2\pm2(|\varkappa|\rho)^{-1}\omega(w)\eta(v)\;.
\end{align*}
Then
\begin{align*}
|u|_{g_{\text{\rm c},\pm}}^2&\le|u|_{g_{\text{\rm b},\pm}}^2+2\varkappa^{-2}\eta(v)^2+\rho^{-2}\omega(w)^2
\le2(1+\varkappa^{-2}\|\eta\|_{L^\infty})|u|_{g_{\text{\rm b},\pm}}^2\;,\\[4pt]
|u|_{g_{\text{\rm b},\pm}}^2&\le|u|_{g_{\text{\rm c},\pm}}^2+\rho^{-2}\omega(w)^2\le2|u|_{g_{\text{\rm c},\pm}}^2\;.\qedhere
\end{align*}
\end{proof}

\subsection{Vector fields}\label{ss: vector fields}

Assume again that $\mathring{\widetilde M}_\pm$ and $\mathring M_\pm$ are endowed with $\tilde g_{\text{\rm b},\pm}$ and $g_{\text{\rm b},\pm}$. Recall that any $A\in\fX(M,\FF)$ induces a vector field  $A_\pm\in\fX(M,\FF)$, whose restriction to $\mathring M_\pm$ is also denoted by $A_\pm$ (\Cref{ss: simple fol flows,ss: suspension - homotheties}).

\begin{prop}\label{p: Y in fX_ub(mathring M_pm mathring FF_pm)}
$Y_\pm\in\fXub(\mathring M_\pm,\mathring\FF_\pm)$.
\end{prop}

\begin{proof}
This follows from the proof of \Cref{p: FF_pm is of bd geom} because $\partial_t$ corresponds to $\varkappa x\partial_x$ by the change of coordinate $t=\varkappa^{-1}\ln|x|$.
\end{proof}

\begin{lem}\label{l: V}
If $V\in\fX_\co(\FF)$, then $V_\pm\in\fXub(\mathring\FF_\pm)$.
\end{lem}

\begin{proof}
Consider the normal foliated charts of $\mathring\FF_\pm$, $\chi_{p,\pm}=(t_{x_0},y_q)=(t,y)$ on $U_{p,\pm}$, like in the proof of \Cref{p: FF_pm is of bd geom}, and the foliated charts of $\widetilde\FF$, $\chi_q=(x,y_q)=(x,y)$ on $U_q$, like in the proof of \Cref{l: d_FF(phi^t(p) p) ge inj_L}. Then $U_q=\varpi^{-1}(B_L(q,r))$, $U_{p,\pm}=U_q\cap\mathring M_\pm$ and $x=e^{\varkappa t}x_0$. Let $\partial_i=\partial_{y^i}$ and $\partial_I=\partial_{i_1}\cdots\partial_{i_m}$ for any multi-index $I=(i_1,\dots,i_m)$. Take a partition of unity subordinated to a finite open cover of the compact manifold $L$ by balls $B_L(q,r)$ ($q\in L$). By using the $\varpi$-lift of this partition of unity to $M$, it easily follows that we can assume $V$ is supported in some $U_q$. Thus we can write $V=f^i(x,y)\partial_i$ on $U_q$ for functions $f^i\in\Cinftyc(\R\times B)\equiv\Cinftyc(U_q)$, and write $V_\pm=h^i(t,y)\partial_i$ on $U_{p,\pm}$ for functions $h^i\in C^\infty(\R_\pm\times B)$. We have
\begin{equation}\label{partial_Ih^i(t y)}
\partial_Ih^i(t,y)=\partial_If^i(e^{\varkappa t}x_0,y)\;.
\end{equation}

\begin{claim}\label{cl: partial_t^k}
For $l\le k$ in $\N$, there are $c_{k,l}\in\N$ such that, on $U_{p,\pm}$, 
\[
\partial_t^k=\sum_lc_{k,l}x^l\partial_x^l=\sum_l(\pm1)^lc_{k,l}\rho^lX^l\;.
\]
\end{claim}

To simplify the notation, we define $c_{k,l}$ for all $k,l\in\Z$ by setting $c_{k,l}=0$ if $\min\{k,l\}<0$, $c_{0,0}=1$, and $c_{k,l}=\varkappa(lc_{k-1,l}+c_{k-1,l-1})$ if $\max\{k,l\}>0$. Note that $c_{k,l}=0$ if $l\le0<k$ or $l>k$. 

The first equality of \Cref{cl: partial_t^k} follows by induction on $k$. The case $k=1$ is true because $\partial_t=\varkappa x\partial_x$. If $k>1$ and the first equality holds for $k-1$, then
\begin{align*}
\partial_t^k&=\varkappa x\partial_x\sum_lc_{k-1,l}x^l\partial_x^l
=\sum_l\varkappa c_{k-1,l}\big(x^{l+1}\partial_x^{l+1}+x\big[\partial_x,x^l\big]\partial_x^l\big)\\
&=\sum_l\varkappa c_{k-1,l}\big(x^{l+1}\partial_x^{l+1}+lx^l\partial_x^l\big)
=\sum_l\varkappa(lc_{k-1,l}+c_{k-1,l-1})x^l\partial_x^l\;.
\end{align*}

The second equality of \Cref{cl: partial_t^k} holds because
\[
\tilde\rho^l\widetilde X^l=e^{lF(y)}|x|^l(e^{-F(y)}\partial_x)^l=|x|^l\partial_x^l=(\pm1)^lx^l\partial_x^l\;.
\]

By~\eqref{partial_Ih^i(t y)} and Claim~\ref{cl: partial_t^k},
\[
\partial_t^{k+1}\partial_Ih^i(t,y)=\sum_{l=1}^k(\pm1)^lc_{k,l}\rho^lX^l\partial_If^i(e^{\varkappa t}x_0,y)\;.
\]
Thus every function $|\partial_t^{k+1}\partial_Ih^i|$ is uniformly bounded on $\R_\pm\times B$ because $f^i\in\Cinftyc(\R\times B)\equiv\Cinftyc(U_q)$ and $X\in\fX(M)$.
\end{proof}

\begin{prop}\label{p: Z -> A}
For any $\epsilon>0$, there is some $A\in\fXcom(M,\FF)$ such that $A_\pm\in\fXub(\mathring M_\pm,\mathring\FF_\pm)$ and $A=Z$ on $T_\epsilon$.
\end{prop}

\begin{proof}
Let $\widetilde V=(0,\widetilde Z_x)\in\fX(\widetilde\FF)$, which projects to a vector field $V\in\fX(\FF)$ by~\eqref{T_gamma tilde phi_x^t = tilde phi_a_gamma x^t T_gamma}. For any $\lambda\in\Cinftyc(M)$ such that $0\le\lambda\le1$ and $\lambda=1$ on $T_\epsilon$, we have $V':=\lambda V\in\fX_\co(\FF)$. Then $V'_\pm\in\fXub(\mathring\FF_\pm)$ by \Cref{l: V}, $A:=Y+V'\in\fXcom(M,\FF)$, $A=Z$ on $T_\epsilon$, and $A_\pm=Y_\pm+V'_\pm\in\fXub(\mathring M_\pm,\mathring\FF_\pm)$ by \Cref{p: Y in fX_ub(mathring M_pm mathring FF_pm)}. 
\end{proof}

\section{Global objects on foliations with simple foliated flows}\label{s: global objects}

Consider the notation of \Cref{ss: folns almost w/o hol,ss: simple fol flows}, where $M$ is compact, $\FF$ is transversely oriented, and $\phi$ is transversely simple.

\subsection{Tubular neighborhoods of $M^0$}\label{ss: tubular neighborhoods of M^0}

In the following, for $L\in\pi_0M^0$ (the set of leaves in $M^0$), we have corresponding objects $\hat h_L$, $h_L$, $\Gamma_L$, $\pi_L:\widetilde L\to L$, $\varkappa_L$ and $a_{L,\gamma}$ (\Cref{ss: simple fol flows}). Consider also the corresponding suspension foliated manifold, $(M'_L,\FF'_L)$, and all other associated objects (\Cref{ss: suspension - homotheties,ss: transv simple flows - suspension,ss: omega - suspension,ss: def func of M^0 - suspension,ss: boundary-def func of M_pm - suspension,ss: g_M,ss: g_b pm,ss: g_c pm}). A prime and the subscript ``$L$'' is added to their notation; for instance, we have $\xi'_L=\{\xi^{\prime\,t}_L\}$, $Y'_L$, $M^{\prime\,0}_L$, $M^{\prime\,1}_L$, $\FF^{\prime\,1}_L$, $\varpi'_L$, $\rho'_L$, $T'_{L,\epsilon}$, $T^{\prime\,1}_{L,\epsilon}$, $X'_L$, $\omega'_L$, $\eta'_L$, $g_{M'_L}$ and $g_{\FF'_L}$. The corresponding disjoint unions or combinations, with $L$ running in $\pi_0M^0$, are denoted by $M'$, $\FF'$, $\xi'=\{\xi^{\prime\,t}\}$, $Y'$, $M^{\prime\,0}$, $M^{\prime\,1}$, $\FF^{\prime\,1}$, $\varpi'$, $\rho'$, $T'_\epsilon$, $T^{\prime\,1}_\epsilon$, $X'$, $\omega'$, $\eta'$, $g_{M'}$ and $g_{\FF'}$, removing the subscript ``$L$''.

By the Reeb's local stability, if $\epsilon>0$ is small enough, there is a tubular neighborhood of every $L$ in $M$, $\varpi_L:T_{L,\epsilon}\to L$, such that $T_{L,\epsilon}$ is diffeomorphic to $T'_{L,\epsilon}$, with $\varpi_L$ and $\FF|_{T_{L,\epsilon}}$ corresponding to $\varpi'_L$ and $\FF'_L|_{T'_{L,\epsilon}}$; we simply write $\varpi_L\equiv\varpi'_L$ and $\FF\equiv\FF'_L$ on $T_{L,\epsilon}\equiv T'_{L,\epsilon}$. We can assume the closures $\overline{T_{L,\epsilon}}$ are disjoint one another. Then the combination of the maps $\varpi_{L,\epsilon}$ is a tubular neighborhood of $M^0$ in $M$,
\[
\varpi\equiv\varpi':T_\epsilon:=\bigcup_LT_{L,\epsilon}\equiv T'_\epsilon\to M^0\equiv M^{\prime\,0}\;.
\]

\subsection{Collar neighborhoods of every $\partial M_l$}\label{ss: collar neighborhoods of partial M_l}

Given any connected component $M^1_l$ of $M^1$, consider only leaves $L\in\pi_0(M^0\cap\overline{M^1_l})\equiv\pi_0(\partial M_l)$. The notation $(M'_{L,l},\FF'_{L,l})$ is used for $(M'_{L,+},\FF'_{L,+})$ (resp., $(M'_{L,-},\FF'_{L,-})$) if the transverse orientation of $\FF_l$ along $L$ points inwards (resp., outwards), like the transverse orientation along $L$ of $\FF'_{L,+}$ (resp., $\FF'_{L,-}$). This kind of change is applied to the rest of notation concerning these foliated manifolds with boundary (\Cref{ss: transv simple flows - suspension,ss: boundary-def func of M_pm - suspension,ss: g_M,ss: g_b pm,ss: g_c pm,ss: vector fields}). For instance, we obtain $\xi'_{L,l}=\{\xi^{\prime\,t}_{L,l}\}$, $Y'_{L,l}$, $\varpi'_{L,l}$, $\rho'_{L,l}$, $\nu'_{L,l}$, $T'_{L,l,\epsilon}$, $\omega'_{\text{\rm b},L,l}$, $\eta'_{L,l}$, $g'_{\text{\rm b},L,l}$ and $g'_{\text{\rm c},L,l}$. Similarly, we have $(M^{\prime\,1}_{L,l},\FF^{\prime\,1}_{L,l})\equiv(\mathring M'_{L,l},\mathring\FF'_{L,l})$, whose Molino's description involves $\widetilde M^{\prime\,1}_{L,l}$, $\widetilde\FF^{\prime\,1}_{L,l}$, $h'_{L,l}:\Gamma_L\to\R$ and $D'_{L,l}:\widetilde M^{\prime\,1}_{L,l}\to\R$. We have $\mathring T'_{L,l,\epsilon}\equiv T'_{L,\epsilon}\cap M^{\prime\,1}_l=:T^{\prime\,1}_{L,l,\epsilon}$. The corresponding disjoint unions or combinations, with $L$ running in $\pi_0(\partial M_l)$, are denoted by $M'_l$, $\FF'_l$, $\xi'_l=\{\xi^{\prime\,t}_l\}$, $Y'_l$, $\varpi'_l$, $\rho'_l$, $\nu'_{L,l}$, $T'_{l,\epsilon}$, $\omega'_{\text{\rm b},l}$, $\eta'_l$, $g'_{\text{\rm b},l}$ and $g'_{\text{\rm c},l}$, deleting the subscript ``$L$''. In the same way, we have $M^{\prime\,1}_l$, $\FF^{\prime\,1}_l$ and $T^{\prime\,1}_{l,\epsilon}=T'_\epsilon\cap M^{\prime\,1}_l\equiv\mathring T'_{l,\epsilon}$.

Next, we delete ``$l$'' from this notation and use boldface for the corresponding disjoint unions or combinations for all $l$, obtaining $\bfM'$, $\bfFF'$, $\bfvarpi'$, $\bfrho'$, $\bfnu'$, $\bfT'_\epsilon$, $\bfomega'_{\text{\rm b}}$, $\bfeta'$, $\bfg'_{\text{\rm b}}$ and $\bfg'_\co$.
 
On the other hand, $\varpi_L:T_{L,\epsilon}\to L$ induces a collar neighborhood $\varpi_{L,l}:T_{L,l,\epsilon}\to L$ of the boundary component $L$ of $M_l$, and the identity $T_{L,\epsilon}\equiv T'_{L,\epsilon}$ induces an identity $T_{L,l,\epsilon}\equiv T'_{L,l,\epsilon}$, and we have $\varpi_{L,l}\equiv\varpi'_{L,l}$ and $\FF_l\equiv\FF'_{L,l}$ on $T_{L,l,\epsilon}\equiv T'_{L,l,\epsilon}$. Moreover $T^{\prime\,1}_{L,l,\epsilon}\equiv T^1_{L,l,\epsilon}:=T_{L,\epsilon}\cap M^1_l\equiv\mathring T_{L,l,\epsilon}$ and $T^{\prime\,1}_{L,l,\epsilon}\equiv T^1_{l,\epsilon}:=T_\epsilon\cap M^1_l\equiv\mathring T_{L,l,\epsilon}$.

The combination of the maps $\varpi_{L,l}$, with $L$ running in $\pi_0M_l$, is a collar neighborhood $\varpi_l\equiv\varpi'_l:T_{l,\epsilon}\equiv T'_{l,\epsilon}\to\partial M_l\equiv\partial M'_l$ of the boundary in $M_l$, where $\FF_l\equiv\FF'_l$. In turn, the combination of the maps $\varpi_l$ is a collar neighborhood
\[
\bfvarpi\equiv\bfvarpi':\bfT_\epsilon:=\bigsqcup_lT_{l,\epsilon}\equiv\bfT'_\epsilon\to\partial\bfM\equiv M^0\sqcup M^0
\]
of the boundary in $\bfM$, and we have $\bfFF\equiv\bfFF'$ on $\bfT_\epsilon\equiv\bfT'_\epsilon$.

\subsection{Globalization}\label{ss: globalization}

For fixed $0<\epsilon<\epsilon_0$ small enough, we can construct the following objects with standard arguments, using a partition of unity subordinated to the open cover $\{T_{\epsilon_0},M\setminus\overline{T_{\epsilon}}\}$ of $M$:
\begin{enumerate}[(A)]
\setcounter{enumi}{4}
\item\label{i-(A): Y'} For any $A'\in\fXcom(M',\FF')$ with $\overline{A'}=\overline{Y'}$, there is some $A\in\fXcom(M,\FF)$ with $\overline{A}=\overline{Z}$, $A\equiv A'$ on $T_\epsilon\equiv T'_\epsilon$ and $A=Z$ on $M\setminus T_{\epsilon_0}$.  Moreover $A$ induces a vector field $A_l\in\fX(M_l,\FF_l)$ (\Cref{ss: simple fol flows}), whose restriction to $\mathring M_l\equiv M^1_l$ is denoted in the same way. In particular, this applies to $Y'\in\fXcom(M',\FF')$, obtaining $Y\in\fX(M,\FF)$ with flow $\xi=\{\xi^t\}$ and $Y_l\in\fX(M_l,\FF_l)$ with flow $\xi_l=\{\xi^t_l\}$. We have $\Fix(\xi)=M^0$, and the orbits of $\xi$ agree with the fibers of $\varpi$ on $T_\epsilon\cap M^1$. Thus $\xi$ has no closed orbit in $T_\epsilon\cap M^1$.

\item\label{i-(A): Z'} Some $Z'\in\fXcom(M',\FF')$, with flow $\phi'=\{\phi^{\prime\,t}\}$, such that $\overline{Z'}=\overline{Y'}$, $Z'\equiv Z$ on $T_\epsilon\equiv T'_\epsilon$, and $Z'=Y'$ on $M'\setminus T'_{\epsilon_0}$.  This $Z'$ induces vector fields $Z'_{L,l}\in\fX(M'_{L,l},\FF'_{L,l})$ with flow $\phi'_{L,l}=\{\phi^{\prime\,t}_{L,l}\}$, and $Z'_l\in\fX(M'_l,\FF'_l)$ with flow $\phi'_l=\{\phi^{\prime\,t}_l\}$.

\item\label{i-(A): g_b l} A bundle-like metric $g_{\text{\rm b},l}$ of every $\FF^1_l\equiv\mathring\FF_l$ on $M^1_l\equiv\mathring M_l$ such that $g_{\text{\rm b},l}\equiv g'_{\text{\rm b},l}$ on $T^1_{l,\epsilon}\equiv T^{\prime\,1}_{l,\epsilon}$. Thus $g_{\text{\rm b},l}$ is the restriction to $\mathring M_l$ of a b-metric on $M_l$, also denoted by $g_{\text{\rm b},l}$. Let $\omega_{\text{\rm b},l}$ be the $g_{\text{\rm b},l}$-transverse volume form, defining the transverse orientation given by $\overline{Z_l}$; thus $\omega_{\text{\rm b},l}\equiv\omega'_{\text{\rm b},l}$ on $T^1_{l,\epsilon}\equiv T^{\prime\,1}_{l,\epsilon}$. Since $\omega'_{\text{\rm b},l}(Y'_l)=\omega'_{\text{\rm b},l}(Z'_l)=1$ (\Cref{ss: boundary-def func of M_pm - suspension}), we can assume $\omega_{\text{\rm b},l}(Y_l)=\omega_{\text{\rm b},l}(Z_l)=1$.

\item\label{i-(A): g_c l} A Riemannian metric $g_{\text{\rm c},l}$ on every $M^1_l\equiv\mathring M_l$ such that $g_{\text{\rm c},l}\equiv g'_{\text{\rm c},l}$ on $T^1_{l,\epsilon}\equiv T^{\prime\,1}_{l,\epsilon}$. Thus $g_{\text{\rm c},l}$ is the restriction to $\mathring M_l$ of a b-metric on $M_l$, also denoted by $g_{\text{\rm c},l}$, and the b-metric $\varkappa_L^2g_{\text{\rm c},l}$ is exact and cylindrical around every boundary component $L$ of $M_l$.

\item\label{i-(A): g_M} A Riemannian metric $g_M$ on $M$ such that $g_M\equiv g_{M'}$ on $T_\epsilon\equiv T'_\epsilon$, $g_M=g_{\text{\rm b},l}$ on every $M^1_l\setminus T_{\epsilon_0}$, and $g_M$ defines the same orthogonal complement of $T\FF$ as $g_{\text{\rm b},l}$ on every $M^1_l$.  We consider the bigrading of $\Lambda M$ defined by the $g_M$-orthogonal complement of the leaves (\Cref{s: dif forms on fold mfds}).

\item\label{i-(A): g_FF} A leafwise Riemannian metric $g_{\FF}$ of $\FF$ such that $g_{\FF}\equiv g_{\FF'}$ on $T_\epsilon\equiv T'_\epsilon$. We can assume it is induced by $g_M$ on $M$, and by $g_{\text{\rm b},l}$ and $g_{\text{\rm c},l}$ on every $M^1_l$. It induces a leafwise metric $g_{\FF_l}$ for every $\FF_l$.

\item\label{i-(A): d omega = eta wedge omega} Differential forms, $\omega\in C^\infty(M;\Lambda^{1,0})$ and $\eta\in C^\infty(M;\Lambda^{0,1})$, such that $\omega$ is the transverse volume form of $\FF$ with respect to $g_M$, and $d\omega=\eta\wedge\omega$. Thus $\ker\omega=T\FF$, $\eta=0$ on $M\setminus T_{\epsilon_0}$, and they extend the forms $\omega$ and $\eta$ we had on $T_\epsilon$. For every $L\in\pi_0M^0$, we may use the notation $\eta_L=\eta|_L$ and $\tilde\eta_L=\pi_L^*\eta_L=d_{\widetilde L}F_L$ for some $F_L\in C^\infty(\widetilde L)$. Moreover, $\eta=\eta_0$ on $T_\epsilon$ with the notation of \Cref{ss: perturbation vs bigrading} because this is true for every $\FF'_L$. 

\item\label{i-(A): rho} A defining function $\rho\equiv\rho'$ of $M^0$ in $T_{\epsilon_0}\equiv T'_{\epsilon_0}$.

\item\label{i-(A): rho_l} A boundary-defining function $\rho=\rho_l$ on every $M_l$ such that $\rho_l\equiv\rho'_l$ on $T_{l,\epsilon}\equiv T'_{l,\epsilon}$, and $\rho_l=1$ on $M^1_l\setminus T^1_{l,\epsilon_0}$. The level hypersurfaces of $\rho_l$ in $T^1_{l,\epsilon}$ are totally geodesic with respect to $g_{\text{\rm c},l}$. Let $\nu=\nu_l$ be the unique smooth trivialization of ${}_+N\partial M_l$ with $d\rho_l(\nu_l)=1$ (\Cref{ss: b-geometry}). Thus $\nu_l\equiv\nu'_l$ via $T_{l,\epsilon}\equiv T'_{l,\epsilon}$.

\end{enumerate}

From \Cref{p: FF_pm is of bd geom,p: (M_pm g_pm) satisfies (B),p: d(ln rho) in C^infty_ub(mathring M_pm Lambda^1),p: Y in fX_ub(mathring M_pm mathring FF_pm),p: Z -> A}, it easily follows that $\FF^1_l$ is of bounded geometry, $(M_l,g_{\text{\rm b},l})$ satisfies the properties~\ref{i: g is of bounded geometry} and~\ref{i: A'} of \Cref{ss: weighted b-Sobolev}, $d(\ln\rho_l)\in C^\infty_{\text{\rm ub}}(\mathring M_l;T^*\mathring M_l)$, and $Y_l\in\fXub(M^1_l,\FF^1_l)$ with respect to $g_{\text{\rm b},l}$, and we can assume $Z'_{L,l}\in\fXub(M^{\prime\,1}_{L,l},\FF^{\prime\,1}_{L,l})$ with respect to $g'_{\text{\rm b},L,l}$. So $Z_l\in\fXub(M^1_l,\FF^1_l)$ with respect to $g_{\text{\rm b},l}$. By \Cref{p: g_c pm & g_pm are quasi-isometric} and since $M^1_l\setminus T^1_{l,\epsilon}$ is compact, we also get that the metrics $g_{\text{\rm b},l}$ and $g_{\text{\rm c},l}$ are quasi-isometric on $M^1_l$; this also follows because both of these metrics are restrictions to $\mathring M_l$ of b-metrics on the compact manifold with boundary $M_l$.

By~\eqref{omega_pm = sign(omega(Z_pm)) rho_pm^-1omega}, we have $\omega=\sign(\omega(Z_l))\,\rho_l\omega_{\text{\rm b},l}$ on $\mathring M_l\cap\bfT_\epsilon\equiv M^1_l\cap T_\epsilon$. This equality is also true on $M^1_l\setminus T_{\epsilon_0}$, where $g_M=g_{\text{\rm b},l}$ and $\rho_l=1$. Indeed, we can choose $\rho_l$ so that this equality holds on the whole of $\mathring M_l\equiv M^1_l$. So
\[
d\omega=\sign(\omega(Z_l))\,d\rho_l\wedge\omega_{\text{\rm b},l}=d\rho_l\wedge\rho_l^{-1}\omega=d(\ln\rho_l)\wedge\omega
\]
on $\mathring M_l\equiv M^1_l$, yielding
\begin{equation}\label{eta_0 = d_0,1(ln rho_l)}
\eta_0=d_{0,1}(\ln\rho_l)\equiv d_{\FF_l}(\ln\rho_l)\;.
\end{equation}

Taking combinations of the above objects on the manifolds $M_l$, we get a boundary-defining function $\bfrho$ on $\bfM$, a trivialization $\bfnu$ of ${}_+N\partial\bfM$, real $1$-forms $\bfomega_{\text{\rm b}}$ and $\bfeta$, and b-metrics $\bfg_{\text{\rm b}}$ and $\bfg_\co$. They agree with $\bfrho'$, $\bfnu'$, $\bfomega'_{\text{\rm b}}$, $\bfeta'$, $\bfg'_{\text{\rm b}}$ and $\bfg'_\co$ on $\bfT_\epsilon\equiv\bfT'_\epsilon$.

\subsection{The components of $M^1$}\label{ss: components of M^1}

Recall that every $\FF^1_l\equiv\mathring\FF_l$ on $M^1_l\equiv\mathring M_l$ is a transversely complete $\R$-Lie foliation, where this transverse structure is defined by $Z_l\in\fXcom(M^1_l,\FF^1_l)$. Of course, the transverse orientation of $\FF^1_l$ defined by $Z_l$ may not agree with the original transverse orientation of $\FF$.

The Fedida's description of $\FF^1_l$ is given by a regular covering $\pi_l:\widetilde M^1_l\to M^1_l$ with group of deck transformations $\Gamma_l$, a holonomy monomorphism $h_l:\Gamma_l\to\R$ and a developing map $D_l:\widetilde M^1_l\to\R$ (\Cref{ss: transverse structures,ss: complete R-Lie folns,}). Note that $\Gamma_l$ has finite rank because $M^1_l\equiv\mathring M_l$ and $M_l$ is compact. Recall that the action of any $\gamma\in\Gamma_l$ on $\widetilde M^1_l$ is denoted by $\tilde p\mapsto\gamma\cdot\tilde p$ or by $T_\gamma$.

Let $Y_l$ and $\xi_l=\{\xi_l^t\}$ denote the restrictions of $Y$ and $\xi$ to every $M^1_l$. Let $\widetilde\FF^1_l$, $\widetilde Y_l$, $\widetilde Z_l$, $\tilde\xi_l=\{\tilde\xi_l^t\}$ and $\tilde\phi_l=\{\tilde\phi_l^t\}$ be the lifts to $\widetilde M^1_l$ of $\FF^1_l$, $Z_l$ and $\phi_l$, respectively. Recall from \Cref{ss: Schwartz kernels} that $\widetilde Z_l$ is $\Gamma_l$-invariant and $D_l$-projectable, and $\tilde\phi_l$ is $\Gamma_l$-equivariant. Moreover we can assume $D_{l*}\widetilde Z_l=\partial_x$, where $x$ denotes the canonical global coordinate of $\R$, and therefore $\tilde\phi_l$ corresponds via $D_l$ to the flow $\bar\phi_l$ on $\R$ defined by $\bar\phi_l^t(x)=t+x$. So $D_l$ restricts to diffeomorphisms between the orbits of $\tilde\phi_l$ and $\R$.

\begin{prop}\label{p: widetilde M^1_l equiv L_l times R}
Given any leaf $L_l$ of $\FF^1_l$, there is a left action of $\Gamma_l$ on $L_l$ and there is an identity $\widetilde M^1_l\equiv \R\times L_l$ such that:
\begin{enumerate}[{\rm(i)}]

\item\label{D_l = the left factor proj} $D_l$ is the left-factor projection;

\item\label{i: widetilde Y_l equiv (partial_x 0)} $\widetilde Y_l\equiv(\partial_x,0)$ and $\tilde\xi_l^t(x,y)=(t+x,y)$;

\item\label{i: gamma . (x,y) = (h_l(gamma) + x gamma . y)} the action of $\Gamma_l$ on $\widetilde M^1_l$ is given by $\gamma\cdot(x,y)=(h_l(\gamma)+x,\gamma\cdot y)$; and

\item\label{i: gamma . y = y => gamma = e} there is some compact $K_l\subset M^1_l$ so that, if $\gamma\cdot y=y$ for some $\gamma\in\Gamma_l$ and $y\in L_l\setminus K_l$, then $\gamma=e$.

\end{enumerate}
\end{prop}

\begin{proof}
Since $\widetilde Y_l$ is projectable by $D_l$ to $\partial_x$ because $\overline{\widetilde Y_l}=\overline{\widetilde Z_l}$, it follows that $D_l$ also restricts to diffeomorphisms of the $\tilde\xi_l$-orbits to $\R$. So, given any leaf $\widetilde L_l$ of $\widetilde\FF^1_l$ over $L_l$, we get $\widetilde M^1_l\equiv\R\times\widetilde L_l\equiv \R\times L_l$ such that~\ref{D_l = the left factor proj} and~\ref{i: widetilde Y_l equiv (partial_x 0)} hold. 

The action of every $\gamma\in\Gamma_l$ on $(x,y)\in \R\times L_l\equiv\widetilde M^1_l$ can be written as $\gamma\cdot(x,y)=(h_l(\gamma)+x,T_\gamma(x,y))$ for some smooth map $T_\gamma:\R\times L_l\to L_l$. Then, since the flow $\tilde\xi_l^t$ is $\Gamma_l$-equivariant, it easily follows that $T_\gamma(x,y)=T_\gamma(t+x,y)$. So $T_\gamma(x,y)$ is independent of $x$, and therefore it can be written as $\gamma\cdot y$. It is easy to check that this defines a left $\Gamma_l$-action on $L_l$, and~\ref{i: gamma . (x,y) = (h_l(gamma) + x gamma . y)} follows.

Let us prove~\ref{i: gamma . y = y => gamma = e}. If $\gamma\cdot y=y$ for some $\gamma\in\Gamma_l\setminus\{e\}$ and $y\in L_l$, then we easily compute $\gamma\cdot\tilde\xi^t(x,y)=\tilde\xi^{h_l(\gamma)+t}(x,y)$ for all $x,t\in\R$. Thus the $\tilde\xi^t$-orbit of $(x,y)$ is invariant by the action of $\gamma$, and therefore the $\xi^t$-orbit of $[x,y]$ is closed because $\gamma\ne e$. Since $Y\equiv Y'$ on $T_\epsilon\equiv T'_\epsilon$, it follows that $y\in L_l\setminus T_\epsilon$, and $M^1_l\setminus T_\epsilon$ is compact in $M^1_l$.
\end{proof}

\begin{rem}  
In \Cref{p: widetilde M^1_l equiv L_l times R}, the projection $L_l\to\Gamma_l\backslash L_l$ may not be a covering map, and therefore $(M^1_l,\FF^1_l)$ may not be given by a suspension. According to its proof, a point $y\in L_l$ is fixed by some $\gamma\in\Gamma_l\setminus\{e\}$ just when $\R\times\{y\}$ projects to a closed orbit of $\xi^t$ in $M^1_l$ whose group of periods contains $h_l(\gamma)$.
\end{rem}

According to \Cref{p: widetilde M^1_l equiv L_l times R}, we may use the notation $[x,y]=\pi_l(x,y)\in M^1_l$ for $(x,y)\in\R\times L_l\equiv\widetilde M^1_l$, and the action of every $\gamma\in\Gamma_l$ on $L_l$ may be also denoted by $T_\gamma$. Like in~\eqref{tilde phi^t(tilde y x) widetilde Z} and~\eqref{T_gamma tilde phi_x^t = tilde phi_a_gamma x^t T_gamma}, we get
\begin{equation}\label{tilde phi_l^t(x tilde y)}
\tilde\phi_l^t(x,\tilde y)=(t+x,\tilde\phi_{l,x}^t(\tilde y))\;,\quad\widetilde Z_l=(\partial_x,\widetilde Z_{l,x})\;,
\end{equation}
for some smooth families, $\{\,\tilde\phi_{l,x}^t\mid x,t\in\R\,\}\subset\Diffeo(L_l)$ and $\{\,\widetilde Z_{l,x}\mid x\in\R\,\}\subset\fX(L_l)$, such that
\[
T_\gamma\tilde\phi_{l,x}^t=\tilde\phi_{l,h_l(\gamma)+x}^tT_\gamma\;,\quad T_{\gamma*}\widetilde Z_x=\widetilde Z_{h_l(\gamma)+x}\;.
\]

Let $c$ be a closed orbit of $\phi_l$ with period $t_0$, and let $p=[x,y]\in c$ and $\tilde p=(x,y)\in\widetilde M^1_l\equiv \R\times L_l$. Then $k=t_0/\ell(c)\in\Z$ and there is a unique $\gamma_0\in\Gamma_l$ such that $\tilde\phi_l^{t_0}(\tilde p)=\gamma_0\cdot\tilde p$. Using~\eqref{tilde phi_l^t(x tilde y)} and \Cref{p: widetilde M^1_l equiv L_l times R}~\ref{i: gamma . (x,y) = (h_l(gamma) + x gamma . y)}, it easily follows that $t_0=h_l(\gamma_0)$ and $\tilde\phi_{l,x}^{t_0}(y)=\gamma_0\cdot y$; i.e., $y$ is a fixed point of the diffeomorphism $T_{\gamma_0}^{-1}\tilde\phi^{t_0}_{l,x}$ of $L_l$. Moreover $y$ is simple if and only if $c$ is simple, and, in this case, $\epsilon_y(T_{\gamma_0}\tilde\phi_{l,x}^{t_0})=\epsilon_c(k,\phi)=\epsilon_c(k)$.

We have $\tilde\omega_{\text{\rm b},l}:=\pi_l^*\omega_{\text{\rm b},l}=D_l^*dx\equiv dx$ because $D_{l*}\widetilde Z_l=\partial_x$ and $\omega_{\text{\rm b},l}(Z_l)=1$ (\Cref{ss: globalization}).

\subsection{Metric properties of the components of $M^1$}\label{ss: metric properties of M^1}

With the notation of \Cref{ss: collar neighborhoods of partial M_l,ss: globalization,ss: components of M^1}, for leaves $L\subset M^0\cap\overline{M^1_l}$ and $0<\epsilon'\le\epsilon$, the open subsets
\[
\widetilde T^{\prime\,1}_{L,l,\epsilon'}=\pi_{M'_{L,l}}^{-1}(T^{\prime\,1}_{L,l,\epsilon'})\subset\widetilde M^{\prime\,1}_{L,l}\;,\quad
\widetilde T^1_{L,l,\epsilon'}=\pi_l^{-1}(T^1_{L,l,\epsilon'})\subset\widetilde M^1_l\;,
\]
are invariant by $\Gamma_L$ and $\Gamma_l$, respectively. Let $\tilde\rho_l=\pi_l^*\rho_l$ and $M^1_{l,\epsilon'}=M^1_l\setminus T_{l,\epsilon'}$, which is a connected compact smooth submanifold with boundary of $M^1_l$. Then $\widetilde T^1_{l,\epsilon'}:=\pi_l^{-1}(T^1_{l,\epsilon'})=\{\tilde\rho_l<\epsilon'\}$ is a $\Gamma_l$-invariant open subspace of $\widetilde M^1_l$, and $\pi_l:\widetilde M^1_{l,\epsilon'}:=\widetilde M^1_l\setminus\widetilde T^1_{l,\epsilon'}\to M^1_{l,\epsilon'}$ is a regular $\Gamma_l$-covering.

Let $d_l$ denote the length-metric on $M^1_l$ defined by $g_{\text{\rm b},l}$. Let $\tilde g_{\text{\rm b},l}$ and $\tilde g_{\text{\rm c},l}$ be the lifts to $\widetilde M^1_l$ of $g_{\text{\rm b},l}$ and $g_{\text{\rm c},l}$. Both of them induce the same leafwise metric $g_{\widetilde\FF_l}$ of $\widetilde\FF_l$, which is the lift of $g_{\FF_l}$. Let $\tilde d_l$ and $\tilde d_{l,\epsilon'}$ denote the length-metrics on $\widetilde M^1_l$ and $\widetilde M^1_{l,\epsilon'}$ defined by $\tilde g_{\text{\rm b},l}$. Similarly, let $\tilde d_{\text{\rm c},l}$ and $\tilde d_{\text{\rm c},l,\epsilon'}$ be the length-metrics on $\widetilde M^1_l$ and $\widetilde M^1_{l,\epsilon'}$ defined by $\tilde g_{\text{\rm c},l}$. Since $g_{\text{\rm b},l}$ and $g_{\text{\rm c},l}$ are quasi-isometric (\Cref{ss: globalization}), the metrics $\tilde g_{\text{\rm b},l}$ and $\tilde g_{\text{\rm c},l}$ are also quasi-isometric. Therefore there is some $C_1\ge1$ such that, for all $\tilde p,\tilde q\in\widetilde M^1_l$,
\begin{equation}\label{C_1^-1 tilde d_l(tilde p tilde q) le tilde d_l epsilon'(tilde p tilde q) le C_1 tilde d_l(tilde p tilde q)}
C_1^{-1}\tilde d_l(\tilde p,\tilde q)\le \tilde d_{c,l}(\tilde p,\tilde q)\le C_1\tilde d_l(\tilde p,\tilde q)\;.
\end{equation}
On the other hand, $\tilde d_l\le \tilde d_{l,\epsilon'}$ on $\widetilde M^1_{l,\epsilon'}$.

\begin{lem}\label{l: tilde d_c l = tilde d_c l epsilon'}
We have $\tilde d_{\text{\rm c},l}=\tilde d_{\text{\rm c},l,\epsilon'}$ on $\widetilde M^1_{l,\epsilon'}$.
\end{lem}

\begin{proof}
It is enough to show that any $\tilde g_{\text{\rm c},l}$-geodesic segment with end-points in $\widetilde M^1_{l,\epsilon'}$ is contained in $\widetilde M^1_{l,\epsilon'}$ ($\widetilde M^1_{l,\epsilon'}$ is $\tilde g_{\text{\rm c},l}$-convex). This follows easily using that the level hypersurfaces of $\tilde\rho_l$ are $\tilde g_{\text{\rm c},l}$-totally geodesic because the level hypersurfaces of $\rho_l$ in $T^1_{l,\epsilon}$ are $g_{\text{\rm c},l}$-totally geodesic (\ref{i-(A): rho_l} of \Cref{ss: globalization}).
\end{proof}

Let $|{\cdot}|=|{\cdot}|_l:\Gamma_l\to\N_0$ and $|{\cdot}|=|{\cdot}|_L:\Gamma_L\to\N_0$ be the word length functions induced by any choice of finite sets of generators of $\Gamma_l$ and $\Gamma_L$. By the compactness of $M^1_{l,\epsilon'}$, there is some $C_2=C_2(\epsilon')\ge1$ such that, for all $\gamma\in\Gamma_l$ and $\tilde p\in\widetilde M^1_{l,\epsilon'}$,
\begin{equation}\label{C_2^-1 |gamma| le tilde d_l epsilon'(tilde p gamma . tilde p) le C_2 |gamma|}
C_2^{-1}|\gamma|\le \tilde d_{l,\epsilon'}(\tilde p,\gamma\cdot\tilde p)\le C_2|\gamma|\;.
\end{equation}
Since $g_{\text{\rm b},l}$ and $g_{\text{\rm c},l}$ are quasi-isometric on $M^1_l$, it follows from~\eqref{C_1^-1 tilde d_l(tilde p tilde q) le tilde d_l epsilon'(tilde p tilde q) le C_1 tilde d_l(tilde p tilde q)},~\eqref{C_2^-1 |gamma| le tilde d_l epsilon'(tilde p gamma . tilde p) le C_2 |gamma|} and \Cref{l: tilde d_c l = tilde d_c l epsilon'} that there is some $C_3=C_3(\epsilon')\ge1$ such that, for all $\gamma\in\Gamma_l$ and $\tilde p\in\widetilde M^1_{l,\epsilon'}$,
\begin{equation}\label{C_3^-1 |gamma| le tilde d_l(tilde p gamma . tilde p) le C_3 |gamma|}
C_3^{-1}|\gamma|\le \tilde d_l(\tilde p,\gamma\cdot\tilde p)\le C_3|\gamma|\;.
\end{equation}

\begin{rem}\label{r: relation between the descriptions}
For any leaf $L\subset M^0\cap\overline{M^1_l}$, the given descriptions of $\FF$ on $T^1_{L,\epsilon}$ and $M^1_l$ have the following relation, whose proof is omitted because it will not be used. There is a monomorphism $H_{L,l}:\Gamma_L\to\Gamma_l$ such that, for every connected component $\widetilde T^1_{L,l,\epsilon',0}$ of $\widetilde T^1_{L,l,\epsilon'}$, the identity $T^{\prime\,1}_{L,l,\epsilon'}\equiv T^1_{L,l,\epsilon'}$ can be lifted to an $H_{L,l}$-equivariant identity $\widetilde T^{\prime\,1}_{L,l,\epsilon'}\equiv\widetilde T^1_{L,l,\epsilon',0}$, which is locally equivariant with respect to the local flows defined by $\tilde\xi'_L$ on $\widetilde T^{\prime\,1}_{L,\epsilon',\pm}$ and $\tilde\xi_l$ on $\widetilde T^1_{L,l,\epsilon',0}$, and so that $D_l$ corresponds to $D'_{L,l}$.
\end{rem}

\chapter{Conormal leafwise reduced cohomology}\label{ch: conormal}

\section{Conormal sequence of leafwise currents}\label{s: conormal seq of leafwise currents}

Let $\FF$ be a transversely orientable smooth foliation of codimension one on a closed manifold $M$ satisfying the conditions~\ref{i: almost w/o hol} and~\ref{i: homotheties} of \Cref{ss: simple fol flows}. Then $M^0$ is determined by $\FF$ in the cases~\ref{i: minimal Lie foln}--\ref{i: all folns FF_l are models (2)} of \Cref{ss: simple fol flows}, whereas $M^0$ must be also given in the case~\ref{i: fiber bundle}. The compactness condition on $M$ is assumed for the sake of simplicity, but all concepts, results and arguments of this section have straightforward extensions to the case where $M$ is not compact and $M^0$ is compact, using compactly supported versions or versions without support restrictions of the spaces of leafwise currents that will be considered. The compactly supported versions, in the non-compact case, will be used in the arguments.

Since $\Diff^1(\FF;\Lambda\FF)\subset\Diff^1(M,M^0;\Lambda\FF)$, the graded LCHS\index{$I(\FF)$} 
\[
I(\FF)=I\Lambda^\bullet(\FF):=I(M,M^0;\Lambda\FF)
\]
becomes a topological complex with $d_\FF$ (\Cref{ss: diff opers on conormal distribs,ss: leafwise complex}). If we take coefficients in some leafwise flat vector bundle $E$, then the notation $I(\FF;E)$ will be used, and all other notations will be modified in the same way. We may even consider $I(\FF;E)$ for an arbitrary vector bundle $E$, missing the leafwise differential map $d_\FF$.

The topological complex $(I(\FF),d_\FF)$ produces the \emph{conormal leafwise cohomology} and \emph{conormal leafwise reduced cohomology} of $\FF$ (or of $(\FF,M^0)$ when $M^0$ is not determined by $\FF$), denoted by $H^\bullet I(\FF)$\index{$H^\bullet I(\FF)$} and $\bar H^\bullet I(\FF)$,\index{$\bar H^\bullet I(\FF)$} which are LCSs (\Cref{ss: top complexes}). The image and kernel of $d_\FF$ in $I(\FF)$ are denoted by $BI(\FF)$ and $ZI(\FF)$, and we write $\bar BI(\FF)=\overline{BI(\FF)}$.

The LCHSs\index{$I^{(s)}(\FF)$}
\[
I^{(s)}(\FF)=I^{(s)}\Lambda^\bullet(\FF):=I^{(s)}(M,M^0;\Lambda\FF)\quad(s\in\R)
\]
also become topological complexes with $d_\FF$ (\Cref{ss: diff opers on conormal distribs}). The notation $H^\bullet I^{(s)}(\FF)$,\index{$H^\bullet I^{(s)}(\FF)$} $\bar H^\bullet I^{(s)}(\FF)$,\index{$\bar H^\bullet I^{(s)}(\FF)$} $BI^{(s)}(\FF)$,\index{$BI^{(s)}(\FF)$} $ZI^{(s)}(\FF)$\index{$ZI^{(s)}(\FF)$} and $\bar BI^{(s)}(\FF)$\index{$\bar BI^{(s)}(\FF)$} is used as before. We have continuous inclusion maps (\Cref{ss: conormal - Sobolev order - compact})
\begin{equation}\label{j_s}
j_s:I^{(s)}(\FF)\hookrightarrow I(\FF)\;,\quad 
j_{s,s'}:I^{(s)}(\FF)\hookrightarrow I^{(s')}(\FF)\quad(s'\le s)\;.
\end{equation}
The induced homomorphism in cohomology and reduced cohomology are denoted by $j_{s*}$, $j_{s,s'*}$, $\bar\jmath_{s*}$ and $\bar\jmath_{s,s'*}$. The homomorphism $j_{s,s'*}$ and $\bar\jmath_{s*}$ form inductive spectra, giving rise to inductive limits as $s\downarrow-\infty$. The maps $j_{s*}$ and $\bar\jmath_{s*}$ induce canonical continuous linear isomorphisms (\Cref{s: injj lims in cohom and reduced cohom}),
\begin{equation}\label{varinjlim H^bullet I^(s)(FF) cong H^bullet I(FF)}
\left\{
\begin{gathered}
\tilde\jmath_*:=\varinjlim j_{s*}:
\widetilde H^\bullet I(\FF):=\varinjlim H^\bullet I^{(s)}(\FF)\xrightarrow{\cong} H^\bullet I(\FF)\;,\\
\hat\jmath_*:=\varinjlim\bar\jmath_{s*}:
\widehat H^\bullet I(\FF):=\varinjlim\bar H^\bullet I^{(s)}(\FF)\xrightarrow{\cong}\bar H^\bullet I(\FF)\;.
\end{gathered}
\right.
\end{equation}
The canonical maps of the steps to the inductive limits are denoted by \index{$\tilde\jmath_{s*}$} \index{$\hat\jmath_{s*}$}
\[
\tilde\jmath_{s*}:H^\bullet I^{(s)}(\FF)\to\widetilde H^\bullet I(\FF)\;,\quad
\hat\jmath_{s*}:\bar H^\bullet I^{(s)}(\FF)\to\widehat H^\bullet I(\FF)\;.
\]

The graded LCHSs,\index{$J(\FF)$} \index{$K(\FF)$}
\[
J(\FF)=J\Lambda^\bullet(\FF):=J(M,M^0;\Lambda\FF)\;,\quad K(\FF)=K\Lambda^\bullet(\FF):=K(M,M^0;\Lambda\FF)\;,
\]
also become topological complexes with $d_\FF$ (\Cref{ss: Diff(M) on the conormal seq}). The above kind of notation is also used for the induced spaces: $BJ(\FF)$, $ZJ(\FF)$, $\bar BJ(\FF)$ and $H^\bullet J(\FF)$,\index{$H^\bullet J(\FF)$} $\bar H^\bullet J(\FF)$,\index{$\bar H^\bullet J(\FF)$} and the same for $K(\FF)$.\index{$H^\bullet K(\FF)$}

Similarly, we have topological complexes $J^{(s)}(\FF)$, $J^m(\FF)$ and $K^{(s)}(\FF)$ ($s,m\in\R$) with $d_\FF$ (\Cref{ss: Diff(M) on the conormal seq}). The analogs of the inclusion maps~\eqref{j_s} for the spaces $J^{(s)}(\FF)$ and $K^{(s)}(\FF)$ are denoted in the same way. The induced homomorphisms in cohomology and reduced cohomology form inductive spectra. Their inductive limits, denoted by $\widetilde H^\bullet K(\FF)$, $\widehat H^\bullet K(\FF)$, $\widetilde H^\bullet J(\FF)$ and $\widehat H^\bullet J(\FF)$, satisfy analogs of~\eqref{varinjlim H^bullet I^(s)(FF) cong H^bullet I(FF)} (proved with the same arguments). In fact, in the case of $K(\FF)$, we have  canonical TVS-identities (\Cref{c: HK(FF) equiv bigoplus_k H_-k-1(M^0)}),
\begin{equation}\label{HK(FF) equiv bar HK(FF)}
\left\{
\begin{gathered}
H^\bullet K(\FF)\equiv\bar H^\bullet K(\FF)\;,\quad H^\bullet K^{(s)}(\FF)\equiv\bar H^\bullet K^{(s)}(\FF)\;,\\
\widetilde H^\bullet K^{(s)}(\FF)\equiv H^\bullet K(\FF)\;,\quad
\widehat H^\bullet K^{(s)}(\FF)\equiv\bar H^\bullet K(\FF)\;.
\end{gathered}
\right.
\end{equation}

There are also continuous inclusion maps (\Cref{ss: J(M L)})
\begin{equation}\label{j_m}
\left\{
\begin{gathered}
j_m:J^m(\FF)\hookrightarrow J(\FF)\;,\quad j_{m,m'}:J^m(\FF)\hookrightarrow J^{m'}(\FF)\quad(m'\le m)\;,\\
j_{s,m}:J^{(s)}(\FF)\hookrightarrow J^m(\FF)\quad(m<s-n/2-1)\;,\\
j_{m,s}:J^m(\FF)\hookrightarrow J^{(s)}(\FF)\quad(s\le m,0)\;, 
\end{gathered}
\right.
\end{equation}
denoted like in~\eqref{j_s} with some abuse of notation. The homomorphisms induced by the maps $j_{m,m'}$ in cohomology and reduced cohomology form inductive spectra whose inductive limits as $m\downarrow-\infty$ agree with the previous ones for $J(\FF)$, and the maps $j_m$ induce a continuous linear isomorphism analogous to~\eqref{varinjlim H^bullet I^(s)(FF) cong H^bullet I(FF)}.

There are similar constructions for the spaces of the symbol-order filtration of $I(\FF)$ and $K(\FF)$, with similar properties, but they will not be used here.

The \emph{leafwise conormal exact sequence} of $\FF$ is the bottom row of~\eqref{CD: conormal seqs} with $\Lambda\FF$, 
\begin{equation}\label{leafwise conormal exact seq}
0\to K(\FF) \xrightarrow{\iota} I(\FF) \xrightarrow{R} J(\FF)\to0\;.
\end{equation}
Besides being exact in the category of continuous linear maps between LCSs, it is compatible with $d_\FF$. The exactness of the induced sequences,
\begin{gather}
0\to H^\bullet K(\FF) \xrightarrow{\iota_*} H^\bullet I(\FF) \xrightarrow{R_*} H^\bullet J(\FF)\to0\;,
\label{exact seq in cohom - conormal}\\
0\to H^\bullet K(\FF) \xrightarrow{\bar\iota_*} \bar H^\bullet I(\FF) \xrightarrow{\bar R_*} \bar H^\bullet J(\FF)\to0\;,
\label{exact seq in reduced cohom - conormal}
\end{gather}
will be proved in \Cref{s: short exact seq - conormal}; in particular, this shows \Cref{t: intro - reduced conormal cohomology exact sequence}.

Concerning notation, the subscript ``$s$'' may be added to the notation of cochain maps between the topological complexes $K^{(s)}(\FF)$, $I^{(s)}(\FF)$ or $J^{(s)}(\FF)$, like \index{$\iota_s$} \index{$R_s$}
\begin{equation}\label{iota_s}
\iota_s=\iota:K^{(s)}(\FF)\to I^{(s)}(\FF)\;,\quad R_s=R:I^{(s)}(\FF)\to J^{(s)}(\FF)\;.
\end{equation}
The subscript ``$s$'' may be also added to the elements of their cohomologies or reduced cohomologies: $[\alpha]_s\in H^\bullet I^{(s)}(\FF)$ \index{$[\alpha]_s$} and $\overline{[\alpha]}_s\in\bar H^\bullet I^{(s)}(\FF)$ \index{$\overline{[\alpha]}_s$} for $\alpha\in ZI^{(s)}(\FF)$.

\section{Injective limits in cohomology and reduced cohomology}\label{s: injj lims in cohom and reduced cohom}

The purpose of this section is to prove that the maps~\eqref{varinjlim H^bullet I^(s)(FF) cong H^bullet I(FF)} are isomorphisms. The details are given for the case of $\bar H^\bullet I(\FF)$. Some remarks indicate how to modify the arguments to show the simpler case of $H^\bullet I(\FF)$.

\subsection{Injectivity of $\hat\jmath_*$}\label{ss: injectivity of hat jmath_*}

Take any element in $\ker\hat\jmath_*$, which is of the form $\hat\jmath_{s*}(\overline{[\alpha]}_s)$ for some $\overline{[\alpha]}_s\in\bar H^\bullet I^{(s)}(\FF)$. Then there is some net $\varphi_l\in I(\FF)$ such that $\alpha=\lim_ld_\FF\varphi_l$ in $I(\FF)$. We can assume $\varphi_l\in C^\infty(M;\Lambda\FF)$ by the density of $C^\infty(M;\Lambda\FF)$ in $I(\FF)$ (\Cref{ss: conormal - Sobolev order - compact}). The set $\{\alpha,d_\FF\varphi_l\}_l$ is compact in $I(\FF)$. Then $\{\alpha,d_\FF\beta_l\}_l$ is contained and compact in some step $I^{(s')}(\FF)$ ($s'\le s$) because $I(\FF)$ is compactly retractive (\Cref{ss: conormal - Sobolev order - compact}). Thus $\alpha=\lim_ld_\FF\varphi_l$ in $I^{(s')}(\FF)$; otherwise, using that $\{\alpha,d_\FF\varphi_l\}_l$ is compact in $I^{(s')}(\FF)$, it is easy to find a subnet $d_\FF\varphi_{l_k}$ convergent to some $\beta\ne\alpha$ in $I^{(s')}(\FF)$, which contradicts the continuity of $j_{s'}:I^{(s')}(\FF)\to I(\FF)$ and the convergence $d_\FF\varphi_l\to\alpha$ in $I(\FF)$. (Indeed, we can assume $d_\FF\phi_l$ is a sequence because $I^{(s')}(\FF)$ is a Fr\'echet space.) So $\overline{[\alpha]}_{s'}=0$ in $\bar H^\bullet I^{(s')}(\FF)$, and therefore $\hat\jmath_{s*}(\overline{[\alpha]}_s)=\hat\jmath_{s'*}(\overline{[\alpha]}_{s'})=0$.

\begin{rem}\label{r: injectivity of tilde jmath_*}
To prove injectivity of $\tilde\jmath_*$, take some $\tilde\jmath_{s*}(\overline{[\alpha]}_s)$ in $\ker\tilde\jmath_*$. Now modify the above argument by using cohomology classes, and taking an element $\varphi\in I(\FF)$ with $d_\FF\varphi=\alpha$ instead of a net $\varphi_l$. Then $\varphi$ and $\alpha$ are in some step $I^{(s')}(\FF)$ ($s'\le s$), yielding $[\alpha]_{s'}=0$ in $H^\bullet I^{(s')}(\FF)$, and therefore $\hat\jmath_{s*}([\alpha]_s)=\hat\jmath_{s'*}([\alpha]_{s'})=0$.
\end{rem}

\subsection{Surjectivity of $\hat\jmath_*$}\label{ss: surjectivity of hat jmath_*}

For any $\overline{[\alpha]}\in\bar H^\bullet I(\FF)$, there is some $s$ such that $\alpha\in I^{(s)}(\FF)$, and therefore $\alpha\in ZI^{(s)}(\FF)$. Hence the element $\overline{[\alpha]}_s\in\bar HI^{(s)}(\FF)$ is defined, and the element $\hat\jmath_{s*}(\overline{[\alpha]}_s)\in\widehat H^\bullet I^{(s)}(\FF)$ is mapped to $\overline{[\alpha]}$ by $\hat\jmath_*$.

\begin{rem}\label{r: surjectivity of tilde jmath_*}
To prove the surjectivity of $\tilde\jmath_*$, simply modify the argument by using cohomology classes instead of reduced cohomology classes.
\end{rem}

\section{Description of $H^\bullet K(\FF)$}\label{s: HK(FF)}

Consider also the notation of \Cref{ss: globalization}. For every $z\in\C$, we have the Witten's complex $d_z=d+z\,{\eta\wedge}$ on $C^\infty(M^0;\Lambda)$, whose cohomology is denoted by $H^\bullet_z(M^0)$ (\Cref{ss: Witten's complex}). Consider also the trivialization of the flat line bundle $\Omega^zNM^0=\Omega^zN\FF|_{M^0}$ defined by $|\omega|^z$. Then, by~\eqref{C^-infty(M Lambda otimes LL^z) equiv C^-infty(M Lambda)} and since $d\omega=\eta\wedge\omega$ (\ref{i-(A): d omega = eta wedge omega} of \Cref{ss: globalization}),
\begin{gather*}
C^{\pm\infty}(M^0;\Lambda\otimes\Omega^zNM^0)\equiv C^{\pm\infty}(M^0;\Lambda)\otimes\R|\omega|^z
\equiv C^{\pm\infty}(M^0;\Lambda)\;,\\
d\equiv d_z\otimes1\equiv d_z\;,\quad H^\bullet(M^0;\Omega^zNM^0)\equiv H_z(M^0)\;.
\end{gather*}
These identities will be applied without further comment. By Reeb's local stability, the following result follows from the case of a suspension foliation, which will be proved in \Cref{ss: (K(Lambda FF) d_FF)} (Corollary~\ref{c: K(M M^0 Lambda FF) equiv bigoplus_L bigoplus_k C^infty(L Lambda) suspension}).

\begin{prop}\label{p: K(Lambda FF) equiv bigoplus_k C^infty(M^0 Lambda)}
We have identities of topological complexes,
\begin{gather*}
K(\FF)\equiv\bigoplus_kC^\infty(M^0;\Lambda)
\equiv\bigoplus_kC^\infty(M^0;\Lambda\otimes\Omega^{-k-1}NM^0)\;,\\
d_\FF\equiv\bigoplus_kd_{-k-1}\equiv\bigoplus_kd\;,
\end{gather*}
where $k$ runs in $\N_0$. Moreover the subcomplex $K^{(s)}(\FF)\subset K(\FF)$ corresponds to the finite direct sum with $k<-s-1/2$.
\end{prop}

\begin{cor}\label{c: HK(FF) equiv bigoplus_k H_-k-1(M^0)}
We have TVS-identities,
\[
H^\bullet K(\FF)\equiv\bigoplus_kH_{-k-1}^\bullet(M^0)
\equiv\bigoplus_kH^\bullet(M^0,\Omega^{-k-1}NM^0)\;.
\]
Moreover $H^\bullet K^{(s)}(\FF)$ is the topological vector subspace of $H^\bullet K(\FF)$ given by the finite direct sum with $k<-s-1/2$. In particular,~\eqref{HK(FF) equiv bar HK(FF)} is satisfied.
\end{cor}

\begin{rem}\label{r: bigoplus_L}
The differential complexes on $M^0$ used in \Cref{p: K(Lambda FF) equiv bigoplus_k C^infty(M^0 Lambda)} obviously split into direct sums of the same complexes given by leaves $L\subset M^0$. The same applies to their cohomologies in Corollary~\ref{c: HK(FF) equiv bigoplus_k H_-k-1(M^0)}.
\end{rem}

\begin{rem}\label{r: C^-infty_M^0(M Lambda FF) equiv bigoplus_k ...}
Like in \Cref{p: K(Lambda FF) equiv bigoplus_k C^infty(M^0 Lambda)}, the isomorphism~\eqref{bigoplus_m C^0_m -> C^-infty_L(M)} gives
\[
C^{-\infty}_{M^0}(M;\Lambda\FF)\equiv\bigoplus_kC^{-\infty}(M^0;\Lambda)
\equiv\bigoplus_kC^{-\infty}(M^0;\Lambda\otimes\Omega^{-k-1}NM^0)\;.
\]
\end{rem}

\section{Description of $\bar H^\bullet J(\FF)$}\label{s: description of bar H^bullet J(FF)}

With the notation of \Cref{ss: folns almost w/o hol,ss: simple fol flows}, by~\eqref{J^m(M L) cong ...} and~\eqref{J(M L) cong ...}, for $m\in\R$,
\begin{gather}
J^m(\FF)\cong\bfrho^m\Hb^\infty(\bfM;\Lambda\bfFF)
\equiv\bfrho^{m+\frac12}H^\infty(\mathring\bfM;\Lambda\mathring\bfFF)\;,
\label{J^m(FF) cong ...}\\
J(\FF)\cong\bigcup_m\bfrho^m\Hb^\infty(\bfM;\Lambda\bfFF)
=\bigcup_m\bfrho^mH^\infty(\mathring\bfM;\Lambda\mathring\bfFF)\;,
\label{J(FF) cong ...}
\end{gather}
as topological complexes with $d_\FF$, $d_{\bfFF}$ or $d_{\mathring\bfFF}$, using the b-metric $\bfg$ to define $\Hb^\infty(\bfM;\Lambda\bfFF)$, and using $\bfg|_{\mathring\bfM}$ to define $H^\infty(\mathring\bfM;\Lambda\mathring\bfFF)$. 

On the other hand, since $\bfeta=d_{\bfFF}(\ln\bfrho)$ on $\mathring\bfM$ by~\eqref{eta_0 = d_0,1(ln rho_l)}, we get isomorphisms of topological complexes,
\begin{equation}\label{bfrho^-m-frac12 - conormal}
\bfrho^{-m-\frac12}:\big(\bfrho^{m+\frac12}H^\infty(\mathring\bfM;\Lambda\mathring\bfFF),d_{\mathring\bfFF}\big)\xrightarrow{\cong}\big(H^\infty(\mathring\bfM;\Lambda\mathring\bfFF),d_{\mathring\bfFF,m+\frac12}\big)\;,
\end{equation}
by the leafwise version of~\eqref{Witten's opers} (\Cref{s: Witten's opers on Riem folns of bd geom}). 

By~\eqref{J^m(FF) cong ...} and~\eqref{bfrho^-m-frac12 - conormal}, and  the analog of~\eqref{leafwise Hodge iso} for $\Delta_{\mathring\bfFF,m+\frac12}$ in $H^\infty(\mathring\bfM;\Lambda\mathring\bfFF)$ (\Cref{s: Witten's opers on Riem folns of bd geom}), we get induced TVS-isomorphisms
\begin{align}
\bar H^\bullet J^m(\FF)&\cong
\bar H^\bullet\big(\bfrho^{m+\frac12}H^\infty(\mathring\bfM;\Lambda\mathring\bfFF),d_{\mathring\bfFF}\big)
\label{bar H^bullet J^m(FF) cong bar H^bullet(...)}\\
&\cong\bar H^\bullet\big(H^\infty(\mathring\bfM;\Lambda\mathring\bfFF),d_{\mathring\bfFF,m+\frac12}\big)
\label{bar H^bullet(...) cong bar H^bullet(...) - conormal}\\
&\cong\ker\Delta_{\mathring\bfFF,m+\frac12}\;.\label{bar H^bullet J^m(FF) cong ker ...}
\end{align}
By the analog of~\eqref{varinjlim H^bullet I^(s)(FF) cong H^bullet I(FF)} for $J(\FF)$ and~\eqref{bar H^bullet J^m(FF) cong ker ...}, the LCHS $\bar H^\bullet J(\FF)$ is an inductive limit of Hilbertian spaces. The isomorphisms~\eqref{bar H^bullet J^m(FF) cong bar H^bullet(...)} and~\eqref{bar H^bullet(...) cong bar H^bullet(...) - conormal} are also true in cohomology.

\Cref{t: intro - bar H^bullet J(FF)} follows from the analog of~\eqref{varinjlim H^bullet I^(s)(FF) cong H^bullet I(FF)} for $J(\FF)$ and~\eqref{J^m(FF) cong ...}--\eqref{bfrho^-m-frac12 - conormal}.

\section{Short exact sequence of conormal reduced cohomology}\label{s: short exact seq - conormal}

The goal of this section is to prove \Cref{t: intro - reduced conormal cohomology exact sequence}; i.e., the exactenss~\eqref{exact seq in reduced cohom - conormal}. Some remarks will indicate how to modify the argument to get also the exactness of~\eqref{exact seq in cohom - conormal}. To begin with, we choose appropriate partial extension maps.

\subsection{Compatibility of the maps $E_m$ with $d_\FF$}\label{ss: compatibility of E_m with d_FF}

For $m\in\R$, take $s\in\R$ such that $s=0$ if $m\ge0$, and $m>s\in\Z^-$ if $m<0$. For fixed $0<\epsilon<1$, using the tubular neighborhood $T:=T_\epsilon$ of $M^0$ in $M$ (\Cref{ss: tubular neighborhoods of M^0}), consider the continuous inclusions of $I^{(s)}_\co(\FF|_T)\subset I^{(s)}(\FF)$ and $J^m_\co(\FF|_T)\subset J^m(\FF)$, using the extension by zero. Let $E_{m,T}:J^m_\co(\FF|_T) \to I^{(s)}_\co(\FF|_T)$ \index{$E_{m,T}$} be the continuous linear partial extension map constructed in the proofs of the compactly supported versions of \Cref{p: E_m,c: E_m: J^m(M L) -> I^(s)(M L)} with $\Lambda\FF|_T$ (see \Cref{r: E_m vector bundle,r: E_m compactly supported}). By the Reeb's local stability, the following result follows from its case for suspension foliations, which will be proved in \Cref{ss: E_m} (\Cref{c: E_m T}).

\begin{prop}\label{p: E_m T d_FF = d_FF E_m T}
$E_{m,T}d_\FF=d_\FF E_{m,T}$.
\end{prop}

Let $\{\lambda,\mu\}$ be a smooth partition of unity of $\R$ subordinated to the open cover $\{(-\epsilon,\epsilon),\R^\times\}$, which induces the smooth partition of unity $\{\lambda(\rho),\mu(\rho)\}$ of $M$ subordinated to the open cover $\{T, M^1\}$, where $\lambda(\rho)$ (resp., $\mu(\rho)$) is extended by $0$ (resp., $1$) to the whole of $M$. According to \Cref{r: E_m m < 0,r: E_m vector bundle}, take the continuous linear partial extension map $E_m:J^m(\FF)\to I^{(s)}(\FF)$ defined by
\begin{equation}\label{E_m alpha = E_m T(lambda(rho) alpha) + mu(rho) alpha}
E_m\alpha=E_{m,T}(\lambda(\rho)\,\alpha)+\mu(\rho)\,\alpha\;.
\end{equation}

\begin{cor}\label{c: E_m d_FF = d_FF E_m}
$E_md_\FF=d_\FF E_m$. 
\end{cor}

\begin{proof}
By the version of \Cref{c: E_m s} for $E_{m,T}$ (\Cref{conormal sequence}), and since $d_\FF\lambda(\rho)=-d_\FF\mu(\rho)$ is supported in $M^1$, we get, for $\alpha\in J^m(\FF)$,
\begin{align*}
E_md_\FF\alpha&=E_{m,T}(\lambda(\rho)\,d_\FF\alpha)+\mu(\rho)\,d_\FF\alpha\\
&=E_{m,T}d_\FF(\lambda(\rho)\,\alpha)-E_{m,T}(d_\FF\lambda(\rho)\wedge\alpha)\\
&\phantom{={}}{}+d_\FF(\mu(\rho)\,\alpha)-d_\FF\mu(\rho)\wedge \alpha\wedge\omega\\
&=d_\FF E_{m,T}(\lambda(\rho)\,\alpha)-d_\FF\lambda(\rho)\wedge\alpha\\
&\phantom{={}}{}+d_\FF(\mu(\rho)\,\alpha)-d_\FF\mu(\rho)\wedge \alpha\\
&=d_\FF E_{m,T}(\lambda(\rho)\,\alpha)+d_\FF(\mu(\rho)\,\alpha)
=d_\FF E_m\alpha\;.\qedhere
\end{align*}
\end{proof}

\subsection{The maps $F_m$}\label{ss: F_m}

For $s\in\R$ and $m<s-n/2-1$, we can consider $R:I^{(s)}(\FF)\to J^m(\FF)$ by the analog of~\eqref{sandwich for AA} for $J(\FF)$ (\Cref{ss: J(M L)}). Taking $s'=0$ if $m\ge0$, and $m>s'\in\Z^-$ if $m<0$, let $E_m:J^m(\FF)\to I^{(s')}(\FF)$ be defined like in \Cref{ss: compatibility of E_m with d_FF}.We can also consider
\begin{equation}\label{E_m: J^(s)(FF) to I^(s')(FF)}
E_m=E_mj_{s,m}:J^{(s)}(\FF)\to I^{(s')}(\FF)\;.
\end{equation}
Then define the continuous linear map \index{$F_m$}
\[
F_m:=1-E_mR:I^{(s)}(\FF)\to K^{(s')}(\FF)\;.
\]
Note that
\begin{gather}
E_mR_s+\iota_{s'}F_m=j_{s,s'}:I^{(s)}(\FF)\to I^{(s')}(\FF)\;,\label{E_m R_s + iota_s' F_m = j_s s'}\\
F_m\iota_s=j_{s,s'}:K^{(s)}(\FF)\to K^{(s')}(\FF)\;.\label{F_m iota_s = j_s s'}
\end{gather}
Moreover, by \Cref{c: E_m d_FF = d_FF E_m},
\begin{equation}\label{F_m d_FF = d_FF F_m}
F_m d_\FF=d_\FF F_m\;.
\end{equation}

Take smaller numbers, $s_1<s$, $m_1<m$ and $s'_1<s'$, satisfying the same inequalities as $s$, $m$ and $s'$. Then, with~\eqref{E_m: J^(s)(FF) to I^(s')(FF)}, the version of \Cref{p: E_m s} with $\Lambda\FF$ (see \Cref{r: E_m s - variants of the defn}) gives
\begin{equation}\label{j_s' s'_1 E_m = E_m_1 j_s s_1}
j_{s',s'_1}E_m=E_{m_1}j_{s,s_1}\;.
\end{equation}
Then, using the definition of $F_m$, we also get
\begin{equation}\label{j_s' s'_1 F_m = F_m_1 j_s s_1}
j_{s',s'_1}F_m=F_{m_1}j_{s,s_1}\;.
\end{equation}

\begin{rem}\label{r: F_m with Omega M}
According to \Cref{r: coefficients in Omega M}, we can also define
\[
F_m:I^{(s)}(\FF;\Omega M)\to K^{(s')}(\FF;\Omega M)
\]
satisfying similar properties, using $d_\FF^\trans$.
\end{rem}

\subsection{The equality $\ker\bar R_*=\im\bar\iota_*$}
\label{ss: ker bar R_* = im bar iota_*}

We already know that $\ker\bar R_*\supset\im\bar\iota_*$. To prove that $\ker\bar R_*\subset\im\bar\iota_*$, take any class $\overline{[\alpha]}\in\ker\bar R_*$ in $\bar H^\bullet I(\FF)$. Hence there is some net $\varphi_l\in J(\FF)$ such that $R\alpha=\lim_ld_\FF\varphi_l$ in $J(\FF)$. We can assume $\varphi_l\in\Cinftyc(M^1;\Lambda\FF)$ by the density of $\Cinftyc(M^1;\Lambda\FF)$ in $J(\FF)$ (\Cref{ss: J(M L)}).

Using that $J(\FF)$ is compactly retractive (\Cref{ss: J(M L)}) and arguing like in \Cref{ss: injectivity of hat jmath_*}, we get that $\{R\alpha,d_\FF\beta_l\}_l$ is contained in some step $J^{(s)}(\FF)$, and $R\alpha=\lim_ld_\FF\varphi_l$ in $J^{(s)}(\FF)$. Moreover, we can assume $d_\FF\phi_l$ is a sequence because $J^{(s)}(\FF)$ is a Fr\'echet space.

Consider the notation of \Cref{ss: F_m}. We have $R\alpha=\lim_l d_\FF\varphi_l$ in $J^m(\FF)$ by the version of~\eqref{AA^m(M) subset AA^m'(M)} for $J(\FF)$ (\Cref{ss: J(M L)}). We have $\beta:=F_m\alpha\in ZJ^{(s')}(\FF)\subset ZJ(\FF)$ by~\eqref{F_m d_FF = d_FF F_m}, obtaining a class $\overline{[\beta]}\in\bar H^\bullet J(\FF)$.

Since the sequence $\varphi_l$ is in $\Cinftyc(M^1;\Lambda\FF)$, it is also in $J^m(\FF)$ and in $I^{(s')}I(\FF)$, and we have $E_m\varphi_l=\varphi_l$ by the version of \Cref{c: E_m s} with $J^m(\FF)$. Hence, by \Cref{c: E_m d_FF = d_FF E_m},
\[
\beta=\alpha-E_mR\alpha=\alpha-\lim_lE_md_\FF\varphi_l=\alpha-\lim_ld_\FF E_m\varphi_l=\alpha-\lim_ld_\FF\varphi_l
\]
 in $I^{(s')}(\FF)$, and therefore also in $I(\FF)$. This shows that $\bar\iota_*(\overline{[\beta]})=\overline{[\alpha]}$ in $\bar H^\bullet I(\FF)$.
 
 \begin{rem}\label{r: ker R_* subset im iota_*}
Using cohomology instead of reduced cohomology and a single element $\varphi$ instead of a net $\varphi_l$, the analogous argument gives $\ker R_*\subset\im\iota_*$, obtaining $\ker R_*=\im\iota_*$.
 \end{rem}

\subsection{Injectivity of $\bar\iota_*$}\label{ss: injectivity of bar iota_*}

Take any $\overline{[\alpha]}\in H^\bullet K(\FF)$ with $\bar\iota_*(\overline{[\alpha]})=0$ in $\bar H^\bullet I(\FF)$. Since $I(\FF)$ is compactly retractive, $C^\infty(M;\Lambda\FF)$ is dense in $I(\FF)$ and every $I^{(s)}(\FF)$ is a Fr\'echet space (\Cref{ss: conormal - Sobolev order - compact}), we get as above that there is some $s$ and a sequence $\varphi_l$ in $C^\infty(M;\Lambda\FF)$ such that $\alpha\in K^{(s)}(\FF)$ and $\alpha=\lim_ld_\FF\varphi_l$ in $I^{(s)}(\FF)$. 

Consider again the notation of \Cref{ss: F_m}. By~\eqref{F_m iota_s = j_s s'},
\[
\alpha=F_m\alpha=\lim_lF_md_\FF\varphi_l=\lim_ld_\FF F_m\varphi_l
\]
in $K^{(s)}(\FF)$, and therefore in $K(\FF)$. So $\alpha\in\bar BK(\FF)=BK(\FF)$ by~\eqref{HK(FF) equiv bar HK(FF)}, and therefore $[\alpha]=0$ in $H^\bullet K(\FF)$.

\begin{rem}\label{r: injectivity of iota_*}
Like in \Cref{r: ker R_* subset im iota_*}, we also get the injectivity of $\iota_*$.
 \end{rem}

\subsection{Surjectivity of $\bar R_*$}\label{ss: surjectivity of bar R_*}

For any class $\overline{[\alpha]}\in\bar H^\bullet J(\FF)$, the representative $\alpha$ is in some step $J^{(s)}(\FF)$, and therefore it is also in $ZJ^{(s)}(\FF)$. With the notation of \Cref{ss: F_m}, $\beta:=E_m\alpha\in ZI^{(s')}(\FF)\subset  ZI(\FF)$ by \Cref{c: E_m d_FF = d_FF E_m}, and we have $R\beta=\alpha$. This shows that $\overline{[\alpha]}=\bar R(\overline{[\beta]})$.

 \begin{rem}\label{r: ker R_* subset im iota_*}
Using cohomology instead of reduced cohomology, the analogous argument gives the surjectivity of $R_*$.
 \end{rem}

\section{Computations in the case of a suspension foliation}\label{s: computations suspension foln}

Consider the notation of \Crefrange{ss: suspension - homotheties}{ss: boundary-def func of M_pm - suspension}, where the case of a suspension foliation with a simple foliated flow was considered.

\subsection{Description of $(K(\FF),d_{\FF})$}\label{ss: (K(Lambda FF) d_FF)}

For every $m\in\N_0$ and $s<-1/2$, consider the injection defined by~\eqref{C^infty(L Omega^-1 NL) cong partial_x^m C^infty(L Omega^-1 NL)} for the vector bundle $\Lambda\FF$, 
\begin{equation}\label{alpha mapsto partial_rho^m delta_M^0^alpha}
C^\infty(L;\Lambda\otimes\Omega^{-1}NL) \to K^{(s-m)}(M,L;\Lambda\FF)\;,\quad \alpha\mapsto\partial_\rho^m\delta_L^\alpha\;.
\end{equation}

\begin{prop}\label{p: d_FF corresponds to d_L - m theta wedge}
Via~\eqref{alpha mapsto partial_rho^m delta_M^0^alpha}, the operator $d_\FF$ on $K^{(s-m)}(M,L;\Lambda\FF)$ corresponds to the operator $d_L-m{\eta\wedge}$ on $C^\infty(L;\Lambda\otimes\Omega^{-1}NL)$.
\end{prop}

\begin{proof}
Consider first the case $m=0$. According to~\eqref{C^-infty(M Lambda^v FF) equiv C^infty_c(M Lambda^n''-v FF otimes Omega N FF)'}, for some degree $v$, take $\alpha\in C^\infty(L;\Lambda^v\otimes\Omega^{-1}NL)$ and $\beta\in\Cinftyc(M;\Lambda^{1,n-1-v})$. We can write $\alpha=\alpha_0\otimes|\omega|^{-1}$ and $\beta=\beta_0\wedge\omega$ for some $\alpha_0\in C^\infty(L;\Lambda^v)$ and $\beta_0\in\Cinftyc(M;\Lambda^{0,n-1-v})$. By~\eqref{C^-infty(M Lambda otimes LL^z) equiv C^-infty(M Lambda)},~\eqref{C^-infty(M Lambda^r) equiv C^infty_c(M Lambda^n-r)'},~\eqref{d_z equiv (-1)^k+1 d_-z^t} (or~\eqref{d antiderivative with d_z and d_-z} and the Stokes' theorem),~\eqref{C^-infty(M Lambda^v FF) equiv C^infty_c(M Lambda^n''-v FF otimes Omega N FF)'} and~\eqref{d_FF equiv (-1)^v+1 d_FF^t}, and since $d\omega=\eta\wedge\omega$,
\begin{align*}
\langle d_\FF\delta_L^\alpha,\beta\rangle
&=-(-1)^v\langle \delta_L^\alpha,d\beta\rangle
=-(-1)^v\langle \delta_L^\alpha,(d\beta_0+(-1)^{n-1-v}\beta_0\wedge\eta)\wedge\omega\rangle\\
&=-(-1)^v\int_L\alpha_0\wedge((d+\eta\wedge)\beta_0)|_L=-(-1)^v\int_L\alpha_0\wedge(d+\eta\wedge)(\beta_0|_L)\\
&=\int_L(d_L-\eta\wedge)\alpha_0\wedge\beta_0|_L
=\int_Ld_L\alpha\wedge\beta|_L=\langle\delta_L^{d_L\alpha},\beta\rangle\;.
\end{align*}

The general case follows from the previous case because $d_\FF\partial_\rho=\partial_\rho(d_\FF-{\eta\wedge})$ on $C^{-\infty}(M;\Lambda\FF)$ by~\eqref{| varkappa | partial_rho = Theta_X} and~\eqref{[Theta_X d_0 1]}, and $\eta\wedge\delta_L^\alpha=\delta_L^{\eta\wedge\alpha}$ by~\eqref{alpha wedge delta_L^beta}.
\end{proof}

\begin{cor}\label{c: K(M M^0 Lambda FF) equiv bigoplus_L bigoplus_k C^infty(L Lambda) suspension}
\Cref{p: K(Lambda FF) equiv bigoplus_k C^infty(M^0 Lambda)} is true in this case.
\end{cor}

\begin{proof}
Apply~\eqref{bigoplus_m C^1_m -> K(M L)},~\eqref{bigoplus_m<-s-1/2 C^1_m cong K^(s)(M L)} and \Cref{p: d_FF corresponds to d_L - m theta wedge}.
\end{proof}

\subsection{A partial extension map on $M_\pm$}\label{ss: E_m pm}

The notation of \Cref{ss: conormality at the boundary - Sobolev order,ss: x^m L^infty(M),ss: description of AA(M) by bounds,ss: dot AA(M) and AA(M) vs I(breve M partial M),ss: conormality at the boundary - symbol order,ss: KK(M)}, concerning conormal distributions at the boundary, is also used here. By \Cref{p: E_m}, there is a continuous linear partial extension map, \index{$E_{m,\pm}$}
\[
E_{m,\pm}:\AA^m(M_\pm;\Lambda\FF_\pm)\to\dot\AA^{(s)}(M_\pm;\Lambda\FF_\pm)\;,
\]
where $s=0$ if $m\ge0$, and $m>s\in\Z^-$ if $m<0$.
According to the proof of \Cref{p: E_m} and \Cref{r: E_m m < 0,r: E_m vector bundle}, in the case $0>m>s\in\Z^-$, the homomorphism $E_{m,\pm}$ can be given by the composition
\[
\AA^m(M_\pm;\Lambda\FF_\pm) \xrightarrow{J^N} \AA^0(M_\pm;\Lambda\FF_\pm) 
\xrightarrow{E_{0,\pm}} \dot\AA^{(0)}(M_\pm;\Lambda\FF_\pm) \xrightarrow{\partial_\rho^N} 
\dot\AA^{(s)}(M_\pm;\Lambda\FF_\pm)\;,
\]
where $N=-s\in\Z^+$, $E_{0,\pm}$ is a continuous inclusion map, and $J$ is the endomorphism of $C^\infty(M_\pm;\Lambda\FF_\pm)$ given by
\[
J\alpha(\rho,y)=\int_1^\rho\alpha(\rho_1,y)\,d\rho_1\;,
\]
using~\eqref{dot AA(M)|_mathring M AA(M) subset C^infty(mathring M)},~\eqref{M_pm equiv [0 infty)_rho times L_varpi} and the identity $\Lambda_{(\rho,y)}\FF\equiv\Lambda_yL$.

Consider also the endomorphism $\widetilde J$ of $C^\infty(\widetilde M_\pm;\Lambda\widetilde\FF_\pm)$ defined like $J$, 
\[
\widetilde J\tilde\alpha(\tilde\rho,\tilde y)=\int_1^{\tilde \rho}\tilde\alpha(\tilde\rho_1,y)\,d\tilde\rho_1\;,
\]
using~\eqref{widetilde M_pm equiv [0 infty)_tilde rho times widetilde L_varpi} and the identity $\Lambda_{(\rho,\tilde y)}\FF_\pm\equiv\Lambda_{\tilde y}\widetilde L$. Clearly, $\widetilde J$ corresponds to $J$ via
\[
\pi_{M_\pm}^*:C^\infty(M_\pm;\Lambda\FF_\pm)\to C^\infty(\widetilde M_\pm;\Lambda\widetilde\FF_\pm)\;.
\]
Using~\eqref{widetilde M_pm equiv [0 infty)_tau times widetilde L_varpi}, we can also write
\begin{equation}\label{widetilde J}
\widetilde J\tilde\alpha(\tau,\tilde y)=e^{F(\tilde y)}\int_{e^{-F(\tilde y)}}^\tau\tilde\alpha(\tau_1,\tilde y)\,d\tau_1\;,
\end{equation}
with the change of variable $\tilde\rho_1=e^{F(\tilde y)}\tau_1$, because $\tilde\rho=\pm e^{F(\tilde y)}\tau$.

For fixed $0<\epsilon<1$, consider the collar neighborhoods $T_\pm:=T_{\pm,\epsilon}$ and $\widetilde T_\pm=\widetilde T_{\pm,\epsilon}=\pi_{M_\pm}^{-1}(T_\pm)$ of the boundaries in $M_\pm$ and $\widetilde M_\pm$. Using~\eqref{Lambda M}, consider $\eta\in C^\infty(M_\pm;\Lambda^1\FF_\pm)$ and $\tilde\eta\in C^\infty(\widetilde M_\pm;\Lambda^1\widetilde\FF_\pm)$.

\begin{prop}\label{p: J d_FF_pm = (d_FF_pm - theta wedge) J}
$Jd_{\FF_\pm}=(d_{\FF_\pm}-\eta\wedge)J$ on $\Cinftyc(T_\pm;\Lambda\FF_\pm)$.
\end{prop}

\begin{proof}
Take open subsets $\widetilde B\subset\widetilde L$ such that $\pi_L:\widetilde B\to B:=\pi_L(\widetilde B)$ is a diffeomorphism. Since the open sets of the form $\varpi_\pm^{-1}(B)\equiv[0,\infty)_\rho\times\widetilde B_{\widetilde\varpi}$ cover $M_\pm$ and $J$ preserves the spaces $\Cinftyc(T_\pm\cap\varpi_\pm^{-1}(B);\Lambda\FF_\pm)$, it is enough to prove the stated equality on $\Cinftyc(T_\pm\cap\varpi_\pm^{-1}(B);\Lambda\FF_\pm)$. In turn, this follows by checking that $\widetilde Jd_{\widetilde \FF_\pm}=(d_{\widetilde \FF_\pm}-\tilde\eta\wedge)\widetilde J$ on $\Cinftyc(\widetilde T_\pm;\Lambda\widetilde\FF_\pm)$ because
\[
\pi_{M_\pm}\equiv\id\times\pi_L:\widetilde\varpi_\pm^{-1}(\widetilde B)\equiv[0,\infty)_{\tilde\rho}\times\widetilde B_{\widetilde\varpi}
\to\varpi_\pm^{-1}(B)\equiv[0,\infty)_\rho\times B_\varpi
\]
is a diffeomorphism. Let $\tilde\alpha\in\Cinftyc(\widetilde T_\pm;\Lambda\widetilde\FF_\pm)$ and $(\tau,\tilde y)\in[0,\infty)\times\widetilde L$ with $\tau<e^{-F(\tilde y)}\epsilon$, which means that $(\tau,\tilde y)$ corresponds to an element of $T_\pm$ via~\eqref{widetilde M_pm equiv [0 infty)_tau times widetilde L_varpi}; in particular, $\tilde\alpha(e^{-F(\tilde y)},\tilde y)=0$. So, by~\eqref{widetilde J} and since $d_{\widetilde\FF_\pm}\tau=0$,
\begin{align*}
\widetilde Jd_{\widetilde\FF_\pm}\tilde\alpha(\tau,\tilde y)
&=e^{F(\tilde y)}\int_{e^{-F(\tilde y)}}^\tau d_{\widetilde\FF_\pm}\tilde\alpha(\tau_1,\tilde y)\,d\tau_1
=e^{F(\tilde y)}d_{\widetilde\FF_\pm}\int_{e^{-F(\tilde y)}}^\tau\tilde\alpha(\tau_1,\tilde y)\,d\tau_1\\
&=d_{\widetilde\FF_\pm}\,e^{F(\tilde y)}\int_{e^{-F(\tilde y)}}^\tau\tilde\alpha(\tau_1,\tilde y)\,d\tau_1
-e^{F(\tilde y)}\tilde\eta(\tilde y)\wedge\int_{e^{-F(\tilde y)}}^\tau\tilde\alpha(\tau_1,\tilde y)\,d\tau_1\\
&=(d_{\widetilde \FF_\pm}-\tilde\eta\wedge)\widetilde J\tilde\alpha(x,\tilde y)\;.\qedhere
\end{align*}
\end{proof}

\begin{prop}\label{p: partial_rho d_FF}
We have $\partial_\rho d_{\FF_\pm}=(d_{\FF_\pm}+\eta\wedge)\partial_\rho$ on $C^{-\infty}(M_\pm;\Lambda\FF_\pm)$.
\end{prop}

\begin{proof}
Apply~\eqref{Lambda M},~\eqref{d_0 1 equiv d_FF} and~\eqref{[partial_rho d_0 1]}. 
\end{proof}


\begin{cor}\label{c: E_m pm d_FF_pm = d_FF_pm E_m pm}
For all $m\in\R$, $E_{m,\pm}d_{\FF_\pm}=d_{\FF_\pm}E_{m,\pm}$ on $\AA^m_\co(T_\pm;\Lambda\FF_\pm)$.
\end{cor}

\begin{proof}
It is enough to consider the case $m<0$. Then apply \Cref{p: J d_FF_pm = (d_FF_pm - theta wedge) J,p: partial_rho d_FF}, using the given definition of $E_{m,\pm}$ and the density of $\Cinftyc(T_\pm;\Lambda\FF_\pm)$ in $\AA^m_\co(T_\pm;\Lambda\FF_\pm)$ (see \Cref{ss: description of AA(M) by bounds}).
\end{proof}

\subsection{A partial extension map on $M$}\label{ss: E_m}

Let us apply the notation of \Cref{s: conormal seq} to the suspension foliation (that notation is compatible with the notation of \Cref{ss: folns almost w/o hol,ss: simple fol flows,s: suspension}). Recall that $\bfM=M_-\sqcup M_+$, $\bfFF$ is the combination of $\FF_\pm$, and $\bfpi:\bfM\to M$ is the combination of $\pi_\pm:M_\pm\to M$. The version of the commutative diagram~\eqref{CD: conormal seqs} for $\Lambda\bfFF\equiv\bfpi^*\Lambda\FF$ is
\begin{equation}\label{commut diag of conormal seqs with folns}
\begin{CD}
\KK(\bfM;\Lambda\bfFF) @>{\iota}>> \dot\AA(\bfM;\Lambda\bfFF) @>R>> \AA(\bfM;\Lambda\bfFF)\phantom{\;.} \\
@V{\bfpi_*}VV @V{\bfpi_*}VV @V{\bfpi_*}V{\cong}V \\
K(\FF) @>{\iota}>> I(\FF) @>R>> J(\FF)\;.
\end{CD}
\end{equation}
Moreover $d_{\bfFF}\in\Diffb(\bfM;\Lambda\bfFF)$ is the lift of $d_\FF$. Hence the operators defined by $d_{\bfFF}$ on the spaces of the top row of~\eqref{commut diag of conormal seqs with folns} correspond to the operators defined by $d_\FF$ on the spaces of its bottom row via the homomorphisms $\bfpi_*$ (\Cref{s: conormal seq}). According to~\Cref{ss: diff ops on dot AA(M) and AA(M)}, $d_{\bfFF}$ preserves the subspaces $\dot\AA^{(s)}(\bfM;\Lambda\bfFF)$ and $\AA^m(\bfM;\Lambda\bfFF)$.

The partial extension maps of \Cref{ss: E_m pm},
\[
E_{m,\pm}:\AA^m(M_\pm;\Lambda\FF_\pm)\to\dot\AA^{(s)}(M_\pm;\Lambda\FF_\pm)\;,
\]
can be combined to define a continuous linear partial extension map
\[
\bfE_m:\AA^m(\bfM;\Lambda\bfFF)\to\dot\AA^{(s)}(\bfM;\Lambda\bfFF)\;.
\]
Then, according to \Cref{c: E_m: J^m(M L) -> I^(s)(M L)} and its proof, a continuous linear partial extension map $E_m:J^m(\FF)\to I^{(s)}(\FF)$ is given by the composition
\[
J^m(\FF) \xrightarrow{\bfpi_*^{-1}} \AA^m(\bfM;\Lambda\bfFF) \xrightarrow{\bfE_m} \dot\AA^{(s)}(\bfM;\Lambda\bfFF) \xrightarrow{\bfpi_*} I^{(s)}(\FF)\;,
\]
which is a continuous inclusion map if $m\ge0$. Recall that $T\equiv(-\epsilon,\epsilon)_\rho\times L_\varpi$ and
\[
\bfT=T_-\sqcup T_+=\bfpi^{-1}(T)\equiv[0,\epsilon)_{\bfrho}\times\partial\bfM_{\bfvarpi}\;.
\]
Like in \Cref{ss: injectivity of bar iota_*}, consider the restriction $E_{m,T}:J^m_\co(T;\Lambda\FF)\to I^{(s)}_\co(T;\Lambda\FF)$ of $E_m$. Suppose $\epsilon<1$, like in \Cref{ss: E_m pm}.

\begin{cor}\label{c: E_m T}
For all $m\in\R$, $E_{m,T}$ satisfies $E_{m,T}d_\FF=d_\FF E_{m,T}$.
\end{cor}

\begin{proof}
Apply \Cref{c: E_m pm d_FF_pm = d_FF_pm E_m pm}.
\end{proof}

\section{Functoriality and leafwise homotopy invariance}\label{s: functor and leafwise homotopy inv}

\subsection{Pull-back of conormal leafwise currents}\label{ss: pull-back of leafwise conormal currents}

Let $M'$ be another closed manifold, and let $\phi:M'\to M$ be a smooth map transverse to $\FF$. Then $\FF':=\phi^*\FF$ is another transversely oriented foliation of codimension one satisfying the conditions~\ref{i: almost w/o hol} and~\ref{i: homotheties} in \Cref{ss: simple fol flows} with $M^{\prime\,0}:=\phi^{-1}(M^0)$.

\begin{rem}\label{r: phi - extension - non-compact}
The results of \Cref{s: functor and leafwise homotopy inv} have direct extensions to the case where $M$ or $M'$ may not be compact, with the condition that $M^0$ and $M^{\prime\,0}$ are compact.
\end{rem}

According to \Cref{ss: pull-back of conormal distribs}, the map~\eqref{phi^* on leafwise forms} has a continuous extension
\begin{equation}\label{phi^*: I(FF) to I(FF')}
\phi^*:I(\FF)\to I(\FF')\;.
\end{equation}
defined as the composition
\begin{equation}\label{pull-back - decomposition - conormal - leafwise}
I(\FF) \xrightarrow{\phi^*} I(M',M^{\prime\,0};\phi^*\Lambda\FF) \xrightarrow{\phi^*} I(\FF')\;,
\end{equation}
like~\eqref{composition - pull-back - conormal currents}, using~\eqref{phi^*: I(M L E) -> I(M' L' phi^*E)} with $E=\Lambda\FF$. We can also describe~\eqref{phi^*: I(FF) to I(FF')} as the restriction of~\eqref{phi^*: I(M L Lambda) to I(M' L' Lambda)} to conormal currents of bidegree $(0,\bullet)$, like in~\eqref{phi^*: C^-infty(M Lambda FF) to (C^-infty(M' Lambda FF')}. The map~\eqref{phi^*: I(FF) to I(FF')} is also a restriction of~\eqref{phi^*: C^-infty(M Lambda FF) to (C^-infty(M' Lambda FF')}. 

Similarly, the analogs of~\eqref{phi^*: I(M L E) -> I(M' L' phi^*E)} with $E=\Lambda\FF$ for~\eqref{phi^*: K(M L) to K(M' L')} and~\eqref{phi^*: J(M L) to J(M' L')} induce continuous homomorphisms
\begin{gather}
\phi^*:K(\FF)\to K(\FF')\;,\label{phi^*: K(FF) to K(FF')}\\
\phi^*:J(\FF)\to J(\FF')\;.\label{phi^*: J(FF) to J(FF')}
\end{gather}
By passing to cohomology and reduced cohomology, we get continuous homomorphisms,
\begin{equation}\label{phi^*: H^bullet K(FF) to H^bullet K(FF')}
\left\{
\begin{gathered}
\phi^*:H^\bullet K(\FF)\to H^\bullet K(\FF')\;,\\
\begin{alignedat}{2}
\phi^*:&H^\bullet I(\FF)\to H^\bullet I(\FF')\;,&\quad\phi^*:&\bar H^\bullet I(\FF)\to\bar H^\bullet I(\FF')\;,\\
\phi^*:&H^\bullet J(\FF)\to H^\bullet J(\FF')\;,&\quad\phi^*:&\bar H^\bullet J(\FF)\to\bar H^\bullet J(\FF')\;.
\end{alignedat}
\end{gathered}
\right.
\end{equation}
The assignment of the homomorphisms~\eqref{phi^*: I(FF) to I(FF')}--\eqref{phi^*: H^bullet K(FF) to H^bullet K(FF')} is functorial.

\subsection{Description of $\phi^*:K(\FF)\to K(\FF')$}\label{ss: description of phi^*:K(FF) to K(FF')}

For $\omega'=\phi^*\omega$ and $\eta'=\phi^*\eta$, we have $T\FF'=\ker\omega'$ and $d\omega'=\eta'\wedge\omega'$ (the Frobenius integrability condition for $\FF'$). Thus
\begin{equation}\label{phi^*: C^infty(M^0 Lambda) to C^infty(M^prime 0 Lambda)}
\phi^*:C^\infty(M^0;\Lambda)\to C^\infty(M^{\prime\,0};\Lambda)
\end{equation}
is a cochain map for $d_{s\eta}$ and $d_{s\eta'}$ ($s\in\R$) (\Cref{ss: Witten vs pull-back and push-forward}). In other words, $\phi$ induces
\begin{equation}
\label{phi^*: C^infty(M^0 Lambda otimes Omega^sNM^0) to C^infty(M^prime 0 Lambda otimes Omega^sNM^prime 0)}
\phi^*:C^\infty(M^0;\Lambda\otimes\Omega^sNM^0)\to C^\infty(M^{\prime\,0};\Lambda\otimes\Omega^sNM^{\prime\,0})\;,
\end{equation} 
given by
\begin{equation}\label{phi^*(alpha otimes |omega|^s) = phi^*alpha otimes |omega'|^s}
\phi^*(\alpha\otimes|\omega|^s)=\phi^*\alpha\otimes|\omega'|^s\;,
\end{equation}
which is another cochain map for the de~Rham differentials defined with the flat bundle structures of $\Omega^sNM^0$ and $\Omega^sNM^{\prime\,0}$.

If $\rho$ is a defining function of $M^0$ in some open neighborhood $T$, then $\rho':=\phi^*\rho$ is a defining function of $M^{\prime\,0}$ in $T'=\phi^{-1}(T)$, and~\eqref{phi^*: C^infty(M^0 Lambda otimes Omega^sNM^0) to C^infty(M^prime 0 Lambda otimes Omega^sNM^prime 0)} satisfies
\begin{equation}\label{phi^*(alpha otimes |d rho|^s) = phi^*alpha otimes |d rho'|^s}
\phi^*(\alpha\otimes|d\rho|^s)=\phi^*\alpha\otimes|d\rho'|^s\;.
\end{equation}
Note the compatibility of~\eqref{phi^*(alpha otimes |omega|^s) = phi^*alpha otimes |omega'|^s} and~\eqref{phi^*(alpha otimes |d rho|^s) = phi^*alpha otimes |d rho'|^s} with~\eqref{d rho}. Furthermore, the inverse image of $T:=T_\epsilon\equiv(-\epsilon,\epsilon)_\rho\times M^0_\varpi$, for $\epsilon>0$ small enough, is a tubular neighborhood $T'\equiv(-\epsilon,\epsilon)_{\rho'}\times M^{\prime\,0}_{\varpi'}$ of $M^{\prime\,0}$ in $M'$, where $\varpi':T'\to M^{\prime\,0}$ satisfies $\phi\varpi'=\varpi\phi$ as maps $T'\to M^0$. Thus
\[
\phi\equiv\id\times\phi:T'\equiv(-\epsilon,\epsilon)\times M^{\prime\,0}\to T\equiv(-\epsilon,\epsilon)\times M^0\;,
\]
which is proper because $M^{\prime\,0}$ is compact. We can use these tubular neighborhoods to define the operators $\partial_\rho$ and $\partial_{\rho'}$ on $\Cinftyc(T;\Lambda\FF)$ and $\Cinftyc(T';\Lambda\FF')$ (\Cref{s: conormal seq}), which are used in the identities of \Cref{p: K(Lambda FF) equiv bigoplus_k C^infty(M^0 Lambda)} for $K(\FF)$ and $K(\FF')$ (\Cref{ss: (K(Lambda FF) d_FF)}). Clearly,
\begin{equation}\label{partial_rho' phi^* = phi^* partial_rho}
\partial_{\rho'}\phi^*=\phi^*\partial_\rho\;,
\end{equation}
as maps $\Cinftyc(T;\Lambda\FF)\to\Cinftyc(T';\Lambda\FF')$,

\begin{prop}\label{p: phi^* equiv bigoplus_k phi^*}
According to \Cref{p: K(Lambda FF) equiv bigoplus_k C^infty(M^0 Lambda)}, the map~\eqref{phi^*: K(FF) to K(FF')} is given by
\[
\phi^*\equiv\bigoplus_k\phi^*\equiv\bigoplus_k\phi^*\;,
\]
where the terms of the first direct sum are given by~\eqref{phi^*: C^infty(M^0 Lambda) to C^infty(M^prime 0 Lambda)}, and the terms of the second direct sum are given by~\eqref{phi^*: C^infty(M^0 Lambda otimes Omega^sNM^0) to C^infty(M^prime 0 Lambda otimes Omega^sNM^prime 0)}, taking $s=-k-1$.
\end{prop}

\begin{proof}
The second identity follows from the first one and~\eqref{phi^*(alpha otimes |omega|^s) = phi^*alpha otimes |omega'|^s}. To prove the first identity, by~\eqref{partial_rho' phi^* = phi^* partial_rho}, it is enough to consider the term with $k=0$. 

For $\alpha\in C^\infty(M^0;\Lambda)$, let $u=\alpha\otimes|d\rho|^{-1}\in C^\infty(M^0;\Lambda\otimes\Omega^{-1}NM^0)$. Using the first identity of \Cref{p: K(Lambda FF) equiv bigoplus_k C^infty(M^0 Lambda)} for $k=0$, we have $u\equiv\delta_{M^0}^u=\varpi^*\alpha\cdot\rho^*\delta_0$ in $K(\FF)$, using Dirac sections (\Cref{ss: Dirac sections}). Here, $\rho^*\delta_0\in K(T,M^0)$ is defined because $\rho:T\to(-\epsilon,\epsilon)$ is transverse to $0$. Moreover $u':=\phi^*u=\phi^*\alpha\cdot|d\rho'|^{-1}$ by~\eqref{phi^*(alpha otimes |d rho|^s) = phi^*alpha otimes |d rho'|^s}. As before, $u'\equiv\delta_{M^{\prime\,0}}^{u'}=\varpi^{\prime*}\phi^*\alpha\cdot\rho^{\prime*}\delta_0$ in $K(\FF')$. Take a sequence $f_i\in\Cinftyc(-\epsilon,\epsilon)$ converging to $\delta_0$ in $C^{-\infty}_\co(-\epsilon,\epsilon)$. Then
\begin{align*}
\phi^*\delta_{M^0}^u&=\phi^*(\varpi^*\alpha\cdot\rho^*\delta_0)
=\lim_i\phi^*(\varpi^*\alpha\cdot\rho^*f_i)
=\lim_i\phi^*\varpi^*\alpha\cdot\phi^*\rho^*f_i\\
&=\lim_i\varpi^{\prime*}\phi^*\alpha\cdot\rho^{\prime*}f_i
=\varpi^{\prime*}\phi^*\alpha\cdot\rho^{\prime*}\delta_0
=\delta_{M^{\prime\,0}}^{u'}\;.\qedhere
\end{align*}
\end{proof}

\begin{rem}\label{r: phi^* equiv prod_k phi^*}
The equality $\partial_{\rho'}\phi^*=\phi^*\partial_\rho$ has a continuous extension as maps $C^{-\infty}_\co(T;\Lambda\FF)\to C^{-\infty}_\co(T';\Lambda\FF')$, and the computations of the above proof also work also with $\alpha\in C^{-\infty}(M^0;\Lambda)$. So we get similar expressions of $\phi^*:C^{-\infty}_{M^0}(M;\Lambda\FF)\to C^{-\infty}_{M^{\prime\,0}}(M';\Lambda\FF')$ according to \Cref{r: C^-infty_M^0(M Lambda FF) equiv bigoplus_k ...}.
\end{rem}

\subsection{Push-forward of conormal leafwise currents}\label{ss: push-forward of conormal leafwise currents}

With the notation and conditions of \Cref{ss: pull-back of leafwise conormal currents}, suppose that moreover $\phi$ is a submersion such that the vertical bundle $\VV$ is oriented (\Cref{ss: push-forward of leafwise currents}). Thus $\phi:M^{\prime\,0}\to M^0$ is also a submersion whose vertical bundle is $\VV|_{M^{\prime\,0}}\subset TM^{\prime\,0}$, also oriented. Then the case of~\eqref{phi_*: C^pm infty_c/cv(M' Lambda FF') to C^pm infty_c/.(M Lambda FF)} on smooth leafwise forms has a continuous extension
\begin{equation}\label{phi_*: I_c/cv(FF') to I_c/cv(FF)}
\phi_*:I_{\co/\cv}(\FF')\to I_{\co/{\cdot}}(\FF)\;.
\end{equation}
This map can be described as the restriction of the map~\eqref{phi_*: I_c/cv(M' L' Lambda) to I_c/.(M L Lambda)} to conormal currents of bidegree $(0,\bullet)$, like~\eqref{phi_*: C^pm infty_c/cv(M' Lambda FF') to C^pm infty_c/.(M Lambda FF)} in \Cref{ss: push-forward of leafwise currents}. We can also describe~\eqref{phi_*: I_c/cv(FF') to I_c/cv(FF)} as the composition
\[
I_{\co/\cv}(M',L';\Lambda\FF)\xrightarrow{\pi_\topd}I_{\co/\cv}(M',L';\phi^*\Lambda\FF\otimes\Omega_\fiber)
\xrightarrow{\phi_*}I_{\co/{\cdot}}(M,L;\Lambda\FF)\;,
\]
like in~\eqref{push-forward - composition - leafwise currents}, where $\phi_*$ is given by~\eqref{phi_*: I_c/cv(M' L' phi^*E otimes Omega_fiber) -> I_c/.(M L E)} for $E=\Lambda\FF$. The map~\eqref{phi_*: I_c/cv(FF') to I_c/cv(FF)} is also a restriction of the case of~\eqref{phi_*: C^pm infty_c/cv(M' Lambda FF') to C^pm infty_c/.(M Lambda FF)} for leafwise currents.

According to \Cref{ss: push-forward of the conormal seq}, the map~\eqref{phi_*: I_c/cv(FF') to I_c/cv(FF)} induces homomorphisms
\begin{gather}
\phi_*:K(\FF')\to K(\FF)\;,\label{phi_*:K(FF') to K(FF)}\\
\phi_*:J_{\co/\cv}(\FF')\to J_{\co/{\cdot}}(\FF)\;.\label{phi_*: J_c/cv(FF') to J_c/.(FF)}
\end{gather}
Like in \Cref{ss: pull-back of leafwise conormal currents}, we get induced continuous homomorphisms,
\begin{equation}\label{HK(FF') to HK(FF) ...}
\left\{
\begin{gathered}
\phi_*:H^\bullet K(\FF')\to H^\bullet K(\FF)\;,\\
\begin{alignedat}{2}
\phi_*:&H^\bullet I_\co(\FF')\to H^\bullet I_\co(\FF)\;,&\quad
\phi_*:&\bar H^\bullet I_\co(\FF')\to\bar H^\bullet I_\co(\FF)\;,\\
\phi_*:&H^\bullet J_\co(\FF')\to H^\bullet J_\co(\FF)\;,&\quad
\phi_*:&\bar H^\bullet J_\co(\FF')\to\bar H^\bullet J_\co(\FF)\;.
\end{alignedat}
\end{gathered}
\right.
\end{equation} 
The assignments of homomorphisms~\eqref{phi_*: I_c/cv(FF') to I_c/cv(FF)}--\eqref{HK(FF') to HK(FF) ...} are clearly functorial.

\subsection{Description of $\phi_*:K(\FF')\to K(\FF)$}\label{ss: description of phi_*: K(FF') to K(FF)}

For $\phi$ as above, consider the notation of \Cref{ss: description of phi^*:K(FF) to K(FF')}. Then
\begin{equation}\label{phi_*: C^infty(M^prime 0 Lambda) to C^infty(M^0 Lambda)}
\phi_*:C^\infty(M^{\prime\,0};\Lambda)\to C^\infty(M^0;\Lambda)
\end{equation}
is a cochain map for $d_{s\eta}$ and $d_{s\eta'}$ ($s\in\R$) (\Cref{ss: Witten vs pull-back and push-forward}). That is, $\phi$ induces
\begin{equation}
\label{phi_*: C^infty(M^prime 0 Lambda otimes Omega^sNM^prime 0) to C^infty(M^0 Lambda otimes Omega^sNM^0)}
\phi_*:C^\infty(M^{\prime\,0};\Lambda\otimes\Omega^sNM^{\prime\,0})\to C^\infty(M^0;\Lambda\otimes\Omega^sNM^0)\;,
\end{equation} 
given by
\begin{equation}\label{phi_*(alpha otimes |omega'|^s) = phi_*alpha otimes |omega|^s}
\phi_*(\alpha\otimes|\omega'|^s)=\phi_*\alpha\otimes|\omega|^s\;,
\end{equation}
which is another cochain map for the de~Rham differentials defined with the flat bundle structures of $\Omega^sNM^0$ and $\Omega^sNM^{\prime\,0}$ induced by the Bott flat $T\FF$-partial connection (\Cref{ss: infinitesmal transfs}). Like in~\eqref{phi^*(alpha otimes |d rho|^s) = phi^*alpha otimes |d rho'|^s} and~\eqref{partial_rho' phi^* = phi^* partial_rho}, we have
\begin{gather}
\phi_*(\alpha\otimes|d\rho'|^s)=\phi_*\alpha\otimes|d\rho|^s\;,
\label{phi_*(alpha otimes |d rho'|^s) = phi_*alpha otimes |d rho|^s}\\
\partial_\rho\phi_*=\phi_*\partial_{\rho'}\;,
\label{partial_rho phi_* = phi_* partial_rho'}
\end{gather}
where~\eqref{partial_rho phi_* = phi_* partial_rho'} holds as maps $\Cinftyc(T';\Lambda\FF')\to\Cinftyc(T;\Lambda\FF)$.

\begin{prop}\label{p: phi_* equiv bigoplus_k phi_*}
According to \Cref{p: K(Lambda FF) equiv bigoplus_k C^infty(M^0 Lambda)}, the map~\eqref{phi_*:K(FF') to K(FF)} is given by
\[
\phi_*\equiv\bigoplus_k\phi_*\equiv\bigoplus_k\phi_*\;,
\]
where the terms of the first direct sum are given by~\eqref{phi_*: C^infty(M^prime 0 Lambda) to C^infty(M^0 Lambda)}, and the terms of the second direct sum are given by~\eqref{phi_*: C^infty(M^prime 0 Lambda otimes Omega^sNM^prime 0) to C^infty(M^0 Lambda otimes Omega^sNM^0)}, taking $s=-k-1$.
\end{prop}

\begin{proof}
The second identity follows from the first one and~\eqref{phi_*(alpha otimes |omega'|^s) = phi_*alpha otimes |omega|^s}. To prove the first identity, by~\eqref{partial_rho phi_* = phi_* partial_rho'}, it is enough to consider the term with $k=0$. 

For $\beta\in C^\infty(M^{\prime\,0};\Lambda)$, let $v'=\beta\otimes|d\rho'|^{-1}\in C^\infty(M^{\prime\,0};\Lambda\otimes\Omega^{-1}NM^{\prime\,0})$. Like in the proof of \Cref{p: phi^* equiv bigoplus_k phi^*}, we have $v'\equiv\delta_{M^{\prime\,0}}^{v'}=\varpi^{\prime*}\beta\cdot\rho^{\prime*}\delta_0$ in $K(\FF')$. Moreover $v:=\phi_*v'=\phi_*\beta\cdot|d\rho|^{-1}$ by~\eqref{phi_*(alpha otimes |d rho'|^s) = phi_*alpha otimes |d rho|^s}, with $v\equiv\delta_{M^0}^v=\varpi^{\prime*}\phi_*\beta'\cdot\rho^*\delta_0$ in $K(\FF)$. Take a sequence $f_i\in\Cinftyc(-\epsilon,\epsilon)$ converging to $\delta_0$ in $C^{-\infty}_\co(-\epsilon,\epsilon)$. We get
\begin{align*}
\phi_*\delta_{M^{\prime\,0}}^{v'}&=\phi_*(\varpi^{\prime*}\beta\cdot\rho^{\prime*}\delta_0)
=\lim_i\phi_*(\varpi^{\prime*}\beta\cdot\rho^{\prime*}f_i)
=\lim_i\phi_*\varpi^*\beta\cdot\rho^*f_i\\
&=\lim_i\varpi^*\phi_*\beta\cdot\rho^*f_i
=\varpi^*\phi^*\beta\cdot\rho^*\delta_0
=\delta_{M^0}^v\;.\qedhere
\end{align*}
\end{proof}

The analog of \Cref{r: phi^* equiv prod_k phi^*} for $\phi_*:K(\FF')\to K(\FF)$ is true.

\subsection{Leafwise homotopy invariance}\label{ss: leafwise homotopy invariance}

With the notation of \Cref{ss: leafwise homotopy opers}, let $H:(M'\times I,\FF'\times I)\to(M,\FF)$ ($I=[0,1]$) be a smooth leafwise homotopy such that $H_0$ is transverse to $M^0$ and $H_0^{-1}(M^0)=M^{\prime\,0}$. Then, for every $p'\in M'$, the map $H_{t*}:N_{p'}\FF'\to N_{H_t(p')}\FF$ is the composition of $H_{0*}:N_{p'}\FF'\to N_{H_0(p')}\FF$ with the parallel transport along the leafwise path $s\in[0,t]\mapsto H_s(p')$. It follows that every $H_t$ is transverse to $M^0$ and $H_t^{-1}(M^0)=M^{\prime\,0}$. Hence $H$ is transverse to $M^0$ and $H^{-1}(M^0)=M^{\prime\,0}\times I$. Then, by~\eqref{sh = fint_I H^* leafwise} and according to \Cref{ss: pull-back of leafwise conormal currents,ss: push-forward of conormal leafwise currents}, the corresponding leafwise homotopy operator $\sh:C^\infty(M;\Lambda\FF)\to C^\infty(M';\Lambda\FF')$ has continuous linear extensions,
\[
\sh:K(\FF)\to K(\FF')\;,\quad\sh:I(\FF)\to I(\FF')\;,\quad\sh:J(\FF)\to J(\FF')\;.
\]
By continuity and according to \Cref{ss: leafwise homotopy opers}, we have $H_1^*-H_0^*=\sh d_\FF+d_{\FF'}\sh$ with $H_0^*$ and $H_1^*$ given by~\eqref{phi^*: I(FF) to I(FF')},~\eqref{phi^*: K(FF) to K(FF')} and~\eqref{phi^*: J(FF) to J(FF')}. Hence we get the following.

\begin{prop}\label{p: leafwise homotopy invariance - H^bullet K(FF)}
 Let $\phi,\psi:(M',\FF')\to(M,\FF)$ be smooth foliated maps transverse to $M^0$ with $\phi^{-1}(M^0)=\psi^{-1}(M^0)=M^{\prime\,0}$. If $\phi$ is leafwise homotopic to $\psi$, then $\phi$ and $\psi$ induce the same homomorphisms~\eqref{phi^*: H^bullet K(FF) to H^bullet K(FF')}.
\end{prop}

\section{Action of foliated flows on the conormal sequence}\label{s: action of foliated flows on the conormal seq}

Let $\phi=\{\phi^t\}$ be a foliated flow with transversely simple preserved leaves on a compact foliated manifold $(M,\FF)$. The homomorphisms~\eqref{phi^*: H^bullet K(FF) to H^bullet K(FF')} induced by the maps $\phi^t$ define actions of $\R$ on $H^\bullet K(\FF)$, $H^\bullet I(\FF)$ and $H^\bullet J(\FF)$, denoted by $\phi^*=\{\phi^{t*}\}$, and actions on $\bar H^\bullet I(\FF)$ and $\bar H^\bullet J(\FF)$, denoted by $\bar\phi^*=\{\bar\phi^{t*}\}$.  By \Cref{p: leafwise homotopy invariance - H^bullet K(FF)}, they only depend on the flow-leafwise-homotopy class of $\phi$ (\Cref{ss: fol maps}).

With the notation of \Cref{ss: globalization}, the foliated flow $\xi=\{\xi^t\}$ has transversely simple preserved leaves and satisfies $\bar\xi=\bar\phi$ and $\xi^t=\id$ on $M^0$. By \Cref{p: there exists a leafwise homotopy}, there is a flow-leafwise homotopy between $\phi$ and $\xi$, and therefore $\phi^*=\xi^*$ on $H^\bullet K(\FF)$. Consider the tubular neighborhood with defining function, $T_\epsilon\equiv(-\epsilon,\epsilon)_\rho\times M^0_\varpi$, like in \Cref{ss: description of phi^*:K(FF) to K(FF')}.

\begin{prop}\label{p: phi^t* equiv bigoplus_L k e^-(k+1) varkappa_L t}
According to \Cref{c: HK(FF) equiv bigoplus_k H_-k-1(M^0),r: bigoplus_L},
\[
\phi^{t*}\equiv\bigoplus_{k,L}e^{-(k+1)\varkappa_Lt}\equiv\bigoplus_{k,L}e^{-(k+1)\varkappa_Lt}
\]
on $H^\bullet K(\FF)$, where $k$ runs in $\N_0$ and $L$ in $\pi_0M^0$.
\end{prop}

\begin{proof}
Since $\xi^{t*}\rho=e^{\varkappa_Lt}\rho$ on every $T_{L,\epsilon}\cap\xi^{-t}(T_{L,\epsilon})$ by~\eqref{xi^t* rho = e^varkappa t rho}, it follows from~\eqref{phi^*(alpha otimes |d rho|^s) = phi^*alpha otimes |d rho'|^s} and \Cref{p: phi^* equiv bigoplus_k phi^*} that
\[
\xi^{t*}\equiv\bigoplus_{k,L}e^{-(k+1)\varkappa_Lt}\equiv\bigoplus_{k,L}e^{-(k+1)\varkappa_Lt}
\]
on $K(\FF)$, according to \Cref{p: K(Lambda FF) equiv bigoplus_k C^infty(M^0 Lambda)} and Remark~\ref{r: bigoplus_L}. Hence $\phi^{t*}=\xi^{t*}$ has the stated expression on $H^\bullet K(\FF)$.
\end{proof}

\Cref{p: K(Lambda FF) equiv bigoplus_k C^infty(M^0 Lambda),c: HK(FF) equiv bigoplus_k H_-k-1(M^0),r: bigoplus_L,p: phi^t* equiv bigoplus_L k e^-(k+1) varkappa_L t} show \Cref{t: intro - H^bullet K(FF)}.

\chapter{Dual-conormal leafwise reduced cohomology}\label{ch: dual-conormal}

\section{Dual-conormal sequence of leafwise differential forms}
\label{s: dual-conormal seq of leafwise diff forms}

Assume the conditions of \Cref{s: conormal seq of leafwise currents} on $(M,\FF)$. According to~\Cref{ss: diff opers on dual-conormal distribs}, the LCHS \index{$I'(\FF)$}
\[
I'(\FF)=I'\Lambda^\bullet (\FF):=I'(M,M^0;\Lambda\FF)
\]
is a topological complex with $d_\FF$. It induces the \emph{dual-conormal leafwise cohomology} and \emph{dual-conormal leafwise reduced cohomology} of $\FF$ (or of $(\FF,M^0)$). The notation $BI'(\FF)$, $ZI'(\FF)$, $\bar BI'(\FF)$, $H^\bullet I'(\FF)$ \index{$H^\bullet I'(\FF)$} and $\bar H^\bullet I'(\FF)$ \index{$\bar H^\bullet I'(\FF)$} is used like in \Cref{s: conormal seq of leafwise currents}.

For a leafwise flat vector bundle $E$, we can also consider the topological complex
\[
I'(\FF;E)=I'\Lambda^\bullet(\FF;E)=I'(M,M^0;\Lambda\FF\otimes E)
\]
with $d_\FF$. The LCHS $I'(\FF;E)$ is also defined for an arbitrary vector bundle $E$, missing the leafwise differential map $d_\FF$.

Moreover, the LCHSs
\[
I^{\prime\,(s)}(\FF)=I^{\prime\,(s)}\Lambda^\bullet(\FF)=I^{\prime\,(s)}(M,M^0;\Lambda\FF)\quad(s\in\R)\;.
\]
also become topological complexes with $d_\FF$. The notation $BI^{\prime\,(s)}(\FF)$, $ZI^{\prime\,(s)}(\FF)$, $\bar BI^{\prime\,(s)}(\FF)$, $H^\bullet I^{\prime\,(s)}(\FF)$\index{$H^\bullet I^{\prime\,(s)}(\FF)$} and $\bar H^\bullet I^{\prime\,(s)}(\FF)$\index{$\bar H^\bullet I^{\prime\,(s)}(\FF)$} is used like in \Cref{s: conormal seq of leafwise currents}. 

\begin{rem}\label{r: coefficients in Omega M}
Although $\Omega M$ has no leafwise flat structure in general, we can assume $\FF$ is oriented by working locally or passing to the double cover of orientations of $\FF$. Then we can apply~\eqref{(Lambda^v FF)^* otimes Omega M equiv Lambda^n' n''-v M}--\eqref{d_FF equiv (-1)^v+1 d_FF^t} and the leafwise flat structure of $\Omega N\FF$ to define $d_\FF$ and $d_\FF^\trans$ on every $I^{(s)}(\FF;\Omega M)\equiv I^{(s)}(\FF;\Omega N\FF)$. Since the condition of being in $I^{(s)}(\FF;\Omega M)$ is local for elements of $C^{-\infty}(M;\Lambda\FF\otimes\Omega)$, this procedure gives the definition of $d_\FF=d_\FF^{\text{\rm tt}}$.
\end{rem}

For $s'\le s$ in $\R$, we have the continuous linear restriction maps (\Cref{ss: dual-conormal distribs - compact})
\begin{equation}\label{j'_s}
j'_s:I'(\FF)\to I^{\prime\,(s)}(\FF)\;,\quad j'_{s,s'}:I^{\prime\,(s)}(\FF)\to I^{\prime\,(s')}(\FF)\;,
\end{equation}
where $j'_s=j_{-s}^\trans$ and $j'_{s,s'}=j_{-s',-s}^\trans$ for the version of~\eqref{j_s} with $\Omega M$. The induced homomorphisms in cohomology and reduced cohomology are denoted by  $j'_{s*}$, $j'_{s,s'*}$, $\bar\jmath'_{s*}$ and $\bar\jmath'_{s,s'*}$. The homomorphisms $j'_{s,s'*}$ and $\bar\jmath'_{s,s'*}$ form projective spectra, giving rise to projective limits as $s\uparrow+\infty$. Like in~\eqref{varinjlim H^bullet I^(s)(FF) cong H^bullet I(FF)}, the maps $j'_{s*}$ and $\bar\jmath'_{s*}$ induce canonical continuous linear maps,
\begin{equation}\label{bar H^bullet I'(FF) cong projlim bar H^bullet I^(prime s)(FF)}
\left\{
\begin{gathered}
\tilde\jmath'_*:=\varprojlim j'_{s*}:H^\bullet I'(\FF)\to\widetilde H^\bullet I'(\FF)
:=\varprojlim H^\bullet I^{\prime\,(s)}(\FF)\;,\\
\hat\jmath'_*:=\varprojlim\bar\jmath'_{s*}:\bar H^\bullet I'(\FF)\xrightarrow{\cong}
\widehat H^\bullet I'(\FF):=\varprojlim\bar H^\bullet I^{\prime\,(s)}(\FF)\;,
\end{gathered}
\right.
\end{equation}
where the second one is a linear isomorphism (\Cref{s: proj lims in reduced cohom}). The canonical maps of the inductive limits to the steps are denoted by \index{$\tilde\jmath_{s*}$} \index{$\hat\jmath_{s*}$}
\[
\tilde\jmath_{s*}:\widetilde H^\bullet I'(\FF)\to H^\bullet I^{\prime\,(s)}(\FF)\;,\quad
\hat\jmath_{s*}:\widehat H^\bullet I'(\FF)\to\bar H^\bullet I^{\prime\,(s)}(\FF)\;.
\]

Using the above type of notation, the LCHSs $J'(\FF)$ and $K'(\FF)$ are also topological complexes with $d_\FF$ (\Cref{ss: action of Diff(M) on the dual-conormal seq}), with corresponding spaces $BJ'(\FF)$, $ZJ'(\FF)$, $\bar BJ'(\FF)$, $H^\bullet J'(\FF)$\index{$H^\bullet J'(\FF)$} and $\bar H^\bullet J'(\FF)$,\index{$\bar H^\bullet J'(\FF)$} and the same for $K'(\FF)$.\index{$H^\bullet K'(\FF)$}

Similarly, we have topological complexes $J^{\prime\,(s)}(\FF)$, $J^{\prime\,m}(\FF)$ and $K^{\prime\,(s)}(\FF)$ ($s,m\in\R$) (\Cref{ss: dual-conormal seq,ss: Diff(M) on the conormal seq}), with corresponding spaces $BJ^{\prime\,(s)}(\FF)$, $ZJ^{\prime\,(s)}(\FF)$, $\bar BJ^{\prime\,(s)}(\FF)$, $H^\bullet J^{\prime\,(s)}(\FF)$ and $\bar H^\bullet J^{\prime\,(s)}(\FF)$, and the same for $J^{\prime\,m}(\FF)$ and $K^{\prime\,(s)}(\FF)$. There are obvious versions for $J^{\prime\,(s)}(\FF)$ and $K^{\prime\,(s)}(\FF)$ of the maps~\eqref{j'_s} (\Cref{ss: K'(M L) and J'(M L)}), also denoted by $j'_s$ and $j'_{s,s'}$, giving rise to projective spectra in cohomology and reduced cohomology, and the corresponding projective limits. In the case of $J'(\FF)$, the maps $j'_{s,s'}$ and $j'_s$ are continuous inclusions (\Cref{ss: K'(M L) and J'(M L)}).

There are also continuous inclusion maps (\Cref{ss: K'(M L) and J'(M L)})
\begin{equation}\label{j_m}
\left\{
\begin{gathered}
j'_m:J'(\FF)\hookrightarrow J^{\prime\,m}(\FF)\;,\quad 
j'_{m,m'}:J^{\prime\,m}(\FF)\hookrightarrow J^{\prime\,m'}(\FF)\quad(m'\le m)\;,\\
j'_{m,s}:J^m(\FF)\hookrightarrow J^{(s)}(\FF)\quad(m>s+n/2+1)\;,\\
j'_{s,m}:J^{(s)}(\FF)\hookrightarrow J^m(\FF)\quad(s\ge m,0)\;, 
\end{gathered}
\right.
\end{equation}
denoted like in~\eqref{j_s} with some abuse of notation. The homomorphisms induced by the maps $j'_{m,m'}$ in cohomology and reduced cohomology form projective spectra whose inductive limits as $m\uparrow+\infty$ agree with the previous ones for $J(\FF)$, and the maps $j'_m$ induce a continuous linear isomorphism analogous to~\eqref{bar H^bullet I'(FF) cong projlim bar H^bullet I^(prime s)(FF)}.

It will be shown (Corollary~\ref{c: H^bullet K'(FF) equiv prod_k H^bullet_k(M^0)}) that the canonical projections are TVS-identities, 
\begin{equation}\label{H^bullet K'(FF) equiv bar H^bullet K'(FF)}
\left\{
\begin{gathered}
H^\bullet K'(\FF)\equiv\bar H^\bullet K'(\FF)\;,\quad 
H^\bullet K^{\prime\,(s)}(\FF)\equiv\bar H^\bullet K^{\prime\,(s)}(\FF)\;,\\
H^\bullet K'(\FF)\equiv\varprojlim H^\bullet K^{\prime\,(s)}(\FF)\;,\quad\bar H^\bullet K'(\FF)\equiv\varprojlim\bar H^\bullet K^{\prime\,(s)}(\FF)\;.
\end{gathered}
\right.
\end{equation}

The version of the bottom row of~\eqref{CD: dual-conormal seqs} with $\Lambda\FF$  is a short exact sequence of continuous homomorphisms of topological complexes,
\begin{equation}\label{leafwise dual-conormal exact seq}
0\leftarrow K'(\FF) \xleftarrow{R'} I'(\FF) \xleftarrow{\iota'} J'(\FF)\leftarrow0\;,
\end{equation}
using the notation $R'=\iota^\trans$ \index{$R'$} and $\iota'=R^\trans$. \index{$\iota'$} The exactness of the induced sequences,
\begin{gather}
0\leftarrow H^\bullet K'(\FF) \xleftarrow{R'_*} H^\bullet I'(\FF) \xleftarrow{\iota'_*} H^\bullet J'(\FF)\leftarrow0\;,
\label{exact seq in cohom - dual-conormal}\\
0\leftarrow H^\bullet K'(\FF) \xleftarrow{\bar R'_*} \bar H^\bullet I'(\FF) \xleftarrow{\bar\iota'_*} \bar H^\bullet J'(\FF)\leftarrow0\;,
\label{exact seq in reduced cohom - dual-conormal}
\end{gather}
will be proved in \Cref{s: short exact seq - dual-conormal}; in particular, this shows \Cref{t: intro - reduced dual-conormal cohomology exact sequence}.

Taking the transpose of the analog of~\eqref{iota_s} with $\Omega M$, we get continuous linear maps \index{$R'_s$} \index{$\iota'_s$}
\begin{equation}\label{R'_s}
R'_s:K^{\prime\,(s)}(\FF) \to I^{\prime\,(s)}(\FF)\;,\quad \iota'_s:I^{\prime\,(s)}(\FF)\to J^{\prime\,(s)}(\FF)\;.
\end{equation}
Like in \Cref{s: conormal seq of leafwise currents}, the subscript ``$s$'' may be also added to the elements of the cohomologies or reduced cohomologies of $K^{\prime\,(s)}(\FF)$, $I^{\prime\,(s)}(\FF)$ or $J^{\prime\,(s)}(\FF)$.

\section{Projective limits in reduced cohomology}\label{s: proj lims in reduced cohom}

The goal of this section is to prove the linear isomorphism~\eqref{bar H^bullet I'(FF) cong projlim bar H^bullet I^(prime s)(FF)}, and its version for $J'(\FF)$. The case of $K'(\FF)$ is given by~\eqref{H^bullet K'(FF) equiv bar H^bullet K'(FF)}.

To simplify the notation, we write \index{$\widetilde H^\bullet I'(\FF)$} \index{$\widehat H^\bullet I'(\FF)$}
\[
\widetilde H^\bullet I'(\FF)=\varprojlim H^\bullet I^{\prime\,(s)}(\FF)\;,\quad
\widehat H^\bullet I'(\FF)=\varprojlim\bar H^\bullet I^{\prime\,(s)}(\FF)\;,
\]
and the canonical maps of the projective limits to the steps are denoted by \index{$\tilde\jmath'_s$} \index{$\hat\jmath'_s$}
\[
\tilde\jmath'_s:\widetilde H^\bullet I'(\FF)\to H^\bullet I^{\prime\,(s)}(\FF)\;,\quad
\hat\jmath'_s:\widehat H^\bullet I'(\FF)\to\bar H^\bullet I^{\prime\,(s)}(\FF)\;.
\]
The same type of notation is used in the cases of $J'(\FF)$ and $K'(\FF)$.

\begin{lem}\label{l: j'_s: BJ'(FF) to BJ^prime (s)(FF) is dense}
$BI'(\FF)$ is dense in every $BI^{\prime\,(s)}(\FF)$ is dense.
\end{lem}

\begin{proof}
Use that the image of $J'(\FF)$ is dense in $J^{\prime\,(s)}(\FF)$ (\Cref{ss: dual-conormal distribs - compact}) and $d_\FF$ is continuous on $J'(\FF)$ and $J^{\prime\,(s)}(\FF)$ (\Cref{ss: action of Diff(M) on the dual-conormal seq}).
\end{proof}

Recall that $\bar BJ'(\FF)$ (resp., $\bar BJ^{\prime\,(s)}(\FF)$) denotes the closure of $BJ'(\FF)$ (resp., $BJ^{\prime\,(s)}(\FF)$) in $J'(\FF)$ (resp., $J^{\prime\,(s)}(\FF)$).

\begin{cor}\label{c: bar BJ'(FF) = bigcap_s bar BJ^prime (s)(FF)}
As vector spaces,
\[
\bar BJ'(\FF)=\bigcap_s\bar BJ^{\prime\,(s)}(\FF)\;.
\]
\end{cor}

\begin{proof}
By the definition of the projective topology of $\bigcap_sJ^{\prime\,(s)}(\FF)$ \cite[Section~II.5]{Schaefer1971} and using \Cref{l: j'_s: BJ'(FF) to BJ^prime (s)(FF) is dense}, we get that $BJ'(\FF)$ is dense in $\bigcap_s\bar BJ^{\prime\,(s)}(\FF)$. Moreover, this intersection is closed in $J'(\FF)$. Then the stated equality is true.
\end{proof}

\begin{lem}\label{l: ZJ'(FF) = bigcap_s ZJ^prime (s)(FF)}
As vector spaces,
\[
ZJ'(\FF)=\bigcap_s ZJ^{\prime\,(s)}(\FF)\;.
\]
\end{lem}

\begin{proof}
Consider the commutative diagram
\[
\begin{CD}
0 @>>> \bigcap_sZJ^{\prime\,(s)}(\FF) @>>> \bigcap_sJ^{\prime\,(s)}(\FF) @>{d_\FF}>> \bigcap_sBJ^{\prime\,(s)}(\FF) \\
&& @AAA @| @AAA & \\
0 @>>> ZJ'(\FF) @>>> J'(\FF) @>{d_\FF}>> BJ'(\FF) @>>> 0\;.
\end{CD}
\]
Here, the central vertical equality is the analog~\eqref{dot AA'(M) = bigcap_s dot AA^prime (s)(M) = bigcap_m dot AA^prime m(M)}, the arrows that are not given by $d_\FF$ and do not go to $0$ denote inclusion maps, and the bottom row is exact. Since the surjective maps $d_\FF:J^{\prime\,(s)}(\FF)\to BJ^{\prime\,(s)}(\FF)$ form a homomorphism between projective spectra whose kernel is the projective spectrum consisting of the spaces $ZJ^{\prime\,(s)}(\FF)$, the top row is also exact \cite[Proposition~3.1.8]{Wengenroth2003}. Thus the left-hand-side vertical arrow is an equality of vector spaces.  
\end{proof}

\begin{prop}\label{p: hat jmath'_* is a linear iso}
The canonical map $\bar H^\bullet J'(\FF)\to\widehat H^\bullet J'(\FF)$ is a linear isomorphism. 
\end{prop}

\begin{proof}
Consider the commutative diagram
\[
\begin{CD}
0 @>>> \bigcap_s\bar BJ^{\prime\,(s)}(\FF) @>>> \bigcap_sZJ^{\prime\,(s)}(\FF) 
@>>> \widehat H^\bullet J'(\FF) @>>>0\phantom{\;.} \\
&& @| @| @AAA & \\
0 @>>> \bar BJ'(\FF) @>>> ZJ'(\FF) @>>> \bar H^\bullet J'(\FF) @>>> 0\;.
\end{CD}
\]
Here,  \Cref{c: bar BJ'(FF) = bigcap_s bar BJ^prime (s)(FF),l: ZJ'(FF) = bigcap_s ZJ^prime (s)(FF)} give the vertical equalities of vector spaces, the vertical arrow is canonical, and the other maps are canonical; in particular, the bottom row is exact. \Cref{l: j'_s: BJ'(FF) to BJ^prime (s)(FF) is dense} also shows that every $BJ^{(s)}(\FF)$ is dense in $BJ^{(s')}(\FF)$ for $s'<s$. Hence the right derived functor $\varprojlim^1$ satisfies $\varprojlim^1\bar BJ^{(s)}(\FF)=0$ as $s\uparrow+\infty$ \cite[Theorem~3.2.1]{Wengenroth2003}, obtaining that the top row is also exact by \cite[Corollary~3.1.5]{Wengenroth2003}. Then the result follows.
\end{proof}

On the other hand, the kind of arguments that will be given in \Cref{s: short exact seq - dual-conormal} can be adapted to show the exactness of the sequence \index{$\widehat R'_*$} \index{$\hat\iota'_*$}
\begin{equation}\label{exact seq in the proj limit of reduced cohoms - dual-conormal}
0\leftarrow H^\bullet K'(\FF) \xleftarrow{\widehat R'_*} \widehat H^\bullet I'(\FF) 
\xleftarrow{\hat\iota'_*} \widehat H^\bullet J'(\FF) \leftarrow0\;,
\end{equation}
where $\widehat R'_*=\varprojlim\bar R'_{s*}$ and $\hat\iota'_*=\varprojlim\bar\iota'_{s*}$, using the homomorphisms induced by~\eqref{R'_s}. This fits into a commutative diagram
 \[
 \begin{CD}
 0 @<<< H^\bullet K'(\FF) @<<< \bar H^\bullet I'(\FF) @<<< \bar H^\bullet J'(\FF) @<<< 0\phantom{\;,} \\
&& @| @VVV @VV{\cong}V  \\
 0 @<<< H^\bullet K'(\FF) @<<< \widehat H^\bullet I'(\FF) @<<< \widehat H^\bullet J'(\FF) @<<<0\;,
 \end{CD}
 \]
where the top row is the exact sequence~\eqref{exact seq in reduced cohom - dual-conormal}, and the vertical arrows are canonical. The last vertical arrow is a linear isomorphism by \Cref{p: hat jmath'_* is a linear iso}. Then the central vertical arrow is also a linear isomorphism by the five lemma.

\section{Description of $H^\bullet K'(\FF)$}\label{s: bar H K'(FF)}

As explained in \Cref{s: dual-conormal seq of leafwise diff forms}, there is no loss of generality in assuming $\FF$ is oriented, and then we can apply~\eqref{(Lambda^v FF)^* otimes Omega M equiv Lambda^n' n''-v M}--\eqref{d_FF equiv (-1)^v+1 d_FF^t} to get $K^{(s)}(\FF;\Omega M)\equiv K^{(s)}(\FF;\Omega N\FF)$, where we can consider $d_\FF$ or $d_\FF^\trans$ using the leafwise flat structure of $\Omega N\FF$. 

Consider the notation of \Cref{s: HK(FF)}. Since $d\omega=\eta\wedge\omega$ and $d_\FF$ satisfies the derivation rule on products of smooth leafwise currents and smooth leafwise forms (\Cref{ss: leafwise currents}), it follows that the version of \Cref{p: K(Lambda FF) equiv bigoplus_k C^infty(M^0 Lambda)} with coefficients in $\Omega N\FF$ states that
\begin{equation}\label{K(FF Omega N FF) equiv bigoplus_k C^infty(M^0 Lambda)}
\left\{
\begin{gathered}
K(\FF;\Omega N\FF)\equiv\bigoplus_kC^\infty(M^0;\Lambda)
\equiv\bigoplus_kC^\infty(M^0;\Lambda\otimes\Omega^{-k}NM^0)\;,\\
d_\FF\equiv\bigoplus_kd_{-k}\equiv\bigoplus_kd\;,
\end{gathered}
\right.
\end{equation}
where $k$ runs in $\N_0$. Moreover the subcomplex $K^{(s)}(\FF;\Omega N\FF)\subset K(\FF;\Omega N\FF)$ corresponds to the finite direct sum with $k<-s-1/2$. Taking dual spaces and transposing maps, using~\eqref{C^-infty(M Lambda^r) equiv C^infty_c(M Lambda^n-r)'} and~\eqref{d_z equiv (-1)^k+1 d_-z^t}, we get the following consequence.

\begin{cor}\label{c: K'(Lambda FF) equiv prod_k C^-infty(M^0 Lambda)}
We have identities of topological complexes,
\begin{gather*}
K'(\FF)\equiv\prod_kC^{-\infty}(M^0;\Lambda)
\equiv\prod_kC^{-\infty}(M^0;\Lambda\otimes\Omega^kNM^0)\;,\\
d_\FF\equiv\prod_kd_k\equiv\prod_kd\;,
\end{gather*}
where $k$ runs in $\N_0$. Moreover the quotient complex $K^{\prime\,(s)}(\FF)$ corresponds to the finite direct sum with $k<s-1/2$.
\end{cor}

\begin{cor}\label{c: H^bullet K'(FF) equiv prod_k H^bullet_k(M^0)}
We have TVS-identities,
\[
H^\bullet K'(\FF)\equiv\prod_kH_k^\bullet(M^0)
\equiv\prod_kH^\bullet(M^0,\Omega^kNM^0)\;,
\]
where $k$ runs in $\N_0$. Moreover $H^\bullet K^{\prime\,(s)}(\FF)$ is the quotient space of $H^\bullet K'(\FF)$ given by the finite product with $k<s-1/2$. In particular,~\eqref{H^bullet K'(FF) equiv bar H^bullet K'(FF)} is satisfied.
\end{cor}

\begin{rem}\label{r: prod_L}
The differential complexes on $M^0$ used in \Cref{c: K'(Lambda FF) equiv prod_k C^-infty(M^0 Lambda)} split into direct sums of the same complexes given by leaves $L\subset M^0$. The same applies to their cohomologies, used in Corollary~\ref{c: H^bullet K'(FF) equiv prod_k H^bullet_k(M^0)}.
\end{rem}

\begin{cor}\label{c: H^bullet K'(FF) equiv H^bullet(M^0) oplus H^n''-bullet K'(FF)'}
There is a canonical TVS-isomorphism,
\[
H^\bullet K'(\FF)\equiv H^\bullet(M^0)\oplus H^{n''-\bullet}K(\FF)'\;,
\]
\end{cor}

\begin{proof}
Apply \Cref{c: HK(FF) equiv bigoplus_k H_-k-1(M^0),c: H^bullet K'(FF) equiv prod_k H^bullet_k(M^0)}, and~\eqref{H_z^k(M) times H_-z^n-kM) to C}.
\end{proof}

\section{Description of $\bar H^\bullet J'(\FF)$}\label{s: description of bar H^bullet J'(FF)}

Like in \Cref{s: description of bar H^bullet J(FF)}, by~\eqref{J^prime m(M L) cong ...} and~\eqref{J'(M L) cong ...}, for $m\in\R$,
\begin{gather}
J^{\prime\,m}(\FF)\cong\bfrho^m\Hb^{-\infty}(\bfM;\Lambda\bfFF)
\equiv\bfrho^{m-\frac12}H^{-\infty}(\mathring\bfM;\Lambda\mathring\bfFF)\;,
\label{J^prime m(FF) cong ...}\\
J'(\FF)\cong\bigcap_m\bfrho^m\Hb^{-\infty}(\bfM;\Lambda\bfFF)
=\bigcap_m\bfrho^mH^{-\infty}(\mathring\bfM;\Lambda\mathring\bfFF)\;,
\label{J'(FF) cong ...}
\end{gather}
as topological complexes with $d_\FF$, $d_{\bfFF}$ and $d_{\mathring\bfFF}$, using the b-metric $\bfg$ to define $\Hb^{-\infty}(\bfM;\Lambda\bfFF)$, and using $\bfg|_{\mathring\bfM}$ to define $H^{-\infty}(\mathring\bfM;\Lambda\mathring\bfFF)$. The leafwise version of~\eqref{Witten's opers} (\Cref{s: Witten's opers on Riem folns of bd geom}) also gives isomorphisms of topological complexes,
\begin{equation}\label{bfrho^-m+frac12 - dual-conormal}
\bfrho^{-m+\frac12}:\big(\bfrho^{m-\frac12}H^\infty(\mathring\bfM;\Lambda\mathring\bfFF),d_{\mathring\bfFF}\big)\xrightarrow{\cong}\big(H^\infty(\mathring\bfM;\Lambda\mathring\bfFF),d_{\mathring\bfFF,m-\frac12}\big)\;.
\end{equation}

By~\eqref{J^prime m(FF) cong ...} and~\eqref{bfrho^-m+frac12 - dual-conormal}, and the analog of~\eqref{leafwise Hodge iso} for $\Delta_{\mathring\bfFF,m-\frac12}$ in $H^{-\infty}(\mathring\bfM;\Lambda\mathring\bfFF)$ (\Cref{s: Witten's opers on Riem folns of bd geom}), we get induced TVS-isomorphisms
\begin{align}
\bar H^\bullet J^{\prime\,m}(\FF)
&\cong\bar H^\bullet\big(\bfrho^{m-\frac12}H^{-\infty}(\mathring\bfM;\Lambda\mathring\bfFF),d_{\mathring\bfFF}\big)
\label{bar H^bullet J^prime m(FF) cong bar H^bullet(...)}\\
&\cong\bar H^\bullet\big(H^{-\infty}(\mathring\bfM;\Lambda\mathring\bfFF),d_{\mathring\bfFF,m-\frac12}\big)
\label{bar H^bullet(...) cong bar H^bullet(...) - dual-conormal}\\
&\cong\ker\Delta_{\mathring\bfFF,m-\frac12}\;.\label{bar H^bullet J^prime m(FF) cong ker ...}
\end{align}
By the analog of~\eqref{bar H^bullet I'(FF) cong projlim bar H^bullet I^(prime s)(FF)} for $J'(\FF)$ and~\eqref{bar H^bullet J^prime m(FF) cong ker ...}, the LCHS $\bar H^\bullet J(\FF)$ is a projective limit of a sequence of Hilbertian spaces, and therefore a Fr\'echet space. The isomorphisms~\eqref{bar H^bullet J^prime m(FF) cong bar H^bullet(...)} and~\eqref{bar H^bullet(...) cong bar H^bullet(...) - dual-conormal} are also true in cohomology.

\Cref{t: intro - bar H^bullet J'(FF)} follows from the analog of~\eqref{bar H^bullet I'(FF) cong projlim bar H^bullet I^(prime s)(FF)} for $J'(\FF)$ and~\eqref{J^prime m(FF) cong ...}--\eqref{bfrho^-m+frac12 - dual-conormal}.

\section{Short exact sequence of dual-conormal reduced cohomology}
\label{s: short exact seq - dual-conormal}

The goal of this section is to prove the exactness of~\eqref{exact seq in reduced cohom - dual-conormal}. Some remarks will indicate how to adapt the proof to show also the exactness of~\eqref{exact seq in cohom - dual-conormal} and~\eqref{exact seq in the proj limit of reduced cohoms - dual-conormal}.

\subsection{The maps $F'_m$}\label{ss: F'_m}

For every $m\in\R$, let \index{$F'_m$}
\[
F'_m=E_{-m}^\trans:I^{\prime\,(s)}(\FF)\to J^{\prime\,m}(\FF)\;,
\]
where $s=0$ if $m\le0$, and $m<s\in\Z^+$ if $m>0$, where
\[
E_{-m}:J^{-m}(\FF;\Omega M)\to I^{(-s)}(\FF;\Omega M)
\]
is given by the version of \Cref{c: E_m d_FF = d_FF E_m} with $\Omega M$ (see \Cref{r: coefficients in Omega M}); thus
\begin{equation}\label{F'_m d_FF = d_FF F'_m}
F'_md_\FF=d_\FF F'_m\;.
\end{equation}
Since $s\ge m,0$, the map $j'_{s,m}$ is defined, and we have
\begin{equation}\label{F'_m iota'_s = j'_s m}
F'_m\iota'_s=E_{-m}^\trans R_{-s}^\trans=(R_{-s}E_{-m})^\trans=j_{-m,-s}^\trans=j'_{s,m}\;.
\end{equation}

\subsection{The maps $E'_m$}\label{ss: E'_m}

For $s\in\R$ and $m>s+n/2+1$, let \index{$E'_m$}
\[
E'_m=F_{-m}^\trans:K^{\prime\,(s')}(\FF)\to I^{\prime\,(s)}(\FF)\;,
\]
where $s'=0$ if $m\le0$, and $m<s'\in\Z^+$ if $m>0$. Here, we use the map
\[
F_{-m}:I^{(-s)}(\FF;\Omega M)\to K^{(-s')}(\FF;\Omega M)
\]
given be the version of \Cref{ss: F_m} with coefficients in $\Omega M$ (\Cref{r: F_m with Omega M}). Consider
\begin{equation}\label{F'_m: I^prime (s')(FF) to J^prime (s)(FF)}
F'_m=j'_{m,s}F'_m:I^{\prime\,(s')}(\FF)\to J^{\prime\,(s)}(\FF)\;,
\end{equation}
which is the transpose of the version of~\eqref{E_m: J^(s)(FF) to I^(s')(FF)} with coefficients in $\Omega M$,
\[
E_{-m}:J^{(-s)}(\FF;\Omega M)\to I^{(-s')}(\FF;\Omega M)\;.
\]
Then~\eqref{F'_m iota'_s = j'_s m} becomes
\begin{equation}\label{F'_m iota'_s = j'_s s'}
F'_m\iota'_s=j'_{s,s'}\;.
\end{equation}
Transposing the versions of~\eqref{E_m R_s + iota_s' F_m = j_s s'}--\eqref{F_m d_FF = d_FF F_m} with coefficients in $\Omega M$, we get
\begin{gather}
\iota'_sF'_m+E'_mR'_{s'}=j'_{s',s}:I^{\prime\,(s')}(\FF)\to I^{\prime\,(s)}(\FF)\;,\label{iota'_s' F'_m + E'_m R'_s = j'_s' s}\\
R'_sE'_m=j'_{s',s}:K^{\prime\,(s')}(\FF)\to K^{\prime\,(s)}(\FF)\;,\label{R'_s E'_m = j'_s' s}\\
E'_m d_\FF=d_\FF E'_m\;.\label{E'_m d_FF = d_FF E'_m}
\end{gather}

Take greater numbers, $s_1>s$, $m_1>m$ and $s'_1>s'$, satisfying the same inequalities as $s$, $m$ and $s'$. Using~\eqref{F'_m: I^prime (s')(FF) to J^prime (s)(FF)}, the transposition of the versions of~\eqref{j_s' s'_1 E_m = E_m_1 j_s s_1} and~\eqref{j_s' s'_1 F_m = F_m_1 j_s s_1} with coefficients in $\Omega M$ give
\begin{align}
F'_mj'_{s'_1,s'}&=j'_{s_1,s}F'_{m_1}\;,\label{F'_m j'_s'_1 s' = j'_s_1 s F'_m_1}\\
E'_mj'_{s'_1,s'}&=j'_{s_1,s}E'_{m_1}\;.\label{E'_m j'_s'_1 s' = j'_s_1 s E'_m_1}
\end{align}

\subsection{The equality $\ker\bar R'_*=\im\bar\iota'_*$}\label{ss: ker bar R'_* = im bar iota'_*}

We already know that $\ker\bar R'_*\supset\im\bar\iota'_*$. To prove $\ker\bar R'_*\subset\im\bar\iota'_*$, take any class $\overline{[u]}\in\ker\bar R'_*$ in $\bar H^\bullet I'(\FF)$. Thus there is some net $v_l$ in $K'(\FF)$ such that $R'u=\lim_ld_\FF v_l$ in $K'(\FF)$. Write $u_s=j'_su\in I^{\prime\,(s)}(\FF)$ and $v_{l,s}=j'_sv_l\in K^{\prime\,(s)}(\FF)$. Take $s$, $m$ and $s'$ satisfying the the conditions of \Cref{ss: E'_m}, obtaining $E'_m:K^{\prime\,(s')}(\FF)\to I^{\prime\,(s)}(\FF)$ and $F'_m:I^{\prime\,(s')}(\FF)\to J^{\prime\,(s)}(\FF)$. Let $a_s=F'_mu_{s'}\in J^{\prime\,(s)}(\FF)$ and $b_{l,s}=E'_mv_{l,s'}\in I^{\prime\,(s)}(\FF)$. By~\eqref{F'_m d_FF = d_FF F'_m},
\[
d_\FF a_s=F'_md_\FF u_{s'}=0\;.
\]
Moreover, by~\eqref{iota'_s' F'_m + E'_m R'_s = j'_s' s} and~\eqref{E'_m d_FF = d_FF E'_m},
\begin{align*}
u_s&=j'_{s',s}u_{s'}=\iota'_sF'_mu_{s'}+E'_mR'_{s'}u_{s'}\\
&=\iota'_sa_s+\lim_lE'_md_\FF v_{l,s'}
=\iota'_sa_s+\lim_ld_\FF b_{l,s}\;.
\end{align*}

Now consider the above notation for greater real numbers $s_1$, $m_1$ and $s'_1$, satisfying the same properties as $s$, $m$ and $s'$. By~\eqref{F'_m j'_s'_1 s' = j'_s_1 s F'_m_1} and~\eqref{E'_m j'_s'_1 s' = j'_s_1 s E'_m_1},
\begin{gather*}
j'_{s_1,s}a_{s_1}=j'_{s_1,s}F'_{m_1}u_{s'_1}=F'_{m}j'_{s'_1,s'}u_{s'_1}=F'_{m}u_{s'}=a_s\;.\\
j'_{s_1,s}b_{l,s_1}=j'_{s_1,s}E'_{m_1}v_{l,s'_1}=E'_{m}j'_{s'_1,s'}v_{l,s'_1}=E'_{m}v_{l,s'}=b_{l,s}\;.
\end{gather*}
Therefore, taking $s\uparrow+\infty$, $m\uparrow+\infty$ and $s'\uparrow+\infty$, satisfying the above relations, the elements $a_s\in J^{\prime\,(s)}(\FF)$ and $b_{l,s}\in I^{\prime\,(s)}(\FF)$ define elements $a:=(a_s)_s\in ZJ'(\FF)$ and $b_l:=(b_{l,s})_s\in I'(\FF)$, and we have $u=\iota'a+\lim_ld_\FF b_l$. Hence $\overline{[u]}=\bar\iota_*(\overline{[a]})$.

\begin{rem}\label{r: ker R'_* = im iota'_*}
A similar argument, taking an element $v\in K'(\FF)$ instead of a net $v_l$, shows the inclusion $\ker R'_*=\im\iota'_*$ in $H^\bullet I'(\FF)$.
\end{rem}

\begin{rem}\label{r: ker widehat R'_* = im hat iota'_*}
As before, to prove $\ker\widehat R'_*=\im\hat\iota'_*$ in~\eqref{exact seq in the proj limit of reduced cohoms - dual-conormal}, we only have to prove ``$\subset$''. For any $\hat u:=(\overline{[u_s]}_s)_s\in\ker\widehat R'_*$, there is some $v\in K'(\FF)$ such that $R'_su_s=d_\FF v_s$, where $v_s=j'_sv$. Moreover, $j'_{s',s}u_{s'}=u_s+\lim_ld_\FF g_{l,s,s'}$ for some net $g_{l,s,s'}$ in $I^{\prime\,(s)}(\FF)$. Take $a_s$ and $b_{s,l}$ as above. The given argument shows that
\begin{gather*}
d_\FF a_s=0\;,\quad u_s+\lim_ld_\FF g_{l,s',s}=\iota'_sa_s+\lim_ld_\FF b_{l,s}\;,\\
j'_{s_1,s}a_{s_1}=a_s+\lim_ld_\FF F'_mg_{l,s'_1,s'}\;,\quad j'_{s_1,s}b_{l,s_1}=b_{l,s}\;.
\end{gather*}
Hence $\hat a:=(\overline{[a_s]}_s)_s\in\widehat H^\bullet J'(\FF)$ is defined and $\hat\iota'_*(\hat a)=\hat u$.
\end{rem}

\subsection{Injectivity of $\bar\iota'_*$}\label{ss: ker bar iota'_* = 0}

Let $\overline{[u]}\in\bar H^rJ'(\FF)$ such that $\bar\iota'_*(\overline{[u]})=0$. This means that there is a net $v_l$ in $I'(\FF)$ such that $\iota'u=\lim_ld_\FF v_l$ in $I'(\FF)$. Write $u_s=j'_su\in ZK^{\prime\,(s)}(\FF)$ and $v_{l,s}=j'_sv_l\in I^{\prime\,(s)}(\FF)$. With the notation of \Cref{ss: ker bar R'_* = im bar iota'_*}, let $b_{l,s}=F'_mv_{l,s'}\in J^{\prime\,(s)}(\FF)$. By~\eqref{F'_m d_FF = d_FF F'_m} and~\eqref{F'_m iota'_s = j'_s s'},
\[
u_s=j'_{s',s}u_{s'}=F'_m\iota'_{s'}u_{s'}=\lim_lF'_md_\FF v_{l,s'}=\lim_ld_\FF b_{l,s}\;.
\]

Like in \Cref{ss: ker bar R'_* = im bar iota'_*}, it can be shown that, taking $s\uparrow+\infty$, $m\uparrow+\infty$ and $s'\uparrow+\infty$ as above, the elements $b_{l,s}\in J^{\prime\,(s)}(\FF)$ define elements $b_l:=(b_{l,s})_s\in J'(\FF)$, and we have $u=\lim_ld_\FF b_l$. Thus $\overline{[u]}=0$ in $\bar H^rJ'(\FF)$.

\begin{rem}\label{r: ker iota'_* = 0}
Like in \Cref{r: ker R'_* = im iota'_*}, we also get the injectivity of $\iota'_*$.
\end{rem}

\begin{rem}\label{r: ker hat iota'_* = 0}
To prove the injectivity of $\hat\iota'_*$, take any $\hat u:=(\overline{[u_s]}_s)_s\in\ker\hat\iota'_*$ in~\eqref{exact seq in the proj limit of reduced cohoms - dual-conormal}. Then there is some net $v_{l,s}$ in every $I^{\prime\,(s)}(\FF)$ such that $\iota'_su_s=\lim_ld_\FF v_{l,s}$ in $I^{\prime\,(s)}(\FF)$. Moreover, $j'_{s',s}u_{s'}=u_s+\lim_ld_\FF g_{l,s,s'}$ for some net $g_{l,s,s'}$ in $J^{\prime\,(s)}(\FF)$. Take $b_{l,s}$ as before. The above argument shows that 
\[
u_s+\lim_ld_\FF g_{l,s',s}=\lim_ld_\FF b_{l,s}\;.
\]
So $\hat u=0$ in $\widehat H^\bullet J'(\FF)$.
\end{rem}

\subsection{Surjectivity of $\bar R'_*$}\label{ss: surjectivity of bar R'_*}

Take any $[u]\in H^\bullet K(\FF)$, and write $u_s=j'_su\in ZK^{(s)}(\FF)$. With the notation of \Cref{ss: E'_m}, we have $v_s:=E'_mv_{s'}\in ZI^{(s')}(\FF)$ by~\eqref{E'_m d_FF = d_FF E'_m}, and $R'_sv_s=j'_{s',s}u_{s'}=u_s$ by~\eqref{R'_s E'_m = j'_s' s}.

Now consider the above notation for greater real numbers $s_1$, $m_1$ and $s'_1$, satisfying the same properties as $s$, $m$ and $s'$. By~\eqref{E'_m j'_s'_1 s' = j'_s_1 s E'_m_1},
\[
j'_{s_1,s}v_{s_1}=j'_{s_1,s}E'_{m_1}v_{s'_1}=E'_mj'_{s'_1,s'}v_{s'_1}=E'_mv_{s'}=v_s\;.
\]
So $v:=(v_s)_s\in ZI'(\FF)$ satisfies $R'v=u$, and therefore $\bar R'_*(\overline{[v]})=\overline{[u]}$.

\begin{rem}\label{r: surjectivity of R'_*}
Using cohomology instead of reduced cohomology, the analogous argument gives the surjectivity of $R'_*$.
\end{rem}

\begin{rem}\label{r: surjectivity of widehat R'_*}
To prove the surjectivity of $\widehat R'_*$ in~\eqref{exact seq in the proj limit of reduced cohoms - dual-conormal}, for any $[u]\in H^\bullet K(\FF)$, define $u_s$ and $v_s$ as above. We also have $R'_sv_s=u_s$ and $j'_{s_1,s}v_{s_1}=v_s$. Thus $\hat v:=(\overline{[v_s]}_s)_s\in\widehat H^\bullet I'(\FF)$ and $\widehat R'_*\hat v=[u]$.
\end{rem}

\section{Functoriality and leafwise homotopy invariance}\label{s: functor and leafwise homotopy inv - dual-conormal}

\subsection{Pull-back of dual-conormal leafwise currents}\label{ss: pull-back of dual-conormal leafwise currents}

Consider the notation and conditions of \Cref{ss: push-forward of conormal leafwise currents} (including the conditions of \Cref{ss: pull-back of leafwise conormal currents}). According to \Cref{ss: pull-back of dual-conormal distributions}, the map~\eqref{phi^* on leafwise forms} has a continuous extension
\begin{equation}\label{phi^*: I'(FF) to I'(FF')}
\phi^*:I'(\FF)\to I'(\FF')\;,
\end{equation}
defined as the composition
\[
I'(\FF) \xrightarrow{\phi^*} I'(M',M^{\prime\,0};\phi^*\Lambda\FF)
\xrightarrow{\phi^*} I'(\FF')\;,
\]
like~\eqref{pull-back - decomposition - conormal - leafwise}, using~\eqref{phi_*: I'_c/cv(M' L' phi^*E otimes Omega_fiber) -> I'_c/.(M L E)} with $E=\Lambda\FF$. We can also describe~\eqref{phi^*: I'(FF) to I'(FF')} as the restriction of~\eqref{phi_*: I_c/cv(M' L' Lambda) to I_c/.(M L Lambda)} to dual-conormal currents of bidegree $(0,\bullet)$, like in~\eqref{phi_*: C^pm infty_c/cv(M' Lambda FF') to C^pm infty_c/.(M Lambda FF)}. The map~\eqref{phi^*: I'(FF) to I'(FF')} is also a restriction of~\eqref{phi_*: C^pm infty_c/cv(M' Lambda FF') to C^pm infty_c/.(M Lambda FF)}. 

~\eqref{phi_*: I'_c/cv(M' L' phi^*E otimes Omega_fiber) -> I'_c/.(M L E)} with $E=\Lambda\FF$

Similarly, the analogs of~\eqref{phi^*: I(M L E) -> I(M' L' phi^*E)} with $E=\Lambda\FF$ for~\eqref{phi^*: K'(M L) to K'(M' L')} and~\eqref{phi^*: J'(M L) to J'(M' L')} induce continuous homomorphisms
\begin{gather}
\phi^*:K'(\FF)\to K'(\FF')\;,\label{phi^*: K'(FF) to K'(FF')}\\
\phi^*:J'(\FF)\to J'(\FF')\;.\label{phi^*: J'(FF) to J'(FF')}
\end{gather}
By passing to cohomology and reduced cohomology, we get continuous homomorphisms,
\begin{equation}\label{phi^*: HK'(FF) to HK'(FF')}
\left\{
\begin{gathered}
\phi^*:H^\bullet K'(\FF)\to H^\bullet K'(\FF')\;,\\
\begin{alignedat}{2}
\phi^*:&H^\bullet I'(\FF)\to H^\bullet I'(\FF')\;,&\quad\phi^*:&\bar H^\bullet I'(\FF)\to\bar H^\bullet I'(\FF')\;,\\
\phi^*:&H^\bullet J'(\FF)\to H^\bullet J'(\FF')\;,&\quad\phi^*:&\bar H^\bullet J'(\FF)\to\bar H^\bullet J'(\FF')\;.
\end{alignedat}
\end{gathered}
\right.
\end{equation}
The assignment of the homomorphisms~\eqref{phi^*: I'(FF) to I'(FF')}--\eqref{phi^*: HK'(FF) to HK'(FF')} is functorial.

\subsection{Description of $\phi^*:K'(\FF)\to K'(\FF')$}\label{ss: description of phi^*:K'(FF) to K'(FF')}

Consider the notation and conditions of \Cref{ss: description of phi^*:K(FF) to K(FF')}, and assume also that $\phi$ is a submersion. By the density of the space of smooth forms in the space of currents, we get from~\eqref{phi^*: C^infty(M^0 Lambda) to C^infty(M^prime 0 Lambda)} that
\begin{equation}\label{phi^*: C^-infty(M^0 Lambda) to C^-infty(M^prime 0 Lambda)}
\phi^*:C^{-\infty}(M^0;\Lambda)\to C^{-\infty}(M^{\prime\,0};\Lambda)
\end{equation}
is a cochain map for $d_{s\eta}$ and $d_{s\eta'}$ ($s\in\R$), and we get from~\eqref{phi^*: C^infty(M^0 Lambda otimes Omega^sNM^0) to C^infty(M^prime 0 Lambda otimes Omega^sNM^prime 0)} that
\begin{equation}
\label{phi^*: C^-infty(M^0 Lambda otimes Omega^sNM^0) to C^-infty(M^prime 0 Lambda otimes Omega^sNM^prime 0)}
\phi^*:C^{-\infty}(M^0;\Lambda\otimes\Omega^sNM^0)\to C^{-\infty}(M^{\prime\,0};\Lambda\otimes\Omega^sNM^{\prime\,0})
\end{equation} 
is another cochain map for the de~Rham differentials defined with the flat bundle structures of $\Omega^sNM^0$ and $\Omega^sNM^{\prime\,0}$.

\begin{prop}\label{p: phi^* equiv prod_k phi^*}
According to \Cref{c: K'(Lambda FF) equiv prod_k C^-infty(M^0 Lambda)}, the map~\eqref{phi^*: K'(FF) to K'(FF')} is given by
\[
\phi^*\equiv\prod_k\phi^*\equiv\prod_k\phi^*\;,
\]
where the terms of the first direct sum are given by~\eqref{phi^*: C^-infty(M^0 Lambda) to C^-infty(M^prime 0 Lambda)}, and the terms of the second direct sum are given by~\eqref{phi^*: C^-infty(M^0 Lambda otimes Omega^sNM^0) to C^-infty(M^prime 0 Lambda otimes Omega^sNM^prime 0)}, taking $s=k$.
\end{prop}

\begin{proof}
Apply \Cref{p: phi_* is transpose of phi^* - forms - currents,p: phi_* is transpose of phi^* - I-currents - I'-currents,p: phi_* equiv bigoplus_k phi_*}.  
\end{proof}

\subsection{Push-forward of dual-conormal leafwise currents}\label{ss: push-forward of dual-conormal leafwise currents}

Consider the notation and conditions of \Cref{ss: push-forward of conormal leafwise currents} (containing those of \Cref{ss: pull-back of leafwise conormal currents}). Then the case of~\eqref{phi_*: C^pm infty_c/cv(M' Lambda FF') to C^pm infty_c/.(M Lambda FF)} on smooth leafwise forms has a continuous extension
\begin{equation}\label{phi_*: I'_c/cv(FF') to I'_c/cv(FF)}
\phi_*:I'_{\co/\cv}(\FF')\to I'_{\co/{\cdot}}(\FF)\;.
\end{equation}
This map can be described as the restriction of the map~\eqref{phi_*: I'_c/cv(M' L' Lambda) to I'_c/.(M L Lambda)} to dual-conormal currents of bidegree $(0,\bullet)$, like~\eqref{phi_*: C^pm infty_c/cv(M' Lambda FF') to C^pm infty_c/.(M Lambda FF)} in \Cref{ss: push-forward of leafwise currents}. We can also describe~\eqref{phi_*: I'_c/cv(FF') to I'_c/cv(FF)} as the composition
\[
I'_{\co/\cv}(M',L';\Lambda\FF)\xrightarrow{\pi_\topd}I'_{\co/\cv}(M',L';\phi^*\Lambda\FF\otimes\Omega_\fiber)
\xrightarrow{\phi_*}I'_{\co/{\cdot}}(M,L;\Lambda\FF)\;,
\]
like in~\eqref{push-forward - composition - leafwise currents}, where $\phi_*$ is given by~\eqref{phi_*: I'_c/cv(M' L' Omega_fiber) -> I'_c/.(M L)} with $E=\Lambda\FF$. The map~\eqref{phi_*: I'_c/cv(FF') to I'_c/cv(FF)} is also a restriction of the case of~\eqref{phi_*: C^pm infty_c/cv(M' Lambda FF') to C^pm infty_c/.(M Lambda FF)} for leafwise currents.

According to \Cref{ss: push-forward of the dual-conormal seq}, the map~\eqref{phi_*: I'_c/cv(FF') to I'_c/cv(FF)} induces homomorphisms
\begin{gather}
\phi_*:K'(\FF')\to K'(\FF)\;,\label{phi_*:K'(FF') to K'(FF)}\\
\phi_*:J'_{\co/\cv}(\FF')\to J'_{\co/{\cdot}}(\FF)\;.\label{phi_*: J'_c/cv(FF') to J'_c/.(FF)}
\end{gather}
Like in \Cref{ss: pull-back of dual-conormal leafwise currents}, we get induced continuous homomorphisms,
\begin{equation}\label{HK'(FF') to HK'(FF) ...}
\left\{
\begin{gathered}
\phi_*:H^\bullet K'(\FF')\to H^\bullet K'(\FF)\;,\\
\begin{alignedat}{2}
\phi_*:&H^\bullet I'_\co(\FF')\to H^\bullet I'_\co(\FF)\;,&\quad
\phi_*:&\bar H^\bullet I'_\co(\FF')\to\bar H^\bullet I'_\co(\FF)\;,\\
\phi_*:&H^\bullet J'_\co(\FF')\to H^\bullet J'_\co(\FF)\;,&\quad
\phi_*:&\bar H^\bullet J'_\co(\FF')\to\bar H^\bullet J'_\co(\FF)\;.
\end{alignedat}
\end{gathered}
\right.
\end{equation} 
The assignments of homomorphisms~\eqref{phi_*: I'_c/cv(FF') to I'_c/cv(FF)}--\eqref{HK'(FF') to HK'(FF) ...} are clearly functorial.

\subsection{Leafwise homotopy invariance}\label{ss: leafwise homotopy invariance - dual-conormal}

Consider the notation and conditions of \Cref{ss: leafwise homotopy invariance}, and assume that every $H_t$ is a submersion. Like in \Cref{ss: leafwise homotopy invariance}, according to \Cref{ss: pull-back of dual-conormal leafwise currents,ss: push-forward of dual-conormal leafwise currents}, the corresponding leafwise homotopy operator $\sh:C^\infty(M;\Lambda\FF)\to C^\infty(M';\Lambda\FF')$ has continuous linear extensions,
\[
\sh:K'(\FF)\to K'(\FF')\;,\quad\sh:I'(\FF)\to I'(\FF')\;,\quad\sh:J'(\FF)\to J'(\FF')\;.
\]
By continuity and according to \Cref{ss: leafwise homotopy opers}, we have $H_1^*-H_0^*=\sh d_\FF+d_{\FF'}\sh$ with $H_0^*$ and $H_1^*$ given by~\eqref{phi^*: I'(FF) to I'(FF')},~\eqref{phi^*: K'(FF) to K'(FF')} and~\eqref{phi^*: J'(FF) to J'(FF')}. Hence we get the following.

\begin{prop}\label{p: leafwise homotopy invariance - H^bullet K'(FF)}
 Let $\phi,\psi:(M',\FF')\to(M,\FF)$ be smooth foliated maps transverse to $M^0$ with $\phi^{-1}(M^0)=\psi^{-1}(M^0)=M^{\prime\,0}$. If $\phi$ is leafwise homotopic to $\psi$, then $\phi$ and $\psi$ induce the same homomorphisms~\eqref{phi^*: HK'(FF) to HK'(FF')}.
\end{prop}

\section{Action of foliated flows on the dual-conormal sequence}\label{s: action of foliated flows on the dual-conormal seq}

Consider the notation and conditions of \Cref{s: action of foliated flows on the conormal seq}.

\begin{prop}\label{p: phi^t* equiv prod_L k e^k varkappa_L t}
According to \Cref{c: H^bullet K'(FF) equiv prod_k H^bullet_k(M^0),r: prod_L},
\[
\phi^{t*}\equiv\prod_{k,L}e^{k\varkappa_Lt}\equiv\prod_{k,L}e^{k\varkappa_Lt}
\]
on $H^\bullet K(\FF)$, where $k$ runs in $\N_0$ and $L$ in $\pi_0M^0$.
\end{prop}

\begin{proof}
Argue like in the proof of \Cref{p: phi^t* equiv bigoplus_L k e^-(k+1) varkappa_L t} and its previous observations, using \Cref{c: K'(Lambda FF) equiv prod_k C^-infty(M^0 Lambda)}, Remark~\ref{r: prod_L} and \Cref{p: leafwise homotopy invariance - H^bullet K'(FF)}.
\end{proof}

\Cref{c: K'(Lambda FF) equiv prod_k C^-infty(M^0 Lambda),c: H^bullet K'(FF) equiv prod_k H^bullet_k(M^0),r: prod_L,p: phi^t* equiv prod_L k e^k varkappa_L t} show \Cref{t: intro - H^bullet K'(FF)}.

\chapter{Contribution from $M^1$}\label{ch: contribution from M^1}

\section{Operators on a suspension foliation}\label{s: opers suspension}

Consider again the notation of \Cref{s: suspension}, where the case of a weakly simple foliated flow $\phi=\{\phi^t\}$ on a suspension foliated manifold $(M,\FF)$ was described. Equip $\mathring M_\pm$ with $g_\pm$, obtaining that $\mathring\FF_\pm$ is of bounded geometry (\Cref{p: FF_pm is of bd geom}). We can assume $\phi$ is of $\R$-local bounded geometry on $\mathring M_\pm$ by \Cref{p: Z -> A} and according to \Cref{ss: families of bd geom}. Thus, on $\mathring{\widetilde M}_\pm\equiv(\mathring{\widetilde M}_\pm,\tilde g_\pm$), $\mathring{\widetilde\FF}_\pm$ is of bounded geometry and $\tilde\phi$ is of $\R$-local bounded geometry. Consider the leafwise perturbed operators for $(\mathring M_\pm,\mathring\FF_\pm)$ and $(\mathring{\widetilde M}_\pm,\mathring{\widetilde\FF}_\pm)$ defined by the leafwise-closed form $\eta_0$ and the leafwise-exact form $\tilde\eta_0$ (\Cref{ss: perturbation vs bigrading} ). For any $\psi\in\AA$, $f\in\Cinftyc(\R)$ and $z\in\C$, the operator \index{$\mathring P_\pm$}
\begin{equation}\label{mathring P_pm suspension}
\mathring P_\pm=\int_{-\infty}^{+\infty}\phi^{t*}_z\,\psi(D_{\mathring\FF_\pm,z})\,f(t)\,dt
\end{equation}
on $H^{-\infty}(\mathring M_\pm;\Lambda\mathring\FF_\pm)$ is a version of~\eqref{P with f} for $\phi^{t*}_z$ and $D_{\mathring\FF_\pm,z}$, and therefore it is smoothing by the corresponding analog of~\eqref{P_psi}. Let $\mathring K_\pm=K_{\mathring P_\pm}$. \index{$\mathring K_\pm$}

By~\eqref{tilde phi^t(tilde y x) widetilde Z} and~\eqref{T_gamma tilde phi_x^t = tilde phi_a_gamma x^t T_gamma}, for $\gamma\in\Gamma$ and $t\in\R$, the equality $\tilde\phi^{t*}_zT_\gamma^*=T_\gamma^*\tilde\phi^{t*}_z$ means that, for all $x\in\R$,
\begin{equation}\label{tilde phi^t*_x z T_gamma^* = T_gamma^* tilde phi^t*_a_gamma x z}
\tilde\phi^{t*}_{x,z}T_\gamma^*=T_\gamma^*\tilde\phi^{t*}_{a_\gamma x,z}
\end{equation}
on $C^\infty(\widetilde L,\Lambda)$ (\Cref{ss: perturbation of pull-back homs}). Consider also the notation of \Cref{ss: perturbed Schwartz kernels on regular coverings of compact mfds} for the regular covering $\pi=\pi_L:\widetilde L\to L$ used in the suspension construction; in particular, recall the notation $\tilde k_z$. Recall that $h_\pm(\gamma)=\varkappa^{-1}\ln a_\gamma$ for $\gamma\in\Gamma$, and $D_\pm(x,\tilde y)=\varkappa^{-1}\ln|x|$ for $(x,\tilde y)\in\mathring{\widetilde M}_\pm$ (\Cref{ss: suspension - homotheties}). Thus, by the version of \eqref{Schwartz kernel} for the leafwise perturbed differential complex (\Cref{s: Witten's opers on Riem folns of bd geom}), and by~\eqref{tilde phi^t(tilde y x) widetilde Z},~\eqref{T_gamma tilde phi_x^t = tilde phi_a_gamma x^t T_gamma} and~\eqref{tilde phi^t*_x z T_gamma^* = T_gamma^* tilde phi^t*_a_gamma x z}, if $\hat\psi\in\Cinftyc(\R)$, then
\begin{equation}\label{mathring K_pm([tilde y x] [tilde y' x'])}
\mathring K_\pm([x,\tilde y],[x',\tilde y'])\equiv\sum_{\gamma\in \Gamma}\mathring K_{\pm,\gamma}([x,\tilde y],[x',\tilde y'])
\end{equation}
for all $(x,\tilde y),(x',\tilde y')\in\mathring{\widetilde M}_\pm$, where
\begin{multline}\label{mathring K_pm gamma([x tilde y] [x' tilde y'])}
\mathring K_{\pm,\gamma}([x,\tilde y],[x',\tilde y'])\\
=\frac{1}{|\varkappa|}
\tilde\phi_{x,z}^{\frac{1}{\varkappa}\ln\frac{x'}{a_\gamma x}*}T_\gamma^*
\tilde k_z\Big(\gamma\cdot\tilde\phi^{\frac{1}{\varkappa}\ln\frac{x'}{a_\gamma x}}_x(\tilde y),\tilde y'\Big)
f\Big(\frac{1}{\varkappa}\ln\frac{x'}{a_\gamma x}\Big)\Big|\frac{dx'}{x'}\Big|\;.
\end{multline}

According to \Cref{ss: b-stretched product}, for the boundary-defining function $\rho$ on $M_\pm$ (\Cref{ss: boundary-def func of M_pm - suspension}), let $\rho$ and $\rho'$ denote its lifts to $(M_\pm)^2$ from the left and right factors, and let $s=\rho/\rho':(M_\pm)^2\to[0,\infty]$. We have corresponding smooth functions $\rho$, $\rho'$ and $s$ on $(M_\pm)^2_{\text{\rm b}}$. Similarly, let $\eta$ and $\eta'$ denote the lifts of $\eta$ from the left and right factors. Using~\eqref{M_pm equiv [0 infty)_rho times L_varpi}, we get
\[
(M_\pm)^2\equiv[0,\infty)_\rho\times[0,\infty)_{\rho'}\times L^2\;,\quad
(\mathring M_\pm)^2\equiv(0,\infty)^2\times L^2\;.
\]
Then
\[
(M_\pm)^2_{\text{\rm b}}\equiv[0,\infty)_\rho\times[0,\infty]_s\times L^2\;,
\]
with boundary components $\lb=\{s=0\}$, $\rb=\{s=\infty\}$ and $\ff=\{\rho=0\}$. Moreover
\[
\Deltab\equiv\{\,(\rho,1,y,y)\mid\rho\ge0,\ y\in L\,\}\;.
\]
With the above identities, the restriction of $\beta_{\text{\rm b}}: (M_\pm)^2_{\text{\rm b}}\to (M_\pm)^2$ to the interior corresponds to the diffeomorphism
\[
(0,\infty)^2\times L^2\to(0,\infty)^2\times L^2\;,\quad(\rho, s,y,y')\mapsto(\rho,\rho s^{-1}, y, y')\;.
\]
Similar observations apply to $\widetilde M_\pm$, using $\widetilde L$ instead of $L$, and using the lifts $\tilde\rho$, $\tilde\rho'$ and $\tilde s$ instead of $\rho$, $\rho'$ and $s$. The subscript ``$\pm$'' will be added to the notation $\Deltab$ and $\Delta_{\text{\rm b},0}=\Deltab\cap\ff$ if needed.

Let $\mathring\kappa_\pm$ \index{$\mathring\kappa_\pm$} be the $C^\infty$ section of $\beta_{\text{\rm b}}^*(\Lambda\FF_\pm \boxtimes (\Lambda\FF^*_\pm\otimes \Omega M_\pm))$ on the interior of $(M_\pm)^2_{\text{\rm b}}$ that corresponds to $\mathring K_\pm$ via $\beta_{\text{\rm b}}$. If $\hat\psi\in\Cinftyc(\R)$, then, using the changes of variables
\[
x=\pm e^{-F(\tilde y)}\tilde\rho\;,\quad x'=\pm e^{-F(\tilde y')}\tilde\rho'\;,
\] 
with
\begin{gather*}
\ln\frac{x'}{x}=F(\tilde y)-F(\tilde y')-\ln\tilde s\;,\quad
\frac{dx'}{x'}=-\tilde\eta'+\frac{d\tilde\rho'}{\tilde\rho'}\;,\quad
\frac{d\tilde s}{\tilde s}=-\frac{d\tilde\rho'}{\tilde\rho'}\;,\\
x'=0\Leftrightarrow\tilde\rho'=0\Leftrightarrow\tilde s=\infty\;,\quad 
x'=\pm\infty\Leftrightarrow\tilde\rho'=\infty\Leftrightarrow\tilde s=0\;,
\end{gather*}
it follows from~\eqref{mathring K_pm([tilde y x] [tilde y' x'])} and~\eqref{mathring K_pm gamma([x tilde y] [x' tilde y'])} that
\begin{align}
\mathring K_\pm(\rho,\rho',[\tilde y],[\tilde y'])&=\sum_{\gamma\in\Gamma}\mathring K_{\pm,\gamma}(\rho,\rho',[\tilde y],[\tilde y'])\;,
\label{mathring K_pm(rho rho' [tilde y] [tilde y'])}\\
\mathring\kappa_\pm(\rho,s,[\tilde y],[\tilde y'])
&=\sum_{\gamma\in\Gamma}\mathring\kappa_{\pm,\gamma}(\rho,s,[\tilde y],[\tilde y'])\;,
\label{mathring kappa_pm(rho s [tilde y] [tilde y'])}
\end{align}
where \index{$\mathring K_{\pm,\gamma}$}
\begin{multline}\label{mathring K_pm gamma(rho rho' [tilde y] [tilde y'])}
\mathring K_{\pm,\gamma}(\rho,\rho',[\tilde y],[\tilde y'])\\
\begin{aligned}[b]
&=\frac{1}{|\varkappa|}
\tilde\phi_{\pm e^{-F(\tilde y)}\rho,z}^{\frac{1}{\varkappa}(F(\tilde y)-F(\tilde y')
+\ln \frac{\rho'}{a_\gamma\rho})*}T_\gamma^*\,\tilde k_z\Big(\gamma\cdot 
\tilde\phi^{\frac{1}{\varkappa}(F(\tilde y)-F(\tilde y')+\ln \frac{\rho'}{a_\gamma\rho})}_{\pm e^{-F(\tilde y)}\rho}(\tilde y),\tilde y'\Big)\\
&\phantom{=\text{}}\text{}\times f\Big(\frac{1}{\varkappa}\Big(F(\tilde y)-F(\tilde y')
+\ln\frac{\rho'}{a_\gamma\rho}\Big)\Big)\Big|\frac{d\rho'}{\rho'}\Big|\;,
\end{aligned}
\end{multline}
and \index{$\mathring\kappa_{\pm,\gamma}$}
\begin{multline}\label{mathring kappa_pm gamma(rho s [tilde y] [tilde y'])}
\mathring\kappa_{\pm,\gamma}(\rho,s,[\tilde y],[\tilde y'])\\
\begin{aligned}[b]
&=\frac{1}{|\varkappa|}\tilde\phi_{\pm e^{-F(\tilde y)}\rho,z}^{\frac{1}{\varkappa}
(F(\tilde y)-F(\tilde y')-\ln a_\gamma s)*}\,T_\gamma^*\,
\tilde k_z\Big(\gamma\cdot \tilde\phi^{\frac{1}{\varkappa}(F(\tilde y)-F(\tilde y')
-\ln a_\gamma s)}_{\pm e^{-F(\tilde y)}\rho}(\tilde y),\tilde y'\Big)\\
&\phantom{:=\text{}}\text{}\times f\Big(\frac{1}{\varkappa}\big(F(\tilde y)-F(\tilde y')-\ln a_\gamma s\big)\Big)
\Big|\frac{ds}{s}\Big|\;.
\end{aligned}
\end{multline}
Let us look for more general conditions on $\psi$ to get~\eqref{mathring kappa_pm(rho s [tilde y] [tilde y'])} by using the Fr\'echet algebra and $\C[z]$-module $\AA$ (\Cref{ss: Witten - mfds of bd geom}). Notice that every $\mathring\kappa_{\pm,\gamma}(\rho,s,[\tilde y],[\tilde y'])$ has a $C^\infty$ extension to $\rho=0$.

\begin{lem}\label{l: kappa_mathring P_pm}
If $\psi\in \AA$, then, given any fundamental domain $\bfF\subset\widetilde L$, the series in~\eqref{mathring kappa_pm(rho s [tilde y] [tilde y'])} converges with all covariant derivatives, uniformly on $\rho\ge0$, $0<s<\infty$ and $\tilde y,\tilde y'\in \bfF$. Moreover its sum is $\mathring\kappa_\pm(\rho,s,[\tilde y],[\tilde y'])$ for $\rho>0$.
\end{lem}

\begin{proof}
Since $\tilde\phi$ is of $\R$-local bounded geometry on $\mathring{\widetilde M}_\pm$ with $\tilde g_\pm$ and $\supp f$ is compact, we can take $R>0$ and $\bfK\subset{\widetilde L}^2$ like in the proof of \Cref{p: d_FF(phi^t(p) p) ge c_1^{-1} |gamma| - c_2} with $\supp f\subset I$ for any compact $I\subset\R$. Using~\eqref{|... tilde k_z(gamma cdot tilde y tilde y')|} with this $\bfK$, for any $W>0$, we get
\[
\big|\tilde k_z(\gamma\cdot\tilde\phi^t_x(\tilde y),\tilde y')\big| \le C'_1 e^{-\frac{W}{c_1}\,|\gamma|}\,\|\psi\|_{\AA, W,N}
\]
for $\gamma\in\Gamma$, $x\in\R^\pm$, $t \in \supp f$ and $\tilde y,\tilde y' \in \bfF$. Using again the $\R$-local bounded geometry of $\tilde\phi$ on $\mathring{\widetilde M}_\pm$ with $\tilde g_\pm$ and compactness of $I$, it follows that there is some $C_2=C_2(z,W)>0$ such that
\begin{equation}\label{|K_gamma(rho s [tilde y] [tilde y'])|}
\big|\mathring\kappa_{\pm,\gamma}(\rho,s,[\tilde y],[\tilde y'])\big|
\le C_2 e^{-\frac{W}{c_1}\,|\gamma|}\,\|\psi\|_{\AA,W,N}\,\|f\|_{I,C^0}
\end{equation}
for $\gamma\in\Gamma$, $\rho\ge0$, $s>0$ and $\tilde y,\tilde y'\in \bfF$. By~\eqref{sum_gamma in Gamma e^-W_0 |gamma| < infty} and~\eqref{|K_gamma(rho s [tilde y] [tilde y'])|}, if $W>c_1W_0$, then the series in~\eqref{mathring kappa_pm(rho s [tilde y] [tilde y'])} converges uniformly on $\rho\ge0$, $s>0$ and $\tilde y,\tilde y'\in \bfF$, and the norm of its sum is${}\le C \|\psi\|_{\AA, W,N}$ for some $C=C(z,W,N)>0$. 

With more generality, by the $\R$-local bounded geometry of $\tilde\phi$ on $\mathring{\widetilde M}_\pm$ and the compactness of $I$, the higher order derivatives of $\tilde\phi^t_x(\tilde y)$ with respect to $x$, $t$ and $\tilde y$ (in normal coordinates) are also uniformly bounded for $x\in\R^\pm$, $t\in I$ and $\tilde y\in\widetilde L$. Hence, for every $m\in\N_0$, it follows from~\eqref{|... tilde k_z(gamma cdot tilde y tilde y')|} that
\[
\Big|\nabla_{\tilde y}^{m_1}\nabla_{\tilde y'}^{m_2}
\tilde k_z(\gamma\cdot\tilde\phi^t_x(\tilde y),\tilde y')\Big|
\le C'_1 e^{-\frac{W}{c_1}\,|\gamma|}\,\|\psi\|_{\AA, W, N+m}
\]
for $\gamma\in\Gamma$, $x\in\R^\pm$, $t \in I$, $\tilde y,\tilde y' \in \bfF$ and $m_1+m_2\le m$. Moreover, since $I$ and $\bfF$ are compact, there is some $c_3\in\R$ such that, for all $\tilde y,\tilde y'\in\bfF$,
\[
\ln a_\gamma s>c_3\Rightarrow\varkappa^{-1}\big(F(\tilde y)-F(\tilde y')-\ln a_\gamma s\big)\notin I\;.
\]
Thus we can assume $s^{-1}<e^{-c_3}a_\gamma$, yielding $s^{-1}<e^{c_0|\gamma|-c_3}$ by~\eqref{|ln a_gamma| le c_0 |gamma|}. Hence there is some $C_3=C_3(z,W,m)>0$ such that
\begin{multline}\label{|nabla^m K_gamma(rho s [tilde y] [tilde y'])|}
\big|\partial^{m_1}_\rho\partial^{m_2}_s\nabla^{m_3}_{[\tilde y]}\nabla^{m_4}_{[\tilde y']}
\mathring\kappa_{\pm,\gamma}(\rho,s,[\tilde y],[\tilde y'])\big|\\
\le C_3 e^{(mc_0-\frac{W}{c_1})|\gamma|}\,\|\psi\|_{\AA,W,N+m}\,\|f\|_{I,C^m}\;,
\end{multline}
for $\gamma\in\Gamma$, $\rho\ge0$, $s>0$, $\tilde y,\tilde y'\in \bfF$ and $m_1+\dots+m_4\le m$. By~\eqref{sum_gamma in Gamma e^-W_0 |gamma| < infty} and~\eqref{|nabla^m K_gamma(rho s [tilde y] [tilde y'])|}, if $W>c_1(mc_0+W_0)$, then the series defined by the covariant derivatives of order${}\le m$ of the terms in~\eqref{mathring kappa_pm(rho s [tilde y] [tilde y'])} is also convergent, uniformly on $\rho\ge0$, $s>0$ and $\tilde y,\tilde y'\in \bfF$, and the norm of its sum is${}\le C' \|\psi\|_{\AA, W,N+m}\,\|f\|_{I,C^m}$ for some $C'=C'(z,W,N,m)>0$.

We already know that the sum of the series in~\eqref{mathring kappa_pm(rho s [tilde y] [tilde y'])} is $\mathring\kappa_\pm(\rho,s,[\tilde y],[\tilde y'])$ for $\rho>0$ if $\hat\psi\in\Cinftyc(\R)$. Then this also holds when $\psi\in\AA$, as follows by taking a convergent sequence $\psi_k\to\psi$ in $\AA$ with $\widehat{\psi_k}\in\Cinftyc(\R)$, and using the above estimates of the sum.
\end{proof}

\begin{rem}
Like in Remark~\ref{r: k(y y') equiv ...}, \Cref{l: kappa_mathring P_pm} is true for any $\psi\in\SS$ since $\Gamma$ is abelian. But $\psi\in\AA$ is needed for the estimates~\eqref{|K_gamma(rho s [tilde y] [tilde y'])|} and~\eqref{|nabla^m K_gamma(rho s [tilde y] [tilde y'])|}, which will be used later. 
\end{rem}

\begin{prop}\label{p: kappa_P_pm}
If $\psi\in\AA$, then $\mathring\kappa_\pm$ has a $C^\infty$ extension $\kappa_\pm$ to $(M_\pm)^2_{\text{\rm b}}$, also given by~\eqref{mathring kappa_pm(rho s [tilde y] [tilde y'])} and~\eqref{mathring kappa_pm gamma(rho s [tilde y] [tilde y'])} using $C^\infty$ extensions $\kappa_{\pm,\gamma}$ of the sections $\mathring\kappa_{\pm,\gamma}$ to $(M_\pm)^2_{\text{\rm b}}$, which vanishes to all orders at $\lb\cup\rb$. Therefore $\kappa_\pm=\kappa_{P_\pm}$ for some $P_\pm\in\Psib^{-\infty}(M_\pm;\Lambda\FF_\pm)$ induced by $\mathring P_\pm$.
\end{prop}

\begin{proof}
By \Cref{l: kappa_mathring P_pm}, $\mathring\kappa_\pm$ extends smoothly to $\mathring\ff$ ($\rho=0$ and $0<s<\infty$).

Take any compact $I\subset\R$ containing $\supp f$. According to~\eqref{mathring kappa_pm gamma(rho s [tilde y] [tilde y'])}, the sum in~\eqref{mathring kappa_pm(rho s [tilde y] [tilde y'])} can be taken for $\gamma\in \Gamma$ with
\[
\varkappa^{-1}(F(\tilde y)-F(\tilde y')-\ln a_\gamma s)\in I\;.
\]
Then, since $\tilde y,\tilde y'\in\bfF$, there exists $R>0$ such that    
\[
\ln s-R<\ln a_\gamma<\ln s+R\;.
\]
Combining this with~\eqref{|ln a_gamma| le c_0 |gamma|}, we get
\begin{equation}\label{... < |gamma|}
c_0^{-1}(\pm\ln s-R)<|\gamma|\;.  
\end{equation}
By~\eqref{sum_gamma in Gamma e^-W_0 |gamma| < infty},~\eqref{|K_gamma(rho s [tilde y] [tilde y'])|} and~\eqref{... < |gamma|}, for any $W>c_1W_0$, there is some $C'_2=C'_2(z,W)>0$ such that, for $\rho\ge0$, $s>0$ and $y,y' \in L$,
\begin{align}
\big|\mathring\kappa_\pm(\rho,s,y,y')\big|  
& < C_2\sum_{|\gamma|>c_0^{-1}(\pm\ln s-R)} e^{-c_1^{-1}W\,|\gamma|}\,\|\psi\|_{\AA,W,N}\,\|f\|_{I,C^0}\notag\\ 
& < C_2  e^{-(c_1^{-1}W-W_0) c_0^{-1}(\pm\ln s-R)}
\sum_{\gamma\in \Gamma} e^{-W_0 |\gamma|}\,\|\psi\|_{\AA,W,N}\,\|f\|_{I,C^0}\notag\\
& < C'_2s^{\mp(c_1^{-1}W-W_0)c_0^{-1}}\,\|\psi\|_{\AA,W,N}\,\|f\|_{I,C^0}\;.
\label{|mathring kappa_pm(rho s y y')| < ... |psi|_AA W N |f|_I C^0}
\end{align}
Using~\eqref{|nabla^m K_gamma(rho s [tilde y] [tilde y'])|} and~\eqref{... < |gamma|}, we similarly get that, for $m\in\N_0$, if $W>c_1(mc_0+W_0)$, then there is some $C'_3=C'_3(z,W,m)>0$ such that
\begin{multline}\label{|partial ... kappa_pm(rho s y y')| <  ... |psi|_AA W N+m |f|_I C^m}
\big|\partial^{m_1}_\rho\partial^{m_2}_s\nabla^{m_3}_y\nabla^{m_4}_{y'}\mathring\kappa_\pm(\rho,s,y,y')\big|\\
<C'_3s^{\mp(c_1^{-1}W-mc_0-W_0)c_0^{-1}}\,\|\psi\|_{\AA,W,N+m}\,\|f\|_{I,C^m}
\end{multline}
for $\rho\ge0$, $s>0$, $y,y' \in L$ and $m_1+m_2+m_3+m_4\le m$. Since $W$ is arbitrarily large, it follows that $\mathring\kappa_\pm$ also extends smoothly to $\lb\cup\rb$ ($s=0,\infty$), where it vanishes to all orders.
\end{proof}

\begin{notation}\label{n: subscripts psi f z}
The subscripts ``$\psi$'', ``$f$'' or ``$z$'' may be added to the notation $\mathring P_\pm$, $\mathring K_\pm$, $\mathring K_{\pm,\gamma}$, $\mathring\kappa_\pm$, $\kappa_\pm$ and $P_\pm$ if needed.
\end{notation}

\begin{prop}\label{p: (psi f) mapsto kappa_pm psi f is cont}
The bilinear map
\[
\AA\times\Cinftyc(\R)\to C^\infty\big((M_\pm)^2_{\text{\rm b}};\beta_{\text{\rm b}}^*(\Lambda\FF_\pm \boxtimes (\Lambda\FF^*_\pm\otimes \Omega M_\pm))\big)\;,\quad(\psi,f)\mapsto\kappa_{\pm,\psi,f}\;,
\]
is continuous.
\end{prop}

\begin{proof}
This is an additional consequence of~\eqref{|mathring kappa_pm(rho s y y')| < ... |psi|_AA W N |f|_I C^0} and~\eqref{|partial ... kappa_pm(rho s y y')| <  ... |psi|_AA W N+m |f|_I C^m}.
\end{proof}

Recall the notation $\phi_L=\{\phi_L^t\}=\{\phi_0^t\}$ on $M^0\equiv L$ and $\tilde \phi_{\widetilde L}=\{\tilde \phi_{\widetilde L}^t\}=\{\tilde \phi_0^t\}$ on $\widetilde M^0\equiv\widetilde L$ (\Cref{ss: transv simple flows - suspension}), and the trivialization $\nu$ of ${}_+N\partial M_\pm$ (\Cref{ss: boundary-def func of M_pm - suspension}). Recall also that the indicial family is defined in \Cref{s: indicial}.

\begin{prop}\label{p: I_nu(P_pm z lambda))}
We have
\[
I_{\nu_\pm}(P_{\pm,z},\lambda)
\equiv\int_{-\infty}^{+\infty}\phi_{L,z+i\lambda}^{t*}\,\psi(D_{L,z+i\lambda})\,e^{i\lambda\varkappa t}f(t)\,dt\;.
\]
\end{prop}

\begin{proof}
By~\eqref{K_I_nu(A lambda)(y y')}, it is enough to show that the Schwartz kernel of the smoothing operator
\[
\int_{-\infty}^{+\infty}\phi_{L,z+i\lambda}^{t*}\,\psi(D_{L,z+i\lambda})\,e^{i\lambda\varkappa t}f(t)\,dt
\]
on $C^{-\infty}(L;\Lambda)$ is given by
\[
\int_0^\infty s^{-i\lambda}\kappa_{\pm,z}(0,s,y,y')\,\frac{ds}{s}\;,
\]
 at every $(y,y')\in L^2$. By \Cref{l: kappa_mathring P_pm,p: kappa_P_pm}, for all $\tilde y,\tilde y'\in\widetilde L$,
\begin{multline*}
\int_0^\infty s^{-i\lambda}\kappa_{\pm,z}(0,s,[\tilde y],[\tilde y'])\,\frac{ds}{s}\\
\begin{aligned}
&= \frac{1}{|\varkappa|}\sum_{\gamma\in \Gamma} \int_0^\infty
s^{-i\lambda} \tilde\phi_{0,z}^{\frac{1}{\varkappa}(F(\tilde y)-F(\tilde y')-\ln a_\gamma s)*}T_\gamma^*\\
&\phantom{=\text{}}\text{}\circ\tilde k_z\Big(\gamma\cdot \tilde\phi^{\frac{1}{\varkappa}(F(\tilde y)
-F(\tilde y')-\ln a_\gamma s)}_{0}(\tilde y),\tilde y'\Big) \\ 
&\phantom{=\text{}}\text{}\times f\Big(\frac{1}{\varkappa}\Big(F(\tilde y)
-F(\tilde y')-\ln a_\gamma s\Big)\Big)\frac{ds}{s}\\
&=\sum_{\gamma\in \Gamma} e^{i\lambda(F(\tilde y')-F(\tilde y)-\ln a_\gamma)} 
\int_{-\infty}^{+\infty}\tilde\phi_{0,z}^{t*}T_\gamma^*\,
\tilde k_z(\gamma\cdot \tilde\phi_0^t(\tilde y),\tilde y')e^{i\lambda\varkappa t}f(t)\,dt\;,
\end{aligned}
\end{multline*}
where we have used the change of variable
\[
t=\varkappa^{-1}(F(\tilde y)-F(\tilde y')-\ln s+\ln a_\gamma)\;,
\]
with
\begin{gather*}
s=e^{F(\tilde y)-F(\tilde y')+\ln a_\gamma-\varkappa t}\;,\quad dt=-\frac{ds}{\varkappa s}\;,\\
s=0\Leftrightarrow t=\sign(\varkappa)\infty\;,
\quad s=\infty\Leftrightarrow t=-\sign(\varkappa)\infty\;.
\end{gather*}
By \Cref{p: k(y y') equiv ...},
\[
k_{z+i\lambda}([\tilde y],[\tilde y']) =\sum_{\gamma\in \Gamma} 
T_\gamma^*\tilde k_{z+i\lambda}(\gamma\cdot\tilde y,\tilde y')\;.
\]
Moreover, by~\eqref{Witten's opers},
\[
\tilde k_{z+i\lambda}(\tilde y,\tilde y')= e^{i\lambda (F(\tilde y')-F(\tilde y))}
\tilde k_z(\tilde y,\tilde y')\;.
\]
So, by~\eqref{T_gamma tilde phi_x^t = tilde phi_a_gamma x^t T_gamma} and~\eqref{T_gamma^* F = F + ln a_gamma},
\begin{multline*}
\sum_{\gamma\in \Gamma} e^{i\lambda(F(\tilde y')-F(\tilde y)-\ln a_\gamma)} 
\tilde\phi_{0,z}^{t*}T_\gamma^*\,\tilde k_z(\gamma\cdot \tilde\phi_0^t(\tilde y),\tilde y')\\
\begin{aligned}
&= \sum_{\gamma\in \Gamma} e^{i\lambda (F(\tilde y')-F(\gamma\cdot \tilde y))}
\tilde\phi_{0,z}^{t*}T_\gamma^*\,\tilde k_z(\gamma\cdot\tilde\phi_0^t(\tilde y),\tilde y')\\
&= \sum_{\gamma\in \Gamma}e^{i\lambda  (\tilde\phi_0^{t*}F -F)(\gamma\cdot \tilde y)}\tilde\phi_{0,z}^{t*}    
e^{i\lambda (F(\tilde y')-F(\gamma\cdot\tilde\phi_0^t(\tilde y)))}T_\gamma^* \,\tilde k_z(\gamma\cdot
\tilde\phi_0^t(\tilde y),\tilde y') \\
&= \tilde\phi_{0,z+i\lambda}^{t*} \sum_{\gamma\in \Gamma}
T_\gamma^*\tilde k_{z+i\lambda}(\gamma\cdot \tilde\phi_0^t(\tilde y),\tilde y')
\equiv\phi_{0,z+i\lambda}^{t*}\,k_{z+i\lambda}(\phi_0^t([\tilde y]),[\tilde y'])\;.\qedhere
\end{aligned}
\end{multline*}
\end{proof}

\begin{notation}\label{n: subscript u}
In \Cref{n: subscripts psi f z}, we may also add the subscript ``$u$'' if we use a family of functions $\psi_u\in\AA$ depending on a parameter $u$. This also applies to $k_z$ and $\tilde k_z$.
\end{notation}

The identity element of $\Gamma$ is denoted by $e$.

\begin{prop}\label{p: (kappa_pm u - kappa_pm e u)|_Deltab to 0}
If $\psi_u(x)=e^{-ux^2}$, then $(\kappa_{\pm,u}-\kappa_{\pm,e,u})|_{\Delta_{\text{\rm b},\pm}}\to0$ as $u\downarrow0$ in the $C^\infty$ topology.
\end{prop}

\begin{proof}
For $\gamma\in\Gamma$ and $p=[x,\tilde y]=[a_\gamma x,\gamma\cdot\tilde y]\in\mathring M_\pm$, by~\eqref{tilde phi^t(tilde y x) widetilde Z} and~\eqref{T_gamma tilde phi_x^t = tilde phi_a_gamma x^t T_gamma},
\[
\phi^{-h_\pm(\gamma)}(p)=\big[x,\phi^{-h_\pm(\gamma)}_{a_\gamma x}(\gamma\cdot\tilde y)\big]
=\big[x,\gamma\cdot\phi^{-h_\pm(\gamma)}_x(\tilde y)\big]\;.
\]
Thus, using that $\pi_M$ defines an isometric diffeomorphism of $\{x\}\times\widetilde L\equiv\widetilde L$ to $L_p$, it follows from \Cref{c: d_FF(phi^t(p) p) ge c_3 |gamma|} that there is some $c_3>0$, independent of $p$ and $\gamma$, such that, if $h_\pm(\gamma)\in\supp f$, then
\[
d_{\widetilde L}\big(\gamma\cdot\phi^{-h_\pm(\gamma)}_x(\tilde y),\tilde y\big)=d_{\FF}\big(\phi^{-h_\pm(\gamma)}(p),p\big)\ge c_3|\gamma|\;.
\]
Therefore, by~\eqref{heat kernel estimates} and since $\phi$ is of $\R$-local bounded geometry, for $m_1,m_2\in\N_0$, $0<u<u_0$, $\gamma\in\Gamma$, $\rho>0$ and $\tilde y\in\widetilde L$, we get
\[
\big|\partial_\rho^{m_1}\nabla_{\tilde y}^{m_2}\tilde k_{u,z}\big(\gamma\cdot\phi^{-h_\pm(\gamma)}_{\pm e^{-F(\tilde y)}\rho}(\tilde y),\tilde y\big)\big|
\le C_1u^{-(n-1+m_1+m_2)/2}e^{-C_2c_3^2|\gamma|^2/u}\;,
\]
where $\tilde k_{u,z}$ is the Schwartz kernel of $\psi_u(D_{\widetilde L,z})=e^{-u\Delta_{\widetilde L,z}}$. Using again the $\R$-local bounded geometry of $\phi$ and the compactness of $\supp f$, it follows that there is some $C_3>0$ such that
\[
|\partial_\rho^{m_1}\nabla_y^{m_2}\kappa_{\gamma,u}(\rho,1,y,y)|\le C_3u^{-(n-1+m_1+m_2)/2}e^{-C_2c_3^2|\gamma|^2/u}
\]
for $m_1,m_2\in\N_0$, $0<u<u_0$, $\gamma\in\Gamma$, $\rho>0$ and $y\in L$. So
\[
|\partial_\rho^{m_1}\nabla_y^{m_2}(\kappa_{\pm,u}-\kappa_{\pm,e,u})(\rho,1,y,y)|
\le C_3u^{-(n-1+m_1+m_2)/2}\sum_{\gamma\in\Gamma\setminus\{e\}} e^{-C_2c_3^2|\gamma|^2/u}\;,
\]
which converges to zero as $u\downarrow0$.
\end{proof}

\begin{cor}\label{c: kappa_pm u|_Deltab to 0}
If $\psi_u(x)=e^{-ux^2}$ and $f(0)=0$, then $\kappa_{\pm,u}|_{\Delta_{\text{\rm b},\pm}}\to0$ as $u\downarrow0$ in the $C^\infty$ topology. 
\end{cor}

Recall the notation $e(\FF_\pm,g_{\FF_\pm})$ if $n-1$ is even (\Cref{ss: e(FF g_FF),ss: g_M}), and also the notation $C^{0,\infty}_{\varpi_\pm}(M_\pm;\bOmega)$ (\Cref{ss: diff ops on fol mfds}).

\begin{cor}\label{c: lim_u to 0 trs(kappa_pm u|_Delta_b pm) equiv f(0) Pf(R_FF_pm) |omega_pm|}
If $\psi_u(x)=e^{-ux^2}$, then
\[
\lim_{u\downarrow0}\str(\kappa_{\pm,u}|_{\Delta_{\text{\rm b},\pm}})\equiv
\begin{cases}
f(0)\,e(\FF_\pm,g_{\FF_\pm})\,|\omega_\pm| & \text{if $n-1$ is even}\\
0 & \text{if $n-1$ is odd}
\end{cases}
\]
in $C^{0,\infty}_{\varpi_\pm}(M_\pm;\bOmega)$, using the identity $\Delta_{\text{\rm b},\pm}\equiv M_\pm$.
\end{cor}

\begin{proof}

By \Cref{p: kappa_P_pm,p: (kappa_pm u - kappa_pm e u)|_Deltab to 0}, and~\eqref{mathring K_pm(rho rho' [tilde y] [tilde y'])},~\eqref{mathring K_pm gamma(rho rho' [tilde y] [tilde y'])} and~\eqref{rho_pm^-1 d rho_pm}, for all $\rho>0$ and $y=[\tilde y]\in L$ with $\tilde y\in\widetilde L$,
\begin{multline*}
\lim_{u\downarrow0}\str\mathring\kappa_{\pm,u}(\rho,1,y,y)\\
\begin{aligned}
&=\lim_{u\downarrow0}\str\mathring\kappa_{\pm,e,u}(\rho,1,y,y)
=\lim_{u\downarrow0}\str\mathring K_{\pm,e,u}(\rho,\rho,y,y)\\
&=\frac{f(0)}{|\varkappa|}\,\Big|\frac{d\rho}{\rho}\Big|\,\lim_{u\downarrow0}\str\tilde k_{z,u}(\tilde y,\tilde y)
=f(0)\,|\omega_\pm|(y)\,\lim_{u\downarrow0}\str\tilde k_{z,u}(\tilde y,\tilde y)\;.
\end{aligned}
\end{multline*} 
But, by \Cref{t: e_z l},
\[
\lim_{u\downarrow0}\str\tilde k_{z,u}(\tilde y,\tilde y)=e(\widetilde L,g_{\widetilde L})(\tilde y)=e(L,g_L)(y)\equiv e(\FF_\pm,g_{\FF_\pm})(\rho,y)
\]
if $n-1$ is even, and
\[
\lim_{u\downarrow0}\str\tilde k_{z,u}(\tilde y,\tilde y)=0
\]
if $n-1$ is odd.
\end{proof}

\section{Operators on the components $M^1_l$}\label{s: opers on M^1_l}

Consider the notation of \Cref{ss: globalization,ss: components of M^1}; in particular, consider the boundary-defining function $\rho=\rho_l$ on every $M_l$ and the trivialization $\nu=\nu_l$ of ${}_+N\partial M_l$. According to \Cref{ss: b-stretched product}, consider also the lifts of $\rho$ to $M_l^2$ from the left and right factors, $\rho$ and $\rho'$, and the function $s=s_l=\rho/\rho':M_l^2\to[0,\infty]$, as well as the corresponding functions $\rho$, $\rho'$ and $s$ on $(M_l)^2_{\text{\rm b}}$. Equip $\mathring M_l$ with the Riemannian metric $g_{\text{\rm b},l}$, so that $\mathring\FF_l$ becomes a Riemannian foliation of bounded geometry (\Cref{ss: transverse structures of bd geometry}). Consider the leafwise perturbed operators for $(\mathring M_l,\mathring\FF_l)$ defined by the leafwise-closed form $\eta_0$, which agrees with $\eta$ on the collar neighborhood of the boundaty we have fixed. For any $\psi\in\AA$, $f\in\Cinftyc(\R)$, $z\in\C$ and every index $l$, the operator \index{$\mathring P_l$}
\[
\mathring P_l=\int_{-\infty}^{+\infty}\phi_{l,z}^{t*}\,\psi(D_{\mathring\FF_{l,z}})f(t)\,dt
\]
on $H^{-\infty}(\mathring M_l;\Lambda\mathring\FF_l)$ is a twisted version of~\eqref{P with f}, which is smoothing by the appropriate analog of~\eqref{P_psi} (\Cref{s: Witten's opers on Riem folns of bd geom}). Let \index{$\mathring K_l$}
\[
\mathring K_l=K_{\mathring P_l}\in C^\infty(\mathring M_l^2;\Lambda\mathring\FF_l\boxtimes (\Lambda\mathring\FF_l^*\otimes \Omega \mathring M_l))\;.
\]

\begin{lem}\label{l: |mathring P_l|_rho^aH^k rho^aH^m}
For any compact $I\subset \R$ containing $\supp f$, and for all $k,m\in\N_0$ and $a\in\R$, there are some $C',C''>0$ and $N\in\N_0$, depending only on $I$, $k$, $m$ and $a$, such that
\[
\big\|\mathring P_l\big\|_{\rho^aH^k,\rho^aH^m} 
\le C''\,\|\psi\|_{\AA,C',N}\,\|f\|_{I,C^N}\;.
\]
\end{lem}

\begin{proof}
By~\eqref{Witten's opers} and~\eqref{eta_0 = d_0,1(ln rho_l)}, $D_{\mathring\FF_l,z}=\rho^aD_{\mathring\FF_l,z+a,z-a}\rho^{-a}$ (see \Cref{ss: leafwise 2 parameters}). So the result follows from the analog of~\eqref{|P_psi|_m m'} for $D_{\mathring\FF_l,z+a,z-a}$ (\Cref{s: Witten's opers on Riem folns of bd geom}).
\end{proof}

\begin{prop}\label{p: mathring K_l}
The kernel $\mathring K_l$ has a $C^\infty$ extension to $M_l^2\setminus (\partial M_l)^2$ that vanishes to all orders on $(\partial M_l \times \mathring M_l)\cup (\mathring M_l \times \partial M_l)$.
\end{prop}

\begin{proof}
We will use the arguments from the proof of \cite[Theorem~5.2.6]{Hormander1983-I}.

For any $q\in \mathring M_l$ and $\alpha\in \Lambda_q\mathring\FF_l\otimes \Omega^{-1}_q\mathring M_l$, we have $\delta_q^\alpha\in H^k(\mathring M_l;\Lambda\mathring\FF_l)$ for any $k<-n/2$, and $\|\delta_q^\alpha\|_k\le C_k\,|\alpha|$, where $C_k>0$ is independent of $q$ and $\alpha$ (\Cref{ss: Dirac sections}). Therefore, by the definition of weighted Sobolev spaces and the properties of Dirac sections at submanifolds (\Cref{ss: weighted sps,ss: Dirac sections}), for all $a\in \R$, we have $\delta^\alpha_q\in\rho^a H^k(\mathring M_l;\Lambda\mathring\FF_l)$ and
\[
\|\delta^\alpha_q\|_{\rho^aH^k}\le C_k\rho(q)^{-a}|\alpha|\;.
\]
Moreover, for any $\alpha\in C^\infty(\mathring M_l;\Lambda\mathring\FF_l\otimes\Omega^{-1}\mathring M_l)$, the map
\[
\mathring M_l\to\rho^a H^k(\mathring M_l;\Lambda\mathring\FF_l)\;,\quad q\mapsto\delta^{\alpha(q)}_q\;,
\]
is continuous by the continuity of~\eqref{(p u) mapsto delta_p^u(p)}.

Fix any compact $I\subset \R$ containing $\supp f$. By \Cref{l: |mathring P_l|_rho^aH^k rho^aH^m}, we have $\mathring P_l\delta^\alpha_q\in\rho^aH^m(\mathring M_l;\Lambda\mathring\FF_l)$ for any $m\in\N_0$, and
\[
\big\|\mathring P_l\delta^\alpha_q\big\|_{\rho^aH^m}\le C'_m\rho(q)^{-a}\,\|\psi\|_{\AA,C',N}\,\|f\|_{I,C^N}\,|\alpha|
\]
for $q\in \mathring M_l$ and $\alpha\in\Lambda_q\mathring\FF_l\otimes\Omega^{-1}_q\mathring M_l$, where $C'_m>0$ is independent of $a$, $q$ and $\alpha$. Moreover, for any $\alpha\in C^\infty(\mathring M_l;\Lambda\mathring\FF_l\otimes\Omega^{-1}\mathring M_l)$, the map
\[
\mathring M_l\to\rho^aH^m(\mathring M_l;\Lambda\mathring\FF_l)\;,\quad 
q\mapsto \mathring P_l\delta^{\alpha(q)}_q\;,
\]
is continuous. On the other hand, by~\eqref{K_A(cdot q)(u) = A delta_q^u}, for all $q\in\mathring M_l$ and $\alpha\in \Lambda_q\mathring\FF_l\otimes \Omega^{-1}_q\mathring M_l$,
\[
\mathring K_l(\cdot, q)(\alpha)=\mathring P_l\delta_q^\alpha\in C^\infty(\mathring M_l;\Lambda\mathring\FF_l)\;.
\]
It follows that the map
\[
\mathring M_l\to \rho^a H^m(\mathring M_l;\Lambda\mathring\FF_l\otimes\Omega_q\mathring M_l)\;,\quad
q\mapsto \mathring K_l(\cdot, q)\;,
\]
is continuous for any $a\in\R$ and $m\in\N_0$, with
\[
\big\|\mathring K_l(\cdot, q)\big\|_{\rho^aH^m}\le C'_m\rho(q)^{-a}\,\|\psi\|_{\AA,C',N}\,\|f\|_{I,C^N}
\]
for all $q\in \mathring M_l$. Using the Sobolev embedding theorem, we conclude that $\mathring K_l$ is continuous on $\mathring M_l^2$, and
\begin{align*}
\big|\mathring K_l(p,q)\big|&\le C\left(\frac{\rho(p)}{\rho(q)}\right)^{a}\,\|\psi\|_{\AA,C',N},\|f\|_{I,C^N}\\
&=C\,s(p,q)^a\,\|\psi\|_{\AA,C',N}\,\|f\|_{I,C^N}\;,
\end{align*}
for all $a\in \R$ and $p,q\in \mathring M_l$, where $C, C'>0$ and $N\in\N_0$ are independent of $a$, $p$ and $q$. So $\mathring K_l$ extends to a continuous section on $M_l^2\setminus (\partial M_l)^2$, which vanishes on $(\partial M_l \times \mathring M_l)\cup (\mathring M_l \times \partial M_l)$.

For any $D_1, D_2\in \Diffb^k(M_l;\Lambda\FF_l)$, applying the above arguments to the operator $D_1\mathring P_lD_2$ and using~\eqref{K_PAQ(x y)}, it follows that, for all $a\in\R$ and $p,q\in \mathring M_l$, 
\begin{equation}\label{| D_1 p D_2 q^t mathring K_l(p q) |}
\big|D_{1,p}\,D_{2,q}^\trans\mathring K_l(p,q)\big| \le C\,s(p,q)^a\,\|\psi\|_{\AA,C',N}\,\|f\|_{I,C^N}\;,
\end{equation}
where $C, C'>0$ and $N\in\N_0$ are independent of $a$, $p$ and $q$.
\end{proof}

Let $\mathring\kappa_l$ be the $C^\infty$ section of $\beta_{\text{\rm b}}^*(\Lambda\FF_l \boxtimes (\Lambda\FF_l^*\otimes \Omega M_l))$ on the interior of $(M_l)^2_{\text{\rm b}}$ that corresponds to $\mathring K_l$ via $\beta_{\text{\rm b}}^*$, using the notation of \Cref{ss: b-stretched product}.  Fix $0<\epsilon<\epsilon_0$ like in \Cref{ss: globalization}, and consider the notation of \Cref{ss: collar neighborhoods of partial M_l,ss: globalization}. Let $L$ be a boundary component of $M_l$, which can be identified with a leaf of $\FF$ in $M^0$. For $0<\sigma\le\epsilon_0$, via the identity $\mathring T_{L,l,\sigma}\equiv\mathring T'_{L,l,\sigma}$, the sections in $C_\co^\infty(\mathring M_l;\Lambda\mathring\FF_l)$ and $C_\co^\infty(\mathring M_l;\Lambda\mathring\FF_l\otimes\Omega)$ supported in $\mathring T_{L,l,\sigma}$ can be identified with sections in $C_\co^\infty(\mathring M'_l;\Lambda\mathring\FF'_l)$ and $C_\co^\infty(\mathring M'_l;\Lambda\mathring\FF'_l\otimes\Omega)$ supported in $\mathring T'_{L,l,\sigma}$. Similarly, according to \Cref{s: opers suspension},
\[
\beta_{\text{\rm b}}^{-1}(\mathring T_{L,l,\sigma}^2)\equiv\beta_{\text{\rm b}}^{-1}(\mathring T^{\prime\,2}_{L,l,\sigma})
\equiv\{\,(\rho,s, y, y')\in(0,\infty)^2\times L^2\mid \rho,\rho s^{-1}<\sigma\,\}\;.
\]
The operator~\eqref{mathring P_pm suspension}, studied in \Cref{s: opers suspension}, is now expressed as \index{$\mathring P'_{L,l}$}
\[
\mathring P'_{L,l}=\int_{-\infty}^{+\infty}\phi_{L,l,z}^{\prime\,t*}\,\psi(D_{\mathring\FF'_{L,l},z})f(t)\,dt\;.
\]
Let $\mathring K'_{L,l}=K_{\mathring P'_{L,l}}$,\index{$\mathring K'_{L,l}$} with lift $\mathring\kappa'_{L,l}$ \index{$\mathring\kappa'_{L,l}$} to the interior of $(M'_{L,l})^2_{\text{\rm b}}$, and let $\kappa'_{L,l}=\kappa_{P'_{L,l}}$ \index{$\kappa'_{L,l}$} denote the extension of $\mathring\kappa'_{L,l}$ to $(M'_{L,l})^2_{\text{\rm b}}$ given by \Cref{p: kappa_P_pm}. 

The subscripts of \Cref{n: subscripts psi f z,n: subscript u} may be also used with $\mathring P_l$, $\mathring K_l$, $\mathring\kappa_l$, $\mathring P'_{L,l}$, $\mathring K'_{L,l}$, $\mathring\kappa'_{L,l}$ and $\kappa'_{L,l}$.

\begin{prop}\label{p: P equiv P'}
Given $\psi\in\AA$ and $u>0$, take $\psi_u\in\AA$ defined by $\psi_u(x)=\psi(ux)$, and consider the restrictions of $\mathring\kappa_{l,u}$ and $\mathring\kappa'_{L,l,u}$ to $\beta_{\text{\rm b}}^{-1}(\mathring T_{L,l,\epsilon}^2)\equiv\beta_{\text{\rm b}}^{-1}(\mathring T^{\prime\,2}_{L,l,\epsilon})$. There is some $0<\epsilon'<\epsilon$ such that, for any $R>0$, $m,N\in\N_0$ and $a\in\R$, there exist $\widehat C,W>0$ and $N'\in\N_0$ so that, for $m_1+m_2+m_3+m_4\le m$, $0<u\le1$ and $(\rho,s, y, y')\in\beta_{\text{\rm b}}^{-1}(\mathring T_{L,l,\epsilon'}^2)$,
\[
\big|\partial^{m_1}_\rho\partial^{m_2}_s\nabla^{m_3}_y\nabla^{m_4}_{y'}(\mathring\kappa_{l,u}-\mathring\kappa'_{L,l,u})(\rho,s, y, y')\big| 
\le\widehat Ce^{-\frac{R}{u}}\rho^Ns^{a}\,\|\psi\|_{\AA,W,N'}\,\|f\|_{I,C^{N'}}\;.
\]
\end{prop}

\begin{proof}
Take $C>0$ and $c\ge1$ like in \Cref{l: d_widetilde L(tilde y tilde y') > C ln(rho([x tilde y]) / c rho([x tilde y'])} and Corollary~\ref{c: Pen_FF_l(T_epsilon R) subset T_ce^R/C epsilon}, and take $c'>0$ like in \Cref{p: phi^-t(T_epsilon) subset T_kappa epsilon}, for the suspension foliation $\FF'_{L,l}$ on $M'_{L,l}$ and any compact $I\subset\R$ containing $\supp f$.

\begin{claim}\label{cl: leafwise unit propagation speed}
For $\alpha,\beta\in\Cinftyc(\mathring M_{L,l};\Lambda\mathring\FF_{L,l})$, $\alpha',\beta'\in\Cinftyc(\mathring M'_{L,l};\Lambda\mathring\FF'_{L,l})$ and $\xi\in\R$, let
\begin{alignat*}{2}
\alpha(\xi)&=e^{i\xi D_{\mathring\FF_{L,l},z}}\alpha\;,&\quad\beta(\xi)&=e^{i\xi D_{\mathring\FF_{L,l},z}}\beta\;,\\
\alpha'(\xi)&=e^{i\xi D_{\mathring\FF'_{L,l},z}}\alpha'\;,&\quad\beta'(\xi)&=e^{i\xi D_{\mathring\FF'_{L,l},z}}\beta'\;.
\end{alignat*}
The following properties hold for $0<\sigma,\tau<\epsilon$:
\begin{enumerate}[{\rm(i)}]

\item\label{i: alpha(u) equiv alpha'(u)} If $\alpha$ and $\alpha'$ are supported in $\mathring T_{L,l,\sigma}\equiv\mathring T'_{L,l,\sigma}$ and agree there, then $\alpha(\xi)$ and $\alpha'(\xi)$ are supported in $\mathring T_{L,l,\epsilon}\equiv\mathring T'_{L,l,\epsilon}$ and agree there for $|\xi|< C\ln\frac{\epsilon}{c\sigma}$.

\item\label{i: beta(u) equiv beta'(u)} If $\beta$ and $\beta'$ agree on $\mathring T_{L,l,\epsilon}\equiv\mathring T'_{L,l,\epsilon}$, then $\beta(\xi)\equiv\beta'(\xi)$ on $\mathring T_{L,l,\tau}\equiv\mathring T'_{L,l,\tau}$ for $|\xi|<C\ln\frac{\epsilon}{c\tau}$.

\end{enumerate}
\end{claim} 

This is a consequence of Corollary~\ref{c: Pen_FF_l(T_epsilon R) subset T_ce^R/C epsilon} and the leafwise twisted version of~\eqref{leafwise unit propagation speed} applied to the equation $\partial_\xi\mu(\xi)=iD_{\mathring\FF_l,z}\mu(\xi)$ on $\mathring T_{L,l,\epsilon}\equiv\mathring T'_{L,l,\epsilon}$, where $\mu(\xi)=\alpha(\xi)\equiv\alpha'(\xi)$ in~\ref{i: alpha(u) equiv alpha'(u)}, and $\mu(\xi)=\beta(\xi)\equiv\beta'(\xi)$ in~\ref{i: beta(u) equiv beta'(u)}.

\begin{claim}\label{cl: phi^t* alpha(u) equiv phi^t* alpha'(u)}
Let $\alpha$, $\alpha'$, $\alpha(\xi)$ and $\alpha'(\xi)$ be defined like in Claim~\ref{cl: leafwise unit propagation speed}, and let $0<\sigma<\epsilon$ and $0<\tau<\epsilon,\epsilon/c'$. If $\alpha$ and $\alpha'$ are supported in $\mathring T_{L,l,\sigma}\equiv\mathring T'_{L,l,\sigma}$ and agree there, then $\phi^{t*}\alpha(\xi)\equiv\phi^{\prime\,t*}\alpha'(\xi)$ on $\mathring T_{L,l,\tau}\equiv\mathring T'_{L,l,\tau}$ for any $t\in I$ and $|\xi|<C(\ln \frac{\epsilon}{c \sigma}+\ln\frac{\epsilon}{cc'\tau})$.
\end{claim}

By Claim~\ref{cl: leafwise unit propagation speed}~\ref{i: alpha(u) equiv alpha'(u)}, if $\xi<C\ln \frac{\epsilon}{c\sigma}$, then $\alpha(\xi)$ and $\alpha'(\xi)$ are supported in $\mathring T_{L,l,\epsilon}\equiv\mathring T'_{L,l,\epsilon}$ and agree there. Thus, by Claim~\ref{cl: leafwise unit propagation speed}~\ref{i: beta(u) equiv beta'(u)}, if $|\zeta|<C\ln \frac{\epsilon}{cc'\tau}$, then $\alpha(\xi+\zeta)\equiv\alpha'(\xi+\zeta)$ on $\mathring T_{L,l,c'\tau}\equiv\mathring T'_{L,l,c'\tau}$. Hence $\phi^{t*}\alpha(\xi+\zeta)\equiv\phi^{\prime\,t*}\alpha'(\xi+\zeta)$ on $\mathring T_{L,l,\tau}\equiv\mathring T'_{L,l,\tau}$ for all $t\in I$ since $\phi^t(\mathring T_{L,l,\tau})\subset\mathring T_{L,l,c'\tau}$ by \Cref{p: phi^-t(T_epsilon) subset T_kappa epsilon}. This shows Claim~\ref{cl: phi^t* alpha(u) equiv phi^t* alpha'(u)}.

Take any $\mu\in C^\infty(\R)$ such that $0\le\mu\le1$, $\supp\mu\subset(-\infty,0]$, and $\mu=1$ on $(-\infty,-\ln2]$. For $0<\sigma<\epsilon$, let $\chi_\sigma=\mu(\ln\rho-\ln\sigma)\in C^\infty_{\text{\rm ub}}(\mathring M_l)$. We have $\chi_\sigma\ge0$, $\supp\chi_\sigma\subset\mathring T_{L,l,\sigma}$, and $\chi_\sigma=1$ on $\mathring T_{L,l,\sigma/2}$. Moreover $\chi_\sigma\in C^\infty_{\text{\rm ub}}(\mathring M_l)$ and $\|\chi_\sigma\|_{C^m_{\text{\rm ub}}}$ is independent of $\sigma$ for $m\in\N_0$ because $d(\ln\rho)\in C^\infty_{\text{\rm ub}}(\mathring M_l;T^*\mathring M_l)$ (\Cref{ss: globalization}). Let also $0<\tau<\epsilon,\epsilon/c'$ and define $\chi_\tau$ as above. Then the operator $\chi_\tau(\mathring P_{l,u}-\mathring P'_{L,l,u})\chi_\sigma$ is well defined on $H^{-\infty}(\mathring M_l;\Lambda\mathring\FF_l)$ via the identity $\mathring T_{L,l,\epsilon}\equiv\mathring T'_{L,l,\epsilon}$.

Let $\alpha\in C_\co^\infty(\mathring M_l;\Lambda\mathring\FF_l)$ and $\beta\in C_\co^\infty(\mathring M_l;\Lambda\mathring\FF_l^*\otimes\Omega)$. By Claim~\ref{cl: phi^t* alpha(u) equiv phi^t* alpha'(u)} and the version of~\eqref{psi(D_0)} for $\xi uD_{\FF_l,z}$ and $\xi uD_{\FF'_{L,l},z}$ instead of $tD_0$ (\Cref{s: Witten's opers on Riem folns of bd geom}),
\[
\big\langle\chi_\tau(\mathring P_{l,u}-\mathring P'_{L,l,u})\chi_\sigma\alpha,\beta\big\rangle
=\frac{1}{2\pi}\int_{|\xi|>\frac Cu\ln\frac{\epsilon^2}{c^2c'\sigma\tau}} \int_{-\infty}^{+\infty}\hat\psi(\xi)A_{l,z,u}(t,\xi)f(t)\, d\xi\,dt\;,
\]
where
\[
A_{l,z,u}(t,\xi)
=\Big\langle\Big(\phi_{l,z}^{t*}\, e^{i\xi uD_{\mathring\FF_l,z}}-\phi_{L,l,z}^{\prime\,t*}\, e^{i\xi uD_{\mathring\FF'_{L,l},z}}\Big)
\chi_\sigma\alpha,\chi_\tau\beta\Big\rangle\;.
\]
Then, by the version of~\eqref{|e^itD_0 alpha|_m} for $\xi uD_{\FF_l,z}$ and $\xi uD_{\FF'_{L,l},z}$ instead of $tD_0$ (\Cref{s: Witten's opers on Riem folns of bd geom}), since $\phi_l$ and $\phi'_{L,l}$ are of $\R$-local bounded geometry, and using that $\|\chi_\sigma\|_{C^k_{\text{\rm ub}}}$ and $\|\chi_\tau\|_{C^k_{\text{\rm ub}}}$ are finite and independent of $\sigma$ and $\tau$ for all $k\in\N_0$, we get that, for all $m\in\R$,
\begin{align*}
|A_{l,z,u}(t,\xi)|
&\le\Big\|\Big(\phi_{l,z}^{t*}\, e^{i\xi uD_{\mathring\FF_l,z}}
-\phi_{L,l,z}^{\prime\,t*}\, e^{i\xi uD_{\mathring\FF'_{L,l},z}}\Big)\chi_\sigma\alpha\Big\|_m\|\chi_\tau\beta\|_{-m}\\
&\le C'e^{C_m|\xi|}\|\chi_\sigma\alpha \|_m\|\chi_\tau\beta \|_{-m}
\le C''e^{C_m|\xi|}\|\alpha \|_m\|\beta \|_{-m}\;,
\end{align*}
for some $C_m,C',C''>0$ independent of $\alpha$, $\beta$, $\sigma$, $\tau$, and $u\in(0,1]$. So, for all $W>0$,
\begin{multline*}
\big|\big\langle\chi_\tau(\mathring P_{l,u}-\mathring P'_{L,l,u})\chi_\sigma\alpha,\beta\big\rangle\big|\\
\begin{aligned}
&\le\frac{1}{2\pi}\int_{|\xi|>\frac Cu(\ln \frac{\epsilon}{c \sigma}+\ln\frac{\epsilon}{cc'\tau})} 
\int_{-\infty}^{+\infty}\big|\hat\psi(\xi)\big|\, |A_{l,z,u}(t,\xi)|\, |f(t)|\, d\xi\,dt \\
&\le C''\|\alpha \|_m\|\beta \|_{-m}\|f\|_{L^1}
\int_{|\xi|>\frac{C}{u}(\ln \frac{\epsilon}{c \sigma}+\ln\frac{\epsilon}{cc'\tau})}e^{C_m|\xi|}\,\big|\hat\psi(\xi)\big|\,d\xi\\
& \le C''\|\alpha \|_m\|\beta \|_{-m}\|f\|_{L^1}e^{-\frac{CW}{u}(\ln \frac{\epsilon}{c \sigma}+\ln\frac{\epsilon}{cc'\tau})} 
\int_{-\infty}^{\infty}e^{(W+C_m)|\xi|}\,\big|\hat\psi(\xi)\big|\,d\xi\;,
\end{aligned}
\end{multline*}
for some $C_m,C',C''>0$ independent of $\alpha$, $\beta$, $\sigma$, $\tau$, and $u\in(0,1]$. Now, assume 
\begin{equation}\label{sigma < frac epsilon ce}
\sigma<\frac{\epsilon}{ce}\;,\quad\tau<\frac{\epsilon}{cc'e}\;.
\end{equation}
Thus $\ln\frac{\epsilon}{c \sigma},\ln\frac{\epsilon}{cc'\tau}>1$, obtaining 
\begin{align*}
e^{-\frac{CW}{u}(\ln \frac{\epsilon}{c \sigma}+\ln\frac{\epsilon}{cc'\tau})}
&\le e^{-\frac{CW}{u}(1+\frac{1}{2}(\ln \frac{\epsilon}{c \sigma}+\ln\frac{\epsilon}{cc'\tau}))}
\le e^{-\frac{CW}{u}}e^{-\frac{CW}{2}(\ln \frac{\epsilon}{c \sigma}+\ln\frac{\epsilon}{cc'\tau})}\\
&=e^{-\frac{CW}{u}}e^{-\frac{CW}{2}(\ln \frac{\epsilon}{c}+\ln\frac{\epsilon}{cc'})}(\sigma\tau)^{\frac{CW}{2}}
\le e^{-\frac{CW}{u}}(\sigma\tau)^{\frac{CW}{2}}\;.
\end{align*}
Hence
\begin{multline*}
\big|\big\langle\chi_\tau(\mathring P_{l,u}-\mathring P'_{L,l,u})\chi_\sigma\alpha,\beta\big\rangle\big|\\
\le C'''e^{-\frac{CW}{u}}(\sigma\tau)^{\frac{CW}{2}}\,\|\psi\|_{\AA,W+C_m,0}\,\|f\|_{I,C^0}\, \|\alpha \|_m\,\|\beta \|_{-m}\;,
\end{multline*}
for some $C'''>0$ independent of $\alpha$, $\beta$, $\sigma$, $\tau$, and $u\in(0,1]$, but involving the length of $I$. Thus, for any $R>0$, $N\in\N_0$ and $m\in \R$, there are some $\widehat C,W>0$, such that, for all $\sigma$ and $\tau$ as in~\eqref{sigma < frac epsilon ce}, and every $u\in(0,1]$,
\[
\big\|\chi_\tau(\mathring P_{l,u}-\mathring P'_{L,l,u})\chi_\sigma\big\|_m
\le\widehat Ce^{-\frac{R}{u}}\sigma^N \tau^N \|\psi\|_{\AA,W,0}\,\|f\|_{I,C^0}\;.
\]
Using the arguments of the proof of~\eqref{| D_1 p D_2 q^t mathring K_l(p q) |}, we similarly get that, for any $R>0$, $N\in\N_0$ and $m,m'\in \R$, there are $\widehat C,W>0$ and $N'\in\N_0$ such that, for all $\sigma$ and $\tau$ as in~\eqref{sigma < frac epsilon ce}, and every $u\in(0,1]$,
\[
\big\|\chi_\tau(\mathring P_{l,u}-\mathring P'_{L,l,u})\chi_\sigma\big\|_{m,m'}
\le\widehat Ce^{-\frac{R}{u}}\sigma^N \tau^N\|\psi\|_{\AA,W,N'}\,\|f\|_{I,C^{N'}}\;.
\]
Moreover, for any $a\in\R$, replacing $\alpha$ with $\rho^{-a}\alpha$ and $\beta$ with $\rho^{a}\beta$ in the above argument, we also get
\begin{multline*}
\big\langle \chi_\tau(\mathring P_{l,u}-\mathring P'_{L,l,u})\chi_\sigma\rho^{-a}\alpha, \rho^{a}\beta\big\rangle\\
=\frac{1}{2\pi}\int_{|\xi|u>C(\ln \frac{\epsilon}{c \sigma}+\ln\frac{\epsilon}{cc'\tau})} \int_{-\infty}^{+\infty}\hat\psi(\xi) B_{l,z}(t,\xi,a) f(t)\, d\xi\,dt\;, 
\end{multline*}
where
\begin{align*}
B_{l,z}(t,\xi,a)
&=\Big\langle\Big(\phi_{l,z}^{t*}\, e^{i\xi uD_{\mathring\FF_l,z}}
-\phi_{L,l,z}^{\prime\,t*}\, e^{i\xi uD_{\mathring\FF'_{L,l},z}}\Big)\chi_\sigma\rho^{-a}\alpha,\chi_\tau\rho^{a}\beta\Big\rangle \\
&=\Big\langle\Big(\rho^{a}\phi_{l,z}^{t*}\, e^{i\xi uD_{\mathring\FF_l,z}}\rho^{-a}
-\rho^{a}\phi_{L,l,z}^{\prime\,t*}\, e^{i\xi uD_{\mathring\FF'_{L,l},z}}\rho^{-a}\Big)\chi_\sigma\alpha,\chi_\tau\beta\Big\rangle \\
&=\Big\langle\Big(\phi_{l,z-a}^{t*} e^{i\xi uD_{\mathring\FF_l,z-a,z+a}}
-\phi_{L,l,z-a}^{\prime\,t*}e^{i\xi uD_{\mathring\FF'_{L,l},z-a,z+a}}\Big)\chi_\sigma\alpha,\chi_\tau\beta\Big\rangle\;.
\end{align*}  
Then, proceeding as above, we obtain
\[
\big\|\chi_\tau(\mathring P_{l,u}-\mathring P'_{L,l,u})\chi_\sigma\big\|_{\rho^aH^m,\rho^a H^{m'}}
\le\widehat Ce^{-\frac{R}{u}}\sigma^N \tau^N\|\psi\|_{\AA,W,N'}\,\|f\|_{I,C^{N'}}\;,
\]
for some $\widehat C,W>0$ and $N'\in\N_0$, depending only on $R$, $N$, $m$, $m'$ and $a$. Using the Sobolev embedding theorem as in \Cref{p: mathring K_l}, it follows that, for any $a\in\R$, $R>0$ and $N\in\N_0$, there are some $\widehat C,W>0$ and $N'\in\N_0$ such that
\begin{multline*}
\big|\chi_\tau(p)(\mathring K_{l,u}(p,q)-\mathring K'_{L,l,u}(p,q))\chi_\sigma(q)\big|\\
\le\widehat Ce^{-\frac{R}{u}} \Big(\frac{\rho(p)}{\rho(q)}\Big)^a \sigma^N \tau^N\,\|\psi\|_{\AA,W,N'}\,\|f\|_{I,C^{N'}}\;,
\end{multline*}
for all $p,q\in\mathring T_{L,l,\epsilon}$ and $u\in(0,1]$, and every $\sigma$ and $\tau$ as in~\eqref{sigma < frac epsilon ce}. Put 
\[
\epsilon'=\min\left(\frac{\epsilon}{4ce}, \frac{\epsilon}{4cc'e}\right)\;.
\]
For $p,q\in\mathring T_{L,l,\epsilon'}$, we set $\tau=3\rho(p)$ and $\sigma=3\rho(q)$. It is clear that $\sigma$ and $\tau$ satisfy~\eqref{sigma < frac epsilon ce} and $\chi_\tau(p)=\chi_\sigma(q)=1$ (since $\rho(p)<\tau/2, \rho(q)<\sigma/2$). Therefore, by the above estimate, we get
\begin{align*}
\big|(\mathring K_{l,u}-\mathring K'_{L,l,u})(p,q)\big|
&\le 9^N\widehat Ce^{-\frac{R}{u}}s(p,q)^a\rho(p)^N\rho(q)^N\,\|\psi\|_{\AA,W,N'}\,\|f\|_{I,C^{N'}}\\
&=9^N\widehat Ce^{-\frac{R}{u}}s(p,q)^{a-N}\rho(p)^{a+N}\|\psi\|_{\AA,W,N'}\,\|f\|_{I,C^{N'}}\;,
\end{align*}
for all $p,q\in\mathring T_{L,l,\epsilon'}$ and $u\in(0,1]$.

For any $k\in\N_0$, taking arbitrary operators $D_1, D_2\in \Diffb^k(M_l;\Lambda\FF_l)$ and $D'_1, D'_2\in \Diffb^k(M'_{L,l};\Lambda\FF'_{L,l})$ with $D_i\equiv D'_i$ on $\mathring T_{L,l,\epsilon}\equiv\mathring T'_{L,l,\epsilon}$ ($i=1,2$), and applying the above arguments to the operators $D_1\mathring P_{l,u}D_2$ and $D'_1\mathring P'_{l,u}D'_2$, we obtain that, for all $a\in\R$ and $N\in\N_0$, there are some $\widehat C,W>0$ and $N'\in\N_0$ such that, for all $p,q\in\mathring T_{L,l,\epsilon'}$ and $u\in(0,1]$,
\begin{equation}\label{|D_1 p D_2 q^t(mathring K_l u - mathring K'_l u)(p q)|}
\big|D_{1,p}\,D_{2,q}^\trans(\mathring K_{l,u}-\mathring K'_{L,l,u})(p,q)\big|
\le\widehat Ce^{-\frac{R}{u}}s(p,q)^a\rho(p)^N\|\psi\|_{\AA,W,N'}\,\|f\|_{I,C^{N'}}\;.
\end{equation}
Consider the vector bundle $S=\Lambda\FF_l\boxtimes(\Lambda\FF_l^*\otimes\Omega M_l)$ over $M_l^2$. Recall that $\Diffb^k((M_l)^2_{\text{\rm b}};\beta_{\text{\rm b}}^*S)$ is $C^\infty((M_l)^2_{\text{\rm b}})$-spanned by the lift of $\Diffb^k(M_l^2;S)$, and $\Diffb^k(M_l^2;S)$ is $C^\infty(M_l^2)$-spanned by the lift of $\Diffb^k(M_l;\Lambda\FF_l)$ from the left-factor projection and the lift of $\Diffb^k(M_l;\Lambda\FF_l^*\otimes\Omega)$ from right-factor projection (\Cref{ss: b-stretched product}). Then it follows from~\eqref{|D_1 p D_2 q^t(mathring K_l u - mathring K'_l u)(p q)|} that, for all $A\in\Diffb^k((M_l)^2_{\text{\rm b}};\beta_{\text{\rm b}}^*S)$, $a\in\R$ and $N\in\N_0$, there are some $\widehat C,W>0$ and $N'\in\N_0$ such that, on $\beta_{\text{\rm b}}^{-1}(\mathring T_{L,l,\epsilon'}^2)$,
\[
\big|A(\mathring\kappa_{l,u}-\mathring\kappa'_{L,l,u})\big|
\le\widehat Ce^{-\frac{R}{u}}s^a\rho^N\|\psi\|_{\AA,W,N'}\,\|f\|_{I,C^{N'}}\;.
\]
Since $a$ and $N$ are arbitrary, this indeed holds with $A\in\Diff^k((M_l)^2_{\text{\rm b}};\beta_{\text{\rm b}}^*S)$, after possibly increasing $\widehat C$, obtaining the stated inequality.
\end{proof}

\Cref{p: P equiv P'} means that $\mathring\kappa_{l,u}-\mathring\kappa'_{L,l,u}$ has a $\dot C^\infty$ extension on the open subset $(T_{L,l,\epsilon})^2_{\text{\rm b}}\subset(M_l)^2_{\text{\rm b}}$ over $T_{\epsilon'}\equiv T'_{\epsilon'}$.

Recall that $\phi_L=\{\phi_L^t\}$ denotes the restriction of $\phi_l$, or of $\phi$, to any boundary leaf $L$ of $\FF_l$.

\begin{cor}\label{c: P_l u in Psib^-infty(M_l bigwedge T^*FF_l)}
The section $\mathring\kappa_l$ has a $C^\infty$ extension $\kappa_l$ to $(M_l)^2_{\text{\rm b}}$, which vanishes to all orders at $\lb\cup\rb$, and therefore $\mathring P_l$ defines an operator $P_l\in\Psib^{-\infty}(M_l;\Lambda\FF_l)$. Moreover
\begin{align*}
I_{\nu_l}(P_l,\lambda)
&\equiv\bigoplus_L\int_{-\infty}^{+\infty}\phi_{L,z+i\lambda}^{t*}\,\psi(D_{L,z+i\lambda})\,e^{i\lambda\varkappa_Lt}f(t)\,dt\\
&\in\Psi^{-\infty}(\partial M_l;\Lambda)\equiv\bigoplus_L\Psi^{-\infty}(L;\Lambda)\;,
\end{align*}
where $L$ runs in $\pi_0(\partial M_l)$.
\end{cor}

\begin{proof}
This follows from \Cref{p: kappa_P_pm,p: I_nu(P_pm z lambda)),p: mathring K_l,p: P equiv P'}.
\end{proof}

The subscripts of \Cref{n: subscripts psi f z,n: subscript u} may be also used with $P_l$ and $\kappa_l$. If needed, the subscript ``$l$'' is also added to the notation of the b-diagonal $\Deltab$ of $(M_l)^2_{\text{\rm b}}$, and to $\Delta_{\text{\rm b},0}=\Deltab\cap\ff$.

\begin{cor}\label{c: (psi f) mapsto kappa_l psi f is cont}
The bilinear map
\[
\AA\times\Cinftyc(\R)\to C^\infty\big((M_l)^2_{\text{\rm b}};\beta_{\text{\rm b}}^*(\Lambda\FF_l \boxtimes (\Lambda\FF^*_l\otimes \Omega M_l))\big)\;,\quad(\psi,f)\mapsto\kappa_{l,\psi,f}\;,
\]
is continuous.
\end{cor}

\begin{proof}
Apply~\eqref{| D_1 p D_2 q^t mathring K_l(p q) |} and \Cref{p: (psi f) mapsto kappa_pm psi f is cont,p: P equiv P'}.
\end{proof}

\begin{cor}\label{c: kappa_l u to 0  on Deltab cap beta_b^-1(T_l epsilon'^2)}
If $\psi_u(x)=e^{-ux^2}$ {\rm(}$u>0${\rm)} and $f(0)=0$, then there is some $0<\epsilon'<\epsilon$ such that $\kappa_{l,u}\to0$ on $\Delta_{\text{\rm b},l}\cap\beta_{\text{\rm b}}^{-1}(T_{l,\epsilon'}^2)\equiv T_{l,\epsilon'}$, in the $C^\infty$ topology, as $u\downarrow0$.
\end{cor}

\begin{proof}
This is a consequence of Corollary~\ref{c: kappa_pm u|_Deltab to 0} and \Cref{p: P equiv P'}.
\end{proof}

\begin{cor}\label{c: lim_u to 0 trs(kappa_l u|_Deltab) equiv f(0) Pf(R_FF_l) |omega_l|}
If $\psi_u(x)=e^{-ux^2}$, then there is some $0<\epsilon'<\epsilon$ such that
\[
\lim_{u\downarrow0}\str(\kappa_{l,u}|_{\Delta_{\text{\rm b},l}})\equiv
\begin{cases}
f(0)\,e(\FF_l,g_{\FF_l})\,|\omega_{\text{\rm b},l}| & \text{if $n-1$ is even}\\
0 & \text{if $n-1$ is odd}
\end{cases}
\]
in $C^{0,\infty}_{\varpi_l}(T_{l,\epsilon'};\bOmega)$, using the identity $\Delta_{\text{\rm b},l}\cap\beta_{\text{\rm b}}^{-1}(T_{l,\epsilon'}^2)\equiv T_{l,\epsilon'}$.
\end{cor}

\begin{proof}
This is a consequence of Corollary~\ref{c: lim_u to 0 trs(kappa_pm u|_Delta_b pm) equiv f(0) Pf(R_FF_pm) |omega_pm|} and \Cref{p: P equiv P'}.
\end{proof}

\begin{prop}\label{p: d_FF_l z in Psi_b^c} 
We have
\[
d_{\FF_l,z}\in\Diffb^1(M_l;\Lambda\FF_l)\;,\quad I_{\nu_l}(d_{\FF_l,z},\lambda)=d_{\partial M_l,z+i\lambda}\;.
\]
\end{prop}

\begin{proof}
By~\eqref{d_FF(f_I dx^prime prime I)}, $d_{\FF_l,z}\in\Diff^1(\FF_l;\Lambda\FF_l)\subset\Diffb^1(M_l;\Lambda\FF_l)$. By~\eqref{I_nu(A lambda) = (x^-i lambda A x^i lambda)_partial} and~\eqref{eta_0 = d_0,1(ln rho_l)},
\[
I_{\nu_l}(d_{\FF_l,z},\lambda)=(\rho^{-i\lambda}d_{\FF_l,z}\rho^{i\lambda})_\partial
=(d_{\FF_l,z}+i\lambda\rho^{-1}d_{\FF_l}\rho\wedge)_\partial=d_{\partial M_l,z+i\lambda}\;.
\]
Alternatively, we can use~\eqref{I_nu(A lambda) for A in Diffb^m(M)} and~\eqref{d_FF(f_I dx^prime prime I)} to describe $I_{\nu_l}(d_{\FF_l,z},\lambda)$.
\end{proof}

Recall that $\bfM\equiv\bigsqcup_l M_l$ and $\bfnu$ is the combination of the sections $\nu_l$ (\Cref{ss: globalization}). This boldface notation of \Cref{ss: transv simple flows - suspension,ss: collar neighborhoods of partial M_l,ss: globalization} allows to simplify the notation of direct sums of section spaces, cohomologies and operators defined on the manifolds $M_l$. For instance, we get the operators
\begin{gather*}
\bfP\equiv\bigoplus_lP_l\in{\textstyle\Psib^{-\infty}(\bfM;\Lambda\bfFF)}
\equiv\bigoplus_l{\textstyle\Psib^{-\infty}(M_l;\Lambda\FF_l)}\;,\\
d_{\bfFF,z}\equiv\bigoplus_ld_{\FF_l,z}\in{\textstyle\Psib^1(\bfM;\Lambda\bfFF)}
\equiv\bigoplus_l{\textstyle\Psib^1(M_l;\Lambda\FF_l)}\;,
\end{gather*}
whose indicial operators are
\begin{gather*}
I_{\bfnu}(\bfP,\lambda)\equiv\bigoplus_lI_{\nu_l}(P_l,\lambda)
\in\Psi^{-\infty}(\partial\bfM;\Lambda)\equiv\bigoplus_L\Psi^{-\infty}(L;\Lambda)\;,\\
I_{\bfnu}(d_{\bfFF,z},\lambda)\equiv\bigoplus_lI_{\nu_l}(d_{\FF_l,z},\lambda)\in\Psi^1(\partial\bfM;\Lambda)\equiv\bigoplus_L\Psi^1(L;\Lambda)\;,
\end{gather*}
where $L$ runs in $\pi_0(\partial \bfM)\equiv\pi_0M^0\sqcup\pi_0M^0$. On the other hand, according to \Cref{p: d_FF_l z in Psi_b^c}, $I_{\bfnu}(d_{\bfFF,z},\lambda)=d_{\partial \bfM,z+i\lambda}$. Let also $\kappa=\kappa_{\bfP}$ on $\bfM^2_{\text{\rm b}}\equiv\bigsqcup_l(M_l)^2_{\text{\rm b}}$, which is the combination of the sections $\kappa_l$. In $\bfM^2_{\text{\rm b}}$, we have $\Deltab\equiv\bigsqcup_l\Delta_{\text{\rm b},l}$ and $\Delta_{\text{\rm b},0}\equiv\bigsqcup_l\Delta_{\text{\rm b},0,l}$. The subscripts of \Cref{n: subscripts psi f z,n: subscript u} may be also used with $\bfP$ and $\kappa$.

When $z=\mu\in\R$ and $\psi(x)=\psi_u(x)=e^{-ux^2}$ ($u>0$), the above $\bfP=\bfP_u$ is the operator $\bfP_{\mu,u,f}$ of \Cref{ss: intro - leafwise Witten}. Thus \Cref{t: intro - bfP_u f} is a consequence of \Cref{c: smallnuint is cont,c: P_l u in Psib^-infty(M_l bigwedge T^*FF_l),c: (psi f) mapsto kappa_l psi f is cont}.

\section{The limit of $\bStr(\bfP_u)$ as $u\downarrow0$}\label{s: lim_u to 0 bTrs(P_u z)}

With the notation of \Cref{ss: simple flows}, let $\CC=\CC(\phi)$, $\PP=\PP(\phi)$, $\CC_l=\CC(\phi_l)$ and $\PP_l=\PP(\phi_l)$. For any leafwise density $\alpha\in C^\infty(M_l;\Omega\FF_l)$, we can consider $\alpha\,|\omega_{\text{\rm b},l}|\in C^\infty(M_l;\bOmega)$. In particular, if $n-1$ is even, the leafwise Euler density $e(\FF_l,g_{\FF_l})\in C^\infty(M_l;\Omega\FF_l)$ (\Cref{ss: e(FF g_FF)}) gives rise to the b-density $e(\FF_l,g_{\FF_l})\,|\omega_{\text{\rm b},l}|\in C^\infty(M_l;\bOmega)$, whose b-integral,
\[
\bchi_{|\omega_{\text{\rm b},l}|}(\FF_l)=\nulint_{M_l}e(\FF_l,g_{\FF_l})\,|\omega_{\text{\rm b},l}|\;,
\]
can be called the \emph{b-Connes $|\omega_{\text{\rm b},l}|$-Euler characteristic} of $\FF_l$. This is a b-normalized version of the Connes $|\omega_{\text{\rm b},l}|$-Euler characteristic, where $|\omega_{\text{\rm b},l}|$ is considered as an invariant transverse measure of $\FF^1_l$. The usual Connes $|\omega_{\text{\rm b},l}|$-Euler characteristic is not defined because $M^1_l$ is not compact. If $n-1$ is odd, let $\bchi_{|\omega_{\text{\rm b},l}|}(\FF_l)=0$.

Recall the operator $\bfP$ defined in \Cref{s: opers on M^1_l}.

\begin{thm}\label{t: lim_u -> 0 bTrs(P_u z)}
If $\psi_u(x)=e^{-ux^2}$ {\rm(}$u>0${\rm)}, then
\[
\lim_{u\downarrow0}\bStr(\bfP_u)=\sum_l\bchi_{|\omega_{\text{\rm b},l}|}(\FF_l)\cdot f(0)+\sum_{c\in\CC}\ell(c)\sum_{k\in\Z^\times}\epsilon_c(k)\cdot f(k\ell(c))\;.
\]
\end{thm}

To prove this theorem, we consider every $P_{l,u}$, separately. Recall that $\mathring\kappa_{l,u,z}$ corresponds to $\mathring k_{l,z,u}$ via the restriction of $\beta_{\text{\rm b}}:(M_l)^2_{\text{\rm b}}\to M_l^2$ to the interiors. Thus we are going to study the asymptotic behaviour of $\mathring k_{l,z,u}$ as $u\downarrow0$. The identities $\mathring M_l\equiv M^1_l$, $\mathring\FF_l\equiv\FF^1_l$ and $\mathring T_{l,\epsilon'}\equiv T^1_{l,\epsilon'}$ ($0<\epsilon'\le\epsilon$) will be used without further comment. With the notation of \Cref{ss: components of M^1,ss: metric properties of M^1}, and adapting the notation of \Cref{ss: Schwartz kernels}, let $\fG_l=\Hol\FF^1_l$ and $\widetilde\fG_l=\Hol\widetilde\FF^1_l$, with source and target projections, $\bfs,\bfr:\fG_l\to M^1_l$ and $\bfs,\bfr:\widetilde\fG_l\to\widetilde M^1_l$. The pairs $(\bfr,\bfs)$ define identities $\fG_l\equiv\RR_l:=\RR_{\FF^1_l}$ and $\widetilde\fG_l\equiv\widetilde\RR_l:=\RR_{\widetilde\FF^1_l}$. Let $\Delta_l\subset\RR_l$ denote the diagonal. Consider also the vector bundles
\[
S_l=\bfs^*\Lambda\FF^1_l\otimes\bfr^*(\Lambda\FF_l^{1*}\otimes\Omega\FF^1_l)\;,\quad
\widetilde S_l=\bfs^*\Lambda\widetilde\FF^1_l\otimes\bfr^*\big(\Lambda\widetilde\FF_l^{1*}\otimes\Omega\widetilde\FF^1_l\big)\;,
\]
over $\fG_l$ and $\widetilde\fG_l$, and the leafwise Schwartz kernel $\tilde k_{l,z,u}$ defined by the Schwartz kernels of the operators $e^{-u\Delta_{\widetilde L',z}}$ on the leaves $\widetilde L'$ of $\widetilde\FF^1_l$, for $z\in\C$ (\Cref{ss: Schwartz kernels}). By~\eqref{Schwartz kernel} and since $\tilde\omega_{\text{\rm b},l}=D_l^*dx$ (\Cref{ss: components of M^1}), for $\tilde p\in\widetilde M^1_l$ and $p=[\tilde p]\in M^1_l$,
\begin{equation}\label{mathring K_l u z(p p)}
\mathring k_{l,z,u}(p,p)
\equiv\sum_{\gamma\in\Gamma_l}
\tilde\phi_{l,z}^{-h_l(\gamma)*}\,T_\gamma^*\,\tilde k_{l,z,u}\big(T_\gamma\tilde\phi_l^{-h_l(\gamma)}(\tilde p),\tilde p\big)\,
f(-h_l(\gamma))\,|\omega_{\text{\rm b},l}|(p)\;,
\end{equation}
using that $\widetilde S_{(\gamma\cdot\tilde p,\tilde p)}\equiv S_{(p,p)}$. This defines a convergent series in $\Cinftyub(\Delta_l;S_l)$.

Any leaf of $\widetilde\FF^1_l$ is of the form $\widetilde L'=\{x\}\times L_l\equiv L_l$ for some $x\in\R$. Then the restriction of $\tilde g_{\text{\rm b},l}$ to $\widetilde L'$ is identified with a metric $\tilde g_{l,x}$ on $L_l$, $\Delta_{\widetilde L',z}$ is identified with the twisted Laplacian $\Delta_{l,x,z}$ on $(L_l,\tilde g_{l,x})$ defined by the restriction of $\tilde\eta_0$, and $\tilde k_{l,z,u}$ on $\widetilde L^{\prime\,2}$ is identified with the Schwartz kernel $\tilde k_{l,x,u,z}$ of $e^{-u\Delta_{l,x,z}}$, defined on $L_l^2$.

\Cref{t: lim_u -> 0 bTrs(P_u z)} follows from the following result.

\begin{prop}\label{p: lim_u -> 0 bTrs(P_l,u,z)}
Let $I\subset\R$ be a compact interval with $\supp f\subset I$. Then the following properties hold:
\begin{enumerate}[{\rm(i)}]
\item\label{i: no periods in I} If $I\subset\R^\times$ and $I\cap\PP_l=\emptyset$, then
\[
\lim_{u\downarrow0}\bStr(P_{l,u})=0\;.
\]
\item\label{i: one period in I} If $I\subset\R^\times$ and $I\cap\PP_l=\{t_0\}$, then
\[
\lim_{u\downarrow0}\bStr(P_{l,u})=f(t_0)\sum_{c\in\CC_{l,t_0}}\ell(c)\,\epsilon_c(t_0/\ell(c))\;,
\]
where $\CC_{l,t_0}$ consists of the orbits $c\in\CC_l$ with period $t_0$.
\item\label{i: 0 in I} If $0\in I$ and $I\cap\PP_l=\emptyset$, then
\[
\lim_{u\downarrow0}\bStr(P_{l,u})=f(0)\,\bchi_{|\omega_{\text{\rm b},l}|}(\FF_l)\;.
\]
\end{enumerate}
\end{prop}

\begin{proof}
Choose some $0<\epsilon'<\epsilon$ satisfying the statements of Corollaries~\ref{c: kappa_l u to 0  on Deltab cap beta_b^-1(T_l epsilon'^2)} and~\ref{c: lim_u to 0 trs(kappa_l u|_Deltab) equiv f(0) Pf(R_FF_l) |omega_l|}, and take some $0<\epsilon''<\epsilon'$. Take some $C_3\ge1$ satisfying~\eqref{C_3^-1 |gamma| le tilde d_l(tilde p gamma . tilde p) le C_3 |gamma|}. Since $\tilde\phi_l$ is of $\R$-local bounded geometry (\Cref{ss: families of bd geom}), there is some $R\ge0$ such that $\tilde d_l(\tilde\phi_l^t(\tilde p),\tilde p)\le R$ for all $\tilde p\in\widetilde M^1_l$ and $t\in I$. So, by~\eqref{C_3^-1 |gamma| le tilde d_l(tilde p gamma . tilde p) le C_3 |gamma|} and the triangle inequality, for all $\tilde p\in\widetilde M^1_{l,\epsilon''}$ and $\gamma\in\Gamma_l$ with $-h_l(\gamma)\in I$, we get
\begin{equation}\label{C_3^-1 |gamma| - R le d_widetilde FF^1_l(gamma^-1 . tilde phi^h_l(gamma)(tilde p) tilde p) le C_3 |gamma| + R}
C_3^{-1}|\gamma|-R\le \tilde d_l\big(\gamma\cdot\tilde\phi_l^{-h_l(\gamma)}(\tilde p),\tilde p\big)\le C_3|\gamma|+R\;,
\end{equation}
using also that $\gamma\cdot\tilde\phi_l^{-h_l(\gamma)}(\tilde p)=\tilde\phi_l^{-h_l(\gamma)}(\gamma\cdot\tilde p)$ with $\gamma\cdot\tilde p\in\widetilde M^1_{l,\epsilon''}$.

By the $\R$-local bounded geometry of $\tilde\phi_l$ and the compactness of $I$, there are $C_4,C_5>0$ such that, for all $t\in I$,
\begin{equation}\label{|tilde phi_l z^t*| le C_4 |f(t)| le C_5}
\big|\tilde\phi_{l,z}^{t*}\big|\le C_4\;,\quad|f(t)|\le C_5\;.
\end{equation}

Assume $I\subset\R^\times$ and $I\cap\PP_l=\emptyset$ to prove~\ref{i: no periods in I}. Thus
\[
\{\,(p,\phi_l^t(p))\mid p\in M^1_{l,\epsilon''},\ t\in I\,\}
\]
is a compact subset of $(M^1_l)^2\setminus\Delta_l$. By \Cref{l: d_FF reaches maximum/minimum}~\ref{i: d_FF reaches minimum}, there is some $C_6>0$ such that $d_{\FF^1_l}(\phi_l^t(p),p)\ge C_6$ for all $p\in M^1_{l,\epsilon''}$ and $t\in I$. So, for all $\tilde p\in\widetilde M^1_{l,\epsilon''}$ and $\gamma\in\Gamma_l$ with $-h_l(\gamma)\in I$,
\begin{equation}\label{d_widetilde FF^1_l(gamma^-1 . tilde phi_l^t(tilde p) tilde p) ge C_6}
d_{\widetilde\FF^1_l}\big(\gamma\cdot\tilde\phi_l^{-h_l(\gamma)}(\tilde p),\tilde p\big)\ge C_6\;.
\end{equation}

Take some $C_7>0$ such that, for all $\gamma\in\Gamma_l$ with $-h_l(\gamma)\in I$,
\[
C_7|\gamma|\le 
\begin{cases}
C_3^{-1}|\gamma|-R & \text{if $|\gamma|>C_3R$}\\
C_6 & \text{if $|\gamma|\le C_3R$}\;.
\end{cases}
\]
Since $\tilde d_l\le d_{\widetilde\FF^1_l}$ (\Cref{ss: leafwise metric}), it follows from~\eqref{C_3^-1 |gamma| - R le d_widetilde FF^1_l(gamma^-1 . tilde phi^h_l(gamma)(tilde p) tilde p) le C_3 |gamma| + R} and~\eqref{d_widetilde FF^1_l(gamma^-1 . tilde phi_l^t(tilde p) tilde p) ge C_6} that, for all $\tilde p\in\widetilde M^1_{l,\epsilon''}$ and $\gamma\in\Gamma_l$ with $-h_l(\gamma)\in I$,
\begin{equation}\label{d_widetilde FF^1_l(gamma^-1 . tilde phi_l^t(tilde p) tilde p) ge C_7 |gamma|}
d_{\widetilde\FF^1_l}\big(\gamma\cdot\tilde\phi_l^{-h_l(\gamma)}(\tilde p),\tilde p\big)\ge C_7|\gamma|\;.
\end{equation}
By~\eqref{heat kernel estimates} and~\eqref{d_widetilde FF^1_l(gamma^-1 . tilde phi_l^t(tilde p) tilde p) ge C_7 |gamma|}, and since the leaves of $\widetilde\FF^1_l$ are of equi-bounded geometry, there are $C_1,C_2,u_0>0$ such that, for all $0<u\le u_0$, $\tilde p\in\widetilde M^1_{l,\epsilon''}$ and $\gamma\in\Gamma_l$ with $-h_l(\gamma)\in I$,
\begin{equation}\label{| tilde phi_l^h_l(gamma)* T_gamma^-1* tilde k_l u z(gamma^-1 tilde phi^h_l(gamma)(tilde p) tilde p) | le ...}
\big|\tilde k_{l,z,u}\big(\gamma\cdot\tilde\phi_l^{-h_l(\gamma)}(\tilde p),\tilde p\big)\big|
\le C_1u^{(n-1)/2}e^{-C_2C_7^2|\gamma|^2/u}\;.
\end{equation}
Hence, by~\eqref{mathring K_l u z(p p)} and~\eqref{|tilde phi_l z^t*| le C_4 |f(t)| le C_5}, for all $0<u\le u_0$ and $p\in M^1_{l,\epsilon''}$,
\begin{equation}\label{| mathring K_l u z(p p) | le ...}
\big|\mathring k_{l,z,u}(p,p)\big|
\le C_4C_5C_1u^{(n-1)/2}\sum_{\gamma\in\Gamma_l}e^{-C_2C_7^2|\gamma|^2/u}\;,
\end{equation}
which converges to zero as $u\downarrow0$ because $\Gamma_l$ is of polynomial growth. Since $M^1_{l,\epsilon''}$ is compact, we get
\[
\lim_{u\downarrow0}\int_{p\in M^1_{l,\epsilon''}}\str\mathring k_{l,z,u}(p,p)=0\;,
\]
and therefore~\ref{i: no periods in I} follows by Corollaries~\ref{c: kappa_l u to 0  on Deltab cap beta_b^-1(T_l epsilon'^2)} and~\ref{c: smallnuint is cont}.

Now assume $I\subset\R^\times$ and $I\cap\PP=\{t_0\}$ to prove~\ref{i: one period in I}, and let $\CC_{l,t_0}=\{c_1,\dots,c_m\}$. \index{$\CC_{l,t_0}$} Then the following properties hold (\Cref{ss: components of M^1}):
\begin{enumerate}[(A)]
\setcounter{enumi}{13}

\item\label{i-(A): t_0 = h_l(gamma_0)} There is a unique $\gamma_0\in\Gamma_l$ such that $t_0=-h_l(\gamma_0)$. 

\item\label{i-(A): k_j = t_0/ell(c_j) in Z} We have $k_j:=t_0/\ell(c_j)\in\Z$ ($j=1,\dots,m$).

\item\label{i-(A): pi_l: R times y_j to c_j} There is some $y_j\in L_l$ such that $\pi_l:\R\times\{y_j\}\to c_j$ is a $C^\infty$ covering map with fundamental domain $[0,\ell(c_j)]\times\{y_j\}$.

\item\label{i-(A): gamma_0 cdot tilde phi_l^t_0(tilde p) = tilde p} For all $\tilde p\in\R\times\{y_j\}$, we have $\gamma_0\cdot\tilde\phi_l^{t_0}(\tilde p)=\tilde p$.

\item\label{i-(A): epsilon_{y_j}(T_gamma_0 tilde phi_l x^t_0) = epsilon_c_j(k_j)} For all $x\in\R$, every $y_j$ is a simple fixed point of the diffeomorphism $T_{\gamma_0}\tilde\phi_{l,x}^{t_0}$ of $L_l$ with $\epsilon_{y_j}(T_{\gamma_0}\tilde\phi_{l,x}^{t_0})=\epsilon_{c_j}(k_j,\phi)=\epsilon_{c_j}(k_j)$.

\end{enumerate}
In particular, there are no other fixed points of $T_{\gamma_0}\tilde\phi_{l,x_j}^{t_0}$ in some open neighborhood $W_j$ of $y_j$ in $L_l$. Then $\pi_l([0,\ell(c_j)]\times W_j)$ is a neighborhood of $c_j$, whose interior is denoted by $V_j$, which does not intersect other closed orbits with period in $I$. Note that $\pi_l:(0,\ell(c_j))\times W_j\to V_j$ is a $C^\infty$ embedding and $V_j\setminus\pi_l((0,\ell(c_j))\times W_j)=\pi_l(\{0\}\times W_j)$ is of measure zero. For every $p\in V_j$, let $\tilde p$ be the unique point in $[0,\ell(c_j))\times W_j$ with $\pi_l(\tilde p)=p$. We have
\begin{multline*}
\int_{V_j}\str\big(\tilde\phi_{l,z}^{t_0*}\,T_{\gamma_0}^*\,
\tilde k_{l,z,u}\big(T_{\gamma_0}\tilde\phi_l^{t_0}(\tilde p),\tilde p\big)\big)
\,f(t_0)\,|\omega_{\text{\rm b},l}|(p)\\
\begin{aligned}
&=\int_{[0,\ell(c_j)]\times W_j}\str\big(\tilde\phi_{l,z}^{t_0*}\,T_{\gamma_0}^*\,
\tilde k_{l,z,u}\big(T_{\gamma_0}\tilde\phi_l^{t_0}(\tilde p),\tilde p\big)\big)
\,f(t_0)\,\big|\tilde\omega_{\text{\rm b},l}\big|(\tilde p)\\
&=f(t_0)\int_0^{\ell(c_j)}\int_{W_j}\str
\big(\tilde\phi_{l,z}^{t_0*}\,T_{\gamma_0}^*\,\tilde k_{l,z,u}\big(\big(x,T_{\gamma_0}\tilde\phi_{l,x}^{t_0}(y)\big),(x,y)\big)\,|dx|\\
&=f(t_0)\int_0^{\ell(c_j)}\int_{W_j}\str
\big(\tilde\phi_{l,x,z}^{t_0*}\,T_{\gamma_0}^*\,\tilde k_{l,x,u,z}\big(T_{\gamma_0}\tilde\phi_{l,x}^{t_0}(y),y\big)\,|dx|\;.
\end{aligned}
\end{multline*}
But, by \Cref{p: local Lefschetz formula},
\[
\lim_{u\downarrow0}\int_{W_j}\str
\big(\tilde\phi_{l,x,z}^{t_0*}\,T_{\gamma_0}^*\,\tilde k_{l,x,u,z}\big(T_{\gamma_0}\tilde\phi_{l,x}^{t_0}(y),y\big)\big)
=\epsilon_{y_j}\big(T_{\gamma_0}\tilde\phi_{l,x}^{t_0}\big)=\epsilon_{c_j}(k_j)\;.
\]
So
\begin{multline}\label{lim_u to 0 ... = f(t_0) ell(c_j) varepsilon_c_j(k_j)}
\lim_{u\downarrow0}\int_{V_j}\str\big(\tilde\phi_{l,z}^{t_0*}\,T_{\gamma_0}^*\,
\tilde k_{l,z,u}\big(T_{\gamma_0}\tilde\phi_l^{t_0}(\tilde p),\tilde p\big)\big)\,f(t_0)\,|\omega_{\text{\rm b},l}|(p) \\
=f(t_0)\ell(c_j)\epsilon_{c_j}(k_j)\;.
\end{multline}

By~\ref{i: no periods in I}, we can assume the length of $I$ is as small as desired. By the $\R$-local bounded geometry of $\tilde\phi_l$, if the length of $I$ is small enough, there is some $0<r<C_3^{-1}/2$ such that $d_l(\phi_l^t(\tilde p),\phi_l^s(\tilde p))\le r$ for all $\tilde p\in\widetilde M^1_l$ and $t,s\in I$. So, by~\eqref{C_3^-1 |gamma| le tilde d_l(tilde p gamma . tilde p) le C_3 |gamma|},~\ref{i-(A): t_0 = h_l(gamma_0)} and~\ref{i-(A): gamma_0 cdot tilde phi_l^t_0(tilde p) = tilde p}, for all $p\in c_j$ and $\gamma\in\Gamma_l\setminus\{\gamma_0\}$ with $-h_l(\gamma)\in I$,
\begin{multline*}
d_l\big(\gamma\cdot\tilde\phi_l^{-h_l(\gamma)}(\tilde p),\tilde p\big)\\
\begin{aligned}
&\ge d_l\big(\gamma\cdot\tilde\phi_l^{-h_l(\gamma)}(\tilde p),\gamma_0\cdot\tilde\phi_l^{-h_l(\gamma)}(\tilde p)\big)\\
&\phantom{=\text{}}\text{}-d_l\big(\gamma_0\cdot\tilde\phi_l^{-h_l(\gamma)}(\tilde p),
\gamma_0\cdot\tilde\phi_l^{-h_l(\gamma_0)}(\tilde p)\big)
-d_l\big(\gamma_0\cdot\tilde\phi_l^{-h_l(\gamma_0)}(\tilde p),\tilde p\big)\\
&=d_l\big(\tilde\phi_l^{-h_l(\gamma)}(\tilde p),\gamma^{-1}\gamma_0\cdot\tilde\phi_l^{-h_l(\gamma)}(\tilde p)\big)
-d_l\big(\tilde\phi_l^{-h_l(\gamma)}(\tilde p),\tilde\phi_l^{-h_l(\gamma_0)}(\tilde p)\big)\\
&\ge C_3^{-1}|\gamma^{-1}\gamma_0|-r\;.
\end{aligned}
\end{multline*}
Thus, by continuity, the neighborhood $W_j$ of every $y_j$ can be chosen so small that, for all $p\in V_j$ and $\gamma\in\Gamma_l\setminus\{\gamma_0\}$ with $-h_l(\gamma)\in I$,
\[
d_l\big(\gamma\cdot\tilde\phi_l^{-h_l(\gamma)}(\tilde p),\tilde p\big)\ge C_3^{-1}|\gamma^{-1}\gamma_0|-2r\ge C_3^{-1}-2r>0\;.
\]
Hence, by~\eqref{heat kernel estimates} and since the leaves of $\widetilde\FF^1_l$ are of equi-bounded geometry, there are $C_1,C_2,u_0>0$ such that, for all $0<u\le u_0$, $p\in V_j$ and $\gamma\in\Gamma_l\setminus\{\gamma_0\}$ with $-h_l(\gamma)\in I$,
\[
\big|\tilde k_{l,z,u}\big(\gamma\cdot\tilde\phi_l^{-h_l(\gamma)}(\tilde p),\tilde p\big)\big|
\le C_1u^{(n-1)/2}e^{-C_2(C_3^{-1}|\gamma^{-1}\gamma_0|-2r)^2/u}\;.
\]
Then, by~\eqref{|tilde phi_l z^t*| le C_4 |f(t)| le C_5}, for all $0<u\le u_0$ and $p\in V_j$,
\begin{multline*}
\bigg|\sum_{\gamma\in\Gamma_l\setminus\{\gamma_0\}}
\tilde\phi_{l,z}^{-h_l(\gamma)*}\,T_\gamma^*\,\tilde k_{l,z,u}\big(T_\gamma\tilde\phi_l^{-h_l(\gamma)}(\tilde p),\tilde p\big)\,f(-h_l(\gamma))\,
|\omega_{\text{\rm b},l}|(p)\bigg|\\
\le C_4C_5C_1u^{(n-1)/2}\sum_{\gamma\in\Gamma_l\setminus\{\gamma_0\}}e^{-C_2(C_3^{-1}|\gamma^{-1}\gamma_0|-2r)^2/u}\;,
\end{multline*}
which converges to zero as $u\downarrow0$ because $\Gamma_l$ is of polynomial growth. So, by~\eqref{mathring K_l u z(p p)} and~\eqref{lim_u to 0 ... = f(t_0) ell(c_j) varepsilon_c_j(k_j)},
\begin{equation}\label{lim_u to 0 int_p in V_j trs mathring K_l u z(p p) = f(t_0) ell(c_j) varepsilon_c_j(k_j)}
\lim_{u\downarrow0}\int_{p\in V_j}\str\mathring k_{l,z,u}(p,p)=f(t_0)\ell(c_j)\epsilon_{c_j}(k_j)\;.
\end{equation}

On the other hand, since $\phi$ has no closed orbits in $T^1_{l,\epsilon}$ (\Cref{ss: globalization}), we can assume $V_j\subset M^1_{l,\epsilon''}$. Let $\widetilde V_j=\pi_l^{-1}(V_j)\subset\widetilde M^1_{l,\epsilon''}$. If $p\in M^1_{l,\epsilon''}\setminus(V_1\cup\dots\cup V_m)$ and $t\in I$, then $\phi^t(p)\ne p$. Hence, like in the proof of~\ref{i: no periods in I}, there are $C_7,C_1,C_2,u_0>0$ such that~\eqref{d_widetilde FF^1_l(gamma^-1 . tilde phi_l^t(tilde p) tilde p) ge C_7 |gamma|} and~\eqref{| tilde phi_l^h_l(gamma)* T_gamma^-1* tilde k_l u z(gamma^-1 tilde phi^h_l(gamma)(tilde p) tilde p) | le ...} hold for all $0<u\le u_0$, $\tilde p\in\widetilde M^1_{l,\epsilon''}\setminus(\widetilde V_1\cup\dots\cup\widetilde V_m)$ and $\gamma\in\Gamma_l$ with $-h_l(\gamma)\in I$. Thus~\eqref{| mathring K_l u z(p p) | le ...} holds for all $p$ in the compact space $M^1_{l,\epsilon''}\setminus(V_1\cup\dots\cup V_m)$, yielding
\[
\lim_{u\downarrow0}\int_{p\in M^1_{l,\epsilon''}\setminus(V_1\cup\dots\cup V_m)}\str\mathring k_{l,z,u}(p,p)=0\;.
\]
So~\ref{i: one period in I} is true by~\eqref{lim_u to 0 int_p in V_j trs mathring K_l u z(p p) = f(t_0) ell(c_j) varepsilon_c_j(k_j)} and \Cref{c: smallnuint is cont,c: kappa_l u to 0  on Deltab cap beta_b^-1(T_l epsilon'^2)}.

Finally, assume $0\in I$ and $I\cap\PP_l=\emptyset$ to prove~\ref{i: 0 in I}. By~\ref{i: no periods in I}, we can suppose again that the length of $I$ is as small as desired. By~\eqref{C_3^-1 |gamma| - R le d_widetilde FF^1_l(gamma^-1 . tilde phi^h_l(gamma)(tilde p) tilde p) le C_3 |gamma| + R}, there are finitely many elements $\gamma\in\Gamma_l$ such that $-h_l(\gamma)\in I$ and, for all $\tilde p\in\widetilde M^1_l$,
\begin{equation}\label{... > 1}
d_{\widetilde\FF^1_l}\big(\gamma\cdot\tilde\phi_l^{-h_l(\gamma)}(\tilde p),\tilde p\big)>1\;.
\end{equation}
Thus, if $I$ is small enough, we can assume~\eqref{... > 1} is true for all $\tilde p\in\widetilde M^1_l$ and $\gamma\in\Gamma_l\setminus\{e\}$ with $-h_l(\gamma)\in I$. Then, like in the proof of~\ref{i: no periods in I}, there are $C_7,C_1,C_2,u_0>0$ such that~\eqref{d_widetilde FF^1_l(gamma^-1 . tilde phi_l^t(tilde p) tilde p) ge C_7 |gamma|} and~\eqref{| tilde phi_l^h_l(gamma)* T_gamma^-1* tilde k_l u z(gamma^-1 tilde phi^h_l(gamma)(tilde p) tilde p) | le ...} hold for all $0<u\le u_0$, $\tilde p\in\widetilde M^1_{l,\epsilon''}$ and $\gamma\in\Gamma_l\setminus\{e\}$ with $-h_l(\gamma)\in I$. Hence, by~\eqref{|tilde phi_l z^t*| le C_4 |f(t)| le C_5}, for all $0<u\le u_0$ and $p\in M^1_{l,\epsilon''}$,
\begin{multline*}
\bigg|\sum_{\gamma\in\Gamma_l\setminus\{e\}}
\tilde\phi_{l,z}^{-h_l(\gamma)*}\,T_\gamma^*\,\tilde k_{l,z,u}\big(T_\gamma\tilde\phi_l^{-h_l(\gamma)}(\tilde p),\tilde p\big)\,f(-h_l(\gamma))\,
|\omega_{\text{\rm b},l}|(p)\bigg|\\
\le C_4C_5C_1u^{(n-1)/2}\sum_{\gamma\in\Gamma_l}e^{-C_2C_7^2|\gamma|^2/u}\;,
\end{multline*}
which converges to zero as $u\downarrow0$ because $\Gamma_l$ is of polynomial growth. On the other hand, by~\eqref{asymptotic expansion} and \Cref{t: e_z l},
\[
\lim_{u\downarrow0}\str\tilde k_{l,z,u}(\tilde p,\tilde p)=
\begin{cases}
e(\widetilde\FF_l, g_{\widetilde\FF_l})(\tilde p)\equiv e(\FF_l,g_{\FF_l})([\tilde p]) & \text{if $n-1$ is even}\\
0 & \text{if $n-1$ is odd}\;,
\end{cases}
\]
uniformly on $\tilde p\in\widetilde M^1_{l,\epsilon''}$. So, by~\eqref{mathring K_l u z(p p)},
\[
\lim_{u\downarrow0}\str\kappa_{l,u}=f(0)\,e(\FF_l,g_{\FF_l})\,|\omega_{\text{\rm b},l}|
\]
uniformly on $\Delta_{\text{\rm b},l}\cap\beta_{\text{\rm b}}^{-1}(M_{l,\epsilon''}^2)\equiv M_{l,\epsilon''}$. Therefore~\ref{i: 0 in I} follows using \Cref{c: smallnuint is cont,c: lim_u to 0 trs(kappa_l u|_Deltab) equiv f(0) Pf(R_FF_l) |omega_l|,r: smallnuint is cont}.
\end{proof}

\begin{rem}\label{r: lim_u -> 0 bTrs(P_u z) if there are no preserved leaves} 
The simpler argument given in \cite{AlvKordy2002,AlvKordy2008a} for the case of \Cref{t: lim_u -> 0 bTrs(P_u z)} with no preserved leaves cannot be applied here because now $\bStr(\bfP_u)$ depends on $u$.
\end{rem}

\Cref{t: intro - lim_u->0 bTrs(bfP_mu u f} is a restatement of \Cref{t: lim_u -> 0 bTrs(P_u z)}.

\section{The limit of $\bStr(\bfP_{\mu,u})$ as $u\uparrow+\infty$ and $\mu\to\pm\infty$}
\label{s: lim_mu to infty lim_u to infty bTrs(P_u mu)}

\subsection{An expression of $\bTr([d_{\bfFF,\mu},\bfP_\mu\sw])$}\label{ss: bTrs([d_bfFF mu P_mu sw])}

From now on, we will only consider $\bfP_z$ for $z=\mu\in\R$; written $\bfP_\mu$. We keep the notation $z=\mu+i\lambda$ for any other $\lambda\in\R$ ($i=\sqrt{-1}$). In the following, $L$ runs in $\pi_0M^0$. Recall that $\eta_0=\eta$ around $M^0$. For $\psi\in\AA$, $\mu\in\R$ and $f\in\Cinftyc(\R)$, let
\begin{align}
S_{L,\mu}&=-\frac1{2\pi}\int_{-\infty}^{+\infty} {\eta_L\wedge}\,\psi(D_{L,z})\hat f(-\varkappa_L\lambda)\,d\lambda\notag\\
&=-\frac1{2\pi|\varkappa_L|}\int_{-\infty}^{+\infty} {\eta_L\wedge}\,\psi(D_{L,z})\widehat{f_L}(\lambda)\,d\lambda\;,\label{S_L mu}
\end{align}
where $f_L(\lambda)=f(-\lambda/\varkappa_L)$. Again, we may also add the subscript ``$\psi$'' or ``$f$'' to the notation $S_{L,\mu}$ if needed. Recall also that $\sw$ denotes the degree involution. Observe that $\sw d_{\bfFF,z}=-d_{\bfFF,z}\sw$ and $I_{\bfnu}(\bfP_\mu\sw,\lambda)=I_{\bfnu}(\bfP_\mu,\lambda)\,\sw$ by~\eqref{K_I_nu(A lambda)(y y')}.

\begin{lem}\label{l: bTr([d_FF mu P_mu bfw]) = 2 sum_L Trs(S_L mu)}
We have
\[
\bTr([d_{\bfFF,\mu},\bfP_\mu\sw])=2\sum_L\Str(S_{L,\mu})\;.
\]
\end{lem}

\begin{proof}
By the version of~\eqref{bTr[A,B]} with a b-differential operator and a b-pseudodifferential operator of order $-\infty$, \Cref{c: P_l u in Psib^-infty(M_l bigwedge T^*FF_l),p: d_FF_l z in Psi_b^c},
\begin{align*}
\bTr([d_{\bfFF,\mu},\bfP_\mu\sw])
&=-\frac{1}{2\pi i}\int_{-\infty}^{+\infty}\Tr(\partial_\lambda I_{\bfnu}(d_{\bfFF,\mu},\lambda)\, I_{\bfnu}(\bfP_\mu,\lambda)\sw)\,d\lambda\\
&=-\frac{1}{\pi}\sum_L\int_{-\infty}^{+\infty}\int_{-\infty}^{+\infty}\Tr({\eta_L\wedge}\,\psi(D_{L,z})\sw)\,e^{i\lambda\varkappa_Lt}f(t)\,dt\,d\lambda\\
&=2\sum_L\Str(S_{L,\mu})\;.\qedhere
\end{align*}
\end{proof}

\subsection{Variation of $\bStr(\bfP_{\mu,u})$ with respect to $u$}\label{ss: variation of bTrs(P_mu u)}

For any $\psi\in\AA$ and $u>0$, let $\psi_u\in\AA$ be defined by $\psi_u(x)=\psi(\sqrt ux)$, and consider the corresponding operator $\bfP_{\mu,u}$. Recall that $\bfP_{\mu,u}$ is the operator $\bfP_{\mu,u,f}$ of \Cref{ss: intro - leafwise Witten} if $\psi(x)=e^{-x^2}$.

\begin{prop}\label{p: d/du bTrs(P_mu u)}
If $\psi\in\AA$ is even, then
\[
\frac{d}{du}\bStr(\bfP_{\mu,u})=-\frac1{\sqrt u}\sum_L\Str(S_{L,(\psi')_u,\mu})\;.
\]
\end{prop}

\begin{proof}
This result follows like in the heat equation proof of the usual Lefschetz trace formula
\cite{AtiyahBott1967,Gilkey1995,Roe1998}, but the stated derivative does not vanish because the b-trace of commutators may not be zero. To simplify the arguments, consider the change of variables $v=\sqrt u$, and let $\psi^v(x)=\psi_u(x)=\psi(vx)$ and $\bfP^v_\mu=\bfP_{\psi^v,\mu,f}$. By \Cref{l: bTr([d_FF mu P_mu bfw]) = 2 sum_L Trs(S_L mu)} and since $\psi'$ is odd,
\begin{multline*}
\bStr\left(\int_{-\infty}^{+\infty}\phi^{t*}_\mu d_{\bfFF,\mu}\,\psi'(vD_{\bfFF,\mu})\,f(t)\,dt\right)\\
\begin{aligned}
&=\bTr\left(\int_{-\infty}^{+\infty}\phi^{t*}_\mu d_{\bfFF,\mu}\,\psi'(vD_{\bfFF,\mu})\,\sw\,f(t)\,dt\right)\\
&=\bTr\left(\int_{-\infty}^{+\infty}\phi^{t*}_\mu \psi'(vD_{\bfFF,\mu})\,\sw\,d_{\bfFF,\mu}f(t)\,dt\right)+2\sum_L\Str(S_{L,(\psi')^v,\mu})\\
&=-\bTr\left(\int_{-\infty}^{+\infty}\phi^{t*}_\mu \delta_{\bfFF,\mu}\,\psi'(vD_{\bfFF,\mu})\,\sw\,f(t)\,dt\right)+2\sum_L\Str(S_{L,(\psi')^v,\mu})\\
&=-\bStr\left(\int_{-\infty}^{+\infty}\phi^{t*}_\mu \delta_{\bfFF,\mu}\,\psi'(vD_{\bfFF,\mu})\,f(t)\,dt\right)+2\sum_L\Str(S_{L,(\psi')^v,\mu})\;.
\end{aligned}
\end{multline*}
So
\begin{align*}
\frac{d}{dv}\bTr^{\text{\rm s}} \bfP^v_\mu
&=\bStr\left(\int_{-\infty}^{+\infty}\phi^{t*}_\mu D_{\bfFF,\mu}\psi'(vD_{\bfFF,\mu})f(t)\,dt\right)\\
&=\bStr\left(\int_{-\infty}^{+\infty}\phi^{t*}_\mu d_{\bfFF,\mu}\psi'(vD_{\bfFF,\mu})f(t)\,dt\right)\\
&\phantom{=\text{}}\text{}+\bStr\left(\int_{-\infty}^{+\infty}\phi^{t*}_\mu\delta_{\bfFF,\mu}\psi'(vD_{\bfFF,\mu})f(t)\,dt\right)\\
&=2\sum_L\Str(S_{L,(\psi')^v,\mu})\;.
\end{align*}
Now apply the chain rule.
\end{proof}

\subsection{The limit of $\bStr(\bfP_{\mu,u})$ as $u\uparrow+\infty$ and $\mu\to\pm\infty$}
\label{ss: lim_mu to infty lim_u to infty bTrs(P_u mu)}

Now take $\psi(x)=e^{-x^2}$. Hence $\psi_u(x)=e^{-ux^2}$ and $(\psi')_u(x)=-2\sqrt uxe^{-ux^2}$. Thus, by~\eqref{S_L mu},
\begin{equation}\label{Trs(S_L (psi')_u mu)}
\Str(S_{L,(\psi')_u,\mu})
=-\frac{1}{\pi|\varkappa_L|}\int_{-\infty}^{+\infty}
\Str\big({\eta_L\wedge}\,\delta_ze^{-u\Delta_{L,z}}\big)\widehat{f_L}(\lambda)\,d\lambda\;.
\end{equation}

\begin{thm}\label{c: bTrs(P_u z)}
For all $\tau\gg0$, we can choose every $\eta_L$ and $g_L$ {\rm(}$L\in\pi_0M^0${\rm)} so that
\[
\lim_{\mu\uparrow+\infty}\Big(\lim_{u_1\uparrow+\infty}\bfP_{\mu,u_1}-\lim_{u_0\downarrow0}\bfP_{\mu,u_0}\Big)=\tau f(0)\;.
\]
If $n-1$ is even, this is true for all $\tau\in\R$ and as $\mu\to\pm\infty$.
\end{thm}

\begin{proof}
By \Cref{t: Z}, if $\tau\gg0$, we can choose every $\eta_L$ and $g_L$ so that~\eqref{Z_mu} defines a tempered distribution $Z_{L,\mu}:=Z(L,g_L,\eta_L)\in\SS'$ for $|\mu|\gg0$, and $Z_{L,\mu}\to\tau\delta_0$ in $\SS'$ as $\mu\to\infty$. If $n-1=\dim L$ is even, then this is true for all $\tau\in\R$ and as $\mu\to\pm\infty$. Then the result follows because, by~\eqref{Z_mu},~\eqref{Trs(S_L (psi')_u mu)} and \Cref{p: d/du bTrs(P_mu u)},
\begin{align*}
\lim_{u_1\uparrow+\infty}\bfP_{\mu,u_1}-\lim_{u_0\downarrow0}\bfP_{\mu,u_0}
&=-\frac1{\sqrt u}\sum_L\int_0^\infty\Str(S_{L,(\psi')_u,\mu})\,du\\
&=-\frac{1}{\pi}\sum_L\frac1{|\varkappa_L|}\langle Z_{L,\mu},f_L\rangle\;.\qedhere
\end{align*}
\end{proof}

\Cref{c: bTrs(P_u z)} gives \Cref{t: intro - lim_mu->pm infty ...} by taking $\tau=0$ when $n-1$ is even. 

\begin{cor}\label{c: bTrs(P_u z)}
For all $\tau\gg0$, we can choose every $\eta_L$ and $g_L$ {\rm(}$L\in\pi_0M^0${\rm)} so that
\[
\lim_{\mu\uparrow+\infty}\lim_{u\uparrow+\infty}\bStr(\bfP_{\mu,u})
=\big(\bchi_{|\bfomega_b|}(\mathring\bfFF)+\tau\big)\,f(0)
+\sum_{c\in\CC}\ell(c)\sum_{k\in\Z^\times}\epsilon_c(k)\,f(k\ell(c))\;.
\]
If $n$ is even, this is true for all $\tau\in\R$ and as $\mu\to\pm\infty$.
\end{cor}

\begin{proof}
Apply \Cref{t: lim_u -> 0 bTrs(P_u z),c: bTrs(P_u z)}.
\end{proof}

\Cref{t: intro - lim_mu->pm infty ...} follows taking $\tau=0$ in \Cref{c: bTrs(P_u z)}.

Like in\cite{AlvKordy2002,AlvKordyLeichtnam2020}, by \eqref{bar H^bullet J^m(FF) cong ker ...} and~\eqref{bar H^bullet J^prime m(FF) cong ker ...}, the distributions
\[
f\mapsto\lim_{u\uparrow+\infty}\bStr\big(\bfP_{m+\frac12,u,f}\big)\;,\quad 
f\mapsto\lim_{u\uparrow+\infty}\bStr\big(\bfP_{m-\frac12,u,f}\big)
\]
can be considered as a distributional supertraces of the action $\bfphi^*$ of $\R$ on $\bar H^\bullet J^m(\FF)$ and $\bar H^\bullet J^{\prime\,m}(\FF)$. So, by the analogs of~\eqref{varinjlim H^bullet I^(s)(FF) cong H^bullet I(FF)} and~\eqref{bar H^bullet I'(FF) cong projlim bar H^bullet I^(prime s)(FF)} for $J(\FF)$ and $J'(\FF)$, and using~\eqref{bar H^bullet J^m(FF) cong ker ...} and~\eqref{bar H^bullet J^prime m(FF) cong ker ...}, the distributions
\[
f\mapsto\lim_{m\downarrow-\infty}\lim_{u\uparrow+\infty}\bStr\big(\bfP_{m+\frac12,u,f}\big)\;,\quad 
f\mapsto\lim_{m\downarrow\infty}\lim_{u\uparrow+\infty}\bStr\big(\bfP_{m-\frac12,u,f}\big)
\]
can be considered as a distributional supertraces of the action $\bfphi^*$ of $\R$ on $\bar H^\bullet J(\FF)$ and $\bar H^\bullet J'(\FF)$, as indicated in \Cref{ss: L_dis(phi)}.


\backmatter

\bibliographystyle{amsalpha}

\bibliography{../../../../Suso}

\providecommand{\bysame}{\leavevmode\hbox to3em{\hrulefill}\thinspace}
\providecommand{\MR}{\relax\ifhmode\unskip\space\fi MR }
\providecommand{\MRhref}[2]{%
  \href{http://www.ams.org/mathscinet-getitem?mr=#1}{#2}
}
\providecommand{\href}[2]{#2}
\begin{thebibliography}{{\'{A}}LKL23}

\bibitem[AB67]{AtiyahBott1967}
M.F. Atiyah and R.~Bott, \emph{A {L}efschetz fixed point formula for elliptic
  complexes. {I}}, Ann. of Math. (2) \textbf{86} (1967), 374--407. \MR{0212836}

\bibitem[Ada75]{Adams1975}
R.A. Adams, \emph{Sobolev spaces}, Pure and Applied Mathematics, vol.~65,
  Academic Press, Inc., New York-San Francisco-London, 1975. \MR{0450957}

\bibitem[{\'{A}}L89]{Alv1989a}
J.A. {\'{A}}lvarez~L{\'{o}}pez, \emph{A finiteness theorem for the spectral
  sequence of a {R}iemannian foliation}, Illinois J. Math. \textbf{33} (1989),
  no.~1, 79--92. \MR{974012}

\bibitem[{\'{A}}LG21]{AlvGilkey2021a}
J.A. {\'{A}}lvarez~L{\'{o}}pez and P.~Gilkey, \emph{The local index density of
  the perturbed de~{R}ham complex}, Czechoslovak Math. J. \textbf{71} (2021),
  no.~3, 901--932. \MR{4295254}

\bibitem[{\'{A}}LK01]{AlvKordy2001}
J.A. {\'{A}}lvarez~L{\'{o}}pez and Y.A. Kordyukov, \emph{Long time behavior of
  leafwise heat flow for {Riemannian} foliations}, Compos. Math. \textbf{125}
  (2001), no.~2, 129--153.
  \MR{1815391}

\bibitem[{\'{A}}LK02]{AlvKordy2002}
\bysame, \emph{Distributional {B}etti numbers of transitive foliations of
  codimension one}, {Foliations: geometry and dynamics. Proceedings of the
  Euroworkshop, Warsaw, Poland, May 29--June 9, 2000} (Singapore), World Sci.
  Publ., 2002, pp.~159--183. \MR{1882768}

\bibitem[{\'{A}}LK08]{AlvKordy2008a}
\bysame, \emph{Lefschetz distribution of {L}ie foliations}, {$C^\ast$}-algebras
  and elliptic theory {II} (Basel), Trends Math., Birkh{\"{a}}user, 2008,
  pp.~1--40. \MR{2408134}

\bibitem[{\'{A}}LKL14]{AlvKordyLeichtnam2014}
J.A. {\'{A}}lvarez~L{\'{o}}pez, Y.A. Kordyukov, and E.~Leichtnam,
  \emph{Riemannian foliations of bounded geometry}, Math. Nachr. \textbf{287}
  (2014), no.~14-15, 1589--1608. \MR{3266125}

\bibitem[{\'{A}}LKL20]{AlvKordyLeichtnam2020}
\bysame, \emph{Analysis on {Riemannian} foliations of bounded geometry},
  M{\"{u}}nster J. Math. \textbf{13} (2020), 221--265, arXiv:1905.12912.

\bibitem[{\'{A}}LKL21]{AlvKordyLeichtnam-ziomf}
\bysame, \emph{Zeta invariants of {M}orse forms}, arXiv:2112.03191, 2021.

\bibitem[{\'{A}}LKL22]{AlvKordyLeichtnam2022}
\bysame, \emph{Simple foliated flows}, Tohoku Math. J. (2) \textbf{74} (2022),
  no.~1, 53--81. \MR{4374665}

\bibitem[{\'{A}}LKL23]{AlvKordyLeichtnam-conormal}
\bysame, \emph{The topology of the space of conormal distributions},
  arXiv:2304.00798, 2023.

\bibitem[{\'{A}}LT91]{AlvTond1991}
J.A. {\'{A}}lvarez~L{\'{o}}pez and P.~Tondeur, \emph{Hodge decomposition along
  the leaves of a {R}iemannian foliation}, J. Funct. Anal. \textbf{99} (1991),
  no.~2, 443--458. \MR{1121621}

\bibitem[Bar81]{Barner1981}
K.~Barner, \emph{On {A}.~{W}eil's explicit formula}, J. Reine Angew. Math.
  \textbf{323} (1981), 139--152. \MR{611448}

\bibitem[BE91]{ButtigEichhorn1991}
I.~Buttig and J.~Eichhorn, \emph{The heat kernel for $p$-forms on manifolds of
  bounded geometry}, Acta Sci. Math. (Szeged) \textbf{55} (1991), no.~1-2,
  33--51. \MR{1124942}

\bibitem[BF97]{BravermanFarber1997}
M.~Braverman and M.~Farber, \emph{Novikov type inequalities for differential
  forms with non-isolated zeros}, Math. Proc. Cambridge Philos. Soc.
  \textbf{122} (1997), 357--375. \MR{1458239}

\bibitem[BGV04]{BerlineGetzlerVergne2004}
N.~Berline, E.~Getzler, and M.~Vergne, \emph{Heat kernels and {D}irac
  operators}, Grundlehren Text Editions, Springer-Verlag, Berlin, 2004,
  Corrected reprint of the 1992 original. \MR{2273508}

\bibitem[Bou14]{Bourles2014}
H.~Bourl{\`{e}}s, \emph{On the closed graph theorem and the open mapping
  theorem}, arXiv:1411.5500 [math.FA], 2014.

\bibitem[BT82]{BottTu1982}
R.~Bott and L.W. Tu, \emph{Differential forms in algebraic topology}, Graduate
  Texts in Mathematics, vol.~82, Springer-Verlag, New York-Heidelberg-Berlin,
  1982. \MR{658304}

\bibitem[BZ92]{BismutZhang1992}
J.-M. Bismut and W.~Zhang, \emph{An extension of a theorem by {C}heeger and
  {M}{\"{u}}ller}, Ast{\'{e}}risque \textbf{205} (1992), 235 pp., with an
  appendix by F.~Laudenbach. \MR{1185803}

\bibitem[CC00]{CandelConlon2000-I}
A.~Candel and L.~Conlon, \emph{Foliations. {I}}, Graduate Studies in
  Mathematics, vol.~23, American Mathematical Society, Providence, RI, 2000.
  \MR{1732868}

\bibitem[CC03]{CandelConlon2003-II}
\bysame, \emph{Foliations. {II}}, Graduate Studies in Mathematics, vol.~60,
  American Mathematical Society, Providence, RI, 2003. \MR{1994394}

\bibitem[CGT82]{CheegerGromovTaylor1982}
J.~Cheeger, M.~Gromov, and M.~Taylor, \emph{Finite propagation speed, kernel
  estimates for functions of the {L}aplace operator, and the geometry of
  complete {R}iemannian manifolds}, J. Differential Geom. \textbf{17} (1982),
  no.~1, 15--53. \MR{658471}

\bibitem[Che73]{Chernoff1973}
P.R. Chernoff, \emph{Essential self-adjointness of powers of generators of
  hyperbolic equations}, J. Funct. Anal. \textbf{12} (1973), 401--414.
  \MR{0369890}

\bibitem[CLN85]{CamachoLinsNeto1985}
C.~Camacho and A.~Lins~Neto, \emph{Geometric theory of foliations},
  Birkh{\"{a}}user, Boston-Basel-Stuttgart, 1985, Translated from the
  Portuguese by Sue E.~Goodman,
  \url{http://dx.doi.org/10.1007/978-1-4612-5292-4}. \MR{824240}

\bibitem[Con79]{Connes1979}
A.~Connes, \emph{Sur la th\'eorie non commutative de l'int\'egration},
  Alg\`ebres d'op\'erateurs ({S}\'em., {L}es {P}lans-sur-{B}ex, 1978) (Berlin),
  Lecture Notes in Math., vol. 725, Springer, 1979, pp.~19--143. \MR{548112}

\bibitem[Con82]{Connes1982}
\bysame, \emph{A survey of foliations and operator algebras}, Operator Algebras
  and Applications, Kingston, 1980, Proc. Sympos. Pure Math., vol. 38-1, 1982,
  pp.~521--628.

\bibitem[Den98]{Deninger1998}
C.~Deninger, \emph{Some analogies between number theory and dynamical systems
  on foliated spaces}, Doc. Math. \textbf{Extra Vol. I} (1998), 163--186,
  Proceedings of the {I}nternational {C}ongress of {M}athematicians, {V}ol. {I}
  ({B}erlin, 1998). \MR{1648030}

\bibitem[Den01]{Deninger2001}
\bysame, \emph{Number theory and dynamical systems on foliated spaces},
  Jahresber. Deutsch. Math.-Verein. \textbf{103} (2001), no.~3, 79--100.
  \MR{1873325}

\bibitem[Den02]{Deninger2002}
\bysame, \emph{On the nature of the ``explicit formulas'' in analytic number
  theory---a simple example}, Number theoretic methods ({I}izuka, 2001), Dev.
  Math., vol.~8, Kluwer Acad. Publ., Dordrecht, 2002,
  pp.~97--118. \MR{1974137}

\bibitem[Den05]{Deninger-Arith_geom_anal_fol_sps}
\bysame, \emph{Arithmetic geometry and analysis on foliated spaces},
  arXiv:math/0505354, 2005.

\bibitem[Den08]{Deninger2008}
\bysame, \emph{Analogies between analysis on foliated spaces and arithmetic
  geometry}, Groups and analysis, London Math. Soc. Lecture Note Ser., vol.
  354, Cambridge Univ. Press, Cambridge, 2008,
  pp.~174--190.
  \MR{2528467}

\bibitem[Den22]{Deninger-Dyn-Syst-for-arith-schemes}
\bysame, \emph{Dynamical systems for arithmetic schemes}, arXiv:1807.06400,
  2022.

\bibitem[Den23]{Deninger-Primes-knots-periodic-orbits}
\bysame, \emph{Primes, knots and periodic orbits}, arXiv:2301.11643, 2023.

\bibitem[dR84]{deRham1984}
G.~de~Rham, \emph{Differentiable manifolds}, Grundlehren der mathematischen
  Wissenschaften, vol. 266, Springer-Verlag, Berlin, 1984, Forms, currents,
  harmonic forms, Translated from the French by F. R. Smith, With an
  introduction by S. S. Chern. \MR{760450}

\bibitem[DS01]{DeningerSinghof2001}
C.~Deninger and W.~Singhof, \emph{A counterexample to smooth leafwise {H}odge
  decomposition for general foliations and to a type of dynamical trace
  formulas}, Ann. Inst. Fourier (Grenoble) \textbf{51} (2001), no.~1, 209--219.
  \MR{1821074}

\bibitem[DS02]{DeningerSinghof2002}
\bysame, \emph{Real polarizable {H}odge structures arising from foliations},
  Ann. Global Anal. Geom. \textbf{21} (2002), no.~4, 377--399. \MR{1910458}

\bibitem[Edw65]{Edwards1965}
R.E. Edwards, \emph{Functional analysis. {T}heory and applications}, Holt,
  Rinehart and Winston, New York-Toronto-London, 1965. \MR{0221256}

\bibitem[Eic91]{Eichhorn1991}
J.~Eichhorn, \emph{The boundedness of connection coefficients and their
  derivatives}, Math. Nachr. \textbf{152} (1991), 145--158. \MR{1121230}

\bibitem[EMT77]{EpsteinMillettTischler1977}
D.B.A. Epstein, K.C. Millett, and D.~Tischler, \emph{Leaves without holonomy},
  J. London Math. Soc. (2) \textbf{16} (1977), no.~3, 548--552. 
  \MR{0464259}

\bibitem[Far95]{Farber1995}
M.~Farber, \emph{Singularities of the analytic torsion}, J. Differential Geom.
  \textbf{41} (1995), no.~3, 528--572. \MR{1338482}

\bibitem[Far04]{Farber2004}
\bysame, \emph{Topology of closed one-forms}, Mathematical Surveys and
  Monographs, vol. 108, Amer. Math. Soc., Providence, RI, 2004. \MR{2034601}

\bibitem[Fed71]{Fedida1971}
E.~Fedida, \emph{Sur les feuilletages de {L}ie}, C. R. Acad. Sci. Paris
  S\'er.~A \textbf{272} (1971), 999--1001. \MR{0285025}

\bibitem[Fed73]{Fedida1973}
\bysame, \emph{Feuilletages de {Lie}, feuilletages du plan}, Lecture Notes in
  Math., vol. 352, Springer, 1973, pp.~183--195.

\bibitem[Gil95]{Gilkey1995}
P.B. Gilkey, \emph{Invariance theory, the heat equation, and the
  {A}tiyah-{S}inger index theorem}, second ed., Studies in Advanced
  Mathematics, CRC Press, Boca Raton, FL, 1995. \MR{1396308}

\bibitem[God91]{Godbillon1991}
C.~Godbillon, \emph{Feuilletages: \'etudes g\'eom\'etriques}, Progress in
  Math., vol.~98, Birkh{\"{a}}user Verlag, Boston-Basel-Stuttgart, 1991.
  \MR{1120547}

\bibitem[GS77]{GuilleminSternberg1977}
V.~Guillemin and S.~Sternberg, \emph{Geometric asymptotics}, Math. Surveys,
  vol.~14, American Mathematical Society, Providence, R.I., 1977. \MR{0516965}

\bibitem[Gui77]{Guillemin1977}
V.~Guillemin, \emph{Lectures on spectral theory of elliptic operators}, Duke
  Math. J. \textbf{44} (1977), no.~3, 485--517. 
  \MR{0448452}

\bibitem[Hae62]{Haefliger1962}
A.~Haefliger, \emph{Vari\'et\'es feuillet\'ees}, Ann. Scuola Norm. Sup. Pisa
  (3) \textbf{16} (1962), 367--397. \MR{0189060}

\bibitem[Hae80]{Haefliger1980}
\bysame, \emph{Some remarks on foliations with minimal leaves}, J. Differential
  Geom. \textbf{15} (1980), no.~2, 269--384. \MR{614370}

\bibitem[Hec72]{Hector1972c}
G.~Hector, \emph{Sur les feuilletages presque sans holonomie}, C. R. Acad. Sci.
  Paris S\'er.~A \textbf{274} (1972), 1703--1706. \MR{0303550}

\bibitem[Hec77]{Hector1977a}
\bysame, \emph{Feuilletages en cylindres}, Geometry and topology ({P}roc. {III}
  {L}atin {A}mer. {S}chool of {M}ath., {I}nst. {M}at. {P}ura {A}plicada {CNP}q,
  {R}io de {J}aneiro, 1976), Lecture Notes in Math., vol. 597, Springer,
  Berlin, 1977, pp.~252--270. \MR{0451260}

\bibitem[Hec78]{Hector1978}
\bysame, \emph{Croissance des feuilletages presque sans holonomie},
  Differential topology, foliations and {G}elfand-{F}uks cohomology ({P}roc.
  {S}ympos., {P}ontif\'\i cia {U}niv. {C}at\'olica, {R}io de {J}aneiro, 1976)
  (Berlin), Lecture Notes in Math., vol. 652, Springer, 1978, pp.~141--182.
  \MR{505659}

\bibitem[HH81]{HectorHirsch1981-A}
G.~Hector and U.~Hirsch, \emph{Introduction to the geometry of foliations.
  {P}art~{A}: Foliations on compact surfaces, fundamentals for arbitrary
  codimension, and holonomy}, Aspects of Mathematics, vol.~E1, Friedr. Vieweg
  \&\ Sohn, Braunschweig, 1981. \MR{639738}

\bibitem[HH83]{HectorHirsch1983-B}
\bysame, \emph{Introduction to the geometry of foliations. {P}art~{B}}, Aspects
  of Mathematics, vol.~E3, Friedr. Vieweg \&\ Sohn, Braunschweig, 1983.

\bibitem[H{\"{o}}r65]{Hormander1965}
L.~H{\"{o}}rmander, \emph{Pseudo-differential operators}, Commun. Pure Appl.
  Math. \textbf{18} (1965), 501--517. \MR{0180740}

\bibitem[Hor66]{Horvath1966-I}
J.~Horv{\'a}th, \emph{Topological vector spaces and distributions}, vol.~I,
  Addison-Wesley Publishing Co., Reading, Mass.-London-Don Mills, Ont., 1966.
  \MR{0205028}

\bibitem[H{\"{o}}r71]{Hormander1971}
L.~H{\"{o}}rmander, \emph{Fourier integral operators. {I}}, Acta Math.
  \textbf{127} (1971), no.~1-2, 79--183. \MR{388463}

\bibitem[H{\"{o}}r83]{Hormander1983-I}
\bysame, \emph{The analysis of linear partial differential operators. {I}.
  {D}istribution theory and {F}ourier analysis}, Grundlehren der Mathematischen
  Wissenschaften, vol. 256, Springer-Verlag, Berlin, 1983. \MR{717035}

\bibitem[H{\"{o}}r85]{Hormander1985-III}
\bysame, \emph{The analysis of linear partial differential operators. {III}.
  {Pseudodifferential} operators}, Grundlehren der Mathematischen
  Wissenschaften, vol. 274, Springer-Verlag, Berlin, 1985.

\bibitem[Kim17]{Kim-fiber_bdls}
J.~Kim, \emph{On the leafwise cohomology and dynamical zeta functions for fiber
  bundles over the circle}, arXiv:1712.04181, 2017.

\bibitem[KN65]{KohnNirenberg1965}
J.J. Kohn and L.~Nirenberg, \emph{An algebra of pseudo-differential operators},
  Commun. Pure Appl. Math. \textbf{18} (1965), 269--305. \MR{0176362}

\bibitem[Kom67]{Komatsu1967}
H.~Komatsu, \emph{Projective and injective limits of weakly compact sequences
  of locally convex spaces}, J. Math. Soc. Japan \textbf{19} (1967), 366--383.
  \MR{217557}

\bibitem[Kop06]{Kopei2006}
F.~Kopei, \emph{A remark on a relation between foliations and number theory},
  Foliations 2005. Proceedings of the international conference, University of
  {\L}od\'z, {\L}od\'z, Poland, June 13--24, 2005 (Hackensack, NJ) (Pawe{\l}
  et~al. Walczak, ed.), World Sci. Publ., 2006, pp.~245--249. \MR{2284785}

\bibitem[Kop11]{Kopei2011}
\bysame, \emph{A foliated analogue of one- and two-dimensional {A}rakelov
  theory}, Abh. Math. Semin. Univ. Hambg. \textbf{81} (2011), 141--189.
  \MR{2836630}

\bibitem[K{\"{o}}t69]{Kothe1969-I}
G.~K{\"{o}}the, \emph{Topological vector spaces. {I}}, Die Grundlehren der
  mathematischen Wissenschaften, vol. 159, Springer-Verlag,
  Berlin-Heidelberg-New York, 1969, translated from the German by
  D.J.H.~Garling. \MR{0248498}

\bibitem[KP15]{KordyukovPavlenko2015}
Y.A. Kordyukov and V.A. Pavlenko, \emph{On {L}efschetz formulas for flows on
  foliated manifolds}, Ufa Math. J. \textbf{7} (2015), no.~2, 71--101.
  \MR{3430750}

\bibitem[Lei08]{Leichtnam2008}
E.~Leichtnam, \emph{On the analogy between arithmetic geometry and foliated
  spaces}, Rend. Mat. Appl. (7) \textbf{28} (2008), no.~2, 163--188.
  \MR{2463936}

\bibitem[Lei14]{Leichtnam2014}
\bysame, \emph{On the analogy between {$L$}-functions and
  {A}tiyah-{B}ott-{L}efschetz trace formulas for foliated spaces}, Rend. Mat.
  Appl. (7) \textbf{35} (2014), no.~1-2, 1--34. \MR{3241361}

\bibitem[Mel81]{Melrose1981}
R.B. Melrose, \emph{Transformation of boundary problems}, Acta Math.
  \textbf{147} (1981), no.~3-4, 149--236. \MR{639039}

\bibitem[Mel93]{Melrose1993}
\bysame, \emph{The {A}tiyah-{P}atodi-{S}inger index theorem}, Research Notes in
  Mathematics, vol.~4, A.K.~Peters, Ltd., Wellesley, MA, 1993. \MR{1348401}

\bibitem[Mel96]{Melrose1996}
\bysame, \emph{Differential analysis on manifolds with corners},
  \url{http://www-math.mit.edu/~rbm/book.html}, 1996.

\bibitem[Mol88]{Molino1988}
P.~Molino, \emph{Riemannian foliations}, Progress in Mathematics, vol.~73,
  Birkh{\"{a}}user Boston, Inc., Boston, MA, 1988, Translated from the French
  by G.~Cairns, With appendices by Cairns, Y.~Carri\`ere, \'E.~Ghys, E.~Salem
  and V.~Sergiescu.
  \MR{932463}

\bibitem[MS88]{MooreSchochet1988}
C.C. Moore and C.~Schochet, \emph{Global analysis on foliated spaces},
  Mathematical Sciences Research Institute Publications, vol.~9,
  Springer-Verlag, New York, 1988, With appendices by S. Hurder, Moore,
  Schochet and Robert J. Zimmer. \MR{0918974 (89h:58184)}

\bibitem[MU08]{MelroseUhlmann2008}
R.B. Melrose and G.~Uhlmann, \emph{An introduction to microlocal analysis},
  Department of Mathematics, Massachusetts Institute of Technology, Cambridge,
  MA, 2008, \url{https://books.google.es/books?id=Os2jswEACAAJ},
  \url{http://www-math.mit.edu/~rbm/book.html}.

\bibitem[M{\"{u}}m06]{Mumken2006}
B.~M{\"{u}}mken, \emph{On tangential cohomology of {R}iemannian foliations},
  Amer. J. Math. \textbf{128} (2006), no.~6, 1391--1408. \MR{2275025}

\bibitem[NB11]{NariciBeckenstein2011}
L.~Narici and E.~Beckenstein, \emph{Topological vector spaces}, second ed.,
  Pure and Applied Mathematics (Boca Raton), vol. 296, CRC Press, Boca Raton,
  FL, 2011. \MR{2723563}

\bibitem[Nov81]{Novikov1981}
S.P. Novikov, \emph{Multivalued functions and functionals. {A}n analog of the
  {M}orse theory}, Soviet. Math., Dokl. \textbf{24} (1981), 222--226.
  \MR{630459}

\bibitem[Nov82]{Novikov1982}
\bysame, \emph{The {H}amiltonian formalism and a multivalued analogue of
  {M}orse theory}, Russian Math. Surveys \textbf{37} (1982), 1--56. \MR{676612}

\bibitem[Nov02]{Novikov2002}
\bysame, \emph{On the exotic {D}e-{R}ham cohomology. {P}erturbation theory as a
  spectral sequence}, arXiv:math-ph/0201019, 2002.

\bibitem[O'N66]{ONeill1966}
B.~O'Neill, \emph{The fundamental equations of a submersion}, Michigan Math. J.
  \textbf{13} (1966), 459--469. \MR{0200865}

\bibitem[Paj87]{Pajitnov1987}
A.V. Pajitnov, \emph{An analytic proof of the real part of {N}ovikov's
  inequalities}, Soviet Math., Dokl. \textbf{35} (1987), 456--457. \MR{891557}

\bibitem[Pet98]{Petersen1998}
P.~Petersen, \emph{Riemannian geometry}, Graduate Texts in Mathematics, vol.
  171, Springer-Verlag, New York, Berlin, Heidelberg, 1998. \MR{1480173}

\bibitem[Poo81]{Poor1981}
W.A. Poor, \emph{Differential geometric structures}, McGraw-Hill Book Co., New
  York, 1981. \MR{647949}

\bibitem[Roe87]{Roe1987}
J.~Roe, \emph{Finite propagation speed and {C}onnes' foliation algebra}, Math.
  Proc. Cambridge Philos. Soc. \textbf{102} (1987), no.~3, 459--466.
  \MR{906620}

\bibitem[Roe88]{Roe1988I}
\bysame, \emph{An index theorem on open manifolds. {I}}, J. Differential Geom.
  \textbf{27} (1988), no.~1, 87--113. \MR{918459}

\bibitem[Roe98]{Roe1998}
\bysame, \emph{Elliptic operators, topology and asymptotic methods}, second
  ed., Pitman Research Notes in Mathematics, vol. 395, Longman, Harlow, 1998.
  \MR{1670907}

\bibitem[San08]{Sanguiao2008}
L.~Sanguiao, \emph{{$L^2$}-invariants of {R}iemannian foliations}, Ann. Global
  Anal. Geom. \textbf{33} (2008), no.~3, 271--292. \MR{2390835}

\bibitem[Sch71]{Schaefer1971}
H.H. Schaefer, \emph{Topological vector spaces}, Graduate Texts in Mathematics,
  vol.~3, Springer-Verlag, New York, Heidelberg, Berlin, 1971.

\bibitem[Sch96]{Schick1996}
T.~Schick, \emph{Analysis on $\partial$-manifolds of bounded geometry,
  {Hodge-De~Rham} isomorphism and {$L^2$}-index theorem}, Ph.D. thesis,
  Johannes Gutenberg Universit{\"{a}}t Mainz, Mainz, 1996.

\bibitem[Sch01]{Schick2001}
\bysame, \emph{Manifolds with boundary and of bounded geometry}, Math. Nachr.
  \textbf{223} (2001), no.~1, 103--120. \MR{1817852}

\bibitem[See64]{Seeley1964}
R.T. Seeley, \emph{Extension of {$C^{\infty}$} functions defined in a half
  space}, Proc. Amer. Math. Soc. \textbf{15} (1964), 625--626. \MR{0165392}

\bibitem[Shu92]{Shubin1992}
M.A. Shubin, \emph{Spectral theory of elliptic operators on noncompact
  manifolds}, Ast\'erisque \textbf{207} (1992), 35--108, M{\'e}thodes
  semi-classiques, Vol. 1 (Nantes, 1991). \MR{1205177}

\bibitem[Sim90]{Simanca1990}
S.R. Simanca, \emph{Pseudo-differential operators}, Pitman Research Notes in
  Mathematics Series, vol. 236, Longman Scientific \& Technical, Harlow;
  copublished in the United States with John Wiley \& Sons, Inc., New York,
  1990. \MR{1075017}

\bibitem[Tay81]{Taylor1981}
M.E. Taylor, \emph{Pseudodifferential operators}, Princeton Mathematical
  Series, vol.~34, Princeton University Press, Princeton, N.J., 1981.
  \MR{618463}

\bibitem[Val89]{Valdivia1989}
M.~Valdivia, \emph{A characterization of totally reflexive {F}r{\'{e}}chet
  spaces}, Math. Z. \textbf{200} (1989), no.~3, 327--346. \MR{978594}

\bibitem[Wen03]{Wengenroth2003}
J.~Wengenroth, \emph{Derived functors in functional analysis}, Lecture Notes in
  Mathematics, vol. 1810, Springer-Verlag, Berlin, 2003. \MR{1977923}

\bibitem[Wit82]{Witten1982}
E.~Witten, \emph{Supersymmetry and {Morse} theory}, J. Differ. Geom.
  \textbf{17} (1982), 661--692. \MR{683171}

\end{thebibliography}

\printindex

\end{document}